\documentclass[10pt]{amsart}
\usepackage[all]{xy}
\usepackage{amsmath}
\usepackage{amsfonts}
\usepackage{amssymb}
\usepackage{amscd}
\usepackage{amsthm}
\usepackage{latexsym}
\usepackage{amsbsy}
\usepackage{graphicx}
\usepackage{epstopdf}
\usepackage{seqsplit}
\usepackage{color}
\usepackage{float}
\AppendGraphicsExtensions{.gif}
\newtheorem{lemma}{Lemma}[section]
\newtheorem{proposition}{Proposition}[section]
\newtheorem{theorem}{Theorem}[section]
\newtheorem{corollary}{Corollary}[section]

\newtheorem{remark}{Remark}[section]

\newcommand{\ddt}{\frac{d}{dt}\bigm|_{t= 0}}

\newcommand{\ddsone}{\partial_1}

\newcommand{\ddstwo}{\partial_2}

\newcommand{\ddsthree}{\partial_3}

\newcommand{\C}{\mathbb{C}}

\newcommand{\R}{\mathbb{R}}

\newcommand{\Z}{\mathbb{Z}}

\newcommand{\NT}{\mathbb{T}_\theta^2}

\newcommand{\NTp}{\mathbb{T}_{\theta'}^2}

\newcommand{\CNTp}{C^\infty(\mathbb{T}_{\theta'}^2)}

\newcommand{\NTpp}{\mathbb{T}_{\theta''}^2}

\newcommand{\CNTpp}{C^\infty(\mathbb{T}_{\theta''}^2)}

\newcommand{\CNT}{C^\infty(\mathbb{T}_\theta^2)}

\newcommand{\done}{\delta_1}

\newcommand{\dtwo}{\delta_2}

\newcommand{\nab}{\nabla}

\newcommand{\del} {\delta}

\newcommand{\dep}{\frac{d}{d\varepsilon}\bigm|_{\varepsilon = 0}}

\newcommand{\vep}{\varepsilon}

\newcommand{\vphi}{\varphi}

\newcommand{\nabep}{\nabla_{\varepsilon}}

\DeclareMathOperator{\Tr}{Trace}

\DeclareMathOperator{\ad}{ad}

\begin{document}

\title{The term \MakeLowercase{\Large{$a_4$}}
in the heat kernel expansion of noncommutative tori}

\author{
Alain Connes$^{^{\MakeLowercase{a}}}$ and  Farzad Fathizadeh$^{^{\MakeLowercase{b}, 1}}$}

\thanks{ {\it $^1$E-mail addresses:} 
alain@connes.org (Alain Connes), farzadf@caltech.edu (Farzad Fathizadeh)}

\date{}

\maketitle

\begin{center}
\address{  {\it $^a$Coll\'ege de France, 3 rue d'Ulm, Paris F-75005 France; I.H.E.S. and Ohio State University} } \\
\address{ 
{\it $^b$Division of Physics, Mathematics and Astronomy, California Institute of Technology,  1200 E. California Blvd., Pasadena, CA 91125, U.S.A.} } 
\end{center}
\smallskip
\smallskip
\smallskip
\smallskip

\begin{abstract} 
We consider the Laplacian associated with a general metric 
in the canonical conformal structure of the 
noncommutative two torus, and calculate a local expression 
for the term $a_4$ that appears in its corresponding small-time heat 
kernel expansion. The final formula involves one variable 
functions and lengthy two, three and four variable functions 
of the modular automorphism of the state that encodes the 
conformal perturbation of the flat metric.  
We confirm the validity of the calculated expressions by showing 
that they satisfy a family of conceptually predicted functional relations. By studying 
these functional relations abstractly, we derive a partial differential system 
which involves  a natural action of cyclic groups of order two, three and four and 
a flow in parameter space. We 
discover  symmetries of the calculated expressions with respect to 
the action of the cyclic groups.  
In passing, we show that the main ingredients of our 
calculations, which come from a rearrangement lemma 
and relations between the derivatives up to order four of the 
conformal factor and those of its logarithm, can be derived by 
finite differences from the generating function of the Bernoulli 
numbers and its multiplicative inverse. We then shed light on the significance of 
exponential polynomials and their smooth fractions in understanding the 
general structure of the noncommutative geometric invariants appearing in 
the heat kernel expansion. 
As an application of our results we obtain the $a_4$ term for noncommutative
four tori which split as products of two tori. These four tori are not conformally flat and
the $a_4$ term gives a first hint of the Riemann curvature and the higher 
dimensional modular structure. 
\end{abstract}

\vskip 0.5cm
\noindent
{%\bf Mathematics Subject Classification (2010).} . 

%\vskip 0.5 cm
%\noindent
%{\bf Keywords.} 

\tableofcontents

\newpage

\section{Introduction}

In noncommutative geometry the paradigm of a geometric space is 
given in spectral terms, and the local geometric invariants such as the 
Riemannian curvature are extracted from the functionals defined by the 
coefficients of heat kernel expansion. One of the new features of the 
theory which is entirely due to non-commutativity is the modular theory 
which associates to a state a one parameter group of automorphisms of 
the ambient von-Neumann algebra measuring to which extent the state 
fails to be  a trace.  When the state is associated to the volume form of a 
metric a deep interplay arises between the local geometric invariants and 
the modular automorphism group.
More specifically, the local geometric 
invariants of the noncommutative two 
torus $\NT$ equipped with a curved metric are 
computed by calculating the coefficients that appear 
in the small-time heat kernel expansion of the Laplacian 
associated with the metric 
\cite{ConTreGB,ConMosModular, FatKhaSC2T, LesMosMorita}. 
The canonical flat metric 
of $\NT$ can be perturbed conformally by means of a Weyl  
factor $e^h \in \CNT$, where the dilaton $h$ is a smooth 
selfadjoint element \cite{ConTreGB}. The effect of this perturbation is that the 
canonical trace or the flat volume form $\vphi_0$ of the 
noncommutative torus is replaced by a state $\vphi$, 
and the trace of the heat kernel $\exp(-t\triangle_\vphi)$ of the 
Laplacian $\triangle_\vphi$ of the curved metric has 
an asymptotic expansion with complicated coefficients.

\smallskip

That is, there are unique elements $a_{2n} \in \CNT$ such that for any 
$a \in \CNT,$ as the time $t \to 0^+$, there is an asymptotic expansion 
of the form\footnote{By the construction,  since on a noncommutative torus of dimension $m$, we 
use the normalized trace as the analog of the integration on the $m$-dimensional torus 
$\mathbb{T}^m = (\R/ 2\pi \Z)^m$, we incorporate in our calculations the overall multiplicative factor 
$(2 \pi)^m$ for the geometric invariants.  
}
\begin{equation} \label{heatexp}
\Tr \left (a \exp(-t \triangle_\vphi) \right)  
\sim 
t^{-1} \left ( \vphi_0(a\, a_0) + \vphi_0(a \, a_2) \, t + \vphi_0(a\, a_4)\, t^2 + \cdots \right ).      
\end{equation}
In fact, each term $a_{2n}$ is a curvature related invariant  of the 
noncommutative torus $\NT$ equipped with the curved metric.

\smallskip 

 The exploration of the interplay between the local geometric invariants and 
the modular automorphism group has involved 
over the years alternating periods of hard calculations and conceptual 
understanding of their meaning \cite{ConTreGB, ConMosModular, FatKhaSC2T, LesDivided, LesMosMorita}. 
So far one has reached a good understanding of the terms $a_0$ and $a_2$ 
in the heat expansion \eqref{heatexp}. The term 
$a_0$ is related to the volume, whose connection with the analog of 
Weyl's law and the trace theorem of \cite{ConAction} is studied in 
\cite{FatKhaTraceThm}. The term $a_2$, which  is related to the analog 
of scalar curvature and the Gauss-Bonnet theorem for the noncommutative 
two torus \cite{ConTreGB, FatKhaGB}, is calculated and studied in 
\cite{ConMosModular, FatKhaSC2T,LesMosMorita}.  The present paper investigates for the first 
time the hard calculation of the term $a_4 \in \CNT$ appearing in \eqref{heatexp}. Due to the fact that 
the process of calculating this term involves exceedingly lengthy expressions and 
at times involves manipulations on a few hundred thousand terms,  
only the final outputs of the calculations are written in this paper.

\smallskip 

Our key result is that we could  
confirm the validity of the lengthy calculations by checking 
that the final expressions satisfy a family of functional relations, conceptually proved along the lines of   
\cite{ConMosModular}. We 
derive a system of partial differential equations from the functional 
relations by specializing them to certain hyperplanes and study symmetries 
of certain combinations of our expressions with respect to a natural action 
of cyclic groups of order 2, 3 and 4 on the differential system. 
We also pay special attention to the general structure of the several 
variable functions that appear in the expression of the term $a_4$, 
which is closely related to the fact that the main ingredients of 
such calculations can be derived by finite differences from the 
generating function of the Bernoulli numbers and its inverse.

\smallskip

The results of this paper are presented in two main different parts. 
The first  consists of the sections \ref{PrelSec} -- \ref{ConclusionsSec},  
where we mainly address the mathematically abstract work carried out 
for the calculation of the term $a_4 \in \CNT$ and exploring its properties. The second part, which is roughly the last 45 pages, is formed by the appendices, where 
we present final outputs of our calculations that have lengthy expressions, in their simplified form. 
The aspect of this part of the paper shows that a far more suitable mean for communicating these lengthy expressions would be to collect them 
in a mathematical program notebook, which we intend to do in the future.   This will, in particular, make it easy  to test  
further phenomena related to the analog of the Riemann curvature  in noncommutative geometry.

\smallskip

In Section \ref{PrelSec} we recall  necessary 
preliminaries about the canonical translation invariant conformal structure on the 
noncommutative two torus $\NT$, the Laplacian associated with a general 
metric in the conformal class, and the explicit formula for the term 
$a_2$ appearing in the expansion \eqref{heatexp}. 
In Section \ref{a_4Sec} we write the final expression for the 
term $a_4 \in \CNT$ in the expansion, which involves 
one, two, three and four variable functions of a modular automorphism. We 
then present one of our main results, namely a family of functional relations 
that the several variable functions satisfy. As we shall see, it is quite interesting that finite 
differences of the inverse of the generating function of the Bernoulli 
numbers play an important role in the functional relations. In Section \ref{DifferentialSystemSec} 
we derive a partial differential system from the functional relations by specializing them to certain 
hyperplanes, and study symmetries of our finite difference expressions 
with respect to a natural action of cyclic groups of order 2, 3 and 4 on the system
as well as a natural flow associated to the differential system.  
In Section \ref{FuncRelationsProofSec}, we explain the method of proving 
the functional relations, and prove a series of lemmas that will be used  in Section 
\ref{GradientCalculationSec} for calculating in terms of finite differences the gradient of 
the map that sends the dilaton $h = h^* \in \CNT$ to the trace of the term $a_4$. 
Section \ref{FuncRelationsforksSec} presents functional relations among functions 
that appear in the differential system. Some of the relations of this type, because of 
their lengthiness, are written in Appendix \ref{lengthykfnrelationsappsec}.

\smallskip

In Section \ref{Calculatea_4Sec} we explain the details of our calculation of the 
term $a_4$, which is based on using the pseudodifferential calculus developed in  
\cite{ConC*algDiffGeo}, a rearrangement lemma, and 
a lemma proved in Section \ref{GradientCalculationSec} that relates the derivates up to order four of the conformal factor $e^h \in \CNT$ 
and those of the dilaton $h$. In particular we show that all functions 
of several variables appearing in these lemmas can be constructed by finite differences 
from the generating function of the Bernoulli numbers and its inverse. 
By employing these tools and performing heavy calculations, we 
find the expression for the term $a_4$ and explicit formulas for the functions 
of one to four variables that appear in the final formula. 
The one and two variable functions are presented in 
Section \ref{ExplicitFormulasSec},  while, because of having algebraically 
lengthy expressions, the three and four variable functions are presented in 
 Appendix \ref{explicit3and4fnsappsec}. We confirm the accuracy of our heavy calculations by 
explaining that the final functions check out the functional relations presented in 
Section \ref{a_4Sec} and in Appendix \ref{lengthyfnrelationsappsec}. 
In Section \ref{StructureSec}, we explain why each local function 
appearing in the expression of the term $a_4$ is a rational function 
in variables $s_i$ and $e^{s_i/2}$, whose denominator has a nice 
product formula that vanishes on certain hyperplanes. Moreover, we show 
that the coefficients of the numerator of each function satisfy a family of linear 
equations.

\smallskip

In Section \ref{NC4toriSec} we explain how the explicit formulas for the term $a_4$ for $\NT$ provide a first glimpse of the Riemann curvature beyond the conformally flat case. Indeed it yields the term $a_4$ for the $4$-dimensional noncommutative tori obtained as products of two noncommutative two-tori and such spaces are generally not conformally flat. Moreover they possess a natural two-dimensional modular structure, given by an action of $\R^2$ by automorphisms, obtained from the modular structure of the factors, while the measure theory given by the  determinant part of the metric only involves the restriction of this action of $\R^2$ to the diagonal. This gives a strong motivation to develop conceptually the more general notion of twisting suggested in particular in \cite{ConTransverse} and which plays a fundamental role in the work of H. Moscovici and the first author \cite{CMos} on the transverse geometry of foliations and the reduction by duality to the almost isometric case. 

\smallskip 

Our main results and conclusions are 
summarized in Section \ref{ConclusionsSec}.

\smallskip

\section{Preliminaries}
\label{PrelSec}

We consider the noncommutative two torus $\NT$, whose algebra $C(\NT)$ is the 
universal $C^*$-algebra generated by two unitary elements $U$ and 
$V$ that satisfy the following commutation relation for a fixed irrational 
real number $\theta$: 
\[
VU = e^{2 \pi i \theta} \,U  V. 
\]
We have 
a $C^*$-dynamical system by considering the following action $\alpha$ of 
the ordinary two torus $\mathbb{T}^2=(\mathbb{R}/2 \pi \mathbb{Z})^2$ 
on the algebra $C(\NT)$ of the  noncommutative torus. For any $(s_1, s_2)\in \mathbb{T}^2$, one can define
\begin{equation} \label{NC2TDynSys}
\alpha_{(s_1, s_2)}(U^m V^n) 
= 
e^{i (ms_1 +ns_2)} U^m V^n, \qquad m, n \in \mathbb{Z}.  
\end{equation}
This definition extends to an automorphism of the $C^*$-algebra $C(\NT)$, 
which is the noncommutative analog of translating a continuous 
function defined on the torus $\mathbb{T}^2$  by $(s_1, s_2)$.

\smallskip

Associated with the above action $\alpha$, there are two infinitesimal 
generators $\delta_1$ and $\delta_2$, which are derivations 
on the space of smooth elements $\CNT$ in $C(\NT)$. 
More precisely, $\CNT$ consists of all elements $x$ in the noncommutative 
torus such that the mapping  
$(s_1, s_2 ) \mapsto \alpha_{(s_1, s_2)}(x)$ 
from $\mathbb{T}^2$ to $C(\NT)$ is a smooth map. Indeed, 
$\CNT$ is a dense subalgebra of $C(\NT)$, and it can alternatively be 
described as the space of all 
\[
x =\sum_{m, n \in \mathbb{Z}} a_{m, n}\, U^m V^n, 
\]
such that the sequence of 
complex coefficients $(a_{m, n})$ is rapidly decaying in the sense 
that 
\[
\sup_{m, n \in Z} |a_{m, n}|(1 + |n| + |m|)^k < \infty, 
\]
for any non-negative integer $k$. Therefore, each of the derivations 
$\done$ and  
$\dtwo$ on $\CNT$ is characterized by its action on the 
generators $U$ and $V$, which are in fact given by 
\[
\done(U)=U, \qquad \done(V) = 0, \qquad \dtwo(U) = 0, \qquad \dtwo(V)=V. 
\]

\smallskip

While $\done$ and $\dtwo$ are respectively the analogs of the differential 
operators $-i (\partial/ \partial s_1)$ and $-i (\partial/ \partial s_2)$ on 
the ordinary two torus, we also have an analog of the integration or volume 
form of the flat metric. The latter is provided by the unique normalized 
tracial state $\vphi_0 : C(\NT) \to \C$, which is defined on the smooth dense 
subalgebra $\CNT$ by  
\begin{equation} \label{NCTtrace}
\vphi_0 \left (  \sum_{m, n \in \mathbb{Z}} a_{m, n} U^m V^n \right ) = a_{0,0}. 
\end{equation}
In fact, the uniqueness of this trace is due to the irrationality of $\theta$. An important 
property of the trace $\vphi_0$ is its invariance under the action $\alpha$, 
which yields 
\[
\vphi_0 \circ \delta_j = 0, \qquad j =1, 2, 
\]
hence the analog of  integration by parts: 
\begin{equation} \label{IntbyParts}
\vphi_0 \left (x_1 \delta_j(x_2) \right ) 
=
- \vphi_0 \left ( \delta_j(x_1) \, x_2 \right ), 
\qquad x_1, x_2 \in \CNT. 
\end{equation}

\smallskip

Following \cite{ConTreGB}, we consider a complex structure on the 
noncommutative two torus by setting the analog of the Dolbeault 
operators to be 
\[
\partial 
= 
\done - i \dtwo: \CNT \subset \mathcal{H}_0 \to \mathcal{H}^{(1, 0)}, 
\qquad
\bar \partial 
= 
\done +i \dtwo: \CNT \subset \mathcal{H}_0 \to \mathcal{H}^{(0, 1)},    
\]
where the Hilbert spaces 
$ \mathcal{H}_0, \mathcal{H}^{(1, 0)}, \mathcal{H}^{(0, 1)}$ 
are defined as follows. The Hilbert space $\mathcal{H}_0$ is the completion of 
$C(\NT)$ with respect to the inner product
\[
\langle x, y \rangle = \vphi_0 (y^* x), \qquad x, y \in C(\NT). 
\]
In order to define the Hilbert space $\mathcal{H}^{(1, 0)}$, which is the 
analog of the space of $(1, 0)$-forms, we need to consider the bimodule over 
$\CNT$ of finite sums of the form $ \sum a_i \partial(b_i), a_i, b_i \in \CNT$, and 
complete it with an inner product that comes from a positive Hochschild 
cocycle. The Hilbert space $\mathcal{H}^{(0, 1)}$ is obtained similarly by 
a completion of the space of finite sums of the form 
$ \sum a_i \bar \partial(b_i), a_i, b_i \in \CNT$.  
The positive Hochschild cocycle, which encodes the conformal structure of the 
metric \cite{ConNCGBook} on $\NT$, is defined by 
\[
\psi(x, y, z) 
= 
- \vphi_0 \left (x \, \partial(y)\, \bar \partial(z) \right), \qquad x, y, z \in \CNT,  
\]
and the inner product for $\mathcal{H}^{(1, 0)}$ is given by 
\[
\langle x_1 \partial(y_1),  x_2 \partial(y_2)  \rangle 
= 
\psi \left( x_2^* x_1, y_1, y_2^* \right ), 
\qquad x_1, y_1, x_2, y_2 \in \CNT. 
\]

\smallskip

With this information, we can now calculate the Laplacian associated with the flat metric  
\[
\triangle := \partial^* \partial = \del_1^2 + \del_2^2, 
\]
which is an unbounded selfadjoint operator acting in the Hilbert space $\mathcal{H}_0$. 
By using a positive invertible element $e^h \in \CNT$, where $h$ is a smooth selfadjoint 
element, one can vary the metric inside the conformal structure. That is, one can 
consider the state $\vphi : C(\NT) \to \C$ defined by 
\begin{equation} \label{KMSstate}
\vphi(x) = \vphi_0(x e^{-h}), \qquad x \in C(\NT), 
\end{equation}
which is the analog of the volume form of the conformal perturbation of the 
flat metric. We consider the Hilbert space $\mathcal{H}_\vphi$ obtained from 
completing $C(\NT)$ with respect to the inner product 
\[
\langle x, y \rangle_\vphi = \vphi(y^* x), \qquad x, y \in C(\NT). 
\]
Clearly, the adjoint of the operator $\partial_\vphi= \partial : \mathcal{H}_\vphi \to \mathcal{H}^{(1, 0)}$
depends on the conformal factor $e^h$. It is shown in \cite{ConTreGB} that there is an anti-unitary equivalence 
between the Hilbert spaces $\mathcal{H}_0$ and $\mathcal{H}_\vphi$ that identifies 
the Laplacian 
\[
\triangle_\vphi := \partial_\vphi^* \partial_\vphi : \mathcal{H}_\vphi \to \mathcal{H}_\vphi, 
\]
with 
the operator $e^{h/2} \triangle e^{h/2}$ acting in $\mathcal{H}_0$. Therefore, we make the 
identification 
\begin{equation} \label{conformalLaplacian}
\triangle_\vphi = e^{h/2} \triangle e^{h/2} : \mathcal{H}_0 \to \mathcal{H}_0.
\end{equation}

\smallskip

A purely noncommutative feature in the calculation of the terms $a_{2n} \in \CNT$ 
in the small-time asymptotic expansion \eqref{heatexp} of $\Tr(a \exp(-t \triangle_\vphi))$ 
is the appearance 
of the modular automorphism of the state $\vphi$ in the final formulas. 
That is, the linear functional $\vphi$ 
given by \eqref{KMSstate} is a KMS state that satisfies the condition 
\[
\vphi(x y) = \vphi(y \, \sigma_i(x)), \qquad x, y \in C(\NT), 
\]
for the 1-parameter group of automorphisms $\{ \sigma_t \}_{t \in \R}$ defined by 
\[
\sigma_t(x) = e^{ith} x e^{-ith}, \qquad x \in C(\NT). 
\]
Clearly, the modular automorphism $\sigma_i$ acts by conjugation with $e^{-h}$,  
and its logarithm is thereby given by 
\[
\nabla(x) := \log \sigma_i (x) = - \ad_{h}(x) = -hx + xh, \qquad x \in C(\NT). 
\]
The final formula for the second term in the expansion \eqref{heatexp}, which 
is calculated in \cite{ConMosModular, FatKhaSC2T}, is given by  
\begin{equation} \label{SCformula}
a_2 = R_1(\nab) \left (\done^2(\ell) + \dtwo^2(\ell) \right) 
+ R_2(\nab, \nab) \left ( \done(\ell) \cdot \done(\ell) + \dtwo(\ell) \cdot \dtwo(\ell) \right ),  
\end{equation}
where
\[
\ell = \frac{h}{2} \in \CNT, 
\] 
\begin{eqnarray*}
R_1(s_1)&=&\frac{4 e^{\frac{s_1}{2}} \pi  \left(2+e^{s_1} \left(-2+s_1\right)+s_1\right)}{\left(-1+e^{s_1}\right){}^2 s_1}, \\
R_2(s_1, s_2)&=&
\end{eqnarray*}
{\tiny
\[
-4\frac{\pi  \left(\cosh\left[s_2\right] s_1 \left(s_1+s_2\right)-\cosh\left[s_1\right] s_2 \left(s_1+s_2\right)-\left(s_1-s_2\right) \left(\sinh\left[s_1\right]+\sinh\left[s_2\right]-\sinh\left[s_1+s_2\right]+s_1+s_2\right)\right)}{\sinh\left[\frac{s_1}{2}\right] \sinh\left[\frac{s_2}{2}\right] \sinh^2\left[\frac{1}{2} \left(s_1+s_2\right)\right] s_1 s_2 \left(s_1+s_2\right)}.
\]
}

\smallskip

Let us explain that, in general, given a rapidly decaying smooth function 
$L$ defined on the Euclidean space $\R^n$ and elements $x_1, \dots, x_n$  in 
the algebra  $C(\NT)$ of the  noncommutative torus, we consider the calculus defined by 
\[
L(\nab, \dots, \nab)(x_1 \cdots x_n) 
= 
\int_{\R^n} \sigma_{t_1}(x_1) \cdots \sigma_{t_n}(x_n) \, g(t_1, \dots, t_n) \,dt_1 \cdots \,dt_n, 
\] 
where the function $g$ is obtained by writing the function $L$ as a Fourier transform: 
\[
L(s_1, \dots, s_n) = \int_{\R^n} e^{-i( t_1s_1+\cdots+ t_n s_n)} g(t_1, \dots, t_n) \, dt_1 \cdots \, dt_n. 
\]

\smallskip

\section{The term $a_4$ and its functional relations} 
\label{a_4Sec}

%\[
%\ell = \frac{h}{2}, \qquad \nab = [-h, \cdot ], \qquad \sigma_t = e^{-it \nab}. 
%\]

We start this section by writing an expression for the term 
$a_4 \in \CNT$ appearing in the small-time heat kernel expansion 
\eqref{heatexp}, which involves 20 functions of one to four variables 
that are denoted by $K_1, \dots, K_{20}$.  We will then proceed to 
present one of our main results, namely a family of functional relations 
that are satisfied by the functions $K_j,$ $j=1, \dots, 20$. Using the notations 
provided in Section \ref{PrelSec}, we have: 
\begin{equation}  \label{a_4expression}
a_4=
\end{equation}
\begin{eqnarray}
&&- e^{2 \ell} \Big ( K_1(\nab) \left ( \delta _1^2 \delta _2^2 ( \ell ) \right ) 
+ 
K_2 (\nab) \left (   \delta _1^4( \ell )+\delta _2^4 (\ell ) \right )  
+
K_3 (\nab, \nab) \left (
 \left(\delta _1 \delta _2(\ell
   )\right) \cdot \left(\delta _1 \delta _2(\ell
   )\right) 
   \right )
\nonumber \\ && 
+ 
K_4 (\nab, \nab) \left (  \delta _1^2(\ell ) \cdot \delta _2^2(\ell )+\delta
   _2^2(\ell ) \cdot \delta _1^2(\ell )\right )  
+
K_5 (\nab, \nab) \left ( \delta _1^2( \ell )\cdot \delta _1^2(\ell )+\delta
   _2^2(\ell ) \cdot \delta _2^2(\ell ) \right ) 
\nonumber \\ && 
+ 
K_6 (\nab, \nab) \left ( 
\delta _1(\ell )\cdot \delta _1^3(\ell )+\delta_1(\ell )\cdot \left(\delta _1 \delta _2^2 (\ell
   )\right)+\delta _2(\ell )\cdot \delta _2^3(\ell
   )+\delta _2(\ell )\cdot \left(\delta _1^2
   \delta _2(\ell )\right)
 \right ) 
\nonumber \\ &&
+ 
K_7 (\nab, \nab) \left ( 
\delta _1^3(\ell )\cdot \delta _1(\ell
   )+\left(\delta _1 \delta _2^2(\ell
   )\right)\cdot \delta _1(\ell )+\delta _2^3(\ell
   )\cdot \delta _2(\ell )+\left(\delta _1^2
   \delta _2(\ell )\right)\cdot \delta _2(\ell )
 \right ) 
\nonumber 
\end{eqnarray}
\begin{eqnarray}
 &&
+ 
K_8 (\nab, \nab, \nab) \left ( 
\delta _1(\ell ) \cdot \delta _1(\ell )\cdot \delta
   _2^2(\ell )+\delta _2(\ell )\cdot \delta
   _2(\ell )\cdot \delta _1^2(\ell )
\right ) 
\nonumber \\ && 
+ 
K_9 (\nab, \nab, \nab) \left ( 
\delta _1(\ell )\cdot \delta _2(\ell
   )\cdot \left(\delta _1 \delta _2(\ell
   )\right)+\delta _2(\ell )\cdot \delta _1(\ell
   )\cdot \left(\delta _1 \delta _2(\ell )\right)
\right ) 
\nonumber \\ && 
+ 
K_{10} (\nab, \nab, \nab) \left ( 
\delta _1(\ell )\cdot \left(\delta _1 \delta
   _2(\ell )\right)\cdot \delta _2(\ell )+\delta
   _2(\ell )\cdot  \left(\delta _1 \delta _2(\ell
   )\right) \cdot  \delta _1(\ell )
\right )
\nonumber \\ && 
+ 
K_{11} (\nab, \nab, \nab) \left ( 
\delta _1(\ell )\cdot \delta _2^2(\ell ) \cdot \delta
   _1(\ell )+\delta _2(\ell )\cdot \delta
   _1^2(\ell ) \cdot \delta _2(\ell )
\right ) 
\nonumber \\ && 
+ 
K_{12} (\nab, \nab, \nab) \left ( 
\delta _1^2(\ell ) \cdot \delta _2(\ell )\cdot \delta
   _2(\ell )+\delta _2^2( \ell ) \cdot \delta
   _1(\ell )\cdot \delta _1(\ell )
\right ) 
\nonumber 
\end{eqnarray}
\begin{eqnarray}
&&+ 
K_{13} (\nab, \nab, \nab) \left ( 
\left(\delta _1 \delta _2(\ell
   )\right)\cdot \delta _1(\ell )\cdot \delta _2(\ell
   )+\left(\delta _1 \delta _2(\ell
   )\right) \cdot \delta _2(\ell )\cdot \delta _1(\ell )
\right ) 
\nonumber \\ && 
+ 
K_{14} (\nab, \nab, \nab) \left ( 
\delta _1^2(\ell ) \cdot \delta _1(\ell )\cdot 
\delta_1(\ell )+\delta _2^2(\ell )\cdot \delta_2(\ell )\cdot \delta _2(\ell )
\right )
\nonumber \\ && 
+ 
K_{15} (\nab, \nab, \nab) \left ( 
\delta _1(\ell ) \cdot \delta _1(\ell )\cdot \delta
   _1^2(\ell )+\delta _2(\ell )\cdot \delta
   _2(\ell )\cdot \delta _2^2(\ell )
\right )
\nonumber \\ && 
+ 
K_{16} (\nab, \nab, \nab) \left ( 
\delta _1(\ell )\cdot \delta _1^2(\ell ) \cdot \delta
   _1(\ell )+\delta _2(\ell )\cdot \delta
   _2^2(\ell )\cdot \delta _2(\ell )
\right ) 
\nonumber 
\end{eqnarray} 
\begin{eqnarray}
&& 
+ 
K_{17} (\nab, \nab, \nab, \nab ) \left ( 
\delta _1(\ell )\cdot \delta _1(\ell )\cdot \delta
   _2(\ell )\cdot \delta _2(\ell )+\delta _2(\ell
   )\cdot \delta _2(\ell )\cdot \delta _1(\ell )\cdot \delta
   _1(\ell )
\right ) 
\nonumber \\ && 
+ 
K_{18} (\nab, \nab, \nab, \nab ) \left ( 
\delta _1(\ell )\cdot \delta _2(\ell )\cdot \delta
   _1(\ell )\cdot \delta _2(\ell )+\delta _2(\ell
   )\cdot \delta _1(\ell )\cdot \delta _2(\ell )\cdot \delta
   _1(\ell )
\right ) 
\nonumber \\ && 
+ 
K_{19} (\nab, \nab, \nab, \nab ) \left ( 
\delta _1(\ell )\cdot \delta _2(\ell )\cdot \delta
   _2(\ell )\cdot \delta _1(\ell )+\delta _2(\ell
   )\cdot \delta _1(\ell )\cdot \delta _1(\ell )\cdot \delta
   _2(\ell )
\right ) 
\nonumber \\ && 
+ 
K_{20} (\nab, \nab, \nab, \nab ) \left ( 
\delta _1(\ell )\cdot \delta _1(\ell )\cdot \delta
   _1(\ell )\cdot \delta _1(\ell )+\delta _2(\ell
   )\cdot \delta _2(\ell )\cdot \delta _2(\ell )\cdot \delta
   _2(\ell )
\right ) \Big ). \nonumber
\end{eqnarray}

\smallskip

Explicit formulas for the one and two variable functions $K_1, \dots, K_7$ 
are given in Section \ref{ExplicitFormulasSec}, and lengthy formulas 
for the three and four variable functions $K_8, \dots, K_{20}$ are 
provided in Appendix \ref{explicit3and4fnsappsec}.

\smallskip

In order to present the functional relations, it is convenient to 
define mild variants $\widetilde K_j$ of the functions $K_j,$ 
$j=1, 2, \dots, 20,$ which are given by  
\begin{eqnarray} \label{Ktildesdef}
\widetilde K_j(s_1) &=& \frac{1}{2}\frac{ \sinh \left(\frac{s_1}{2}\right)}{\frac{s_1}{2}} K_j(s_1), \qquad j=1, 2, \\ 
\widetilde K_j (s_1, s_2) &=& \frac{1}{2^2}\frac{ \sinh\left(\frac{s_1+s_2}{2}\right)}{\frac{s_1+s_2}{2}} K_j(s_1, s_2), \qquad j=3, 4, \dots, 7, \nonumber 
\end{eqnarray}
\begin{eqnarray*}
\widetilde K_j (s_1, s_2, s_3) &=& \frac{1}{2^3}\frac{ \sinh \left(\frac{s_1+s_2+s_3}{2}\right)}{\frac{s_1+s_2+s_3}{2}} K_j(s_1, s_2, s_3), \qquad j=8, 9, \dots, 16,  \\
\widetilde K_j (s_1, s_2, s_3, s_4) &=& \frac{1}{2^4}\frac{ \sinh\left(\frac{s_1+s_2+s_3+s_4}{2}\right)}{\frac{s_1+s_2+s_3+s_4}{2}} K_j(s_1, s_2, s_3, s_4), \qquad j=17, 18, 19, 20.  
\end{eqnarray*}
We also need to introduce the following functions $k_j$, $j=3,4, \dots, 20,$ 
which are derived from the main functions by setting
\begin{eqnarray} \label{littlekfunctions}
k_j (s_1) &=& K_j(s_1, -s_1), \qquad j=3, \dots, 7, \\
k_j (s_1, s_2) &=& K_j(s_1, s_2, -s_1-s_2), \qquad j=8, 9, \dots, 16, \nonumber \\
k_j (s_1, s_2, s_3) &=& K_j(s_1, s_2, s_3, -s_1-s_2-s_3), \qquad j=17, 18, 19, 20. \nonumber  
\end{eqnarray}

\smallskip

The following functions $G_1, G_2, G_3, G_4$, which are constructed in 
lemmas \ref{ToLogBaby}, \ref{BabyFunctions1} and \ref{BabyFunctions2}, 
play an important role as well 
in the functional relations. They are given explicitly by  
\begin{eqnarray}  \label{ExplicitBabyFunctions}
G_1(s_1)&=&\frac{e^{s_1}-1}{s_1}, \\ 
G_2(s_1, s_2) &=& \frac{e^{s_1} \left(\left(e^{s_2}-1\right) s_1-s_2\right)+s_2}{s_1 s_2 \left(s_1+s_2\right)}, \nonumber \\
G_3(s_1, s_2, s_3) &=& \nonumber
\end{eqnarray}
\[
\frac{e^{s_1} \left(e^{s_2+s_3} s_1 s_2 \left(s_1+s_2\right)+\left(s_1+s_2+s_3\right) \left(\left(s_1+s_2\right) s_3-e^{s_2} s_1 \left(s_2+s_3\right)\right)\right)-s_2 s_3 \left(s_2+s_3\right)}{s_1 s_2 \left(s_1+s_2\right) s_3 \left(s_2+s_3\right) \left(s_1+s_2+s_3\right)}, 
\]
\begin{eqnarray*}
G_4(s_1, s_2, s_3, s_4)&=&\frac{e^{s_1+s_2} \left(\frac{1}{s_2 \left(s_1+s_2\right) \left(s_3+s_4\right)}-\frac{e^{s_3}}{\left(s_2+s_3\right) \left(s_1+s_2+s_3\right) s_4}\right)}{s_3} \\
&&+\frac{\frac{1}{\left(s_1+s_2\right) \left(s_1+s_2+s_3\right) \left(s_1+s_2+s_3+s_4\right)}-\frac{e^{s_1}}{s_2 \left(s_2+s_3\right) \left(s_2+s_3+s_4\right)}}{s_1}\\
&&+\frac{e^{s_1+s_2+s_3+s_4}}{s_4 \left(s_3+s_4\right) \left(s_2+s_3+s_4\right) \left(s_1+s_2+s_3+s_4\right)}. 
\end{eqnarray*}

\smallskip

At this stage, we are ready to present the functional relations explicitly. 
It is worth emphasizing that the proof of 
these relations requires a significant amount of work, which is 
carried out in Section \ref{FuncRelationsProofSec} and Section 
\ref{GradientCalculationSec}. That is, the proof is based on calculating 
the gradient of the map that sends a general dilaton $h = h^* \in \CNT$ 
to $\vphi_0(a_4),$ in two different ways: first, by using a fundamental 
identity proved in \cite{ConMosModular}, and second, by using the Duhamel 
formula to find a formula in terms of finite differences.  

\begin{theorem} \label{FuncRelationsThm}
Using the above notations, the functions of one to four variables 
$K_1, \dots , K_{20}$ appearing in the expression \eqref{a_4expression} for the term 
$a_4 \in \CNT$ satisfy the functional  relations presented in 
the following subsections and in Appendix \ref{lengthyfnrelationsappsec}, in which each 
variant $\widetilde K_j$ of the function $K_j$ is expressed as finite differences of the 
involved functions. 
\begin{proof}
The functional relations are derived by comparing the corresponding terms 
in the final formulas for the gradient
\[
\dep \vphi_0 \left ( a_4(h + \vep a) \right ),  
\]
where $h, a \in \CNT$ are selfadjoint elements, calculated in two different ways. 
The first method is explained in the beginning of Section \ref{FuncRelationsProofSec}, 
which gives rise to the formula \eqref{Gradientofa4} where the variants $\widetilde K_j$ 
of the functions $K_j$ appear. The second method is based on employing the 
lemmas proved in Section \ref{FuncRelationsProofSec} and performing the 
calculations explained in Section \ref{GradientCalculationSec} to calculate the 
above gradient in terms of finite differences. 
\end{proof}
\end{theorem}

\smallskip

In  each of the 
following subsections, the functional relations associated with the 
functions that depend on the same number of variables are 
given.

\subsection{The functions $\widetilde K_1, \widetilde K_2$}

In this subsection we present the functional relations associated with 
the one variable functions $\widetilde K_1,$ $\widetilde K_2$ defined 
by \eqref{Ktildesdef}.

\subsubsection{The function $\widetilde K_1$} By using the 
identity \eqref{Gradientofa4} and considering the 
specific function of $\nabla$ that acts on $\delta _1^2 \delta _2^2 ( h )$ 
in the final formula for the gradient $\dep \vphi_0(a_4(h+\vep a))$ 
as calculated in Section \ref{GradientCalculationSec}, we have:

\begin{equation} \label{basicK1eqn}
 \widetilde K_1(s_1) =
\end{equation}

\begin{center} 
\begin{math}
-\frac{1}{15} \pi  G_1\left(s_1\right)+\frac{1}{4} e^{s_1} k_3\left(-s_1\right)+\frac{1}{4} k_3\left(s_1\right)+\frac{1}{2} e^{s_1} k_4\left(-s_1\right)+\frac{1}{2} k_4\left(s_1\right)-\frac{1}{2} e^{s_1} k_6\left(-s_1\right)-\frac{1}{2} k_6\left(s_1\right)-\frac{1}{2} e^{s_1} k_7\left(-s_1\right)-\frac{1}{2} k_7\left(s_1\right)-\frac{\pi  \left(e^{s_1}-1\right)}{15 s_1}. 
\end{math}
\end{center}

\subsubsection{The function $\widetilde K_2$} It is clear from 
the explicit formulas presented 
in Section \ref{ExplicitFormulasSec} that the function $\widetilde K_2$ is a scalar multiple of 
the function $\widetilde K_1$. Therefore, it is interesting to see that the following 
functional relation, which has different ingredients compared to those for 
$\widetilde K_1$, gives the same function up to a scalar multiplication. For the second 
function we have: 

\begin{equation} \label{basicK2eqn}
\widetilde K_2(s_1) =
\end{equation}

\begin{center} 
\begin{math}
 -\frac{1}{30} \pi  G_1\left(s_1\right)+\frac{1}{4} e^{s_1} k_5\left(-s_1\right)+\frac{1}{4} k_5\left(s_1\right)-\frac{1}{4} e^{s_1} k_6\left(-s_1\right)-\frac{1}{4} k_6\left(s_1\right)-\frac{1}{4} e^{s_1} k_7\left(-s_1\right)-\frac{1}{4} k_7\left(s_1\right)-\frac{\pi  \left(e^{s_1}-1\right)}{30 s_1}. 
\end{math}
\end{center}

\subsection{The functions $\widetilde K_3, \dots, \widetilde K_7$}
In this subsection we present the functional relations associated with 
the two variable functions $\widetilde K_3, \dots, \widetilde K_7$ defined 
in \eqref{Ktildesdef}. 

\subsubsection{The function $\widetilde K_3$}  
\label{basicK3} 
Similarly to the case 
of the one variable functions, by using the 
identity \eqref{Gradientofa4} and considering the 
specific function of $\nabla$ that acts on $ \delta _1 \delta _2(h) \cdot \delta _1 \delta _2(h) $ 
in the final formula for the gradient $\dep \vphi_0(a_4(h+\vep a))$ 
as calculated in Section \ref{GradientCalculationSec},
we obtain the following functional relation for the function 
$\widetilde K_3:$

\begin{equation} \label{basicK3eqn}
\widetilde K_3(s_1, s_2) =
\end{equation}

\begin{center}
\begin{math} 
 \frac{1}{15} (-4) \pi  G_2\left(s_1,s_2\right)+\frac{1}{2} k_8\left(s_1,s_2\right)+\frac{1}{4} k_9\left(s_1,s_2\right)-\frac{1}{4} e^{s_1+s_2} k_9\left(-s_1-s_2,s_1\right)-\frac{1}{4} e^{s_1} k_9\left(s_2,-s_1-s_2\right)-\frac{1}{4} k_{10}\left(s_1,s_2\right)-\frac{1}{4} e^{s_1+s_2} k_{10}\left(-s_1-s_2,s_1\right)+\frac{1}{4} e^{s_1} k_{10}\left(s_2,-s_1-s_2\right)+\frac{1}{2} e^{s_1} k_{11}\left(s_2,-s_1-s_2\right)+\frac{1}{2} e^{s_1+s_2} k_{12}\left(-s_1-s_2,s_1\right)-\frac{1}{4} k_{13}\left(s_1,s_2\right)+\frac{1}{4} e^{s_1+s_2} k_{13}\left(-s_1-s_2,s_1\right)-\frac{1}{4} e^{s_1} k_{13}\left(s_2,-s_1-s_2\right)+\frac{1}{4} e^{s_2} G_1\left(s_1\right) k_3\left(-s_2\right)+\frac{1}{4} G_1\left(s_1\right) k_3\left(s_2\right)-G_1\left(s_1\right) k_6\left(s_2\right)-e^{s_2} G_1\left(s_1\right) k_7\left(-s_2\right)+\frac{\left(e^{s_1+s_2}-1\right) k_3\left(s_1\right)}{4 \left(s_1+s_2\right)}+\frac{k_3\left(s_2\right)-k_3\left(s_1+s_2\right)}{4 s_1}+\frac{k_3\left(s_1+s_2\right)-k_3\left(s_1\right)}{4 s_2}+\frac{k_6\left(s_1\right)-k_6\left(s_1+s_2\right)}{s_2}+\frac{k_6\left(s_1+s_2\right)-k_6\left(s_2\right)}{s_1}+\frac{e^{s_1} \left(k_7\left(-s_1\right)-e^{s_2} k_7\left(-s_1-s_2\right)\right)}{s_2}+\frac{e^{s_2} \left(e^{s_1} k_7\left(-s_1-s_2\right)-k_7\left(-s_2\right)\right)}{s_1}-\frac{e^{s_2} \left(e^{s_1} k_3\left(-s_1-s_2\right)-k_3\left(-s_2\right)\right)}{4 s_1}-\frac{e^{s_1} \left(k_3\left(-s_1\right)-e^{s_2} k_3\left(-s_1-s_2\right)\right)}{4 s_2}-\frac{e^{s_1} \left(k_3\left(-s_1\right)+e^{s_2} k_3\left(s_1\right)-e^{s_2} k_3\left(-s_2\right)-k_3\left(s_2\right)\right)}{4 \left(s_1+s_2\right)}.  
\end{math}
\end{center}

\subsubsection{The function $\widetilde K_4$} 
\label{basicK4}
By making another comparison 
between the term containing  the function $\widetilde K_4$ in \eqref{Gradientofa4}, 
and the corresponding function coming out of the calculations in Section 
\ref{GradientCalculationSec}, 
we have: 

\begin{equation} \label{basicK4eqn}
\widetilde K_4(s_1, s_2) =
\end{equation}

\begin{center}
\begin{math}
 -\frac{1}{15} \pi  G_2\left(s_1,s_2\right)-\frac{1}{8} e^{s_1+s_2} k_8\left(-s_1-s_2,s_1\right)-\frac{1}{8} e^{s_1} k_8\left(s_2,-s_1-s_2\right)+\frac{1}{8} k_9\left(s_1,s_2\right)+\frac{1}{8} e^{s_1} k_{10}\left(s_2,-s_1-s_2\right)-\frac{1}{8} k_{11}\left(s_1,s_2\right)-\frac{1}{8} e^{s_1+s_2} k_{11}\left(-s_1-s_2,s_1\right)-\frac{1}{8} k_{12}\left(s_1,s_2\right)-\frac{1}{8} e^{s_1} k_{12}\left(s_2,-s_1-s_2\right)+\frac{1}{8} e^{s_1+s_2} k_{13}\left(-s_1-s_2,s_1\right)+\frac{1}{4} e^{s_2} G_1\left(s_1\right) k_4\left(-s_2\right)+\frac{1}{4} G_1\left(s_1\right) k_4\left(s_2\right)-\frac{1}{4} G_1\left(s_1\right) k_6\left(s_2\right)-\frac{1}{4} e^{s_2} G_1\left(s_1\right) k_7\left(-s_2\right)+\frac{\left(e^{s_1+s_2}-1\right) k_4\left(s_1\right)}{4 \left(s_1+s_2\right)}+\frac{k_4\left(s_2\right)-k_4\left(s_1+s_2\right)}{4 s_1}+\frac{k_4\left(s_1+s_2\right)-k_4\left(s_1\right)}{4 s_2}+\frac{k_6\left(s_1\right)-k_6\left(s_1+s_2\right)}{4 s_2}+\frac{k_6\left(s_1+s_2\right)-k_6\left(s_2\right)}{4 s_1}+\frac{e^{s_1} \left(k_7\left(-s_1\right)-e^{s_2} k_7\left(-s_1-s_2\right)\right)}{4 s_2}+\frac{e^{s_2} \left(e^{s_1} k_7\left(-s_1-s_2\right)-k_7\left(-s_2\right)\right)}{4 s_1}-\frac{e^{s_2} \left(e^{s_1} k_4\left(-s_1-s_2\right)-k_4\left(-s_2\right)\right)}{4 s_1}-\frac{e^{s_1} \left(k_4\left(-s_1\right)-e^{s_2} k_4\left(-s_1-s_2\right)\right)}{4 s_2}-\frac{e^{s_1+s_2} \left(k_4\left(s_1\right)-k_4\left(-s_2\right)\right)}{4 \left(s_1+s_2\right)}-\frac{e^{s_1} \left(k_4\left(-s_1\right)-k_4\left(s_2\right)\right)}{4 \left(s_1+s_2\right)}. 
\end{math}
\end{center}

\subsubsection{The function $\widetilde K_5$}
\label{basicK5}
By a similar comparison, we have: 

\begin{equation} \label{basicK5eqn}
\widetilde K_5(s_1, s_2) =
\end{equation}

\begin{center}
\begin{math}
 -\frac{1}{5} \pi  G_2\left(s_1,s_2\right)-\frac{1}{8} k_{14}\left(s_1,s_2\right)+\frac{1}{4} e^{s_1+s_2} k_{14}\left(-s_1-s_2,s_1\right)-\frac{1}{8} e^{s_1} k_{14}\left(s_2,-s_1-s_2\right)+\frac{1}{4} k_{15}\left(s_1,s_2\right)-\frac{1}{8} e^{s_1+s_2} k_{15}\left(-s_1-s_2,s_1\right)-\frac{1}{8} e^{s_1} k_{15}\left(s_2,-s_1-s_2\right)-\frac{1}{8} k_{16}\left(s_1,s_2\right)-\frac{1}{8} e^{s_1+s_2} k_{16}\left(-s_1-s_2,s_1\right)+\frac{1}{4} e^{s_1} k_{16}\left(s_2,-s_1-s_2\right)+\frac{1}{4} e^{s_2} G_1\left(s_1\right) k_5\left(-s_2\right)+\frac{1}{4} G_1\left(s_1\right) k_5\left(s_2\right)-\frac{3}{4} G_1\left(s_1\right) k_6\left(s_2\right)-\frac{3}{4} e^{s_2} G_1\left(s_1\right) k_7\left(-s_2\right)+\frac{\left(e^{s_1+s_2}-1\right) k_5\left(s_1\right)}{4 \left(s_1+s_2\right)}+\frac{k_5\left(s_2\right)-k_5\left(s_1+s_2\right)}{4 s_1}+\frac{k_5\left(s_1+s_2\right)-k_5\left(s_1\right)}{4 s_2}+\frac{3 \left(k_6\left(s_1\right)-k_6\left(s_1+s_2\right)\right)}{4 s_2}+\frac{3 e^{s_1} \left(k_7\left(-s_1\right)-e^{s_2} k_7\left(-s_1-s_2\right)\right)}{4 s_2}+\frac{3 e^{s_2} \left(e^{s_1} k_7\left(-s_1-s_2\right)-k_7\left(-s_2\right)\right)}{4 s_1}-\frac{e^{s_2} \left(e^{s_1} k_5\left(-s_1-s_2\right)-k_5\left(-s_2\right)\right)}{4 s_1}-\frac{3 \left(k_6\left(s_2\right)-k_6\left(s_1+s_2\right)\right)}{4 s_1}-\frac{e^{s_1} \left(k_5\left(-s_1\right)-e^{s_2} k_5\left(-s_1-s_2\right)\right)}{4 s_2}-\frac{e^{s_1} \left(k_5\left(-s_1\right)+e^{s_2} k_5\left(s_1\right)-e^{s_2} k_5\left(-s_2\right)-k_5\left(s_2\right)\right)}{4 \left(s_1+s_2\right)}. 
\end{math}
\end{center} 

\subsubsection{The function $\widetilde K_6$}
\label{basicK6}

In \eqref{Gradientofa4} we see that the operator $ \widetilde K_6(\nab, \nab) $ 
acts on two different elements that are not the same up to switching $\done$ 
and $\dtwo$, namely  $\done(h) \cdot \done^3(h)$ and  $\done(h) \cdot \done \dtwo^2(h)$. 
By looking at the corresponding finite difference expressions in the result of the second 
gradient calculation performed in Section \ref{FuncRelationsProofSec} we find the following 
basic equations for $\widetilde K_6(s_1, s_2) $. From the expression associated with the 
term  $\done(h) \cdot \done^3(h)$ we find that 

\begin{equation} \label{basicK6eqnfirst}
 \widetilde K_6(s_1, s_2) =
\end{equation}
\begin{center}
\begin{math}
\frac{1}{15} (-2) \pi  G_2\left(s_1,s_2\right)+\frac{1}{8} e^{s_1+s_2} k_{14}\left(-s_1-s_2,s_1\right)-\frac{1}{8} e^{s_1} k_{14}\left(s_2,-s_1-s_2\right)+\frac{1}{8} k_{15}\left(s_1,s_2\right)-\frac{1}{8} e^{s_1+s_2} k_{15}\left(-s_1-s_2,s_1\right)-\frac{1}{8} k_{16}\left(s_1,s_2\right)+\frac{1}{8} e^{s_1} k_{16}\left(s_2,-s_1-s_2\right)+\frac{1}{2} e^{s_2} G_1\left(s_1\right) k_5\left(-s_2\right)+\frac{1}{2} G_1\left(s_1\right) k_5\left(s_2\right)-\frac{1}{4} e^{s_2} G_1\left(s_1\right) k_6\left(-s_2\right)-\frac{3}{4} G_1\left(s_1\right) k_6\left(s_2\right)-\frac{3}{4} e^{s_2} G_1\left(s_1\right) k_7\left(-s_2\right)-\frac{1}{4} G_1\left(s_1\right) k_7\left(s_2\right)+\frac{k_5\left(s_2\right)-k_5\left(s_1+s_2\right)}{2 s_1}+\frac{\left(e^{s_1+s_2}-1\right) k_6\left(s_1\right)}{4 \left(s_1+s_2\right)}+\frac{e^{s_2} \left(e^{s_1} k_6\left(-s_1-s_2\right)-k_6\left(-s_2\right)\right)}{4 s_1}+\frac{k_6\left(s_1\right)-k_6\left(s_1+s_2\right)}{4 s_2}+\frac{e^{s_1} \left(k_7\left(-s_1\right)-e^{s_2} k_7\left(-s_1-s_2\right)\right)}{4 s_2}+\frac{3 e^{s_2} \left(e^{s_1} k_7\left(-s_1-s_2\right)-k_7\left(-s_2\right)\right)}{4 s_1}+\frac{k_7\left(s_1+s_2\right)-k_7\left(s_2\right)}{4 s_1}-\frac{e^{s_2} \left(e^{s_1} k_5\left(-s_1-s_2\right)-k_5\left(-s_2\right)\right)}{2 s_1}-\frac{3 \left(k_6\left(s_2\right)-k_6\left(s_1+s_2\right)\right)}{4 s_1}-\frac{e^{s_1+s_2} \left(k_6\left(s_1\right)-k_6\left(-s_2\right)\right)}{4 \left(s_1+s_2\right)}-\frac{e^{s_1} \left(k_7\left(-s_1\right)-k_7\left(s_2\right)\right)}{4 \left(s_1+s_2\right)}. 
\end{math}
\end{center} 

\smallskip 
Moreover, the expression associated with the term  $\done(h) \cdot \done \dtwo^2(h)$ 
as explained above yields: 
\begin{equation} \label{basicK6eqnsecond}
 \widetilde K_6(s_1, s_2) =
\end{equation}
\begin{center}
\begin{math}
\frac{1}{15} (-2) \pi  G_2\left(s_1,s_2\right)+\frac{1}{4} e^{s_2} G_1\left(s_1\right) k_3\left(-s_2\right)+\frac{1}{4} G_1\left(s_1\right) k_3\left(s_2\right)+\frac{1}{2} e^{s_2} G_1\left(s_1\right) k_4\left(-s_2\right)+\frac{1}{2} G_1\left(s_1\right) k_4\left(s_2\right)+\frac{\left(1-e^{s_1+s_2}\right) k_6\left(s_1\right)}{4 \left(-s_1-s_2\right)}-\frac{1}{4} e^{s_2} G_1\left(s_1\right) k_6\left(-s_2\right)+\frac{e^{s_1+s_2} \left(k_6\left(-s_2\right)-k_6\left(s_1\right)\right)}{4 \left(s_1+s_2\right)}+\frac{e^{s_1+s_2} k_6\left(-s_1-s_2\right)-e^{s_2} k_6\left(-s_2\right)}{4 s_1}-\frac{3}{4} G_1\left(s_1\right) k_6\left(s_2\right)+\frac{3 \left(k_6\left(s_1+s_2\right)-k_6\left(s_2\right)\right)}{4 s_1}-\frac{3}{4} e^{s_2} G_1\left(s_1\right) k_7\left(-s_2\right)+\frac{3 \left(e^{s_1+s_2} k_7\left(-s_1-s_2\right)-e^{s_2} k_7\left(-s_2\right)\right)}{4 s_1}-\frac{1}{4} G_1\left(s_1\right) k_7\left(s_2\right)+\frac{e^{s_1} \left(k_7\left(s_2\right)-k_7\left(-s_1\right)\right)}{4 \left(s_1+s_2\right)}+\frac{k_7\left(s_1+s_2\right)-k_7\left(s_2\right)}{4 s_1}+\frac{1}{8} k_8\left(s_1,s_2\right)-\frac{1}{8} e^{s_1+s_2} k_8\left(-s_1-s_2,s_1\right)+\frac{1}{8} k_9\left(s_1,s_2\right)-\frac{1}{8} e^{s_1+s_2} k_9\left(-s_1-s_2,s_1\right)-\frac{1}{8} k_{10}\left(s_1,s_2\right)+\frac{1}{8} e^{s_1} k_{10}\left(s_2,-s_1-s_2\right)-\frac{1}{8} k_{11}\left(s_1,s_2\right)+\frac{1}{8} e^{s_1} k_{11}\left(s_2,-s_1-s_2\right)+\frac{1}{8} e^{s_1+s_2} k_{12}\left(-s_1-s_2,s_1\right)-\frac{1}{8} e^{s_1} k_{12}\left(s_2,-s_1-s_2\right)+\frac{1}{8} e^{s_1+s_2} k_{13}\left(-s_1-s_2,s_1\right)-\frac{1}{8} e^{s_1} k_{13}\left(s_2,-s_1-s_2\right)-\frac{e^{s_1+s_2} k_4\left(-s_1-s_2\right)-e^{s_2} k_4\left(-s_2\right)}{2 s_1}-\frac{k_4\left(s_1+s_2\right)-k_4\left(s_2\right)}{2 s_1}-\frac{e^{s_1+s_2} k_3\left(-s_1-s_2\right)-e^{s_2} k_3\left(-s_2\right)}{4 s_1}-\frac{k_3\left(s_1+s_2\right)-k_3\left(s_2\right)}{4 s_1}-\frac{k_6\left(s_1+s_2\right)-k_6\left(s_1\right)}{4 s_2}-\frac{e^{s_1+s_2} k_7\left(-s_1-s_2\right)-e^{s_1} k_7\left(-s_1\right)}{4 s_2}. 
\end{math}
\end{center}

\subsubsection{The function $\widetilde K_7$}
\label{basicK7}
The situation for the last two variable function $\widetilde K_7$ is similar to 
that of $\widetilde K_6$ in the sense that the operator  $\widetilde K_7(\nab, \nab)$ 
in \eqref{Gradientofa4}
acts on two different types of elements that are different even modulo switching 
$\done$ and $\dtwo$. By finding the finite difference expression of the modular operator 
that acts on  $\done^3(h) \cdot \done(h)$ in the second gradient calculation of Section 
\ref{FuncRelationsProofSec} we have: 

\begin{equation} \label{basicK7eqnfirst}
\widetilde K_7(s_1, s_2) = 
\end{equation}
\begin{center}
\begin{math}
 \frac{1}{15} (-2) \pi  G_2\left(s_1,s_2\right)-\frac{1}{8} k_{14}\left(s_1,s_2\right)+\frac{1}{8} e^{s_1+s_2} k_{14}\left(-s_1-s_2,s_1\right)+\frac{1}{8} k_{15}\left(s_1,s_2\right)-\frac{1}{8} e^{s_1} k_{15}\left(s_2,-s_1-s_2\right)-\frac{1}{8} e^{s_1+s_2} k_{16}\left(-s_1-s_2,s_1\right)+\frac{1}{8} e^{s_1} k_{16}\left(s_2,-s_1-s_2\right)-\frac{1}{4} G_1\left(s_1\right) k_6\left(s_2\right)-\frac{1}{4} e^{s_2} G_1\left(s_1\right) k_7\left(-s_2\right)+\frac{k_5\left(s_1+s_2\right)-k_5\left(s_1\right)}{2 s_2}+\frac{e^{s_1} \left(k_6\left(-s_1\right)-e^{s_2} k_6\left(-s_1-s_2\right)\right)}{4 s_2}+\frac{3 \left(k_6\left(s_1\right)-k_6\left(s_1+s_2\right)\right)}{4 s_2}+\frac{k_6\left(s_1+s_2\right)-k_6\left(s_2\right)}{4 s_1}+\frac{\left(e^{s_1+s_2}-1\right) k_7\left(s_1\right)}{4 \left(s_1+s_2\right)}+\frac{3 e^{s_1} \left(k_7\left(-s_1\right)-e^{s_2} k_7\left(-s_1-s_2\right)\right)}{4 s_2}+\frac{e^{s_2} \left(e^{s_1} k_7\left(-s_1-s_2\right)-k_7\left(-s_2\right)\right)}{4 s_1}+\frac{k_7\left(s_1\right)-k_7\left(s_1+s_2\right)}{4 s_2}-\frac{e^{s_1} \left(k_5\left(-s_1\right)-e^{s_2} k_5\left(-s_1-s_2\right)\right)}{2 s_2}-\frac{e^{s_1} \left(k_6\left(-s_1\right)-k_6\left(s_2\right)\right)}{4 \left(s_1+s_2\right)}-\frac{e^{s_1+s_2} \left(k_7\left(s_1\right)-k_7\left(-s_2\right)\right)}{4 \left(s_1+s_2\right)}. 
\end{math}
\end{center} 

\smallskip

The finite difference expression associated with the element 
$\done \dtwo^2(h) \cdot \done(h)$ as explained above gives 
another basic identity: 

\begin{equation} \label{basicK7eqnsecond}
\widetilde K_7(s_1, s_2) = 
\end{equation}
\begin{center}
\begin{math}
\frac{1}{15} (-2) \pi  G_2\left(s_1,s_2\right)+\frac{e^{s_1+s_2} k_3\left(-s_1-s_2\right)-e^{s_1} k_3\left(-s_1\right)}{4 s_2}+\frac{k_3\left(s_1+s_2\right)-k_3\left(s_1\right)}{4 s_2}+\frac{e^{s_1+s_2} k_4\left(-s_1-s_2\right)-e^{s_1} k_4\left(-s_1\right)}{2 s_2}+\frac{k_4\left(s_1+s_2\right)-k_4\left(s_1\right)}{2 s_2}-\frac{1}{4} G_1\left(s_1\right) k_6\left(s_2\right)+\frac{e^{s_1} \left(k_6\left(s_2\right)-k_6\left(-s_1\right)\right)}{4 \left(s_1+s_2\right)}+\frac{k_6\left(s_1+s_2\right)-k_6\left(s_2\right)}{4 s_1}+\frac{\left(1-e^{s_1+s_2}\right) k_7\left(s_1\right)}{4 \left(-s_1-s_2\right)}-\frac{1}{4} e^{s_2} G_1\left(s_1\right) k_7\left(-s_2\right)+\frac{e^{s_1+s_2} \left(k_7\left(-s_2\right)-k_7\left(s_1\right)\right)}{4 \left(s_1+s_2\right)}+\frac{e^{s_1+s_2} k_7\left(-s_1-s_2\right)-e^{s_2} k_7\left(-s_2\right)}{4 s_1}+\frac{1}{8} k_8\left(s_1,s_2\right)-\frac{1}{8} e^{s_1} k_8\left(s_2,-s_1-s_2\right)+\frac{1}{8} k_9\left(s_1,s_2\right)-\frac{1}{8} e^{s_1} k_9\left(s_2,-s_1-s_2\right)-\frac{1}{8} e^{s_1+s_2} k_{10}\left(-s_1-s_2,s_1\right)+\frac{1}{8} e^{s_1} k_{10}\left(s_2,-s_1-s_2\right)-\frac{1}{8} e^{s_1+s_2} k_{11}\left(-s_1-s_2,s_1\right)+\frac{1}{8} e^{s_1} k_{11}\left(s_2,-s_1-s_2\right)-\frac{1}{8} k_{12}\left(s_1,s_2\right)+\frac{1}{8} e^{s_1+s_2} k_{12}\left(-s_1-s_2,s_1\right)-\frac{1}{8} k_{13}\left(s_1,s_2\right)+\frac{1}{8} e^{s_1+s_2} k_{13}\left(-s_1-s_2,s_1\right)-\frac{e^{s_1+s_2} k_6\left(-s_1-s_2\right)-e^{s_1} k_6\left(-s_1\right)}{4 s_2}-\frac{3 \left(k_6\left(s_1+s_2\right)-k_6\left(s_1\right)\right)}{4 s_2}-\frac{3 \left(e^{s_1+s_2} k_7\left(-s_1-s_2\right)-e^{s_1} k_7\left(-s_1\right)\right)}{4 s_2}-\frac{k_7\left(s_1+s_2\right)-k_7\left(s_1\right)}{4 s_2}. 
\end{math}
\end{center}

\subsection{The functions $\widetilde K_8,$ $\dots,$ $\widetilde K_{16}$}
In this subsection we present the functional relation associated with 
the three variable functions $\widetilde K_8$ defined 
by \eqref{Ktildesdef}. Since the functional relations 
for the functions $\widetilde K_9$  $\dots,$ $\widetilde K_{16}$ are similarly lengthy, 
for the sake of completeness, they are provided in Appendix \ref{lengthyfnrelationsappsec}. 
Similarly to the previous functions,  
these functional relations are derived by using the 
identity \eqref{Gradientofa4} and by making a comparison with 
the corresponding terms in the result of the gradient calculation carried out  in 
Section \ref{GradientCalculationSec}.

\subsubsection{The function $\widetilde K_8$}
\label{basicK8}
We have: 

\begin{equation} \label{basicK8eqn}
\widetilde K_8(s_1, s_2, s_3) =
\end{equation}

\begin{center}
\begin{math}
 \frac{1}{15} (-2) \pi  G_3\left(s_1,s_2,s_3\right)+\frac{1}{2} e^{s_3} G_2\left(s_1,s_2\right) k_4\left(-s_3\right)-\frac{e^{s_3} \left(e^{s_2} s_1 k_4\left(-s_2-s_3\right)+e^{s_2} s_2 k_4\left(-s_2-s_3\right)-e^{s_1+s_2} s_2 k_4\left(-s_1-s_2-s_3\right)-s_1 k_4\left(-s_3\right)\right)}{2 s_1 s_2 \left(s_1+s_2\right)}+\frac{1}{2} G_2\left(s_1,s_2\right) k_4\left(s_3\right)+\frac{G_1\left(s_1\right) \left(k_4\left(s_3\right)-k_4\left(s_2+s_3\right)\right)}{2 s_2}+\frac{s_1 k_4\left(s_3\right)-s_1 k_4\left(s_2+s_3\right)-s_2 k_4\left(s_2+s_3\right)+s_2 k_4\left(s_1+s_2+s_3\right)}{2 s_1 s_2 \left(s_1+s_2\right)}-\frac{1}{2} G_2\left(s_1,s_2\right) k_6\left(s_3\right)+\frac{G_1\left(s_1\right) \left(k_6\left(s_2\right)-k_6\left(s_2+s_3\right)\right)}{4 s_3}+\frac{k_6\left(s_2\right)-k_6\left(s_1+s_2\right)-k_6\left(s_2+s_3\right)+k_6\left(s_1+s_2+s_3\right)}{4 s_1 s_3}+\frac{-s_3 k_6\left(s_1\right)+s_2 k_6\left(s_1+s_2\right)+s_3 k_6\left(s_1+s_2\right)-s_2 k_6\left(s_1+s_2+s_3\right)}{4 s_2 s_3 \left(s_2+s_3\right)}+\frac{-s_1 k_6\left(s_3\right)+s_1 k_6\left(s_2+s_3\right)+s_2 k_6\left(s_2+s_3\right)-s_2 k_6\left(s_1+s_2+s_3\right)}{2 s_1 s_2 \left(s_1+s_2\right)}+\frac{e^{s_2} G_1\left(s_1\right) \left(k_7\left(-s_2\right)-e^{s_3} k_7\left(-s_2-s_3\right)\right)}{4 s_3}-\frac{e^{s_1} \left(s_3 k_7\left(-s_1\right)-e^{s_2} s_2 k_7\left(-s_1-s_2\right)-e^{s_2} s_3 k_7\left(-s_1-s_2\right)+e^{s_2+s_3} s_2 k_7\left(-s_1-s_2-s_3\right)\right)}{4 s_2 s_3 \left(s_2+s_3\right)}-\frac{e^{s_2} \left(e^{s_1} k_7\left(-s_1-s_2\right)-k_7\left(-s_2\right)+e^{s_3} k_7\left(-s_2-s_3\right)-e^{s_1+s_3} k_7\left(-s_1-s_2-s_3\right)\right)}{4 s_1 s_3}+\frac{e^{s_3} G_1\left(s_1\right) \left(e^{s_2} k_7\left(-s_2-s_3\right)-k_7\left(-s_3\right)\right)}{2 s_2}-\frac{1}{2} e^{s_3} G_2\left(s_1,s_2\right) k_7\left(-s_3\right)+\frac{e^{s_3} \left(e^{s_2} s_1 k_7\left(-s_2-s_3\right)+e^{s_2} s_2 k_7\left(-s_2-s_3\right)-e^{s_1+s_2} s_2 k_7\left(-s_1-s_2-s_3\right)-s_1 k_7\left(-s_3\right)\right)}{2 s_1 s_2 \left(s_1+s_2\right)}+\frac{\left(-1+e^{s_1+s_2+s_3}\right) k_8\left(s_1,s_2\right)}{8 \left(s_1+s_2+s_3\right)}+\frac{k_8\left(s_1,s_2+s_3\right)-k_8\left(s_1,s_2\right)}{8 s_3}-\frac{1}{8} e^{s_2+s_3} G_1\left(s_1\right) k_8\left(-s_2-s_3,s_2\right)+\frac{e^{s_1+s_2+s_3} \left(k_8\left(-s_1-s_2-s_3,s_1\right)-k_8\left(-s_1-s_2-s_3,s_1+s_2\right)\right)}{8 s_2}+\frac{1}{8} G_1\left(s_1\right) k_9\left(s_2,s_3\right)+\frac{k_9\left(s_2,s_3\right)-k_9\left(s_1+s_2,s_3\right)}{8 s_1}+\frac{k_9\left(s_1+s_2,s_3\right)-k_9\left(s_1,s_2+s_3\right)}{8 s_2}+\frac{1}{8} e^{s_2} G_1\left(s_1\right) k_{10}\left(s_3,-s_2-s_3\right)+\frac{e^{s_2} \left(k_{10}\left(s_3,-s_2-s_3\right)-e^{s_1} k_{10}\left(s_3,-s_1-s_2-s_3\right)\right)}{8 s_1}+\frac{e^{s_1} \left(e^{s_2} k_{10}\left(s_3,-s_1-s_2-s_3\right)-k_{10}\left(s_2+s_3,-s_1-s_2-s_3\right)\right)}{8 s_2}-\frac{1}{8} G_1\left(s_1\right) k_{11}\left(s_2,s_3\right)+\frac{k_{11}\left(s_1,s_2+s_3\right)-k_{11}\left(s_1+s_2,s_3\right)}{8 s_2}+\frac{k_{11}\left(s_1+s_2,s_3\right)-k_{11}\left(s_2,s_3\right)}{8 s_1}-\frac{1}{8} e^{s_2} G_1\left(s_1\right) k_{12}\left(s_3,-s_2-s_3\right)+\frac{1}{8} e^{s_2+s_3} G_1\left(s_1\right) k_{13}\left(-s_2-s_3,s_2\right)+\frac{e^{s_2+s_3} \left(k_{13}\left(-s_2-s_3,s_2\right)-e^{s_1} k_{13}\left(-s_1-s_2-s_3,s_1+s_2\right)\right)}{8 s_1}-\frac{1}{16} k_{17}\left(s_1,s_2,s_3\right)-\frac{1}{16} e^{s_1+s_2} k_{17}\left(s_3,-s_1-s_2-s_3,s_1\right)-\frac{1}{16} e^{s_1} k_{19}\left(s_2,s_3,-s_1-s_2-s_3\right)-\frac{1}{16} e^{s_1+s_2+s_3} k_{19}\left(-s_1-s_2-s_3,s_1,s_2\right)-\frac{e^{s_2+s_3} \left(k_8\left(-s_2-s_3,s_2\right)-e^{s_1} k_8\left(-s_1-s_2-s_3,s_1+s_2\right)\right)}{8 s_1}-\frac{e^{s_2} \left(k_{12}\left(s_3,-s_2-s_3\right)-e^{s_1} k_{12}\left(s_3,-s_1-s_2-s_3\right)\right)}{8 s_1}-\frac{e^{s_3} G_1\left(s_1\right) \left(e^{s_2} k_4\left(-s_2-s_3\right)-k_4\left(-s_3\right)\right)}{2 s_2}-\frac{G_1\left(s_1\right) \left(k_6\left(s_3\right)-k_6\left(s_2+s_3\right)\right)}{2 s_2}-\frac{e^{s_1} \left(e^{s_2} k_{12}\left(s_3,-s_1-s_2-s_3\right)-k_{12}\left(s_2+s_3,-s_1-s_2-s_3\right)\right)}{8 s_2}-\frac{e^{s_1+s_2+s_3} \left(k_{13}\left(-s_1-s_2-s_3,s_1\right)-k_{13}\left(-s_1-s_2-s_3,s_1+s_2\right)\right)}{8 s_2}-\frac{e^{s_1} \left(k_{11}\left(s_2,-s_1-s_2\right)-k_{11}\left(s_2+s_3,-s_1-s_2-s_3\right)\right)}{8 s_3}-\frac{e^{s_1+s_2} \left(k_{12}\left(-s_1-s_2,s_1\right)-e^{s_3} k_{12}\left(-s_1-s_2-s_3,s_1\right)\right)}{8 s_3}-\frac{e^{s_1+s_2+s_3} \left(k_8\left(s_1,s_2\right)-k_8\left(-s_2-s_3,s_2\right)\right)}{8 \left(s_1+s_2+s_3\right)}-\frac{e^{s_1} \left(k_{11}\left(s_2,-s_1-s_2\right)-k_{11}\left(s_2,s_3\right)\right)}{8 \left(s_1+s_2+s_3\right)}-\frac{e^{s_1+s_2} \left(k_{12}\left(-s_1-s_2,s_1\right)-k_{12}\left(s_3,-s_2-s_3\right)\right)}{8 \left(s_1+s_2+s_3\right)}. 
\end{math}
\end{center}

\subsection{The functions $\widetilde K_{17}, \dots, \widetilde K_{20}$}
In this subsection we present the functional relation associated with 
the four variable function $\widetilde K_{17},$  defined 
by \eqref{Ktildesdef}. Like this functional relation, the corresponding basic equations for the 
functions $\widetilde K_{18}, \widetilde K_{19}, \widetilde K_{20}$, which are written in Appendix  \ref{lengthyfnrelationsappsec}, 
have  quite lengthy expressions.  
These relations will in particular show that the 
the four variable functions $K_{17}, \dots, K_{20}$ appearing in the expression \eqref{a_4expression} 
for the term $a_4$ can be constructed from the one and two 
variable functions given explicitly in Section \ref{ExplicitFormulasSec} 
and the three variable functions provided explicitly in Appendix \ref{explicit3and4fnsappsec}, 
with the aid of the functions $G_1, G_2, G_3, G_4$ given by \eqref{ExplicitBabyFunctions}.

\subsubsection{The function $\widetilde K_{17}$} 
\label{basicK17}
In this case also, by using the 
identity \eqref{Gradientofa4} and considering the 
specific function of $\nabla$ that acts on $ \delta _1(h )\cdot \delta _1(h )\cdot 
\delta_2(h )\cdot \delta _2(h ) $ 
in the final formula for the gradient $\dep \vphi_0(a_4(h+\vep a))$ 
as calculated in Section \ref{GradientCalculationSec},
we obtain the functional relation associated with the function 
$\widetilde K_{17}$.   
We have: 

\begin{equation} \label{basicK17eqn}
\widetilde K_{17}(s_1, s_2, s_3, s_4) =
\end{equation}
\begin{center}
\begin{math}
 \frac{1}{15} (-4) \pi  G_4\left(s_1,s_2,s_3,s_4\right)+\frac{s_3 k_6\left(s_1\right)}{2 s_2 \left(s_2+s_3\right) \left(s_3+s_4\right) \left(s_2+s_3+s_4\right)}+\frac{s_4 k_6\left(s_1\right)}{2 s_2 \left(s_2+s_3\right) \left(s_3+s_4\right) \left(s_2+s_3+s_4\right)}+\frac{k_6\left(s_1+s_2\right)}{2 s_1 s_3 \left(s_3+s_4\right)}+\frac{G_1\left(s_1\right) k_6\left(s_3\right)}{2 s_2 s_4}+\frac{G_2\left(s_1,s_2\right) k_6\left(s_3\right)}{2 s_4}+\frac{k_6\left(s_3\right)}{2 s_2 \left(s_1+s_2\right) s_4}+\frac{G_1\left(s_1\right) k_6\left(s_2+s_3\right)}{2 s_3 \left(s_3+s_4\right)}+\frac{G_1\left(s_1\right) k_6\left(s_2+s_3\right)}{2 s_4 \left(s_3+s_4\right)}+\frac{k_6\left(s_2+s_3\right)}{2 s_1 s_3 \left(s_3+s_4\right)}+\frac{k_6\left(s_2+s_3\right)}{2 s_1 s_4 \left(s_3+s_4\right)}+\frac{k_6\left(s_1+s_2+s_3\right)}{2 s_1 \left(s_1+s_2\right) s_4}+\frac{k_6\left(s_1+s_2+s_3\right)}{\left(s_2+s_3\right) \left(s_3+s_4\right) \left(s_2+s_3+s_4\right)}+\frac{s_2 k_6\left(s_1+s_2+s_3\right)}{2 s_3 \left(s_2+s_3\right) \left(s_3+s_4\right) \left(s_2+s_3+s_4\right)}+\frac{s_4 k_6\left(s_1+s_2+s_3\right)}{2 s_3 \left(s_2+s_3\right) \left(s_3+s_4\right) \left(s_2+s_3+s_4\right)}+\frac{s_2 k_6\left(s_1+s_2+s_3\right)}{2 \left(s_2+s_3\right) s_4 \left(s_3+s_4\right) \left(s_2+s_3+s_4\right)}+\frac{s_3 k_6\left(s_1+s_2+s_3\right)}{2 \left(s_2+s_3\right) s_4 \left(s_3+s_4\right) \left(s_2+s_3+s_4\right)}-\frac{1}{2} G_3\left(s_1,s_2,s_3\right) k_6\left(s_4\right)+\frac{G_1\left(s_1\right) k_6\left(s_3+s_4\right)}{2 s_2 \left(s_2+s_3\right)}+\frac{G_1\left(s_1\right) k_6\left(s_3+s_4\right)}{2 s_3 \left(s_2+s_3\right)}+\frac{G_2\left(s_1,s_2\right) k_6\left(s_3+s_4\right)}{2 s_3}+\frac{k_6\left(s_3+s_4\right)}{\left(s_1+s_2\right) \left(s_2+s_3\right) \left(s_1+s_2+s_3\right)}+\frac{s_1 k_6\left(s_3+s_4\right)}{2 s_2 \left(s_1+s_2\right) \left(s_2+s_3\right) \left(s_1+s_2+s_3\right)}+\frac{s_3 k_6\left(s_3+s_4\right)}{2 s_2 \left(s_1+s_2\right) \left(s_2+s_3\right) \left(s_1+s_2+s_3\right)}+\frac{s_1 k_6\left(s_3+s_4\right)}{2 \left(s_1+s_2\right) s_3 \left(s_2+s_3\right) \left(s_1+s_2+s_3\right)}+\frac{s_2 k_6\left(s_3+s_4\right)}{2 \left(s_1+s_2\right) s_3 \left(s_2+s_3\right) \left(s_1+s_2+s_3\right)}+\frac{G_1\left(s_1\right) k_6\left(s_2+s_3+s_4\right)}{2 s_2 s_4}+\frac{k_6\left(s_2+s_3+s_4\right)}{2 s_1 \left(s_1+s_2\right) s_4}+\frac{k_6\left(s_2+s_3+s_4\right)}{2 s_2 \left(s_1+s_2\right) s_4}+\frac{s_2 k_6\left(s_1+s_2+s_3+s_4\right)}{2 s_1 \left(s_1+s_2\right) \left(s_2+s_3\right) \left(s_1+s_2+s_3\right)}+\frac{s_3 k_6\left(s_1+s_2+s_3+s_4\right)}{2 s_1 \left(s_1+s_2\right) \left(s_2+s_3\right) \left(s_1+s_2+s_3\right)}+\frac{k_6\left(s_1+s_2+s_3+s_4\right)}{2 s_1 s_4 \left(s_3+s_4\right)}+\frac{e^{s_1} s_3 k_7\left(-s_1\right)}{2 s_2 \left(s_2+s_3\right) \left(s_3+s_4\right) \left(s_2+s_3+s_4\right)}+\frac{e^{s_1} s_4 k_7\left(-s_1\right)}{2 s_2 \left(s_2+s_3\right) \left(s_3+s_4\right) \left(s_2+s_3+s_4\right)}+\frac{e^{s_1+s_2} k_7\left(-s_1-s_2\right)}{2 s_1 s_3 \left(s_3+s_4\right)}+\frac{e^{s_2+s_3} G_1\left(s_1\right) k_7\left(-s_2-s_3\right)}{2 s_3 \left(s_3+s_4\right)}+\frac{e^{s_2+s_3} G_1\left(s_1\right) k_7\left(-s_2-s_3\right)}{2 s_4 \left(s_3+s_4\right)}+\frac{e^{s_2+s_3} k_7\left(-s_2-s_3\right)}{2 s_1 s_3 \left(s_3+s_4\right)}+\frac{e^{s_2+s_3} k_7\left(-s_2-s_3\right)}{2 s_1 s_4 \left(s_3+s_4\right)}+\frac{e^{s_1+s_2+s_3} k_7\left(-s_1-s_2-s_3\right)}{2 s_1 \left(s_1+s_2\right) s_4}+\frac{e^{s_1+s_2+s_3} k_7\left(-s_1-s_2-s_3\right)}{\left(s_2+s_3\right) \left(s_3+s_4\right) \left(s_2+s_3+s_4\right)}+\frac{e^{s_1+s_2+s_3} s_2 k_7\left(-s_1-s_2-s_3\right)}{2 s_3 \left(s_2+s_3\right) \left(s_3+s_4\right) \left(s_2+s_3+s_4\right)}+\frac{e^{s_1+s_2+s_3} s_4 k_7\left(-s_1-s_2-s_3\right)}{2 s_3 \left(s_2+s_3\right) \left(s_3+s_4\right) \left(s_2+s_3+s_4\right)}+\frac{e^{s_1+s_2+s_3} s_2 k_7\left(-s_1-s_2-s_3\right)}{2 \left(s_2+s_3\right) s_4 \left(s_3+s_4\right) \left(s_2+s_3+s_4\right)}+\frac{e^{s_1+s_2+s_3} s_3 k_7\left(-s_1-s_2-s_3\right)}{2 \left(s_2+s_3\right) s_4 \left(s_3+s_4\right) \left(s_2+s_3+s_4\right)}+\frac{e^{s_3} G_1\left(s_1\right) k_7\left(-s_3\right)}{2 s_2 s_4}+\frac{e^{s_3} G_2\left(s_1,s_2\right) k_7\left(-s_3\right)}{2 s_4}+\frac{e^{s_3} k_7\left(-s_3\right)}{2 s_2 \left(s_1+s_2\right) s_4}+\frac{e^{s_3+s_4} G_1\left(s_1\right) k_7\left(-s_3-s_4\right)}{2 s_2 \left(s_2+s_3\right)}+\frac{e^{s_3+s_4} G_1\left(s_1\right) k_7\left(-s_3-s_4\right)}{2 s_3 \left(s_2+s_3\right)}+\frac{e^{s_3+s_4} G_2\left(s_1,s_2\right) k_7\left(-s_3-s_4\right)}{2 s_3}+\frac{e^{s_3+s_4} k_7\left(-s_3-s_4\right)}{\left(s_1+s_2\right) \left(s_2+s_3\right) \left(s_1+s_2+s_3\right)}+\frac{e^{s_3+s_4} s_1 k_7\left(-s_3-s_4\right)}{2 s_2 \left(s_1+s_2\right) \left(s_2+s_3\right) \left(s_1+s_2+s_3\right)}+\frac{e^{s_3+s_4} s_3 k_7\left(-s_3-s_4\right)}{2 s_2 \left(s_1+s_2\right) \left(s_2+s_3\right) \left(s_1+s_2+s_3\right)}+\frac{e^{s_3+s_4} s_1 k_7\left(-s_3-s_4\right)}{2 \left(s_1+s_2\right) s_3 \left(s_2+s_3\right) \left(s_1+s_2+s_3\right)}+\frac{e^{s_3+s_4} s_2 k_7\left(-s_3-s_4\right)}{2 \left(s_1+s_2\right) s_3 \left(s_2+s_3\right) \left(s_1+s_2+s_3\right)}+\frac{e^{s_2+s_3+s_4} G_1\left(s_1\right) k_7\left(-s_2-s_3-s_4\right)}{2 s_2 s_4}+\frac{e^{s_2+s_3+s_4} k_7\left(-s_2-s_3-s_4\right)}{2 s_1 \left(s_1+s_2\right) s_4}+\frac{e^{s_2+s_3+s_4} k_7\left(-s_2-s_3-s_4\right)}{2 s_2 \left(s_1+s_2\right) s_4}+\frac{e^{s_1+s_2+s_3+s_4} s_2 k_7\left(-s_1-s_2-s_3-s_4\right)}{2 s_1 \left(s_1+s_2\right) \left(s_2+s_3\right) \left(s_1+s_2+s_3\right)}+\frac{e^{s_1+s_2+s_3+s_4} s_3 k_7\left(-s_1-s_2-s_3-s_4\right)}{2 s_1 \left(s_1+s_2\right) \left(s_2+s_3\right) \left(s_1+s_2+s_3\right)}+\frac{e^{s_1+s_2+s_3+s_4} k_7\left(-s_1-s_2-s_3-s_4\right)}{2 s_1 s_4 \left(s_3+s_4\right)}-\frac{1}{2} e^{s_4} G_3\left(s_1,s_2,s_3\right) k_7\left(-s_4\right)+\frac{k_8\left(s_1,s_2\right)}{4 s_3 \left(s_3+s_4\right)}+\frac{k_8\left(s_1,s_2+s_3+s_4\right)}{4 s_4 \left(s_3+s_4\right)}+\frac{G_1\left(s_1\right) k_8\left(s_3,s_4\right)}{4 s_2}+\frac{1}{4} G_2\left(s_1,s_2\right) k_8\left(s_3,s_4\right)+\frac{k_8\left(s_3,s_4\right)}{4 s_2 \left(s_1+s_2\right)}+\frac{k_8\left(s_1+s_2+s_3,s_4\right)}{4 s_1 \left(s_1+s_2\right)}+\frac{k_9\left(s_1,s_2+s_3\right)}{8 s_2 s_4}+\frac{k_9\left(s_1,s_2+s_3+s_4\right)}{8 s_2 \left(s_2+s_3\right)}+\frac{G_1\left(s_1\right) k_9\left(s_2,s_3+s_4\right)}{8 s_4}+\frac{k_9\left(s_2,s_3+s_4\right)}{8 s_1 s_4}+\frac{k_9\left(s_1+s_2,s_3\right)}{8 s_1 s_4}+\frac{k_9\left(s_1+s_2,s_3+s_4\right)}{8 s_1 s_3}+\frac{k_9\left(s_1+s_2,s_3+s_4\right)}{8 s_2 s_4}+\frac{G_1\left(s_1\right) k_9\left(s_2+s_3,s_4\right)}{8 s_3}+\frac{k_9\left(s_2+s_3,s_4\right)}{8 s_1 s_3}+\frac{k_9\left(s_1+s_2+s_3,s_4\right)}{8 s_3 \left(s_2+s_3\right)}+\frac{e^{s_1+s_2} k_{10}\left(s_3,-s_1-s_2-s_3\right)}{8 s_1 s_4}+\frac{e^{s_1} k_{10}\left(s_2+s_3,-s_1-s_2-s_3\right)}{8 s_2 s_4}+\frac{e^{s_2+s_3} G_1\left(s_1\right) k_{10}\left(s_4,-s_2-s_3-s_4\right)}{8 s_3}+\frac{e^{s_2+s_3} k_{10}\left(s_4,-s_2-s_3-s_4\right)}{8 s_1 s_3}+\frac{e^{s_1+s_2+s_3} k_{10}\left(s_4,-s_1-s_2-s_3-s_4\right)}{8 s_3 \left(s_2+s_3\right)}+\frac{e^{s_2} G_1\left(s_1\right) k_{10}\left(s_3+s_4,-s_2-s_3-s_4\right)}{8 s_4}+\frac{e^{s_2} k_{10}\left(s_3+s_4,-s_2-s_3-s_4\right)}{8 s_1 s_4}+\frac{e^{s_1+s_2} k_{10}\left(s_3+s_4,-s_1-s_2-s_3-s_4\right)}{8 s_1 s_3}+\frac{e^{s_1+s_2} k_{10}\left(s_3+s_4,-s_1-s_2-s_3-s_4\right)}{8 s_2 s_4}+\frac{e^{s_1} k_{10}\left(s_2+s_3+s_4,-s_1-s_2-s_3-s_4\right)}{8 s_2 \left(s_2+s_3\right)}+\frac{e^{s_1} k_{11}\left(s_2,-s_1-s_2\right)}{4 s_3 \left(s_3+s_4\right)}+\frac{e^{s_3} G_1\left(s_1\right) k_{11}\left(s_4,-s_3-s_4\right)}{4 s_2}+\frac{1}{4} e^{s_3} G_2\left(s_1,s_2\right) k_{11}\left(s_4,-s_3-s_4\right)+\frac{e^{s_3} k_{11}\left(s_4,-s_3-s_4\right)}{4 s_2 \left(s_1+s_2\right)}+\frac{e^{s_1+s_2+s_3} k_{11}\left(s_4,-s_1-s_2-s_3-s_4\right)}{4 s_1 \left(s_1+s_2\right)}+\frac{e^{s_1} k_{11}\left(s_2+s_3+s_4,-s_1-s_2-s_3-s_4\right)}{4 s_4 \left(s_3+s_4\right)}+\frac{e^{s_1+s_2} k_{12}\left(-s_1-s_2,s_1\right)}{4 s_3 \left(s_3+s_4\right)}+\frac{e^{s_3+s_4} G_1\left(s_1\right) k_{12}\left(-s_3-s_4,s_3\right)}{4 s_2}+\frac{1}{4} e^{s_3+s_4} G_2\left(s_1,s_2\right) k_{12}\left(-s_3-s_4,s_3\right)+\frac{e^{s_3+s_4} k_{12}\left(-s_3-s_4,s_3\right)}{4 s_2 \left(s_1+s_2\right)}+\frac{e^{s_1+s_2+s_3+s_4} k_{12}\left(-s_1-s_2-s_3-s_4,s_1\right)}{4 s_4 \left(s_3+s_4\right)}+\frac{e^{s_1+s_2+s_3+s_4} k_{12}\left(-s_1-s_2-s_3-s_4,s_1+s_2+s_3\right)}{4 s_1 \left(s_1+s_2\right)}+\frac{e^{s_1+s_2+s_3} k_{13}\left(-s_1-s_2-s_3,s_1\right)}{8 s_2 s_4}+\frac{e^{s_1+s_2+s_3} k_{13}\left(-s_1-s_2-s_3,s_1+s_2\right)}{8 s_1 s_4}+\frac{e^{s_2+s_3+s_4} G_1\left(s_1\right) k_{13}\left(-s_2-s_3-s_4,s_2\right)}{8 s_4}+\frac{e^{s_2+s_3+s_4} k_{13}\left(-s_2-s_3-s_4,s_2\right)}{8 s_1 s_4}+\frac{e^{s_2+s_3+s_4} G_1\left(s_1\right) k_{13}\left(-s_2-s_3-s_4,s_2+s_3\right)}{8 s_3}+\frac{e^{s_2+s_3+s_4} k_{13}\left(-s_2-s_3-s_4,s_2+s_3\right)}{8 s_1 s_3}+\frac{e^{s_1+s_2+s_3+s_4} k_{13}\left(-s_1-s_2-s_3-s_4,s_1\right)}{8 s_2 \left(s_2+s_3\right)}+\frac{e^{s_1+s_2+s_3+s_4} k_{13}\left(-s_1-s_2-s_3-s_4,s_1+s_2\right)}{8 s_1 s_3}+\frac{e^{s_1+s_2+s_3+s_4} k_{13}\left(-s_1-s_2-s_3-s_4,s_1+s_2\right)}{8 s_2 s_4}+\frac{e^{s_1+s_2+s_3+s_4} k_{13}\left(-s_1-s_2-s_3-s_4,s_1+s_2+s_3\right)}{8 s_3 \left(s_2+s_3\right)}+\frac{k_{17}\left(s_1,s_2,s_3\right)}{16 s_4}+\frac{k_{17}\left(s_1,s_2,s_3+s_4\right)}{16 s_3}+\frac{e^{s_1+s_2} k_{17}\left(s_3,-s_1-s_2-s_3,s_1\right)}{16 s_4}-\frac{1}{16} e^{s_2} G_1\left(s_1\right) k_{17}\left(s_3,s_4,-s_2-s_3-s_4\right)+\frac{e^{s_1+s_2} k_{17}\left(s_3,s_4,-s_2-s_3-s_4\right)}{16 \left(s_1+s_2+s_3+s_4\right)}+\frac{e^{s_1+s_2} k_{17}\left(s_3,s_4,-s_1-s_2-s_3-s_4\right)}{16 s_1}+\frac{e^{s_1} k_{17}\left(s_2+s_3,s_4,-s_1-s_2-s_3-s_4\right)}{16 s_2}-\frac{1}{16} e^{s_2+s_3+s_4} G_1\left(s_1\right) k_{17}\left(-s_2-s_3-s_4,s_2,s_3\right)+\frac{e^{s_1+s_2+s_3+s_4} k_{17}\left(-s_2-s_3-s_4,s_2,s_3\right)}{16 \left(s_1+s_2+s_3+s_4\right)}+\frac{e^{s_1+s_2+s_3+s_4} k_{17}\left(-s_1-s_2-s_3-s_4,s_1,s_2+s_3\right)}{16 s_2}+\frac{e^{s_1+s_2+s_3+s_4} k_{17}\left(-s_1-s_2-s_3-s_4,s_1+s_2,s_3\right)}{16 s_1}+\frac{e^{s_1+s_2} k_{17}\left(s_3+s_4,-s_1-s_2-s_3-s_4,s_1\right)}{16 s_3}+\frac{k_{19}\left(s_1,s_2+s_3,s_4\right)}{16 s_2}+\frac{e^{s_1} k_{19}\left(s_2,s_3,-s_1-s_2-s_3\right)}{16 s_4}-\frac{1}{16} G_1\left(s_1\right) k_{19}\left(s_2,s_3,s_4\right)+\frac{e^{s_1} k_{19}\left(s_2,s_3,s_4\right)}{16 \left(s_1+s_2+s_3+s_4\right)}+\frac{e^{s_1} k_{19}\left(s_2,s_3+s_4,-s_1-s_2-s_3-s_4\right)}{16 s_3}+\frac{k_{19}\left(s_1+s_2,s_3,s_4\right)}{16 s_1}+\frac{e^{s_1+s_2+s_3} k_{19}\left(-s_1-s_2-s_3,s_1,s_2\right)}{16 s_4}+\frac{e^{s_1+s_2+s_3+s_4} k_{19}\left(-s_1-s_2-s_3-s_4,s_1,s_2\right)}{16 s_3}-\frac{1}{16} e^{s_2+s_3} G_1\left(s_1\right) k_{19}\left(s_4,-s_2-s_3-s_4,s_2\right)+\frac{e^{s_1+s_2+s_3} k_{19}\left(s_4,-s_2-s_3-s_4,s_2\right)}{16 \left(s_1+s_2+s_3+s_4\right)}+\frac{e^{s_1+s_2+s_3} k_{19}\left(s_4,-s_1-s_2-s_3-s_4,s_1\right)}{16 s_2}+\frac{e^{s_1+s_2+s_3} k_{19}\left(s_4,-s_1-s_2-s_3-s_4,s_1+s_2\right)}{16 s_1}-\frac{e^{s_2} k_{17}\left(s_3,s_4,-s_2-s_3-s_4\right)}{16 s_1}-\frac{e^{s_2+s_3+s_4} k_{17}\left(-s_2-s_3-s_4,s_2,s_3\right)}{16 s_1}-\frac{k_{19}\left(s_2,s_3,s_4\right)}{16 s_1}-\frac{e^{s_2+s_3} k_{19}\left(s_4,-s_2-s_3-s_4,s_2\right)}{16 s_1}-\frac{G_1\left(s_1\right) k_8\left(s_2+s_3,s_4\right)}{4 s_2}-\frac{e^{s_2+s_3} G_1\left(s_1\right) k_{11}\left(s_4,-s_2-s_3-s_4\right)}{4 s_2}-\frac{e^{s_2+s_3+s_4} G_1\left(s_1\right) k_{12}\left(-s_2-s_3-s_4,s_2+s_3\right)}{4 s_2}-\frac{e^{s_1+s_2} k_{17}\left(s_3,s_4,-s_1-s_2-s_3-s_4\right)}{16 s_2}-\frac{e^{s_1+s_2+s_3+s_4} k_{17}\left(-s_1-s_2-s_3-s_4,s_1+s_2,s_3\right)}{16 s_2}-\frac{k_{19}\left(s_1+s_2,s_3,s_4\right)}{16 s_2}-\frac{e^{s_1+s_2+s_3} k_{19}\left(s_4,-s_1-s_2-s_3-s_4,s_1+s_2\right)}{16 s_2}-\frac{k_8\left(s_2+s_3,s_4\right)}{4 s_1 \left(s_1+s_2\right)}-\frac{e^{s_2+s_3} k_{11}\left(s_4,-s_2-s_3-s_4\right)}{4 s_1 \left(s_1+s_2\right)}-\frac{e^{s_2+s_3+s_4} k_{12}\left(-s_2-s_3-s_4,s_2+s_3\right)}{4 s_1 \left(s_1+s_2\right)}-\frac{k_8\left(s_2+s_3,s_4\right)}{4 s_2 \left(s_1+s_2\right)}-\frac{e^{s_2+s_3} k_{11}\left(s_4,-s_2-s_3-s_4\right)}{4 s_2 \left(s_1+s_2\right)}-\frac{e^{s_2+s_3+s_4} k_{12}\left(-s_2-s_3-s_4,s_2+s_3\right)}{4 s_2 \left(s_1+s_2\right)}-\frac{G_2\left(s_1,s_2\right) k_6\left(s_4\right)}{2 s_3}-\frac{e^{s_4} G_2\left(s_1,s_2\right) k_7\left(-s_4\right)}{2 s_3}-\frac{G_1\left(s_1\right) k_9\left(s_2,s_3+s_4\right)}{8 s_3}-\frac{e^{s_2} G_1\left(s_1\right) k_{10}\left(s_3+s_4,-s_2-s_3-s_4\right)}{8 s_3}-\frac{e^{s_2+s_3+s_4} G_1\left(s_1\right) k_{13}\left(-s_2-s_3-s_4,s_2\right)}{8 s_3}-\frac{k_{17}\left(s_1,s_2+s_3,s_4\right)}{16 s_3}-\frac{e^{s_1+s_2+s_3} k_{17}\left(s_4,-s_1-s_2-s_3-s_4,s_1\right)}{16 s_3}-\frac{e^{s_1} k_{19}\left(s_2+s_3,s_4,-s_1-s_2-s_3-s_4\right)}{16 s_3}-\frac{e^{s_1+s_2+s_3+s_4} k_{19}\left(-s_1-s_2-s_3-s_4,s_1,s_2+s_3\right)}{16 s_3}-\frac{k_9\left(s_2,s_3+s_4\right)}{8 s_1 s_3}-\frac{k_9\left(s_1+s_2+s_3,s_4\right)}{8 s_1 s_3}-\frac{e^{s_1+s_2+s_3} k_{10}\left(s_4,-s_1-s_2-s_3-s_4\right)}{8 s_1 s_3}-\frac{e^{s_2} k_{10}\left(s_3+s_4,-s_2-s_3-s_4\right)}{8 s_1 s_3}-\frac{e^{s_2+s_3+s_4} k_{13}\left(-s_2-s_3-s_4,s_2\right)}{8 s_1 s_3}-\frac{e^{s_1+s_2+s_3+s_4} k_{13}\left(-s_1-s_2-s_3-s_4,s_1+s_2+s_3\right)}{8 s_1 s_3}-\frac{G_1\left(s_1\right) k_6\left(s_2+s_3+s_4\right)}{2 s_2 \left(s_2+s_3\right)}-\frac{e^{s_2+s_3+s_4} G_1\left(s_1\right) k_7\left(-s_2-s_3-s_4\right)}{2 s_2 \left(s_2+s_3\right)}-\frac{k_9\left(s_1+s_2,s_3+s_4\right)}{8 s_2 \left(s_2+s_3\right)}-\frac{e^{s_1+s_2} k_{10}\left(s_3+s_4,-s_1-s_2-s_3-s_4\right)}{8 s_2 \left(s_2+s_3\right)}-\frac{e^{s_1+s_2+s_3+s_4} k_{13}\left(-s_1-s_2-s_3-s_4,s_1+s_2\right)}{8 s_2 \left(s_2+s_3\right)}-\frac{G_1\left(s_1\right) k_6\left(s_4\right)}{2 s_3 \left(s_2+s_3\right)}-\frac{e^{s_4} G_1\left(s_1\right) k_7\left(-s_4\right)}{2 s_3 \left(s_2+s_3\right)}-\frac{k_9\left(s_1+s_2,s_3+s_4\right)}{8 s_3 \left(s_2+s_3\right)}-\frac{e^{s_1+s_2} k_{10}\left(s_3+s_4,-s_1-s_2-s_3-s_4\right)}{8 s_3 \left(s_2+s_3\right)}-\frac{e^{s_1+s_2+s_3+s_4} k_{13}\left(-s_1-s_2-s_3-s_4,s_1+s_2\right)}{8 s_3 \left(s_2+s_3\right)}-\frac{k_6\left(s_2+s_3+s_4\right)}{\left(s_1+s_2\right) \left(s_2+s_3\right) \left(s_1+s_2+s_3\right)}-\frac{e^{s_2+s_3+s_4} k_7\left(-s_2-s_3-s_4\right)}{\left(s_1+s_2\right) \left(s_2+s_3\right) \left(s_1+s_2+s_3\right)}-\frac{s_2 k_6\left(s_2+s_3+s_4\right)}{2 s_1 \left(s_1+s_2\right) \left(s_2+s_3\right) \left(s_1+s_2+s_3\right)}-\frac{s_3 k_6\left(s_2+s_3+s_4\right)}{2 s_1 \left(s_1+s_2\right) \left(s_2+s_3\right) \left(s_1+s_2+s_3\right)}-\frac{e^{s_2+s_3+s_4} s_2 k_7\left(-s_2-s_3-s_4\right)}{2 s_1 \left(s_1+s_2\right) \left(s_2+s_3\right) \left(s_1+s_2+s_3\right)}-\frac{e^{s_2+s_3+s_4} s_3 k_7\left(-s_2-s_3-s_4\right)}{2 s_1 \left(s_1+s_2\right) \left(s_2+s_3\right) \left(s_1+s_2+s_3\right)}-\frac{s_1 k_6\left(s_2+s_3+s_4\right)}{2 s_2 \left(s_1+s_2\right) \left(s_2+s_3\right) \left(s_1+s_2+s_3\right)}-\frac{s_3 k_6\left(s_2+s_3+s_4\right)}{2 s_2 \left(s_1+s_2\right) \left(s_2+s_3\right) \left(s_1+s_2+s_3\right)}-\frac{e^{s_2+s_3+s_4} s_1 k_7\left(-s_2-s_3-s_4\right)}{2 s_2 \left(s_1+s_2\right) \left(s_2+s_3\right) \left(s_1+s_2+s_3\right)}-\frac{e^{s_2+s_3+s_4} s_3 k_7\left(-s_2-s_3-s_4\right)}{2 s_2 \left(s_1+s_2\right) \left(s_2+s_3\right) \left(s_1+s_2+s_3\right)}-\frac{s_1 k_6\left(s_4\right)}{2 \left(s_1+s_2\right) s_3 \left(s_2+s_3\right) \left(s_1+s_2+s_3\right)}-\frac{s_2 k_6\left(s_4\right)}{2 \left(s_1+s_2\right) s_3 \left(s_2+s_3\right) \left(s_1+s_2+s_3\right)}-\frac{e^{s_4} s_1 k_7\left(-s_4\right)}{2 \left(s_1+s_2\right) s_3 \left(s_2+s_3\right) \left(s_1+s_2+s_3\right)}-\frac{e^{s_4} s_2 k_7\left(-s_4\right)}{2 \left(s_1+s_2\right) s_3 \left(s_2+s_3\right) \left(s_1+s_2+s_3\right)}-\frac{G_2\left(s_1,s_2\right) k_6\left(s_3+s_4\right)}{2 s_4}-\frac{e^{s_3+s_4} G_2\left(s_1,s_2\right) k_7\left(-s_3-s_4\right)}{2 s_4}-\frac{G_1\left(s_1\right) k_9\left(s_2,s_3\right)}{8 s_4}-\frac{e^{s_2} G_1\left(s_1\right) k_{10}\left(s_3,-s_2-s_3\right)}{8 s_4}-\frac{e^{s_2+s_3} G_1\left(s_1\right) k_{13}\left(-s_2-s_3,s_2\right)}{8 s_4}-\frac{k_{17}\left(s_1,s_2,s_3+s_4\right)}{16 s_4}-\frac{e^{s_1+s_2} k_{17}\left(s_3+s_4,-s_1-s_2-s_3-s_4,s_1\right)}{16 s_4}-\frac{e^{s_1} k_{19}\left(s_2,s_3+s_4,-s_1-s_2-s_3-s_4\right)}{16 s_4}-\frac{e^{s_1+s_2+s_3+s_4} k_{19}\left(-s_1-s_2-s_3-s_4,s_1,s_2\right)}{16 s_4}-\frac{k_9\left(s_2,s_3\right)}{8 s_1 s_4}-\frac{k_9\left(s_1+s_2,s_3+s_4\right)}{8 s_1 s_4}-\frac{e^{s_2} k_{10}\left(s_3,-s_2-s_3\right)}{8 s_1 s_4}-\frac{e^{s_1+s_2} k_{10}\left(s_3+s_4,-s_1-s_2-s_3-s_4\right)}{8 s_1 s_4}-\frac{e^{s_2+s_3} k_{13}\left(-s_2-s_3,s_2\right)}{8 s_1 s_4}-\frac{e^{s_1+s_2+s_3+s_4} k_{13}\left(-s_1-s_2-s_3-s_4,s_1+s_2\right)}{8 s_1 s_4}-\frac{G_1\left(s_1\right) k_6\left(s_2+s_3\right)}{2 s_2 s_4}-\frac{G_1\left(s_1\right) k_6\left(s_3+s_4\right)}{2 s_2 s_4}-\frac{e^{s_2+s_3} G_1\left(s_1\right) k_7\left(-s_2-s_3\right)}{2 s_2 s_4}-\frac{e^{s_3+s_4} G_1\left(s_1\right) k_7\left(-s_3-s_4\right)}{2 s_2 s_4}-\frac{k_9\left(s_1,s_2+s_3+s_4\right)}{8 s_2 s_4}-\frac{k_9\left(s_1+s_2,s_3\right)}{8 s_2 s_4}-\frac{e^{s_1+s_2} k_{10}\left(s_3,-s_1-s_2-s_3\right)}{8 s_2 s_4}-\frac{e^{s_1} k_{10}\left(s_2+s_3+s_4,-s_1-s_2-s_3-s_4\right)}{8 s_2 s_4}-\frac{e^{s_1+s_2+s_3} k_{13}\left(-s_1-s_2-s_3,s_1+s_2\right)}{8 s_2 s_4}-\frac{e^{s_1+s_2+s_3+s_4} k_{13}\left(-s_1-s_2-s_3-s_4,s_1\right)}{8 s_2 s_4}-\frac{k_6\left(s_2+s_3\right)}{2 s_1 \left(s_1+s_2\right) s_4}-\frac{k_6\left(s_1+s_2+s_3+s_4\right)}{2 s_1 \left(s_1+s_2\right) s_4}-\frac{e^{s_2+s_3} k_7\left(-s_2-s_3\right)}{2 s_1 \left(s_1+s_2\right) s_4}-\frac{e^{s_1+s_2+s_3+s_4} k_7\left(-s_1-s_2-s_3-s_4\right)}{2 s_1 \left(s_1+s_2\right) s_4}-\frac{k_6\left(s_2+s_3\right)}{2 s_2 \left(s_1+s_2\right) s_4}-\frac{k_6\left(s_3+s_4\right)}{2 s_2 \left(s_1+s_2\right) s_4}-\frac{e^{s_2+s_3} k_7\left(-s_2-s_3\right)}{2 s_2 \left(s_1+s_2\right) s_4}-\frac{e^{s_3+s_4} k_7\left(-s_3-s_4\right)}{2 s_2 \left(s_1+s_2\right) s_4}-\frac{G_1\left(s_1\right) k_6\left(s_2\right)}{2 s_3 \left(s_3+s_4\right)}-\frac{e^{s_2} G_1\left(s_1\right) k_7\left(-s_2\right)}{2 s_3 \left(s_3+s_4\right)}-\frac{k_8\left(s_1,s_2+s_3\right)}{4 s_3 \left(s_3+s_4\right)}-\frac{e^{s_1} k_{11}\left(s_2+s_3,-s_1-s_2-s_3\right)}{4 s_3 \left(s_3+s_4\right)}-\frac{e^{s_1+s_2+s_3} k_{12}\left(-s_1-s_2-s_3,s_1\right)}{4 s_3 \left(s_3+s_4\right)}-\frac{k_6\left(s_2\right)}{2 s_1 s_3 \left(s_3+s_4\right)}-\frac{k_6\left(s_1+s_2+s_3\right)}{2 s_1 s_3 \left(s_3+s_4\right)}-\frac{e^{s_2} k_7\left(-s_2\right)}{2 s_1 s_3 \left(s_3+s_4\right)}-\frac{e^{s_1+s_2+s_3} k_7\left(-s_1-s_2-s_3\right)}{2 s_1 s_3 \left(s_3+s_4\right)}-\frac{G_1\left(s_1\right) k_6\left(s_2+s_3+s_4\right)}{2 s_4 \left(s_3+s_4\right)}-\frac{e^{s_2+s_3+s_4} G_1\left(s_1\right) k_7\left(-s_2-s_3-s_4\right)}{2 s_4 \left(s_3+s_4\right)}-\frac{k_8\left(s_1,s_2+s_3\right)}{4 s_4 \left(s_3+s_4\right)}-\frac{e^{s_1} k_{11}\left(s_2+s_3,-s_1-s_2-s_3\right)}{4 s_4 \left(s_3+s_4\right)}-\frac{e^{s_1+s_2+s_3} k_{12}\left(-s_1-s_2-s_3,s_1\right)}{4 s_4 \left(s_3+s_4\right)}-\frac{k_6\left(s_1+s_2+s_3\right)}{2 s_1 s_4 \left(s_3+s_4\right)}-\frac{k_6\left(s_2+s_3+s_4\right)}{2 s_1 s_4 \left(s_3+s_4\right)}-\frac{e^{s_1+s_2+s_3} k_7\left(-s_1-s_2-s_3\right)}{2 s_1 s_4 \left(s_3+s_4\right)}-\frac{e^{s_2+s_3+s_4} k_7\left(-s_2-s_3-s_4\right)}{2 s_1 s_4 \left(s_3+s_4\right)}-\frac{k_6\left(s_1+s_2\right)}{\left(s_2+s_3\right) \left(s_3+s_4\right) \left(s_2+s_3+s_4\right)}-\frac{e^{s_1+s_2} k_7\left(-s_1-s_2\right)}{\left(s_2+s_3\right) \left(s_3+s_4\right) \left(s_2+s_3+s_4\right)}-\frac{s_3 k_6\left(s_1+s_2\right)}{2 s_2 \left(s_2+s_3\right) \left(s_3+s_4\right) \left(s_2+s_3+s_4\right)}-\frac{s_4 k_6\left(s_1+s_2\right)}{2 s_2 \left(s_2+s_3\right) \left(s_3+s_4\right) \left(s_2+s_3+s_4\right)}-\frac{e^{s_1+s_2} s_3 k_7\left(-s_1-s_2\right)}{2 s_2 \left(s_2+s_3\right) \left(s_3+s_4\right) \left(s_2+s_3+s_4\right)}-\frac{e^{s_1+s_2} s_4 k_7\left(-s_1-s_2\right)}{2 s_2 \left(s_2+s_3\right) \left(s_3+s_4\right) \left(s_2+s_3+s_4\right)}-\frac{s_2 k_6\left(s_1+s_2\right)}{2 s_3 \left(s_2+s_3\right) \left(s_3+s_4\right) \left(s_2+s_3+s_4\right)}-\frac{s_4 k_6\left(s_1+s_2\right)}{2 s_3 \left(s_2+s_3\right) \left(s_3+s_4\right) \left(s_2+s_3+s_4\right)}-\frac{e^{s_1+s_2} s_2 k_7\left(-s_1-s_2\right)}{2 s_3 \left(s_2+s_3\right) \left(s_3+s_4\right) \left(s_2+s_3+s_4\right)}-\frac{e^{s_1+s_2} s_4 k_7\left(-s_1-s_2\right)}{2 s_3 \left(s_2+s_3\right) \left(s_3+s_4\right) \left(s_2+s_3+s_4\right)}-\frac{s_2 k_6\left(s_1+s_2+s_3+s_4\right)}{2 \left(s_2+s_3\right) s_4 \left(s_3+s_4\right) \left(s_2+s_3+s_4\right)}-\frac{s_3 k_6\left(s_1+s_2+s_3+s_4\right)}{2 \left(s_2+s_3\right) s_4 \left(s_3+s_4\right) \left(s_2+s_3+s_4\right)}-\frac{e^{s_1+s_2+s_3+s_4} s_2 k_7\left(-s_1-s_2-s_3-s_4\right)}{2 \left(s_2+s_3\right) s_4 \left(s_3+s_4\right) \left(s_2+s_3+s_4\right)}-\frac{e^{s_1+s_2+s_3+s_4} s_3 k_7\left(-s_1-s_2-s_3-s_4\right)}{2 \left(s_2+s_3\right) s_4 \left(s_3+s_4\right) \left(s_2+s_3+s_4\right)}-\frac{k_{17}\left(s_1,s_2,s_3\right)}{16 \left(s_1+s_2+s_3+s_4\right)}-\frac{e^{s_1+s_2} k_{17}\left(s_3,-s_1-s_2-s_3,s_1\right)}{16 \left(s_1+s_2+s_3+s_4\right)}-\frac{e^{s_1} k_{19}\left(s_2,s_3,-s_1-s_2-s_3\right)}{16 \left(s_1+s_2+s_3+s_4\right)}-\frac{e^{s_1+s_2+s_3} k_{19}\left(-s_1-s_2-s_3,s_1,s_2\right)}{16 \left(s_1+s_2+s_3+s_4\right)}. 
\end{math}
\end{center}

\smallskip 

As we mentioned earlier, the remaining basic functional relations of the above type, because of their similarly lengthy expressions,  are 
recorded in Appendix \ref{lengthyfnrelationsappsec}.

\smallskip

\section{Differential system from functional relations and action of cyclic groups}
\label{DifferentialSystemSec}

In this section we present a system of differential equations that the functions 
$k_3, \dots, k_{20}$ satisfy. We recall that $k_3, \dots, k_{20}$, which are defined by 
\eqref{littlekfunctions}, are derived from the two, three and four variable functions $K_3, \dots, K_{20}$ 
by specializing them to $s_1+\cdots+s_n=0$, where $n \in \{2, 3, 4\}$ is the number 
of variables that each $K_j$ depends on. The functional relations stated  in 
Theorem \ref{FuncRelationsThm} and presented subsequently and in Appendix \ref{lengthyfnrelationsappsec} 
allow us to abstractly use specific finite differences 
of the functions  $k_3, \dots, k_{20}$ to express the functions 
\[
\widetilde K_j(s_1, \dots, s_n) 
= 
\frac{1}{2^n} \frac{\sinh \left ( (s_1+\dots+s_n)/2 \right )}{(s_1+\dots+s_n)/2} K_j(s_1, \dots, s_n) . 
\]
An interesting question that arises naturally here is whether one can extract more 
abstract information from these relations. We answer this question in this section.

\subsection{System of differential equations for the functions $k_3, \dots, k_{20}$}
\label{differentialsystemderived}

We start by noting that  
the factor that changes $K_j$ to $\widetilde K_j$, namely 
\[
\frac{1}{2^n} \frac{\sinh \left ( (s_1+\dots+s_n)/2 \right )}{(s_1+\dots+s_n)/2}, 
\]
specializes to $\frac{1}{2^n}$ on $s_1+\dots+s_n = 0$, hence the left side 
of each functional relation, under this  specialization, behaves as follows:  
\[
\widetilde K_j(s_1, \dots, s_n)  \to \frac{1}{2^n} k_j(s_1, \dots, s_{n-1}). 
\]
Now we should analyze the behavior of the right side of each functional relation 
under the specialization to $s_1+ \cdots + s_n=0$. We shall see shortly that 
on the right side of each equation, most of the terms behave nicely with respect 
to a direct replacement of $s_n$ by $-s_1- \cdots - s_{n-1}$, except for few 
terms that have $s_1+\dots+s_n = 0$ in their denominators. We will specialize 
each of such terms by calculating a limit with the aid of an estimate of the form 
$s_1+\cdots+s_n= \vep \sim 0$, and will see that partial derivatives of the 
involved functions will appear. Hence, a system of differential equations will be derived.

\smallskip

We first demonstrate this process on the basic equation  \eqref{basicK3eqn}:
\[
\widetilde K_3(s_1, s_2) =
\]
\begin{center}
\begin{math}
\frac{1}{15} (-4) \pi  G_2\left(s_1,s_2\right)+\frac{1}{2} k_8\left(s_1,s_2\right)+\frac{1}{4} k_9\left(s_1,s_2\right)-\frac{1}{4} e^{s_1+s_2} k_9\left(-s_1-s_2,s_1\right)-\frac{1}{4} e^{s_1} k_9\left(s_2,-s_1-s_2\right)-\frac{1}{4} k_{10}\left(s_1,s_2\right)-\frac{1}{4} e^{s_1+s_2} k_{10}\left(-s_1-s_2,s_1\right)+\frac{1}{4} e^{s_1} k_{10}\left(s_2,-s_1-s_2\right)+\frac{1}{2} e^{s_1} k_{11}\left(s_2,-s_1-s_2\right)+\frac{1}{2} e^{s_1+s_2} k_{12}\left(-s_1-s_2,s_1\right)-\frac{1}{4} k_{13}\left(s_1,s_2\right)+\frac{1}{4} e^{s_1+s_2} k_{13}\left(-s_1-s_2,s_1\right)-\frac{1}{4} e^{s_1} k_{13}\left(s_2,-s_1-s_2\right)+\frac{1}{4} e^{s_2} G_1\left(s_1\right) k_3\left(-s_2\right)+\frac{1}{4} G_1\left(s_1\right) k_3\left(s_2\right)-G_1\left(s_1\right) k_6\left(s_2\right)-e^{s_2} G_1\left(s_1\right) k_7\left(-s_2\right)+\frac{\left(e^{s_1+s_2}-1\right) k_3\left(s_1\right)}{4 \left(s_1+s_2\right)}+\frac{k_3\left(s_2\right)-k_3\left(s_1+s_2\right)}{4 s_1}+\frac{k_3\left(s_1+s_2\right)-k_3\left(s_1\right)}{4 s_2}+\frac{k_6\left(s_1\right)-k_6\left(s_1+s_2\right)}{s_2}+\frac{k_6\left(s_1+s_2\right)-k_6\left(s_2\right)}{s_1}+\frac{e^{s_1} \left(k_7\left(-s_1\right)-e^{s_2} k_7\left(-s_1-s_2\right)\right)}{s_2}+\frac{e^{s_2} \left(e^{s_1} k_7\left(-s_1-s_2\right)-k_7\left(-s_2\right)\right)}{s_1}-\frac{e^{s_2} \left(e^{s_1} k_3\left(-s_1-s_2\right)-k_3\left(-s_2\right)\right)}{4 s_1}-\frac{e^{s_1} \left(k_3\left(-s_1\right)-e^{s_2} k_3\left(-s_1-s_2\right)\right)}{4 s_2}-\frac{e^{s_1} \left(k_3\left(-s_1\right)+e^{s_2} k_3\left(s_1\right)-e^{s_2} k_3\left(-s_2\right)-k_3\left(s_2\right)\right)}{4 \left(s_1+s_2\right)},   
\end{math}
\end{center}
which we want to specialize for $s_1+s_2=0$. In that respect the term 
$\widetilde K_3(s_1, s_2)$ gives $\frac 14 k_3(s_1)$ since the coefficient 
$\frac{1}{2^2}\frac{ \sinh\left((s_1+s_2)/2\right)}{(s_1+s_2)/2} $ is well 
behaved and gives $\frac 14$.

\smallskip

The two interesting terms in the above expression are those which have 
$s_1+s_2$ in the denominator. The first one is 
$$
\frac{\left(e^{s_1+s_2}-1\right) k_3\left(s_1\right)}{4 \left(s_1+s_2\right)}, 
$$
which contributes exactly by $\frac 14 k_3(s_1)$ and thus cancels the 
previous term $\widetilde K_3(s_1, s_2)$ which also gives $\frac 14 k_3(s_1)$. 
The second crucial term in the right hand side of the basic equation is 
$$
-\frac{e^{s_1} \left(k_3\left(-s_1\right)+e^{s_2} k_3\left(s_1\right)-e^{s_2} k_3\left(-s_2\right)-k_3\left(s_2\right)\right)}{4 \left(s_1+s_2\right)}, 
$$
which can be treated directly using $s_1+s_2 = \vep \sim 0$ as
$$
-\frac{e^{s_1} \left(k_3\left(-s_1\right)-k_3\left(-s_1+\vep\right)\right)+e^{s_1+s_2}\left( k_3\left(s_1\right)- k_3\left(s_1-\vep \right)\right)}{4 \vep}, 
$$
and it gives all the terms involving the derivative of $k_3$, namely 
$$
\frac14 e^{s_1}k'_3\left(-s_1\right)-\frac14 k'_3\left(s_1\right), 
$$
which corresponds to the differential equation \eqref{diffeqK3} since all 
other terms are specialized without problem using  $s_1+s_2=0$.

\smallskip

As we mentioned above, this approach can be applied to each 
functional relation to derive a differential equation. 
\begin{theorem}
The functions $k_3, \dots, k_{20}$ satisfy a system of partial differential equations as described below. 
\end{theorem}

We explained above the details of the derivation of the differential equation 
associated with the functional relation for the function $\widetilde K_3$. 
We can now write explicitly the differential equation, and will then perform a 
similar analysis on each functional relation to obtain explicitly a system of partial 
differential equations. 

\subsubsection{Differential equation derived from the expression for $\widetilde K_3$} We have:

\begin{equation} \label{diffeqK3}
\frac{1}{4} e^{s_1} k_3'\left(-s_1\right)-\frac{1}{4} k_3'\left(s_1\right) = 
\end{equation}
\begin{center}
\begin{math}
\frac{1}{60 s_1} \Big ( 16 \pi  s_1 G_2(s_1,-s_1)-30 s_1 k_8(s_1,-s_1)+15 s_1 k_9(0,s_1)+15 e^{s_1} s_1 k_9(-s_1,0)-15 s_1 k_9(s_1,-s_1)+15 s_1 k_{10}(0,s_1)-15 e^{s_1} s_1 k_{10}(-s_1,0)+15 s_1 k_{10}(s_1,-s_1)-30 e^{s_1} s_1 k_{11}(-s_1,0)-30 s_1 k_{12}(0,s_1)-15 s_1 k_{13}(0,s_1)+15 e^{s_1} s_1 k_{13}(-s_1,0)+15 s_1 k_{13}(s_1,-s_1)-15 s_1 G_1(s_1) k_3(-s_1)-15 e^{-s_1} s_1 G_1(s_1) k_3(s_1)+60 s_1 G_1(s_1) k_6(-s_1)+60 e^{-s_1} s_1 G_1(s_1) k_7(s_1)-15 e^{s_1} k_3(-s_1)-15 k_3(-s_1)-15 e^{-s_1} k_3(s_1)-15 k_3(s_1)+60 k_6(-s_1)+60 k_6(s_1)+60 e^{s_1} k_7(-s_1)+60 e^{-s_1} k_7(s_1)+60 k_3(0)-120 k_6(0)-120 k_7(0)   \Big ).  
\end{math}
\end{center}

\subsubsection{Differential equation derived from the expression for $\widetilde K_4$} 
In the basic equation \eqref{basicK4eqn} for $\widetilde K_4(s_1, s_2),$ the terms 
that need special care in order to specialize to $s_1+s_2 =0$ are 
\[
\frac{\left(e^{s_1+s_2}-1\right) k_4\left(s_1\right)}{4 \left(s_1+s_2\right)}, 
\qquad 
-\frac{e^{s_1+s_2} \left(k_4\left(s_1\right)-k_4\left(-s_2\right)\right)}{4 \left(s_1+s_2\right)}, 
\qquad 
-\frac{e^{s_1} \left(k_4\left(-s_1\right)-k_4\left(s_2\right)\right)}{4 \left(s_1+s_2\right)}. 
\]
By writing $s_1 + s_2 = \vep \sim 0$, we find that the above terms 
respectively specialize to 
\[
\frac{1}{4} k_4\left(s_1\right), 
\qquad 
-\frac{1}{4} k_4'\left(s_1\right), 
\qquad 
\frac{1}{4} e^{s_1} k_4'\left(-s_1\right). 
\]
Therefore by a direct replacement of $s_2$ by $-s_1$ 
in the rest of the terms we find that 
\begin{equation} \label{diffeqK4}
\frac{1}{4} e^{s_1} k_4'\left(-s_1\right)-\frac{1}{4} k_4'\left(s_1\right)=
\end{equation}
\begin{center}
\begin{math} 
\frac{1}{120 s_1} \Big ( 
8 \pi  s_1 G_2\left(s_1,-s_1\right)+15 s_1 k_8\left(0,s_1\right)+15 e^{s_1} s_1 k_8\left(-s_1,0\right)-15 s_1 k_9\left(s_1,-s_1\right)-15 e^{s_1} s_1 k_{10}\left(-s_1,0\right)+15 s_1 k_{11}\left(0,s_1\right)+15 s_1 k_{11}\left(s_1,-s_1\right)+15 e^{s_1} s_1 k_{12}\left(-s_1,0\right)+15 s_1 k_{12}\left(s_1,-s_1\right)-15 s_1 k_{13}\left(0,s_1\right)-30 s_1 G_1\left(s_1\right) k_4\left(-s_1\right)-30 e^{-s_1} s_1 G_1\left(s_1\right) k_4\left(s_1\right)+30 s_1 G_1\left(s_1\right) k_6\left(-s_1\right)+30 e^{-s_1} s_1 G_1\left(s_1\right) k_7\left(s_1\right)-30 e^{s_1} k_4\left(-s_1\right)-30 k_4\left(-s_1\right)-30 e^{-s_1} k_4\left(s_1\right)-30 k_4\left(s_1\right)+30 k_6\left(-s_1\right)+30 k_6\left(s_1\right)+30 e^{s_1} k_7\left(-s_1\right)+30 e^{-s_1} k_7\left(s_1\right)+120 k_4(0)-60 k_6(0)-60 k_7(0) \Big ). 
\end{math}
\end{center}

\subsubsection{Differential equation derived from the expression for $\widetilde K_5$} 
In the expression given in \eqref{basicK5eqn} for $\widetilde K_5,$ the only terms  that do not 
specialize to $s_1 + s_2 = 0$ by a simple replacement of $s_2$ by $- s_1$ are 
\[
\frac{\left(e^{s_1+s_2}-1\right) k_5\left(s_1\right)}{4 \left(s_1+s_2\right)}, 
\qquad 
-\frac{e^{s_1} \left(k_5\left(-s_1\right)+e^{s_2} k_5\left(s_1\right)-e^{s_2} k_5\left(-s_2\right)-k_5\left(s_2\right)\right)}{4 \left(s_1+s_2\right)}. 
\]
Therefore, we use the estimate $s_1 + s_2 = \vep \sim 0$ to find that the above terms 
respectively tend to the following functions on $s_1 + s_2 =0$: 
\[
\frac{1}{4} k_5\left(s_1\right), 
\qquad 
\frac{1}{4} e^{s_1} k_5'\left(-s_1\right)-\frac{1}{4} k_5'\left(s_1\right). 
\]
Since the rest of the terms behave nicely when we replace $s_2$ by $-s_1$, we 
find the following identity:
\begin{equation} \label{diffeqK5}
\frac{1}{4} e^{s_1} k_5'\left(-s_1\right)-\frac{1}{4} k_5'\left(s_1\right) = 
\end{equation}
\begin{center}
\begin{math}
\frac{1}{40 s_1} \Big (   8 \pi  s_1 G_2\left(s_1,-s_1\right)-10 s_1 k_{14}\left(0,s_1\right)+5 e^{s_1} s_1 k_{14}\left(-s_1,0\right)+5 s_1 k_{14}\left(s_1,-s_1\right)+5 s_1 k_{15}\left(0,s_1\right)+5 e^{s_1} s_1 k_{15}\left(-s_1,0\right)-10 s_1 k_{15}\left(s_1,-s_1\right)+5 s_1 k_{16}\left(0,s_1\right)-10 e^{s_1} s_1 k_{16}\left(-s_1,0\right)+5 s_1 k_{16}\left(s_1,-s_1\right)-10 s_1 G_1\left(s_1\right) k_5\left(-s_1\right)-10 e^{-s_1} s_1 G_1\left(s_1\right) k_5\left(s_1\right)+30 s_1 G_1\left(s_1\right) k_6\left(-s_1\right)+30 e^{-s_1} s_1 G_1\left(s_1\right) k_7\left(s_1\right)-10 e^{s_1} k_5\left(-s_1\right)-10 k_5\left(-s_1\right)-10 e^{-s_1} k_5\left(s_1\right)-10 k_5\left(s_1\right)+30 k_6\left(-s_1\right)+30 k_6\left(s_1\right)+30 e^{s_1} k_7\left(-s_1\right)+30 e^{-s_1} k_7\left(s_1\right)+40 k_5(0)-60 k_6(0)-60 k_7(0) \Big ). 
\end{math}
\end{center}

\subsubsection{Differential equations derived from the two expressions for $\widetilde K_6$} 
As we explained in \S\,\ref{basicK6}, since in the formula \eqref{Gradientofa4}, the 
operator associated with $\widetilde K_6$ acts on two different kinds of elements of 
$\CNT$ that are not the same modulo switching $\done$ and $\dtwo$, we have 
two different basic equations for $K_6$. In the first expression given by \eqref{basicK6eqnfirst}, 
the terms that do not specialize directly to $s_1 + s_2 = 0$ are 
\[
\frac{\left(e^{s_1+s_2}-1\right) k_6\left(s_1\right)}{4 \left(s_1+s_2\right)}, 
\qquad 
-\frac{e^{s_1+s_2} \left(k_6\left(s_1\right)-k_6\left(-s_2\right)\right)}{4 \left(s_1+s_2\right)}, 
\qquad 
-\frac{e^{s_1} \left(k_7\left(-s_1\right)-k_7\left(s_2\right)\right)}{4 \left(s_1+s_2\right)}. 
\]
By using the estimate $s_1 + s_2 = \vep \sim 0$, and calculating a limit in each case, 
we find that the above functions respectively turn to the following functions on $s_1 + s_2 =0:$
\[
\frac{1}{4} k_6\left(s_1\right), 
\qquad
 -\frac{1}{4} k_6'\left(s_1\right), 
 \qquad 
\frac{1}{4} e^{s_1} k_7'\left(-s_1\right).  
\]
Then, by a simple replacement of $s_2$ by $-s_1$ in the rest of the terms 
we find that 
\begin{equation} \label{diffeqK6first}
\frac{1}{4} e^{s_1} k_7'\left(-s_1\right)-\frac{1}{4} k_6'\left(s_1\right) = 
\end{equation}
\begin{center}
\begin{math}
\frac{1}{120 s_1}\Big ( 16 \pi  s_1 G_2\left(s_1,-s_1\right)-15 s_1 k_{14}\left(0,s_1\right)+15 e^{s_1} s_1 k_{14}\left(-s_1,0\right)+15 s_1 k_{15}\left(0,s_1\right)-15 s_1 k_{15}\left(s_1,-s_1\right)-15 e^{s_1} s_1 k_{16}\left(-s_1,0\right)+15 s_1 k_{16}\left(s_1,-s_1\right)-60 s_1 G_1\left(s_1\right) k_5\left(-s_1\right)-60 e^{-s_1} s_1 G_1\left(s_1\right) k_5\left(s_1\right)+90 s_1 G_1\left(s_1\right) k_6\left(-s_1\right)+30 e^{-s_1} s_1 G_1\left(s_1\right) k_6\left(s_1\right)+30 s_1 G_1\left(s_1\right) k_7\left(-s_1\right)+90 e^{-s_1} s_1 G_1\left(s_1\right) k_7\left(s_1\right)-60 k_5\left(-s_1\right)-60 e^{-s_1} k_5\left(s_1\right)+90 k_6\left(-s_1\right)+30 e^{-s_1} k_6\left(s_1\right)+30 k_6\left(s_1\right)+30 e^{s_1} k_7\left(-s_1\right)+30 k_7\left(-s_1\right)+90 e^{-s_1} k_7\left(s_1\right)+120 k_5(0)-150 k_6(0)-150 k_7(0)  \Big ). 
\end{math}
\end{center}

\smallskip

In the second expression for $\widetilde K_6$ given by \eqref{basicK6eqnsecond},  
the terms that seemingly have singularity on $s_1 + s_2 = 0$ are 
\[
\frac{\left(e^{s_1+s_2}-1\right) k_6\left(s_1\right)}{4 \left(s_1+s_2\right)}, 
\qquad
-\frac{e^{s_1+s_2} \left(k_6\left(s_1\right)-k_6\left(-s_2\right)\right)}{4 \left(s_1+s_2\right)}, 
\qquad 
-\frac{e^{s_1} \left(k_7\left(-s_1\right)-k_7\left(s_2\right)\right)}{4 \left(s_1+s_2\right)},  
\]
which are exactly the same as the seemingly singular terms in the first expression. 
Since we already know the restriction of these terms to $s_1+s_2=0$, 
we can  derive the following identity by 
replacing $s_2$ by $-s_1$ in the remaining terms: 
\begin{equation} \label{diffeqK6second}
\frac{1}{4} e^{s_1} k_7'\left(-s_1\right)-\frac{1}{4} k_6'\left(s_1\right) = 
\end{equation}
\begin{center}
\begin{math}
\frac{1}{120 s_1} \Big ( 16 \pi  s_1 G_2\left(s_1,-s_1\right)+15 s_1 k_8\left(0,s_1\right)-15 s_1 k_8\left(s_1,-s_1\right)+15 s_1 k_9\left(0,s_1\right)-15 s_1 k_9\left(s_1,-s_1\right)-15 e^{s_1} s_1 k_{10}\left(-s_1,0\right)+15 s_1 k_{10}\left(s_1,-s_1\right)-15 e^{s_1} s_1 k_{11}\left(-s_1,0\right)+15 s_1 k_{11}\left(s_1,-s_1\right)-15 s_1 k_{12}\left(0,s_1\right)+15 e^{s_1} s_1 k_{12}\left(-s_1,0\right)-15 s_1 k_{13}\left(0,s_1\right)+15 e^{s_1} s_1 k_{13}\left(-s_1,0\right)-30 s_1 G_1\left(s_1\right) k_3\left(-s_1\right)-30 e^{-s_1} s_1 G_1\left(s_1\right) k_3\left(s_1\right)-60 s_1 G_1\left(s_1\right) k_4\left(-s_1\right)-60 e^{-s_1} s_1 G_1\left(s_1\right) k_4\left(s_1\right)+90 s_1 G_1\left(s_1\right) k_6\left(-s_1\right)+30 e^{-s_1} s_1 G_1\left(s_1\right) k_6\left(s_1\right)+30 s_1 G_1\left(s_1\right) k_7\left(-s_1\right)+90 e^{-s_1} s_1 G_1\left(s_1\right) k_7\left(s_1\right)-30 k_3\left(-s_1\right)-30 e^{-s_1} k_3\left(s_1\right)-60 k_4\left(-s_1\right)-60 e^{-s_1} k_4\left(s_1\right)+90 k_6\left(-s_1\right)+30 e^{-s_1} k_6\left(s_1\right)+30 k_6\left(s_1\right)+30 e^{s_1} k_7\left(-s_1\right)+30 k_7\left(-s_1\right)+90 e^{-s_1} k_7\left(s_1\right)+60 k_3(0)+120 k_4(0)-150 k_6(0)-150 k_7(0) \Big ). 
\end{math}
\end{center}

\smallskip

Since the left hand side of the equation \eqref{diffeqK6first} matches precisely 
with that of \eqref{diffeqK6second}, we have the following functional relation. 

\begin{corollary} \label{FnRlnK6DiffEqns}
We have 
\begin{center}
\begin{math}
-s_1 k_8\left(0,s_1\right)+s_1 k_8\left(s_1,-s_1\right)-s_1 k_9\left(0,s_1\right)+s_1 k_9\left(s_1,-s_1\right)+e^{s_1} s_1 k_{10}\left(-s_1,0\right)-s_1 k_{10}\left(s_1,-s_1\right)+e^{s_1} s_1 k_{11}\left(-s_1,0\right)-s_1 k_{11}\left(s_1,-s_1\right)+s_1 k_{12}\left(0,s_1\right)-e^{s_1} s_1 k_{12}\left(-s_1,0\right)+s_1 k_{13}\left(0,s_1\right)-e^{s_1} s_1 k_{13}\left(-s_1,0\right)-s_1 k_{14}\left(0,s_1\right)+e^{s_1} s_1 k_{14}\left(-s_1,0\right)+s_1 k_{15}\left(0,s_1\right)-s_1 k_{15}\left(s_1,-s_1\right)-e^{s_1} s_1 k_{16}\left(-s_1,0\right)+s_1 k_{16}\left(s_1,-s_1\right)+2 s_1 G_1\left(s_1\right) k_3\left(-s_1\right)+2 e^{-s_1} s_1 G_1\left(s_1\right) k_3\left(s_1\right)+4 s_1 G_1\left(s_1\right) k_4\left(-s_1\right)+4 e^{-s_1} s_1 G_1\left(s_1\right) k_4\left(s_1\right)-4 s_1 G_1\left(s_1\right) k_5\left(-s_1\right)-4 e^{-s_1} s_1 G_1\left(s_1\right) k_5\left(s_1\right)+2 k_3\left(-s_1\right)+2 e^{-s_1} k_3\left(s_1\right)+4 k_4\left(-s_1\right)+4 e^{-s_1} k_4\left(s_1\right)-4 k_5\left(-s_1\right)-4 e^{-s_1} k_5\left(s_1\right)-4 k_3(0)-8 k_4(0)+8 k_5(0) 
\end{math}
\end{center}
\[
=0.
\]
\end{corollary}

\subsubsection{Differential equations derived from the two expressions for $\widetilde K_7$} 
Like the situation for $\widetilde K_6$, in \eqref{Gradientofa4} the 
operator associated with $\widetilde K_7$ acts on two different kinds of elements of 
$\CNT$ that are not the same up to switching $\done$ and $\dtwo$. We explained in 
\S\,\ref{basicK7} that we 
get two basic functional relations in this case as well. In the first expression  for 
$\widetilde K_7$ given by \eqref{basicK7eqnfirst}, which is associated with the action of the corresponding operator 
on $\done^3(\ell) \cdot \done(\ell)$, the terms that do not specialize trivially to $s_1+s_2 = 0$ 
are
\[
-\frac{e^{s_1} \left(k_6\left(-s_1\right)-k_6\left(s_2\right)\right)}{4 \left(s_1+s_2\right)}, 
\qquad 
\frac{\left(e^{s_1+s_2}-1\right) k_7\left(s_1\right)}{4 \left(s_1+s_2\right)}, 
\qquad
-\frac{e^{s_1+s_2} \left(k_7\left(s_1\right)-k_7\left(-s_2\right)\right)}{4 \left(s_1+s_2\right)}. 
\]
By taking a limit in each case with the aid of the estimation $s_1 + s_2 = \vep \sim 0$, 
we find that the above terms respectively specialize on $s_1+s_2 =0$ to 
\[
\frac{1}{4} e^{s_1} k_6'\left(-s_1\right), 
\qquad 
\frac{1}{4} k_7\left(s_1\right), 
\qquad 
-\frac{1}{4} k_7'\left(s_1\right). 
\]
We can then replace $s_2$ by $-s_1$ in the remaining terms and 
obtain the following identity:
\begin{equation} \label{diffeqK7first}
\frac{1}{4} e^{s_1} k_6'\left(-s_1\right)-\frac{1}{4} k_7'\left(s_1\right) = 
\end{equation}
\begin{center}
\begin{math}
\frac{1}{120s_1} \Big ( 16 \pi  s_1 G_2\left(s_1,-s_1\right)-15 s_1 k_{14}\left(0,s_1\right)+15 s_1 k_{14}\left(s_1,-s_1\right)+15 e^{s_1} s_1 k_{15}\left(-s_1,0\right)-15 s_1 k_{15}\left(s_1,-s_1\right)+15 s_1 k_{16}\left(0,s_1\right)-15 e^{s_1} s_1 k_{16}\left(-s_1,0\right)+30 s_1 G_1\left(s_1\right) k_6\left(-s_1\right)+30 e^{-s_1} s_1 G_1\left(s_1\right) k_7\left(s_1\right)-60 e^{s_1} k_5\left(-s_1\right)-60 k_5\left(s_1\right)+30 e^{s_1} k_6\left(-s_1\right)+30 k_6\left(-s_1\right)+90 k_6\left(s_1\right)+90 e^{s_1} k_7\left(-s_1\right)+30 e^{-s_1} k_7\left(s_1\right)+30 k_7\left(s_1\right)+120 k_5(0)-150 k_6(0)-150 k_7(0) \Big ). 
\end{math}
\end{center}

\smallskip

In the second expression given by \eqref{basicK7eqnsecond} for $\widetilde K_7$, which corresponds to   
$\done\dtwo^2(\ell) \cdot \done(\ell)$, the terms that do not behave nicely 
under the replacement of $s_2$ by $-s_1$ are
\[
-\frac{e^{s_1} \left(k_6\left(-s_1\right)-k_6\left(s_2\right)\right)}{4 \left(s_1+s_2\right)}, 
\qquad 
\frac{\left(e^{s_1+s_2}-1\right) k_7\left(s_1\right)}{4 \left(s_1+s_2\right)}, 
\qquad 
-\frac{e^{s_1+s_2} \left(k_7\left(s_1\right)-k_7\left(-s_2\right)\right)}{4 \left(s_1+s_2\right)}, 
\] 
which are exactly the same as those of the first expression for $\widetilde K_7$. 
Since we just worked out the restriction of these terms to $s_1+s_2 = 0$, 
we replace $s_2$ by $-s_1$ in the rest of the terms and obtain the following 
identity: 

\begin{equation} \label{diffeqK7second}
\frac{1}{4} e^{s_1} k_6'\left(-s_1\right)-\frac{1}{4} k_7'\left(s_1\right) = 
\end{equation}
\begin{center}
\begin{math}
\frac{1}{120s_1}\Big ( 16 \pi  s_1 G_2\left(s_1,-s_1\right)+15 e^{s_1} s_1 k_8\left(-s_1,0\right)-15 s_1 k_8\left(s_1,-s_1\right)+15 e^{s_1} s_1 k_9\left(-s_1,0\right)-15 s_1 k_9\left(s_1,-s_1\right)+15 s_1 k_{10}\left(0,s_1\right)-15 e^{s_1} s_1 k_{10}\left(-s_1,0\right)+15 s_1 k_{11}\left(0,s_1\right)-15 e^{s_1} s_1 k_{11}\left(-s_1,0\right)-15 s_1 k_{12}\left(0,s_1\right)+15 s_1 k_{12}\left(s_1,-s_1\right)-15 s_1 k_{13}\left(0,s_1\right)+15 s_1 k_{13}\left(s_1,-s_1\right)+30 s_1 G_1\left(s_1\right) k_6\left(-s_1\right)+30 e^{-s_1} s_1 G_1\left(s_1\right) k_7\left(s_1\right)-30 e^{s_1} k_3\left(-s_1\right)-30 k_3\left(s_1\right)-60 e^{s_1} k_4\left(-s_1\right)-60 k_4\left(s_1\right)+30 e^{s_1} k_6\left(-s_1\right)+30 k_6\left(-s_1\right)+90 k_6\left(s_1\right)+90 e^{s_1} k_7\left(-s_1\right)+30 e^{-s_1} k_7\left(s_1\right)+30 k_7\left(s_1\right)+60 k_3(0)+120 k_4(0)-150 k_6(0)-150 k_7(0) \Big ). 
\end{math}
\end{center}

\smallskip

Since the left hand sides of \eqref{diffeqK7first} and \eqref{diffeqK7second} are identical, 
we find the following functional relation by setting equal their right hand sides. 

\begin{corollary} \label{FnRlnK7DiffEqns}
We have 
\begin{center}
\begin{math}
-e^{s_1} s_1 k_8\left(-s_1,0\right)+s_1 k_8\left(s_1,-s_1\right)-e^{s_1} s_1 k_9\left(-s_1,0\right)+s_1 k_9\left(s_1,-s_1\right)-s_1 k_{10}\left(0,s_1\right)+e^{s_1} s_1 k_{10}\left(-s_1,0\right)-s_1 k_{11}\left(0,s_1\right)+e^{s_1} s_1 k_{11}\left(-s_1,0\right)+s_1 k_{12}\left(0,s_1\right)-s_1 k_{12}\left(s_1,-s_1\right)+s_1 k_{13}\left(0,s_1\right)-s_1 k_{13}\left(s_1,-s_1\right)-s_1 k_{14}\left(0,s_1\right)+s_1 k_{14}\left(s_1,-s_1\right)+e^{s_1} s_1 k_{15}\left(-s_1,0\right)-s_1 k_{15}\left(s_1,-s_1\right)+s_1 k_{16}\left(0,s_1\right)-e^{s_1} s_1 k_{16}\left(-s_1,0\right)+2 e^{s_1} k_3\left(-s_1\right)+2 k_3\left(s_1\right)+4 e^{s_1} k_4\left(-s_1\right)+4 k_4\left(s_1\right)-4 e^{s_1} k_5\left(-s_1\right)-4 k_5\left(s_1\right)-4 k_3(0)-8 k_4(0)+8 k_5(0) 
\end{math}
\end{center}
\[
= 0. 
\]
\end{corollary}

\subsubsection{Differential equation derived from the expression for $\widetilde K_8$} Let us start from the basic equation 
for the function 
\[
\widetilde K_8(s_1, s_2, s_3) = \frac{1}{2^3} \frac{\sinh((s_1+s_2+s_3)/2)}{(s_1+s_2+s_3)/2} K_8(s_1, s_2, s_3), 
\]
which is given by \eqref{basicK8eqn}. We specialize the equation to $s_1+s_2+s_3 = 0$ as follows. 
On the left hand side,  $\widetilde K_8(s_1, s_2, s_3)$ specializes to $\frac{1}{8} k_8(s_1, s_2)$, 
and on the right hand side, most of the terms behave nicely except the following 
four terms. The first term that has $s_1+s_2+s_3$ in the denominator is 
\[
\frac{\left(e^{s_1+s_2+s_3}-1\right) k_8\left(s_1,s_2\right)}{8 \left(s_1+s_2+s_3\right)},  
\]
which contributes $\frac{1}{8} k_8\left(s_1,s_2\right)$ when specialized to $s_1+s_2+s_3=0$. 
The second term is 
\[
-\frac{e^{s_1+s_2+s_3} \left(k_8\left(s_1,s_2\right)-k_8\left(-s_2-s_3,s_2\right)\right)}{8 \left(s_1+s_2+s_3\right)}, 
\]
which under the estimation $s_1+s_2+s_3 = \vep \sim 0$ turns to 
\[
-\frac{e^{\vep} \left(k_8\left(s_1,s_2\right)-k_8\left(s_1 - \vep,s_2\right)\right)}{8 \vep } 
\to \frac{-1}{8} \ddsone k_8(s_1, s_2). 
\]
The third term is 
\[
-\frac{e^{s_1} \left(k_{11}\left(s_2,-s_1-s_2\right)-k_{11}\left(s_2,s_3\right)\right)}{8 \left(s_1+s_2+s_3\right)}, 
\]
which under the above estimate specializes to 
\[
-\frac{e^{s_1} \left(k_{11}\left(s_2,-s_1-s_2 \right)-k_{11}\left(s_2, -s_1-s_2+\vep \right)\right)}{8 \vep } 
\to 
\frac{1}{8} e^{s_1} \ddstwo k_{11}\left(s_2,-s_1-s_2\right). 
\]
The fourth term is 
\[
-\frac{e^{s_1+s_2} \left(k_{12}\left(-s_1-s_2,s_1\right)-k_{12}\left(s_3,-s_2-s_3\right)\right)}{8 \left(s_1+s_2+s_3\right)}, 
\]
which, under the estimate $s_1+s_2+s_3 = \vep \sim 0$, turns to 
\[
-e^{s_1+s_2} \frac{
k_{12}(-s_1-s_2,s_1) - k_{12}(-s_1-s_2,s_1-\vep)
 }{8 \vep}
 \]
 \[
 -e^{s_1+s_2} \frac{
k_{12}(-s_1-s_2, s_1-\vep)-k_{12}(-s_1-s_2+\vep,s_1-\vep)
 }{8 \vep}
\]
\[
\to -\frac{1}{8} e^{s_1+s_2} \left(\ddstwo k_{12}\left(-s_1-s_2,s_1\right)- \ddsone k_{12}\left(-s_1-s_2,s_1\right)\right) 
\]
\[
=  -\frac{1}{8} e^{s_1+s_2} \left( (\ddstwo - \ddsone) k_{12}\left(-s_1-s_2,s_1\right) \right ). 
\]

\smallskip
Putting together the above bad behaving terms, one can see that they add up to the following expression: 
\[
\widetilde K_{8, \,\textnormal{s}}(s_1, s_2, s_3) = 
\]
\[
\frac{1}{8 \left(s_1+s_2+s_3\right)} \Big (  -k_8\left(s_1,s_2\right)+e^{s_1+s_2+s_3} k_8\left(-s_2-s_3,s_2\right)-e^{s_1} k_{11}\left(s_2,-s_1-s_2\right)\]
\[+e^{s_1} k_{11}\left(s_2,s_3\right)-e^{s_1+s_2} k_{12}\left(-s_1-s_2,s_1\right)+e^{s_1+s_2} k_{12}\left(s_3,-s_2-s_3\right) \Big ).
\]
Therefore, by a simple replacement of $s_3$ by $-s_1-s_2$ in the rest of the terms, we obtain the following identity:
\begin{equation} \label{diffeqK8}
\frac{1}{8} e^{s_1} \ddstwo k_{11}{}\left(s_2,-s_1-s_2\right)
-\frac{1}{8} e^{s_1+s_2} \ddstwo k_{12}{}\left(-s_1-s_2,s_1\right)
\end{equation}
\[
-\frac{1}{8} \ddsone k_8{} \left(s_1,s_2\right)+\frac{1}{8} e^{s_1+s_2} \ddsone k_{12}{}\left(-s_1-s_2,s_1\right) = 
\]
\begin{center}
\begin{math}
-\Big ( \widetilde K_{8}(s_1, s_2, s_3) -   \widetilde K_{8, \,\textnormal{s}}(s_1, s_2, s_3)
\Big )\bigm|_{s_3 = -s_1-s_2}. 
\end{math}
\end{center}

\subsubsection{Differential equation derived from the expression for $\widetilde K_9$} 
Similarly to the previous case, when we 
specialize the basic equation  \eqref{basicK9eqn} for $\widetilde K_9$ to $s_1+s_2+s_3=0$, 
the left hand side gives $\frac{1}{8}k_9(s_1, s_2)$. On the right hand side, the following 
are the terms that have $s_1+s_2+s_3$ in their denominators. 
We write each of these terms accompanied with the limit that is calculated by using 
the estimate $s_1+s_2+s_3 = \vep \sim 0:$ 
\[
\frac{\left(e^{s_1+s_2+s_3}-1\right) k_9\left(s_1,s_2\right)}{8 \left(s_1+s_2+s_3\right)} 
\to 
\frac{1}{8} k_9\left(s_1,s_2\right), 
\] 
\[
-\frac{e^{s_1+s_2+s_3} \left(k_9\left(s_1,s_2\right)-k_9\left(-s_2-s_3,s_2\right)\right)}{8 \left(s_1+s_2+s_3\right)} 
\to 
-\frac{1}{8} \ddsone k_9{} \left(s_1,s_2\right), 
\]
\[
-\frac{e^{s_1} \left(k_{10}\left(s_2,-s_1-s_2\right)-k_{10}\left(s_2,s_3\right)\right)}{8 \left(s_1+s_2+s_3\right)} 
\to 
\frac{1}{8} e^{s_1} \ddstwo k_{10}{} \left(s_2,-s_1-s_2\right),  
\]
\[
-\frac{e^{s_1+s_2} \left(k_{13}\left(-s_1-s_2,s_1\right)-k_{13}\left(s_3,-s_2-s_3\right)\right)}{8 \left(s_1+s_2+s_3\right)} 
\to 
\]
\[
-\frac{1}{8} e^{s_1+s_2} \left(\ddstwo k_{13}{} \left(-s_1-s_2,s_1\right)- \ddsone k_{13}{} \left(-s_1-s_2,s_1\right)\right) 
\]
\[
= -\frac{1}{8} e^{s_1+s_2} \left ( (\ddstwo - \ddsone)  k_{13}{} \left(-s_1-s_2,s_1\right) \right ). 
\]

\smallskip

One can see that the above bad behaving terms add up to the following expression: 
\[
\widetilde K_{9, \,\textnormal{s}}(s_1, s_2, s_3) = 
\]
\[
\frac{1}{8 \left(s_1+s_2+s_3\right)} \Big (  -k_9\left(s_1,s_2\right)+e^{s_1+s_2+s_3} k_9\left(-s_2-s_3,s_2\right)-e^{s_1} k_{10}\left(s_2,-s_1-s_2\right)
 \]
\[
+e^{s_1} k_{10}\left(s_2,s_3\right)-e^{s_1+s_2} k_{13}\left(-s_1-s_2,s_1\right)+e^{s_1+s_2} k_{13}\left(s_3,-s_2-s_3\right) \Big ). 
\]
Then, by simply replacing $s_3$ by $-s_1-s_2$ in the remaining terms, we obtain the following 
identity:
\begin{equation} \label{diffeqK9}
\frac{1}{8} e^{s_1} \ddstwo k_{10}{} \left(s_2,-s_1-s_2\right)-\frac{1}{8} e^{s_1+s_2} \ddstwo k_{13}{} \left(-s_1-s_2,s_1\right)
\end{equation}
\[
-\frac{1}{8} \ddsone k_9{} \left(s_1,s_2\right)+\frac{1}{8} e^{s_1+s_2} \ddsone k_{13}{}\left(-s_1-s_2,s_1\right) = 
\] 
\begin{center}
\begin{math}
- \Big ( \widetilde K_{9}(s_1, s_2, s_3) -   \widetilde K_{9, \,\textnormal{s}}(s_1, s_2, s_3)
\Big )\bigm|_{s_3 = -s_1-s_2}. 
\end{math}
\end{center}

\subsubsection{Differential equation derived from the expression for $\widetilde K_{10}$} 
In the expression given by  \eqref{basicK10eqn} for 
$\widetilde K_{10}$, the following are the terms that have $s_1+s_2+s_3$
in their denominators,  
which are written along with their restrictions to $s_1+s_2+s_3 = 0 $ that 
are calculated by using the estimate $s_1+s_2+s_3=\vep \sim 0:$ 
\[
-\frac{e^{s_1+s_2} \left(k_9\left(-s_1-s_2,s_1\right)-k_9\left(s_3,-s_2-s_3\right)\right)}{8 \left(s_1+s_2+s_3\right)}
\to
\]
\[
-\frac{1}{8} e^{s_1+s_2} \left( \ddstwo k_9{}\left(-s_1-s_2,s_1\right)-\ddsone k_9{}\left(-s_1-s_2,s_1\right)\right)
\]
\[
= -\frac{1}{8} e^{s_1+s_2} \left ( (\ddstwo - \ddsone ) k_9{}\left(-s_1-s_2,s_1\right) \right ), 
\]
\[
\frac{\left(e^{s_1+s_2+s_3}-1\right) k_{10}\left(s_1,s_2\right)}{8 \left(s_1+s_2+s_3\right)} 
\to 
\frac{1}{8} k_{10}\left(s_1,s_2\right), 
\]
\[
-\frac{e^{s_1+s_2+s_3} \left(k_{10}\left(s_1,s_2\right)-k_{10}\left(-s_2-s_3,s_2\right)\right)}{8 \left(s_1+s_2+s_3\right)}
\to 
-\frac{1}{8} \ddsone k_{10}{}\left(s_1,s_2\right), 
\]
\[
-\frac{e^{s_1} \left(k_{13}\left(s_2,-s_1-s_2\right)-k_{13}\left(s_2,s_3\right)\right)}{8 \left(s_1+s_2+s_3\right)} 
\to 
\frac{1}{8} e^{s_1} \ddstwo k_{13}{}\left(s_2,-s_1-s_2\right). 
\]

\smallskip

A simple calculation shows that the above bad behaving terms add up to the following expression: 
\[
\widetilde K_{10, \,\textnormal{s}}(s_1, s_2, s_3) = 
\]
\[
\frac{1}{8 \left(s_1+s_2+s_3\right)} \Big ( -e^{s_1+s_2} k_9\left(-s_1-s_2,s_1\right)+e^{s_1+s_2} k_9\left(s_3,-s_2-s_3\right) 
\]
\[-k_{10}\left(s_1,s_2\right)+e^{s_1+s_2+s_3} k_{10}\left(-s_2-s_3,s_2\right)-e^{s_1} k_{13}\left(s_2,-s_1-s_2\right)+e^{s_1} k_{13}\left(s_2,s_3\right)  \Big ). 
\]
Since $\widetilde K_{10}(s_1, s_2, s_3)$ specializes to $\frac{1}{8} k_{10}(s_1, s_2)$, and in the 
remaining terms we can simply replace $s_3$ by $-s_1-s_2$, we find the following 
identity:
\begin{equation} \label{diffeqK10}
-\frac{1}{8} e^{s_1+s_2} \ddstwo k_9{}\left(-s_1-s_2,s_1\right)+\frac{1}{8} e^{s_1} \ddstwo k_{13}{}\left(s_2,-s_1-s_2\right)
\end{equation}
\[
+\frac{1}{8} e^{s_1+s_2} \ddsone k_9{} \left(-s_1-s_2,s_1\right)-\frac{1}{8} \ddsone k_{10}{}\left(s_1,s_2\right) = 
\] 
\begin{center}
\begin{math}
- \Big ( \widetilde K_{10}(s_1, s_2, s_3) -   \widetilde K_{10, \,\textnormal{s}}(s_1, s_2, s_3)
\Big )\bigm|_{s_3 = -s_1-s_2}. 
\end{math}
\end{center}

\subsubsection{Differential equation derived from the expression for $\widetilde K_{11}$} 
We consider the basic equation 
\eqref{basicK11eqn} for $\widetilde K_{11}$ and specialize it to $s_1+s_2+s_3 =0$, 
in which $\widetilde K_{11}(s_1, s_2, s_3)$ turns to $\frac{1}{8} k_{11}(s_1, s_2)$ 
and we treat the other side of the equation as follows. As in the previous cases, 
most of the terms on the other side behave nicely under the replacement of 
$s_3$ by $-s_1-s_2$, except the following terms which are written along with 
their restrictions to $s_1+s_2+s_3=0$ that are calculated by using the 
estimate $s_1+s_2+s_3=\vep \sim 0:$
\[
-\frac{e^{s_1+s_2} \left(k_8\left(-s_1-s_2,s_1\right)-k_8\left(s_3,-s_2-s_3\right)\right)}{8 \left(s_1+s_2+s_3\right)} 
\to
\]
\[
-\frac{1}{8} e^{s_1+s_2} \left(  \ddstwo k_8{}\left(-s_1-s_2,s_1\right)- \ddsone k_8{}\left(-s_1-s_2,s_1\right)\right)
\]
\[
=-\frac{1}{8} e^{s_1+s_2} \left ( (  \ddstwo - \ddsone ) k_8{}\left(-s_1-s_2,s_1\right) \right ),
\]
\[
\frac{\left(e^{s_1+s_2+s_3}-1\right) k_{11}\left(s_1,s_2\right)}{8 \left(s_1+s_2+s_3\right)}
\to 
\frac{1}{8} k_{11}\left(s_1,s_2\right), 
\]
\[
-\frac{e^{s_1+s_2+s_3} \left(k_{11}\left(s_1,s_2\right)-k_{11}\left(-s_2-s_3,s_2\right)\right)}{8 \left(s_1+s_2+s_3\right)}
\to 
-\frac{1}{8} \ddsone k_{11}{}\left(s_1,s_2\right), 
\]
\[
-\frac{e^{s_1} \left(k_{12}\left(s_2,-s_1-s_2\right)-k_{12}\left(s_2,s_3\right)\right)}{8 \left(s_1+s_2+s_3\right)}
\to 
\frac{1}{8} e^{s_1} \ddstwo k_{12}{}\left(s_2,-s_1-s_2\right). 
\]

\smallskip
Putting together the above bad behaving terms, one can see that their sum is equal to: 
\[
\widetilde K_{11, \,\textnormal{s}}(s_1, s_2, s_3) = 
\]
\[
\frac{1}{8 \left(s_1+s_2+s_3\right)}\Big (  -e^{s_1+s_2} k_8\left(-s_1-s_2,s_1\right)+e^{s_1+s_2} k_8\left(s_3,-s_2-s_3\right)
\]
\[
-k_{11}\left(s_1,s_2\right)+e^{s_1+s_2+s_3} k_{11}\left(-s_2-s_3,s_2\right)-e^{s_1} k_{12}\left(s_2,-s_1-s_2\right)+e^{s_1} k_{12}\left(s_2,s_3\right)  \Big ). 
\]
Therefore we find the following identity: 
\begin{equation} \label{diffeqK11}
-\frac{1}{8} e^{s_1+s_2} \ddstwo k_8{}\left(-s_1-s_2,s_1\right)+\frac{1}{8} e^{s_1} \ddstwo k_{12}{}\left(s_2,-s_1-s_2\right)
\end{equation}
\[
+\frac{1}{8} e^{s_1+s_2} \ddsone k_8{} \left(-s_1-s_2,s_1\right)-\frac{1}{8} \ddsone k_{11}{}\left(s_1,s_2\right)=
\]
\begin{center}
\begin{math}
- \Big ( \widetilde K_{11}(s_1, s_2, s_3) -   \widetilde K_{11, \,\textnormal{s}}(s_1, s_2, s_3)
\Big )\bigm|_{s_3 = -s_1-s_2}. 
\end{math}
\end{center}

\subsubsection{Differential equation derived from the expression for $\widetilde K_{12}$} 
We consider the basic equation  \eqref{basicK12eqn} for $\widetilde K_{12}$, 
whose left hand side specializes to $\frac{1}{8} k_{12}(s_1, s_2)$, and we treat the 
other side as follows. The following are the terms on the right hand side of 
this equation that have $s_1+s_2+s_3$ 
in their denominators, which, by using the estimate $s_1+s_2+s_3 = \vep \sim 0$, 
their behavior on $s_1+s_2+s_3=0$ can be determined: 
\[
-\frac{e^{s_1} \left(k_8\left(s_2,-s_1-s_2\right)-k_8\left(s_2,s_3\right)\right)}{8 \left(s_1+s_2+s_3\right)} 
\to 
\frac{1}{8} e^{s_1} \ddstwo k_8{}\left(s_2,-s_1-s_2\right), 
\]
\[
-\frac{e^{s_1+s_2} \left(k_{11}\left(-s_1-s_2,s_1\right)-k_{11}\left(s_3,-s_2-s_3\right)\right)}{8 \left(s_1+s_2+s_3\right)}
\to
\]
\[
-\frac{1}{8} e^{s_1+s_2} \left(\ddstwo k_{11}{} \left(-s_1-s_2,s_1\right)- \ddsone k_{11}\left(-s_1-s_2,s_1\right)\right) 
\]
\[
=-\frac{1}{8} e^{s_1+s_2} \left (  (\ddstwo - \ddsone) k_{11}{} \left(-s_1-s_2,s_1\right) \right ),
\]
\[
\frac{\left(e^{s_1+s_2+s_3}-1\right) k_{12}\left(s_1,s_2\right)}{8 \left(s_1+s_2+s_3\right)} 
\to 
\frac{1}{8} k_{12}\left(s_1,s_2\right), 
\]
\[
-\frac{e^{s_1+s_2+s_3} \left(k_{12}\left(s_1,s_2\right)-k_{12}\left(-s_2-s_3,s_2\right)\right)}{8 \left(s_1+s_2+s_3\right)}
\to 
-\frac{1}{8} \ddsone k_{12}{}\left(s_1,s_2\right). 
\]

\smallskip
The above bad behaving terms add up to the following expression
\[
\widetilde K_{12, \,\textnormal{s}}(s_1, s_2, s_3) = 
\]
\[
\frac{1}{8 \left(s_1+s_2+s_3\right)} \Big (   -e^{s_1} k_8\left(s_2,-s_1-s_2\right)+e^{s_1} k_8\left(s_2,s_3\right) 
-e^{s_1+s_2} k_{11}\left(-s_1-s_2,s_1\right)
\]
\[
+e^{s_1+s_2} k_{11}\left(s_3,-s_2-s_3\right)-k_{12}\left(s_1,s_2\right)+e^{s_1+s_2+s_3} k_{12}\left(-s_2-s_3,s_2\right) \Big ). 
\] 
In the rest of the terms we can simply replace $s_3$ by $-s_1-s_2$ and obtain the 
following identity:
\begin{equation} \label{diffeqK12}
\frac{1}{8} e^{s_1} \ddstwo k_8{}\left(s_2,-s_1-s_2\right)-\frac{1}{8} e^{s_1+s_2} \ddstwo k_{11}{}\left(-s_1-s_2,s_1\right)
\end{equation}
\[
+\frac{1}{8} e^{s_1+s_2} \ddsone k_{11}{}\left(-s_1-s_2,s_1\right)-\frac{1}{8} \ddsone k_{12}{} \left(s_1,s_2\right) = 
\]
\begin{center}
\begin{math}
- \Big ( \widetilde K_{12}(s_1, s_2, s_3) -   \widetilde K_{12, \,\textnormal{s}}(s_1, s_2, s_3)
\Big )\bigm|_{s_3 = -s_1-s_2}.   
\end{math}
\end{center}

\subsubsection{Differential equation derived from the expression for $\widetilde K_{13}$} 
We use the basic identity given by \eqref{basicK13eqn}, in which 
$\widetilde K_{13}(s_1, s_2, s_3)$ specializes to $\frac{1}{8} k_{13}(s_1, s_2)$ 
on $s_1+s_2+s_3 =0$. On the right 
side of this identity, the following are the terms that have $s_1+s_2+s_3$ in their denominators, 
which we specialize to $s_1+s_2+s_3 =0$ by using the estimate $s_1+s_2+s_3 = \vep \sim 0:$
\[
-\frac{e^{s_1} \left(k_9\left(s_2,-s_1-s_2\right)-k_9\left(s_2,s_3\right)\right)}{8 \left(s_1+s_2+s_3\right)}
\to 
\frac{1}{8} e^{s_1} \ddstwo k_9{}\left(s_2,-s_1-s_2\right), 
\]
\[
-\frac{e^{s_1+s_2} \left(k_{10}\left(-s_1-s_2,s_1\right)-k_{10}\left(s_3,-s_2-s_3\right)\right)}{8 \left(s_1+s_2+s_3\right)}
\to 
\]
\[
-\frac{1}{8} e^{s_1+s_2} \left( \ddstwo k_{10}{}\left(-s_1-s_2,s_1\right)- \ddsone k_{10}{}\left(-s_1-s_2,s_1\right)\right)
\]
\[
= -\frac{1}{8} e^{s_1+s_2} \left ( ( \ddstwo - \ddsone ) k_{10}{}\left(-s_1-s_2,s_1\right) \right ), 
\]
\[
\frac{\left(e^{s_1+s_2+s_3}-1\right) k_{13}\left(s_1,s_2\right)}{8 \left(s_1+s_2+s_3\right)} 
\to 
\frac{1}{8} k_{13}\left(s_1,s_2\right), 
\]
\[
-\frac{e^{s_1+s_2+s_3} \left(k_{13}\left(s_1,s_2\right)-k_{13}\left(-s_2-s_3,s_2\right)\right)}{8 \left(s_1+s_2+s_3\right)}
\to
-\frac{1}{8} \ddsone k_{13}{}\left(s_1,s_2\right). 
\]

\smallskip
Adding up the above bad behaving terms, one can see that their sum is the following expression: 
\[
\widetilde K_{13, \,\textnormal{s}}(s_1, s_2, s_3) = 
\]
\[
\frac{1}{8 \left(s_1+s_2+s_3\right)} \Big (  -e^{s_1} k_9\left(s_2,-s_1-s_2\right)+e^{s_1} k_9\left(s_2,s_3\right)-e^{s_1+s_2} k_{10}\left(-s_1-s_2,s_1\right)
\]
\[
+e^{s_1+s_2} k_{10}\left(s_3,-s_2-s_3\right)-k_{13}\left(s_1,s_2\right)+e^{s_1+s_2+s_3} k_{13}\left(-s_2-s_3,s_2\right) \Big ). 
\]
Then, by a simple replacement of $s_3$ by $-s_1-s_2$ in the rest of the terms,  
we find that 
\begin{equation} \label{diffeqK13}
\frac{1}{8} e^{s_1} \ddstwo k_9{}\left(s_2,-s_1-s_2\right)-\frac{1}{8} e^{s_1+s_2} \ddstwo k_{10}{}\left(-s_1-s_2,s_1\right)
\end{equation}
\[
+\frac{1}{8} e^{s_1+s_2} \ddsone k_{10}{}\left(-s_1-s_2,s_1\right)-\frac{1}{8} \ddsone k_{13}{}\left(s_1,s_2\right) = 
\]
\begin{center}
\begin{math}
- \Big ( \widetilde K_{13}(s_1, s_2, s_3) -   \widetilde K_{13, \,\textnormal{s}}(s_1, s_2, s_3)
\Big )\bigm|_{s_3 = -s_1-s_2}.  
\end{math}
\end{center}

\subsubsection{Differential equation derived from the expression for $\widetilde K_{14}$} 
In the basic equation written in \eqref{basicK14eqn}, 
$\widetilde K_{14}(s_1, s_2, s_3)$ specializes to $\frac{1}{8} k_{14}(s_1, s_2)$ on $s_1+s_2+s_3=0$, 
and on the right side we specialize  by replacing $s_3$ by $-s_1-s_2$ 
in the terms that behave nicely, and the following are the terms that have 
$s_1+s_2+s_3$ in their denominators, for which we calculate their restriction to 
$s_1+s_2+s_3=0$ by using the estimate $s_1+s_2+s_3=\vep \sim 0:$
\[
\frac{\left(e^{s_1+s_2+s_3}-1\right) k_{14}\left(s_1,s_2\right)}{8 \left(s_1+s_2+s_3\right)}
\to
\frac{1}{8} k_{14}\left(s_1,s_2\right), 
\]
\[
-\frac{e^{s_1+s_2+s_3} \left(k_{14}\left(s_1,s_2\right)-k_{14}\left(-s_2-s_3,s_2\right)\right)}{8 \left(s_1+s_2+s_3\right)}
\to 
-\frac{1}{8} \ddsone k_{14}{}\left(s_1,s_2\right), 
\]
\[
-\frac{e^{s_1} \left(k_{15}\left(s_2,-s_1-s_2\right)-k_{15}\left(s_2,s_3\right)\right)}{8 \left(s_1+s_2+s_3\right)}
\to 
\frac{1}{8} e^{s_1} \ddstwo k_{15}{}\left(s_2,-s_1-s_2\right), 
\]
\[
-\frac{e^{s_1+s_2} \left(k_{16}\left(-s_1-s_2,s_1\right)-k_{16}\left(s_3,-s_2-s_3\right)\right)}{8 \left(s_1+s_2+s_3\right)}
\to 
\]
\[
-\frac{1}{8} e^{s_1+s_2} \left( \ddstwo k_{16}{}\left(-s_1-s_2,s_1\right)- \ddsone k_{16}{}\left(-s_1-s_2,s_1\right)\right) 
\]
\[
= -\frac{1}{8} e^{s_1+s_2} \left ( ( \ddstwo - \ddsone ) k_{16}{}\left(-s_1-s_2,s_1\right) \right ). 
\]

\smallskip 

The above bad behaving terms simply add up to the following expression: 
\[
\widetilde K_{14, \,\textnormal{s}}(s_1, s_2, s_3) = 
\]
\[
\frac{1}{8 \left(s_1+s_2+s_3\right)} \Big (  -k_{14}\left(s_1,s_2\right)+e^{s_1+s_2+s_3} k_{14}\left(-s_2-s_3,s_2\right)-e^{s_1} k_{15}\left(s_2,-s_1-s_2\right)
\]
\[
+e^{s_1} k_{15}\left(s_2,s_3\right)-e^{s_1+s_2} k_{16}\left(-s_1-s_2,s_1\right)+e^{s_1+s_2} k_{16}\left(s_3,-s_2-s_3\right)  \Big ).
\]
Therefore we obtain the following identity: 
\begin{equation} \label{diffeqK14}
\frac{1}{8} e^{s_1} \ddstwo k_{15}{}\left(s_2,-s_1-s_2\right)-\frac{1}{8} e^{s_1+s_2} \ddstwo k_{16}{}\left(-s_1-s_2,s_1\right)
\end{equation}
\[
-\frac{1}{8} \ddsone k_{14}{}\left(s_1,s_2\right)+\frac{1}{8} e^{s_1+s_2} \ddsone k_{16}{}\left(-s_1-s_2,s_1\right) = 
\]
\begin{center}
\begin{math}
- \Big ( \widetilde K_{14}(s_1, s_2, s_3) -   \widetilde K_{14, \,\textnormal{s}}(s_1, s_2, s_3)
\Big )\bigm|_{s_3 = -s_1-s_2}. 
\end{math}
\end{center}

\subsubsection{Differential equation derived from the expression for $\widetilde K_{15}$} 
We start from the basic equation  \eqref{basicK15eqn} for 
$\widetilde K_{15}$ and specialize it to $s_1+s_2+s_3 = 0$. In this process, 
$\widetilde K_{15}(s_1, s_2, s_3)$ specializes to $\frac{1}{8}k_{15}(s_1, s_2)$, and 
except for the following terms, the rest of the terms behave nicely under the replacement 
of $s_3$ by $-s_1-s_2$. The terms that need special  care with the aid 
of the estimate $s_1+s_2+s_3 = \vep \sim 0$ are 
\[
-\frac{e^{s_1+s_2} \left(k_{14}\left(-s_1-s_2,s_1\right)-k_{14}\left(s_3,-s_2-s_3\right)\right)}{8 \left(s_1+s_2+s_3\right)}
\to 
\]
\[
-\frac{1}{8} e^{s_1+s_2} \left( \ddstwo k_{14}{}\left(-s_1-s_2,s_1\right)- \ddsone k_{14}{}\left(-s_1-s_2,s_1\right)\right)
\]
\[
= -\frac{1}{8} e^{s_1+s_2} \left ( ( \ddstwo - \ddsone) k_{14}{}\left(-s_1-s_2,s_1\right) \right ), 
\]
\[
\frac{\left(e^{s_1+s_2+s_3}-1\right) k_{15}\left(s_1,s_2\right)}{8 \left(s_1+s_2+s_3\right)} 
\to 
\frac{1}{8} k_{15}\left(s_1,s_2\right), 
\]
\[
-\frac{e^{s_1+s_2+s_3} \left(k_{15}\left(s_1,s_2\right)-k_{15}\left(-s_2-s_3,s_2\right)\right)}{8 \left(s_1+s_2+s_3\right)}
\to 
-\frac{1}{8}  \ddsone k_{15}{} \left(s_1,s_2\right), 
\]
\[
-\frac{e^{s_1} \left(k_{16}\left(s_2,-s_1-s_2\right)-k_{16}\left(s_2,s_3\right)\right)}{8 \left(s_1+s_2+s_3\right)}
\to 
\frac{1}{8} \ddstwo e^{s_1} k_{16}{}\left(s_2,-s_1-s_2\right). 
\]

\smallskip 

We put the bad behaving terms together and see that they add up to the following expression: 
\[
\widetilde K_{15, \,\textnormal{s}}(s_1, s_2, s_3) = 
\]
\[
\frac{1}{8 \left(s_1+s_2+s_3\right)} \Big ( -e^{s_1+s_2} k_{14}\left(-s_1-s_2,s_1\right)+e^{s_1+s_2} k_{14}\left(s_3,-s_2-s_3\right)-k_{15}\left(s_1,s_2\right)\]
\[
+e^{s_1+s_2+s_3} k_{15}\left(-s_2-s_3,s_2\right)-e^{s_1} k_{16}\left(s_2,-s_1-s_2\right)+e^{s_1} k_{16}\left(s_2,s_3\right) \Big ). 
\]
Thus, we find that the following identity holds: 
\begin{equation} \label{diffeqK15}
-\frac{1}{8} e^{s_1+s_2}  \ddstwo k_{14}{}\left(-s_1-s_2,s_1\right)+\frac{1}{8} e^{s_1} \ddstwo k_{16}{}\left(s_2,-s_1-s_2\right)
\end{equation}
\[
+\frac{1}{8} e^{s_1+s_2} \ddsone k_{14}{} \left(-s_1-s_2,s_1\right)-\frac{1}{8} \ddsone k_{15}{} \left(s_1,s_2\right) = 
\]
\begin{center}
\begin{math}
- \Big ( \widetilde K_{15}(s_1, s_2, s_3) -   \widetilde K_{15, \,\textnormal{s}}(s_1, s_2, s_3)
\Big )\bigm|_{s_3 = -s_1-s_2}.   
\end{math}
\end{center}

\subsubsection{Differential equation derived from the expression for $\widetilde K_{16}$} 
When we specialize the basic equation 
\eqref{basicK16eqn} for $\widetilde K_{16}$ to $s_1+s_2+s_3 =0$, similarly to the previous cases, 
$\widetilde K_{16}(s_1, s_2, s_3)$ turns to $\frac{1}{8}k_{16}(s_1, s_2)$, and 
we have terms that behave nicely under the replacement of $s_3$ by $-s_1-s_2$, 
while the following are the terms that need to be specialized by using the estimate 
$s_1+s_2+s_3 = \vep \sim 0$: 
\[
-\frac{e^{s_1} \left(k_{14}\left(s_2,-s_1-s_2\right)-k_{14}\left(s_2,s_3\right)\right)}{8 \left(s_1+s_2+s_3\right)} 
\to 
\frac{1}{8} e^{s_1} \ddstwo k_{14}{}\left(s_2,-s_1-s_2\right), 
\]
\[
-\frac{e^{s_1+s_2} \left(k_{15}\left(-s_1-s_2,s_1\right)-k_{15}\left(s_3,-s_2-s_3\right)\right)}{8 \left(s_1+s_2+s_3\right)} 
\to 
\]
\[
-\frac{1}{8} e^{s_1+s_2} \left( \ddstwo k_{15}{}\left(-s_1-s_2,s_1\right)- \ddsone k_{15}{}\left(-s_1-s_2,s_1\right)\right) 
\]
\[
=-\frac{1}{8} e^{s_1+s_2}  \left ( ( \ddstwo - \ddsone) k_{15}{}\left(-s_1-s_2,s_1\right) \right ),
\]
\[
\frac{\left(e^{s_1+s_2+s_3}-1\right) k_{16}\left(s_1,s_2\right)}{8 \left(s_1+s_2+s_3\right)} 
\to
\frac{1}{8} k_{16}\left(s_1,s_2\right), 
\]
\[
-\frac{e^{s_1+s_2+s_3} \left(k_{16}\left(s_1,s_2\right)-k_{16}\left(-s_2-s_3,s_2\right)\right)}{8 \left(s_1+s_2+s_3\right)} 
\to 
-\frac{1}{8} \ddsone k_{16}{}\left(s_1,s_2\right). 
\]

\smallskip

The sum of the above bad behaving terms is the following expression: 
\[
\widetilde K_{16, \,\textnormal{s}}(s_1, s_2, s_3) = 
\]
\[
\frac{1}{8 \left(s_1+s_2+s_3\right)} \Big ( -e^{s_1} k_{14}\left(s_2,-s_1-s_2\right)+e^{s_1} k_{14}\left(s_2,s_3\right)-e^{s_1+s_2} k_{15}\left(-s_1-s_2,s_1\right)
\]
\[
+e^{s_1+s_2} k_{15}\left(s_3,-s_2-s_3\right)-k_{16}\left(s_1,s_2\right)+e^{s_1+s_2+s_3} k_{16}\left(-s_2-s_3,s_2\right) \Big ).  
\]
Therefore, the result of specializing the basic equation for $\widetilde K_{16}$  to $s_1+s_2+s_3=0$ is the following identity: 
\begin{equation} \label{diffeqK16}
\frac{1}{8} e^{s_1} \ddstwo k_{14}{}\left(s_2,-s_1-s_2\right)-\frac{1}{8} e^{s_1+s_2} \ddstwo k_{15}{}\left(-s_1-s_2,s_1\right)
\end{equation}
\[
+\frac{1}{8} e^{s_1+s_2} \ddsone k_{15}{} \left(-s_1-s_2,s_1\right)-\frac{1}{8} \ddsone k_{16}{}\left(s_1,s_2\right) =
\]
\begin{center}
\begin{math}
- \Big ( \widetilde K_{16}(s_1, s_2, s_3) -   \widetilde K_{16, \,\textnormal{s}}(s_1, s_2, s_3)
\Big )\bigm|_{s_3 = -s_1-s_2}.  
\end{math}
\end{center}

\begin{remark} \label{denK8-K16diffsys}
After replacing the lengthy expressions for $\widetilde K_{8},$ $\dots,$ 
$\widetilde K_{16},$ given by \eqref{basicK8eqn}, \eqref{basicK9eqn} -- \eqref{basicK16eqn} in 
the right hand side of the equations \eqref{diffeqK8} -- \eqref{diffeqK16}, and after taking the 
common denominator and putting the terms together, in each case, one finds a fraction whose 
numerator is a finite difference expression of the involved functions and its denominator is the expression 
\[
s_1 s_2 \left(s_1+s_2\right). 
\]
\end{remark}

\subsubsection{Differential equation derived from the expression for $\widetilde K_{17}$} In order to derive a differential equation 
from the basic equation  \eqref{basicK17eqn} for $\widetilde K_{17}$, we specialize it 
to $s_1+s_2+s_3+s_4 = 0$. On the 
left side of the equation $\widetilde K_{17}(s_1, s_2, s_3, s_4)$ turns to $\frac{1}{16} k_{17}(s_1, s_2, s_3)$, 
and on the other hand we have terms that behave nicely under the replacement of $s_4$ by 
$-s_1-s_2-s_3$, while we also have the following five terms which can be specialized with the aid 
of the estimate $s_1+s_2+s_3+s_4 = \vep \sim 0$. The first one is    
\begin{eqnarray*}
\frac{\left(e^{s_1+s_2+s_3+s_4}-1\right) k_{17}\left(s_1,s_2,s_3\right)}{16 \left(s_1+s_2+s_3+s_4\right)} 
\to
\frac{1}{16} k_{17}\left(s_1,s_2,s_3\right).  
\end{eqnarray*}
The second term that needs to be treated with the above estimate is
\begin{eqnarray*}
&&
-\frac{e^{s_1+s_2} \left(k_{17}\left(s_3,-s_1-s_2-s_3,s_1\right)-k_{17}\left(s_3,s_4,-s_2-s_3-s_4\right)\right)}{16 \left(s_1+s_2+s_3+s_4\right)} \\
&=&
-\frac{e^{s_1+s_2} \left(k_{17}\left(s_3,-s_1-s_2-s_3,s_1\right)-k_{17}\left(s_3,-s_1-s_2-s_3+\vep ,s_1-\vep \right)\right)}{16 \vep }\\
&\to& 
-\frac{1}{16} e^{s_1+s_2} \left( \ddsthree k_{17}{}\left(s_3,-s_1-s_2-s_3,s_1\right)- \ddstwo k_{17}{} \left(s_3,-s_1-s_2-s_3,s_1\right)\right) \\
&=&-\frac{1}{16} e^{s_1+s_2} \left(  ( \ddsthree - \ddstwo) k_{17}{}\left(s_3,-s_1-s_2-s_3,s_1\right) \right).  
\end{eqnarray*}
For the third term of this kind we have
\begin{eqnarray*}
&& -\frac{e^{s_1+s_2+s_3+s_4} \left(k_{17}\left(s_1,s_2,s_3\right)-k_{17}\left(-s_2-s_3-s_4,s_2,s_3\right)\right)}{16 \left(s_1+s_2+s_3+s_4\right)} \\
&=& 
-\frac{e^{\vep } \left(k_{17}\left(s_1,s_2,s_3\right)-k_{17}\left(s_1-\vep ,s_2,s_3\right)\right)}{16 \vep }
\to
-\frac{1}{16} \ddsone k_{17}{}\left(s_1,s_2,s_3\right). 
\end{eqnarray*}
For the fourth term we can write 
\begin{eqnarray*}
&&
-\frac{e^{s_1} \left(k_{19}\left(s_2,s_3,-s_1-s_2-s_3\right)-k_{19}\left(s_2,s_3,s_4\right)\right)}{16 \left(s_1+s_2+s_3+s_4\right)} 
\\
&=& 
-\frac{e^{s_1} \left(k_{19}\left(s_2,s_3,-s_1-s_2-s_3\right)-k_{19}\left(s_2,s_3,-s_1-s_2-s_3+\vep \right)\right)}{16 \vep }\\
&\to&
 \frac{1}{16} e^{s_1} \ddsthree k_{19}{} \left(s_2,s_3,-s_1-s_2-s_3\right).  
\end{eqnarray*}
Finally, we treat the last term that needs to be specialized to $s_1+s_2+s_3+s_4=0$ by computing a limit: 
\begin{eqnarray*}
&&-\frac{e^{s_1+s_2+s_3} \left(k_{19}\left(-s_1-s_2-s_3,s_1,s_2\right)-k_{19}\left(s_4,-s_2-s_3-s_4,s_2\right)\right)}{16 \left(s_1+s_2+s_3+s_4\right)} 
\\
&=& -\frac{e^{s_1+s_2+s_3} \left(k_{19}\left(-s_1-s_2-s_3,s_1,s_2\right)-k_{19}\left(-s_1-s_2-s_3+\vep ,s_1-\vep ,s_2\right)\right)}{16 \vep } \\
&\to& 
-\frac{1}{16} e^{s_1+s_2+s_3} \left( \ddstwo k_{19}{} \left(-s_1-s_2-s_3,s_1,s_2\right)- \ddsone k_{19}{}\left(-s_1-s_2-s_3,s_1,s_2\right)\right) \\
&=&
-\frac{1}{16} e^{s_1+s_2+s_3} \left( ( \ddstwo - \ddsone ) k_{19}{} \left(-s_1-s_2-s_3,s_1,s_2\right) \right). 
\end{eqnarray*}

\smallskip

Putting together the above bad behaving terms, one can see that they add up to the following expression:
\[
 \widetilde K_{17, \,\textnormal{s}}(s_1, s_2, s_3, s_4) =
\]
\[
\frac{1}{16 \left(s_1+s_2+s_3+s_4\right)} \Big (   
-k_{17}\left(s_1,s_2,s_3\right)-e^{s_1+s_2} k_{17}\left(s_3,-s_1-s_2-s_3,s_1\right)
\]
\[
+e^{s_1+s_2} k_{17}\left(s_3,s_4,-s_2-s_3-s_4\right)
+e^{s_1+s_2+s_3+s_4} k_{17}\left(-s_2-s_3-s_4,s_2,s_3\right)
\]
\[
-e^{s_1} k_{19}\left(s_2,s_3,-s_1-s_2-s_3\right)+e^{s_1} k_{19}\left(s_2,s_3,s_4\right)
\]
\[
-e^{s_1+s_2+s_3} k_{19}\left(-s_1-s_2-s_3,s_1,s_2\right)+e^{s_1+s_2+s_3} k_{19}\left(s_4,-s_2-s_3-s_4,s_2\right) \Big )
\]
Therefore, the above specialization of the basic equation for $\widetilde K_{17}$ yields the following identity: 
\begin{equation} \label{diffeqK17}
-\frac{1}{16} e^{s_1+s_2} \ddsthree k_{17}{}\left(s_3,-s_1-s_2-s_3,s_1\right)+\frac{1}{16} e^{s_1} \ddsthree k_{19}{}\left(s_2,s_3,-s_1-s_2-s_3\right)
\end{equation}
\[
+\frac{1}{16} e^{s_1+s_2} \ddstwo k_{17}{}\left(s_3,-s_1-s_2-s_3,s_1\right)-\frac{1}{16} e^{s_1+s_2+s_3} \ddstwo k_{19}{}\left(-s_1-s_2-s_3,s_1,s_2\right)
\]
\[
-\frac{1}{16} \ddsone k_{17}{}\left(s_1,s_2,s_3\right)+\frac{1}{16} e^{s_1+s_2+s_3} \ddsone k_{19}{}\left(-s_1-s_2-s_3,s_1,s_2\right) = 
\] 
\begin{center}
\begin{math}
- 
\Big ( \widetilde K_{17}(s_1, s_2, s_3, s_4) -   \widetilde K_{17, \,\textnormal{s}}(s_1, s_2, s_3, s_4)
\Big )\bigm|_{s_4 = -s_1-s_2-s_3}. 
\end{math}
\end{center}

\subsubsection{Differential equation derived from the expression for $\widetilde K_{18}$} 
We  specialize the basic equation \eqref{basicK18eqn} 
for $\widetilde K_{18}$ to $s_1+s_2+s_3+s_4=0$, which on the left side turns to 
$\frac{1}{16}k_{18}(s_1, s_2, s_3)$, and the other side is treated as follows. 
The following are the terms on the right side that have $s_1+s_2+s_3+s_4$ in their 
denominators and need to be treated with the estimate $s_1+s_2+s_3+s_4 = \vep \sim 0:$  
\[
\frac{\left(e^{s_1+s_2+s_3+s_4}-1\right) k_{18}\left(s_1,s_2,s_3\right)}{16 \left(s_1+s_2+s_3+s_4\right)} 
\to 
\frac{1}{16} k_{18}\left(s_1,s_2,s_3\right), 
\]
\[
-\frac{e^{s_1} \left(k_{18}\left(s_2,s_3,-s_1-s_2-s_3\right)-k_{18}\left(s_2,s_3,s_4\right)\right)}{16 \left(s_1+s_2+s_3+s_4\right)}
\to 
\]
\[
\frac{1}{16} e^{s_1} \ddsthree k_{18}{}\left(s_2,s_3,-s_1-s_2-s_3\right), 
\]
\[
-\frac{e^{s_1+s_2} \left(k_{18}\left(s_3,-s_1-s_2-s_3,s_1\right)-k_{18}\left(s_3,s_4,-s_2-s_3-s_4\right)\right)}{16 \left(s_1+s_2+s_3+s_4\right)} \to
\]
\[
-\frac{1}{16} e^{s_1+s_2} \left( \ddsthree k_{18}{}\left(s_3,-s_1-s_2-s_3,s_1\right)- \ddstwo k_{18}{}\left(s_3,-s_1-s_2-s_3,s_1\right)\right), 
\]
\[
-\frac{e^{s_1+s_2+s_3+s_4} \left(k_{18}\left(s_1,s_2,s_3\right)-k_{18}\left(-s_2-s_3-s_4,s_2,s_3\right)\right)}{16 \left(s_1+s_2+s_3+s_4\right)} 
\to 
-\frac{1}{16} \ddsone k_{18}{}\left(s_1,s_2,s_3\right), 
\]
\[
-\frac{e^{s_1+s_2+s_3} \left(k_{18}\left(-s_1-s_2-s_3,s_1,s_2\right)-k_{18}\left(s_4,-s_2-s_3-s_4,s_2\right)\right)}{16 \left(s_1+s_2+s_3+s_4\right)} \to
\]
\[
-\frac{1}{16} e^{s_1+s_2+s_3}  \left( \ddstwo k_{18}{}\left(-s_1-s_2-s_3,s_1,s_2\right)- \ddsone k_{18}{}\left(-s_1-s_2-s_3,s_1,s_2\right)\right). 
\]

\smallskip

The above bad behaving terms add up to the following expression: 
\[
 \widetilde K_{18, \,\textnormal{s}}(s_1, s_2, s_3, s_4) =
\]
\[
\frac{1}{16 \left(s_1+s_2+s_3+s_4\right)} \Big ( -k_{18}\left(s_1,s_2,s_3\right)-e^{s_1} k_{18}\left(s_2,s_3,-s_1-s_2-s_3\right)
\]
\[
+e^{s_1} k_{18}\left(s_2,s_3,s_4\right)-e^{s_1+s_2+s_3} k_{18}\left(-s_1-s_2-s_3,s_1,s_2\right)
\]
\[
-e^{s_1+s_2} k_{18}\left(s_3,-s_1-s_2-s_3,s_1\right) 
+e^{s_1+s_2} k_{18}\left(s_3,s_4,-s_2-s_3-s_4\right)
\]
\[
+e^{s_1+s_2+s_3+s_4} k_{18}\left(-s_2-s_3-s_4,s_2,s_3\right)
+e^{s_1+s_2+s_3} k_{18}\left(s_4,-s_2-s_3-s_4,s_2\right)  \Big ). 
\]
Note that in the rest of the terms we can directly replace $s_4$ by $-s_1-s_2-s_3$. Therefore, the result of this analysis, which allows us to specialize the basic equation \eqref{basicK18eqn} to $s_1+s_2+s_3+s_4=0$, 
is the following identity: 
\begin{equation} \label{diffeqK18}
\frac{1}{16} e^{s_1} \ddsthree k_{18}{}\left(s_2,s_3,-s_1-s_2-s_3\right)-\frac{1}{16} e^{s_1+s_2} \ddsthree k_{18}{}\left(s_3,-s_1-s_2-s_3,s_1\right)
\end{equation}
\[
-\frac{1}{16} e^{s_1+s_2+s_3} \ddstwo k_{18}{}\left(-s_1-s_2-s_3,s_1,s_2\right)+\frac{1}{16} e^{s_1+s_2} 
\ddstwo  k_{18}{}\left(s_3,-s_1-s_2-s_3,s_1\right)
\]
\[
-\frac{1}{16} \ddsone k_{18}{}\left(s_1,s_2,s_3\right)+\frac{1}{16} e^{s_1+s_2+s_3} \ddsone k_{18}{}\left(-s_1-s_2-s_3,s_1,s_2\right) =
\]
\begin{center}
\begin{math}
- 
\Big ( \widetilde K_{18}(s_1, s_2, s_3, s_4) -   \widetilde K_{18, \,\textnormal{s}}(s_1, s_2, s_3, s_4)
\Big )\bigm|_{s_4 = -s_1-s_2-s_3}.  
\end{math}
\end{center}

\subsubsection{Differential equation derived from the expression for 
$\widetilde K_{19}$} In a similar way, we can specialize the basic 
equation written in \eqref{basicK19eqn} for $\widetilde K_{19}$ to $s_1+s_2+s_3+s_4=0$. On the left side we get 
$\frac{1}{16} k_{19}(s_1, s_2, s_3)$, and on the right side we replace $s_4$ by $-s_1-s_2-s_3$ 
in the terms that behave nicely with respect to this replacement, while 
the bad behaving terms, which are listed below, can be treated using the estimate $s_1+s_2+s_3+s_4 =\vep \sim 0$ 
as follows: 
\[
-\frac{e^{s_1} \left(k_{17}\left(s_2,s_3,-s_1-s_2-s_3\right)-k_{17}\left(s_2,s_3,s_4\right)\right)}{16 \left(s_1+s_2+s_3+s_4\right)}
\to 
\]
\[
\frac{1}{16} e^{s_1} \ddsthree k_{17}{}\left(s_2,s_3,-s_1-s_2-s_3\right), 
\]
\[
-\frac{e^{s_1+s_2+s_3} \left(k_{17}\left(-s_1-s_2-s_3,s_1,s_2\right)-k_{17}\left(s_4,-s_2-s_3-s_4,s_2\right)\right)}{16 \left(s_1+s_2+s_3+s_4\right)} 
\to 
\]
\[
-\frac{1}{16} e^{s_1+s_2+s_3} \left( \ddstwo k_{17}{}\left(-s_1-s_2-s_3,s_1,s_2\right)- \ddsone k_{17}{}\left(-s_1-s_2-s_3,s_1,s_2\right)\right), 
\]
\[
\frac{\left(e^{s_1+s_2+s_3+s_4}-1\right) k_{19}\left(s_1,s_2,s_3\right)}{16 \left(s_1+s_2+s_3+s_4\right)} 
\to 
\frac{1}{16} k_{19}\left(s_1,s_2,s_3\right), 
\]
\[
-\frac{e^{s_1+s_2} \left(k_{19}\left(s_3,-s_1-s_2-s_3,s_1\right)-k_{19}\left(s_3,s_4,-s_2-s_3-s_4\right)\right)}{16 \left(s_1+s_2+s_3+s_4\right)} \to 
\]
\[
-\frac{1}{16} e^{s_1+s_2}  \left(\ddsthree k_{19}{}\left(s_3,-s_1-s_2-s_3,s_1\right)- \ddstwo k_{19}{}\left(s_3,-s_1-s_2-s_3,s_1\right)\right), 
\]
\[
-\frac{e^{s_1+s_2+s_3+s_4} \left(k_{19}\left(s_1,s_2,s_3\right)-k_{19}\left(-s_2-s_3-s_4,s_2,s_3\right)\right)}{16 \left(s_1+s_2+s_3+s_4\right)} \to 
-\frac{1}{16} \ddsone k_{19}{}\left(s_1,s_2,s_3\right).
\]

\smallskip

One can see that the above bad behaving terms add up to the following expression: 
\[
 \widetilde K_{19, \,\textnormal{s}}(s_1, s_2, s_3, s_4) = 
\]
\[
\frac{1}{16 \left(s_1+s_2+s_3+s_4\right)} \Big (  -e^{s_1} k_{17}\left(s_2,s_3,-s_1-s_2-s_3\right)+e^{s_1} k_{17}\left(s_2,s_3,s_4\right)
\]
\[
-e^{s_1+s_2+s_3} k_{17}\left(-s_1-s_2-s_3,s_1,s_2\right)+e^{s_1+s_2+s_3} k_{17}\left(s_4,-s_2-s_3-s_4,s_2\right)
\]
\[
-k_{19}\left(s_1,s_2,s_3\right)-e^{s_1+s_2} k_{19}\left(s_3,-s_1-s_2-s_3,s_1\right)
\]
\[
+e^{s_1+s_2} k_{19}\left(s_3,s_4,-s_2-s_3-s_4\right)+e^{s_1+s_2+s_3+s_4} k_{19}\left(-s_2-s_3-s_4,s_2,s_3\right)  \Big ).  
\]
Therefore we obtain the following identity, in which $ \widetilde K_{19}(s_1, s_2, s_3, s_4)$ denotes the 
right side of the equation \eqref{basicK19eqn}, and the direct replacement of $s_4$ by $-s_1-s_2-s_3$ behaves 
nicely, since following the above discussion, we have apparently removed the bad behaving terms: 
\begin{equation} \label{diffeqK19}
\frac{1}{16} e^{s_1} \ddsthree k_{17}{}\left(s_2,s_3,-s_1-s_2-s_3\right)-\frac{1}{16} e^{s_1+s_2} \ddsthree k_{19}{}\left(s_3,-s_1-s_2-s_3,s_1\right)
\end{equation}
\[
-\frac{1}{16} e^{s_1+s_2+s_3} \ddstwo k_{17}{}\left(-s_1-s_2-s_3,s_1,s_2\right)
+\frac{1}{16} e^{s_1+s_2} \ddstwo k_{19}{}\left(s_3,-s_1-s_2-s_3,s_1\right)
\]
\[
+\frac{1}{16} e^{s_1+s_2+s_3} \ddsone k_{17}{}\left(-s_1-s_2-s_3,s_1,s_2\right)-\frac{1}{16} \ddsone k_{19}{}\left(s_1,s_2,s_3\right) = 
\]
\begin{center}
\begin{math}
- 
\Big ( \widetilde K_{19}(s_1, s_2, s_3, s_4) -   \widetilde K_{19, \,\textnormal{s}}(s_1, s_2, s_3, s_4)
\Big )\bigm|_{s_4 = -s_1-s_2-s_3}. 
\end{math}
\end{center}

\subsubsection{Differential equation derived from the expression for $\widetilde K_{20}$} Finally we 
specialize the basic identity given by \eqref{basicK20eqn} for $\widetilde K_{20}$ to 
$s_1+s_2+s_3+s_4=0$, and derive the last differential equation. 
On the left side of the basic identity, $\widetilde K_{20}(s_1, s_2, s_3, s_4)$ 
specializes to $\frac{1}{16}k_{20}(s_1, s_2, s_3)$, and on the right 
side we use the estimate $s_1+s_2+s_3+s_4 = \vep \sim 0$ to specialize 
the terms that do not behave nicely with respect to the replacement of 
$s_4$ by $-s_1-s_2-s_3$. In this respect, the bad behaving terms are:  
\[
\frac{\left(e^{s_1+s_2+s_3+s_4}-1\right) k_{20}\left(s_1,s_2,s_3\right)}{16 \left(s_1+s_2+s_3+s_4\right)} 
\to 
\frac{1}{16} k_{20}\left(s_1,s_2,s_3\right), 
\]
\[
-\frac{e^{s_1} \left(k_{20}\left(s_2,s_3,-s_1-s_2-s_3\right)-k_{20}\left(s_2,s_3,s_4\right)\right)}{16 \left(s_1+s_2+s_3+s_4\right)}
\to 
\]
\[
\frac{1}{16} e^{s_1} \ddsthree k_{20}{} \left(s_2,s_3,-s_1-s_2-s_3\right), 
\]
\[
-\frac{e^{s_1+s_2} \left(k_{20}\left(s_3,-s_1-s_2-s_3,s_1\right)-k_{20}\left(s_3,s_4,-s_2-s_3-s_4\right)\right)}{16 \left(s_1+s_2+s_3+s_4\right)} 
\to 
\]
\[
-\frac{1}{16} e^{s_1+s_2} \left( \ddsthree k_{20}{} \left(s_3,-s_1-s_2-s_3,s_1\right)- \ddstwo k_{20}{}\left(s_3,-s_1-s_2-s_3,s_1\right)\right), 
\]
\[
-\frac{e^{s_1+s_2+s_3+s_4} \left(k_{20}\left(s_1,s_2,s_3\right)-k_{20}\left(-s_2-s_3-s_4,s_2,s_3\right)\right)}{16 \left(s_1+s_2+s_3+s_4\right)} 
\to 
\]
\[
-\frac{1}{16} \ddsone k_{20}{}\left(s_1,s_2,s_3\right), 
\]
\[
-\frac{e^{s_1+s_2+s_3} \left(k_{20}\left(-s_1-s_2-s_3,s_1,s_2\right)-k_{20}\left(s_4,-s_2-s_3-s_4,s_2\right)\right)}{16 \left(s_1+s_2+s_3+s_4\right)} 
\to 
\]
\[
-\frac{1}{16} e^{s_1+s_2+s_3}  \left( \ddstwo k_{20}{}\left(-s_1-s_2-s_3,s_1,s_2\right)- \ddsone k_{20}{}\left(-s_1-s_2-s_3,s_1,s_2\right)\right). 
\]

\smallskip

Putting together the above bad behaving terms, one can see that their sum is the following expression: 
\[
 \widetilde K_{20, \,\textnormal{s}}(s_1, s_2, s_3, s_4) = 
\]
\[
\frac{1}{16 \left(s_1+s_2+s_3+s_4\right)} \Big ( -k_{20}\left(s_1,s_2,s_3\right)-e^{s_1} k_{20}\left(s_2,s_3,-s_1-s_2-s_3\right)
\]
\[
+e^{s_1} k_{20}\left(s_2,s_3,s_4\right)-e^{s_1+s_2+s_3} k_{20}\left(-s_1-s_2-s_3,s_1,s_2\right)
\]
\[
-e^{s_1+s_2} k_{20}\left(s_3,-s_1-s_2-s_3,s_1\right)+e^{s_1+s_2} k_{20}\left(s_3,s_4,-s_2-s_3-s_4\right)
\]
\[
+e^{s_1+s_2+s_3+s_4} k_{20}\left(-s_2-s_3-s_4,s_2,s_3\right)+e^{s_1+s_2+s_3} k_{20}\left(s_4,-s_2-s_3-s_4,s_2\right) \Big ).
\]
Therefore, this analysis yields the following identity, in which  $\widetilde K_{20}(s_1, s_2, s_3, s_4)$ 
denotes the right side of the basic identity \eqref{basicK20eqn}, and the direct replacement of $s_4$ by 
$-s_1-s_2-s_3$ behaves nicely since, following the above discussion, the terms that $s_1+s_2+s_3+s_4$ 
appears in their denominators are removed: 
\begin{equation} \label{diffeqK20}
\frac{1}{16} e^{s_1} \ddsthree k_{20}{} \left(s_2,s_3,-s_1-s_2-s_3\right)-\frac{1}{16} e^{s_1+s_2} \ddsthree k_{20}{} 
\left(s_3,-s_1-s_2-s_3,s_1\right) 
\end{equation}
\[
-\frac{1}{16} e^{s_1+s_2+s_3} \ddstwo k_{20}{} \left(-s_1-s_2-s_3,s_1,s_2\right)+\frac{1}{16} e^{s_1+s_2} \ddstwo k_{20}{}\left(s_3,-s_1-s_2-s_3,s_1\right)
\]
\[
-\frac{1}{16} \ddsone k_{20}{}\left(s_1,s_2,s_3\right)+\frac{1}{16} e^{s_1+s_2+s_3} \ddsone k_{20}{}\left(-s_1-s_2-s_3,s_1,s_2\right) = 
\]
\begin{center}
\begin{math}
- 
\Big ( \widetilde K_{20}(s_1, s_2, s_3, s_4) -   \widetilde K_{20, \,\textnormal{s}}(s_1, s_2, s_3, s_4)
\Big )\bigm|_{s_4 = -s_1-s_2-s_3}. 
\end{math}
\end{center}

\smallskip

In the equations \eqref{diffeqK3} -- \eqref{diffeqK7second}, we see $s_1$ as the denominator 
on the right hand side of each equation. As we explained in Remark \ref{denK8-K16diffsys}, for 
the equations \eqref{diffeqK8} -- \eqref{diffeqK16} 
the expression $s_1 s_2 (s_1+s_2)$ appears as the denominator on the right hand side 
of the equations. 
Since we have similarly used a concise way of writing the right hand side of each of the 
equations \eqref{diffeqK17} -- \eqref{diffeqK20}, their denominators are not written explicitly 
and it is important to elaborate in the following remark on this matter. 

\begin{remark} \label{denK17-K20diffsys}
After replacing the lengthy expressions for $\widetilde K_{17},$ $\dots,$ 
$\widetilde K_{20},$ given by \eqref{basicK17eqn}, \eqref{basicK18eqn} -- \eqref{basicK20eqn}, in the right hand side of the 
equations \eqref{diffeqK17} -- \eqref{diffeqK20}, and after taking the common denominator 
and putting the terms together, in each case, one finds that the denominator is the expression 
\[
s_1 s_2 \left(s_1+s_2\right) s_3 \left(s_2+s_3\right) \left(s_1+s_2+s_3\right). 
\]
\end{remark}

\subsection{Cyclic groups involved in the differential system}

In the differential equations \eqref{diffeqK3} -- \eqref{diffeqK7second} it is apparent 
that the action of $\mathbb{Z}/2 \mathbb{Z}$ on $\mathbb{R}$ given by 
\[
T_2(s_1) = -s_1, \qquad s_1 \in \R, 
\]  
is involved. Also on the right hand side of each of these equations, the expression 
in the denominator is $s_1$, which up to sign is invariant under the above 
transformation. We shall see shortly that a similar situation holds for the remaining 
differential equations in the differential system presented in Subsection \ref{differentialsystemderived}. 
First we present a result about a transformation property  of the expressions appearing on 
either side of the differential equations \eqref{diffeqK3} -- \eqref{diffeqK7second} under the above 
action of $\mathbb{Z}/2 \mathbb{Z}$. The left hand side of each of these equations, up to an 
overall factor of $\frac{1}{4}$,  
is of the form 
\begin{equation} \label{leftK3K7groupactionform}
- k'_{j_0}(s_1) + e^{s_1} k'_{j_1}(-s_1) = - k'_{j_0}(s_1)  + T_2( e^{-s_1} k'_{j_1})(s_1), 
\end{equation}
where 
\begin{eqnarray*}
\textrm{for the equation } \eqref{diffeqK3}:   &&\qquad (j_0, j_1)=(3, 3),  \\
\textrm{for the equation } \eqref{diffeqK4}:   &&\qquad (j_0, j_1) = (4, 4),\\
\textrm{for the equation } \eqref{diffeqK5}:   &&\qquad (j_0, j_1) = (5, 5),\\
\textrm{for the equation } \eqref{diffeqK6first}:   &&\qquad (j_0, j_1) = (6, 7),\\
\textrm{for the equation } \eqref{diffeqK6second}:   &&\qquad (j_0, j_1) = (6, 7),\\
\textrm{for the equation } \eqref{diffeqK7first}:   &&\qquad (j_0, j_1) = (7, 6), \\
\textrm{for the equation } \eqref{diffeqK7second}:   &&\qquad (j_0, j_1) = (7, 6). 
\end{eqnarray*}
Therefore we can divide the above differential equations to the following four groups: 
the first three groups respectively consist of  the equations \eqref{diffeqK3},  \eqref{diffeqK4} and 
 \eqref{diffeqK5} since they respectively involve only derivatives of $k_3, k_4$ and 
 $k_5$, and the fourth group consists of the equations \eqref{diffeqK6first} -- \eqref{diffeqK7second} 
 which involve the derivatives of $k_6$ and $k_7$. 
In fact it will also be useful to use the following notation $f_{j_0, j_1}(s_1)$  to rewrite the differential 
equations \eqref{diffeqK3} -- \eqref{diffeqK7second} in a concise form:
\begin{equation} \label{diffeqnK3K7consice}
- k'_{j_0}(s_1) + e^{s_1} k'_{j_1}(-s_1) = \frac{f_{j_0, j_1}(s_1)}{s_1}.  
\end{equation}
Apparently each $f_{j_0, j_1}(s_1)$ is given in terms of finite differences of  the restriction to $s_1+s_2=0$ 
of the functions $k_j$ appearing on the right hand sides of the equations, and note that 
using the equations \eqref{diffeqK6first} and \eqref{diffeqK6second}, we have two 
such expressions for  $f_{6, 7}(s_1)$. Similarly, in the case of the differential equations 
\eqref{diffeqK7first} and \eqref{diffeqK7second}, we obtain two finite difference expressions for 
$f_{7, 6}(s_1)$.

\smallskip

We can now start using the above group action to investigate invariance properties and symmetries 
of the expressions appearing in the differential equations. 
\begin{theorem} \label{k1tok7eqnsinvariance}
The expression 
\[
e^{-\frac{s_1}{2}} \big ( - (k'_{j_0}(s_1) + k'_{j_1}(s_1) )+ e^{s_1} \left( k'_{j_0}(-s_1)+ k'_{j_1}(-s_1) \right ) \big ), 
\]
is in the kernel of $1+T_2$ for any pair of integers $j_0, j_1$ in $\{3, 4, 5, 6, 7\}$. 
In particular, considering the differential equations \eqref{diffeqK3} -- \eqref{diffeqK7second} and 
the above discussion following the equation \eqref{leftK3K7groupactionform}, we have obtained 
expressions that are in terms of finite differences of the $k_j$ and are in the kernel of $1+T_2$: 
\begin{enumerate}
\item When $(j_0, j_1) = (3, 3)$.
\item When $(j_0, j_1) = (4, 4)$.
\item When $(j_0, j_1) = (5, 5)$.
\item When $(j_0, j_1) = (6, 7)$.
\end{enumerate}
 
\begin{proof} We multiply \eqref{leftK3K7groupactionform} by 
\[
\alpha_1(s_1):=e^{-\frac{s_1}{2}}, 
\]
which yields
\[
e^{-\frac{s_1}{2}} \left (- k'_{j_0}(s_1) + e^{s_1} k'_{j_1}(-s_1)  \right)= 
- (\alpha_2 \cdot k'_{j_0})(s_1) + T_2(  \alpha_2 \cdot k'_{j_1} )(s_1). 
\]
By switching $j_0$ and $j_1$ in this equation we obtain a new equation, 
and summing up the two equations we get 
\begin{eqnarray*}
&&e^{-\frac{s_1}{2}} \big ( - (k'_{j_0}(s_1) + k'_{j_1}(s_1) )+ e^{s_1} \left( k'_{j_0}(-s_1)+ k'_{j_1}(-s_1) \right ) \big ) \\
&=& - \left (\alpha_1 \cdot (k'_{j_0}+ k'_{j_1}) \right)(s_1) + T_2 \left(  \alpha_1 \cdot ( k'_{j_0}+k'_{j_1}) \right )(s_1) \\
&=& (-1 + T_2) \left (  \alpha_1 \cdot (k'_{j_0}+ k'_{j_1}) \right ) (s_1), 
\end{eqnarray*}
which is in the kernel of $1+T_2$ since $T_2^2=1. $
\end{proof}
\end{theorem}

We elaborate on the second part of the statement of the above theorem. Using the 
notation introduced by the equation \eqref{diffeqnK3K7consice} and the first part of the 
above theorem, we have proved that the following expressions are invariant under the 
transformation $T_2: s_1 \mapsto -s_1$ on $\mathbb{R}$:  
\[
s_1 e^{-\frac{s_1}{2}} \big ( - k'_{3}(s_1) + e^{s_1}  k'_{3}(-s_1) \big ) = e^{-\frac{s_1}{2}} f_{3, 3}(s_1),
\]
\[
s_ 1e^{-\frac{s_1}{2}} \big ( - k'_{4}(s_1) + e^{s_1}  k'_{4}(-s_1) \big ) = e^{-\frac{s_1}{2}} f_{4, 4}(s_1) ,
\]
\[
s_1 e^{-\frac{s_1}{2}} \big ( - k'_{5}(s_1) + e^{s_1}  k'_{5}(-s_1) \big ) = e^{-\frac{s_1}{2}} f_{5, 5}(s_1),
\]
\[
s_1 e^{-\frac{s_1}{2}} \big ( - (k'_{6}(s_1) + k'_{7}(s_1) )+ e^{s_1} \left( k'_{6}(-s_1)+ k'_{7}(-s_1) \right ) \big ) 
= e^{-\frac{s_1}{2}} \left( f_{6, 7}(s_1) + f_{7, 6}(s_1) \right). 
\]

\smallskip

We now wish to investigate invariance properties of expressions that can be associated with the 
differential equations \eqref{diffeqK8} -- \eqref{diffeqK16}. 
In these equations, where partial derivatives of the  functions $k_8(s_1, s_2),  \dots, k_{16}(s_1, s_2)$ 
are involved, two transformations are apparently present, namely:  
$$
(s_1,s_2)\mapsto (-s_1-s_2,s_1), \qquad  (s_1,s_2)\mapsto (s_2,-s_1-s_2). 
$$
They correspond to the matrices
$$
T_3 =
\left(
 \begin{array}{cc}
 -1 & -1 \\
 1 & 0 \\
\end{array}
\right), \qquad 
T_3^2 = \left(
\begin{array}{cc}
 0 & 1 \\
 -1 & -1 \\
\end{array}
\right), 
$$
and together with the identity matrix they form a cyclic group of order $3$. 
The partial differential system involves the values on the orbit of points for this group 
and hence these values should be considered simultaneously. Taking a close look on the 
left hand side of the equations \eqref{diffeqK8} -- \eqref{diffeqK16}, it is clear that, up to an 
overall factor of $\frac{1}{8}$,  
they are all of the form 
\begin{equation} \label{leftK8K16groupactionform}
 -\ddsone k_{j_0}{}\left(s_1,s_2\right) - e^{s_1+s_2} (\ddstwo - \ddsone )k_{j_1}{}\left(-s_1-s_2,s_1\right) + 
 e^{s_1} \ddstwo k_{j_2}{}\left(s_2,-s_1-s_2\right)
\end{equation}
\[
= -\ddsone k_{j_0}{}\left(s_1,s_2\right) 
- T_3( e^{-s_1} (\ddstwo - \ddsone )k_{j_1} )\left(s_1,s_2\right)
+ T_3^2(e^{-s_1-s_2} \ddstwo k_{j_2}) \left(s_1, s_2 \right), 
\]
where
\begin{eqnarray} \label{diffeqK8K17indices}
\textrm{for the equation } \eqref{diffeqK8}:   &&\qquad (j_0, j_1, j_2) = ( 8, 12, 11),   \nonumber \\
\textrm{for  the equation } \eqref{diffeqK9}:   &&\qquad (j_0, j_1, j_2) = ( 9, 13, 10), \nonumber \\
\textrm{for the equation } \eqref{diffeqK10}:   &&\qquad (j_0, j_1, j_2) = ( 10, 9, 13), \nonumber \\
\textrm{for the equation } \eqref{diffeqK11}:   &&\qquad (j_0, j_1, j_2) = ( 11, 8, 12), \nonumber \\
\textrm{for the equation } \eqref{diffeqK12}:   &&\qquad (j_0, j_1, j_2) = ( 12, 11, 8),  \\
\textrm{for the equation } \eqref{diffeqK13}:   &&\qquad (j_0, j_1, j_2) = ( 13, 10, 9), \nonumber \\
\textrm{for the equation } \eqref{diffeqK14}:   &&\qquad (j_0, j_1, j_2) = ( 14, 16, 15),\nonumber \\
\textrm{for the equation } \eqref{diffeqK15}:   &&\qquad (j_0, j_1, j_2) = ( 15, 14, 16), \nonumber \\
\textrm{for the equation } \eqref{diffeqK16}:   &&\qquad (j_0, j_1, j_2) = ( 16, 15, 14). \nonumber 
\end{eqnarray}

\smallskip

We note that while ranging from the equation \eqref{diffeqK8} to the equation 
\eqref{diffeqK16}, each $j_0, j_1, j_2$ attains each of the integers in $\{ 8, 9, \dots, 16 \}$ 
exactly once. Moreover, we can put the equations in the following three groups.  
The first group consists of the equations \eqref{diffeqK8}, \eqref{diffeqK11}, \eqref{diffeqK12}, 
which involve partial derivatives of the functions $k_8, k_{11}, k_{12}$. We can put the equations  
\eqref{diffeqK9}, \eqref{diffeqK10}, \eqref{diffeqK13} in the second group since 
they  involve partial derivatives of only $k_9, k_{10}, k_{13}$. The last group consists of 
the equations \eqref{diffeqK14}, \eqref{diffeqK15}, \eqref{diffeqK16}, which 
involve partial derivatives of $k_{14}, k_{15}, k_{16}$.

\smallskip

It will be useful to introduce the following notation $f_{j_0, j_1, j_2}(s_1, s_2)$ to rewrite the differential equations 
\eqref{diffeqK8} -- \eqref{diffeqK16} concisely: 
\begin{equation} \label{diffeqsK8K16concise}
 -\ddsone k_{j_0}{}\left(s_1,s_2\right) - e^{s_1+s_2} (\ddstwo - \ddsone )k_{j_1}{}\left(-s_1-s_2,s_1\right) + 
 e^{s_1} \ddstwo k_{j_2}{}\left(s_2,-s_1-s_2\right) 
 \end{equation}
 \[
 = 
 \frac{f_{j_0, j_1, j_2}(s_1, s_2)}{s_1 s_2 (s_1 + s_2)}, 
 \]
 where for each differential equation, $(j_0, j_1, j_2)$ is as in \eqref{diffeqK8K17indices}. 
 It is clear from the equations that each $f_{j_0, j_1, j_2}(s_1, s_2)$ is a finite difference 
 expression of restriction of the involved functions $k_j$ to $s_1+s_2+s_3=0$. 
Another important 
observation that we have made is that on the right hand side of each of the equations  \eqref{diffeqK8} -- \eqref{diffeqK16}, 
the denominator is the expression  $s_1 s_2 (s_1+s_2)$. We see in the following lemma that 
the above action of $\mathbb{Z}/3\mathbb{Z}$ leaves this expression invariant. The quadratic form written in 
this lemma can be obtained by averaging the standard quadratic form using the action of the cyclic group. 

\begin{lemma} \label{quadratic2Lemma}
The following expressions are invariants of the action of $\mathbb{Z}/3\mathbb{Z}$ on $\mathbb{R}^2$ given 
by the above matrices: 
\begin{enumerate}
   \item The positive definite quadratic form 
    \[
    Q_2(s) = s_1^2 + s_2^2 + s_1 s_2, \qquad s = (s_1, s_2) \in \mathbb{R}^2. 
    \] 
    \item The product $s_1 s_2 \left(s_1+s_2\right)$. 
\end{enumerate}
\end{lemma}

\begin{figure}
\begin{center}
\includegraphics[scale=0.5]{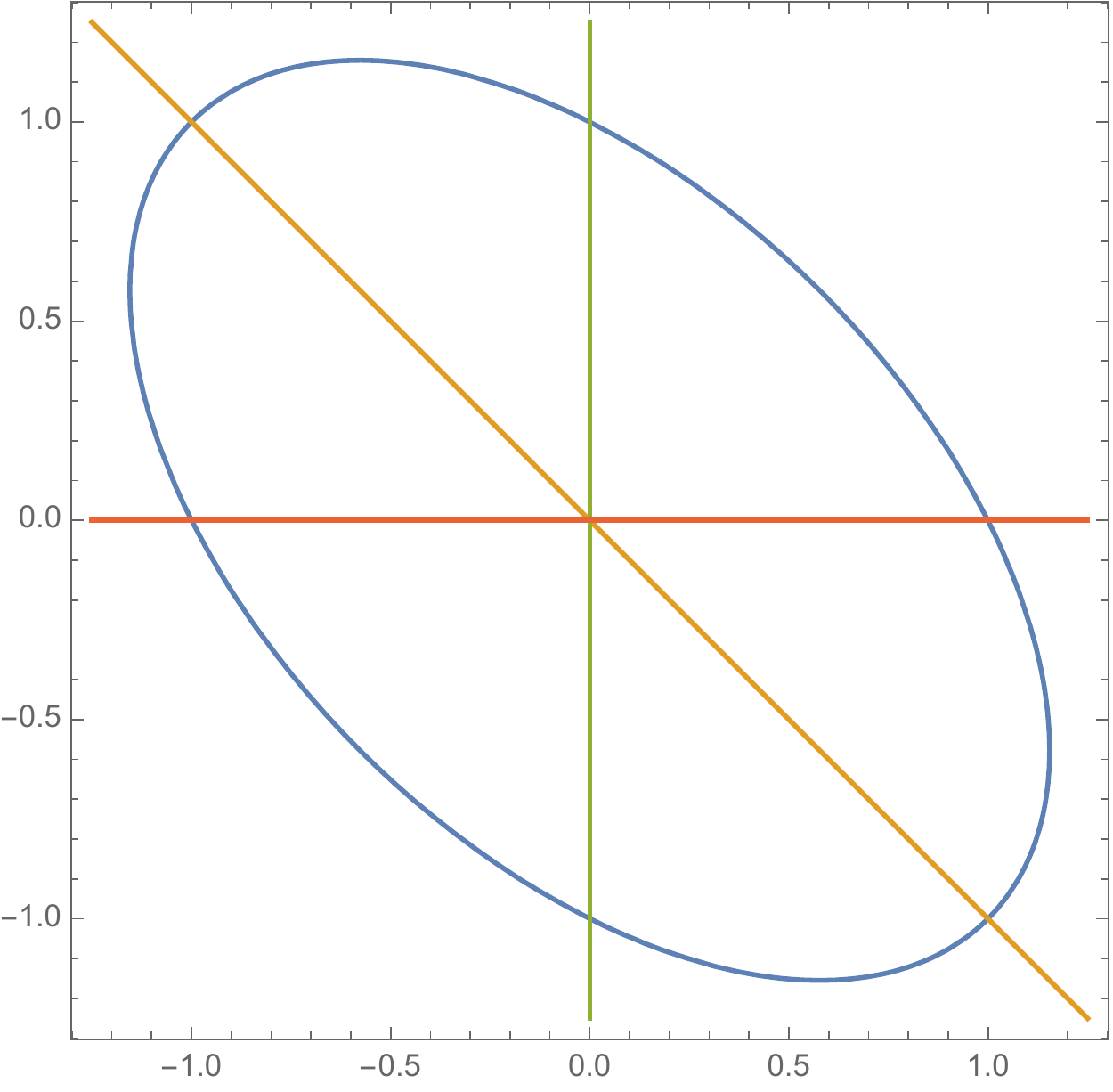}
\end{center}
\caption{The blue ellipse is the unit ball for the quadratic form $Q_2$, and the 
orange, green and  yellow lines are respectively the lines $s_1 =0, s_2=0$, 
and $s_1+s_2=0$.\label{graphquadratic2} }
\end{figure}

\smallskip

We can now present expressions that are associated with the functions appearing in the 
differential equations and are in kernel of the transformation $1+T_3+T_3^2$, where $T_3$ is 
the above action of the cyclic group of order 3 on $\mathbb{R}^2$. 

\begin{theorem} \label{k8tok16eqnsinvariance}
The  expression  
\begin{eqnarray*}
&&e^{-\frac{2s_1}{3}-\frac{s_2}{3} } \Big ( 
 -\ddsone (k_{j_0}+k_{j_1}+k_{j_2})\left(s_1,s_2\right) \\
 && \qquad \qquad - e^{s_1+s_2} (\ddstwo - \ddsone )(k_{j_0}+k_{j_1}+k_{j_2}) 
\left(-s_1-s_2,s_1\right) \\
&& \qquad \qquad + e^{s_1} \ddstwo (k_{j_0}+k_{j_1}+k_{j_2})\left(s_2,-s_1-s_2\right) \Big ) 
\end{eqnarray*}
is in the kernel of $1+T_3+T_3^2$ for any integers $j_0, j_1, j_2$ in $\{8, 9, \dots, 16 \}$. 
In particular, considering the differential equations \eqref{diffeqK8} -- \eqref{diffeqK16} and 
the above discussion following the equation \eqref{leftK8K16groupactionform}, 
in the following cases we have obtained expressions that are given by finite 
differences of the functions $k_j$ and are in the kernel of $1+T_3+T_3^2$: 
\begin{enumerate}
\item  When $(j_0, j_1, j_2)=( 8, 12, 11)$. %in $\{ ( 8, 12, 11), ( 11, 8, 12), ( 12, 11, 8) \}$. 
\item When $(j_0, j_1, j_2)= ( 9, 13, 10)$. %in $\{ ( 9, 13, 10), ( 10, 9, 13), ( 13, 10, 9) \}$. 
\item When $(j_0, j_1, j_2)= ( 14, 16, 15)$. %in $\{ ( 14, 16, 15), ( 15, 14, 16), ( 16, 15, 14) \}$. 
\end{enumerate}
\begin{proof}
By multiplying  \eqref{leftK8K16groupactionform} by 
\[
\alpha_2(s_1, s_2):=e^{-\frac{2s_1}{3}-\frac{s_2}{3} },
\] 
since 
\begin{eqnarray*}
T_3 \alpha_2 \left(s_1,s_2\right)  &=& e^{s_1+s_2} \alpha_2(s_1, s_2)= e^{\frac{1}{3} \left(s_1+2 s_2\right)}, \\
T_3^2 \alpha_2 \left(s_1,s_2\right) &=& e^{s_1} \alpha_2(s_1, s_2)= e^{\frac{1}{3} \left(s_1-s_2\right)}, 
\end{eqnarray*}
we obtain: 
\[
 e^{-\frac{2s_1}{3}-\frac{s_2}{3} }\big (
 -\ddsone k_{j_0}{}\left(s_1,s_2\right) - e^{s_1+s_2} (\ddstwo - \ddsone )k_{j_1}{}\left(-s_1-s_2,s_1\right) + 
 e^{s_1} \ddstwo k_{j_2}{}\left(s_2,-s_1-s_2\right) \big ) 
 \]
\begin{eqnarray*}
&=& -e^{-\frac{2s_1}{3}-\frac{s_2}{3}}
 \ddsone k_{j_0}{}\left(s_1,s_2\right) - e^{\frac{s_1}{3}+\frac{2 s_2}{3}} (\ddstwo - \ddsone )k_{j_1}{}\left(-s_1-s_2,s_1\right) \\
 &&+ 
 e^{\frac{s_1}{3} -\frac{s_2}{3}} \ddstwo k_{j_2}{}\left(s_2,-s_1-s_2\right) \\
&=& - \alpha_2(s_1, s_2)  \ddsone k_{j_0}{}\left(s_1,s_2\right) 
 - T_3 \left(  \alpha_2 \cdot \ddstwo k_{j_1} \right) \left(s_1,s_2\right)  \\
&&+ T_3 \left(  \alpha_2 \cdot \ddsone k_{j_1} \right) \left(s_1,s_2\right)  
+ T_3^2 \left(  \alpha_2 \cdot \ddstwo k_{j_2} \right ) \left(s_1,s_2\right)  \\
&=& - \big ( \alpha_2 \cdot \ddsone k_{j_0}  \left(s_1,s_2\right) + 
T_3 \left ( \alpha_2 \cdot \ddstwo k_{j_1} \right )  \left(s_1,s_2\right) \big ) \\
&&+ T_3 \big (   \alpha_2 \cdot \ddsone k_{j_1} + T_3 \left ( \alpha_2 \cdot \ddstwo k_{j_2} \right )    \big )\left(s_1,s_2\right).   
\end{eqnarray*}
Thinking of this formula as being associated with the tripe $(j_0, j_1, j_2)$, we 
consider the formulas corresponding to the cyclic permutations $(j_2, j_0, j_1)$ and $(j_1, j_2, j_0)$ of 
$(j_0, j_1, j_2)$ to get two more equations, and summing up the three equations we have:  
\[
e^{-\frac{2s_1}{3}-\frac{s_2}{3} } \Big ( 
 -\ddsone (k_{j_0}+k_{j_1}+k_{j_2})\left(s_1,s_2\right) - e^{s_1+s_2} (\ddstwo - \ddsone )(k_{j_0}+k_{j_1}+k_{j_2}) 
\left(-s_1-s_2,s_1\right)
\]
\[
 + e^{s_1} \ddstwo (k_{j_0}+k_{j_1}+k_{j_2})\left(s_2,-s_1-s_2\right) \Big ) 
\]
\[
= 
 - \big ( \alpha(s_1, s_2) \ddsone ( k_{j_0} + k_{j_1} + k_{j_2} ) \left(s_1,s_2\right) + T_3 \left ( \alpha_2 \cdot \ddstwo ( k_{j_0} + k_{j_1} + k_{j_2} )  \right )  \left(s_1,s_2\right) \big ) 
 \]
 \[
+ T_3 \big (   \alpha_2 \cdot \ddsone ( k_{j_0} + k_{j_1} + k_{j_2} )  + T_3 \left ( \alpha_2 \cdot \ddstwo ( k_{j_0} + k_{j_1} + k_{j_2} )  \right )    \big )\left(s_1,s_2\right)
\]
\[
= (-1 + T_3) \Big (  \alpha_2 \cdot \ddsone ( k_{j_0} + k_{j_1} + k_{j_2} )  + T_3 \left ( \alpha_2 \cdot \ddstwo ( k_{j_0} + k_{j_1} + k_{j_2} )  \right )  \Big ) \left(s_1,s_2\right).
\]
The latter is clearly in the kernel of $1+T_3+T_3^2$ since $T_3^3 =1$. 
\end{proof}

\end{theorem}

Let us elaborate on the second part of the statement of the above theorem. 
Using the notation \eqref{diffeqsK8K16concise}, Lemma \ref{quadratic2Lemma}, and the above 
theorem we have proved that the following expressions are in the kernel 
of the transformation $1+T_3+T_3^2$: 
\[
e^{-\frac{2s_1}{3}-\frac{s_2}{3} } \left ( f_{8, 12, 11}(s_1, s_2) + f_{11, 8, 12}(s_1, s_2) + f_{12, 11, 8}(s_1, s_2)  \right ), 
\]
\[
e^{-\frac{2s_1}{3}-\frac{s_2}{3} } \left ( f_{9, 13, 10}(s_1, s_2) + f_{13, 10, 9}(s_1, s_2) + f_{10, 9, 13}(s_1, s_2)  \right ), 
\]
\[
e^{-\frac{2s_1}{3}-\frac{s_2}{3} } \left ( f_{14, 15, 16}(s_1, s_2) + f_{15, 16, 14}(s_1, s_2) + f_{16, 14, 15}(s_1, s_2)  \right ). 
\]

\smallskip

We now turn our focus to the functions $k_{17}(s_1, s_2, s_3), \dots, k_{20}(s_1, s_2, s_3)$. By 
considering the differential equations \eqref{diffeqK17} -- \eqref{diffeqK20}, one gets a 
cyclic group of order $4$ which corresponds to the matrices:
  $$
T_4= \left(
\begin{array}{ccc}
 -1 & -1 & -1 \\
 1 & 0 & 0 \\
 0 & 1 & 0 \\
\end{array}
\right)
,\ \
T_4^2=
\left(
\begin{array}{ccc}
 0 & 0 & 1 \\
 -1 & -1 & -1 \\
 1 & 0 & 0 \\
\end{array}
\right), \ \ 
T_4^3=
 \left(
\begin{array}{ccc}
 0 & 1 & 0 \\
 0 & 0 & 1 \\
 -1 & -1 & -1 \\
\end{array}
\right). 
$$
Together with the identity matrix, these matrices form a cyclic group of order 4. 
The  transformations corresponding to the above matrices are 
\[
(s_1,s_2,s_3)\mapsto (-s_1-s_2-s_3,s_1,s_2),
\] 
\[
(s_1,s_2,s_3)\mapsto (s_3,-s_1-s_2-s_3,s_1), 
\]
\[
(s_1,s_2,s_3)\mapsto (s_2,s_3,-s_1-s_2-s_3).
\] 
We  investigate 
the compatibility of these transformations  with the denominators and the various 
functions involved.

\smallskip

In Remark \ref{denK17-K20diffsys}, we pointed out that the denominator on the 
right hand side of the differential equations \eqref{diffeqK17} -- \eqref{diffeqK20} 
is the expression 
\[
s_1 s_2 \left(s_1+s_2\right) s_3 \left(s_2+s_3\right) \left(s_1+s_2+s_3\right). 
\]
The following lemma  shows the invariance of this expression under the above 
transformations.

\begin{lemma} \label{quadratic3Lemma}
	The action of $\Z/4\Z$ on $\R^3$ given by the above matrices leaves the following 
	expressions invariant up to sign:
	\begin{enumerate}
	\item The positive definite quadratic form 
	 $$
Q_3(s)= s_1^2+s_2^2+s_3^2+s_1 s_2+s_2 s_3+s_3 s_1, \qquad s=(s_1, s_2, s_3) \in \mathbb{R}^3.  
 $$
	\item The linear form $s_1+s_3$ whose kernel defines a plane in which a generator 
	of $\Z/4\Z$ acts by a rotation of angle $\pi/2$ for the metric  induced by $Q_3$.
	\item 	The product $\left(s_1+s_2\right)  \left(s_2+s_3\right)$ and the line 
	$\{s= (s_1,-s_1,s_1)\mid s_1 \in \R\}$ on which it acts as the symmetry $s\mapsto -s$.
	\item The product $
s_1 s_2 s_3\left(s_1+s_2+s_3\right)
$.
	\end{enumerate}

\end{lemma}
This means that when using the positive definite quadratic form $Q_3$, the generator of the action of $\Z/4\Z$ has the following form: it preserves the line $\{s=(s_1,-s_1,s_1)\mid s_1 \in \R\}$ on which it acts as the symmetry $s\mapsto -s$. The orthogonal of this line for the quadratic form $Q_3$ is the plane $s_1+s_3=0$. On this plane the generator of the action of $\Z/4\Z$ is a rotation of angle $\pi/2$ for the metric  induced by $Q_3$ which agrees with $s_1^2+s_2^2$ since $s_1+s_3=0$. 

\begin{figure}
\begin{center}
\includegraphics[scale=1]{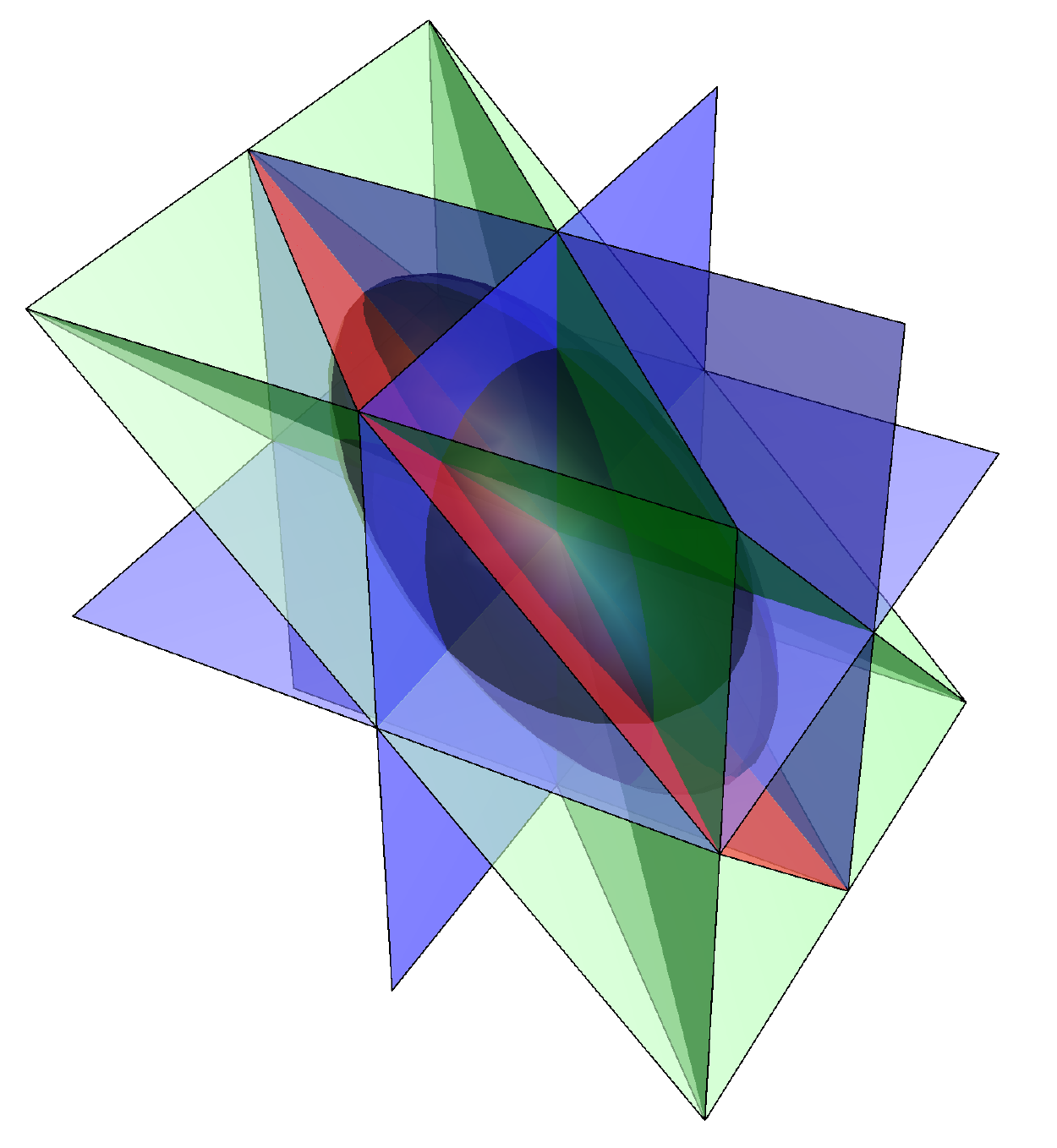}
\end{center}
\caption{The black ellipsoid is the unit ball for the quadratic form $Q_3$ and the various planes are: in blue the $s_j=0$, in red $s_1+s_2+s_3=0$ and in light green the plane $s_1+s_3=0$ in which the generator of the action of $\Z/4\Z$ is a rotation of angle $\pi/2$. The remaining two planes in dark green correspond to the product $\left(s_1+s_2\right)  \left(s_2+s_3\right)$.\label{symmetries} }
\end{figure}

%\begin{figure}
%\begin{center}
%\includegraphics[scale=1]{symmnctorus}
%\end{center}
%\caption{The black ellipsoid is the unit ball for the quadratic form $Q_3$ and the various planes are: in blue the $s_j=0$, in red $s_1+s_2+s_3=0$ and in light green the plane $s_1+s_3=0$ in which the generator of the action of $\Z/4\Z$ is a rotation of angle $\pi/2$. The remaining two planes in dark green correspond to the product $\left(s_1+s_2\right)  \left(s_2+s_3\right)$.\label{symmetries} }
%\end{figure}

\smallskip

Note also that the natural lattice $L=\Z^3\subset \R^3$ is contained in its dual $L^\perp$ with respect to the quadratic form $Q_3$ which is defined by 
$$
L^\perp=\{s=(s_1,s_2,s_3)\mid Q_3(s, t)\in\Z, \ \forall t \in L\}. 
$$
One finds that the quotient $H=L^\perp/L$ is the group $\Z/4\Z$ generated by $(\frac 14,\frac 14,\frac 14)$.

\smallskip

The left hand side of the differential equations \eqref{diffeqK17} -- \eqref{diffeqK20}, up to multiplication 
by $\frac{1}{16}$, is of the form 
\begin{eqnarray} \label{leftK17K20groupactionform}
&& -  \ddsone k_{j_0}{}\left(s_1,s_2,s_3\right) 
- e^{s_1+s_2+s_3} (\ddstwo - \ddsone) k_{j_1}(-s_1-s_2-s_3, s_1, s_2) \nonumber \\
&&- e^{s_1+s_2} (\ddsthree - \ddstwo) k_{j_0} (s_3, -s_1-s_2-s_3, s_1) 
+e^{s_1} \ddsthree k_{j_1} (s_2, s_3, -s_1-s_2-s_3) \nonumber \\
&=& -  \ddsone k_{j_0}{}\left(s_1,s_2,s_3\right)  - 
T_4(e^{-s_1} (\ddstwo - \ddsone) k_{j_1}) \left(s_1,s_2,s_3\right)  \nonumber \\
&& - T_4^2(e^{-s_1-s_2} (\ddsthree - \ddstwo) k_{j_0})\left(s_1,s_2,s_3\right) 
+  T_4^3(e^{-s_1-s_2-s_3} k_{j_1})   \left(s_1,s_2,s_3\right),  
\end{eqnarray}
where 
\begin{eqnarray} \label{diffeqsK17K20indices}
\textrm{for the equation } \eqref{diffeqK17}:   &&\qquad (j_0, j_1) = ( 17, 19), \nonumber \\
\textrm{for the equation } \eqref{diffeqK18}:   &&\qquad (j_0, j_1) = ( 18, 18), \nonumber \\
\textrm{for the equation } \eqref{diffeqK19}:   &&\qquad (j_0, j_1) = ( 19, 17),\\
\textrm{for the equation } \eqref{diffeqK20}:   &&\qquad (j_0, j_1) = ( 20, 20). \nonumber
\end{eqnarray}
This suggests that we can put the differential equations \eqref{diffeqK17} -- \eqref{diffeqK20} 
into three groups: one consisting of the equation \eqref{diffeqK17} and \eqref{diffeqK19} 
which involve partial derivatives of only $k_{17}$ and $k_{19}$, one group consisting 
of the equation \eqref{diffeqK18} since it involves partial derivatives of  $k_{18}$, and 
finally we can consider \eqref{diffeqK20} as the last group since it involves partial 
derivatives of only $k_{20}$.  Also we will introduce the following concise notation 
$f_{j_0, j_1}(s_1, s_2, s_3)$ to rewrite the differential equations \eqref{diffeqK17} -- \eqref{diffeqK20} 
as 
\begin{eqnarray} \label{diffeqsK17K20concise}
&& -  \ddsone k_{j_0}{}\left(s_1,s_2,s_3\right) 
- e^{s_1+s_2+s_3} (\ddstwo - \ddsone) k_{j_1}(-s_1-s_2-s_3, s_1, s_2) \nonumber \\
&&- e^{s_1+s_2} (\ddsthree - \ddstwo) k_{j_0} (s_3, -s_1-s_2-s_3, s_1) \\
&&+e^{s_1} \ddsthree k_{j_1} (s_2, s_3, -s_1-s_2-s_3) \nonumber \\
&=& \frac{f_{j_0, j_1}(s_1, s_2, s_3)}{s_1 s_2 \left(s_1+s_2\right) s_3 \left(s_2+s_3\right) \left(s_1+s_2+s_3\right)}, \nonumber
\end{eqnarray}
where in each case, $(j_0, j_1)$ is given as in \eqref{diffeqsK17K20indices}. Note that 
each $f_{j_0, j_1}(s_1, s_2, s_3)$ is written in terms of finite differences of restriction to $s_1 + s_2 +s_3 + s_4=0$ of the 
involved functions $k_j$.  

\smallskip

We can now exploit the above action of the cyclic group of order 4 on $\mathbb{R}^3$ generated by 
$T_4$ to derive 
expressions from our differential equations that are in the kernel of the transformation 
$1+ T_4 + T_4^2 + T_4^3$. 
\begin{theorem} \label{k17tok20eqnsinvariance}
The expression 
\begin{eqnarray*}
&&e^{-\frac{3s_1}{4}-\frac{s_2}{2}-\frac{s_3}{4}} \Big (  -  \ddsone (k_{j_0}{} + k_{j_1})\left(s_1,s_2,s_3\right) \\
&& \qquad \qquad \qquad - e^{s_1+s_2+s_3} (\ddstwo - \ddsone) (k_{j_0}{} + k_{j_1}) (-s_1-s_2-s_3, s_1, s_2)  \\
&& \qquad \qquad \qquad - e^{s_1+s_2} (\ddsthree - \ddstwo) (k_{j_0}{} + k_{j_1}) (s_3, -s_1-s_2-s_3, s_1) \\
&& \qquad  \qquad \qquad +e^{s_1} \ddsthree (k_{j_0}{} + k_{j_1}) (s_2, s_3, -s_1-s_2-s_3) \Big ) 
\end{eqnarray*}
is in the kernel of $1 + T_4 + T_4^2 + T_4^3$ for any pair of integers $j_0, j_1$ in $\{17, 18, 19, 20 \}$.  
In particular, considering the differential equations \eqref{diffeqK17} -- \eqref{diffeqK20} and 
the above discussion following the equation \eqref{leftK17K20groupactionform}, we have obtained 
in the following cases expressions that are given by finite differences of the $k_j$ and are 
in the kernel of $1 + T_4 + T_4^2 + T_4^3$: 
\begin{enumerate}
\item When $(j_0, j_1) = (17, 19)$.
\item When $(j_0, j_1) =(18, 18)$. 
\item When $(j_0, j_1)= (20, 20)$. 
\end{enumerate}
\begin{proof}
We multiply \eqref{leftK17K20groupactionform} by 
\[
\alpha_3 \left (s_1, s_2, s_3 \right) := e^{-\frac{3s_1}{4}-\frac{s_2}{2}-\frac{s_3}{4}}, 
\]
and since 
\begin{eqnarray*}
T_4 \alpha_3 \left (s_1, s_2, s_3 \right) &=& \alpha_3(s_1, s_2, s_3) \, e^{s_1+s_2+s_3} 
= e^{\frac{1}{4} \left(s_1+2 s_2+3 s_3\right)}, \\
T_4^2 \alpha_3 \left (s_1, s_2, s_3 \right) &=& \alpha_3(s_1, s_2, s_3) \, e^{s_1+s_2}  
= e^{\frac{1}{4} \left(s_1+2 s_2-s_3\right)}, \\
T_4^3 \alpha_3 \left (s_1, s_2, s_3 \right) &=& \alpha_3(s_1, s_2, s_3) \, e^{s_1} 
= e^{\frac{1}{4} \left(s_1-2 s_2-s_3\right)}, 
\end{eqnarray*}
we find that 
\begin{eqnarray*}
&& e^{-\frac{3s_1}{4}-\frac{s_2}{2}-\frac{s_3}{4}} \Big (  -  \ddsone k_{j_0}{}\left(s_1,s_2,s_3\right) 
- e^{s_1+s_2+s_3} (\ddstwo - \ddsone) k_{j_1}(-s_1-s_2-s_3, s_1, s_2)  \\
&&- e^{s_1+s_2} (\ddsthree - \ddstwo) k_{j_0} (s_3, -s_1-s_2-s_3, s_1) 
+e^{s_1} \ddsthree k_{j_1} (s_2, s_3, -s_1-s_2-s_3) \Big )  \\
&=& - ( \alpha_3 \cdot \ddsone k_{j_0})  \left (s_1, s_2, s_3 \right) - T_4(\alpha_3 \cdot \ddstwo k_{j_1} )  \left (s_1, s_2, s_3 \right)
+ T_4(\alpha_3 \cdot \ddsone k_{j_1} )  \left (s_1, s_2, s_3 \right)- \\ 
&& T_4^2(\alpha_3 \cdot \ddsthree k_{j_0} )  \left (s_1, s_2, s_3 \right)
+ T_4^2(\alpha_3 \cdot \ddstwo k_{j_0} )  \left (s_1, s_2, s_3 \right)+ T_4^3(\alpha_3 \cdot \ddsthree k_{j_1} ) \left (s_1, s_2, s_3 \right). 
\end{eqnarray*}
We switch $j_0$ and $j_1$ in this equation to get  a new one, and adding the two equations we 
have  
\begin{eqnarray*}
&&
 - ( \alpha_3 \cdot \ddsone (k_{j_0}+k_{j_1}))  
 - T_4(\alpha_3 \cdot \ddstwo (k_{j_0}+k_{j_1}) ) 
+ T_4(\alpha_3 \cdot \ddsone (k_{j_0}+k_{j_1}) )   \\ 
&& -T_4^2(\alpha_3 \cdot \ddsthree (k_{j_0}+k_{j_1}) )  
+ T_4^2(\alpha_3 \cdot \ddstwo (k_{j_0}+k_{j_1}) ) + T_4^3(\alpha_3 \cdot \ddsthree (k_{j_0}+k_{j_1}) ) \\
&=& - \left ( \alpha_3 \cdot \ddsone(k_{j_0}+k_{j_1}) + T_4( \alpha_3 \cdot \ddstwo (k_{j_0}+k_{j_1}))
+ T_4^2( \alpha_3 \cdot \ddsthree (k_{j_0}+k_{j_1})) \right ) \\
&&+ T_4  \left ( \alpha_3 \cdot \ddsone(k_{j_0}+k_{j_1}) + T_4( \alpha_3 \cdot \ddstwo (k_{j_0}+k_{j_1}))
+ T_4^2( \alpha_3 \cdot \ddsthree (k_{j_0}+k_{j_1})) \right ) \\
&=& (-1+T_4) \left ( \alpha_3 \cdot \ddsone(k_{j_0}+k_{j_1}) + T_4( \alpha_3 \cdot \ddstwo (k_{j_0}+k_{j_1}))
+ T_4^2( \alpha_3 \cdot \ddsthree (k_{j_0}+k_{j_1})) \right ),  
\end{eqnarray*}
which is in the kernel of $1 + T_4 + T_4^2 + T_4^3$ since $T_4^4=1$. 
\end{proof}
\end{theorem}

\smallskip

In order to explain the second statement in this theorem, it is important to mention that 
with the aid of the notation used in \eqref{diffeqsK17K20concise}, Lemma \ref{quadratic3Lemma}, and 
the above theorem, we have shown that the following expressions are in the kernel of the 
transformation $1 + T_4 + T_4^2 + T_4^3$: 
\[
e^{-\frac{3s_1}{4}-\frac{s_2}{2}-\frac{s_3}{4}} \left ( f_{17, 19}(s_1, s_2, s_3) + f_{19, 17}(s_1, s_2, s_3) \right ), 
\]
\[
e^{-\frac{3s_1}{4}-\frac{s_2}{2}-\frac{s_3}{4}}  f_{18, 18}(s_1, s_2, s_3),
\]
\[
e^{-\frac{3s_1}{4}-\frac{s_2}{2}-\frac{s_3}{4}} f_{20, 20}(s_1, s_2, s_3). 
\]

\subsection{Flow on the orbits and the differential system}

For any 1-variable function $k(s_1)$ let us consider the orbit $\mathcal{O}k$ of the 
function under the action of $T_2$ on $\R$: 
\[
\mathcal{O}k (s_1) = \left (k(s_1), k(-s_1) \right ). 
\]
Since 
\[
\ddt \mathcal{O}k(s_1 +t) = \ddt \left (k(s_1+t), k(-s_1-t) \right ) = \left ( k'(s_1), - k'(-s_1) \right ), 
\]
the differential parts of the differential equations \eqref{diffeqK3} -- \eqref{diffeqK7second} 
when put together as in the statement of Theorem \ref{k1tok7eqnsinvariance} are generated by 
the following flow when passed to the orbit using the action of $T_2$ on $\R$:  
\[
s_1 \mapsto s_1 + t. 
\]
On the other hand from the proof of Theorem \ref{k1tok7eqnsinvariance}, we can use the 
function $\alpha_1(s_1)=e^{-\frac{s_1}{2}}$ with $\mathcal{O} \alpha_1(s_1) = (e^{-\frac{s_1}{2}}, e^{\frac{s_1}{2}})$ to obtain the following statement, in which we use the notation given by 
\eqref{diffeqnK3K7consice}.

\begin{theorem} Considering the cases for $(j_0, j_1)$ listed below and by setting for each case  
\[
k= k_{j_0} + k_{j_1}, \qquad f = f_{j_0, j_1} + f_{j_1, j_0},  
\]
we have 
\[
\left ( \ddt   \mathcal{O} k (s_1 +t) \right )  \cdot  \left (\mathcal{O} \alpha_1 (s_1) \right ) 
 = -
 \frac{\alpha_1 (s_1) \, f(s_1)}{s_1}. 
\]
The following are the cases: 
\begin{enumerate}
\item When $(j_0, j_1) = (3, 3)$.
\item When $(j_0, j_1) = (4, 4)$.
\item When $(j_0, j_1) = (5, 5)$.
\item When $(j_0, j_1) = (6, 7)$.
\end{enumerate}
\end{theorem} 

We note that Theorem \ref{k1tok7eqnsinvariance} asserts that in each of the above cases  
the function $ \alpha_1 (s_1) \, f(s_1)/s_1$ is in the kernel of transformation $1+T_2$, 
which sends via averaging any function defined on $\R$ to a function defined on  the quotient space.  

\smallskip

Now for any 2-variable function $k(s_1, s_2)$ we similarly define $\mathcal{O}k$ to be the 
orbit of the function $k$ under the action of $T_3$ on $\R^2$: 
\[
\mathcal{O}k (s_1, s_2) = \left (  k(s_1, s_2), k(-s_1-s_2, s_1), k(s_2, -s_1 -s_2)  \right ). 
\]
It is easy to see that 
\begin{eqnarray*}
&&\ddt \mathcal{O}k (s_1 +t, s_2) \\
&=&   \ddt  \left (  k(s_1+t, s_2), k(-s_1-t -s_2, s_1+t), k(s_2, -s_1-t -s_2)  \right ) \\
&=&  \left ( \ddsone k(s_1, s_2),  \, (\ddstwo - \ddsone) k (-s_1-s_2, s_1), \, - \ddstwo k(s_2, -s_1-s_2) \right ).
\end{eqnarray*}
This means that the flow induced on the orbit from the flow on $\R^2$ given by 
\[
(s_1, s_2) \mapsto (s_1+t, s_2)
\]
generates the differential components of the differential equations \eqref{diffeqK8} -- \eqref{diffeqK16} 
when put together as in the statement of Theorem \ref{k8tok16eqnsinvariance}.  
Also we know from the proof of this theorem that the coefficients appearing on the left 
hand sides of these equations are 
induced by the action of $T_3$ starting from the function 
\[
\alpha_2(s_1, s_2) = e^{-\frac{2s_1}{3}-\frac{s_2}{3} },
\] 
since 
\[
\mathcal{O} \alpha_2 (s_1, s_2) = \left (  e^{-\frac{2s_1}{3}-\frac{s_2}{3} }, e^{\frac{1}{3} \left(s_1+2 s_2\right)}, e^{\frac{1}{3} \left(s_1-s_2\right)} \right ). 
\]
Therefore, using the  notation introduced by \eqref{diffeqsK8K16concise}  we have the following statement.  
\begin{theorem}
For each of the cases listed below we set 
\[
k= k_{j_0} + k_{j_1} + k_{j_2}, \qquad f = f_{j_0, j_1, j_2} + f_{j_2, j_0, j_1} + f_{j_1, j_2, j_0}. 
\]
Then we have 
\begin{eqnarray*}
 \left ( \ddt   \mathcal{O} k (s_1 +t, s_2) \right )  \cdot  \left (\mathcal{O} \alpha_2 (s_1, s_2) \right ) 
 = -
 \frac{\alpha_2 (s_1, s_2) \, f(s_1, s_2)}{s_1 s_2 (s_1 + s_2)},  
\end{eqnarray*}
where the following are the cases for $(j_0, j_1, j_2)$:  
\begin{enumerate}
\item  When $(j_0, j_1, j_2)=( 8, 12, 11)$. 
\item When $(j_0, j_1, j_2)= ( 9, 13, 10)$. 
\item When $(j_0, j_1, j_2)= ( 14, 16, 15)$.  
\end{enumerate}

\end{theorem}

We recall from  Theorem \ref{k8tok16eqnsinvariance} that in each of the cases 
mentioned in the above theorem, 
the function $\alpha_2 (s_1, s_2) f(s_1, s_2)$ is in the kernel of the averaging 
operator $1+T_3 + T_3^2$, which means that its average (defined on the quotient space)   
vanishes.   We also note from Lemma \ref{quadratic2Lemma} that the denominator $s_1 s_2 (s_1 + s_2)$ 
is invariant under the action of $T_3$. 

\smallskip

Similarly, for any 3-variable function $k(s_1, s_2, s_3)$ we define 
$\mathcal{O}k$ to be orbit of the function 
$k$ under the action of $T_4$ on $\R^3$: 
\[
\mathcal{O}k (s_1, s_2, s_3) =
\]
\[
 \left (  k(s_1, s_2, s_3), k(-s_1-s_2-s_3, s_1, s_2), k(s_3, -s_1 -s_2-s_3, s_1), 
k(s_2, s_3, -s_1 -s_2-s_3)  \right ). 
\]
In this case as well  the flow induced on the orbit from the flow on $\R^3$ given by 
\[
(s_1, s_2, s_3) \mapsto (s_1+t, s_2, s_3)
\]
generates the differential components of the differential equations \eqref{diffeqK17} -- 
\eqref{diffeqK20} when considered together as mentioned in the statement of 
Theorem \ref{k17tok20eqnsinvariance}. That is, one can 
see easily that 
\begin{eqnarray*}
&&\ddt \mathcal{O}k (s_1 +t, s_2, s_3) \\
&=&  \ddt  \Big (  k(s_1+t, s_2, s_3), k(-s_1-t-s_2-s_3, s_1+t, s_2),  \\
&&
\qquad \qquad \qquad \qquad  k(s_3, -s_1 -t -s_2-s_3, s_1+t), 
k(s_2, s_3, -s_1-t -s_2-s_3)  \Big ) \\
&=&  \Big ( \ddsone k(s_1, s_2, s_3),  \, (\ddstwo - \ddsone) k (-s_1-s_2-s_3, s_1, s_2), \\
&&
 \qquad \qquad \, ( \ddsthree - \ddstwo ) k(s_3, -s_1 -s_2-s_3, s_1), 
\, - \ddsthree k (s_2, s_3, -s_1-s_2-s_3) \Big ).
\end{eqnarray*}

As we observed in the proof of Theorem \ref{k17tok20eqnsinvariance}  the coefficients appearing 
on the left hand sides of these equations are 
induced by the transformation $T_4$ starting from the function 
\[
\alpha_3 \left (s_1, s_2, s_3 \right) = e^{-\frac{3s_1}{4}-\frac{s_2}{2}-\frac{s_3}{4}}, 
\]
because we have
\[
\mathcal{O} \alpha_3 (s_1, s_2, s_3) = 
\left (  e^{-\frac{3s_1}{4}-\frac{s_2}{2}-\frac{s_3}{4}},   
e^{\frac{1}{4} \left(s_1+2 s_2+3 s_3\right)},  
e^{\frac{1}{4} \left(s_1+2 s_2-s_3\right)},
 e^{\frac{1}{4} \left(s_1-2 s_2-s_3\right)} 
\right ). 
\]
Hence, using the notation introduced by \eqref{diffeqsK17K20concise} we have: 

\begin{theorem}
For each case from the list given below set
\[
k= k_{j_0} + k_{j_1}, \qquad f = f_{j_0, j_1} + f_{j_1, j_0}. 
\]
Then we have: 
\begin{eqnarray*}
&& \left ( \ddt   \mathcal{O} k (s_1 +t, s_2, s_3) \right )  \cdot  
 \left (\mathcal{O} \alpha_3 (s_1, s_2, s_3) \right )   \\
 &=& 
- \frac{\alpha_3 (s_1, s_2, s_3) \, f(s_1, s_2, s_3)}{s_1 s_2 \left(s_1+s_2\right) s_3 \left(s_2+s_3\right) \left(s_1+s_2+s_3\right)},  
\end{eqnarray*}
where the cases are the following: 
\begin{enumerate}
\item When $(j_0, j_1) = (17, 19)$.
\item When $(j_0, j_1) =(18, 18)$. 
\item When $(j_0, j_1)= (20, 20)$. 
\end{enumerate}
\end{theorem}

We note that we proved in Theorem \ref{k17tok20eqnsinvariance} that in each of the above cases 
the function $\alpha_3 (s_1, s_2, s_3) f(s_1, s_2, s_3)$ is in the kernel of the averaging 
operator $1+T_4 + T_4^2 + T_4^3$, hence its average (defined on the quotient space)  
vanishes.   We also recall from Lemma \ref{quadratic3Lemma} that the denominator 
$s_1 s_2 \left(s_1+s_2\right) s_3 \left(s_2+s_3\right) \left(s_1+s_2+s_3\right)$ is invariant 
under the action of $T_4$.

\smallskip

\section{The method and tools for proving the functional relations}
\label{FuncRelationsProofSec}

In this section we provide the tools for 
calculating in two different ways the gradient of the map that sends a 
general dilaton $h = h^* \in \CNT$ to $\vphi_0(a_4)$, which will 
eventually lead to the functional relations stated in Theorem \ref{FuncRelationsThm} in  
Section \ref{a_4Sec}. 
For the sake of clarity, let us write $a_4 = a_4(h)$ so that we 
can clearly view it as a function of the dilaton  
$h$. Also, we will use the following notation for the 
conformally perturbed Laplacian given 
by \eqref{conformalLaplacian}:  
\[
\triangle_h = e^{h/2} \triangle e^{h/2}.    
\]

\smallskip

The methods are in fact closely based on the techniques initiated in 
\cite{ConMosModular}. The first method employs a fundamental 
identity proved in 
\cite{ConMosModular} as follows. For a general element 
$ a \in \CNT,$  consider the spectral zeta function 
defined by 
\[
\zeta_{h}(a, s) = \Tr(a \, \triangle_h^{-s}), 
\qquad s \in \C, \,\, \Re(s) \gg 0. 
\]
The zeta function is initially defined when the real part of the 
complex number $s$ is large enough, and it has a meromorphic 
extension to the complex plane. 
It is proved in \cite{ConMosModular} that if $h$ and $a$ are smooth selfadjoint 
elements in $C(\NT)$, then 
\begin{equation} \label{GradientIdentity}
\dep \zeta_{h + \vep a}(1, s) 
=  
-\frac{s}{2}\, \zeta_h \left( \widetilde a , s \right),  
\end{equation}
where 
\[
\widetilde a = \int_{-1}^1 e^{uh/2} a e^{-uh/2} \, du. 
\]

\smallskip

Using the Mellin transform and the asymptotic expansion \eqref{heatexp}, 
one can see that 
\[
\zeta_h(a, -1) = - \vphi_0(a \, a_4(h)), \qquad a \in \CNT, \,\, h=h^* \in \CNT.  
\]
Therefore, it follows from \eqref{GradientIdentity} that 
\begin{equation} \label{tildea4start}
\dep \vphi_0(a_4(h + \vep a)) 
= 
-\frac{1}{2} \zeta_h (\widetilde a, -1) 
=
\frac{1}{2} \vphi_0(\widetilde a \, a_4(h)). 
\end{equation} 
At this stage, following the method in \cite{ConMosModular}, it is 
crucial to observe that, in general, given a smooth function 
$L$ and elements $x_1, \dots, x_n \in \CNT$, we have:  
\begin{eqnarray} \label{tildegeneral}
&&\vphi_0 \left ( \widetilde a e^h L(\nab, \dots, \nab)(x_1 \cdots x_n) \right )  \\
&=& \vphi_0 \left (  \int_{-1}^1 e^{uh/2} a e^h e^{-uh/2} \, du \, L(\nab, \dots, \nab)(x_1 \cdots x_n) \right )  \nonumber \\
&=&  \vphi_0 \left (  a e^h \int_{-1}^1   e^{-uh/2}  \, L(\nab, \dots, \nab)(x_1 \cdots x_n) \, e^{uh/2}  \, du \right ) \nonumber \\
&=& \vphi_0 \left ( a e^h \left ( 2 \frac{\sinh(\nab/2)}{\nab/2} \right ) 
\left (L(\nab, \dots, \nab)(x_1 \cdots x_n) \right ) \right ) \nonumber \\
&=& \vphi_0 \left ( a e^h \left( 2 \frac{\sinh((s_1+ \cdots + s_n)/2)}{(s_1+ \cdots + s_n)/2} L(s_1, \dots, s_n) \right )\bigm|_{s_j = \nab} (x_1 \cdots x_n)  \right ). \nonumber
\end{eqnarray}
This justifies the reason for defining the variants $\widetilde K_j$ of the 
functions $K_j$ as in \eqref{Ktildesdef}.

\smallskip

By using  the identities  \eqref{tildea4start} and 
\eqref{tildegeneral}, and considering the 
expression \eqref{a_4expression} for the term $a_4(h)$ with $\ell = h/2$, 
we can conclude that under the trace $\vphi_0$, we have: 
\begin{equation} \label{Gradientofa4}
\dep \vphi_0 \left ( a_4(h + \vep a) \right ) =
\end{equation}
\begin{eqnarray}
&&- a e^{h} \Big (  \widetilde K_1(\nab) \left ( \delta _1^2 \delta _2^2 ( h ) \right ) 
+ 
\widetilde K_2 (\nab) \left (   \delta _1^4( h )+\delta _2^4 (h ) \right )  
+
\widetilde K_3 (\nab, \nab) \left (
 \left(\delta _1 \delta _2(h
   )\right) \cdot \left(\delta _1 \delta _2(h
   )\right) 
   \right )
\nonumber \\ && 
+ 
\widetilde K_4 (\nab, \nab) \left (  \delta _1^2(h ) \cdot \delta _2^2(h )+\delta
   _2^2(h ) \cdot \delta _1^2(h )\right )  
+
\widetilde K_5 (\nab, \nab) \left ( \delta _1^2( h )\cdot \delta _1^2(h )+\delta
   _2^2(h ) \cdot \delta _2^2(h ) \right ) 
\nonumber \\ && 
+ 
\widetilde K_6 (\nab, \nab) \left ( 
\delta _1(h )\cdot \delta _1^3(h )+\delta_1(h )\cdot \left(\delta _1 \delta _2^2 (h
   )\right)+\delta _2(h )\cdot \delta _2^3(h
   )+\delta _2(h )\cdot \left(\delta _1^2
   \delta _2(h )\right)
 \right ) 
\nonumber \\ &&
+ 
\widetilde K_7 (\nab, \nab) \left ( 
\delta _1^3(h )\cdot \delta _1(h
   )+\left(\delta _1 \delta _2^2(h
   )\right)\cdot \delta _1(h )+\delta _2^3(h
   )\cdot \delta _2(h )+\left(\delta _1^2
   \delta _2(h )\right)\cdot \delta _2(h )
 \right ) 
\nonumber 
\end{eqnarray}
\begin{eqnarray}
 &&
+ 
\widetilde K_8 (\nab, \nab, \nab) \left ( 
\delta _1(h ) \cdot \delta _1(h )\cdot \delta
   _2^2(h )+\delta _2(h )\cdot \delta
   _2(h )\cdot \delta _1^2(h )
\right ) 
\nonumber \\ && 
+ 
\widetilde K_9 (\nab, \nab, \nab) \left ( 
\delta _1(h )\cdot \delta _2(h
   )\cdot \left(\delta _1 \delta _2(h
   )\right)+\delta _2(h )\cdot \delta _1(h
   )\cdot \left(\delta _1 \delta _2(h )\right)
\right ) 
\nonumber \\ && 
+ 
\widetilde K_{10} (\nab, \nab, \nab) \left ( 
\delta _1(h )\cdot \left(\delta _1 \delta
   _2(h )\right)\cdot \delta _2(h )+\delta
   _2(h )\cdot  \left(\delta _1 \delta _2(h
   )\right) \cdot  \delta _1(h )
\right )
\nonumber \\ && 
+ 
\widetilde K_{11} (\nab, \nab, \nab) \left ( 
\delta _1(h )\cdot \delta _2^2(h ) \cdot \delta
   _1(h )+\delta _2(h )\cdot \delta
   _1^2(h ) \cdot \delta _2(h )
\right ) 
\nonumber \\ && 
+ 
\widetilde K_{12} (\nab, \nab, \nab) \left ( 
\delta _1^2(h ) \cdot \delta _2(h )\cdot \delta
   _2(h )+\delta _2^2( h ) \cdot \delta
   _1(h )\cdot \delta _1(h )
\right ) 
\nonumber 
\end{eqnarray}
\begin{eqnarray}
&&+ 
\widetilde K_{13} (\nab, \nab, \nab) \left ( 
\left(\delta _1 \delta _2(h
   )\right)\cdot \delta _1(h )\cdot \delta _2(h
   )+\left(\delta _1 \delta _2(h
   )\right) \cdot \delta _2(h )\cdot \delta _1(h )
\right ) 
\nonumber \\ && 
+ 
\widetilde K_{14} (\nab, \nab, \nab) \left ( 
\delta _1^2(h ) \cdot \delta _1(h )\cdot 
\delta_1(h )+\delta _2^2(h )\cdot \delta_2(h )\cdot \delta _2(h )
\right )
\nonumber \\ && 
+ 
\widetilde K_{15} (\nab, \nab, \nab) \left ( 
\delta _1(h ) \cdot \delta _1(h )\cdot \delta
   _1^2(h )+\delta _2(h )\cdot \delta
   _2(h )\cdot \delta _2^2(h )
\right )
\nonumber \\ && 
+ 
\widetilde K_{16} (\nab, \nab, \nab) \left ( 
\delta _1(h )\cdot \delta _1^2(h ) \cdot \delta
   _1(h )+\delta _2(h )\cdot \delta
   _2^2(h )\cdot \delta _2(h )
\right ) 
\nonumber 
\end{eqnarray} 
\begin{eqnarray}
&& 
+ 
\widetilde K_{17} (\nab, \nab, \nab, \nab ) \left ( 
\delta _1(h )\cdot \delta _1(h )\cdot \delta
   _2(h )\cdot \delta _2(h )+\delta _2(h
   )\cdot \delta _2(h )\cdot \delta _1(h )\cdot \delta
   _1(h )
\right ) 
\nonumber \\ && 
+ 
\widetilde K_{18} (\nab, \nab, \nab, \nab ) \left ( 
\delta _1(h )\cdot \delta _2(h )\cdot \delta
   _1(h )\cdot \delta _2(h )+\delta _2(h
   )\cdot \delta _1(h )\cdot \delta _2(h )\cdot \delta
   _1(h )
\right ) 
\nonumber \\ && 
+ 
\widetilde K_{19} (\nab, \nab, \nab, \nab ) \left ( 
\delta _1(h )\cdot \delta _2(h )\cdot \delta
   _2(h )\cdot \delta _1(h )+\delta _2(h
   )\cdot \delta _1(h )\cdot \delta _1(h )\cdot \delta
   _2(h )
\right ) 
\nonumber \\ && 
+ 
\widetilde K_{20} (\nab, \nab, \nab, \nab ) \left ( 
\delta _1(h )\cdot \delta _1(h )\cdot \delta
   _1(h )\cdot \delta _1(h )+\delta _2(h
   )\cdot \delta _2(h )\cdot \delta _2(h )\cdot \delta
   _2(h )
\right ) \Big ). \nonumber
\end{eqnarray}

\smallskip

For our second calculation of the above gradient, we use 
Duhamel's formula in a crucial way, which, given a family of 
operators $A_s$, allows one to write: 
\[
\frac{d}{ds} e^{-A_s} = - \int_0^1  e^{-u A_s} \frac{d A_s}{ds} e^{-(1-u)A_s} \,ds. 
\]
This method will give rise to a formula for the desired gradient 
in terms of finite differences, which is described by performing 
explicit calculation in Section \ref{GradientCalculationSec}. However, first, 
we need to prove a series of lemmas, which are given in the following subsections.

\subsection{Gradients of functional calculi with $\nab$} In this subsection, 
we find explicit formulas for certain gradients of functions of the modular 
automorphism acting on elements of the noncommutative torus. For convenience, 
we will use the notation 
\[
\nabep = \nab_{h + \vep a}= \ad_{-h - \vep a} = \nab - \vep \ad_{a}, 
\]
where $\vep$ is a real number, and $h$ and $a$ are selfadjoint elements  
in $\CNT$.

\begin{lemma} \label{GrofFuncCalc1} Let $L(s_1)$ be a smooth function and 
$x_1, x_2$ be elements of the algebra  $C(\NT)$ of the  noncommutative torus. 
Under the trace $\varphi_0$, we can write: 
\begin{eqnarray*} \label{gr1}
e^h  \left ( \dep L(\nabep)(x_1 ) \right )  x_2 
&=& a e^h L^\vep_{1,1}(\nab, \nab) (x_1 \cdot x_2) \\
&&+
a e^h L^\vep_{1,2}(\nab, \nab)(x_2  \cdot x_1),  
\end{eqnarray*}
where 
\begin{eqnarray*}
L^\vep_{1,1}(s_1, s_2) &=& e^{s_1+s_2} \frac{L(-s_2) - L(s_1)}{s_1+s_2}, \\
L^\vep_{1,2}(s_1, s_2) &=& e^{s_1} \frac{L(s_2) - L(-s_1)}{s_1+s_2}.  
\end{eqnarray*}
\end{lemma}

\begin{proof}
Writing $L$ as a Fourier transform $L(v) = \int e^{-itv} g(t)\, dt$,  and 
using Duhamel's formula and trace property of $\vphi_0$, we can write under 
the latter:  
\begin{eqnarray*}
&&e^h  \left ( \dep L(\nabep)(x_1 ) \right )  x_2 \\
&=&
e^h \int \int_0^1 it \, \sigma_{ut} \ad_a \sigma_{(1-u)t} (x_1) \, g(t)\, x_2 \, du \, dt \\
&=& e^h \int \int_0^1 it \, a\, \sigma_{t-ut} (x_1)  \sigma_{-ut} (x_2) \, g(t) \, du \, dt \\
&&- e^h \int \int_0^1 it \, \sigma_{t-ut} (x_1)  \, a\, \sigma_{-ut} (x_2) \, g(t) \, du \, dt \\
&=& a  L_1(\nab, \nab) (x_1 x_2 e^h) + a  L_2(\nab, \nab) (x_2  e^h x_1),  
\end{eqnarray*}
where 
\begin{eqnarray*}
 L_1( s_1, s_2) 
&:=& 
\int_0^1 \int it e^{-i(t-ut)s_1+iuts_2} g(t)\, dt \, du
=
\frac{G(-s_2)-G(s_1)}{s_1+s_2},\\
 L_2( s_1, s_2) 
&:=& 
-\int_0^1 \int it e^{iuts_1-i(t-ut)s_2} g(t)\, dt \, du
=
\frac{G(s_2)-G(-s_1)}{s_1+s_2}. 
\end{eqnarray*}

\end{proof}

We need a version of the above lemma that involves functions of higher number of 
variables involved. In the following we treat the two and three variable cases, which 
will suffice for our purposes in Section \ref{GradientCalculationSec}. 

\begin{lemma} \label{GrofFuncCalc2} Let $L(s_1, s_2)$ be a smooth function and 
$x_1, x_2, x_3$ be elements in $C(\NT)$. 
Under the trace $\varphi_0$, we have: 
\begin{eqnarray*}
e^h  \left ( \dep L(\nabep, \nabep)(x_1 \cdot x_2) \right )  x_3 
&=& a e^h L^\vep_{2,1}(\nab, \nab, \nab)(x_1 \cdot x_2 \cdot x_3) \\
&&+
a e^h L^\vep_{2,2}(\nab, \nab, \nab)(x_2 \cdot x_3 \cdot x_1) \\
&&+
a e^h L^\vep_{2,3}(\nab, \nab, \nab)(x_3 \cdot x_1 \cdot x_2),  
\end{eqnarray*}
where 
\begin{eqnarray*}
L^\vep_{2,1}(s_1, s_2, s_3) &:=& e^{s_1+s_2+s_3} \frac{L(-s_2-s_3, s_2) - L(s_1, s_2)}{s_1+s_2+s_3}, \\
L^\vep_{2,2}(s_1, s_2, s_3) &:=& e^{s_1+s_2} \frac{L(s_3, - s_2- s_3) - L(-s_1 - s_2, s_1)}{s_1+s_2+s_3}, \\
L^\vep_{2,3}(s_1, s_2, s_3) &:=& e^{s_1} \frac{L(s_2, s_3) - L(s_2, -s_1-s_2)}{s_1+s_2+s_3}. 
\end{eqnarray*}
\end{lemma}

\begin{proof}

Writing $L(s_1, s_2) = \int e^{-i t_1 s_1 - i t_2 s_2} g(t_1, t_2) \, dt_1 \, dt_2$, we have
\begin{eqnarray*}
L(\nabep, \nabep)(x_1 x_2) 
= 
\int e^{-it_1 \nabep}(x_1) e^{-it_2 \nabep}(x_2) g(t_1, t_2) \, dt_1 \, dt_2. 
\end{eqnarray*}
We can then use the Duhamel formula to write the following equalities under the trace $\vphi_0$:
\begin{eqnarray*}
&&\dep L(\nabep, \nabep)(x_1 x_2) \\
&=&  \int \left ( \dep e^{-it_1 \nabep}(x_1) \right ) e^{-it_2 \nabep}(x_2) g(t_1, t_2) \, dt_1 \, dt_2 \\
&&+ 
\int e^{-it_1 \nabep}(x_1) \left ( \dep e^{-it_2 \nabep}(x_2) \right ) g(t_1, t_2) \, dt_1 \, dt_2 \\
&=& \int \int_0^1 i t_1 \sigma_{ut_1} \ad_a \sigma_{(1-u)t_1}(x_1) \, \sigma_{t_2}(x_2) g(t_1, t_2) \, du \, dt_1 \,dt_2 \\
&&+ \int \int_0^1 i t_2 \sigma_{t_1}(x_1) \sigma_{ut_2} \ad_a \sigma_{(1-u)t_2}(x_2)  g(t_1, t_2) \, du\, dt_1 \,dt_2 
\end{eqnarray*}
\begin{eqnarray*}
&=& \int_0^1 \int i t_1 \sigma_{ut_1} (a) \sigma_{t_1}(x_1) \sigma_{t_2}(x_2) g(t_1, t_2) \, 
dt_1 \, dt_2\, du 
\\
&&- 
\int_0^1 \int i t_1 \sigma_{t_1} (x_1) \sigma_{ut_1}(a) \sigma_{t_2}(x_2) g(t_1, t_2) \, 
dt_1 \, dt_2\, du 
\\
&&+ 
\int_0^1 \int i t_2 \sigma_{t_1} (x_1) \sigma_{ut_2}(a) \sigma_{t_2}(x_2) g(t_1, t_2) \, 
dt_1 \, dt_2\, du \\
&&- 
\int_0^1 \int i t_2 \sigma_{t_1} (x_1) \sigma_{t_2}(x_2) \sigma_{ut_2}(a) g(t_1, t_2) \, 
dt_1 \, dt_2\, du 
\end{eqnarray*}
\begin{eqnarray*}
&=& L_1(\nab, \nab, \nab) (a x_1 x_2) + L_2(\nab, \nab, \nab) ( x_1 a x_2) \\
&& + 
L_3(\nab, \nab, \nab) ( x_1 a x_2) + L_4(\nab, \nab, \nab) ( x_1 x_2 a),  
\end{eqnarray*}
where 
\begin{eqnarray*}
L_1(s_1, s_2, s_3) &=& \int_0^1 \int it_1 e^{-iut_1 s_1-it_1 s_2 -it_2s_3} g(t_1, t_2) 
\, dt_1 \, dt_2 \, du \\
&=& -\frac{L(s_1+s_2, s_3) - L(s_2, s_3)}{s_1}, 
\end{eqnarray*}
and with similar calculations we have 
\begin{eqnarray*}
L_2(s_1, s_2, s_3) &=& \frac{L(s_1+s_2, s_3) - L(s_1, s_3)}{s_2}, \\
L_3(s_1, s_2, s_3) &=& -\frac{L(s_1, s_2 + s_3) - L(s_1, s_3)}{s_2}, \\
L_4(s_1, s_2, s_3) &=& \frac{L(s_1, s_2 + s_3) - L(s_1, s_2)}{s_3}. 
\end{eqnarray*}
The statement of the lemma then follows if one uses the trace property of 
$\vphi_0$ to move the element $a$ cyclically to the very left, and then uses the modular 
automorphism. 
\end{proof}

We also need a version of the above lemmas for the case when $L$ is a function 
of three variables, which can be proved in a similar way:  

\begin{lemma}  \label{GrofFuncCalc3} Let $L(s_1, s_2, s_3)$ be a smooth function and 
$x_1, x_2, x_3, x_4$ be elements in $C(\NT)$. 
Under the trace $\varphi_0$, we have: 
\begin{eqnarray*}
&& e^h  \left ( \dep L(\nabep, \nabep, \nabep)(x_1 \cdot x_2 \cdot x_3) \right )  x_4 \\
&=& a e^h L^\vep_{3,1}(\nab, \nab, \nab, \nab)(x_1 \cdot x_2 \cdot x_3 \cdot x_4) 
+
a e^h L^\vep_{3,2}(\nab, \nab, \nab, \nab)(x_2 \cdot x_3 \cdot x_4 \cdot x_1) \\
&&+
a e^h L^\vep_{3,3}(\nab, \nab, \nab, \nab)(x_3 \cdot x_4 \cdot x_1 \cdot x_2) 
+
a e^h L^\vep_{3,4}(\nab, \nab, \nab, \nab)(x_4 \cdot x_1 \cdot x_2 \cdot x_3),    
\end{eqnarray*}
where 
\begin{eqnarray*}
L^\vep_{3,1}(s_1, s_2, s_3, s_4) &:=& e^{s_1+s_2+s_3+s_4} \frac{L(-s_2-s_3-s_4, s_2, s_3) - L(s_1, s_2, s_3)}{s_1+s_2+s_3+s_4}, \\
L^\vep_{3,2}(s_1, s_2, s_3, s_4) &:=& e^{s_1+s_2+s_3} \frac{L(s_4, - s_2- s_3 -s_4, s_2) - L(-s_1 - s_2 -s_3, s_1, s_2)}{s_1+s_2+s_3+s_4}, \\
L^\vep_{3,3}(s_1, s_2, s_3, s_4) &:=& e^{s_1+s_2} \frac{L( s_3, s_4, -s_2 -s_3 -s_4) - L(s_3, -s_1-s_2-s_3, s_1)}{s_1+s_2+s_3+s_4}, \\
L^\vep_{3,4}(s_1, s_2, s_3, s_4) &:=& e^{s_1} \frac{L(s_2, s_3, s_4) - L(s_2, s_3, -s_1-s_2-s_3)}{s_1+s_2+s_3+s_4}. 
\end{eqnarray*}
\end{lemma}

It should be noted that the reason for considering the particular expressions in the 
statement of the above lemmas is that we wish to prepare the ground for making 
the comparison between the gradients calculated in the beginning of this Section 
and in Section \ref{GradientCalculationSec}, and derive the functional relations 
presented in Section \ref{a_4Sec}.

\subsection{Derivatives $\delta_j$ of functional calculi with $\nab$} For our purposes, 
given a function of the modular automorphism acting on an 
element of $C(\NT)$, it is important to have an explicit formula for the  derivative 
$\delta_j,$ $j=1, 2,$  of such a term. In this subsection we work out  the necessary 
explicit formulas. 
The one variable case, which is given in the following lemma, was also proved in 
\cite{ConMosModular}  
and it is given here for the sake of explaining the idea in the simplest case. We present 
the cases when two and three variable functions are involved in Lemma \ref{DerFuncCalc2} and 
Lemma \ref{DerFuncCalc3}. 

\begin{lemma}
\label{DerFuncCalc1}
Let $L(s_1)$ be a smooth function and 
$x_1$ be an element of the algebra  $C(\NT)$ of the  noncommutative torus. We have: 
\begin{eqnarray*}
\delta_j \left (  L( \nab)(x_1) \right ) 
=
L(\nab)( \delta_j(x_1) )  
+ L^{\delta}_{1, 1}(\nab, \nab) (\delta_j(h) \cdot x_1)  
+ L^{\delta}_{1, 2}(\nab, \nab) (  x_1 \cdot \delta_j(h)),  
\end{eqnarray*}
where 
\begin{eqnarray*}
L^{\delta}_{1, 1}(s_1, s_2) &:=& \frac{L(s_2) - L(s_1+s_2)}{s_1}, \\
L^{\delta}_{1, 2}(s_1, s_2) &:=& \frac{L(s_1+s_2) - L(s_1)}{s_2}. 
\end{eqnarray*}

\end{lemma}
\begin{proof}
One can start by writing 
\begin{eqnarray*}
\delta_j \sigma_t(x_1) - \sigma_t \del_j(x_1) &=& 
- \int_0^1 e^{-iut \nab} \,[\delta_j, it \nab] \, e^{-i(1-u)t\nab}(x_1) \, du  \\
&=&
it \int_0^1 \sigma_{ut} \ad_{\delta_j(h)} \sigma_{(1-u)t}(x_1) \, du. 
\end{eqnarray*}
Therefore, after writing $L(s_1) = \int e^{-it s_1} g(t)\, dt$, one gets
\begin{eqnarray*}
\del_j (L(\nab)(x_1)) - L(\nab)(\del_j(x_1)) &=& 
\int_0^1 \int it \sigma_{ut}(\delta_j(h)) \, \sigma_t(x_1) \, g(t) \, dt \, du \\
&&- \int_0^1 \int it \, \sigma_{t}(x_1) \, \sigma_{ut}(\del_j(h)) \, g(t) \, dt \, du \\
&=& L^{\delta}_{1, 1}(\nab, \nab) (\delta_j(h) \cdot x_1) 
+ L^{\delta}_{1, 2}(\nab, \nab) (  x_1 \cdot \delta_j(h)),
\end{eqnarray*}
where 
\begin{eqnarray*}
L^{\delta}_{1, 1}(s_1, s_2) &=& \int_0^1 \int it e^{-iuts_1-itS_2} g(t) \, dt \, du = 
 \frac{L(s_2) - L(s_1+s_2)}{s_1}, \\
L^{\delta}_{1, 2}(s_1, s_2) &=& - \int_0^1 \int it e^{-its_1-iutS_2} g(t) \, dt \, du = 
 \frac{L(s_1+s_2) - L(s_1)}{s_2}.
\end{eqnarray*}
\end{proof}

Now we prove a generalization of the above lemma for the case when the function $L$ 
depends smoothly on two variables. 

\begin{lemma}\label{DerFuncCalc2}
Let $L(s_1, s_2)$ be a smooth function and 
$x_1, x_2$ be elements of  $C(\NT)$. We have: 
\begin{eqnarray*}
&&\delta_j \left (  L(\nab, \nab)(x_1 \cdot x_2) \right ) \\
&=& 
L(\nab, \nab)( \delta_j(x_1) \cdot x_2) 
+ L(\nab, \nab)( x_1 \cdot \delta_j(x_2)) 
+ L^{\delta}_{2, 1}(\nab, \nab, \nab) (\delta_j(h) \cdot x_1 \cdot x_2) \\
&& 
+ L^{\delta}_{2, 2}(\nab, \nab, \nab) (  x_1 \cdot \delta_j(h) \cdot x_2)  
+ L^{\delta}_{2, 3}(\nab, \nab, \nab) (  x_1  \cdot x_2 \cdot \delta_j(h)),  
\end{eqnarray*}
where 
\begin{eqnarray*}
L^{\delta}_{2, 1}(s_1, s_2, s_3) &:=& \frac{L(s_2, s_3) - L(s_1+s_2, s_3)}{s_1}, \\
L^{\delta}_{2, 2}(s_1, s_2, s_3) &:=& \frac{L(s_1+s_2, s_3) - L(s_1, s_2+s_3)}{s_2}, \\
L^{\delta}_{2, 3}(s_1, s_2, s_3) &:=& \frac{L(s_1, s_2+s_3) - L(s_1, s_2)}{s_3}. 
\end{eqnarray*}

\end{lemma}

\begin{proof}
We start by considering the identity 
\begin{eqnarray*}
&&\delta_j \left ( \sigma_{t_1}(x_1) \sigma_{t_2}(x_2) \right ) 
- \sigma_{t_1}(\delta_j(x_1)) \sigma_{t_2}(x_2) 
- \sigma_{t_1}(x_1) \sigma_{t_2}(\delta_j(x_2)) \\
&=& [\del_j, \sigma_{t_1}](x_1)\, \sigma_{t_2}(x_2) + 
\sigma_{t_1}(x_1)\, [\delta_j, \sigma_{t_2}](x_2). 
\end{eqnarray*}
Therefore, by writing $L(s_1, s_2) = \int e^{-i t_1 s_1 - i t_2 s_2} g(t_1, t_2) \, dt_1 \, dt_2$, 
we have 
\begin{eqnarray*}
&&\delta_j L(\nab, \nab)(x_1 x_2) - L(\nab, \nab)(\delta_j(x_1) x_2) - L(\nab, \nab)(x_1 \delta_j(x_2) ) \\
&=& \int it_1 \int_0^1 \sigma_{ut_1} \ad_{\delta_j(h)} \sigma_{(1-u)t_1}(x_1) \, du \, \sigma_{t_2} (x_2) \, g(t_1, t_2) \, dt_1 \, dt_2 \\
&&+ \int \sigma_{t_1}(x_1) it_2 \int_0^1 \sigma_{ut_2} \ad_{\delta_j(h)}  \sigma_{(1-u)t_2}(x_2) \, du \, g(t_1, t_2) \, dt_1 \, dt_2 
\end{eqnarray*}
\begin{eqnarray*}
&=&\int_0^1 \int i t_1 \sigma_{ut_1} (\delta_j(h)) \sigma_{t_1}(x_1)  \sigma_{t_2}(x_2) g(t_1, t_2) \, dt_1 \, dt_2 \, du \\
&& - \int_0^1 \int i t_1 \sigma_{t_1} (x_1) \sigma_{ut_1}(\delta_j(h))  \sigma_{t_2}(x_2) g(t_1, t_2) \, dt_1 \, dt_2 \, du \\
&&+\int_0^1 \int i t_2 \sigma_{t_1} (x_1) \sigma_{ut_2}(\delta_j(h))  \sigma_{t_2}(x_2) g(t_1, t_2) \, dt_1 \, dt_2 \, du \\
&&- \int_0^1 \int i t_2 \sigma_{t_1} (x_1) \sigma_{t_2}(x_2)  \sigma_{ut_2}(\delta_j(h)) g(t_1, t_2) \, dt_1 \, dt_2 \, du
\end{eqnarray*}
\begin{eqnarray*}
&=&
 L^{\delta}_{2, 1}(\nab, \nab, \nab) (\delta_j(h) \cdot x_1 \cdot x_2) 
+ L^{\delta}_{2, 2}(\nab, \nab, \nab) (  x_1 \cdot \delta_j(h) \cdot x_2)  \\
&&+ L^{\delta}_{2, 3}(\nab, \nab, \nab) (  x_1  \cdot x_2 \cdot \delta_j(h)),  
\end{eqnarray*}
where 
\begin{eqnarray*}
L^{\delta}_{2, 1}(s_1, s_2, s_3) &=& 
\int_0^1 \int i t_1 e^{-iut_1s_1 -it_1s_2 -it_2s_3} g(t_1, t_2) \,dt_1 \, dt_2 \, du \\
&=&  \frac{L(s_2, s_3) - L(s_1+s_2, s_3)}{s_1}, 
\end{eqnarray*}
\begin{eqnarray*}
L^{\delta}_{2, 2}(s_1, s_2, s_3) &=& 
-\int_0^1 \int i t_1 e^{-it_1s_1 -iut_1s_2 -it_2s_3} g(t_1, t_2) \,dt_1 \, dt_2 \, du \\
&&+\int_0^1 \int i t_2 e^{-it_1s_1 -iut_2s_2 -it_2s_3} g(t_1, t_2) \,dt_1 \, dt_2 \, du \\
&=& \frac{L(s_1+s_2, s_3) - L(s_1, s_3)}{s_2} -  \frac{L(s_1,s_2+ s_3) - L(s_1, s_3)}{s_2} \\
&=&  \frac{L(s_1+s_2, s_3) - L(s_1, s_2+s_3)}{s_2},
\end{eqnarray*}
\begin{eqnarray*}
L^{\delta}_{2, 3}(s_1, s_2, s_3) &=& 
-\int_0^1 \int i t_2 e^{-it_1s_1 -it_2s_2 -iut_2s_3} g(t_1, t_2) \,dt_1 \, dt_2 \, du \\
&=&  \frac{L(s_1, s_2+ s_3) - L(s_1, s_2)}{s_3}.
\end{eqnarray*}

\end{proof}

Finally we treat the case when a smooth three variable function is involved, and find an 
explicit formula for the derivative $\delta_j,$ $j=1, 2$, of a general associated element 
of the noncommutative torus. 

 \begin{lemma} \label{DerFuncCalc3}
Let $L(s_1, s_2, s_3)$ be a smooth function and 
$x_1, x_2, x_3$ be elements of  $C(\NT)$. We have: 
\begin{eqnarray*}
&&\delta_j \left (  L(\nab, \nab, \nab )(x_1 \cdot x_2 \cdot x_3) \right ) \\
&=& 
L(\nab, \nab, \nab)( \delta_j(x_1) \cdot x_2 \cdot x_3) 
+ L(\nab, \nab, \nab)( x_1 \cdot \delta_j(x_2) \cdot x_3) \\
&&
+ L(\nab, \nab, \nab)( x_1 \cdot x_2 \cdot \delta_j(x_3))
+ L^{\delta}_{3, 1}(\nab, \nab, \nab, \nab) (\delta_j(h) \cdot x_1 \cdot x_2 \cdot x_3) \\
&& 
+ L^{\delta}_{3, 2}(\nab, \nab, \nab, \nab) (  x_1 \cdot \delta_j(h) \cdot x_2 \cdot x_3) 
+ L^{\delta}_{3, 3}(\nab, \nab, \nab, \nab) (  x_1  \cdot x_2 \cdot \delta_j(h) \cdot x_3) \\
&& 
+ L^{\delta}_{3, 4}(\nab, \nab, \nab, \nab) (  x_1  \cdot x_2 \cdot x_3 \cdot \delta_j(h) ),  
\end{eqnarray*}
where 
\begin{eqnarray*}
L^{\delta}_{3, 1}(s_1, s_2, s_3, s_4) &:=& \frac{L(s_2, s_3, s_4) - L(s_1+s_2, s_3, s_4)}{s_1}, 
\\
L^{\delta}_{3, 2}(s_1, s_2, s_3, s_4) &:=& \frac{L(s_1+s_2, s_3, s_4) - L(s_1, s_2+s_3, s_4)}{s_2}, \\
L^{\delta}_{3, 3}(s_1, s_2, s_3, s_4) &:=& \frac{L(s_1, s_2+s_3, s_4) - L(s_1, s_2, s_3+s_4)}{s_3}, \\
L^{\delta}_{3, 4}(s_1, s_2, s_3, s_4) &:=& \frac{L(s_1, s_2, s_3+s_4) - L(s_1, s_2, s_3)}{s_4}. 
\end{eqnarray*}

\end{lemma}
\begin{proof}
One can prove this lemma by starting with the following identity and by 
continuing as in the proof of Lemma \ref{DerFuncCalc2}: 
\begin{eqnarray*}
&&\delta_j \left ( \sigma_{t_1}(x_1) \sigma_{t_2}(x_2) \sigma_{t_3}(x_3)\right ) 
- \sigma_{t_1}(\delta_j(x_1)) \sigma_{t_2}(x_2) \sigma_{t_3}(x_3) \\
&&- \sigma_{t_1}(x_1) \sigma_{t_2}(\delta_j(x_2)) \sigma_{t_3}(x_3) 
- \sigma_{t_1}(x_1) \sigma_{t_2}(x_2) \sigma_{t_3}(\del_j(x_3))  \\
&=& [\del_j, \sigma_{t_1}](x_1)\, \sigma_{t_2}(x_2)\sigma_{t_3}(x_3)  + 
\sigma_{t_1}(x_1)\, [\delta_j, \sigma_{t_2}](x_2) \sigma_{t_3}(x_3) \\
&&+ \sigma_{t_1}(x_1)\, \sigma_{t_2}(x_2) [\delta_j, \sigma_{t_3}](x_3) . 
\end{eqnarray*}
\end{proof}

It is interesting that in the explicit formulas derived in this subsection, certain finite 
differences of an original function determine the final formulas.

\section{Calculation of gradient of $h \mapsto \vphi_0(a_4)$ in terms of finite differences}
\label{GradientCalculationSec}

In the beginning of Section \ref{FuncRelationsProofSec}, we explained how, for selfadjoint 
elements $h, a \in \CNT$, the following 
gradient can be calculated by using an important 
identity proved in \cite{ConMosModular}: 
\[
 \dep \vphi_0(a_4(h+ \vep a)).
\]
In this section, we demonstrate a second way of calculating the above gradient, 
which is based on using the Duhamel formula and the lemmas proved in the 
subsections of Section 
\ref{FuncRelationsProofSec}. This method gives rise to expressions that involve 
finite differences, and by comparing the final outcome with the first formula 
\eqref{Gradientofa4} derived for 
the above gradient, we find the 
functional relations presented in Theorem \ref{FuncRelationsThm} in 
Section \ref{a_4Sec}. Then, we confirm 
the accuracy of the lengthy formulas presented in Section \ref{ExplicitFormulasSec} 
and in the appendices for the functions $K_1, \dots,  K_{20}$ 
appearing in formula \eqref{a_4expression} for the $a_4$, by 
checking that they satisfy the expected functional relations. 
Before starting the second method of calculating the desired gradient, 
we need some lemmas.

\subsection{Cyclic permutations and functional caculi with $\nab$}
The first type of lemmas that we will need concern exploiting the trace 
and invariance property of $\vphi_0$ to prepare the outcome of 
the second calculation of the mentioned gradient for comparison 
with the first formula given by \eqref{Gradientofa4}.

\begin{lemma} \label{LowerwithTrace1} 
Let $K$ be a smooth function of $n$ variables and $x_1, \dots, x_n$ belong to the noncommutative 
torus $C(\NT)$. If $n=1$, 
under the trace $\vphi_0$ we have 
\[
 K(\nab) (x_1) = K(0) \, x_1, 
\]
and if $n>1$, under $\vphi_0$ we have
\[
K(\nab, \dots, \nab)(x_1 \cdots x_n ) =  L(\nab, \dots, \nab) (x_1 \cdots x_{n-1}) \, x_n,
\]
where 
\[
L(s_1, \dots, s_{n-1}) := K(s_1, \dots, s_{n-1}, -s_1 - \cdots - s_{n-1}). 
\]
\end{lemma}
\begin{proof}
Writing  $K(s_1,  \dots, s_n) = \int e^{-it_1 s_1 -  \cdots- it_ns_n } g(t_1, \dots, t_n) \, dt_1  \cdots \, dt_n,$ we have 
\[ 
K(\nab, \dots, \nab)(x_1 \cdots x_n) = 
\int \sigma_{t_1}(x_1) \cdots \sigma_{t_n}(x_n) g(t_1, \dots, t_n) \, dt_1 \cdots \, dt_n.  
\]
Under the trace $\vphi_0$, the latter is equal to
\begin{eqnarray*}
&&\int \sigma_{t_1-t_n}(x_1) \cdots \sigma_{t_{n-1}-t_n}(x_{n-1}) g(t_1, \dots, t_n) \, dt_1 \cdots dt_n \, \cdot x_n \\
&=&  L(\nab, \dots, \nab) (x_1 \cdots x_{n-1})\,  x_n, 
\end{eqnarray*}
where
\begin{eqnarray*}
L(s_1, \dots, s_{n-1}) &=& \int e^{-i(t_1 - t_n) s_1- \cdots -i(t_{n-1}-t_n)s_n} \, g(t_1, \dots, t_n)\,dt_1  \cdots \, dt_n\\ 
&=& K(s_1, \dots, s_{n-1}, -s_1 - \cdots - s_{n-1}). 
\end{eqnarray*}
\end{proof}

 The following lemma  allows us to cyclically permute elements 
of $C(\NT)$ by varying appropriately the function of the modular automorphism that is 
involved in each calculation. 

\begin{lemma}\label{ShiftwithTrace}
Let $L$ be a smooth function of $n-1$ variables and $x_1, \dots, x_n$ belong 
to $C(\NT)$. For any $1 \leq j \leq n-1$, under the trace $\vphi_0$ we have
\[
L(\nab, \dots, \nab)(x_1 \cdots x_{n-1} )\, x_n =  x_j\, L^c_j (\nab, \dots, \nab) (x_{j+1} \cdots x_{n} \cdot x_1 \cdots x_{j-1}),  
\]
where 
\[
L^c_j(s_1, \dots, s_{n-1}) := L(s_{j+2}, \dots, s_{n-1} ,-s_1- s_2 - \cdots -s_{n-1}, s_1, s_2, \dots, s_j). 
\]
\end{lemma}
\begin{proof}
Writing 
\[
L(s_1,  \dots, s_{n-1}) = \int e^{-it_1 s_1 -  \cdots- it_{n-1}s_{n-1} } g(t_1, \dots, t_{n-1}) \, dt_1  \cdots \, dt_{n-1},
\] 
and using the trace property and the invariance of $\vphi_0$, under the latter 
we can write
\begin{eqnarray*}
&&L(\nab, \dots, \nab)(x_1 \cdots x_{n-1} )\, x_n  \\ 
&=&  \int \sigma_{t_1}(x_1) \cdots \sigma_{t_j}(x_j)\cdots \sigma_{t_{n-1}}(x_{n-1}) \, x_n \,g(t_1, \dots, t_{n-1}) \, dt_1 \cdots dt_j \cdots dt_{n-1} \\
&=& \int  \sigma_{t_j}(x_j)\cdots \sigma_{t_{n-1}}(x_{n-1}) \, x_n \, \sigma_{t_1}(x_1) \cdots \sigma_{t_{j-1}}(x_{j-1}) \, g(t_1, \dots, t_{n-1}) \, dt_1  \cdots dt_{n-1} \\
&=& x_j \int  \sigma_{t_{j+1}-t_j}(x_{j+1})\cdots \sigma_{t_{n-1}-t_j}(x_{n-1}) \, \sigma_{-t_j} (x_n) \, \sigma_{t_1-t_j}(x_1) \cdots \sigma_{t_{j-1}-t_j}(x_{j-1}) \times \, \\
&& \qquad \qquad \qquad \qquad \qquad \qquad    g(t_1, \dots, t_{n-1}) \, dt_1  \cdots dt_{n-1} \\ 
&=&x_j\, L^c_j (\nab, \dots, \nab) (x_{j+1} \cdots x_{n} \cdot x_1 \cdots x_{j-1}), 
\end{eqnarray*}
where 
\begin{eqnarray*}
&&L^c_j(s_1, \dots, s_{n-1}) \\
&=& 
\int e^{-i(t_{j+1}-t_j)s_1- \cdots - i(t_{n-1}-t_j) s_{n-1-j} + i t_j s_{n-j} - i (t_1-t_j)s_{n-j+1}- \cdots - i (t_{j-1}-t_j)s_{n-1} }\, \times \\
&& \qquad \qquad \qquad \qquad \qquad \qquad   g(t_1, \dots, t_{n-1}) \, dt_1  \cdots dt_{n-1}\\
&=& L(s_{j+2}, \dots, s_{n-1} ,-s_1- s_2 - \cdots -s_{n-1}, s_1, s_2, \dots, s_j). 
\end{eqnarray*}
\end{proof}

\subsection{Finite differences of $G_1(\nab)$ relating derivatives 
of $h$ and $e^{h}$} 
The second type of lemmas that we shall 
need in the remainder of this section concern writing the 
derivatives up to order four of the conformal factor $e^h \in \CNT$ in terms 
of the derivatives of the dilaton $h$.  
In \cite{ConMosModular, FatKhaSC2T}, for the purpose of 
writing the final formula for the scalar curvature in terms 
of derivatives of $h$, 
the functions $G_1(s_1)$ and $G_2(s_1, s_2)$ were found such that  
for $j_1, j_2 \in \{1, 2 \}$:
\[
e^{-h} \delta_{j_1}(e^h) = G_1(\nab ) ( \delta _{j_1}(h) ), 
\]
\[
e^{-h}  \delta_{j_1} \delta_{j_2} (e^h) = G_1(\nab) ( \delta _{j_1} \delta _{j_2}(h) ) + 
G_2(\nab, \nab) \left ( \delta _{j_1}(h) \cdot \delta _{j_2}(h)+\delta _{j_2}(h) \cdot \delta
   _{j_1}(h) \right ). 
\]

\smallskip

For our purposes in this paper,  we will need in the sequel  to treat terms 
of the form $e^{-h}  \delta_{j_1} \delta_{j_2} \delta_{j_3} (e^h)$ 
and of the form $e^{-h}  \delta_{j_1} \delta_{j_2} \delta_{j_3} \delta_{j_4}(e^h) $. 
This is done in the following lemmas,  where we first give a general form of 
these expressions in terms of derivatives of $h$ and will then describe the 
functions $G_1, G_2, G_3, G_4$ in 
Lemma \ref{BabyFunctions1} and Lemma \ref{BabyFunctions2}. 

\begin{lemma} \label{ToLogBaby}
Let $j_1, j_2, j_3, j_4$ be either equal to 1 or 2. Using the functions 
$G_1, G_2,$ $G_3, G_4$ presented in Lemma \ref{BabyFunctions1} and Lemma \ref{BabyFunctions2}, we can write: 
\begin{eqnarray*}
&& e^{-h}  \delta_{j_1} \delta_{j_2} \delta_{j_3} (e^h) 
= 
G_1(\nab) ( \Box_{3, 1}(h) ) 
+ 
G_2(\nab, \nab) ( \Box_{3, 2}(h) )  
+ 
G_3(\nab, \nab, \nab) ( \Box_{3, 3}(h) ),
\end{eqnarray*}
where \\

$
\Box_{3, 1}(h) := \delta _{j_1} \delta _{j_2} \delta _{j_3}(h), 
$\\

$\Box_{3, 2}(h) := \delta _{j_1}(h)\cdot \left(\delta _{j_2} \delta
   _{j_3}\right)(h)+\delta
   _{j_2}(h)\cdot \left(\delta _{j_1} \delta
   _{j_3}\right)(h)+\left(\delta _{j_1}
   \delta _{j_2}\right)(h)\cdot \delta
   _{j_3}(h)+\delta _{j_3}(h)\cdot \left(\delta
   _{j_1} \delta
   _{j_2}\right)(h)+\left(\delta _{j_1}
   \delta _{j_3}\right)(h)\cdot \delta
   _{j_2}(h)+\left(\delta _{j_2} \delta
   _{j_3}\right)(h)\cdot \delta _{j_1}(h), $\\

$ \Box_{3, 3}(h) := \delta _{j_1}(h)\cdot \delta _{j_2}(h)\cdot \delta
   _{j_3}(h)+\delta _{j_1}(h)\cdot \delta
   _{j_3}(h)\cdot \delta _{j_2}(h)+\delta
   _{j_2}(h)\cdot \delta _{j_1}(h)\cdot \delta
   _{j_3}(h)+\delta _{j_2}(h)\cdot \delta
   _{j_3}(h)\cdot \delta _{j_1}(h)+\delta
   _{j_3}(h)\cdot \delta _{j_1}(h)\cdot \delta
   _{j_2}(h)+\delta _{j_3}(h)\cdot \delta
   _{j_2}(h)\cdot \delta _{j_1}(h). $\\

\smallskip

Moreover, for the case when the order of differentiation is 4, we have: 
\begin{eqnarray*}
 e^{-h}  \delta_{j_1} \delta_{j_2} \delta_{j_3} \delta_{j_4} (e^h) 
&=& 
G_1(\nab) ( \Box_{4, 1}(h) ) 
+ 
G_2(\nab, \nab) ( \Box_{4, 2}(h) )  
\\ 
&&
+ 
G_3(\nab, \nab, \nab) ( \Box_{4, 3}(h) ) 
+
G_4(\nab, \nab, \nab, \nab) ( \Box_{4, 4}(h) ),  
\end{eqnarray*}
where 
\[
\Box_{4, 1} (h) := \left(\delta _{j_1} \delta _{j_2} \delta
   _{j_3} \delta _{j_4}\right)(h),
\] \\

$
\Box_{4, 2} (h) := 
\delta _{j_1}(h)\cdot \left(\delta _{j_2} \delta
   _{j_3} \delta _{j_4}\right)(h)+\delta
   _{j_2}(h)\cdot \left(\delta _{j_1} \delta
   _{j_3} \delta
   _{j_4}\right)(h)+\left(\delta _{j_1}
   \delta _{j_2}\right)(h)\cdot \left(\delta
   _{j_3} \delta _{j_4}\right)(h)+\delta
   _{j_3}(h)\cdot \left(\delta _{j_1} \delta
   _{j_2} \delta
   _{j_4}\right)(h)+\left(\delta _{j_1}
   \delta _{j_3}\right)(h)\cdot \left(\delta
   _{j_2} \delta
   _{j_4}\right)(h)+\left(\delta _{j_2}
   \delta _{j_3}\right)(h)\cdot \left(\delta
   _{j_1} \delta
   _{j_4}\right)(h)+\left(\delta _{j_1}
   \delta _{j_2} \delta
   _{j_3}\right)(h)\cdot \delta _{j_4}(h)+\delta
   _{j_4}(h)\cdot \left(\delta _{j_1} \delta
   _{j_2} \delta
   _{j_3}\right)(h)+\left(\delta _{j_1}
   \delta _{j_4}\right)(h)\cdot \left(\delta
   _{j_2} \delta
   _{j_3}\right)(h)+\left(\delta _{j_2}
   \delta _{j_4}\right)(h)\cdot \left(\delta
   _{j_1} \delta
   _{j_3}\right)(h)+\left(\delta _{j_1}
   \delta _{j_2} \delta
   _{j_4}\right)(h)\cdot \delta
   _{j_3}(h)+\left(\delta _{j_3} \delta
   _{j_4}\right)(h)\cdot \left(\delta _{j_1}
   \delta _{j_2}\right)(h)+\left(\delta
   _{j_1} \delta _{j_3} \delta
   _{j_4}\right)(h)\cdot \delta
   _{j_2}(h)+\left(\delta _{j_2} \delta
   _{j_3} \delta _{j_4}\right)(h)\cdot \delta
   _{j_1}(h), 
$\\

$
\Box_{4,3}(h) := 
\delta _{j_1}(h)\cdot \delta
   _{j_2}(h)\cdot \left(\delta _{j_3} \delta
   _{j_4}\right)(h)+\delta _{j_1}(h)\cdot \delta
   _{j_3}(h)\cdot \left(\delta _{j_2} \delta
   _{j_4}\right)(h)+\delta
   _{j_1}(h)\cdot \left(\delta _{j_2} \delta
   _{j_3}\right)(h)\cdot \delta _{j_4}(h)+\delta
   _{j_1}(h)\cdot \delta _{j_4}(h)\cdot \left(\delta
   _{j_2} \delta _{j_3}\right)(h)+\delta
   _{j_1}(h)\cdot \left(\delta _{j_2} \delta
   _{j_4}\right)(h)\cdot \delta _{j_3}(h)+\delta
   _{j_1}(h)\cdot \left(\delta _{j_3} \delta
   _{j_4}\right)(h)\cdot \delta _{j_2}(h)+\delta
   _{j_2}(h)\cdot \delta _{j_1}(h)\cdot \left(\delta
   _{j_3} \delta _{j_4}\right)(h)+\delta
   _{j_2}(h)\cdot \delta _{j_3}(h)\cdot \left(\delta
   _{j_1} \delta _{j_4}\right)(h)+\delta
   _{j_2}(h)\cdot \left(\delta _{j_1} \delta
   _{j_3}\right)(h)\cdot \delta _{j_4}(h)+\delta
   _{j_2}(h)\cdot \delta _{j_4}(h)\cdot \left(\delta
   _{j_1} \delta _{j_3}\right)(h)+\delta
   _{j_2}(h)\cdot \left(\delta _{j_1} \delta
   _{j_4}\right)(h)\cdot \delta _{j_3}(h)+\delta
   _{j_2}(h)\cdot \left(\delta _{j_3} \delta
   _{j_4}\right)(h)\cdot \delta
   _{j_1}(h)+\left(\delta _{j_1} \delta
   _{j_2}\right)(h)\cdot \delta _{j_3}(h)\cdot \delta
   _{j_4}(h)+\left(\delta _{j_1} \delta
   _{j_2}\right)(h)\cdot \delta _{j_4}(h)\cdot \delta
   _{j_3}(h)+\delta _{j_3}(h)\cdot \delta
   _{j_1}(h)\cdot \left(\delta _{j_2} \delta
   _{j_4}\right)(h)+\delta _{j_3}(h)\cdot \delta
   _{j_2}(h)\cdot \left(\delta _{j_1} \delta
   _{j_4}\right)(h)+\delta
   _{j_3}(h)\cdot \left(\delta _{j_1} \delta
   _{j_2}\right)(h)\cdot \delta _{j_4}(h)+\delta
   _{j_3}(h)\cdot \delta _{j_4}(h)\cdot \left(\delta
   _{j_1} \delta _{j_2}\right)(h)+\delta
   _{j_3}(h)\cdot \left(\delta _{j_1} \delta
   _{j_4}\right)(h)\cdot \delta _{j_2}(h)+\delta
   _{j_3}(h)\cdot \left(\delta _{j_2} \delta
   _{j_4}\right)(h)\cdot \delta
   _{j_1}(h)+\left(\delta _{j_1} \delta
   _{j_3}\right)(h)\cdot \delta _{j_2}(h)\cdot \delta
   _{j_4}(h)+\left(\delta _{j_1} \delta
   _{j_3}\right)(h)\cdot \delta _{j_4}(h)\cdot \delta
   _{j_2}(h)+\left(\delta _{j_2} \delta
   _{j_3}\right)(h)\cdot \delta _{j_1}(h)\cdot \delta
   _{j_4}(h)+\left(\delta _{j_2} \delta
   _{j_3}\right)(h)\cdot \delta _{j_4}(h)\cdot \delta
   _{j_1}(h)+\delta _{j_4}(h)\cdot \delta
   _{j_1}(h)\cdot \left(\delta _{j_2} \delta
   _{j_3}\right)(h)+\delta _{j_4}(h)\cdot \delta
   _{j_2}(h)\cdot \left(\delta _{j_1} \delta
   _{j_3}\right)(h)+\delta
   _{j_4}(h)\cdot \left(\delta _{j_1} \delta
   _{j_2}\right)(h)\cdot \delta _{j_3}(h)+\delta
   _{j_4}(h)\cdot \delta _{j_3}(h)\cdot \left(\delta
   _{j_1} \delta _{j_2}\right)(h)+\delta
   _{j_4}(h)\cdot \left(\delta _{j_1} \delta
   _{j_3}\right)(h)\cdot \delta _{j_2}(h)+\delta
   _{j_4}(h)\cdot \left(\delta _{j_2} \delta
   _{j_3}\right)(h)\cdot \delta
   _{j_1}(h)+\left(\delta _{j_1} \delta
   _{j_4}\right)(h)\cdot \delta _{j_2}(h)\cdot \delta
   _{j_3}(h)+\left(\delta _{j_1} \delta
   _{j_4}\right)(h)\cdot \delta _{j_3}(h)\cdot \delta
   _{j_2}(h)+\left(\delta _{j_2} \delta
   _{j_4}\right)(h)\cdot \delta _{j_1}(h)\cdot \delta
   _{j_3}(h)+\left(\delta _{j_2} \delta
   _{j_4}\right)(h)\cdot \delta _{j_3}(h)\cdot \delta
   _{j_1}(h)+\left(\delta _{j_3} \delta
   _{j_4}\right)(h)\cdot \delta _{j_1}(h)\cdot \delta
   _{j_2}(h)+\left(\delta _{j_3} \delta
   _{j_4}\right)(h)\cdot \delta _{j_2}(h)\cdot \delta
   _{j_1}(h), 
$\\

$\Box_{4, 4}(h) :=
\delta _{j_1}(h)\cdot \delta _{j_2}(h)\cdot \delta
   _{j_3}(h)\cdot \delta _{j_4}(h)+\delta
   _{j_1}(h)\cdot \delta _{j_2}(h)\cdot \delta
   _{j_4}(h)\cdot \delta _{j_3}(h)+\delta
   _{j_1}(h)\cdot \delta _{j_3}(h)\cdot \delta
   _{j_2}(h)\cdot \delta _{j_4}(h)+\delta
   _{j_1}(h)\cdot \delta _{j_3}(h)\cdot \delta
   _{j_4}(h)\cdot \delta _{j_2}(h)+\delta
   _{j_1}(h)\cdot \delta _{j_4}(h)\cdot \delta
   _{j_2}(h)\cdot \delta _{j_3}(h)+\delta
   _{j_1}(h)\cdot \delta _{j_4}(h)\cdot \delta
   _{j_3}(h)\cdot \delta _{j_2}(h)+\delta
   _{j_2}(h)\cdot \delta _{j_1}(h)\cdot \delta
   _{j_3}(h)\cdot \delta _{j_4}(h)+\delta
   _{j_2}(h)\cdot \delta _{j_1}(h)\cdot \delta
   _{j_4}(h)\cdot \delta _{j_3}(h)+\delta
   _{j_2}(h)\cdot \delta _{j_3}(h)\cdot \delta
   _{j_1}(h)\cdot \delta _{j_4}(h)+\delta
   _{j_2}(h)\cdot \delta _{j_3}(h)\cdot \delta
   _{j_4}(h)\cdot \delta _{j_1}(h)+\delta
   _{j_2}(h)\cdot \delta _{j_4}(h)\cdot \delta
   _{j_1}(h)\cdot \delta _{j_3}(h)+\delta
   _{j_2}(h)\cdot \delta _{j_4}(h)\cdot \delta
   _{j_3}(h)\cdot \delta _{j_1}(h)+\delta
   _{j_3}(h)\cdot \delta _{j_1}(h)\cdot \delta
   _{j_2}(h)\cdot \delta _{j_4}(h)+\delta
   _{j_3}(h)\cdot \delta _{j_1}(h)\cdot \delta
   _{j_4}(h)\cdot \delta _{j_2}(h)+\delta
   _{j_3}(h)\cdot \delta _{j_2}(h)\cdot \delta
   _{j_1}(h)\cdot \delta _{j_4}(h)+\delta
   _{j_3}(h)\cdot \delta _{j_2}(h)\cdot \delta
   _{j_4}(h)\cdot \delta _{j_1}(h)+\delta
   _{j_3}(h)\cdot \delta _{j_4}(h)\cdot \delta
   _{j_1}(h)\cdot \delta _{j_2}(h)+\delta
   _{j_3}(h)\cdot \delta _{j_4}(h)\cdot \delta
   _{j_2}(h)\cdot \delta _{j_1}(h)+\delta
   _{j_4}(h)\cdot \delta _{j_1}(h)\cdot \delta
   _{j_2}(h)\cdot \delta _{j_3}(h)+\delta
   _{j_4}(h)\cdot \delta _{j_1}(h)\cdot \delta
   _{j_3}(h)\cdot \delta _{j_2}(h)+\delta
   _{j_4}(h)\cdot \delta _{j_2}(h)\cdot \delta
   _{j_1}(h)\cdot \delta _{j_3}(h)+\delta
   _{j_4}(h)\cdot \delta _{j_2}(h)\cdot \delta
   _{j_3}(h)\cdot \delta _{j_1}(h)+\delta
   _{j_4}(h)\cdot \delta _{j_3}(h)\cdot \delta
   _{j_1}(h)\cdot \delta _{j_2}(h)+\delta
   _{j_4}(h)\cdot \delta _{j_3}(h)\cdot \delta
   _{j_2}(h)\cdot \delta _{j_1}(h). 
$
\end{lemma}

\smallskip

The above lemma can be proved by writing an expansional formula as 
in Section 6.1 of \cite{ConMosModular} (or the method used in \cite{ConTreGB, FatKhaGB}). 
This method will also show that the functions $G_1, G_2, G_3, G_4$ can be constructed 
as follows. 

\begin{lemma} \label{BabyFunctions1} The functions $G_1, G_2, G_3, G_4$ 
can be constructed recursively by setting 
\[
G_0=1, 
\]
and by writing
\[
G_n(s_1, \dots, s_n) 
= 
\int_{0}^1 r^{n-1} e^{s_1 r} \,G_{n-1} (r s_2, r s_3, \dots, r s_n)  \, dr. 
\]
Therefore, we have 
\[
G_1(s_1) = \frac{e^{s_1}-1}{s_1}, 
\]
and 
\[
G_n(s_1, \dots, s_n) 
= 
\]
\[
\int_0^1 \int_0^1 \dots \int_0^1 
r_1^{n-1} r_2^{n-2} \cdots r_{n-1} \, e^{s_1 r_1 + s_2 r_1 r_2 + \cdots + s_{n-1} r_1 r_2 \cdots r_{n-1}} \times \]
\[
\qquad \qquad \qquad G_1(r_1 r_2 \cdots r_{n-1} s_n) \, dr_1 \, dr_2 \cdots \, dr_{n-1}
\]
\[
= \int_0^1 \int_0^1 \dots \int_0^1 \prod_{j=1}^{n-1} \left ( r_j^{n-j} \, e^{s_j \prod_{i=1}^j r_i} \right )  G_1(r_1 r_2 \cdots r_{n-1} s_n) \, dr_1 \, dr_2 \cdots \, dr_{n-1}. 
\]
\end{lemma}

\smallskip

Now, we can explain how the functions $G_2, G_3, G_4$ can be obtained 
by finite differences of the function $G_1$:

\begin{lemma} \label{BabyFunctions2}
Starting from the function
\[
G_1(s_1)
=
\frac{e^{s_1}-1}{s_1}, 
\]
we have: 
\[
G_2(s_1, s_2)
= 
\frac{1}{s_2} \left( G_1(s_1+s_2) - G_1(s_1)  \right), 
\]
\[
G_3(s_1, s_2, s_3) 
= 
\frac{1}{s_3} \left ( \frac{G_1(s_1+ s_2 + s_3) - G_1(s_1)}{s_2 + s_3} - \frac{G_1(s_1+s_2)-G_1(s_1)}{s_2} \right ), 
\]
\begin{eqnarray*}
&& G_4(s_1, s_2, s_3, s_4) 
= 
\\
&&
\qquad \frac{1}{s_4} \frac{1}{s_3 + s_4} \left(  \frac{G_1(s_1+s_2+s_3+s_4) - G_1(s_1)}{s_2 + s_3 + s_4}  
- \frac{G_1(s_1+s_2) - G_1(s_1)}{s_2}\right ) 
\\
&&
\qquad -\frac{1}{s_4} \frac{1}{s_3} 
\left(  \frac{G_1(s_1+s_2+s_3) - G_1(s_1)}{s_2 + s_3 }  
- \frac{G_1(s_1+s_2) - G_1(s_1)}{s_2}\right ).  
\end{eqnarray*}
\end{lemma}

\smallskip

We recall that we provided the explicit formulas for the 
functions $G_1, G_2, G_3, G_4$ by \eqref{ExplicitBabyFunctions}, 
since they play an important role in the functional 
relations stated in Theorem \ref{FuncRelationsThm} in Section \ref{a_4Sec}.

\smallskip

Now we  are absolutely ready to start the second calculation of the gradient 
\[
\dep \vphi_0(a_4(h+ \vep a)), 
\] 
where $h, a \in \CNT$ are selfadjoint elements, 
by computing 
the gradient corresponding to each term appearing in the 
expression \eqref{a_4expression} for the term $a_4$. In fact, the terms 
that have functions of the modular automorphism with the same 
number of variables involved give rise to, although lengthy and different,  
but somewhat similar calculations. Therefore, 
we have chosen four different terms from \eqref{a_4expression} that 
respectively involve  one, two, three and four variable functions, and 
will demonstrate the gradient calculation for each individual case 
in the following subsections.

\subsection{Gradient of $ 
\frac{1}{2} \, \varphi_0 \left ( e^{h} K_1(\nab)  ( \delta _1^2 \delta _2^2 ( h )  \right )$}

In order to compute 
\[
\dep \varphi_0 \left ( e^{h+\vep a} K_1(\nab_{\vep })  ( \delta _1^2 \delta _2^2 ( h+\vep a )  \right ), 
\]
first we use Lemma \ref{LowerwithTrace1} to write 
\[
\varphi_0 \left ( e^{h} K_1(\nab)  ( \delta _1^2 \delta _2^2 ( h )  \right )
= 
\frac{-2 \pi}{15} \, \varphi_0 \left( e^h \del_1^2 \del_2^2(h)  \right ).  
\] 
Therefore it suffices to calculate 
\[
\dep \Omega (h+ \vep a), 
\]
where
\[
\Omega(h) = \varphi_0 \left( e^h \del_1^2 \del_2^2(h)  \right ).
\]
Since 
\[
\Omega(h + \vep a) = \vphi_0 \left ( e^{h+ \vep a} \del_1^2 \del_2^2(h) \right ) +
\vep \vphi_0 \left (  e^{h+ \vep a} \del_1^2 \del_2^2(a) \right ),   
\]
under the trace $\vphi_0$ we have
\begin{eqnarray*}
\dep \Omega (h+ \vep a) &=& \dep e^{h+ \vep a} \, \del_1^2 \del_2^2(h) + 
e^h \del_1^2 \del_2^2 (a) \\
&=& \frac{1-e^{-\nab}}{\nab}(a) e^h \del_1^2 \del_2^2 (h) + e^h \del_1^2 \del_2^2 (a) \\
&=& ae^h \frac{1-e^\nab}{-\nab}(\del_1^2 \del_2^2(h)) + a e^h e^{-h} \del_1^2 \del_2^2(e^h). 
\end{eqnarray*}
In this expression, the term $e^{-h} \del_1^2 \del_2^2(e^h)$ can then be expanded 
using Lemma \ref{ToLogBaby}. 

\smallskip

Putting everything together, up to multiplying the right hand side of the following equality by $a e^h$, 
under the trace $\vphi_0$, we have 
\[
\dep \frac{1}{2} \, \varphi_0 \left ( e^{h +\vep a} K_1(\nab)  ( \delta _1^2 \delta _2^2 (h +\vep a  )  \right )
\]
\begin{center}
%\begin{math}
$=
\left(\frac{\left(1-e^{s_1}\right) \pi }{15 s_1}-\frac{1}{15} \pi  G_1\left(s_1\right)\right)\bigm|_{s_j = \nab}\left(\delta _1^2 \delta _2^2(h)\right)$

$+\left(-\frac{1}{15} \pi  G_2\left(s_1,s_2\right)\right)\bigm|_{s_j = \nab}\delta _1^2(h)\delta _2^2(h)$

$+\left(-\frac{1}{15} \pi  G_2\left(s_1,s_2\right)\right)\bigm|_{s_j = \nab}\delta _2^2(h)\delta _1^2(h)$

$+\left(\frac{1}{15} (-2) \pi  G_2\left(s_1,s_2\right)\right)\bigm|_{s_j = \nab}\delta _1(h)\left(\delta _1 \delta _2^2(h)\right)$

$+\left(\frac{1}{15} (-2) \pi  G_2\left(s_1,s_2\right)\right)\bigm|_{s_j = \nab}\delta _2(h)\left(\delta _1^2 \delta _2(h)\right)$

$+\left(\frac{1}{15} (-2) \pi  G_2\left(s_1,s_2\right)\right)\bigm|_{s_j = \nab}\left(\delta _1^2 \delta _2(h)\right)\delta _2(h)$

$+\left(\frac{1}{15} (-2) \pi  G_2\left(s_1,s_2\right)\right)\bigm|_{s_j = \nab}\left(\delta _1 \delta _2^2(h)\right)\delta _1(h)$

$+\left(\frac{1}{15} (-4) \pi  G_2\left(s_1,s_2\right)\right)\bigm|_{s_j = \nab}\left(\delta _1 \delta _2(h)\right)\left(\delta _1 \delta _2(h)\right)$

$+\left(\frac{1}{15} (-2) \pi  G_3\left(s_1,s_2,s_3\right)\right)\bigm|_{s_j = \nab}\delta _1(h)\delta _1(h)\delta _2^2(h)$

$+\left(\frac{1}{15} (-2) \pi  G_3\left(s_1,s_2,s_3\right)\right)\bigm|_{s_j = \nab}\delta _1(h)\delta _2^2(h)\delta _1(h)$

$+\left(\frac{1}{15} (-2) \pi  G_3\left(s_1,s_2,s_3\right)\right)\bigm|_{s_j = \nab}\delta _1^2(h)\delta _2(h)\delta _2(h)$

$+\left(\frac{1}{15} (-2) \pi  G_3\left(s_1,s_2,s_3\right)\right)\bigm|_{s_j = \nab}\delta _2(h)\delta _1^2(h)\delta _2(h)$

$+\left(\frac{1}{15} (-2) \pi  G_3\left(s_1,s_2,s_3\right)\right)\bigm|_{s_j = \nab}\delta _2(h)\delta _2(h)\delta _1^2(h)$

$+\left(\frac{1}{15} (-2) \pi  G_3\left(s_1,s_2,s_3\right)\right)\bigm|_{s_j = \nab}\delta _2^2(h)\delta _1(h)\delta _1(h)$

$+\left(\frac{1}{15} (-4) \pi  G_3\left(s_1,s_2,s_3\right)\right)\bigm|_{s_j = \nab}\delta _1(h)\delta _2(h)\left(\delta _1 \delta _2(h)\right)$

$+\left(\frac{1}{15} (-4) \pi  G_3\left(s_1,s_2,s_3\right)\right)\bigm|_{s_j = \nab}\delta _1(h)\left(\delta _1 \delta _2(h)\right)\delta _2(h)$

$+\left(\frac{1}{15} (-4) \pi  G_3\left(s_1,s_2,s_3\right)\right)\bigm|_{s_j = \nab}\delta _2(h)\delta _1(h)\left(\delta _1 \delta _2(h)\right)$

$+\left(\frac{1}{15} (-4) \pi  G_3\left(s_1,s_2,s_3\right)\right)\bigm|_{s_j = \nab}\delta _2(h)\left(\delta _1 \delta _2(h)\right)\delta _1(h)$

$+\left(\frac{1}{15} (-4) \pi  G_3\left(s_1,s_2,s_3\right)\right)\bigm|_{s_j = \nab}\left(\delta _1 \delta _2(h)\right)\delta _1(h)\delta _2(h)$

$+\left(\frac{1}{15} (-4) \pi  G_3\left(s_1,s_2,s_3\right)\right)\bigm|_{s_j = \nab}\left(\delta _1 \delta _2(h)\right)\delta _2(h)\delta _1(h)$

$+\left(\frac{1}{15} (-4) \pi  G_4\left(s_1,s_2,s_3,s_4\right)\right)\bigm|_{s_j = \nab}\delta _1(h)\delta _1(h)\delta _2(h)\delta _2(h)$

$+\left(\frac{1}{15} (-4) \pi  G_4\left(s_1,s_2,s_3,s_4\right)\right)\bigm|_{s_j = \nab}\delta _1(h)\delta _2(h)\delta _1(h)\delta _2(h)$

$+\left(\frac{1}{15} (-4) \pi  G_4\left(s_1,s_2,s_3,s_4\right)\right)\bigm|_{s_j = \nab}\delta _1(h)\delta _2(h)\delta _2(h)\delta _1(h)$

$+\left(\frac{1}{15} (-4) \pi  G_4\left(s_1,s_2,s_3,s_4\right)\right)\bigm|_{s_j = \nab}\delta _2(h)\delta _1(h)\delta _1(h)\delta _2(h)$

$+\left(\frac{1}{15} (-4) \pi  G_4\left(s_1,s_2,s_3,s_4\right)\right)\bigm|_{s_j = \nab}\delta _2(h)\delta _1(h)\delta _2(h)\delta _1(h)$

$+\left(\frac{1}{15} (-4) \pi  G_4\left(s_1,s_2,s_3,s_4\right)\right)\bigm|_{s_j = \nab}\delta _2(h)\delta _2(h)\delta _1(h)\delta _1(h). 
$
%\end{math}
\end{center}

\subsection{Gradient of $ \frac{1}{4}\,\varphi_0 \left ( e^{h} K_3 (\nab, \nab) \left (
 \left(\delta _1 \delta _2(h
   )\right) \cdot \left(\delta _1 \delta _2(h
   )\right) 
   \right ) \right )$}
In order to compute 
\[
\dep \Omega (h+ \vep a)
\]
where
\[
\Omega(h) = \varphi_0 
\left ( e^h K_3 (\nab, \nab) \left (
 \left(\delta _1 \delta _2(h
   )\right) \cdot \left(\delta _1 \delta _2(h
   )\right) 
   \right ) \right ),
\]
first we use Lemma \ref{LowerwithTrace1} to write 
\[
\Omega(h)= \varphi_0 \left( e^h K(\nab)(\del_1 \del_2(h)) \del_1 \del_2 (h) \right ),  
\] 
where the function $K$ is defined by 
\[
K(s_1)=K_3(s_1, -s_1).
\]
Therefore, under the trace $\vphi_0$, 
\begin{eqnarray}
\Omega(h+ \vep a) &=&   e^{h+\vep a} K(\nab_\vep)(\del_1 \del_2 (h))\, \del_1 \del_2 (h) \nonumber \\
&&+ \vep e^{h+\vep a} K(\nab_\vep)(\del_1 \del_2 (a)) \, \del_1 \del_2 (h) \nonumber \\
&&+  \vep e^{h+\vep a} K(\nab_\vep)(\del_1 \del_2 (h)) \, \del_1 \del_2 (a) \nonumber \\
&& + \vep^2 (\cdots). \nonumber
\end{eqnarray}
Therefore 
\begin{eqnarray} \label{preGr2}
\dep \Omega (h + \vep a) &=& 
\dep e^{h+\vep a}\cdot K(\nab)(\del_1 \del_2(h)) \, \del_1\del_2(h)  \\
&&+ e^h \dep K(\nab_\vep) (\del_1 \del_2(h)) \, \del_1 \del_2(h) \nonumber \\
&&+  e^h  K(\nab) (\del_1 \del_2(a)) \, \del_1 \del_2(h) \nonumber \\
&&+  e^h  K(\nab) (\del_1 \del_2(h)) \, \del_1 \del_2(a). \nonumber
\end{eqnarray}

\smallskip

The first term in \eqref{preGr2}, under $\vphi_0$, is equal to 
\begin{eqnarray*}
&&\frac{1-e^{-\nab}}{\nab}(a) \,e^h K(\nab) ( \del_1 \del_2(h)  ) \, \del_1 \del_2(h) \\
&=&a e^h \frac{1-e^{-\nab}}{\nab} \left( K(\nab) ( \del_1 \del_2(h)  ) \, \del_1 \del_2(h)  \right ) \\
&=& a e^h \left ( \frac{1-e^{s_1+s_2}}{-(s_1+s_2)} K(s_1) \right )\bigm|_{s_1=\nab, s_2=\nab}
(\del_1 \del_2 (h) \, \del_1 \del_2(h))
\end{eqnarray*}

\smallskip

Using Lemma \ref{GrofFuncCalc1}, for the second term in \eqref{preGr2}, under $\vphi_0$, we have 
\begin{eqnarray*}
&&e^h \dep K(\nab_\vep) (\del_1 \del_2(h)) \, \del_1 \del_2(h) \\
&=&  a e^h \left(  e^{s_1+s_2}\frac{K(-s_2)-K(s_1)}{s_1+s_2} \right ) \bigm|_{s_1, s_2=\nab}(\del_1 \del_2 (h) \, \del_1 \del_2(h))\\
&&+  a e^h \left(  e^{s_1}\frac{K(s_2)-K(-s_1)}{s_1+s_2} \right ) \bigm|_{s_1, s_2=\nab}(\del_1 \del_2 (h) \, \del_1 \del_2(h)). 
\end{eqnarray*}

\smallskip

For the third term in \eqref{preGr2}, under $\vphi_0$, we can write 
\begin{eqnarray*}
&&e^h  K(\nab) (\del_1 \del_2(a)) \, \del_1 \del_2(h)  \\
&=& \del_1\del_2(a) e^h K(-\nab) e^\nab(\del_1 \del_2(h)) \\
&=& a e^h e^{-h} \del_1 \del_2(e^h) K_v(\nab)(\del_1 \del_2(h)) + 
a e^h e^{-h} \del_1 (e^h) \del_2 (K_v(\nab)(\del_1 \del_2(h))) \\
&&+ a e^h e^{-h} \del_2 (e^h) \, \del_1 \left (K_v(\nab)(\del_1 \del_2(h)) \right) + 
a e^h \del_1 \del_2  \left( K_v(\nab)(\del_1 \del_2(h)) \right ),  
\end{eqnarray*}
where 
\[
K_v(s_1) = K(-s_1)e^{s_1}.
\] 
In the above expression, terms of the 
form $e^{-h} \del_{j_1}(e^h)$ and $e^{-h} \del_{j_1} \del_{j_2}(e^h)$ can then 
be expanded using Lemma \ref{ToLogBaby}. Moreover, we use  Lemma \ref{DerFuncCalc1} and Lemma \ref{DerFuncCalc2} to expand the terms of the form $ \del_2 (K_v(\nab)(\del_1 \del_2(h)))$ and $\del_1 \del_2  \left( K_v(\nab)(\del_1 \del_2(h)) \right )$.

\smallskip

Similarly, for the fourth term in \eqref{preGr2}, under $\vphi_0$, we can write 
\begin{eqnarray*}
&& e^h  K(\nab) (\del_1 \del_2(h)) \, \del_1 \del_2(a) \\
&=& a \del_1 \del_2 (e^h  K(\nab) (\del_1 \del_2(h)))  \\
&=& a \, e^h e^{-h} \del_1 \del_2 (e^h)  K(\nab) (\del_1 \del_2(h)) + 
ae^h e^{-h} \del_1 (e^h) \del_2 \left(  K(\nab) (\del_1 \del_2(h)) \right ) \\
&&+ ae^h e^{-h} \del_2 (e^h) \, \del_1 \left(  K(\nab) (\del_1 \del_2(h)) \right ) 
+  ae^h   \del_1 \del_2 \left(  K(\nab) (\del_1 \del_2(h)) \right ).
\end{eqnarray*}
In this expression also,  terms of the 
form $e^{-h} \del_{j_1}(e^h)$ and $e^{-h} \del_{j_1} \del_{j_2}(e^h)$ can then 
be expanded using Lemma \ref{ToLogBaby}, and the terms of the form 
$\del_2 \left(  K(\nab) (\del_1 \del_2(h)) \right )$ and   
$\del_1 \del_2 \left(  K(\nab) (\del_1 \del_2(h)) \right )$ need to be expanded 
using Lemma \ref{DerFuncCalc1} and Lemma \ref{DerFuncCalc2}. 

\smallskip

Finally, putting everything together, up to multiplying the right hand side of the following equality by $a e^h$, 
under the trace $\vphi_0$, we have 
\[
\dep
\frac{1}{4}\,\varphi_0 \left ( e^{h+\vep a} K_3 (\nab_{\vep}, \nab_{\vep }) \left (
 \left(\delta _1 \delta _2(h+\vep a
   )\right) \cdot \left(\delta _1 \delta _2(h+\vep a
   )\right) 
   \right ) \right )
\]
\begin{center}
\begin{math}
=(\frac{1}{4} e^{s_1} k_3(-s_1)+\frac{1}{4} k_3(s_1))\bigm|_{s_j = \nab}(\delta _1^2 \delta _2^2(h))
\end{math}
\end{center}
\begin{center}
\begin{math}
+
(-\frac{e^{s_1+s_2} k_3(-s_1-s_2)}{4 s_1}+\frac{1}{4} e^{s_2} G_1(s_1) k_3(-s_2)+\frac{e^{s_2} k_3(-s_2)}{4 s_1}+\frac{1}{4} G_1(s_1) k_3(s_2)+\frac{k_3(s_2)}{4 s_1}-\frac{k_3(s_1+s_2)}{4 s_1})\bigm|_{s_j = \nab}\delta _1(h)(\delta _1 \delta _2^2(h))
\end{math}
\end{center}
\begin{center}
\begin{math}
+(-\frac{e^{s_1+s_2} k_3(-s_1-s_2)}{4 s_1}+\frac{1}{4} e^{s_2} G_1(s_1) k_3(-s_2)+\frac{e^{s_2} k_3(-s_2)}{4 s_1}+\frac{1}{4} G_1(s_1) k_3(s_2)+\frac{k_3(s_2)}{4 s_1}-\frac{k_3(s_1+s_2)}{4 s_1})\bigm|_{s_j = \nab}\delta _2(h)(\delta _1^2 \delta _2(h))
\end{math}
\end{center}
\begin{center}
\begin{math}
+(-\frac{e^{s_1} k_3(-s_1)}{4 s_2}+\frac{e^{s_1+s_2} k_3(-s_1-s_2)}{4 s_2}+\frac{k_3(s_1+s_2)}{4 s_2}-\frac{k_3(s_1)}{4 s_2})\bigm|_{s_j = \nab}(\delta _1^2 \delta _2(h))\delta _2(h)
\end{math}
\end{center}
\begin{center}
\begin{math}
+(-\frac{e^{s_1} k_3(-s_1)}{4 s_2}+\frac{e^{s_1+s_2} k_3(-s_1-s_2)}{4 s_2}+\frac{k_3(s_1+s_2)}{4 s_2}-\frac{k_3(s_1)}{4 s_2})\bigm|_{s_j = \nab}(\delta _1 \delta _2^2(h))\delta _1(h)
\end{math}
\end{center}
\begin{center}
\begin{math}
+(-\frac{e^{s_1} k_3(-s_1)}{2 (s_1+s_2)}-\frac{e^{s_1} s_1 k_3(-s_1)}{4 s_2 (s_1+s_2)}+\frac{e^{s_1+s_2} s_1 k_3(-s_1-s_2)}{4 s_2 (s_1+s_2)}+\frac{e^{s_2} s_1 G_1(s_1) k_3(-s_2)}{4 (s_1+s_2)}+\frac{e^{s_2} s_2 G_1(s_1) k_3(-s_2)}{4 (s_1+s_2)}+\frac{e^{s_2} k_3(-s_2)}{4 (s_1+s_2)}+\frac{e^{s_1+s_2} k_3(-s_2)}{4 (s_1+s_2)}+\frac{e^{s_2} s_2 k_3(-s_2)}{4 s_1 (s_1+s_2)}+\frac{s_1 G_1(s_1) k_3(s_2)}{4 (s_1+s_2)}+\frac{s_2 G_1(s_1) k_3(s_2)}{4 (s_1+s_2)}+\frac{e^{s_1} k_3(s_2)}{4 (s_1+s_2)}+\frac{k_3(s_2)}{4 (s_1+s_2)}+\frac{s_2 k_3(s_2)}{4 s_1 (s_1+s_2)}+\frac{s_1 k_3(s_1+s_2)}{4 s_2 (s_1+s_2)}-\frac{k_3(s_1)}{2 (s_1+s_2)}-\frac{e^{s_1+s_2} s_2 k_3(-s_1-s_2)}{4 s_1 (s_1+s_2)}-\frac{s_2 k_3(s_1+s_2)}{4 s_1 (s_1+s_2)}-\frac{s_1 k_3(s_1)}{4 s_2 (s_1+s_2)})\bigm|_{s_j = \nab}(\delta _1 \delta _2(h))(\delta _1 \delta _2(h))
\end{math}
\end{center}
\begin{center}
\begin{math}
+(-\frac{e^{s_2+s_3} G_1(s_1) k_3(-s_2-s_3)}{4 (s_1+s_2)}-\frac{e^{s_2+s_3} k_3(-s_2-s_3)}{4 s_1 (s_1+s_2)}-\frac{e^{s_2+s_3} k_3(-s_2-s_3)}{4 s_2 (s_1+s_2)}-\frac{e^{s_2+s_3} s_1 G_1(s_1) k_3(-s_2-s_3)}{4 s_2 (s_1+s_2)}+\frac{e^{s_1+s_2+s_3} k_3(-s_1-s_2-s_3)}{4 s_1 (s_1+s_2)}+\frac{e^{s_3} G_1(s_1) k_3(-s_3)}{4 (s_1+s_2)}+\frac{e^{s_3} s_1 G_1(s_1) k_3(-s_3)}{4 s_2 (s_1+s_2)}+\frac{e^{s_3} s_1 G_2(s_1,s_2) k_3(-s_3)}{4 (s_1+s_2)}+\frac{e^{s_3} s_2 G_2(s_1,s_2) k_3(-s_3)}{4 (s_1+s_2)}+\frac{e^{s_3} k_3(-s_3)}{4 s_2 (s_1+s_2)}+\frac{G_1(s_1) k_3(s_3)}{4 (s_1+s_2)}+\frac{s_1 G_1(s_1) k_3(s_3)}{4 s_2 (s_1+s_2)}+\frac{s_1 G_2(s_1,s_2) k_3(s_3)}{4 (s_1+s_2)}+\frac{s_2 G_2(s_1,s_2) k_3(s_3)}{4 (s_1+s_2)}+\frac{k_3(s_3)}{4 s_2 (s_1+s_2)}+\frac{k_3(s_1+s_2+s_3)}{4 s_1 (s_1+s_2)}-\frac{G_1(s_1) k_3(s_2+s_3)}{4 (s_1+s_2)}-\frac{k_3(s_2+s_3)}{4 s_1 (s_1+s_2)}-\frac{s_1 G_1(s_1) k_3(s_2+s_3)}{4 s_2 (s_1+s_2)}-\frac{k_3(s_2+s_3)}{4 s_2 (s_1+s_2)})\bigm|_{s_j = \nab}\delta _1(h)\delta _2(h)(\delta _1 \delta _2(h))
\end{math}
\end{center}
\begin{center}
\begin{math}
+(-\frac{e^{s_2+s_3} G_1(s_1) k_3(-s_2-s_3)}{4 (s_1+s_2)}-\frac{e^{s_2+s_3} k_3(-s_2-s_3)}{4 s_1 (s_1+s_2)}-\frac{e^{s_2+s_3} k_3(-s_2-s_3)}{4 s_2 (s_1+s_2)}-\frac{e^{s_2+s_3} s_1 G_1(s_1) k_3(-s_2-s_3)}{4 s_2 (s_1+s_2)}+\frac{e^{s_1+s_2+s_3} k_3(-s_1-s_2-s_3)}{4 s_1 (s_1+s_2)}+\frac{e^{s_3} G_1(s_1) k_3(-s_3)}{4 (s_1+s_2)}+\frac{e^{s_3} s_1 G_1(s_1) k_3(-s_3)}{4 s_2 (s_1+s_2)}+\frac{e^{s_3} s_1 G_2(s_1,s_2) k_3(-s_3)}{4 (s_1+s_2)}+\frac{e^{s_3} s_2 G_2(s_1,s_2) k_3(-s_3)}{4 (s_1+s_2)}+\frac{e^{s_3} k_3(-s_3)}{4 s_2 (s_1+s_2)}+\frac{G_1(s_1) k_3(s_3)}{4 (s_1+s_2)}+\frac{s_1 G_1(s_1) k_3(s_3)}{4 s_2 (s_1+s_2)}+\frac{s_1 G_2(s_1,s_2) k_3(s_3)}{4 (s_1+s_2)}+\frac{s_2 G_2(s_1,s_2) k_3(s_3)}{4 (s_1+s_2)}+\frac{k_3(s_3)}{4 s_2 (s_1+s_2)}+\frac{k_3(s_1+s_2+s_3)}{4 s_1 (s_1+s_2)}-\frac{G_1(s_1) k_3(s_2+s_3)}{4 (s_1+s_2)}-\frac{k_3(s_2+s_3)}{4 s_1 (s_1+s_2)}-\frac{s_1 G_1(s_1) k_3(s_2+s_3)}{4 s_2 (s_1+s_2)}-\frac{k_3(s_2+s_3)}{4 s_2 (s_1+s_2)})\bigm|_{s_j = \nab}\delta _2(h)\delta _1(h)(\delta _1 \delta _2(h))
\end{math}
\end{center}
\begin{center}
\begin{math}
+(\frac{e^{s_1+s_2} k_3(-s_1-s_2)}{4 s_1 s_3}+\frac{k_3(s_1+s_2)}{4 s_1 s_3}+\frac{e^{s_2+s_3} G_1(s_1) k_3(-s_2-s_3)}{4 s_3}+\frac{e^{s_2+s_3} k_3(-s_2-s_3)}{4 s_1 s_3}+\frac{G_1(s_1) k_3(s_2+s_3)}{4 s_3}+\frac{k_3(s_2+s_3)}{4 s_1 s_3}-\frac{e^{s_2} G_1(s_1) k_3(-s_2)}{4 s_3}-\frac{G_1(s_1) k_3(s_2)}{4 s_3}-\frac{e^{s_2} k_3(-s_2)}{4 s_1 s_3}-\frac{k_3(s_2)}{4 s_1 s_3}-\frac{e^{s_1+s_2+s_3} k_3(-s_1-s_2-s_3)}{4 s_1 s_3}-\frac{k_3(s_1+s_2+s_3)}{4 s_1 s_3})\bigm|_{s_j = \nab}\delta _1(h)(\delta _1 \delta _2(h))\delta _2(h)
\end{math}
\end{center}
\begin{center}
\begin{math}
+(\frac{e^{s_1+s_2} k_3(-s_1-s_2)}{4 s_1 s_3}+\frac{k_3(s_1+s_2)}{4 s_1 s_3}+\frac{e^{s_2+s_3} G_1(s_1) k_3(-s_2-s_3)}{4 s_3}+\frac{e^{s_2+s_3} k_3(-s_2-s_3)}{4 s_1 s_3}+\frac{G_1(s_1) k_3(s_2+s_3)}{4 s_3}+\frac{k_3(s_2+s_3)}{4 s_1 s_3}-\frac{e^{s_2} G_1(s_1) k_3(-s_2)}{4 s_3}-\frac{G_1(s_1) k_3(s_2)}{4 s_3}-\frac{e^{s_2} k_3(-s_2)}{4 s_1 s_3}-\frac{k_3(s_2)}{4 s_1 s_3}-\frac{e^{s_1+s_2+s_3} k_3(-s_1-s_2-s_3)}{4 s_1 s_3}-\frac{k_3(s_1+s_2+s_3)}{4 s_1 s_3})\bigm|_{s_j = \nab}\delta _2(h)(\delta _1 \delta _2(h))\delta _1(h)
\end{math}
\end{center}
\begin{center}
\begin{math}
+(\frac{e^{s_1} k_3(-s_1)}{4 s_2 (s_2+s_3)}+\frac{k_3(s_1)}{4 s_2 (s_2+s_3)}+\frac{e^{s_1+s_2+s_3} k_3(-s_1-s_2-s_3)}{4 s_3 (s_2+s_3)}+\frac{k_3(s_1+s_2+s_3)}{4 s_3 (s_2+s_3)}-\frac{e^{s_1+s_2} k_3(-s_1-s_2)}{4 s_2 (s_2+s_3)}-\frac{k_3(s_1+s_2)}{4 s_2 (s_2+s_3)}-\frac{e^{s_1+s_2} k_3(-s_1-s_2)}{4 s_3 (s_2+s_3)}-\frac{k_3(s_1+s_2)}{4 s_3 (s_2+s_3)})\bigm|_{s_j = \nab}(\delta _1 \delta _2(h))\delta _1(h)\delta _2(h)
\end{math}
\end{center}
\begin{center}
\begin{math}
+(\frac{e^{s_1} k_3(-s_1)}{4 s_2 (s_2+s_3)}+\frac{k_3(s_1)}{4 s_2 (s_2+s_3)}+\frac{e^{s_1+s_2+s_3} k_3(-s_1-s_2-s_3)}{4 s_3 (s_2+s_3)}+\frac{k_3(s_1+s_2+s_3)}{4 s_3 (s_2+s_3)}-\frac{e^{s_1+s_2} k_3(-s_1-s_2)}{4 s_2 (s_2+s_3)}-\frac{k_3(s_1+s_2)}{4 s_2 (s_2+s_3)}-\frac{e^{s_1+s_2} k_3(-s_1-s_2)}{4 s_3 (s_2+s_3)}-\frac{k_3(s_1+s_2)}{4 s_3 (s_2+s_3)})\bigm|_{s_j = \nab}(\delta _1 \delta _2(h))\delta _2(h)\delta _1(h). 
\end{math}
\end{center}

\subsection{Gradient of $\frac{1}{8}\, \vphi_0 \left (  e^{h} K_{13} (\nab, \nab, \nab) \left ( 
\left(\delta _1 \delta _2(h
   )\right)\cdot \delta _1(h )\cdot \delta _2(h
   )
\right ) \right )$} We use Lemma \ref{LowerwithTrace1} to write
\begin{eqnarray*}
\Omega(h) &=& \vphi_0 \left (  e^h K_{13} (\nab, \nab, \nab) \left ( 
\left(\delta _1 \delta _2(h
   )\right)\cdot \delta _1(h )\cdot \delta _2(h
   )
\right ) \right ) \\
&=& \vphi_0 \left (  e^h K (\nab, \nab) \left ( 
\left(\delta _1 \delta _2(h
   )\right)\cdot \delta _1(h )
\right ) \cdot \delta _2(h
   ) \right ),
\end{eqnarray*}
where the function $K$ is defined by 
\[
K(s_1, s_2) = K_{13}(s_1, s_2, -s_1-s_2). 
\]
Therefore, under the trace $\vphi_0$, we have 
\begin{eqnarray*}
\Omega(h+ \vep a) &=& e^{h+ \vep a} K(\nab_\vep, \nab_\vep)(\del_1 \del_2(h) \, \del_1(h)) \, \del_2(h) \\
&&+ \vep e^{h+ \vep a} K(\nab_\vep, \nab_\vep)(\del_1 \del_2(h) \, \del_1(h)) \, \del_2(a) \\
&& +  \vep e^{h+ \vep a} K(\nab_\vep, \nab_\vep)(\del_1 \del_2(h) \, \del_1(a)) \, \del_2(h)\\
&& +  \vep e^{h+ \vep a} K(\nab_\vep, \nab_\vep)(\del_1 \del_2(a) \, \del_1(h)) \, \del_2(h)\\
&& + \vep^2 (\cdots)+ \vep^3(\cdots). 
\end{eqnarray*} 
Under $\vphi_0$, we can thus continue to write 
\begin{eqnarray} \label{preGr3}
\dep \Omega(h+\vep a) &=&
\dep e^{h+\vep a}\, K(\nab, \nab) (\del_1 \del_2(h) \, \del_1(h)) \, \del_2(h)  \\
&&+  e^h\, \dep K(\nab_\vep, \nab_\vep) (\del_1 \del_2(h) \, \del_1(h)) \, \del_2(h) \nonumber \\
&&+   e^h\,  K(\nab, \nab) (\del_1 \del_2(h) \, \del_1(h)) \, \del_2(a) \nonumber \\
&&+    e^h\,  K(\nab, \nab) (\del_1 \del_2(h) \, \del_1(a)) \, \del_2(h) \nonumber \\ 
&&+    e^h\,  K(\nab, \nab) (\del_1 \del_2(a) \, \del_1(h)) \, \del_2(h). \nonumber
\end{eqnarray}

\smallskip

The first term in the above expression, under $\vphi_0$, is equal to 
\begin{eqnarray*}
&&\frac{1-e^{-\nab}}{\nab}(a) \, e^h \, K(\nab, \nab) (\del_1 \del_2(h) \, \del_1(h)) \, \del_2(h) \\
&=& a e^h \frac{1-e^{\nab}}{-\nab} \left ( K(\nab, \nab) (\del_1 \del_2(h) \, \del_1(h))  \right ) \\
&=& a e^h \left ( \frac{1-e^{s_1+s_2+s_3}}{-s_1-s_2-s_3} K(s_1, s_2)\right )  \bigm|_{s_1, s_2, s_3 = \nab} (\del_1 \del_2(h) \, \del_1(h)  \, \del_2(h) ). 
\end{eqnarray*}

\smallskip

Using Lemma \ref{GrofFuncCalc2}, for the second term in \eqref{preGr3}, under $\vphi_0$, we have 
\begin{eqnarray*}
&&e^h\, \dep K(\nab_\vep, \nab_\vep) (\del_1 \del_2(h) \, \del_1(h)) \, \del_2(h) \\
&=& a e^h \left ( e^{s_1+s_2+s_3} \frac{ K(s_1, s_2) - K(-s_2-s_3,s_2) }{-s_1-s_2-s_3}\right ) \bigm|_{s_1, s_2, s_3 = \nab} ( \del_1\del_2(h) \del_1(h) \del_2(h)) + \\
&& a e^h \left ( e^{s_1+s_2} \frac{ K(-s_1-s_2, s_2) - K(s_3,-s_2-s_3) }{-s_1-s_2-s_3}\right ) \bigm|_{s_1, s_2, s_3 = \nab} (  \del_1(h) \del_2(h) \del_1\del_2(h))\\
&&+ a e^h \left ( e^{s_1} \frac{ K(s_2,-s_1 -s_2) - K(s_2,s_3) }{-s_1-s_2-s_3}\right ) \bigm|_{s_1, s_2, s_3 = \nab} (   \del_2(h) \del_1\del_2(h) \del_1(h)). 
\end{eqnarray*}

\smallskip

For the third term in \eqref{preGr3}, under $\vphi_0$, we can write 
\begin{eqnarray*}
&& e^h\,  K(\nab, \nab) (\del_1 \del_2(h) \, \del_1(h)) \, \del_2(a) \\
&=& \del_2(a) e^h\,  K(\nab, \nab) (\del_1 \del_2(h) \, \del_1(h)) \\
&=& - a e^h e^{-h} \del_2(e^h)\,  K(\nab, \nab) (\del_1 \del_2(h) \, \del_1(h)) 
-  a e^h\, \del_2  ( K(\nab, \nab) (\del_1 \del_2(h) \, \del_1(h))  ). 
\end{eqnarray*}

\smallskip

The forth term in \eqref{preGr3}, under $\vphi_0$ and by using Lemma \ref{ShiftwithTrace}, can be written as
\begin{eqnarray*}
&&e^h\,  K(\nab, \nab) (\del_1 \del_2(h) \, \del_1(a)) \, \del_2(h) \\
&=& \del_1(a) K_v(\nab, \nab)(\del_2(h) \cdot e^h\del_1\del_2(h)) \\
&=& \del_1(a) e^h K_{vv}(\nab, \nab)(\del_2(h) \, \del_1\del_2(h)) \\
&=& - a e^h e^{-h} \del_1(e^h) K_{vv}(\nab, \nab)(\del_2(h) \, \del_1\del_2(h)) \\
&&- a  e^h \del_1(K_{vv}(\nab, \nab)(\del_2(h) \, \del_1\del_2(h))),  
\end{eqnarray*}
where 
\[
K_v(s_1, s_2)= K(s_2, -s_1-s_2), \qquad K_{vv}(s_1, s_2)= e^{s_1} K_v(s_1, s_2).  
\]
In the above expression, the  term of the form 
$e^{-h} \del_{j_1}(e^h)$ can  
be expanded using Lemma \ref{ToLogBaby}.  We use  Lemma \ref{DerFuncCalc2}  to expand $ \del_1(K_{vv}(\nab, \nab)(\del_2(h) \, \del_1\del_2(h)))$.

\smallskip

Similarly, for the fifth term in \eqref{preGr3}, under $\vphi_0$ and using Lemma \ref{ShiftwithTrace}, we write
\begin{eqnarray*}
&& e^h\,  K(\nab, \nab) (\del_1 \del_2(a) \, \del_1(h)) \, \del_2(h) \\
&=& e^h \del_1 \del_2(a)  K_w(\nab, \nab) ( \del_1(h) \, \del_2(h))  \\
&=& \del_1 \del_2(a) e^h K_{ww}(\nab, \nab) ( \del_1(h) \, \del_2(h))   \\
&=& a e^h e^{-h}\del_1 \del_2(e^h) K_{ww}(\nab, \nab) ( \del_1(h) \, \del_2(h)) \\
&&+ a e^h e^{-h} \del_1(e^h) \,\del_2( K_{ww}(\nab, \nab) ( \del_1(h) \, \del_2(h)) ) \\
&&+  a e^h e^{-h} \del_2(e^h) \, \del_1( K_{ww}(\nab, \nab) ( \del_1(h) \, \del_2(h)) ) \\
&&+ ae^h \, \del_1 \del_2( K_{ww}(\nab, \nab) ( \del_1(h) \, \del_2(h)) ), 
\end{eqnarray*}
where 
\[
K_w(s_1, s_2)= K(-s_1-s_2, s_1), \qquad K_{ww}(s_1, s_2)= e^{s_1+s_2} K_w(s_1, s_2).  
\]
Here also, one can use Lemma \ref{ToLogBaby} for expanding the 
terms of the form $e^{-h} \del_{j_1}(e^h)$ and $e^{-h} \del_{j_1} \del_{j_2}(e^h)$, 
Lemma \ref{DerFuncCalc2} can be used for 
$\del_{j_1}( K_{ww}(\nab, \nab) ( \del_1(h) \, \del_2(h)) )$, and Lemma \ref{DerFuncCalc2} 
and Lemma \ref{DerFuncCalc3} allow one to 
expand the term  $\del_1 \del_2( K_{ww}(\nab, \nab) ( \del_1(h) \, \del_2(h)) )$. 

\smallskip

Putting everything together, up to multiplying the right hand side of the following equality by $a e^h$, 
under the trace $\vphi_0$, we have 
\[
\dep
\frac{1}{8}\, \vphi_0 \left (  e^{h+\vep a} K_{13} (\nab_\vep, \nab_\vep, \nab_\vep) \left ( 
\left(\delta _1 \delta _2( h+\vep a
   )\right)\cdot \delta _1( h+\vep a )\cdot \delta _2(h+\vep a 
   )
\right ) \right )
\]
{\tiny 
\begin{center}
\begin{math}
= (-\frac{1}{8} k_{13}(s_1,s_2))\bigm|_{s_j = \nab}(\delta _1 \delta _2^2(h))\delta _1(h)
\end{math}
\end{center}
\begin{center}
\begin{math}
+(\frac{1}{8} e^{s_1+s_2} k_{13}(-s_1-s_2,s_1))\bigm|_{s_j = \nab}\delta _1(h)(\delta _1 \delta _2^2(h)) 
\end{math}
\end{center}
\begin{center}
\begin{math}
+(\frac{1}{8} e^{s_1+s_2} k_{13}(-s_1-s_2,s_1))\bigm|_{s_j = \nab}\delta _1^2(h)\delta _2^2(h)
\end{math}
\end{center}
\begin{center}
\begin{math}
+(\frac{1}{8} e^{s_1+s_2} k_{13}(-s_1-s_2,s_1))\bigm|_{s_j = \nab}(\delta _1^2 \delta _2(h))\delta _2(h)
\end{math}
\end{center}
\begin{center}
\begin{math}
+(-\frac{1}{8} e^{s_1} k_{13}(s_2,-s_1-s_2))\bigm|_{s_j = \nab}\delta _2(h)(\delta _1^2 \delta _2(h))
\end{math}
\end{center}
\begin{center}
\begin{math}
+(-\frac{1}{8} k_{13}(s_1,s_2)+\frac{1}{8} e^{s_1+s_2} k_{13}(-s_1-s_2,s_1)-\frac{1}{8} e^{s_1} k_{13}(s_2,-s_1-s_2))\bigm|_{s_j = \nab}(\delta _1 \delta _2(h))(\delta _1 \delta _2(h))
\end{math}
\end{center}
\begin{center}
\begin{math}
+(\frac{1}{8} e^{s_2+s_3} G_1(s_1) k_{13}(-s_2-s_3,s_2)+\frac{e^{s_2+s_3} k_{13}(-s_2-s_3,s_2)}{8 s_1}-\frac{e^{s_1+s_2+s_3} k_{13}(-s_1-s_2-s_3,s_1+s_2)}{8 s_1})\bigm|_{s_j = \nab}\delta _2(h)\delta _1^2(h)\delta _2(h)
\end{math}
\end{center}
\begin{center}
\begin{math}
+(\frac{1}{8} e^{s_2+s_3} G_1(s_1) k_{13}(-s_2-s_3,s_2)+\frac{e^{s_2+s_3} k_{13}(-s_2-s_3,s_2)}{8 s_1}+\frac{e^{s_1+s_2+s_3} k_{13}(-s_1-s_2-s_3,s_1+s_2)}{8 s_2}-\frac{e^{s_1+s_2+s_3} k_{13}(-s_1-s_2-s_3,s_1+s_2)}{8 s_1}-\frac{e^{s_1+s_2+s_3} k_{13}(-s_1-s_2-s_3,s_1)}{8 s_2})\bigm|_{s_j = \nab}\delta _1(h)\delta _1(h)\delta _2^2(h)
\end{math}
\end{center}
\begin{center}
\begin{math}
+(-\frac{e^{s_1+s_2} k_{13}(-s_1-s_2,s_1)}{8 s_3}+\frac{1}{8} e^{s_2+s_3} G_1(s_1) k_{13}(-s_2-s_3,s_2)+\frac{e^{s_2+s_3} k_{13}(-s_2-s_3,s_2)}{8 s_1}+\frac{e^{s_1+s_2+s_3} k_{13}(-s_1-s_2-s_3,s_1)}{8 s_3}+\frac{e^{s_1+s_2+s_3} k_{13}(-s_1-s_2-s_3,s_1+s_2)}{8 s_2}-\frac{e^{s_1+s_2+s_3} k_{13}(-s_1-s_2-s_3,s_1+s_2)}{8 s_1}-\frac{e^{s_1+s_2+s_3} k_{13}(-s_1-s_2-s_3,s_1)}{8 s_2})\bigm|_{s_j = \nab}\delta _1(h)(\delta _1 \delta _2(h))\delta _2(h)
\end{math}
\end{center}
\begin{center}
\begin{math}
+(-\frac{e^{s_1+s_2} k_{13}(-s_1-s_2,s_1)}{8 s_3}+\frac{e^{s_1+s_2+s_3} k_{13}(-s_1-s_2-s_3,s_1)}{8 s_3}+\frac{e^{s_1+s_2+s_3} k_{13}(-s_1-s_2-s_3,s_1+s_2)}{8 s_2}-\frac{e^{s_1+s_2+s_3} k_{13}(-s_1-s_2-s_3,s_1)}{8 s_2})\bigm|_{s_j = \nab}\delta _1^2(h)\delta _2(h)\delta _2(h)
\end{math}
\end{center}
\begin{center}
\begin{math}
+(\frac{1}{8} e^{s_2+s_3} G_1(s_1) k_{13}(-s_2-s_3,s_2)+\frac{e^{s_2+s_3} k_{13}(-s_2-s_3,s_2)}{8 s_1}+\frac{e^{s_1} k_{13}(s_2+s_3,-s_1-s_2-s_3)}{8 s_2}-\frac{e^{s_1+s_2+s_3} k_{13}(-s_1-s_2-s_3,s_1+s_2)}{8 s_1}-\frac{e^{s_1+s_2} k_{13}(s_3,-s_1-s_2-s_3)}{8 s_2})\bigm|_{s_j = \nab} \delta _2(h)\delta _1(h)(\delta _1 \delta _2(h))
\end{math}
\end{center}
\begin{center}
\begin{math}
+(\frac{k_{13}(s_1,s_2+s_3)}{8 s_2}+\frac{e^{s_1+s_2+s_3} k_{13}(-s_1-s_2-s_3,s_1)}{8 s_3}-\frac{k_{13}(s_1+s_2,s_3)}{8 s_2}-\frac{e^{s_1+s_2} k_{13}(-s_1-s_2,s_1)}{8 s_3})\bigm|_{s_j = \nab}(\delta _1 \delta _2(h))\delta _2(h)\delta _1(h)
\end{math}
\end{center}
\begin{center}
\begin{math}
+(\frac{e^{s_1+s_2+s_3} k_{13}(-s_1-s_2-s_3,s_1)}{8 s_3}-\frac{e^{s_1+s_2} k_{13}(-s_1-s_2,s_1)}{8 s_3})\bigm|_{s_j = \nab}\delta _1(h)\delta _2^2(h)\delta _1(h)
\end{math}
\end{center}
\begin{center}
\begin{math}
+(-\frac{e^{s_1+s_2} k_{13}(-s_1-s_2,s_1)}{4 (s_1+s_2+s_3)}-\frac{e^{s_1+s_2} s_1 k_{13}(-s_1-s_2,s_1)}{8 s_3 (s_1+s_2+s_3)}-\frac{e^{s_1+s_2} s_2 k_{13}(-s_1-s_2,s_1)}{8 s_3 (s_1+s_2+s_3)}+\frac{e^{s_1+s_2+s_3} s_1 k_{13}(-s_1-s_2-s_3,s_1)}{8 s_3 (s_1+s_2+s_3)}+\frac{e^{s_1+s_2+s_3} s_2 k_{13}(-s_1-s_2-s_3,s_1)}{8 s_3 (s_1+s_2+s_3)}+\frac{e^{s_1+s_2+s_3} k_{13}(-s_1-s_2-s_3,s_1+s_2)}{8 (s_1+s_2+s_3)}+\frac{e^{s_1+s_2+s_3} s_1 k_{13}(-s_1-s_2-s_3,s_1+s_2)}{8 s_2 (s_1+s_2+s_3)}+\frac{e^{s_1+s_2+s_3} s_3 k_{13}(-s_1-s_2-s_3,s_1+s_2)}{8 s_2 (s_1+s_2+s_3)}+\frac{e^{s_1+s_2} k_{13}(s_3,-s_2-s_3)}{8 (s_1+s_2+s_3)}+\frac{e^{s_1+s_2} k_{13}(s_3,-s_1-s_2-s_3)}{8 (s_1+s_2+s_3)}+\frac{e^{s_1+s_2} s_2 k_{13}(s_3,-s_1-s_2-s_3)}{8 s_1 (s_1+s_2+s_3)}+\frac{e^{s_1+s_2} s_3 k_{13}(s_3,-s_1-s_2-s_3)}{8 s_1 (s_1+s_2+s_3)}-\frac{e^{s_2} k_{13}(s_3,-s_2-s_3)}{8 (s_1+s_2+s_3)}-\frac{e^{s_2} s_1 G_1(s_1) k_{13}(s_3,-s_2-s_3)}{8 (s_1+s_2+s_3)}-\frac{e^{s_2} s_2 G_1(s_1) k_{13}(s_3,-s_2-s_3)}{8 (s_1+s_2+s_3)}-\frac{e^{s_2} s_3 G_1(s_1) k_{13}(s_3,-s_2-s_3)}{8 (s_1+s_2+s_3)}-\frac{e^{s_2} s_2 k_{13}(s_3,-s_2-s_3)}{8 s_1 (s_1+s_2+s_3)}-\frac{e^{s_2} s_3 k_{13}(s_3,-s_2-s_3)}{8 s_1 (s_1+s_2+s_3)}-\frac{e^{s_1+s_2+s_3} s_1 k_{13}(-s_1-s_2-s_3,s_1)}{8 s_2 (s_1+s_2+s_3)}-\frac{e^{s_1+s_2+s_3} s_3 k_{13}(-s_1-s_2-s_3,s_1)}{8 s_2 (s_1+s_2+s_3)})\bigm|_{s_j = \nab}\delta _1(h)\delta _2(h)(\delta _1 \delta _2(h))
\end{math}
\end{center}
\begin{center}
\begin{math}
+(\frac{s_1 k_{13}(s_1,s_2)}{8 s_3 (s_1+s_2+s_3)}+\frac{s_2 k_{13}(s_1,s_2)}{8 s_3 (s_1+s_2+s_3)}+\frac{e^{s_2+s_3} s_1 G_1(s_1) k_{13}(-s_2-s_3,s_2)}{8 (s_1+s_2+s_3)}+\frac{e^{s_2+s_3} s_2 G_1(s_1) k_{13}(-s_2-s_3,s_2)}{8 (s_1+s_2+s_3)}+\frac{e^{s_2+s_3} s_3 G_1(s_1) k_{13}(-s_2-s_3,s_2)}{8 (s_1+s_2+s_3)}+\frac{e^{s_2+s_3} k_{13}(-s_2-s_3,s_2)}{8 (s_1+s_2+s_3)}+\frac{e^{s_1+s_2+s_3} k_{13}(-s_2-s_3,s_2)}{8 (s_1+s_2+s_3)}+\frac{e^{s_2+s_3} s_2 k_{13}(-s_2-s_3,s_2)}{8 s_1 (s_1+s_2+s_3)}+\frac{e^{s_2+s_3} s_3 k_{13}(-s_2-s_3,s_2)}{8 s_1 (s_1+s_2+s_3)}+\frac{e^{s_1+s_2+s_3} s_1 k_{13}(-s_1-s_2-s_3,s_1+s_2)}{8 s_2 (s_1+s_2+s_3)}+\frac{e^{s_1+s_2+s_3} s_3 k_{13}(-s_1-s_2-s_3,s_1+s_2)}{8 s_2 (s_1+s_2+s_3)}-\frac{k_{13}(s_1,s_2+s_3)}{8 (s_1+s_2+s_3)}-\frac{e^{s_1+s_2+s_3} k_{13}(-s_1-s_2-s_3,s_1)}{8 (s_1+s_2+s_3)}-\frac{e^{s_1+s_2+s_3} s_2 k_{13}(-s_1-s_2-s_3,s_1+s_2)}{8 s_1 (s_1+s_2+s_3)}-\frac{e^{s_1+s_2+s_3} s_3 k_{13}(-s_1-s_2-s_3,s_1+s_2)}{8 s_1 (s_1+s_2+s_3)}-\frac{e^{s_1+s_2+s_3} s_1 k_{13}(-s_1-s_2-s_3,s_1)}{8 s_2 (s_1+s_2+s_3)}-\frac{e^{s_1+s_2+s_3} s_3 k_{13}(-s_1-s_2-s_3,s_1)}{8 s_2 (s_1+s_2+s_3)}-\frac{s_1 k_{13}(s_1,s_2+s_3)}{8 s_3 (s_1+s_2+s_3)}-\frac{s_2 k_{13}(s_1,s_2+s_3)}{8 s_3 (s_1+s_2+s_3)})\bigm|_{s_j = \nab}(\delta _1 \delta _2(h))\delta _1(h)\delta _2(h)
\end{math}
\end{center}
\begin{center}
\begin{math}
+(\frac{e^{s_1} s_1 k_{13}(s_2,-s_1-s_2)}{8 s_3 (s_1+s_2+s_3)}+\frac{e^{s_1} s_2 k_{13}(s_2,-s_1-s_2)}{8 s_3 (s_1+s_2+s_3)}+\frac{e^{s_1} k_{13}(s_2,s_3)}{8 (s_1+s_2+s_3)}+\frac{k_{13}(s_1+s_2,s_3)}{8 (s_1+s_2+s_3)}+\frac{s_2 k_{13}(s_1+s_2,s_3)}{8 s_1 (s_1+s_2+s_3)}+\frac{s_3 k_{13}(s_1+s_2,s_3)}{8 s_1 (s_1+s_2+s_3)}-\frac{s_1 G_1(s_1) k_{13}(s_2,s_3)}{8 (s_1+s_2+s_3)}-\frac{s_2 G_1(s_1) k_{13}(s_2,s_3)}{8 (s_1+s_2+s_3)}-\frac{s_3 G_1(s_1) k_{13}(s_2,s_3)}{8 (s_1+s_2+s_3)}-\frac{k_{13}(s_2,s_3)}{8 (s_1+s_2+s_3)}-\frac{e^{s_1} k_{13}(s_2+s_3,-s_1-s_2-s_3)}{8 (s_1+s_2+s_3)}-\frac{s_2 k_{13}(s_2,s_3)}{8 s_1 (s_1+s_2+s_3)}-\frac{s_3 k_{13}(s_2,s_3)}{8 s_1 (s_1+s_2+s_3)}-\frac{e^{s_1} s_1 k_{13}(s_2+s_3,-s_1-s_2-s_3)}{8 s_3 (s_1+s_2+s_3)}-\frac{e^{s_1} s_2 k_{13}(s_2+s_3,-s_1-s_2-s_3)}{8 s_3 (s_1+s_2+s_3)})\bigm|_{s_j = \nab}\delta _2(h)(\delta _1 \delta _2(h))\delta _1(h)
\end{math}
\end{center}
\begin{center}
\begin{math}
+(\frac{e^{s_3+s_4} G_1(s_1) k_{13}(-s_3-s_4,s_3)}{8 (s_1+s_2)}+\frac{e^{s_3+s_4} s_1 G_1(s_1) k_{13}(-s_3-s_4,s_3)}{8 s_2 (s_1+s_2)}+\frac{e^{s_3+s_4} s_1 G_2(s_1,s_2) k_{13}(-s_3-s_4,s_3)}{8 (s_1+s_2)}+\frac{e^{s_3+s_4} s_2 G_2(s_1,s_2) k_{13}(-s_3-s_4,s_3)}{8 (s_1+s_2)}+\frac{e^{s_3+s_4} k_{13}(-s_3-s_4,s_3)}{8 s_2 (s_1+s_2)}+\frac{e^{s_2+s_3+s_4} s_1 G_1(s_1) k_{13}(-s_2-s_3-s_4,s_2+s_3)}{8 (s_1+s_2) s_3}+\frac{e^{s_2+s_3+s_4} s_2 G_1(s_1) k_{13}(-s_2-s_3-s_4,s_2+s_3)}{8 (s_1+s_2) s_3}+\frac{e^{s_2+s_3+s_4} k_{13}(-s_2-s_3-s_4,s_2+s_3)}{8 (s_1+s_2) s_3}+\frac{e^{s_2+s_3+s_4} s_2 k_{13}(-s_2-s_3-s_4,s_2+s_3)}{8 s_1 (s_1+s_2) s_3}+\frac{e^{s_1+s_2+s_3+s_4} k_{13}(-s_1-s_2-s_3-s_4,s_1+s_2)}{8 (s_1+s_2) s_3}+\frac{e^{s_1+s_2+s_3+s_4} s_2 k_{13}(-s_1-s_2-s_3-s_4,s_1+s_2)}{8 s_1 (s_1+s_2) s_3}+\frac{e^{s_1+s_2+s_3+s_4} k_{13}(-s_1-s_2-s_3-s_4,s_1+s_2+s_3)}{8 s_1 (s_1+s_2)}-\frac{e^{s_2+s_3+s_4} G_1(s_1) k_{13}(-s_2-s_3-s_4,s_2+s_3)}{8 (s_1+s_2)}-\frac{e^{s_2+s_3+s_4} k_{13}(-s_2-s_3-s_4,s_2+s_3)}{8 s_1 (s_1+s_2)}-\frac{e^{s_2+s_3+s_4} k_{13}(-s_2-s_3-s_4,s_2+s_3)}{8 s_2 (s_1+s_2)}-\frac{e^{s_2+s_3+s_4} s_1 G_1(s_1) k_{13}(-s_2-s_3-s_4,s_2+s_3)}{8 s_2 (s_1+s_2)}-\frac{e^{s_2+s_3+s_4} k_{13}(-s_2-s_3-s_4,s_2)}{8 (s_1+s_2) s_3}-\frac{e^{s_2+s_3+s_4} s_1 G_1(s_1) k_{13}(-s_2-s_3-s_4,s_2)}{8 (s_1+s_2) s_3}-\frac{e^{s_2+s_3+s_4} s_2 G_1(s_1) k_{13}(-s_2-s_3-s_4,s_2)}{8 (s_1+s_2) s_3}-\frac{e^{s_1+s_2+s_3+s_4} k_{13}(-s_1-s_2-s_3-s_4,s_1+s_2+s_3)}{8 (s_1+s_2) s_3}-\frac{e^{s_2+s_3+s_4} s_2 k_{13}(-s_2-s_3-s_4,s_2)}{8 s_1 (s_1+s_2) s_3}-\frac{e^{s_1+s_2+s_3+s_4} s_2 k_{13}(-s_1-s_2-s_3-s_4,s_1+s_2+s_3)}{8 s_1 (s_1+s_2) s_3})\bigm|_{s_j = \nab}\delta _2(h)\delta _1(h)\delta _1(h)\delta _2(h)
\end{math}
\end{center}
\begin{center}
\begin{math}
+(-\frac{e^{s_2+s_3} s_2 G_1(s_1) k_{13}(-s_2-s_3,s_2)}{8 (s_2+s_3) s_4}-\frac{e^{s_2+s_3} s_3 G_1(s_1) k_{13}(-s_2-s_3,s_2)}{8 (s_2+s_3) s_4}-\frac{e^{s_2+s_3} s_2 k_{13}(-s_2-s_3,s_2)}{8 s_1 (s_2+s_3) s_4}-\frac{e^{s_2+s_3} s_3 k_{13}(-s_2-s_3,s_2)}{8 s_1 (s_2+s_3) s_4}+\frac{e^{s_1+s_2+s_3} k_{13}(-s_1-s_2-s_3,s_1)}{8 (s_2+s_3) s_4}+\frac{e^{s_1+s_2+s_3} s_3 k_{13}(-s_1-s_2-s_3,s_1)}{8 s_2 (s_2+s_3) s_4}+\frac{e^{s_1+s_2+s_3} s_2 k_{13}(-s_1-s_2-s_3,s_1+s_2)}{8 s_1 (s_2+s_3) s_4}+\frac{e^{s_1+s_2+s_3} s_3 k_{13}(-s_1-s_2-s_3,s_1+s_2)}{8 s_1 (s_2+s_3) s_4}+\frac{e^{s_2+s_3+s_4} s_2 G_1(s_1) k_{13}(-s_2-s_3-s_4,s_2)}{8 (s_2+s_3) s_4}+\frac{e^{s_2+s_3+s_4} s_3 G_1(s_1) k_{13}(-s_2-s_3-s_4,s_2)}{8 (s_2+s_3) s_4}+\frac{e^{s_2+s_3+s_4} s_2 k_{13}(-s_2-s_3-s_4,s_2)}{8 s_1 (s_2+s_3) s_4}+\frac{e^{s_2+s_3+s_4} s_3 k_{13}(-s_2-s_3-s_4,s_2)}{8 s_1 (s_2+s_3) s_4}+\frac{e^{s_2+s_3+s_4} G_1(s_1) k_{13}(-s_2-s_3-s_4,s_2+s_3)}{8 (s_2+s_3)}+\frac{e^{s_2+s_3+s_4} s_2 G_1(s_1) k_{13}(-s_2-s_3-s_4,s_2+s_3)}{8 s_3 (s_2+s_3)}+\frac{e^{s_2+s_3+s_4} k_{13}(-s_2-s_3-s_4,s_2+s_3)}{8 s_1 (s_2+s_3)}+\frac{e^{s_2+s_3+s_4} s_2 k_{13}(-s_2-s_3-s_4,s_2+s_3)}{8 s_1 s_3 (s_2+s_3)}+\frac{e^{s_1+s_2+s_3+s_4} k_{13}(-s_1-s_2-s_3-s_4,s_1)}{8 s_2 (s_2+s_3)}+\frac{e^{s_1+s_2+s_3+s_4} k_{13}(-s_1-s_2-s_3-s_4,s_1+s_2)}{8 s_1 (s_2+s_3)}+\frac{e^{s_1+s_2+s_3+s_4} s_2 k_{13}(-s_1-s_2-s_3-s_4,s_1+s_2)}{8 s_1 s_3 (s_2+s_3)}+\frac{e^{s_1+s_2+s_3+s_4} k_{13}(-s_1-s_2-s_3-s_4,s_1+s_2)}{8 (s_2+s_3) s_4}+\frac{e^{s_1+s_2+s_3+s_4} s_3 k_{13}(-s_1-s_2-s_3-s_4,s_1+s_2)}{8 s_2 (s_2+s_3) s_4}+\frac{e^{s_1+s_2+s_3+s_4} k_{13}(-s_1-s_2-s_3-s_4,s_1+s_2+s_3)}{8 s_3 (s_2+s_3)}-\frac{e^{s_2+s_3+s_4} G_1(s_1) k_{13}(-s_2-s_3-s_4,s_2)}{8 (s_2+s_3)}-\frac{e^{s_2+s_3+s_4} k_{13}(-s_2-s_3-s_4,s_2)}{8 s_1 (s_2+s_3)}-\frac{e^{s_1+s_2+s_3+s_4} k_{13}(-s_1-s_2-s_3-s_4,s_1+s_2+s_3)}{8 s_1 (s_2+s_3)}-\frac{e^{s_1+s_2+s_3+s_4} k_{13}(-s_1-s_2-s_3-s_4,s_1+s_2)}{8 s_2 (s_2+s_3)}-\frac{e^{s_2+s_3+s_4} s_2 G_1(s_1) k_{13}(-s_2-s_3-s_4,s_2)}{8 s_3 (s_2+s_3)}-\frac{e^{s_1+s_2+s_3+s_4} k_{13}(-s_1-s_2-s_3-s_4,s_1+s_2)}{8 s_3 (s_2+s_3)}-\frac{e^{s_2+s_3+s_4} s_2 k_{13}(-s_2-s_3-s_4,s_2)}{8 s_1 s_3 (s_2+s_3)}-\frac{e^{s_1+s_2+s_3+s_4} s_2 k_{13}(-s_1-s_2-s_3-s_4,s_1+s_2+s_3)}{8 s_1 s_3 (s_2+s_3)}-\frac{e^{s_1+s_2+s_3} k_{13}(-s_1-s_2-s_3,s_1+s_2)}{8 (s_2+s_3) s_4}-\frac{e^{s_1+s_2+s_3+s_4} k_{13}(-s_1-s_2-s_3-s_4,s_1)}{8 (s_2+s_3) s_4}-\frac{e^{s_1+s_2+s_3+s_4} s_2 k_{13}(-s_1-s_2-s_3-s_4,s_1+s_2)}{8 s_1 (s_2+s_3) s_4}-\frac{e^{s_1+s_2+s_3+s_4} s_3 k_{13}(-s_1-s_2-s_3-s_4,s_1+s_2)}{8 s_1 (s_2+s_3) s_4}-\frac{e^{s_1+s_2+s_3} s_3 k_{13}(-s_1-s_2-s_3,s_1+s_2)}{8 s_2 (s_2+s_3) s_4}-\frac{e^{s_1+s_2+s_3+s_4} s_3 k_{13}(-s_1-s_2-s_3-s_4,s_1)}{8 s_2 (s_2+s_3) s_4})\bigm|_{s_j = \nab}\delta _1(h)\delta _1(h)\delta _2(h)\delta _2(h)
\end{math}
\end{center}
\begin{center}
\begin{math}
+(-\frac{e^{s_2+s_3} G_1(s_1) k_{13}(-s_2-s_3,s_2)}{8 s_4}-\frac{e^{s_2+s_3} k_{13}(-s_2-s_3,s_2)}{8 s_1 s_4}+\frac{e^{s_1+s_2+s_3} k_{13}(-s_1-s_2-s_3,s_1+s_2)}{8 s_1 s_4}+\frac{e^{s_2+s_3+s_4} G_1(s_1) k_{13}(-s_2-s_3-s_4,s_2)}{8 s_4}+\frac{e^{s_2+s_3+s_4} k_{13}(-s_2-s_3-s_4,s_2)}{8 s_1 s_4}-\frac{e^{s_1+s_2+s_3+s_4} k_{13}(-s_1-s_2-s_3-s_4,s_1+s_2)}{8 s_1 s_4})\bigm|_{s_j = \nab}\delta _2(h)\delta _1(h)\delta _2(h)\delta _1(h)
\end{math}
\end{center}
\begin{center}
\begin{math}
+(\frac{e^{s_1+s_2} k_{13}(-s_1-s_2,s_1) s_2^2}{8 (s_1+s_2) s_3 (s_2+s_3) (s_3+s_4)}+\frac{e^{s_3+s_4} s_3 G_2(s_1,s_2) k_{13}(-s_3-s_4,s_3) s_2^2}{8 (s_1+s_2) (s_2+s_3) (s_3+s_4)}+\frac{e^{s_3+s_4} s_4 G_2(s_1,s_2) k_{13}(-s_3-s_4,s_3) s_2^2}{8 (s_1+s_2) (s_2+s_3) (s_3+s_4)}+\frac{e^{s_1+s_2+s_3+s_4} k_{13}(-s_1-s_2-s_3-s_4,s_1) s_2^2}{8 (s_1+s_2) (s_2+s_3) s_4 (s_3+s_4)}-\frac{e^{s_1+s_2+s_3} k_{13}(-s_1-s_2-s_3,s_1) s_2^2}{8 (s_1+s_2) s_3 (s_2+s_3) (s_3+s_4)}-\frac{e^{s_1+s_2+s_3} k_{13}(-s_1-s_2-s_3,s_1) s_2^2}{8 (s_1+s_2) (s_2+s_3) s_4 (s_3+s_4)}+\frac{e^{s_1+s_2} k_{13}(-s_1-s_2,s_1) s_2}{8 (s_1+s_2) (s_2+s_3) (s_3+s_4)}+\frac{e^{s_1+s_2} s_1 k_{13}(-s_1-s_2,s_1) s_2}{8 (s_1+s_2) s_3 (s_2+s_3) (s_3+s_4)}+\frac{e^{s_3+s_4} s_3 G_1(s_1) k_{13}(-s_3-s_4,s_3) s_2}{8 (s_1+s_2) (s_2+s_3) (s_3+s_4)}+\frac{e^{s_3+s_4} s_4 G_1(s_1) k_{13}(-s_3-s_4,s_3) s_2}{8 (s_1+s_2) (s_2+s_3) (s_3+s_4)}+\frac{e^{s_3+s_4} s_3^2 G_2(s_1,s_2) k_{13}(-s_3-s_4,s_3) s_2}{8 (s_1+s_2) (s_2+s_3) (s_3+s_4)}+\frac{e^{s_3+s_4} s_1 s_3 G_2(s_1,s_2) k_{13}(-s_3-s_4,s_3) s_2}{8 (s_1+s_2) (s_2+s_3) (s_3+s_4)}+\frac{e^{s_3+s_4} s_1 s_4 G_2(s_1,s_2) k_{13}(-s_3-s_4,s_3) s_2}{8 (s_1+s_2) (s_2+s_3) (s_3+s_4)}+\frac{e^{s_3+s_4} s_3 s_4 G_2(s_1,s_2) k_{13}(-s_3-s_4,s_3) s_2}{8 (s_1+s_2) (s_2+s_3) (s_3+s_4)}+\frac{e^{s_1+s_2+s_3+s_4} s_1 k_{13}(-s_1-s_2-s_3-s_4,s_1) s_2}{8 (s_1+s_2) (s_2+s_3) s_4 (s_3+s_4)}+\frac{e^{s_1+s_2+s_3+s_4} s_3 k_{13}(-s_1-s_2-s_3-s_4,s_1) s_2}{8 (s_1+s_2) (s_2+s_3) s_4 (s_3+s_4)}+\frac{e^{s_1+s_2+s_3+s_4} k_{13}(-s_1-s_2-s_3-s_4,s_1+s_2+s_3) s_2}{8 (s_1+s_2) (s_2+s_3) (s_3+s_4)}+\frac{e^{s_1+s_2+s_3+s_4} s_3 k_{13}(-s_1-s_2-s_3-s_4,s_1+s_2+s_3) s_2}{8 s_1 (s_1+s_2) (s_2+s_3) (s_3+s_4)}+\frac{e^{s_1+s_2+s_3+s_4} s_4 k_{13}(-s_1-s_2-s_3-s_4,s_1+s_2+s_3) s_2}{8 s_1 (s_1+s_2) (s_2+s_3) (s_3+s_4)}+\frac{e^{s_1+s_2+s_3+s_4} s_4 k_{13}(-s_1-s_2-s_3-s_4,s_1+s_2+s_3) s_2}{8 (s_1+s_2) s_3 (s_2+s_3) (s_3+s_4)}-\frac{e^{s_1+s_2+s_3} k_{13}(-s_1-s_2-s_3,s_1) s_2}{8 (s_1+s_2) (s_2+s_3) (s_3+s_4)}-\frac{e^{s_2+s_3+s_4} s_3 G_1(s_1) k_{13}(-s_2-s_3-s_4,s_2+s_3) s_2}{8 (s_1+s_2) (s_2+s_3) (s_3+s_4)}-\frac{e^{s_2+s_3+s_4} s_4 G_1(s_1) k_{13}(-s_2-s_3-s_4,s_2+s_3) s_2}{8 (s_1+s_2) (s_2+s_3) (s_3+s_4)}-\frac{e^{s_1+s_2+s_3+s_4} k_{13}(-s_1-s_2-s_3-s_4,s_1+s_2) s_2}{8 (s_1+s_2) (s_2+s_3) (s_3+s_4)}-\frac{e^{s_2+s_3+s_4} s_3 k_{13}(-s_2-s_3-s_4,s_2+s_3) s_2}{8 s_1 (s_1+s_2) (s_2+s_3) (s_3+s_4)}-\frac{e^{s_2+s_3+s_4} s_4 k_{13}(-s_2-s_3-s_4,s_2+s_3) s_2}{8 s_1 (s_1+s_2) (s_2+s_3) (s_3+s_4)}-\frac{e^{s_1+s_2+s_3} s_1 k_{13}(-s_1-s_2-s_3,s_1) s_2}{8 (s_1+s_2) s_3 (s_2+s_3) (s_3+s_4)}-\frac{e^{s_1+s_2+s_3+s_4} s_4 k_{13}(-s_1-s_2-s_3-s_4,s_1+s_2) s_2}{8 (s_1+s_2) s_3 (s_2+s_3) (s_3+s_4)}-\frac{e^{s_1+s_2+s_3} s_1 k_{13}(-s_1-s_2-s_3,s_1) s_2}{8 (s_1+s_2) (s_2+s_3) s_4 (s_3+s_4)}-\frac{e^{s_1+s_2+s_3} s_3 k_{13}(-s_1-s_2-s_3,s_1) s_2}{8 (s_1+s_2) (s_2+s_3) s_4 (s_3+s_4)}+\frac{e^{s_1+s_2} s_1 k_{13}(-s_1-s_2,s_1)}{8 (s_1+s_2) (s_2+s_3) (s_3+s_4)}+\frac{e^{s_3+s_4} s_3^2 G_1(s_1) k_{13}(-s_3-s_4,s_3)}{8 (s_1+s_2) (s_2+s_3) (s_3+s_4)}+\frac{e^{s_3+s_4} s_1 s_3 G_1(s_1) k_{13}(-s_3-s_4,s_3)}{8 (s_1+s_2) (s_2+s_3) (s_3+s_4)}+\frac{e^{s_3+s_4} s_1 s_4 G_1(s_1) k_{13}(-s_3-s_4,s_3)}{8 (s_1+s_2) (s_2+s_3) (s_3+s_4)}+\frac{e^{s_3+s_4} s_3 s_4 G_1(s_1) k_{13}(-s_3-s_4,s_3)}{8 (s_1+s_2) (s_2+s_3) (s_3+s_4)}+\frac{e^{s_3+s_4} s_1 s_3^2 G_1(s_1) k_{13}(-s_3-s_4,s_3)}{8 s_2 (s_1+s_2) (s_2+s_3) (s_3+s_4)}+\frac{e^{s_3+s_4} s_1 s_3 s_4 G_1(s_1) k_{13}(-s_3-s_4,s_3)}{8 s_2 (s_1+s_2) (s_2+s_3) (s_3+s_4)}+\frac{e^{s_3+s_4} s_1 s_3^2 G_2(s_1,s_2) k_{13}(-s_3-s_4,s_3)}{8 (s_1+s_2) (s_2+s_3) (s_3+s_4)}+\frac{e^{s_3+s_4} s_1 s_3 s_4 G_2(s_1,s_2) k_{13}(-s_3-s_4,s_3)}{8 (s_1+s_2) (s_2+s_3) (s_3+s_4)}+\frac{e^{s_3+s_4} s_3 k_{13}(-s_3-s_4,s_3)}{8 (s_1+s_2) (s_2+s_3) (s_3+s_4)}+\frac{e^{s_3+s_4} s_4 k_{13}(-s_3-s_4,s_3)}{8 (s_1+s_2) (s_2+s_3) (s_3+s_4)}+\frac{e^{s_3+s_4} s_3^2 k_{13}(-s_3-s_4,s_3)}{8 s_2 (s_1+s_2) (s_2+s_3) (s_3+s_4)}+\frac{e^{s_3+s_4} s_3 s_4 k_{13}(-s_3-s_4,s_3)}{8 s_2 (s_1+s_2) (s_2+s_3) (s_3+s_4)}+\frac{e^{s_1+s_2+s_3+s_4} s_3 k_{13}(-s_1-s_2-s_3-s_4,s_1)}{8 (s_1+s_2) (s_2+s_3) (s_3+s_4)}+\frac{e^{s_1+s_2+s_3+s_4} s_4 k_{13}(-s_1-s_2-s_3-s_4,s_1)}{8 (s_1+s_2) (s_2+s_3) (s_3+s_4)}+\frac{e^{s_1+s_2+s_3+s_4} s_1 s_3 k_{13}(-s_1-s_2-s_3-s_4,s_1)}{8 s_2 (s_1+s_2) (s_2+s_3) (s_3+s_4)}+\frac{e^{s_1+s_2+s_3+s_4} s_1 s_4 k_{13}(-s_1-s_2-s_3-s_4,s_1)}{8 s_2 (s_1+s_2) (s_2+s_3) (s_3+s_4)}+\frac{e^{s_1+s_2+s_3+s_4} s_1 s_3 k_{13}(-s_1-s_2-s_3-s_4,s_1)}{8 (s_1+s_2) (s_2+s_3) s_4 (s_3+s_4)}+\frac{e^{s_1+s_2+s_3+s_4} s_1 k_{13}(-s_1-s_2-s_3-s_4,s_1+s_2+s_3)}{8 (s_1+s_2) (s_2+s_3) (s_3+s_4)}+\frac{e^{s_1+s_2+s_3+s_4} s_3^2 k_{13}(-s_1-s_2-s_3-s_4,s_1+s_2+s_3)}{8 s_1 (s_1+s_2) (s_2+s_3) (s_3+s_4)}+\frac{e^{s_1+s_2+s_3+s_4} s_3 s_4 k_{13}(-s_1-s_2-s_3-s_4,s_1+s_2+s_3)}{8 s_1 (s_1+s_2) (s_2+s_3) (s_3+s_4)}+\frac{e^{s_1+s_2+s_3+s_4} s_1 s_4 k_{13}(-s_1-s_2-s_3-s_4,s_1+s_2+s_3)}{8 (s_1+s_2) s_3 (s_2+s_3) (s_3+s_4)}-\frac{e^{s_1+s_2+s_3} s_1 k_{13}(-s_1-s_2-s_3,s_1)}{8 (s_1+s_2) (s_2+s_3) (s_3+s_4)}-\frac{e^{s_2+s_3+s_4} s_3 k_{13}(-s_2-s_3-s_4,s_2+s_3)}{8 (s_1+s_2) (s_2+s_3) (s_3+s_4)}-\frac{e^{s_2+s_3+s_4} s_4 k_{13}(-s_2-s_3-s_4,s_2+s_3)}{8 (s_1+s_2) (s_2+s_3) (s_3+s_4)}-\frac{e^{s_2+s_3+s_4} s_3^2 G_1(s_1) k_{13}(-s_2-s_3-s_4,s_2+s_3)}{8 (s_1+s_2) (s_2+s_3) (s_3+s_4)}-\frac{e^{s_2+s_3+s_4} s_1 s_3 G_1(s_1) k_{13}(-s_2-s_3-s_4,s_2+s_3)}{8 (s_1+s_2) (s_2+s_3) (s_3+s_4)}-\frac{e^{s_2+s_3+s_4} s_1 s_4 G_1(s_1) k_{13}(-s_2-s_3-s_4,s_2+s_3)}{8 (s_1+s_2) (s_2+s_3) (s_3+s_4)}-\frac{e^{s_2+s_3+s_4} s_3 s_4 G_1(s_1) k_{13}(-s_2-s_3-s_4,s_2+s_3)}{8 (s_1+s_2) (s_2+s_3) (s_3+s_4)}-\frac{e^{s_1+s_2+s_3+s_4} s_1 k_{13}(-s_1-s_2-s_3-s_4,s_1+s_2)}{8 (s_1+s_2) (s_2+s_3) (s_3+s_4)}-\frac{e^{s_1+s_2+s_3+s_4} s_3 k_{13}(-s_1-s_2-s_3-s_4,s_1+s_2)}{8 (s_1+s_2) (s_2+s_3) (s_3+s_4)}-\frac{e^{s_1+s_2+s_3+s_4} s_4 k_{13}(-s_1-s_2-s_3-s_4,s_1+s_2)}{8 (s_1+s_2) (s_2+s_3) (s_3+s_4)}-\frac{e^{s_2+s_3+s_4} s_3^2 k_{13}(-s_2-s_3-s_4,s_2+s_3)}{8 s_1 (s_1+s_2) (s_2+s_3) (s_3+s_4)}-\frac{e^{s_2+s_3+s_4} s_3 s_4 k_{13}(-s_2-s_3-s_4,s_2+s_3)}{8 s_1 (s_1+s_2) (s_2+s_3) (s_3+s_4)}-\frac{e^{s_2+s_3+s_4} s_3^2 k_{13}(-s_2-s_3-s_4,s_2+s_3)}{8 s_2 (s_1+s_2) (s_2+s_3) (s_3+s_4)}-\frac{e^{s_2+s_3+s_4} s_3 s_4 k_{13}(-s_2-s_3-s_4,s_2+s_3)}{8 s_2 (s_1+s_2) (s_2+s_3) (s_3+s_4)}-\frac{e^{s_2+s_3+s_4} s_1 s_3^2 G_1(s_1) k_{13}(-s_2-s_3-s_4,s_2+s_3)}{8 s_2 (s_1+s_2) (s_2+s_3) (s_3+s_4)}-\frac{e^{s_2+s_3+s_4} s_1 s_3 s_4 G_1(s_1) k_{13}(-s_2-s_3-s_4,s_2+s_3)}{8 s_2 (s_1+s_2) (s_2+s_3) (s_3+s_4)}-\frac{e^{s_1+s_2+s_3+s_4} s_1 s_3 k_{13}(-s_1-s_2-s_3-s_4,s_1+s_2)}{8 s_2 (s_1+s_2) (s_2+s_3) (s_3+s_4)}-\frac{e^{s_1+s_2+s_3+s_4} s_1 s_4 k_{13}(-s_1-s_2-s_3-s_4,s_1+s_2)}{8 s_2 (s_1+s_2) (s_2+s_3) (s_3+s_4)}-\frac{e^{s_1+s_2+s_3+s_4} s_1 s_4 k_{13}(-s_1-s_2-s_3-s_4,s_1+s_2)}{8 (s_1+s_2) s_3 (s_2+s_3) (s_3+s_4)}-\frac{e^{s_1+s_2+s_3} s_1 s_3 k_{13}(-s_1-s_2-s_3,s_1)}{8 (s_1+s_2) (s_2+s_3) s_4 (s_3+s_4)})\bigm|_{s_j = \nab}\delta _1(h)\delta _2(h)\delta _1(h)\delta _2(h)
\end{math}
\end{center}
\begin{center}
\begin{math}
+(\frac{e^{s_1+s_2} k_{13}(-s_1-s_2,s_1)}{8 s_3 (s_3+s_4)}+\frac{e^{s_1+s_2+s_3} k_{13}(-s_1-s_2-s_3,s_1)}{8 s_2 (s_3+s_4)}+\frac{e^{s_1+s_2+s_3} s_3 k_{13}(-s_1-s_2-s_3,s_1)}{8 s_2 s_4 (s_3+s_4)}+\frac{e^{s_1+s_2+s_3+s_4} k_{13}(-s_1-s_2-s_3-s_4,s_1)}{8 s_4 (s_3+s_4)}+\frac{e^{s_1+s_2+s_3+s_4} k_{13}(-s_1-s_2-s_3-s_4,s_1+s_2)}{8 s_2 (s_3+s_4)}+\frac{e^{s_1+s_2+s_3+s_4} s_3 k_{13}(-s_1-s_2-s_3-s_4,s_1+s_2)}{8 s_2 s_4 (s_3+s_4)}-\frac{e^{s_1+s_2+s_3} k_{13}(-s_1-s_2-s_3,s_1+s_2)}{8 s_2 (s_3+s_4)}-\frac{e^{s_1+s_2+s_3+s_4} k_{13}(-s_1-s_2-s_3-s_4,s_1)}{8 s_2 (s_3+s_4)}-\frac{e^{s_1+s_2+s_3} k_{13}(-s_1-s_2-s_3,s_1)}{8 s_3 (s_3+s_4)}-\frac{e^{s_1+s_2+s_3} k_{13}(-s_1-s_2-s_3,s_1)}{8 s_4 (s_3+s_4)}-\frac{e^{s_1+s_2+s_3} s_3 k_{13}(-s_1-s_2-s_3,s_1+s_2)}{8 s_2 s_4 (s_3+s_4)}-\frac{e^{s_1+s_2+s_3+s_4} s_3 k_{13}(-s_1-s_2-s_3-s_4,s_1)}{8 s_2 s_4 (s_3+s_4)})\bigm|_{s_j = \nab}\delta _1(h)\delta _2(h)\delta _2(h)\delta _1(h). 
\end{math}
\end{center}
}

\subsection{Gradient of $ \frac{1}{16} \, \vphi_0 \left ( e^{h} K_{18} (\nab, \nab, \nab, \nab ) \left ( 
\delta _1(h )\cdot \delta _2(h )\cdot \delta
   _1(h)\cdot \delta _2(h )
\right )  \right )$ }
We need to calculate 
\[
\dep \Omega(h+\vep a), 
\]
where 
\[
\Omega(h) = \vphi_0 \left ( e^h K_{18} (\nab, \nab, \nab, \nab ) \left ( 
\delta _1(h )\cdot \delta _2(h )\cdot \delta
   _1(h )\cdot \delta _2(h)
\right )  \right ). 
\]
Using Lemma  \ref{LowerwithTrace1}, we have 
\[
\Omega(h)=  \vphi_0 \left ( e^h K (\nab, \nab, \nab ) \left ( 
\delta _1(h )\cdot \delta _2(h )\cdot \delta
   _1(h )
\right )  \right ) \, \delta _2(h),  
\]
where the function $K$ is defined as  
\[
K(s_1, s_2, s_3) = K_{18}(s_1, s_2, s_3, -s_1-s_2-s_3). 
\]
Therefore, under the trace $\vphi_0$, we have 
\begin{eqnarray*}
\Omega(h+\vep a) &=& e^{h+\vep a} K (\nab_\vep, \nab_\vep, \nab_\vep ) (\delta _1(h )\, \delta _2(h )\, \delta_1(h )) \, \del_2(h) \\
&&+\vep e^{h+\vep a} K (\nab_\vep, \nab_\vep, \nab_\vep ) (\delta _1(a)\, \delta _2(h )\, \delta_1(h )) \, \del_2(h)  \\
&&+\vep e^{h+\vep a} K (\nab_\vep, \nab_\vep, \nab_\vep ) (\delta _1(h)\, \delta _2(a)\, \delta_1(h )) \, \del_2(h) \\
&&+\vep e^{h+\vep a} K (\nab_\vep, \nab_\vep, \nab_\vep ) (\delta _1(h)\, \delta _2(h)\, \delta_1(a)) \, \del_2(h) \\
&&+\vep e^{h+\vep a} K (\nab_\vep, \nab_\vep, \nab_\vep ) (\delta _1(h)\, \delta _2(h)\, \delta_1(h)) \, \del_2(a) \\
&&+ \vep^2 (\cdots) + \vep^3 (\cdots)+ \vep^4 (\cdots). 
\end{eqnarray*}
Therefore 
\begin{eqnarray}  \label{preGr4}
\dep \Omega(h+\vep a) &=& \dep e^{h+\vep a}\,  K (\nab, \nab, \nab ) (\delta _1(h )\, \delta _2(h )\, \delta_1(h )) \, \del_2(h) \nonumber  \\
&&+ e^h \dep K (\nab_\vep, \nab_\vep, \nab_\vep ) (\delta _1(h )\, \delta _2(h )\, \delta_1(h )) \, \del_2(h)  \\
&& + e^{h} K (\nab, \nab, \nab ) (\delta _1(a)\, \delta _2(h )\, \delta_1(h )) \, \del_2(h) \nonumber \\
&& + e^{h} K (\nab, \nab, \nab ) (\delta _1(h)\, \delta _2(a )\, \delta_1(h )) \, \del_2(h) \nonumber \\
&& + e^{h} K (\nab, \nab, \nab ) (\delta _1(h)\, \delta _2(h )\, \delta_1(a )) \, \del_2(h) \nonumber \\
&& + e^{h} K (\nab, \nab, \nab ) (\delta _1(h)\, \delta _2(h )\, \delta_1(h )) \, \del_2(a).  \nonumber
\end{eqnarray}
The first term in \eqref{preGr4}, under the trace $\vphi_0$, is equal to 
\begin{eqnarray*}
&&\frac{1-e^{-\nab}}{\nab}(a) \, e^h \,  K (\nab, \nab, \nab ) (\delta _1(h )\, \delta _2(h )\, \delta_1(h )) \, \del_2(h) \\
&=&a e^h \left ( \frac{1-e^{s_1+s_2+s_3+s_4}}{-s_1-s_2-s_3-s_4} K(s_1, s_2, s_3)\right )  \bigm|_{s_1, s_2, s_3, s_4 = \nab} (\del_1(h) \, \del_2(h) \,\del_1(h) \, \del_2(h) ). 
\end{eqnarray*}

\smallskip

Using Lemma \ref{GrofFuncCalc3}, for the second term in \eqref{preGr4}, under $\vphi_0$, 
we have 
\[
 e^h \dep K (\nab_\vep, \nab_\vep, \nab_\vep ) (\delta _1(h )\, \delta _2(h )\, \delta_1(h )) \, \del_2(h) = 
\]
{\small
\[ a e^h  \Big \{ \left (  \frac{ K(s_1, s_2, s_3) - K(-s_2-s_3-s_4, s_2, s_3) }{e^{-s_1-s_2-s_3-s_4}(-s_1-s_2-s_3-s_4)}\right ) \bigm|_{s_j = \nab} ( \del_1(h)\del_2(h) \del_1(h) \del_2(h)) 
\]
\[
+  \left (  \frac{ K(-s_1-s_2-s_3, s_1, s_2) - K(s_4,-s_2-s_3-s_4, s_2) }{e^{-s_1-s_2-s_3}(-s_1-s_2-s_3-s_4)}\right ) \bigm|_{s_j = \nab} ( \del_2(h)\del_1(h) \del_2(h) \del_1(h))
\]
\[
+  \left (  \frac{ K(s_3,-s_1-s_2-s_3, s_1) - K(s_3, s_4,-s_2-s_3-s_4) }{e^{-s_1-s_2}(-s_1-s_2-s_3-s_4)}\right ) \bigm|_{s_j = \nab} ( \del_1(h)\del_2(h) \del_1(h) \del_2(h))
\]
\[
+  \left (  \frac{ K(s_2, s_3,-s_1-s_2-s_3) - K(s_2, s_3, s_4) }{e^{-s_1}(-s_1-s_2-s_3-s_4)}\right ) \bigm|_{s_j = \nab} ( \del_2(h) \del_1(h) \del_2(h) \del_1(h)) \Big \}. 
\]}

\smallskip

For the third term in \eqref{preGr4}, using Lemma \ref{ShiftwithTrace}, 
under the trace $\vphi_0$, we write: 
\begin{eqnarray*}
&& e^{h} K (\nab, \nab, \nab ) (\delta _1(a)\, \delta _2(h )\, \delta_1(h )) \, \del_2(h)\\
&=& e^h \del_1(a) K_v(\nab, \nab, \nab) (\del_2(h) \, \del_1(h)\, \del_2(h)) \\
&=& \del_1(a) e^h K_{vv}(\nab, \nab, \nab) (\del_2(h) \, \del_1(h)\, \del_2(h)) \\
&=& - a e^h e^{-h} \del_1(e^h) K_{vv}(\nab, \nab, \nab) (\del_2(h) \, \del_1(h)\, \del_2(h)) \\
&& - a e^h \del_1 (K_{vv}(\nab, \nab, \nab) (\del_2(h) \, \del_1(h)\, \del_2(h)) ),   
\end{eqnarray*}
where the functions $K_v$ and $K_{vv}$ are defined by 
\[
K_v(s_1, s_2, s_3) = K(-s_1-s_2-s_3, s_1, s_2), 
\qquad 
K_{vv}(s_1, s_2, s_3) = e^{s_1+s_2+s_3} K_v(s_1, s_2, s_3).   
\]
Here, $e^{-h} \del_1(e^h) $ can be expanded using Lemma \ref{ToLogBaby} and 
\[
\del_1 (K_{vv}(\nab, \nab, \nab) (\del_2(h) \, \del_1(h)\, \del_2(h)) )
\] 
can be expanded using Lemma \ref{DerFuncCalc3}.

\smallskip

For the fourth term in \eqref{preGr4}, using Lemma \ref{ShiftwithTrace}, 
under the trace $\vphi_0$, we write:
\begin{eqnarray*}
&& e^{h} K (\nab, \nab, \nab ) (\delta _1(h)\, \delta _2(a )\, \delta_1(h )) \, \del_2(h)\\
&=& \del_2(a) K_w(\nab, \nab, \nab )(\del_1(h) \cdot \del_2(h)\cdot e^h \del_1(h)) \\
&=& \del_2(a) e^h K_{ww}(\nab, \nab, \nab )(\del_1(h) \, \del_2(h)\,\del_1(h)) \\
&=& - ae^h e^{-h} \del_2(e^h) K_{ww}(\nab, \nab, \nab )(\del_1(h) \, \del_2(h)\,\del_1(h)) \\
&&-  ae^h  \del_2( K_{ww}(\nab, \nab, \nab )(\del_1(h) \, \del_2(h)\,\del_1(h)) ), 
\end{eqnarray*}
where the functions $K_w$ and $K_{ww}$ are given by 
\[
K_w(s_1, s_2, s_3) = K(s_3, -s_1-s_2-s_3, s_1), 
\qquad 
K_{ww}(s_1, s_2, s_3) = e^{s_1+s_2} K_w(s_1, s_2, s_3).
\]
Here also, $e^{-h} \del_2(e^h) $ can be expanded using Lemma \ref{ToLogBaby} and 
\[
\del_2( K_{ww}(\nab, \nab, \nab )(\del_1(h) \, \del_2(h)\,\del_1(h)) )
\] can be expanded 
using Lemma \ref{DerFuncCalc3}.

\smallskip

For the fifth term in \eqref{preGr4}, using Lemma \ref{ShiftwithTrace}, 
under the trace $\vphi_0$, we have:
\begin{eqnarray*}
&&e^{h} K (\nab, \nab, \nab ) (\delta _1(h)\, \delta _2(h )\, \delta_1(a )) \, \del_2(h) \\
&=&\del_1(a) K_z(\nab, \nab, \nab )(\del_2(h) \cdot e^h \del_1(h)\cdot  \del_2(h)) \\
&=&\del_1(a)e^h K_{zz}(\nab, \nab, \nab )(\del_2(h) \,  \del_1(h)\,  \del_2(h)) \\ 
&=& - a e^h e^{-h} \del_1(e^h) K_{zz}(\nab, \nab, \nab )(\del_2(h) \,  \del_1(h)\,  \del_2(h)) \\
&&-  a e^h  \del_1(K_{zz}(\nab, \nab, \nab )(\del_2(h) \,  \del_1(h)\,  \del_2(h))), 
\end{eqnarray*}
where 
\[
K_z(s_1, s_2, s_3) = K(s_2, s_3, -s_1-s_2-s_3), 
\qquad
K_{zz}(s_1, s_2, s_3) = e^{s_1} K_z(s_1, s_2, s_3). 
\]
Similar to the previous cases, Lemma \ref{ToLogBaby} and Lemma \ref{DerFuncCalc3} 
can be used for further expansions of the above expression. 

\smallskip

Finally, for the sixth term in \eqref{preGr4}, under $\vphi_0$, we have 
\begin{eqnarray*}
&&e^{h} K (\nab, \nab, \nab ) (\delta _1(h)\, \delta _2(h )\, \delta_1(h )) \, \del_2(a) \\
&=& \del_2(a) e^{h} K (\nab, \nab, \nab ) (\delta _1(h)\, \delta _2(h )\, \delta_1(h )) \\
&=& -ae^h e^{-h} \del_2(e^{h}) K (\nab, \nab, \nab ) (\delta _1(h)\, \delta _2(h )\, \delta_1(h )) \\
&& - a e^h \del_2(K (\nab, \nab, \nab ) (\delta _1(h)\, \delta _2(h )\, \delta_1(h ))). 
\end{eqnarray*}
Similarly to the previous cases, we apply Lemma \ref{ToLogBaby} and Lemma \ref{DerFuncCalc3} 
for expanding the latter further. 

\smallskip

Putting everything together, up to multiplying the right hand side of the following equality by $a e^h$, 
under the trace $\vphi_0$, we have 
{\small
\[
\dep
\frac{1}{16} \, \vphi_0 \left ( e^{h+\vep a} K_{18} (\nab_\vep, \nab_\vep, \nab_\vep, \nab_\vep ) \left ( 
\delta _1(h + \vep a)\cdot \delta _2(h + \vep a )\cdot \delta
   _1(h + \vep a)\cdot \delta _2(h + \vep a )
\right )  \right )
\]
}
{\small
\begin{center}
\begin{math}
=(-\frac{1}{16} e^{s_1} k_{18}(s_2,s_3,-s_1-s_2-s_3)-\frac{1}{16} e^{s_1+s_2+s_3} k_{18}(-s_1-s_2-s_3,s_1,s_2))\bigm|_{s_j = \nab}\delta _2(h)\delta _1(h)(\delta _1 \delta _2(h))
\end{math}
\end{center}
\begin{center}
\begin{math}
+(-\frac{1}{16} e^{s_1} k_{18}(s_2,s_3,-s_1-s_2-s_3)-\frac{1}{16} e^{s_1+s_2+s_3} k_{18}(-s_1-s_2-s_3,s_1,s_2))\bigm|_{s_j = \nab}\delta _2(h)\delta _1^2(h)\delta _2(h)
\end{math}
\end{center}
\begin{center}
\begin{math}
+(-\frac{1}{16} e^{s_1} k_{18}(s_2,s_3,-s_1-s_2-s_3)-\frac{1}{16} e^{s_1+s_2+s_3} k_{18}(-s_1-s_2-s_3,s_1,s_2))\bigm|_{s_j = \nab}(\delta _1 \delta _2(h))\delta _1(h)\delta _2(h)
\end{math}
\end{center}
\begin{center}
\begin{math}
+(-\frac{1}{16} k_{18}(s_1,s_2,s_3)-\frac{1}{16} e^{s_1+s_2} k_{18}(s_3,-s_1-s_2-s_3,s_1))\bigm|_{s_j = \nab}\delta _1(h)\delta _2(h)(\delta _1 \delta _2(h))
\end{math}
\end{center}
\begin{center}
\begin{math}
+(-\frac{1}{16} k_{18}(s_1,s_2,s_3)-\frac{1}{16} e^{s_1+s_2} k_{18}(s_3,-s_1-s_2-s_3,s_1))\bigm|_{s_j = \nab}\delta _1(h)\delta _2^2(h)\delta _1(h)
\end{math}
\end{center}
\begin{center}
\begin{math}
+(-\frac{1}{16} k_{18}(s_1,s_2,s_3)-\frac{1}{16} e^{s_1+s_2} k_{18}(s_3,-s_1-s_2-s_3,s_1))\bigm|_{s_j = \nab}(\delta _1 \delta _2(h))\delta _2(h)\delta _1(h)
\end{math}
\end{center}
\begin{center}
\begin{math}
+(-\frac{e^{s_1} k_{18}(s_2+s_3,s_4,-s_1-s_2-s_3-s_4)-e^{s_1} k_{18}(s_2,s_3+s_4,-s_1-s_2-s_3-s_4)}{16 s_3}-\frac{e^{s_1+s_2} k_{18}(s_3,s_4,-s_1-s_2-s_3-s_4)-e^{s_1} k_{18}(s_2+s_3,s_4,-s_1-s_2-s_3-s_4)}{16 s_2}-\frac{e^{s_1+s_2+s_3+s_4} k_{18}(-s_1-s_2-s_3-s_4,s_1,s_2+s_3)-e^{s_1+s_2+s_3+s_4} k_{18}(-s_1-s_2-s_3-s_4,s_1,s_2)}{16 s_3}-\frac{e^{s_1+s_2+s_3+s_4} k_{18}(-s_1-s_2-s_3-s_4,s_1+s_2,s_3)-e^{s_1+s_2+s_3+s_4} k_{18}(-s_1-s_2-s_3-s_4,s_1,s_2+s_3)}{16 s_2})\bigm|_{s_j = \nab}\delta _2(h)\delta _1(h)\delta _1(h)\delta _2(h)
\end{math}
\end{center}
\begin{center}
\begin{math}
+(-\frac{1}{16} G_1(s_1) k_{18}(s_2,s_3,s_4)+\frac{e^{s_1} (k_{18}(s_2,s_3,s_4)-k_{18}(s_2,s_3,-s_1-s_2-s_3))}{16 (s_1+s_2+s_3+s_4)}+\frac{k_{18}(s_1+s_2,s_3,s_4)-k_{18}(s_2,s_3,s_4)}{16 s_1}-\frac{e^{s_1+s_2+s_3+s_4} k_{18}(-s_1-s_2-s_3-s_4,s_1,s_2)-e^{s_1+s_2+s_3} k_{18}(-s_1-s_2-s_3,s_1,s_2)}{16 s_4}-\frac{1}{16} e^{s_2+s_3} G_1(s_1) k_{18}(s_4,-s_2-s_3-s_4,s_2)+\frac{e^{s_1+s_2+s_3} (k_{18}(s_4,-s_2-s_3-s_4,s_2)-k_{18}(-s_1-s_2-s_3,s_1,s_2))}{16 (s_1+s_2+s_3+s_4)}+\frac{e^{s_1+s_2+s_3} k_{18}(s_4,-s_1-s_2-s_3-s_4,s_1+s_2)-e^{s_2+s_3} k_{18}(s_4,-s_2-s_3-s_4,s_2)}{16 s_1}-\frac{e^{s_1} k_{18}(s_2,s_3+s_4,-s_1-s_2-s_3-s_4)-e^{s_1} k_{18}(s_2,s_3,-s_1-s_2-s_3)}{16 s_4})\bigm|_{s_j = \nab}\delta _2(h)\delta _1(h)\delta _2(h)\delta _1(h)
\end{math}
\end{center}
\begin{center}
\begin{math}
+(\frac{(1-e^{s_1+s_2+s_3+s_4}) k_{18}(s_1,s_2,s_3)}{16 (-s_1-s_2-s_3-s_4)}-\frac{1}{16} e^{s_2} G_1(s_1) k_{18}(s_3,s_4,-s_2-s_3-s_4)+\frac{e^{s_1+s_2} (k_{18}(s_3,s_4,-s_2-s_3-s_4)-k_{18}(s_3,-s_1-s_2-s_3,s_1))}{16 (s_1+s_2+s_3+s_4)}+\frac{e^{s_1+s_2} k_{18}(s_3,s_4,-s_1-s_2-s_3-s_4)-e^{s_2} k_{18}(s_3,s_4,-s_2-s_3-s_4)}{16 s_1}-\frac{1}{16} e^{s_2+s_3+s_4} G_1(s_1) k_{18}(-s_2-s_3-s_4,s_2,s_3)+\frac{e^{s_1+s_2+s_3+s_4} (k_{18}(-s_2-s_3-s_4,s_2,s_3)-k_{18}(s_1,s_2,s_3))}{16 (s_1+s_2+s_3+s_4)}+\frac{e^{s_1+s_2+s_3+s_4} k_{18}(-s_1-s_2-s_3-s_4,s_1+s_2,s_3)-e^{s_2+s_3+s_4} k_{18}(-s_2-s_3-s_4,s_2,s_3)}{16 s_1}-\frac{e^{s_1+s_2} k_{18}(s_3+s_4,-s_1-s_2-s_3-s_4,s_1)-e^{s_1+s_2} k_{18}(s_3,-s_1-s_2-s_3,s_1)}{16 s_4}-\frac{k_{18}(s_1,s_2,s_3+s_4)-k_{18}(s_1,s_2,s_3)}{16 s_4})\bigm|_{s_j = \nab}\delta _1(h)\delta _2(h)\delta _1(h)\delta _2(h)
\end{math}
\end{center}
\begin{center}
\begin{math}
+(-\frac{k_{18}(s_1,s_2+s_3,s_4)-k_{18}(s_1,s_2,s_3+s_4)}{16 s_3}-\frac{e^{s_1+s_2+s_3} k_{18}(s_4,-s_1-s_2-s_3-s_4,s_1+s_2)-e^{s_1+s_2+s_3} k_{18}(s_4,-s_1-s_2-s_3-s_4,s_1)}{16 s_2}-\frac{e^{s_1+s_2+s_3} k_{18}(s_4,-s_1-s_2-s_3-s_4,s_1)-e^{s_1+s_2} k_{18}(s_3+s_4,-s_1-s_2-s_3-s_4,s_1)}{16 s_3}-\frac{k_{18}(s_1+s_2,s_3,s_4)-k_{18}(s_1,s_2+s_3,s_4)}{16 s_2})\bigm|_{s_j = \nab}\delta _1(h)\delta _2(h)\delta _2(h)\delta _1(h). 
\end{math}
\end{center}
}

\section{Functional relations among the functions $k_3, \dots, k_{20}$}
\label{FuncRelationsforksSec}

In the results presented so far, there are indications about 
functional relations that involve only the functions $k_3, \dots, k_{20}$. 
For example, the basic functional relations presented in Section \ref{a_4Sec} 
and in Appendix \ref{lengthyfnrelationsappsec} 
are derived by setting equal a function $\widetilde K_j $ that appears 
in the first calculation of the gradient in Section \ref{FuncRelationsProofSec} with the corresponding 
term in the outcome of the second calculation of the same gradient in terms 
of finite differences of the functions $k_j$ in Section \ref{GradientCalculationSec}. 
Since the operator associated with 
a function $\widetilde K_j $ acts on different elements of $\CNT$, this raises 
the question whether there are different finite difference expressions for the 
$\widetilde K_j$, which would provide a functional relation that  involves  
only the functions $k_j$. 

\smallskip
By studying the result of the second gradient 
calculation from this perspective, we find that if in the formula \eqref{Gradientofa4} 
the operator associated with a function $\widetilde K_j$ acts on two 
elements of $\CNT$ that are the same modulo switching $\done$ 
and $\dtwo$, then the corresponding finite difference expressions 
in the second calculation are precisely the same. However, there are 
two cases, namely for $\widetilde K_6$ and $\widetilde K_7$, where 
the corresponding operators act on different kinds of elements of $\CNT$ 
that are not the same up to  switching  $\done$ and $\dtwo$. It turns 
out that in these cases, the corresponding finite difference expressions 
have different ingredients, hence functional relations of the desired type. 
We should also recall that in a closely related manner, in Corollary \ref{FnRlnK6DiffEqns} 
and Corollary \ref{FnRlnK7DiffEqns} we presented functional relations that 
involve restrictions of the functions $k_j$. We will see shortly that there are more 
functional relations between the $k_j$, which are derived from the fact that the final 
explicit expressions for some of the functions $K_j$ are identical up to multiplication by 
scalars. We should note that the explicit expressions for $K_j$ are given in Section \ref{ExplicitFormulasSec} 
and in  Appendix \ref{explicit3and4fnsappsec}.

\begin{theorem} \label{funcrlnsamongksThm}
There are  functional relations that involve only the functions $k_3,$ $\dots,$  $k_{20},$ 
as presented in the following subsections and in Appendix \ref{lengthykfnrelationsappsec}. 
\end{theorem}

We now start writing the new functional relations and will explain more specifically how each 
relation is derived. 

\subsection{Functional relation between $k_3,$ $k_4,$ $k_5,$ $k_8,$ $ k_9,$ 
$\dots, $ $ k_{16}$ }

In \eqref{Gradientofa4}, the operator $\widetilde K_6(\nab, \nab)$ acts on two kinds of elements, 
namely $\done(h) \cdot \done^3(h)$ and $\done(h) \cdot \done \dtwo^2(h)$, from which 
one can obtain the other two elements by switching $\done$ and $\dtwo$. On the other hand, 
in the outcome of the second calculation of the same gradient the finite difference expressions 
of $\nab$ acting on these elements have different ingredients, hence we have two different 
finite different expressions for $\widetilde K_6(\nab, \nab)$ and thereby obtain the following 
functional relation: 

\begin{center}
\begin{math}
-s_1 k_8(s_1,s_2)+e^{s_1+s_2} s_1 k_8(-s_1-s_2,s_1)-s_1 k_9(s_1,s_2)+e^{s_1+s_2} s_1 k_9(-s_1-s_2,s_1)+s_1 k_{10}(s_1,s_2)-e^{s_1} s_1 k_{10}(s_2,-s_1-s_2)+s_1 k_{11}(s_1,s_2)-e^{s_1} s_1 k_{11}(s_2,-s_1-s_2)-e^{s_1+s_2} s_1 k_{12}(-s_1-s_2,s_1)+e^{s_1} s_1 k_{12}(s_2,-s_1-s_2)-e^{s_1+s_2} s_1 k_{13}(-s_1-s_2,s_1)+e^{s_1} s_1 k_{13}(s_2,-s_1-s_2)+e^{s_1+s_2} s_1 k_{14}(-s_1-s_2,s_1)-e^{s_1} s_1 k_{14}(s_2,-s_1-s_2)+s_1 k_{15}(s_1,s_2)-e^{s_1+s_2} s_1 k_{15}(-s_1-s_2,s_1)-s_1 k_{16}(s_1,s_2)+e^{s_1} s_1 k_{16}(s_2,-s_1-s_2)-2 e^{s_2} (s_1 G_1(s_1)+1) k_3(-s_2)-2 s_1 G_1(s_1) k_3(s_2)-4 e^{s_2} s_1 G_1(s_1) k_4(-s_2)-4 s_1 G_1(s_1) k_4(s_2)+4 e^{s_2} s_1 G_1(s_1) k_5(-s_2)+4 s_1 G_1(s_1) k_5(s_2)+2 e^{s_1+s_2} k_3(-s_1-s_2)-2 k_3(s_2)+2 k_3(s_1+s_2)+4 e^{s_1+s_2} k_4(-s_1-s_2)-4 e^{s_2} k_4(-s_2)-4 k_4(s_2)+4 k_4(s_1+s_2)-4 e^{s_1+s_2} k_5(-s_1-s_2)+4 e^{s_2} k_5(-s_2)+4 k_5(s_2)-4 k_5(s_1+s_2) = 0. 
\end{math}
\end{center}

\subsection{Another functional relation between  $k_3,$ $k_4,$ $k_5,$ $k_8,$ $ k_9,$ $\dots,$ $ k_{16}$  }
The situation in formula \eqref{Gradientofa4} for $\widetilde K_7(\nab, \nab)$
is similar to the case of its preceding term with $\widetilde K_6(\nab, \nab)$ involved: 
there are two kinds of elements that $\widetilde K_7(\nab, \nab)$ acts on, 
namely $\done^3(h) \cdot \done(h)$ and $\done \dtwo^2(h) \cdot \done(h)$, 
that are not the same up to switching $\done$ and $\dtwo$, and in the second 
gradient calculation, there are correspondingly two different finite difference 
expressions for $\widetilde K_7$. This yields the following functional relation: 
\begin{center}
\begin{math}
-s_2 k_8(s_1,s_2)+e^{s_1} s_2 k_8(s_2,-s_1-s_2)-s_2 k_9(s_1,s_2)+e^{s_1} s_2 k_9(s_2,-s_1-s_2)+e^{s_1+s_2} s_2 k_{10}(-s_1-s_2,s_1)-e^{s_1} s_2 k_{10}(s_2,-s_1-s_2)+e^{s_1+s_2} s_2 k_{11}(-s_1-s_2,s_1)-e^{s_1} s_2 k_{11}(s_2,-s_1-s_2)+s_2 k_{12}(s_1,s_2)-e^{s_1+s_2} s_2 k_{12}(-s_1-s_2,s_1)+s_2 k_{13}(s_1,s_2)-e^{s_1+s_2} s_2 k_{13}(-s_1-s_2,s_1)-s_2 k_{14}(s_1,s_2)+e^{s_1+s_2} s_2 k_{14}(-s_1-s_2,s_1)+s_2 k_{15}(s_1,s_2)-e^{s_1} s_2 k_{15}(s_2,-s_1-s_2)-e^{s_1+s_2} s_2 k_{16}(-s_1-s_2,s_1)+e^{s_1} s_2 k_{16}(s_2,-s_1-s_2)+2 e^{s_1} k_3(-s_1)+2 k_3(s_1)-2 e^{s_1+s_2} k_3(-s_1-s_2)-2 k_3(s_1+s_2)+4 e^{s_1} k_4(-s_1)+4 k_4(s_1)-4 e^{s_1+s_2} k_4(-s_1-s_2)-4 k_4(s_1+s_2)-4 e^{s_1} k_5(-s_1)-4 k_5(s_1)+4 e^{s_1+s_2} k_5(-s_1-s_2)+4 k_5(s_1+s_2) = 0. 
\end{math}
\end{center}

 \subsection{Functional relation between $k_3,$ $k_4,$ $k_5$} 
 The explicit final formulas for the functions $K_1, \dots, K_{20}$ are provided in 
 Section \ref{ExplicitFormulasSec} and in Appendix \ref{explicit3and4fnsappsec}. One can see by using these 
 explicit formulas that some of these functions are scalar multiples of each other. 
 For example we have $2 K_2(s_1) = K_1(s_1)$, which clearly implies that  
 $2 \widetilde K_2 = \widetilde K_1$. However, the functional relations given 
 by \eqref{basicK1eqn} and \eqref{basicK2eqn} for $\widetilde K_1$ and $\widetilde K_2$ apparently 
 have different ingredients, which yields the following functional relation: 
 \[
-e^{s_1} k_3\left(-s_1\right)-k_3\left(s_1\right)+2 \left(-e^{s_1} k_4\left(-s_1\right)-k_4\left(s_1\right)+e^{s_1} k_5\left(-s_1\right)+k_5\left(s_1\right)\right)
=0. 
\]

\subsection{Functional relation between  $k_3,$ $k_4,$ $k_6,$ $k_7,$ $k_8,$ $\dots,$ $ k_{13},$ $ k_{17},$ $ k_{18},$ $ k_{19}$ } In a very similar manner, 
one can use the final explicit formulas given in Appendix \ref{explicit3and4fnsappsec} to see that 
$2 \widetilde K_{11} = \widetilde K_{10}$. Therefore, one can then use the finite 
difference expressions given by \eqref{basicK10eqn} and \eqref{basicK11eqn} to obtain the following functional relation: 

\smallskip

{\small
\begin{center}
\begin{math}
\frac{e^{s_1+s_2} k_3(-s_1-s_2)}{4 s_1 s_3}+\frac{k_3(s_1+s_2)}{4 s_1 s_3}+\frac{e^{s_2+s_3} G_1(s_1) k_3(-s_2-s_3)}{4 s_3}+\frac{e^{s_2+s_3} k_3(-s_2-s_3)}{4 s_1 s_3}+\frac{G_1(s_1) k_3(s_2+s_3)}{4 s_3}+\frac{k_3(s_2+s_3)}{4 s_1 s_3}+\frac{e^{s_2} G_1(s_1) k_4(-s_2)}{s_3}+\frac{e^{s_2} k_4(-s_2)}{s_1 s_3}+\frac{G_1(s_1) k_4(s_2)}{s_3}+\frac{k_4(s_2)}{s_1 s_3}+\frac{e^{s_1+s_2+s_3} k_4(-s_1-s_2-s_3)}{s_1 s_3}+\frac{k_4(s_1+s_2+s_3)}{s_1 s_3}+\frac{k_6(s_1)}{2 s_2 (s_2+s_3)}+\frac{k_6(s_1)}{2 s_3 (s_2+s_3)}+\frac{k_6(s_3)}{4 s_1 (s_1+s_2)}+\frac{k_6(s_3)}{4 s_2 (s_1+s_2)}+\frac{k_6(s_1+s_2+s_3)}{4 s_1 s_2}+\frac{k_6(s_1+s_2+s_3)}{2 s_2 s_3}+\frac{e^{s_1} k_7(-s_1)}{2 s_2 (s_2+s_3)}+\frac{e^{s_1} k_7(-s_1)}{2 s_3 (s_2+s_3)}+\frac{e^{s_1+s_2+s_3} k_7(-s_1-s_2-s_3)}{4 s_1 s_2}+\frac{e^{s_1+s_2+s_3} k_7(-s_1-s_2-s_3)}{2 s_2 s_3}+\frac{e^{s_3} k_7(-s_3)}{4 s_1 (s_1+s_2)}+\frac{e^{s_3} k_7(-s_3)}{4 s_2 (s_1+s_2)}+\frac{k_8(s_1,s_2+s_3)}{4 s_2}+\frac{k_8(s_1,s_2+s_3)}{4 s_3}+\frac{e^{s_1+s_2} k_8(-s_1-s_2,s_1)}{4 (s_1+s_2+s_3)}+\frac{1}{4} G_1(s_1) k_8(s_2,s_3)+\frac{k_8(s_2,s_3)}{4 s_1}+\frac{e^{s_1+s_2+s_3} k_8(-s_1-s_2-s_3,s_1)}{4 s_3}+\frac{1}{4} e^{s_2} G_1(s_1) k_8(s_3,-s_2-s_3)+\frac{e^{s_2} k_8(s_3,-s_2-s_3)}{4 s_1}+\frac{k_9(s_1,s_2)}{8 s_3}+\frac{e^{s_1+s_2} k_9(-s_1-s_2,s_1)}{8 s_3}-\frac{1}{8} G_1(s_1) k_9(s_2,s_3)+\frac{k_9(s_1+s_2,s_3)}{8 s_1}+\frac{k_9(s_1+s_2,s_3)}{8 s_2}-\frac{1}{8} e^{s_2} G_1(s_1) k_9(s_3,-s_2-s_3)+\frac{e^{s_1+s_2} k_9(s_3,-s_2-s_3)}{8 (s_1+s_2+s_3)}+\frac{e^{s_1+s_2} k_9(s_3,-s_1-s_2-s_3)}{8 s_1}+\frac{k_{10}(s_1,s_2)}{8 (-s_1-s_2-s_3)}+\frac{k_{10}(s_1,s_2)}{8 s_3}+\frac{e^{s_1} k_{10}(s_2,-s_1-s_2)}{8 s_3}-\frac{1}{8} e^{s_2+s_3} G_1(s_1) k_{10}(-s_2-s_3,s_2)+\frac{e^{s_1+s_2+s_3} k_{10}(-s_2-s_3,s_2)}{8 (s_1+s_2+s_3)}+\frac{e^{s_1+s_2+s_3} k_{10}(-s_1-s_2-s_3,s_1+s_2)}{8 s_1}-\frac{1}{8} e^{s_2} G_1(s_1) k_{10}(s_3,-s_2-s_3)+\frac{e^{s_1+s_2} k_{10}(s_3,-s_1-s_2-s_3)}{8 s_1}+\frac{e^{s_1+s_2} k_{10}(s_3,-s_1-s_2-s_3)}{8 s_2}+\frac{e^{s_1+s_2+s_3} k_{11}(s_1,s_2)}{4 (-s_1-s_2-s_3)}+\frac{e^{s_1+s_2+s_3} k_{11}(s_1,s_2)}{4 (s_1+s_2+s_3)}+\frac{k_{11}(s_1,s_2+s_3)}{4 s_3}+\frac{1}{4} e^{s_2+s_3} G_1(s_1) k_{11}(-s_2-s_3,s_2)+\frac{e^{s_2+s_3} k_{11}(-s_2-s_3,s_2)}{4 s_1}+\frac{1}{4} e^{s_2} G_1(s_1) k_{11}(s_3,-s_2-s_3)+\frac{e^{s_2} k_{11}(s_3,-s_2-s_3)}{4 s_1}+\frac{e^{s_1} k_{11}(s_2+s_3,-s_1-s_2-s_3)}{4 s_2}+\frac{e^{s_1} k_{11}(s_2+s_3,-s_1-s_2-s_3)}{4 s_3}+\frac{e^{s_1} k_{12}(s_2,-s_1-s_2)}{4 (s_1+s_2+s_3)}+\frac{1}{4} G_1(s_1) k_{12}(s_2,s_3)+\frac{k_{12}(s_2,s_3)}{4 s_1}+\frac{1}{4} e^{s_2+s_3} G_1(s_1) k_{12}(-s_2-s_3,s_2)+\frac{e^{s_2+s_3} k_{12}(-s_2-s_3,s_2)}{4 s_1}+\frac{e^{s_1+s_2+s_3} k_{12}(-s_1-s_2-s_3,s_1)}{4 s_2}+\frac{e^{s_1+s_2+s_3} k_{12}(-s_1-s_2-s_3,s_1)}{4 s_3}+\frac{e^{s_1} k_{12}(s_2+s_3,-s_1-s_2-s_3)}{4 s_3}+\frac{e^{s_1+s_2} k_{13}(-s_1-s_2,s_1)}{8 s_3}+\frac{e^{s_1} k_{13}(s_2,-s_1-s_2)}{8 s_3}-\frac{1}{8} G_1(s_1) k_{13}(s_2,s_3)+\frac{e^{s_1} k_{13}(s_2,s_3)}{8 (s_1+s_2+s_3)}+\frac{k_{13}(s_1+s_2,s_3)}{8 s_1}-\frac{1}{8} e^{s_2+s_3} G_1(s_1) k_{13}(-s_2-s_3,s_2)+\frac{e^{s_1+s_2+s_3} k_{13}(-s_1-s_2-s_3,s_1+s_2)}{8 s_1}+\frac{e^{s_1+s_2+s_3} k_{13}(-s_1-s_2-s_3,s_1+s_2)}{8 s_2}-\frac{1}{16} k_{17}(s_1,s_2,s_3)-\frac{1}{16} e^{s_1} k_{17}(s_2,s_3,-s_1-s_2-s_3)-\frac{1}{16} e^{s_1+s_2+s_3} k_{17}(-s_1-s_2-s_3,s_1,s_2)-\frac{1}{16} e^{s_1+s_2} k_{17}(s_3,-s_1-s_2-s_3,s_1)+\frac{1}{8} k_{18}(s_1,s_2,s_3)+\frac{1}{8} e^{s_1} k_{18}(s_2,s_3,-s_1-s_2-s_3)+\frac{1}{8} e^{s_1+s_2+s_3} k_{18}(-s_1-s_2-s_3,s_1,s_2)+\frac{1}{8} e^{s_1+s_2} k_{18}(s_3,-s_1-s_2-s_3,s_1)-\frac{1}{16} k_{19}(s_1,s_2,s_3)-\frac{1}{16} e^{s_1} k_{19}(s_2,s_3,-s_1-s_2-s_3)-\frac{1}{16} e^{s_1+s_2+s_3} k_{19}(-s_1-s_2-s_3,s_1,s_2)-\frac{1}{16} e^{s_1+s_2} k_{19}(s_3,-s_1-s_2-s_3,s_1)-\frac{k_8(s_1+s_2,s_3)}{4 s_1}-\frac{e^{s_1+s_2} k_8(s_3,-s_1-s_2-s_3)}{4 s_1}-\frac{e^{s_1+s_2+s_3} k_{11}(-s_1-s_2-s_3,s_1+s_2)}{4 s_1}-\frac{e^{s_1+s_2} k_{11}(s_3,-s_1-s_2-s_3)}{4 s_1}-\frac{k_{12}(s_1+s_2,s_3)}{4 s_1}-\frac{e^{s_1+s_2+s_3} k_{12}(-s_1-s_2-s_3,s_1+s_2)}{4 s_1}-\frac{k_9(s_2,s_3)}{8 s_1}-\frac{e^{s_2} k_9(s_3,-s_2-s_3)}{8 s_1}-\frac{e^{s_2+s_3} k_{10}(-s_2-s_3,s_2)}{8 s_1}-\frac{e^{s_2} k_{10}(s_3,-s_2-s_3)}{8 s_1}-\frac{k_{13}(s_2,s_3)}{8 s_1}-\frac{e^{s_2+s_3} k_{13}(-s_2-s_3,s_2)}{8 s_1}-\frac{k_8(s_1+s_2,s_3)}{4 s_2}-\frac{e^{s_1+s_2} k_{11}(s_3,-s_1-s_2-s_3)}{4 s_2}-\frac{e^{s_1+s_2+s_3} k_{12}(-s_1-s_2-s_3,s_1+s_2)}{4 s_2}-\frac{k_9(s_1,s_2+s_3)}{8 s_2}-\frac{e^{s_1} k_{10}(s_2+s_3,-s_1-s_2-s_3)}{8 s_2}-\frac{e^{s_1+s_2+s_3} k_{13}(-s_1-s_2-s_3,s_1)}{8 s_2}-\frac{k_6(s_3)}{4 s_1 s_2}-\frac{e^{s_3} k_7(-s_3)}{4 s_1 s_2}-\frac{k_6(s_1+s_2+s_3)}{4 s_1 (s_1+s_2)}-\frac{e^{s_1+s_2+s_3} k_7(-s_1-s_2-s_3)}{4 s_1 (s_1+s_2)}-\frac{k_6(s_1+s_2+s_3)}{4 s_2 (s_1+s_2)}-\frac{e^{s_1+s_2+s_3} k_7(-s_1-s_2-s_3)}{4 s_2 (s_1+s_2)}-\frac{k_{11}(s_1,s_2)}{4 (-s_1-s_2-s_3)}-\frac{e^{s_1+s_2+s_3} k_{10}(s_1,s_2)}{8 (-s_1-s_2-s_3)}-\frac{e^{s_2+s_3} G_1(s_1) k_4(-s_2-s_3)}{s_3}-\frac{G_1(s_1) k_4(s_2+s_3)}{s_3}-\frac{e^{s_2} G_1(s_1) k_3(-s_2)}{4 s_3}-\frac{G_1(s_1) k_3(s_2)}{4 s_3}-\frac{k_8(s_1,s_2)}{4 s_3}-\frac{e^{s_1+s_2} k_8(-s_1-s_2,s_1)}{4 s_3}-\frac{k_{11}(s_1,s_2)}{4 s_3}-\frac{e^{s_1} k_{11}(s_2,-s_1-s_2)}{4 s_3}-\frac{e^{s_1+s_2} k_{12}(-s_1-s_2,s_1)}{4 s_3}-\frac{e^{s_1} k_{12}(s_2,-s_1-s_2)}{4 s_3}-\frac{k_9(s_1,s_2+s_3)}{8 s_3}-\frac{e^{s_1+s_2+s_3} k_9(-s_1-s_2-s_3,s_1)}{8 s_3}-\frac{k_{10}(s_1,s_2+s_3)}{8 s_3}-\frac{e^{s_1} k_{10}(s_2+s_3,-s_1-s_2-s_3)}{8 s_3}-\frac{e^{s_1+s_2+s_3} k_{13}(-s_1-s_2-s_3,s_1)}{8 s_3}-\frac{e^{s_1} k_{13}(s_2+s_3,-s_1-s_2-s_3)}{8 s_3}-\frac{e^{s_1+s_2} k_4(-s_1-s_2)}{s_1 s_3}-\frac{k_4(s_1+s_2)}{s_1 s_3}-\frac{e^{s_2+s_3} k_4(-s_2-s_3)}{s_1 s_3}-\frac{k_4(s_2+s_3)}{s_1 s_3}-\frac{e^{s_2} k_3(-s_2)}{4 s_1 s_3}-\frac{k_3(s_2)}{4 s_1 s_3}-\frac{e^{s_1+s_2+s_3} k_3(-s_1-s_2-s_3)}{4 s_1 s_3}-\frac{k_3(s_1+s_2+s_3)}{4 s_1 s_3}-\frac{k_6(s_1)}{2 s_2 s_3}-\frac{e^{s_1} k_7(-s_1)}{2 s_2 s_3}-\frac{k_6(s_1+s_2+s_3)}{2 s_2 (s_2+s_3)}-\frac{e^{s_1+s_2+s_3} k_7(-s_1-s_2-s_3)}{2 s_2 (s_2+s_3)}-\frac{k_6(s_1+s_2+s_3)}{2 s_3 (s_2+s_3)}-\frac{e^{s_1+s_2+s_3} k_7(-s_1-s_2-s_3)}{2 s_3 (s_2+s_3)}-\frac{e^{s_1+s_2} k_8(s_3,-s_2-s_3)}{4 (s_1+s_2+s_3)}-\frac{e^{s_1+s_2+s_3} k_{11}(-s_2-s_3,s_2)}{4 (s_1+s_2+s_3)}-\frac{e^{s_1} k_{12}(s_2,s_3)}{4 (s_1+s_2+s_3)}-\frac{e^{s_1+s_2} k_9(-s_1-s_2,s_1)}{8 (s_1+s_2+s_3)}-\frac{e^{s_1+s_2+s_3} k_{10}(s_1,s_2)}{8 (s_1+s_2+s_3)}-\frac{e^{s_1} k_{13}(s_2,-s_1-s_2)}{8 (s_1+s_2+s_3)} = 0. 
\end{math}
\end{center}
}

\smallskip

There are more functional relations of this type, which are written in Appendix \ref{lengthykfnrelationsappsec} 
because of their algebraically  lengthy expressions.

\section{Calculation of the term $a_4$} 
\label{Calculatea_4Sec}

In this section we provide the details of the calculation of the term $a_4$, 
whose final formula is given by \eqref{a_4expression}. In order to derive the 
small-time asymptotic expansion \eqref{heatexp}, one can start by using 
the Cauchy integral formula to write: 
\[
e^{-t \triangle_\vphi} 
= 
\int_\gamma e^{-t \lambda} (\triangle_\vphi- \lambda)^{-1} \, d\lambda,  
\]
where the contour $\gamma$ goes clockwise around the non-negative real axis, where 
the eigenvalues of $\triangle_\vphi$ are located.  Then one can use the pseudodifferential 
calculus developed in \cite{ConC*algDiffGeo}  for $C^*$-dynamical systems to approximate the parametrix of 
$\triangle_\vphi- \lambda.$ We elaborate on this 
procedure in the following subsection. This approach is indeed an implementation of 
the heat kernel method explained in \cite{GilBook} in a noncommutative setting.  

\subsection{Heat expansion}

The pseudodifferential calculus developed in \cite{ConC*algDiffGeo}  
reduces to the following case for the noncommutative two torus $\NT$. The symbols 
are smooth maps  $\rho : \R^2 \to \CNT$ such that the norm of $\rho(\xi)$ and 
its derivatives satisfy certain growth conditions depending on the order of the symbol, 
see \cite{ConC*algDiffGeo} for details. 
For example, the symbol of a differential operator 
$\sum a_{i,j} \done^i \dtwo^j,$ where $a_{i,j} \in \CNT,$ is the 
polynomial $\sum a_{i,j} \xi_1^i \xi_2^j$ whose coefficients apparently belong to 
$\CNT$.  The correspondence between a general symbol $\rho$ 
and its associated pseudodifferential operator $P_\rho$ acting on $\CNT$ is established by 
the formula 
\begin{equation} \label{PDOformula}
P_\rho(x) = (2 \pi)^{-2} \int \int e^{-i s \cdot \xi}\, \rho(\xi) \, \alpha_s(x) \, ds \, d\xi, \qquad x \in \CNT, 
\end{equation}
where the dynamics $\alpha_s$ is given by \eqref{NC2TDynSys}. 
Given two symbols $\rho_1$ and $\rho_2$, the symbol of the composition of their  
corresponding pseudodifferential operators is asymptotically given by 
\[ 
\rho_1 \circ \rho_2 \sim 
\sum_{\alpha_1, \alpha_2 \in \Z_{\geq 0}} \frac{1}{\alpha_1 ! \alpha_2 !} \left (
\frac{\partial^{\alpha_1 + \alpha_2} }{\partial \xi_1^{\alpha_1} \partial \xi_2^{\alpha_2}} \rho_1 (\xi) \right )
\, \done^{\alpha_1} \dtwo^{\alpha_2}\left (\rho_2(\xi) \right ). 
\]

\smallskip

Using this calculus, the parametrix $R_\lambda$ of $\triangle_\vphi - \lambda$ can be 
approximated by a recursive procedure. That is, one can start by assuming that the symbol $\sigma(R_\lambda)$ 
of $R_\lambda$ has an expansion of the form $\sum_{j=0}^\infty r_j(\xi, \lambda)$, where 
each $r_j$ is homogeneous of order $-2 - j$ in the sense that for any $t>0$: 
\[
r_j(t \xi, t^2 \lambda) = t^{-2-j} r_j(\xi, \lambda). 
\]
The reason for starting from order $-2$ is that $\triangle_\vphi - \lambda$ is of order 2 
whose symbol is of the form $(p_2(\xi) - \lambda)+ p_1(\xi)+ p_0(\xi)$, where as calculated 
in \cite{ConTreGB}, we have:
\begin{eqnarray} \label{LaplacianSymbol}
p_2(\xi) &=& (\xi_1^2 + \xi_2^2)e^h, \\
p_1(\xi) &=& 2 \xi_1 e^{h/2} \done(e^{h/2})+ 2 \xi_2 e^{h/2} \dtwo(e^{h/2}), \nonumber \\
p_0(\xi) &=& e^{h/2} \done^2(e^{h/2}) + e^{h/2} \dtwo^2(e^{h/2}). \nonumber
\end{eqnarray}

\smallskip

As it is explained in detail in \cite{FatKhaGB}, using the calculus of symbols 
mentioned above, one can solve the following equation recursively: 
\[
\left ( \sum_{j=0}^\infty r_j (\xi, \lambda) \right )\circ \left ( (p_2(\xi) - \lambda)+ p_1(\xi)+ p_0(\xi) \right ) \sim 1. 
\]
In fact one finds that 
\[
r_0(\xi, \lambda) = (p_2(\xi) - \lambda )^{-1}, 
\]
and for each $n>1:$
\[
r_n (\xi, \lambda) = - \sum \frac{1}{ \alpha_1 ! \alpha_2 !}
\frac{\partial^{\alpha_1 + \alpha_2} r_j(\xi, \lambda)}{\partial \xi_1^{\alpha_1} \partial \xi_2^{\alpha_2}} 
\, \done^{\alpha_1} \dtwo^{\alpha_2} \left (p_k(\xi) \right ) \, r_0(\xi, \lambda), 
\]
where the summation is over all $\alpha_1, \alpha_2 \in \Z_{\geq 0}$, $j \in \{1, \dots, n-1\}$ and $k \in \{0, 1, 2\}$ such that $2+j+\alpha_1+ \alpha_2-k=n$. 

\smallskip

Finally, after finding each $r_{n}$ recursively, one can show that each $a_{2n} \in \CNT$ appearing 
in the small-time heat kernel expansion \eqref{heatexp} can be written as 
\begin{equation} \label{generala_2nformula}
a_{2n} = \frac{1}{2 \pi i}\int_{\R^2} \int_{\gamma} e^{-\lambda} r_{2n}(\xi, \lambda) \, d\lambda \, d\xi. 
\end{equation}
Using a homogeneity argument and classical complex analysis  
one can explicitly calculate the integrals involved in the above expression, except for 
a final part that has a purely noncommutative obstruction. That is, one has to 
calculate the integral of certain $\CNT$-valued functions defined over the positive 
real line (due to passing to the polar coordinates and the necessity to integrate 
in the radial direction). This task can be accomplished by the so called rearrangement 
lemmas, see \cite{ConTreGB, BhuMarRicci, ConMosModular, 
FatKhaSC2T, LesDivided}. We discuss the necessary rearrangement lemma for the calculation 
of the term $a_4$ in the following subsection. 

\subsection{Rearrangement lemma}

As we explained above, 
the appearance of the modular automorphism $\sigma_i(x)=e^\nab(x)=e^{-h}xe^h$, 
$x \in C(\NT)$, in the calculation of the term $a_4$ comes from the 
fact that in the process of applying the method explained in the previous subsection, one  encounters integrals of $\CNT$-valued functions defined on the positive 
real line. This type of integrals can be worked out by means of the so called 
rearrangement lemmas. In Lemma \ref{Rearrfora4} we present the 
rearrangement lemma that is needed for the 
calculation of the term $a_4$. However, since we wish to explain a 
recursive procedure for constructing the functions appearing in this lemma, 
it is useful to first recall the rearrangement lemma that was used in 
\cite{ConMosModular, FatKhaSC2T} 
for the calculation of the term $a_2$. 

\begin{lemma} \label{Rearrfora2}
For any $m = (m_0, m_1, \dots, m_{n}) \in \Z_{>0}^{n+1}$ and $\rho_1, \dots, \rho_n \in \CNT$, one has: 
\begin{eqnarray*}
&&\int_0^\infty (e^h u +1)^{-m_0} \prod_{j=1}^n \rho_j (e^h u +1)^{-m_j} u^{|m|-2} \, du \\
&=&
e^{-(|m|-1)h} F_m(\sigma_i, \dots, \sigma_i)(\rho_1 \cdots \rho_n), 
\end{eqnarray*}
where the function $F_m$ is defined by 
\[
F_m(u_1, \dots, u_n)
=
\int_{0}^\infty \frac{x^{|m|-2}}{(x+1)^{m_0}} \prod_{j=1}^n \left (x \prod_{\nu=1}^j u_\nu +1 \right )^{-m_j} \, dx. 
\]
\end{lemma}

\smallskip

In the following proposition we present functional relations that allow 
us to construct recursively any function $F_m(u_1, \dots, u_n)$ appearing in the 
previous lemma, essentially starting from the generating function of the Bernoulli numbers.

\begin{proposition} \label{RecursiveforRearr1}
For the functions $F_m$ appearing in the statement of 
Lemma \ref{Rearrfora2}, we have: 
\[
F_{(1,1)}(u_1) = \frac{\log u_1}{u_1 - 1}, 
\]
if $m_0, m_1 \geq 2$ and $m_3, \dots, m_n \geq 1$, then 
\begin{eqnarray*}
F_{(m_0, m_1, \dots, m_n)}(u_1, u_2,  \dots, u_n) &=& \frac{1}{u_1-1} F_{(m_0, m_1-1, m_2, \dots, m_n)}(u_1, \dots, u_n) - \\
&& \frac{1}{u_1-1} F_{(m_0-1, m_1, \dots, m_n)}(u_1, \dots, u_n),
\end{eqnarray*}
if $m_1 \geq 2$ and $m_2, \dots, m_n \geq 1$, then 
\begin{eqnarray*}
F_{(1, m_1, m_2,  \dots, m_n)}(u_1, u_2,  \dots, u_n) &=& \frac{1}{u_1-1} F_{(1, m_1-1, m_2, \dots, m_n)}(u_1, \dots, u_n) -\\
&& \frac{u_1^{-m_1-m_2-\cdots - m_n+1}}{u_1-1} F_{(m_1, m_2, \dots, m_n)}(u_2, u_3, \dots, u_n),
\end{eqnarray*}
if $m_0 \geq 2$ and $m_2, \dots, m_n \geq 1$, then 
\begin{eqnarray*}
F_{(m_0, 1, m_2,  \dots, m_n)}(u_1, u_2, \dots, u_n) &=& \frac{1}{u_1-1} F_{(m_0, m_2, \dots, m_n)}(u_1 u_2, u_3, \dots, u_n) -\\
&& \frac{1}{u_1-1} F_{(m_0-1, 1, m_2, \dots, m_n)}(u_1, u_2, \dots, u_n),  
\end{eqnarray*}
and if $m_2, \dots, m_n \geq 1$, then 
\begin{eqnarray*}
F_{(1, 1, m_2,  \dots, m_n)}(u_1, u_2, \dots, u_n) &=& \frac{1}{u_1-1} F_{(1, m_2, \dots, m_n)}(u_1 u_2, u_3, \dots, u_n) -\\
&& \frac{u_1^{-m_2-m_3 - \cdots - m_n}}{u_1-1} F_{(m_0-1, 1, m_2, \dots, m_n)}(u_1, u_2, \dots, u_n).  
\end{eqnarray*}
\begin{proof}
It follows from considering the definition of the functions and writing a partial fraction decomposition for $\frac{1}{(x+1)(xu+1)}$.   
\end{proof}
\end{proposition}

For the calculation of $a_4$, since the pseudodifferential symbol that we have 
to use has a different homogeneity order than the one for $a_2$, we need 
a variant of Lemma \ref{Rearrfora2}, which is given in the following lemma.  

\begin{lemma}  \label{Rearrfora4}
For any $m = (m_0, m_1, \dots, m_{n}) \in \Z_{>0}^{n+1}$ and $\rho_1, \dots, \rho_n \in \CNT$, we have: 
\begin{eqnarray*}
&&\int_0^\infty (e^h u +1)^{-m_0} \prod_{j=1}^n \rho_j (e^h u +1)^{-m_j} u^{|m|-3} \, du \\
&=&
e^{-(|m|-2)h} F^v_m(\sigma_i, \dots, \sigma_i)(\rho_1 \cdots \rho_n), 
\end{eqnarray*}
where the function $F^v_m$ is defined by 
\[
F^v_m(u_1, \dots, u_n)
=
\int_{0}^\infty \frac{x^{|m|-3}}{(x+1)^{m_0}} \prod_{j=1}^n \left (x \prod_{\nu=1}^j u_\nu +1 \right )^{-m_j} \, dx. 
\]
\end{lemma}
\begin{proof}
The proof of Lemma 6.2 of \cite{ConMosModular} works verbatim, by choosing 
the positive numbers $\alpha_j$ such that they add up to 2 instead of 1. 
\end{proof}

In fact, the functions $F^v_m$ appearing in the statement of Lemma 
\ref{Rearrfora4} are closely related to those appearing in Lemma \ref{Rearrfora2}:

\begin{proposition} \label{RecursiveforRearr2}
If $m_0, m_1 \geq 1$ and $m_3, \dots, m_n \geq 0$, then 
\begin{eqnarray*} 
F^v_{(m_0, m_1, \dots, m_n)}(u_1, \dots, u_n) 
&=& 
\frac{1}{u_1-1} F_{(m_0, m_1-1 , \dots, m_n)}(u_1, \dots, u_n) \\
&& + \frac{u_1}{u_1-1} F_{(m_0-1, m_1\, \dots, m_n)}(u_1, \dots, u_n). 
\end{eqnarray*}
\end{proposition}

\smallskip

Hence, all functions $F^v_m$ that are necessary for the calculation of the 
term $a_4$ can be constructed recursively, as explained in this subsection, 
 from  $F_{(1, 1)}(u_1)=$ $\log u_1 /(u_1 -1)$.  

\smallskip

The methods and the tools that we have discussed so far in this section allow us to calculate 
an expression for the term $a_4 \in \CNT$ that involves functions of the 
operator $\nab$ acting on terms that involve derivatives of the conformal factor 
$e^h \in \CNT$. However, it is important to write the final expression 
in terms of derivatives of $h$. This task can be achieved by using 
Lemma \ref{ToLogBaby} to employ the functions $G_1, G_2, G_3, 
G_4$, which are constructed in Lemma \ref{BabyFunctions1} and 
Lemma \ref{BabyFunctions2}, to find explicitly an expression of the 
form $\eqref{a_4expression}$ for the term $a_4$. We note that 
such a procedure was performed for the final formula for the term 
$a_2$ as well in \cite{ConMosModular, FatKhaSC2T}, where 
making use of the functions $G_1$ and $G_2$ was sufficient.

\newpage

\section{Explicit formulas for the local functions describing $a_4$}
\label{ExplicitFormulasSec}

After performing the calculation of the term $a_4 \in \CNT$ as explained in 
Section \ref{Calculatea_4Sec}, which involved heavy calculations, we find 
explicit formulas for the 
functions $K_1, \dots, K_{20}$ that appear in the expression \eqref{a_4expression}. 
In this section we present the explicit formulas for the one and two variable functions 
$K_1, \dots, K_7$. The lengthy final expressions for the three and four variable  
functions $K_8, \dots, K_{20}$ are given in Appendix \ref{explicit3and4fnsappsec}. Although the final 
formulas are quite lengthy, it should be noted 
that they are the results of enormous amount of cancellations and 
simplifications that typically occur in this type of calculations, 
cf. \cite{ConMosModular, FatKhaSC2T, FatRes}. More importantly, in order to 
confirm the accuracy of the final formulas, we have checked that 
they satisfy the functional relations stated in Theorem \ref{FuncRelationsThm}.

\subsection{The functions of one variable}
The function $K_1$ has a simple expression: 
\begin{equation} \label{K_1explicitformula}
K_1(s_1)= -\frac{4 \pi  e^{\frac{3 s_1}{2}}
   \left(\left(4 e^{s_1}+e^{2 s_1}+1\right)
   s_1-3 e^{2
   s_1}+3\right)}{\left(e^{s_1}-1\right){}^4
   s_1}. 
\end{equation}
The graph of this function is given in Figure \ref{K1graph}.

\begin{figure}[H]
\includegraphics[scale=0.6]{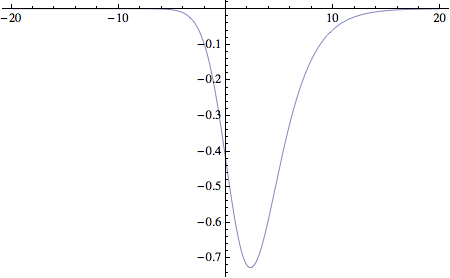}
\caption{Graph of $K_1$.}
\label{K1graph}
\end{figure}

\smallskip

The function $K_2$ turns out to be a scalar multiple of $K_1$: 
\[
K_2(s_1) = -\frac{2 \pi  e^{\frac{3 s_1}{2}}
   \left(\left(4 e^{s_1}+e^{2 s_1}+1\right)
   s_1-3 e^{2
   s_1}+3\right)}{\left(e^{s_1}-1\right){}^4
   s_1} 
= \frac{1}{2} K_1(s_1). 
\]

\begin{figure}[H]
\includegraphics[scale=0.6]{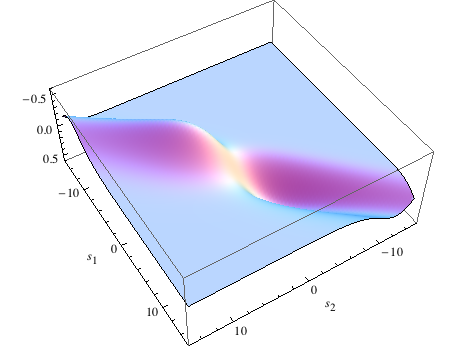}
\caption{Graph of $K_3$.}
\label{K3graph}
\end{figure}

\subsection{The functions of two variables}
The functions $K_3, \dots, K_7$ in the expression 
\eqref{a_4expression} are the two variable functions. 

\smallskip

The function $K_3$ is of the form 
\begin{equation} \label{K_3explicitformula}
K_3(s_1, s_2) = \frac{K_3^{\text{num}}(s_1, s_2)}{
\left(e^{s_1}-1\right){}^2 \left(e^{s_2}-1\right){}^2
   \left(e^{s_1+s_2}-1\right){}^4 s_1 s_2 \left(s_1+s_2\right)}, 
   \end{equation}
where the function in the numerator  is given by 
\begin{eqnarray*}
  K_3^{\text{num}}(s_1, s_2)&=&
16 \, e^{\frac{3 s_1}{2}+\frac{3 s_2}{2}} \pi \Big  [\left(e^{s_1}-1\right) \left(e^{s_2}-1\right) \left(e^{s_1+s_2}-1\right)  \big \{ 
\\
&&
\left(-5 e^{s_1}-e^{s_2}+6 e^{s_1+s_2}-e^{2 s_1+s_2}-5 e^{s_1+2 s_2}+3 e^{2 s_1+2 s_2}+3\right) s_1+
\\
&&
\left(e^{s_1}+5 e^{s_2}-6 e^{s_1+s_2}+5 e^{2 s_1+s_2}+e^{s_1+2 s_2}-3 e^{2 s_1+2 s_2}-3\right)s_2 \big \}
\\
&&
  -2 \left(e^{s_1}-e^{s_2}\right) \left(e^{s_1+s_2}-1\right)
\\
&&
\left(-e^{s_1}-e^{s_2}-e^{2 s_1+s_2}-e^{s_1+2 s_2}+2 e^{2 s_1+2 s_2}+2\right) s_1 s_2
\\
&&
+2 e^{s_1} \left(e^{s_2}-1\right){}^3 \left(e^{s_1}-e^{s_1+s_2}+2 e^{2 s_1+s_2}-2\right) s_1^2 
\\
&&
-2 e^{s_2} \left(e^{s_1}-1\right){}^3\left(e^{s_2}-e^{s_1+s_2}+2 e^{s_1+2 s_2}-2\right) s_2^2 \Big ].
\end{eqnarray*}

\smallskip

Clearly, the function $K_3^{\text{num}}$ is a polynomial in 
$s_1, s_2, e^{s_1/2}, e^{s_2/2}$. By plotting the points $(i, j)$ 
and $(m, n)$ such that $s_1^i s_2^j e^{ms_1/2} e^{n s_2/2}$ 
appears in $K_3^{\text{num}}$, we find that these exponents appear 
in an orderly fashion. This can be seen in Figure \ref{K3sEtospowersgraph},  
where the points $(i, j)$ are the blue dots on the lower left corner, and the 
points $(m, n)$ are the yellow dots on the upper right side.  The graph 
of the function $K_3$ is provided in Figure \ref{K3graph}.

\begin{figure}[H] 
\includegraphics[scale=0.3]{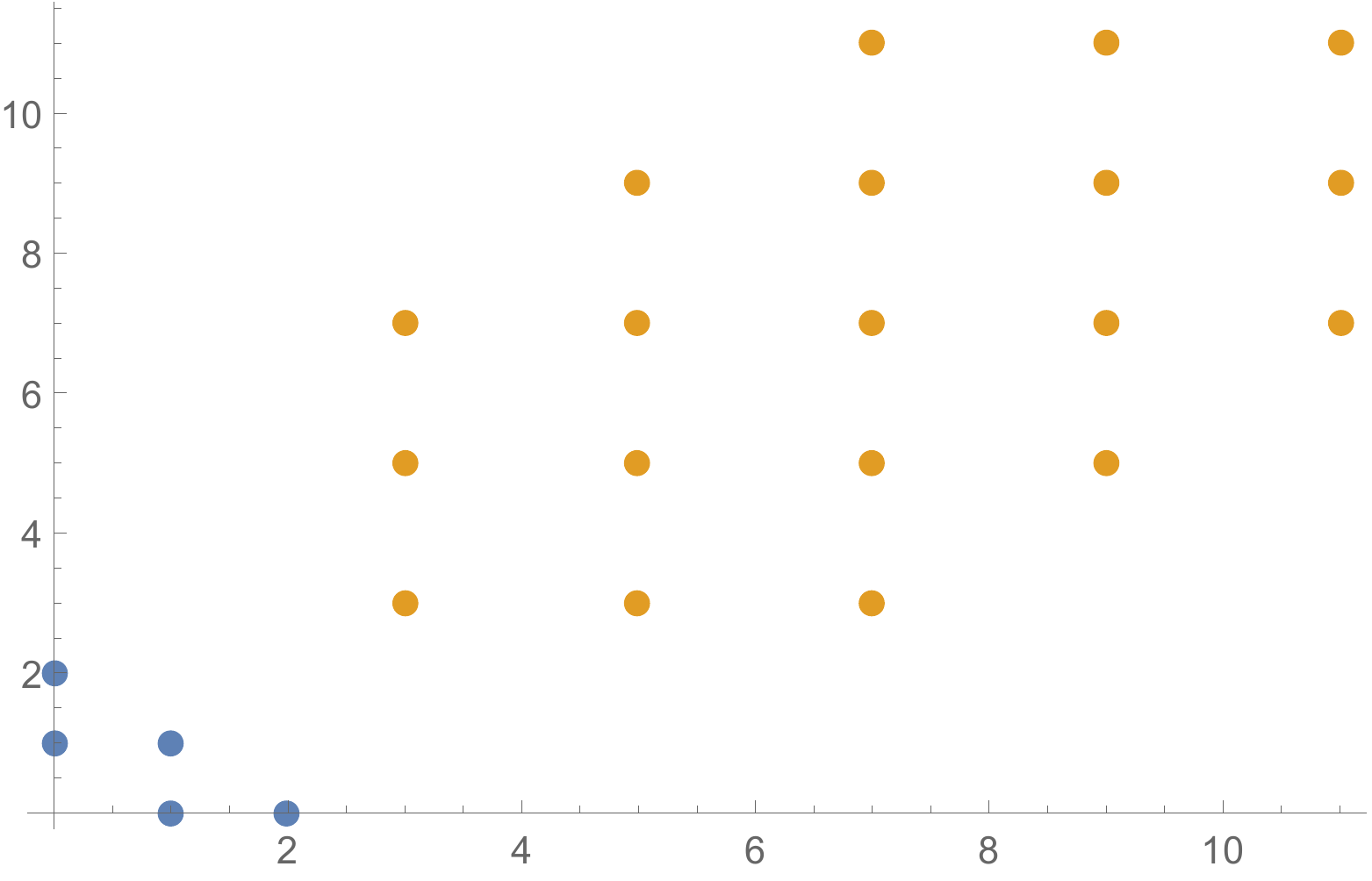}
\caption{The points $(i, j)$ in blue and $(m, n)$ in yellow such that 
$s_1^i \, s_2^j \,e^{m s_1/2}\, e^{n s_2/2}$  appears in the expression for $K_3^{\text{num}}$.}
\label{K3sEtospowersgraph}
\end{figure}

\begin{figure}[H]
\includegraphics[scale=0.5]{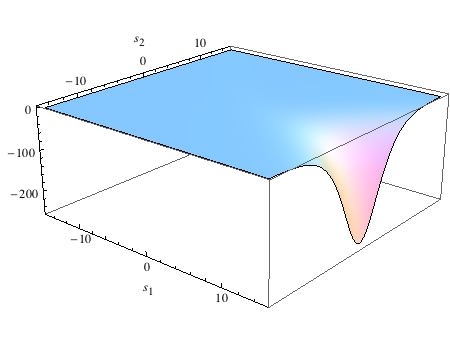}
\caption{Graph of $K_4$.}
\label{K4graph}
\end{figure}

\smallskip

The function $K_4$ is given by 
\[ K_4(s_1, s_2) = \frac{K_4^{\text{num}}(s_1, s_2)}{
\left(e^{s_1}-1\right){}^2 \left(e^{s_2}-1\right){}^2
   \left(e^{s_1+s_2}-1\right){}^4 s_1 s_2 \left(s_1+s_2\right) }, 
\]
where 
\begin{eqnarray*}
 K_4^{\text{num}}(s_1, s_2) &=& 4 e^{\frac{3 s_1}{2}+\frac{s_2}{2}} \pi \Big [\left(e^{s_1}-1\right) \left(e^{s_2}-1\right) \left(e^{s_1+s_2}-1\right) \Big \{ 
\end{eqnarray*}
$$
\Big (5 e^{s_2}-3 e^{2 s_2}-e^{s_1+s_2}-5 e^{2 \left(s_1+s_2\right)}+2 e^{3 \left(s_1+s_2\right)}+6 e^{s_1+2 s_2}-3 e^{s_1+3 s_2}+e^{2 s_1+3 s_2}-2 \Big ) s_1+$$   
$$
\Big (-e^{s_2}+3 e^{2 s_2}+5 e^{s_1+s_2}-6 e^{s_1+2 s_2}+e^{2 s_1+2 s_2}+3 e^{s_1+3 s_2}-5 e^{2 s_1+3 s_2}+2 e^{3 s_1+3 s_2}-2 \Big )s_2 \Big \}
$$
$$
  - \left(e^{s_1+s_2}-1\right) \left(-2 e^{s_2}+e^{s_1+s_2}+1\right)
$$
$$
\Big (e^{s_1}-e^{s_2}+3 e^{s_1+s_2}+3 e^{2 \left(s_1+s_2\right)}+e^{3 \left(s_1+s_2\right)}-6 e^{2 s_1+s_2}-2 e^{s_1+2 s_2}+e^{3 s_1+2 s_2}-e^{2 s_1+3 s_2}+1 \Big) s_1 s_2
$$
$$
+\left(e^{s_2}-1\right){}^3 \left(-e^{s_1}-4 e^{s_1+s_2}-e^{2 \left(s_1+s_2\right)}+6 e^{2 s_1+s_2}+e^{3 s_1+2 s_2}-1\right) s_1^2
$$
$$
-e^{2 s_2}  \left(e^{s_1}-1\right){}^3 \left(e^{s_2}-6 e^{s_1+s_2}+e^{2 \left(s_1+s_2\right)}+4 e^{s_1+2 s_2}+e^{2 s_1+3 s_2}-1\right) s_2^2 \Big ]. 
$$

\smallskip

Clearly, the function $K_4^{\text{num}}$ is also a polynomial in 
$s_1, s_2, e^{s_1/2}, e^{s_2/2}$. In Figure \ref{K4sEtospowersgraph},  
the points $(i, j)$ and $(m, n)$ such that $s_1^i s_2^j e^{ms_1/2} e^{n s_2/2}$ 
appears in $K_4^{\text{num}}$ are plotted. The graph of $K_4$ 
is given in Figure \ref{K4graph}.

\begin{figure}[H]
\includegraphics[scale=0.6]{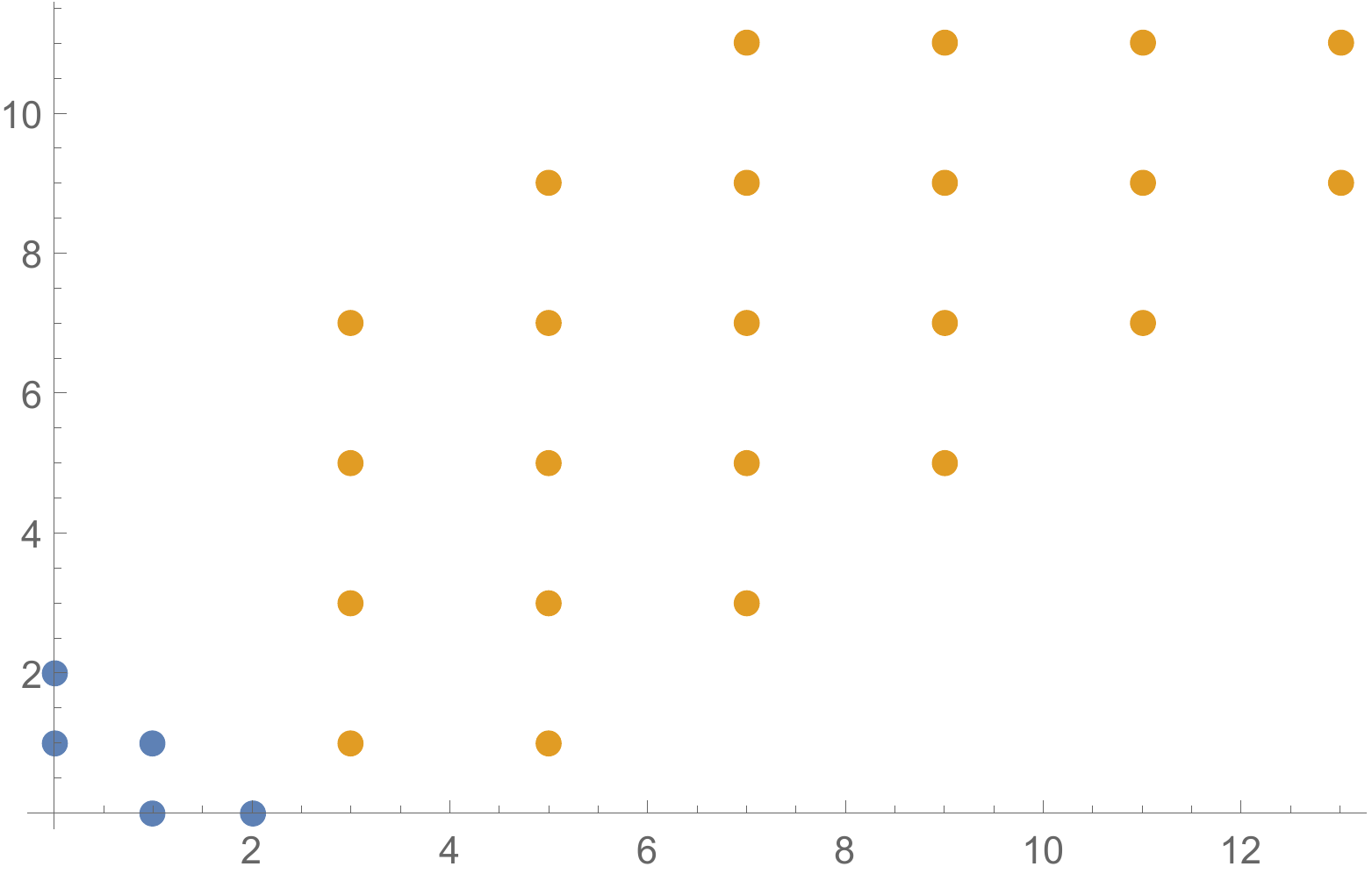}
\caption{The points $(i, j)$ in blue and $(m, n)$ in yellow such that 
$s_1^i \, s_2^j \, e^{m s_1/2} \, e^{n s_2/2}$  appears in the expressions  
for $K_4^{\text{num}}$ and $K_5^{\text{num}}$.}  
\label{K4sEtospowersgraph}
\end{figure}

\smallskip

The function $K_5$ is given by the quotient 
\[
K_5(s_1, s_2) = \frac{K_5^{\text{num}}(s_1, s_2)}{
\left(e^{s_1}-1\right){}^2 \left(e^{s_2}-1\right){}^2
   \left(e^{s_1+s_2}-1\right){}^4 s_1 s_2 \left(s_1+s_2\right)}, 
\]
where the function $K_5^{\text{num}}$ is given by 

\begin{eqnarray*}
  K_5^{\text{num}}(s_1, s_2)&=& 4 e^{\frac{3 s_1}{2}+\frac{s_2}{2}} \pi \Big [\left(e^{s_1}-1\right) \left(e^{s_2}-1\right) \left(e^{s_1+s_2}-1\right) \Big \{ 
 \end{eqnarray*}
$$\Big (11 e^{s_2}-5 e^{2 s_2}-11 e^{s_1+s_2}-7 e^{2 \left(s_1+s_2\right)}+2 e^{3 \left(s_1+s_2\right)}+18 e^{s_1+2 s_2}-13 e^{s_1+3 s_2}+7 e^{2 s_1+3 s_2}-2 \Big ) s_1+$$   
     $$
\Big (-7 e^{s_2}+13 e^{2 s_2}+7 e^{s_1+s_2}-18 e^{s_1+2 s_2}+11 e^{2 s_1+2 s_2}+5 e^{s_1+3 s_2}-11 e^{2 s_1+3 s_2}+2 e^{3 s_1+3 s_2}-2 \Big )s_2 \Big \}
   $$
   $$
  \left(e^{s_1+s_2}-1\right) \Big (-e^{s_1}+3 e^{s_2}+6 e^{2 s_2}-4 e^{3 s_2}-10 e^{s_1+s_2}+9 e^{2 s_1+s_2}+9 e^{s_1+2 s_2}-18 e^{2 s_1+2 s_2}+9 e^{3 s_1+2 s_2}$$
  $$-4 e^{s_1+3 s_2}+9 e^{2 s_1+3 s_2}-10 e^{3 s_1+3 s_2}-e^{4 s_1+3 s_2}-4 e^{s_1+4 s_2}+6 e^{2 s_1+4 s_2}+3 e^{3 s_1+4 s_2}-e^{4 s_1+4 s_2}-1 \Big )
  s_1 s_2
$$
$$
+\left(e^{s_2}-1\right){}^3 \left(-e^{s_1}-12 e^{s_1+s_2}+10 e^{2 s_1+s_2}-5 e^{2 s_1+2 s_2}+9 e^{3 s_1+2 s_2}-1\right) s_1^2
$$
$$
e^{2 s_2}  \left(e^{s_1}-1\right){}^3 \left(-5 e^{s_2}+10 e^{s_1+s_2}-12 e^{s_1+2 s_2}-e^{2 s_1+2 s_2}-e^{2 s_1+3 s_2}+9\right) s_2^2 \Big ]. 
$$

The points $(i, j)$ and $(m, n)$ such that $s_1^i s_2^j e^{ms_1/2} e^{n s_2/2}$ 
appears in $K_5^{\text{num}}$ are the same as those for $K_4^{\text{num}}$, 
given in Figure \ref{K4sEtospowersgraph}. The graph of $K_5$ is given in Figure 
\ref{K5graph}.

\begin{figure}[H]
\includegraphics[scale=0.6]{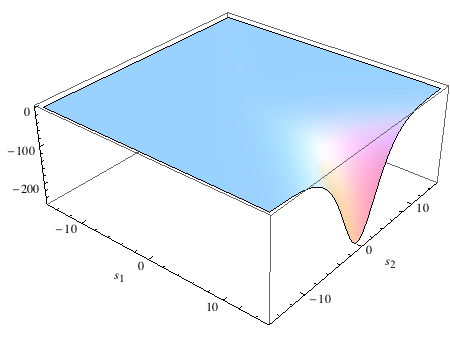}
\caption{Graph of $K_5$.}
\label{K5graph}
\end{figure}
\vspace{2cm}

\smallskip

The function $K_6$ is of the form 
\[
K_6(s_1, s_2) = \frac{K_6^{\text{num}}(s_1, s_2)}{\left(e^{s_1}-1\right) 
   \left(e^{s_2}-1\right){}^3 \left(e^{s_1+s_2}-1\right){}^4 s_1 s_2
   \left(s_1+s_2\right)}, 
\]
where 
\begin{eqnarray*}
  K_6^{\text{num}}(s_1, s_2)&=&4 e^{\frac{3 s_1}{2}+\frac{3s_2}{2}} \pi \Big [\left(e^{s_1}-1\right) \left(e^{s_2}-1\right) \left(e^{s_1+s_2}-1\right)\big \{ 
\end{eqnarray*}
    $$\left(-3 e^{s_2}+e^{2 s_2}-3 e^{s_1+s_2}-14 e^{s_1+2 s_2}+e^{2 s_1+2 s_2}+5 e^{s_1+3 s_2}+5 e^{2 s_1+3 s_2}+8\right) s_1+$$   
     $$
\left(21 e^{s_2}-11 e^{2 s_2}-15 e^{s_1+s_2}+10 e^{s_1+2 s_2}+e^{2 s_1+2 s_2}-7 e^{s_1+3 s_2}+5 e^{2 s_1+3 s_2}-4\right)s_2 \big \}
   $$
   $$
  -2 \Big (-2 e^{s_1}+7 e^{s_2}-6 e^{2 s_2}+2 e^{3 s_2}-6 e^{s_1+s_2}+6 e^{2 s_1+s_2}+2 e^{s_1+2 s_2}-6 e^{2 s_1+2 s_2}+2 e^{3 s_1+2 s_2}-10 e^{s_1+3 s_2} $$
  $$+24 e^{2 s_1+3 s_2}-14 e^{3 s_1+3 s_2}+4 e^{s_1+4 s_2}-6 e^{2 s_1+4 s_2}+2 e^{3 s_1+4 s_2}+2 e^{4 s_1+4 s_2}-2 e^{3 s_1+5 s_2}+e^{4 s_1+5 s_2}\Big )
  s_1 s_2
$$
$$
+2 e^{s_1} \left(e^{s_2}-1\right){}^4 \left(e^{s_1+s_2}+2\right) s_1^2
$$
$$
e^{2 s_2}  \left(e^{s_1}-1\right){}^2 \big (6 e^{s_2}-2 e^{2 s_2}-2 e^{s_1+s_2}+14 e^{s_1+2 s_2}-6 e^{s_1+3 s_2}-2 e^{2 s_1+3 s_2}-e^{2 s_1+4 s_2}-7\big ) s_2^2 \Big ].
$$

\smallskip

The points $(i, j)$ and $(m, n)$ such that $s_1^i s_2^j e^{ms_1/2} e^{n s_2/2}$ 
appears in $K_6^{\text{num}}$ are presented in Figure  
\ref{K6sEtospowersgraph}, and the graph of $K_6$ is given in 
Figure \ref{K6graph}. 

\begin{figure}[H]
\includegraphics[scale=0.4]{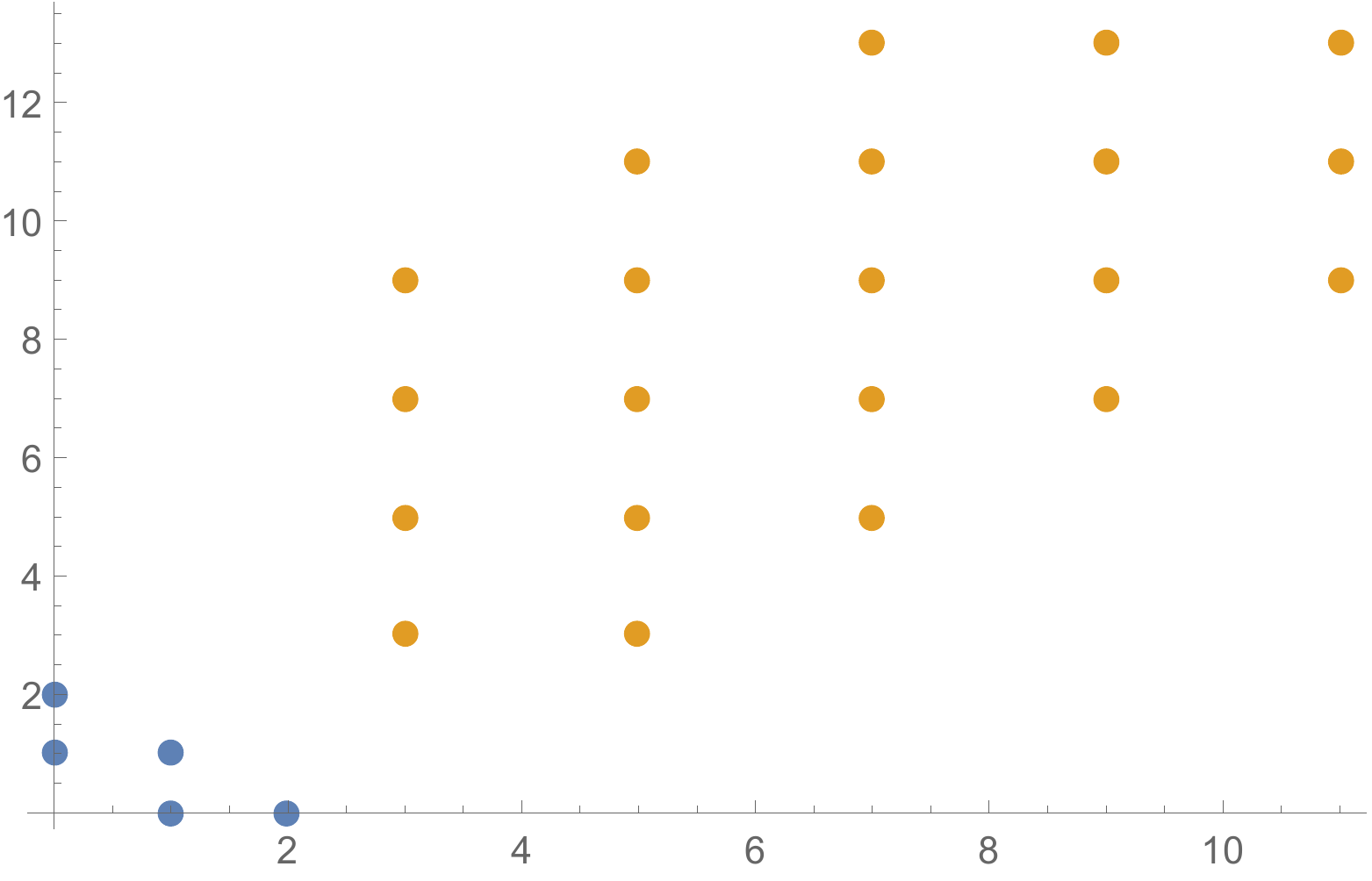}
\caption{The points $(i, j)$ in blue and $(m, n)$ in yellow such that 
$s_1^i \, s_2^j \, e^{m s_1/2} \, e^{n s_2/2}$  appears in the expression   
for $K_6^{\text{num}}$. }  
\label{K6sEtospowersgraph}
\end{figure}

\begin{figure}[H]
\includegraphics[scale=0.5]{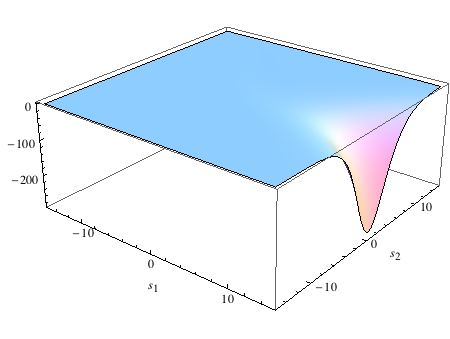}
\caption{Graph of $K_6$.}
\label{K6graph}
\end{figure}

\smallskip

The function $K_7$ is also a fraction of the form 
\[
K_7(s_1, s_2)=\frac{K_7^{\text{num}}(s_1, s_2)}{ \left(e^{s_1}-1\right){}^3 \left(e^{s_2}-1\right)
    \left(e^{s_1+s_2}-1\right){}^4 s_1
   s_2 \left(s_1+s_2\right)},  
\]
where 
\begin{eqnarray*}
  K_7^{\text{num}}(s_1, s_2)&=& 4 \pi  e^{\frac{3 s_1}{2}+\frac{s_2}{2}} \Big [\left(e^{s_1}-1\right) \left(e^{s_2}-1\right) \left(e^{s_1+s_2}-1\right) \big \{ 
\end{eqnarray*}
    $$\left(-e^{s_1}+7 e^{s_2}-10 e^{s_1+s_2}+15 e^{2 s_1+s_2}+11 e^{s_1+2 s_2}-21 e^{2 s_1+2 s_2}+4 e^{3 s_1+2 s_2}-5\right) s_1+$$   
     $$
\left(-e^{s_1}-5 e^{s_2}+14 e^{s_1+s_2}+3 e^{2 s_1+s_2}-e^{s_1+2 s_2}+3 e^{2 s_1+2 s_2}-8 e^{3 s_1+2 s_2}-5\right)s_2 \big \}
   $$
   $$
  -2 \Big (-2 e^{s_1}+2 e^{s_2}-2 e^{s_1+s_2}-24 e^{2 \left(s_1+s_2\right)}-2 e^{3 \left(s_1+s_2\right)}-7 e^{4 \left(s_1+s_2\right)}+14 e^{2 s_1+s_2}-2 e^{3 s_1+s_2}+6 e^{s_1+2 s_2}$$
  $$+6 e^{3 s_1+2 s_2}-6 e^{4 s_1+2 s_2}-4 e^{s_1+3 s_2}+10 e^{2 s_1+3 s_2}+6 e^{4 s_1+3 s_2}+2 e^{5 s_1+3 s_2}-2 e^{2 s_1+4 s_2}+6 e^{3 s_1+4 s_2}-1 \Big )
  s_1 s_2
$$
$$
+2 \left(e^{s_2}-1\right){}^2 \big (2 e^{s_1}+6 e^{s_1+s_2}-14 e^{2 s_1+s_2}+2 e^{3 s_1+s_2}+2 e^{2 s_1+2 s_2}-6 e^{3 s_1+2 s_2}+7 e^{4 s_1+2 s_2}+1\big ) s_1^2
$$
$$
-2 e^{2 s_2} \left(e^{s_1}-1\right){}^4\left(2 e^{s_1+s_2}+1\right) s_2^2 \Big]. 
$$

\smallskip

The points $(i, j)$ and $(m, n)$ such that $s_1^i s_2^j e^{ms_1/2} e^{n s_2/2}$ 
appears in $K_7^{\text{num}}$ are plotted in
\ref{K7sEtospowersgraph} and the graph of $K_7$ is given in 
Figure \ref{K7graph}. 
 
 \begin{figure}[H]
\includegraphics[scale=0.4]{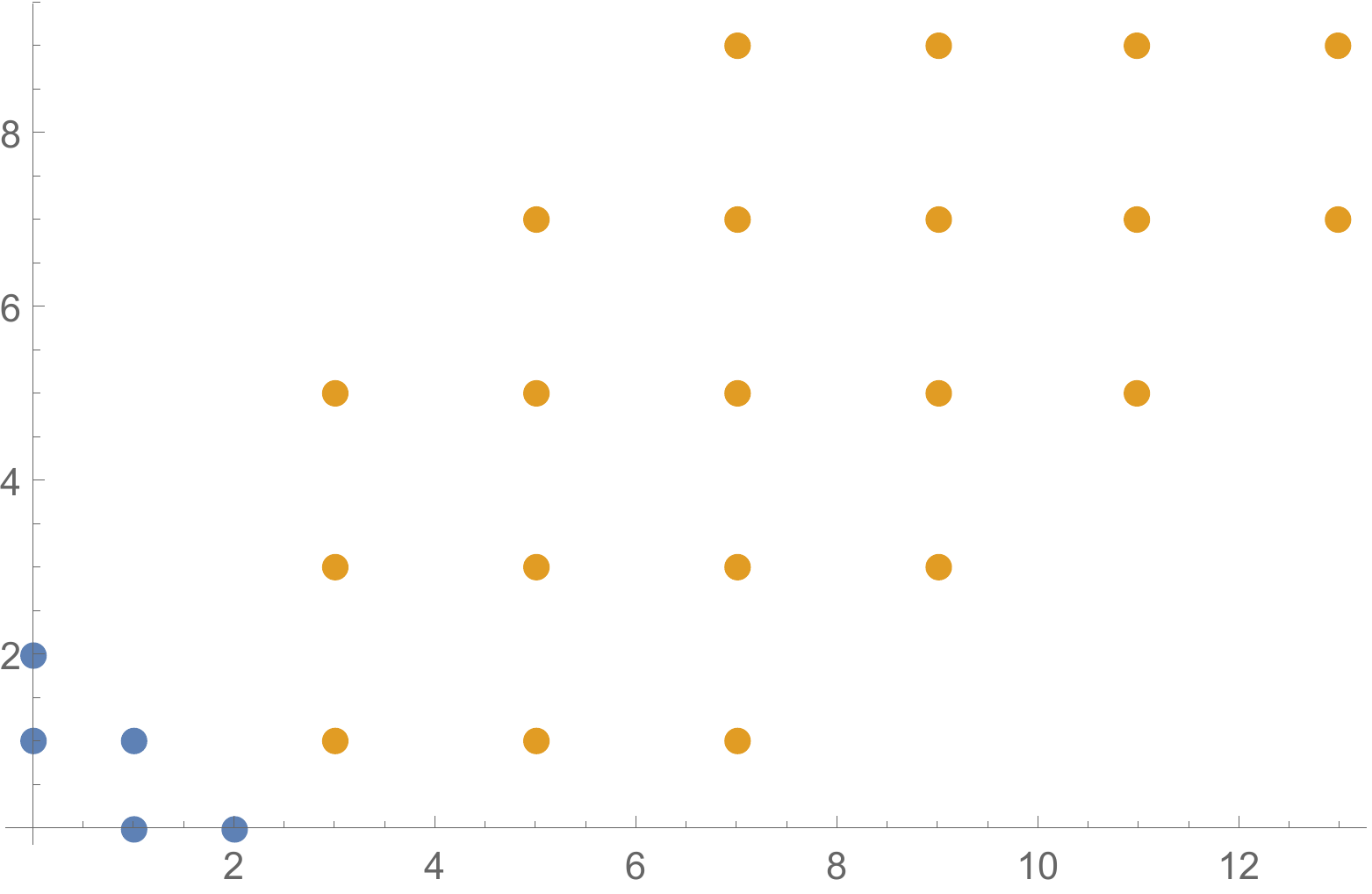}
\caption{The points $(i, j)$ in blue and $(m, n)$ in yellow such that 
$s_1^i\, s_2^j\, e^{m s_1/2}\, e^{n s_2/2}$  appears in the expression   
for $K_7^{\text{num}}$. }  
\label{K7sEtospowersgraph}
\end{figure}

\begin{figure}[H]
\includegraphics[scale=0.6]{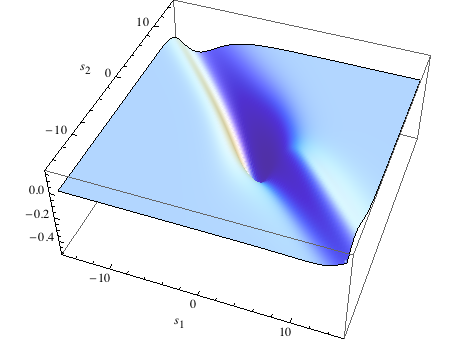}
\caption{Graph of $K_7$.}
\label{K7graph}
\end{figure}

\smallskip

\section{General structure of the noncommutative local  invariants}
\label{StructureSec}

In Section \ref{ExplicitFormulasSec} and in Appendix \ref{explicit3and4fnsappsec}, 
we have presented the explicit formulas for the 
functions $K_1, \dots, K_{20}$ that appear in the local expression 
\eqref{a_4expression} for the term $a_4 \in \CNT$. Also, we observed 
that each function $K_j,$ $j=1, \dots, 20,$ is a rational function 
in $s_1, \dots, s_n,$ $e^{s_1/2}, \dots , e^{s_n/2}$, where $n \in \{1, 2, 3, 4\}$ 
is the number of variables that each function depends on. This fact can indeed be justified 
by using Lemma \ref{BabyFunctions2} and our results in Section \ref{Calculatea_4Sec} where we construct the main 
ingredients of our calculations with finite differences from the generating function of 
the Bernoulli numbers and its inverse $G_1(s_1)=(e^{s_1}-1)/s_1$. 
Moreover, these considerations  explain why the denominators of the functions $K_j$ 
have nice product formulas. At the first glance, it might seem odd that 
the fractions $s_i/2$ appear in the exponents of the calculated functions, since 
finite differences of $G_1(s_1)$ and its inverse yield rational functions in 
$s_i$ and $e^{s_i}$. The reason for the appearance of the $e^{s_i/2}$ in the 
final formulas is that the pseudodifferential symbol of the Laplacian $\triangle_\vphi$,  
which is given by \eqref{LaplacianSymbol}, has $e^{h/2}$ and its derivatives in 
its expression. Therefore, when we apply the rearrangement lemma, namely 
Lemma \ref{Rearrfora4}, to calculate the integral in the formula \eqref{generala_2nformula} 
in the case of $a_4$, the elements $\rho_j$ in the statement of the lemma 
might be of the form $e^{-h/2} \delta_1(h/2) e^{h/2} = e^{\nab/2}(\delta_1(h/2))$ or 
$e^{-3h/2} \delta_1\delta_2(h/2) e^{3h/2} = e^{3 \nab/2}(\delta_1 \delta_2(h/2))$, 
hence the appearance of the fractions $s_i/2$ in the exponents of the final formulas.  

\smallskip

This discussion justifies that each function $K_j,$ $j=1, \dots, 20,$ is a smooth 
rational function in variables $s_i$ and $e^{s_i/2}$, whose denominator 
has a concise product formula of a general type, see for example 
the explicit formulas \eqref{K_1explicitformula}, \eqref{K_3explicitformula}, 
\eqref{K8den}, \eqref{K17den}  presented in Section \ref{ExplicitFormulasSec} 
and in Appendix \ref{explicit3and4fnsappsec}. We now 
wish to argue that the coefficients of the numerator of each function 
$K_j$, considering a monomial ordering, belong to a linear space that 
can potentially be of low dimension. 
Therefore, up to multiplication by a constant, the concise denominator 
and the exponents appearing in the numerator can have significant 
information about each function $K_j$.

\smallskip

Let us first analyse  the function $K_1$ from this perspective. We have 
\begin{eqnarray*}
K_1(s_1)&=&
-\frac{4 \pi  e^{\frac{3 s_1}{2}} \left(\left(4 e^{s_1}+e^{2 s_1}+1\right) s_1-3 e^{2 s_1}+3\right)}{\left(e^{s_1}-1\right){}^4 s_1} \\
&=& 
\frac{-4 \pi  e^{\frac{3 s_1}{2}} s_1-16 \pi  e^{\frac{5 s_1}{2}} s_1-4 \pi  e^{\frac{7 s_1}{2}} s_1-12 \pi  e^{\frac{3 s_1}{2}}+12 \pi  e^{\frac{7 s_1}{2}}}{\left(e^{s_1}-1\right){}^4 s_1}. 
\end{eqnarray*} 
By considering the monomial ordering of the numerator inherited from the dictionary 
ordering of the pairs $(i, j)$ such that $s_1^i e^{js_1/2}$ appears in the numerator, 
we replace the specific coefficients by $c_1, \dots, c_5$ and write: 
\[
K_1(s_1) 
= 
\frac{c_1 e^{\frac{3 s_1}{2}}+c_2 e^{\frac{7 s_1}{2}}+c_3 e^{\frac{3 s_1}{2}} s_1+c_4 e^{\frac{5 s_1}{2}} s_1+c_5 e^{\frac{7 s_1}{2}} s_1}{\left(e^{s_1}-1\right){}^4 s_1}. 
\]
Since for the denominator we have 
\[
K_1^{\text{den}}(s_1)
=
\left(e^{s_1}-1\right){}^4 s_1 
= 
O(s_1^5), 
\]
and the function $K_1$ is smooth, we conclude that 
\[ 
c_1 e^{\frac{3 s_1}{2}}+c_2 e^{\frac{7 s_1}{2}}+
c_3 e^{\frac{3 s_1}{2}} s_1+c_4 e^{\frac{5 s_1}{2}} s_1+
c_5 e^{\frac{7 s_1}{2}} s_1 = O(s_1^5). 
\]
Therefore, by writing the Taylor series for each term of the left 
hand side of the above expression, we have: 
\[
c_1\left (\frac{27 s_1^4}{128}+\frac{9 s_1^3}{16}+\frac{9 s_1^2}{8}+\frac{3 s_1}{2}+1 \right )
+c_2 \left (\frac{2401 s_1^4}{384}+\frac{343 s_1^3}{48}+\frac{49 s_1^2}{8}+\frac{7 s_1}{2}+1 \right ) 
\]
\[
+c_3 \left ( \frac{9 s_1^4}{16}+\frac{9 s_1^3}{8}+\frac{3 s_1^2}{2}+s_1 \right ) 
+c_4 \left ( \frac{125 s_1^4}{48}+\frac{25 s_1^3}{8}+\frac{5 s_1^2}{2}+s_1 \right ) 
\]
\[
+c_5 \left ( \frac{343 s_1^4}{48}+\frac{49 s_1^3}{8}+\frac{7 s_1^2}{2}+s_1 \right ) = 0. 
\]
This implies that the vector $(c_1, c_2, c_3, c_4, c_5)$ belongs 
to the kernel of the $5 \times 5$ matrix 
\[
\left(
\begin{array}{ccccc}
 1 & 1 & 0 & 0 & 0 \\
 3/2 & 7/2 & 1 & 1 & 1 \\
 9/8 & 49/8 & 3/2 & 5/2 & 7/2 \\
 9/16 & 343/48 & 9/8 & 25/8 & 49/8 \\
 27/128 & 2401/384 & 9/16 & 125/48& 343/48 \\
\end{array}
\right), 
\]
whose rank is $4$. Therefore, the vector of the coefficients 
$(c_1, c_2, c_3, c_4, c_5)$ of the numerator of the function 
$K_1$ belongs to a 1-dimensional linear space, 
and up to multiplication by a constant, the numerator of 
$K_1$ is determined by the exponents appearing in its 
numerator, which are plotted in Figure \ref{ExponentsofNumofK1Fig}.

\begin{figure} 
\includegraphics[scale=0.4]{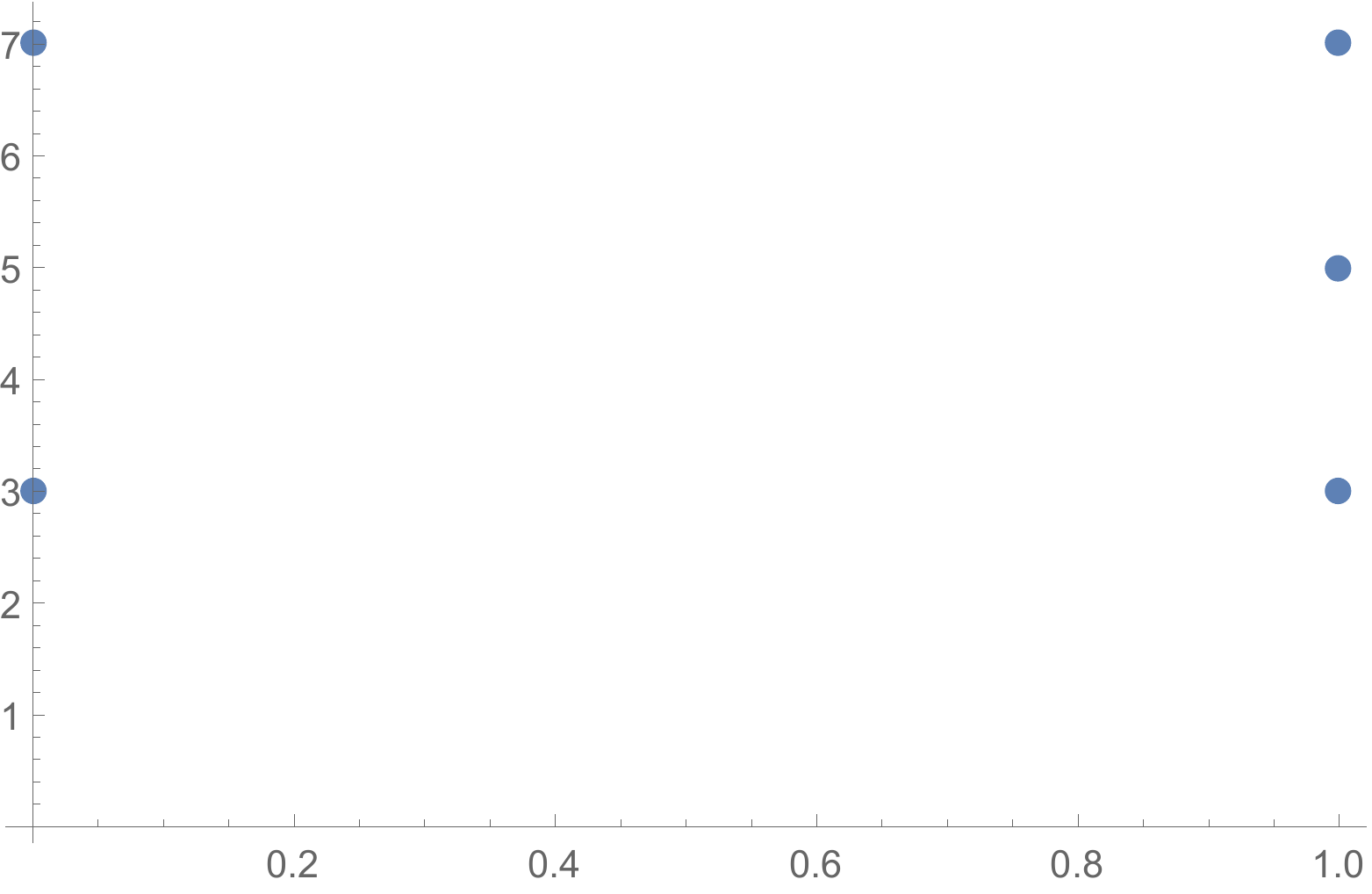}
\caption{The points $(i, j)$ such that $s_1^i e^{j s_1/2}$ 
appears in the numerator of $K_1$.}
\label{ExponentsofNumofK1Fig}
\end{figure}

\smallskip

A similar analysis can in fact be performed for all functions $K_j$, $j=1, \dots, 20,$ 
as far as each monomial in the numerator cannot afford to provide a 
non-zero smooth quotient by itself. We just need to confirm that chosen a number 
of such monomials for the numerator, a system of linear equations for their 
coefficients will be equivalent to smoothness of the quotient.  
This is easy to see as we discuss it for example for the function $K_3$. 
This function is given by 
\[
K_3(s_1, s_2) = \frac{K_3^{\text{num}}(s_1, s_2)}{K_3^{\text{den}}(s_1, s_2)}, 
\]
where
\[
K_3^{\text{den}}(s_1, s_2)
=
\left(e^{s_1}-1\right){}^2 \left(e^{s_2}-1\right){}^2
   \left(e^{s_1+s_2}-1\right){}^4 s_1 s_2 \left(s_1+s_2\right), 
\]
and $K_3^{\text{num}}$ is the sum of 74 different  monomials in 
$s_1, s_2, e^{s_1/2} e^{s_2/2}$: 
\[
K_3^{\text{num}}(s_1, s_2)= 
\]
\begin{center}
\begin{math}
64 e^{\frac{5 s_1}{2}+\frac{3 s_2}{2}} \pi  s_1^2-32 e^{\frac{7 s_1}{2}+\frac{3 s_2}{2}} \pi  s_1^2-192 e^{\frac{5 s_1}{2}+\frac{5 s_2}{2}} \pi  s_1^2+128 e^{\frac{7 s_1}{2}+\frac{5 s_2}{2}} \pi  s_1^2-64 e^{\frac{9 s_1}{2}+\frac{5 s_2}{2}} \pi  s_1^2+192 e^{\frac{5 s_1}{2}+\frac{7 s_2}{2}} \pi  s_1^2-192 e^{\frac{7 s_1}{2}+\frac{7 s_2}{2}} \pi  s_1^2+192 e^{\frac{9 s_1}{2}+\frac{7 s_2}{2}} \pi  s_1^2-64 e^{\frac{5 s_1}{2}+\frac{9 s_2}{2}} \pi  s_1^2+128 e^{\frac{7 s_1}{2}+\frac{9 s_2}{2}} \pi  s_1^2-192 e^{\frac{9 s_1}{2}+\frac{9 s_2}{2}} \pi  s_1^2-32 e^{\frac{7 s_1}{2}+\frac{11 s_2}{2}} \pi  s_1^2+64 e^{\frac{9 s_1}{2}+\frac{11 s_2}{2}} \pi  s_1^2+64 e^{\frac{5 s_1}{2}+\frac{3 s_2}{2}} \pi  s_2 s_1-32 e^{\frac{7 s_1}{2}+\frac{3 s_2}{2}} \pi  s_2 s_1-64 e^{\frac{3 s_1}{2}+\frac{5 s_2}{2}} \pi  s_2 s_1-64 e^{\frac{7 s_1}{2}+\frac{5 s_2}{2}} \pi  s_2 s_1+32 e^{\frac{3 s_1}{2}+\frac{7 s_2}{2}} \pi  s_2 s_1+64 e^{\frac{5 s_1}{2}+\frac{7 s_2}{2}} \pi  s_2 s_1+64 e^{\frac{9 s_1}{2}+\frac{7 s_2}{2}} \pi  s_2 s_1+32 e^{\frac{11 s_1}{2}+\frac{7 s_2}{2}} \pi  s_2 s_1-64 e^{\frac{7 s_1}{2}+\frac{9 s_2}{2}} \pi  s_2 s_1-64 e^{\frac{11 s_1}{2}+\frac{9 s_2}{2}} \pi  s_2 s_1-32 e^{\frac{7 s_1}{2}+\frac{11 s_2}{2}} \pi  s_2 s_1+64 e^{\frac{9 s_1}{2}+\frac{11 s_2}{2}} \pi  s_2 s_1-48 e^{\frac{3 s_1}{2}+\frac{3 s_2}{2}} \pi  s_1+128 e^{\frac{5 s_1}{2}+\frac{3 s_2}{2}} \pi  s_1-80 e^{\frac{7 s_1}{2}+\frac{3 s_2}{2}} \pi  s_1+64 e^{\frac{3 s_1}{2}+\frac{5 s_2}{2}} \pi  s_1-192 e^{\frac{5 s_1}{2}+\frac{5 s_2}{2}} \pi  s_1+64 e^{\frac{7 s_1}{2}+\frac{5 s_2}{2}} \pi  s_1+64 e^{\frac{9 s_1}{2}+\frac{5 s_2}{2}} \pi  s_1-16 e^{\frac{3 s_1}{2}+\frac{7 s_2}{2}} \pi  s_1+128 e^{\frac{5 s_1}{2}+\frac{7 s_2}{2}} \pi  s_1-128 e^{\frac{9 s_1}{2}+\frac{7 s_2}{2}} \pi  s_1+16 e^{\frac{11 s_1}{2}+\frac{7 s_2}{2}} \pi  s_1-64 e^{\frac{5 s_1}{2}+\frac{9 s_2}{2}} \pi  s_1-64 e^{\frac{7 s_1}{2}+\frac{9 s_2}{2}} \pi  s_1+192 e^{\frac{9 s_1}{2}+\frac{9 s_2}{2}} \pi  s_1-64 e^{\frac{11 s_1}{2}+\frac{9 s_2}{2}} \pi  s_1+80 e^{\frac{7 s_1}{2}+\frac{11 s_2}{2}} \pi  s_1-128 e^{\frac{9 s_1}{2}+\frac{11 s_2}{2}} \pi  s_1+48 e^{\frac{11 s_1}{2}+\frac{11 s_2}{2}} \pi  s_1-64 e^{\frac{3 s_1}{2}+\frac{5 s_2}{2}} \pi  s_2^2+192 e^{\frac{5 s_1}{2}+\frac{5 s_2}{2}} \pi  s_2^2-192 e^{\frac{7 s_1}{2}+\frac{5 s_2}{2}} \pi  s_2^2+64 e^{\frac{9 s_1}{2}+\frac{5 s_2}{2}} \pi  s_2^2+32 e^{\frac{3 s_1}{2}+\frac{7 s_2}{2}} \pi  s_2^2-128 e^{\frac{5 s_1}{2}+\frac{7 s_2}{2}} \pi  s_2^2+192 e^{\frac{7 s_1}{2}+\frac{7 s_2}{2}} \pi  s_2^2-128 e^{\frac{9 s_1}{2}+\frac{7 s_2}{2}} \pi  s_2^2+32 e^{\frac{11 s_1}{2}+\frac{7 s_2}{2}} \pi  s_2^2+64 e^{\frac{5 s_1}{2}+\frac{9 s_2}{2}} \pi  s_2^2-192 e^{\frac{7 s_1}{2}+\frac{9 s_2}{2}} \pi  s_2^2+192 e^{\frac{9 s_1}{2}+\frac{9 s_2}{2}} \pi  s_2^2-64 e^{\frac{11 s_1}{2}+\frac{9 s_2}{2}} \pi  s_2^2+48 e^{\frac{3 s_1}{2}+\frac{3 s_2}{2}} \pi  s_2-64 e^{\frac{5 s_1}{2}+\frac{3 s_2}{2}} \pi  s_2+16 e^{\frac{7 s_1}{2}+\frac{3 s_2}{2}} \pi  s_2-128 e^{\frac{3 s_1}{2}+\frac{5 s_2}{2}} \pi  s_2+192 e^{\frac{5 s_1}{2}+\frac{5 s_2}{2}} \pi  s_2-128 e^{\frac{7 s_1}{2}+\frac{5 s_2}{2}} \pi  s_2+64 e^{\frac{9 s_1}{2}+\frac{5 s_2}{2}} \pi  s_2+80 e^{\frac{3 s_1}{2}+\frac{7 s_2}{2}} \pi  s_2-64 e^{\frac{5 s_1}{2}+\frac{7 s_2}{2}} \pi  s_2+64 e^{\frac{9 s_1}{2}+\frac{7 s_2}{2}} \pi  s_2-80 e^{\frac{11 s_1}{2}+\frac{7 s_2}{2}} \pi  s_2-64 e^{\frac{5 s_1}{2}+\frac{9 s_2}{2}} \pi  s_2+128 e^{\frac{7 s_1}{2}+\frac{9 s_2}{2}} \pi  s_2-192 e^{\frac{9 s_1}{2}+\frac{9 s_2}{2}} \pi  s_2+128 e^{\frac{11 s_1}{2}+\frac{9 s_2}{2}} \pi  s_2-16 e^{\frac{7 s_1}{2}+\frac{11 s_2}{2}} \pi  s_2+64 e^{\frac{9 s_1}{2}+\frac{11 s_2}{2}} \pi  s_2-48 e^{\frac{11 s_1}{2}+\frac{11 s_2}{2}} \pi  s_2. 
\end{math}
\end{center}

\smallskip

By considering the monomial ordering of the numerator $K_3^{\text{num}}$ 
inherited from the dictionary ordering of the 4-tuples 
$(\alpha_i,  \beta_i, \gamma_i,  \nu_i)$ such that 
$s_1^{\alpha_i} s_2^{\beta_i} e^{\gamma_i s_1/2} e^{\nu_i s_2/2}$ 
appears in the numerator, let us replace the specific coefficients by 
$c_1, \dots, c_{74}$ and write:
\[
K_3^{\text{num}}(s_1, s_2) 
= 
\sum_{i=1}^{74} c_i \,s_1^{\alpha_i} \,s_2^{\beta_i}\, e^{\gamma_i s_1/2} \, e^{\nu_i s_2/2}. 
\]

By passing to the polar coordinates 
\[
s_1 = r \cos \theta, \qquad s_2 = r \sin \theta,
\] for the 
denominator of $K_3$ we have: 
\[
K_3^{\text{den}}(s_1, s_2) 
=
\left(e^{s_1}-1\right){}^2 \left(e^{s_2}-1\right){}^2
   \left(e^{s_1+s_2}-1\right){}^4 s_1 s_2 \left(s_1+s_2\right) 
= 
O(r^{11}), 
\] 
which, combined with the smoothness of the function $K_3$ at the origin, 
implies that 
\[
K_3^{\text{num}}(s_1, s_2)= \sum_{i=1}^{74} c_i \,s_1^{\alpha_i} \,s_2^{\beta_i}\, e^{\gamma_i s_1/2} \, e^{\nu_i s_2/2} = O(r^{11}). 
\]
By considering the Taylor series of each term in the middle expression 
in the latter, we obtain an $11 \times 74$ matrix, whose entries are 
trigonometric functions of $\theta$, and the vector $(c_1, \dots, c_{74})$ 
of the coefficients of $K_3^{\text{num}}(s_1, s_2)$ belongs to the kernel 
of this matrix for each $\theta$. These linear equations only correspond to smoothness 
of the following mapping at the origin: 
\begin{equation} \label{K3relatedquotient}
(s_1, s_2) \mapsto 
\frac{\sum_{i=1}^{74} c_i \,s_1^{\alpha_i} \,s_2^{\beta_i}\, e^{\gamma_i s_1/2} \, e^{\nu_i s_2/2}}{K_3^{\text{den}}(s_1, s_2) }.
\end{equation}

\smallskip

Since the zeros of $K_3^{\text{den}}(s_1, s_2)$ are located on the lines 
$s_1=0,$ $s_2=0,$ $s_1+s_2=0$, we still need to observe that there 
are further linear equations that guarantee smoothness 
of the mapping given by \eqref{K3relatedquotient}. This can also be seen as follows. 
Since in the Taylor series of $K_3^{\text{den}}(s_1, s_2)$, 
the coefficients of $s_1^p s_2^q$ are zero for $p, q \in \{0, 1, 2, 3, 4, 5 \}$, 
the corresponding coefficients are required to be zero in Taylor series of 
$\sum_{i=1}^{74} c_i \,s_1^{\alpha_i} \,s_2^{\beta_i}\, 
e^{\gamma_i s_1/2} \, e^{\nu_i s_2/2}$ in order to have a smooth 
quotient, which settles the fact that smoothness of the quotient 
on $s_1=0$ and $s_2=0$ is equivalent to a system of linear equations 
for the coefficients $c_1, \dots, c_{74}$. Also, in order to treat the 
smoothness of the quotient \eqref{K3relatedquotient} on the line $s_1+s_2=0,$ 
we note that 
\[ 
\ddsone^j K_3^{\text{den}}(-s_2, s_2) 
= 
\ddstwo^j  K_3^{\text{den}}(s_1,-  s_1) = 0, \qquad 
0 \leq j \leq 4.  
\]
Therefore the limit of the map given by \eqref{K3relatedquotient} as $s_1 \to -s_2$ 
or $s_2 \to - s_1$ exists if and only if for any $0 \leq j \leq 4:$
\begin{eqnarray*}
&& \ddsone^j  \left ( 
\sum_{i=1}^{74} c_i \,s_1^{\alpha_i} \,s_2^{\beta_i}\, e^{\gamma_i s_1/2} \, e^{\nu_i s_2/2} 
\right ) \bigm|_{s_1 = -s_2} \\
&=&
\ddstwo^j \left ( 
\sum_{i=1}^{74} c_i \,s_1^{\alpha_i} \,s_2^{\beta_i}\, e^{\gamma_i s_1/2} \, e^{\nu_i s_2/2} 
\right ) \bigm|_{s_2 = -s_1} \\
&=& 0. 
\end{eqnarray*}
This finishes the justification of our claim that the smoothness of the quotient map given 
by \eqref{K3relatedquotient} is equivalent to a system of linear equations 
for the coefficients $c_1, \dots, c_{74}$.

\smallskip

Since the condition and the arguments used above apply to all functions 
appearing in the expression \eqref{a_4expression} for the term $a_4$, 
and based on the generality of the arguments given for 
proving Lemma \ref{BabyFunctions2} and the results in Section \ref{Calculatea_4Sec}, 
it is natural to expect a similar phenomena to occur for the following 
terms in the expansion \eqref{heatexp}. That is, we predict 
that each term $a_{2n} \in \CNT$ appearing in this expansion 
is described by smooth rational functions of variables $s_1, \dots, s_{2n},$ 
$e^{s_1/2}, \dots, e^{s_{2n}/2}$, whose denominators have 
concise product formulas, which vanish on certain hyperplanes, 
and the coefficients in the numerator of each function 
satisfy a family of linear equations.

\smallskip

\section{The term $a_4$ for certain noncommutative four tori}  
\label{NC4toriSec}

As a corollary of the main calculation of the present paper one obtains, for noncommutative 
four tori  with curved metrics of the product  
form $\NTp \times \NTpp$,  the  explicit form of the local  geometric invariant 
given by the term $a_4$. 
As explained below this required the computation of the term $a_4$ for each component of the product. 
The reason why this gives  a first  hint of  the analog of the Riemann curvature in the general noncommutative twisted case is that product metrics of the above form are in general not conformally flat while in the traditional Riemannian case  the term $a_4$ already involves complicated expressions (\cite{GilBook} Theorem 4.8.18),   in terms of the curvature tensor. More explicitly the classical (commutative) version of the product metric that we consider is the following. Using the 
local coordinates $(x_1, y_1, x_2, y_2) \in$  $\mathbb{T}^4 = (\R/ 2\pi \Z)^4$ on the four torus, 
the metric is written as 
\begin{equation*} 
g =  e^{-h_1(x_1, y_1)} \left ( dx_1^2 + dy_1^2\right ) + e^{-h_2(x_2, y_2)} \left ( dx_2^2 + dy_2^2\right ),  
\end{equation*}
where $h_1$ and $h_2$ are smooth real valued functions. The following non-vanishing components 
of the Weyl curvature tensor of this metric determine its Weyl curvature: 
\[
C_{1212} =
\]
\begin{center}
\begin{math}
\frac{1}{6} e^{-h_1\left(x_1,y_1\right)} \partial_{y_1}^2 h_1{}\left(x_1,y_1\right)
+\frac{1}{6} e^{h_2\left(x_2,y_2\right) -2 h_1\left(x_1,y_1\right)}  \partial_{y_2}^2 h_2{}\left(x_2,y_2\right)
+\frac{1}{6} e^{-h_1\left(x_1,y_1\right)}  \partial_{x_1}^2 h_1\left(x_1,y_1\right)
+\frac{1}{6} e^{h_2\left(x_2,y_2\right) -2 h_1\left(x_1,y_1\right)}  \partial_{x_2}^2h_2\left(x_2,y_2\right), 
\end{math}
\end{center}
\[
C_{1313}= -\frac{1}{2}e^{-h_2\left(x_2,y_2\right) + h_1\left(x_1,y_1\right)} C_{1212}, 
\]
\[
C_{2424} = C_{2323} = C_{1414}= C_{1313}, 
\]
\[
C_{3434} = e^{-2 h_2\left(x_2,y_2\right) +2 h_1\left(x_1,y_1\right)} C_{1212}. 
\]

\smallskip 

In the noncommutative case for the product metric 
on $\NTp \times \NTpp$,  the modular automorphism groups of both factors combine to give an action of $\R^2$ which defines a more refined twisting than the classical  determinant twisting of spectral triples, which itself corresponds to the restriction of the above action of $\R^2$ to the diagonal $\R\subset \R^2$.  This begs  to investigate the more general notion of twisting suggested in particular in \cite{ConTransverse} and which plays a fundamental role in the work of H. Moscovici and the first author \cite{CMos} on the transverse geometry of foliations and the reduction by duality to the almost isometric case.

\smallskip

For now we simply explain how to derive the $a_4$ term for the product metric 
on $\NTp \times \NTpp$.   
That is, 
let us consider a noncommutative four torus of the form 
$\NTp \times \NTpp$, whose algebra has four unitary generators 
$U_1, V_1, U_2, V_2$ such that each element of the pair $(U_1, V_1)$ 
commutes with each element of the pair $(U_2, V_2)$, and we have 
the following commutation relations for fixed 
irrational real numbers $\theta'$ and $\theta''$: 
\[
V_1 \, U_1 = e^{2 \pi i \theta'}\, U_1 \,V_1, 
\qquad 
V_2 \,U_2= e^{2 \pi i \theta''}\, U_2\, V_2.    
\] 
By conformally perturbing the flat metric on each two torus factor of such 
a noncommutative four torus, one obtains the Laplacian associated 
with a curved metric on this noncommutative space. That is, by using 
conformal factors $e^{-h'}$ and $e^{-h''}$, where $h'$ and $h''$ are 
respectively selfadjoint elements in $\CNTp$ and $\CNTpp$, 
we can consider the corresponding perturbed Laplacians 
$\triangle_{\vphi'}$ and $\triangle_{\vphi''}$ of the form 
given by \eqref{conformalLaplacian}, and form the following 
Laplacian on the noncommutative four torus: 
\[
\triangle_{\vphi', \vphi''} 
= 
\triangle_{\vphi'} \otimes 1 + 1 \otimes \triangle_{\vphi''}. 
\]

\smallskip

As the notation suggests, $\vphi'$ and $\vphi''$ are respectively the states on 
$C(\NTp)$ and $C(\NTpp)$ obtained from the corresponding canonical traces 
$\vphi'_0$ and $\vphi''_0$, using the conformal factors $e^{-h'}$ and $e^{-h''}$. 
There are unique elements $a_{2n} \in C^\infty(\NTp \times \NTpp)$ such that 
for any $a \in C^\infty(\NTp \times \NTpp)$, as $t \to 0^+,$ we have: 
\begin{eqnarray} \label{heatexp4d}
&& \Tr(a \exp ( -t \triangle_{\vphi', \vphi''}))  \sim \nonumber  \\
&& \qquad \qquad  t^{-2} \left ( ( \vphi'_0 \otimes \vphi''_0 ) (a\, a_0) + (\vphi'_0 \otimes \vphi''_0 ) (a\, a_2)\, t + ( \vphi'_0 \otimes \vphi''_0) (a\, a_4)\, t^2 + \cdots \right ).   
\end{eqnarray}
The term $a_4$ is the most fundamental and desirable term in this expansion since it is the first 
term in which  the analog of the Riemann curvature tensor manifests itself, whereas $a_0$ and $a_2$ are 
only the analogs of the volume form and the scalar curvature, respectively.

\smallskip

An extremely difficult way of calculating the desired $a_{4}$ appearing in the expansion \eqref{heatexp4d} 
is to use the method 
described in Section \ref{Calculatea_4Sec}, now in a four dimensional case. However,  
the expansion \eqref{heatexp} confirms the existence of the unique elements $a'_{2n} \in \CNTp$ 
and $a''_{2n} \in \CNTpp$ such that for any $a' \in \CNTp$ and $a'' \in \CNTpp,$ we have the following 
small-time expansions: 
\begin{eqnarray*}
\Tr( a' \exp (-t \triangle_{\vphi'}) )  
&\sim& 
t^{-1} \left ( \vphi'_0(a' \, a'_0) + \vphi'_0( a'\, a'_2) \, t + \vphi'_0(a'\, a'_4) \,t^2 + \cdots \right ),  \\
\Tr( a'' \exp ( -t \triangle_{\vphi''}) )  
&\sim& 
t^{-1} \left ( \vphi''_0(a'' \, a''_0) + \vphi''_0(a''\, a''_2) \, t + \vphi''_0(a''\, a''_4)\, t^2 + \cdots \right ). 
\end{eqnarray*}
Therefore, the uniqueness of the terms $a_{2n}$ appearing in the expansion  \eqref{heatexp4d}, 
combined with these expansions and making use of simple tensors as test elements, readily 
implies that we have:  
\[
a_{2n} = \sum_{i=0}^{n} a'_{2i} \otimes a''_{2(n-i)}   \in   C^\infty(\NTp \times \NTpp). 
\]
We recall from  \cite{FatKhaTraceThm} that the terms $a'_0$ and $a''_0$ are quite easy to 
calculate, namely $a'_0 = \pi e^{-h'}$ and $a''_0 = \pi e^{-h''}$. The terms $a'_2$ and $a''_2$ are given by 
\eqref{SCformula} which was calculated in \cite{ConMosModular, FatKhaSC2T}, 
and, most importantly, the $a'_4$ and $a''_4$ are now given explicitly by the 
formula \eqref{a_4expression}, emphasizing that 
we have achieved explicit formulas for all components of the latter. Hence, we have available an explicit 
formula for the term   
\[
a_4 = a'_0 \otimes a''_4 + a'_2 \otimes a''_2 + a'_4 \otimes a''_0 \in C^\infty(\NTp \times \NTpp), 
\]
which appears in the heat kernel expansion \eqref{heatexp4d} associated with a curved 
noncommutative geometry of spectral dimension four, hence explicit information about the 
analog of the Riemann curvature tensor in the noncommutative setting. 

\smallskip

\section{Conclusions}
\label{ConclusionsSec}

The local geometric invariants of the noncommutative 
two torus $\NT$ equipped with a conformally flat metric 
have complicated dependance upon the modular automorphism of the state that encodes the volume form of the 
metric and plays the role of the Weyl factor. This dependance  involves lengthy several 
variable functions of the modular automorphism 
of the state. In this paper we have calculated this dependance for the invariant 
$a_4 \in \CNT,$ which determines the third term 
in the small-time asymptotic expansion of the trace of the 
heat kernel of the Laplacian associated with the conformally 
flat metric. After performing  heavy calculations 
and simplifying the result, we have confirmed 
the accuracy of the final functions appearing 
in the expression for the term $a_4$ by checking 
that they satisfy a family of functional relations which were abstractly predicted before performing the computation. 
The derivation of these functional relations 
is based on using a fundamental spectral identity 
proved in \cite{ConMosModular} for the gradient 
of a functional, and by calculating the same gradient 
with finite differences. This method was indeed 
used in \cite{ConMosModular} for checking the 
validity of the term $a_2$, which is related to the 
scalar curvature for $\NT$. However, the tools 
needed for performing this check for the term $a_4$ and 
the consequent functional relations are by far more 
involved and complicated, compared to the case 
of the term $a_2$. Also by studying the functional relations 
abstractly, we have derived a partial differential system, on which 
the cyclic groups of order two, three and four act naturally and we have 
studied invariance properties and symmetries of the calculated expressions with respect 
to these actions. Moreover, we have found a natural flow that is associated with the 
differential system. 

\smallskip

We should stress that the basic functional relations for the functions 
$\widetilde K_1, \dots, \widetilde K_{20}$ stated in Theorem \ref{FuncRelationsThm} 
were conceptually predicted, and  lead to the differential system and the discovery 
of symmetric expressions, when studied abstractly. However, the functional relations among the functions $k_3, \dots, k_{20}$ 
stated in Theorem \ref{funcrlnsamongksThm} were found after comparing the final explicit  formulas for the 
functions $K_3, \dots, K_{20}$ and it is an open question to understand them a priori  in a conceptual manner.

\smallskip

We have shown that the main ingredients of our calculations 
can be derived by finite differences from the generating 
function of the Bernoulli numbers and its inverse, and we have paid special 
attention to the general structure of the final functions that 
describe the term $a_4$. That is, due to noncommutativity,  
several variable functions of the modular automorphism are needed 
for integrating certain $\CNT$-valued functions defined on the positive real line, 
and for writing the final formula for the term $a_4$ in terms of the derivatives 
of the logarithm of the conformal factor. We have argued 
that each of the final functions is a rational function of variables $s_i$ 
and $e^{s_i/2}$, whose denominator has a concise product 
formula, and the coefficients of its numerator satisfy a family of 
linear equations, which, given a monomial ordering, restrict them to 
belong to a potentially low dimensional linear space. 
This begs for bringing in the front scene   the algebraic geometry, familiar in 
transcendence theory \cite{Waldschmidt}, of exponential polynomials and their smooth 
fractions in understanding the general structure of the noncommutative local invariants.

\smallskip

\smallskip

It is also important to stress that, because of their geometric nature, 
the terms $a_{2n}$ in \eqref{heatexp4d}  are local geometric invariants of the curved metric 
on the noncommutative four torus in the following  sense. In noncommutative geometry the notion of locality is less obvious
than the classical naive notion. This notion nevertheless makes sense
based on Fourier transform, and what it means concretely in the above context is that only
the high frequency contributions are relevant in the computation of 
the coefficients appearing in the expansion of the trace of the heat kernel. This illuminates the locality of the terms $a_{2n} \in \CNT$ 
since when one considers using the pseudodifferential calculus built 
on a phase space with a noncommutative configuration space \cite{ConC*algDiffGeo}, the high 
frequencies appear preponderantly in the derivation of the heat expansion \eqref{heatexp4d}. 
This locality principle is already at the core of the local index formula of 
\cite{ConMosLocalIndex} and also plays an important simplifying role in the case of quantum groups.

\smallskip

As is well known in the case of  Riemannian manifolds, the local geometric invariants appearing in the heat expansion, are certain expressions in terms of the Riemann curvature tensor and its contractions 
and covariant derivatives. These expressions are very complicated, however, because of their 
local nature, one can  identify them using invariance theory \cite{GilBook}, noting that only a few terms 
have in practice been identified due to the rapid growth in complexity of the formulas. As indicated earlier, 
the first two terms are given by the volume form and the scalar curvature, and the third term, the analog 
of which in our noncommutative setting is the  term $a_4$,  
is the first place where the  Riemann curvature tensor manifests itself beyond the curvature scalar. Thus, we view   
the results achieved in this paper as an important step towards the exploration of  further properties of the analog of the Riemann curvature tensor in 
noncommutative geometry as explained briefly in Section \ref{NC4toriSec} for $4$-tori. This paper should be viewed as providing a vast reservoir of concrete data, in the form of the obtained explicit formulas,  
which  can be exploited further for testing ideas.  In fact, research in this area of noncommutative geometry possesses an experimental nature,  
and we intend to share in a mathematical program notebook the invaluable resource of data that we have 
obtained for the $a_4$ term of the heat expansion.   This will  pave the way for researchers to 
discover many more and far different phenomena in noncommutative differential geometry 
\cite{ConNDG, ConNCGBook, ConMosLocalIndex}, awaiting to be discovered in our data. Our final formulas are securely tested  since they satisfy a highly nontrivial 
family of conceptually predicted functional relations, derived by comparing the outcomes of two different  
abstract calculations of a gradient. This method, first introduced in \cite{ConMosModular}, plays in our setting the role of invariance theory in the classical 
case for performing a check on a calculated expression. The symmetries that we have discovered  
in the calculated expressions are quite striking and are closely related 
to the appearance of the action of the cyclic groups in the differential system.  
Note also that  considering the work carried out in \cite{LiuToric}, 
our calculations for noncommutative tori are universal 
in the sense that they extend to noncommutative toric manifolds \cite{ConLanTheta}.

\smallskip 

We end our conclusions by pointing out that while the results of this paper have a very concrete aspect,  
they suggest, as explained above in Section \ref{NC4toriSec} for $4$-tori, the need to  explore at the conceptual level the more general twisting of spectral triples already arising in the transverse geometry of foliations, which goes well beyond  the non-tracial type III nature of the measure theory since the latter is limited to the behavior of the determinant of the metric.     
The theory of  
twisted spectral triples has reached a satisfactory status \cite{ConMosTypeIII, ChaConScaleInv, ConTreGB, PonWanConformal, 
FatGabChernGB}.  The general case appears as an open land where  the case of the transverse geometry of foliations provides a wealth of examples and where general principles such as the reduction to the almost isometric case \cite{ConTransverse, CMos} should be valid in general and allow one to 
 apply the local geometric methods of \cite{ConNDG, ConMosLocalIndex, CMos2}.

\appendix

\smallskip

\section{Lengthy functional relations}
\label{lengthyfnrelationsappsec}

We explained in Section \ref{a_4Sec} that like the 
basic functional relations given respectively by the equations \eqref{basicK8eqn} and 
\eqref{basicK17eqn} for the functions $\widetilde K_9$ and $\widetilde K_{17}$, the 
remaining three and four variable functions have algebraically lengthy basic functional  identities. 
We present these lengthy expressions in this appendix.

\subsection{Functional relations for $\widetilde K_{9}, \dots, \widetilde K_{16}$} First we cover the 
remaining three variable functions.

\subsubsection{The function $\widetilde K_9$}
\label{basicK9}
We have: 

\begin{equation} \label{basicK9eqn}
\widetilde K_9(s_1, s_2, s_3) =
\end{equation}

\begin{center}
\begin{math}
  \frac{1}{15} (-4) \pi  G_3\left(s_1,s_2,s_3\right)+\frac{1}{4} e^{s_3} G_2\left(s_1,s_2\right) k_3\left(-s_3\right)-\frac{e^{s_3} \left(e^{s_2} s_1 k_3\left(-s_2-s_3\right)+e^{s_2} s_2 k_3\left(-s_2-s_3\right)-e^{s_1+s_2} s_2 k_3\left(-s_1-s_2-s_3\right)-s_1 k_3\left(-s_3\right)\right)}{4 s_1 s_2 \left(s_1+s_2\right)}+\frac{1}{4} G_2\left(s_1,s_2\right) k_3\left(s_3\right)+\frac{G_1\left(s_1\right) \left(k_3\left(s_3\right)-k_3\left(s_2+s_3\right)\right)}{4 s_2}+\frac{s_1 k_3\left(s_3\right)-s_1 k_3\left(s_2+s_3\right)-s_2 k_3\left(s_2+s_3\right)+s_2 k_3\left(s_1+s_2+s_3\right)}{4 s_1 s_2 \left(s_1+s_2\right)}-G_2\left(s_1,s_2\right) k_6\left(s_3\right)+\frac{G_1\left(s_1\right) \left(k_6\left(s_2\right)-k_6\left(s_2+s_3\right)\right)}{2 s_3}+\frac{k_6\left(s_2\right)-k_6\left(s_1+s_2\right)-k_6\left(s_2+s_3\right)+k_6\left(s_1+s_2+s_3\right)}{2 s_1 s_3}+\frac{-s_3 k_6\left(s_1\right)+s_2 k_6\left(s_1+s_2\right)+s_3 k_6\left(s_1+s_2\right)-s_2 k_6\left(s_1+s_2+s_3\right)}{2 s_2 s_3 \left(s_2+s_3\right)}+\frac{-s_1 k_6\left(s_3\right)+s_1 k_6\left(s_2+s_3\right)+s_2 k_6\left(s_2+s_3\right)-s_2 k_6\left(s_1+s_2+s_3\right)}{s_1 s_2 \left(s_1+s_2\right)}+\frac{e^{s_2} G_1\left(s_1\right) \left(k_7\left(-s_2\right)-e^{s_3} k_7\left(-s_2-s_3\right)\right)}{2 s_3}+\frac{e^{s_2} \left(-e^{s_1} k_7\left(-s_1-s_2\right)+k_7\left(-s_2\right)-e^{s_3} k_7\left(-s_2-s_3\right)+e^{s_1+s_3} k_7\left(-s_1-s_2-s_3\right)\right)}{4 s_1 s_3}-\frac{e^{s_1} \left(s_3 k_7\left(-s_1\right)-e^{s_2} s_2 k_7\left(-s_1-s_2\right)-e^{s_2} s_3 k_7\left(-s_1-s_2\right)+e^{s_2+s_3} s_2 k_7\left(-s_1-s_2-s_3\right)\right)}{2 s_2 s_3 \left(s_2+s_3\right)}-\frac{e^{s_2} \left(e^{s_1} k_7\left(-s_1-s_2\right)-k_7\left(-s_2\right)+e^{s_3} k_7\left(-s_2-s_3\right)-e^{s_1+s_3} k_7\left(-s_1-s_2-s_3\right)\right)}{4 s_1 s_3}+\frac{e^{s_3} G_1\left(s_1\right) \left(e^{s_2} k_7\left(-s_2-s_3\right)-k_7\left(-s_3\right)\right)}{s_2}-e^{s_3} G_2\left(s_1,s_2\right) k_7\left(-s_3\right)+\frac{e^{s_3} \left(e^{s_2} s_1 k_7\left(-s_2-s_3\right)+e^{s_2} s_2 k_7\left(-s_2-s_3\right)-e^{s_1+s_2} s_2 k_7\left(-s_1-s_2-s_3\right)-s_1 k_7\left(-s_3\right)\right)}{s_1 s_2 \left(s_1+s_2\right)}+\frac{1}{4} G_1\left(s_1\right) k_8\left(s_2,s_3\right)+\frac{k_8\left(s_2,s_3\right)-k_8\left(s_1+s_2,s_3\right)}{4 s_1}+\frac{k_8\left(s_1+s_2,s_3\right)-k_8\left(s_1,s_2+s_3\right)}{4 s_2}+\frac{\left(-1+e^{s_1+s_2+s_3}\right) k_9\left(s_1,s_2\right)}{8 \left(s_1+s_2+s_3\right)}+\frac{k_9\left(s_1,s_2+s_3\right)-k_9\left(s_1,s_2\right)}{8 s_3}+\frac{1}{8} G_1\left(s_1\right) k_9\left(s_2,s_3\right)+\frac{k_9\left(s_2,s_3\right)-k_9\left(s_1+s_2,s_3\right)}{8 s_1}+\frac{k_9\left(s_1+s_2,s_3\right)-k_9\left(s_1,s_2+s_3\right)}{8 s_2}-\frac{1}{8} e^{s_2+s_3} G_1\left(s_1\right) k_9\left(-s_2-s_3,s_2\right)+\frac{e^{s_1+s_2+s_3} \left(k_9\left(-s_1-s_2-s_3,s_1\right)-k_9\left(-s_1-s_2-s_3,s_1+s_2\right)\right)}{8 s_2}-\frac{1}{8} G_1\left(s_1\right) k_{10}\left(s_2,s_3\right)+\frac{k_{10}\left(s_1,s_2+s_3\right)-k_{10}\left(s_1+s_2,s_3\right)}{8 s_2}+\frac{k_{10}\left(s_1+s_2,s_3\right)-k_{10}\left(s_2,s_3\right)}{8 s_1}+\frac{1}{8} e^{s_2} G_1\left(s_1\right) k_{10}\left(s_3,-s_2-s_3\right)+\frac{e^{s_2} \left(k_{10}\left(s_3,-s_2-s_3\right)-e^{s_1} k_{10}\left(s_3,-s_1-s_2-s_3\right)\right)}{8 s_1}+\frac{e^{s_1} \left(e^{s_2} k_{10}\left(s_3,-s_1-s_2-s_3\right)-k_{10}\left(s_2+s_3,-s_1-s_2-s_3\right)\right)}{8 s_2}+\frac{1}{4} e^{s_2} G_1\left(s_1\right) k_{11}\left(s_3,-s_2-s_3\right)+\frac{e^{s_2} \left(k_{11}\left(s_3,-s_2-s_3\right)-e^{s_1} k_{11}\left(s_3,-s_1-s_2-s_3\right)\right)}{4 s_1}+\frac{e^{s_1} \left(e^{s_2} k_{11}\left(s_3,-s_1-s_2-s_3\right)-k_{11}\left(s_2+s_3,-s_1-s_2-s_3\right)\right)}{4 s_2}+\frac{1}{4} e^{s_2+s_3} G_1\left(s_1\right) k_{12}\left(-s_2-s_3,s_2\right)+\frac{e^{s_2+s_3} \left(k_{12}\left(-s_2-s_3,s_2\right)-e^{s_1} k_{12}\left(-s_1-s_2-s_3,s_1+s_2\right)\right)}{4 s_1}+\frac{1}{8} e^{s_2+s_3} G_1\left(s_1\right) k_{13}\left(-s_2-s_3,s_2\right)+\frac{e^{s_2+s_3} \left(k_{13}\left(-s_2-s_3,s_2\right)-e^{s_1} k_{13}\left(-s_1-s_2-s_3,s_1+s_2\right)\right)}{8 s_1}-\frac{1}{8} e^{s_2} G_1\left(s_1\right) k_{13}\left(s_3,-s_2-s_3\right)-\frac{1}{16} e^{s_1} k_{17}\left(s_2,s_3,-s_1-s_2-s_3\right)-\frac{1}{16} e^{s_1+s_2+s_3} k_{17}\left(-s_1-s_2-s_3,s_1,s_2\right)-\frac{1}{16} k_{18}\left(s_1,s_2,s_3\right)-\frac{1}{16} e^{s_1} k_{18}\left(s_2,s_3,-s_1-s_2-s_3\right)-\frac{1}{16} e^{s_1+s_2+s_3} k_{18}\left(-s_1-s_2-s_3,s_1,s_2\right)-\frac{1}{16} e^{s_1+s_2} k_{18}\left(s_3,-s_1-s_2-s_3,s_1\right)-\frac{1}{16} k_{19}\left(s_1,s_2,s_3\right)-\frac{1}{16} e^{s_1+s_2} k_{19}\left(s_3,-s_1-s_2-s_3,s_1\right)-\frac{e^{s_2+s_3} \left(k_9\left(-s_2-s_3,s_2\right)-e^{s_1} k_9\left(-s_1-s_2-s_3,s_1+s_2\right)\right)}{8 s_1}-\frac{e^{s_2} \left(k_{13}\left(s_3,-s_2-s_3\right)-e^{s_1} k_{13}\left(s_3,-s_1-s_2-s_3\right)\right)}{8 s_1}-\frac{G_1\left(s_1\right) \left(k_6\left(s_3\right)-k_6\left(s_2+s_3\right)\right)}{s_2}-\frac{e^{s_3} G_1\left(s_1\right) \left(e^{s_2} k_3\left(-s_2-s_3\right)-k_3\left(-s_3\right)\right)}{4 s_2}-\frac{e^{s_1+s_2+s_3} \left(k_{12}\left(-s_1-s_2-s_3,s_1\right)-k_{12}\left(-s_1-s_2-s_3,s_1+s_2\right)\right)}{4 s_2}-\frac{e^{s_1+s_2+s_3} \left(k_{13}\left(-s_1-s_2-s_3,s_1\right)-k_{13}\left(-s_1-s_2-s_3,s_1+s_2\right)\right)}{8 s_2}-\frac{e^{s_1} \left(e^{s_2} k_{13}\left(s_3,-s_1-s_2-s_3\right)-k_{13}\left(s_2+s_3,-s_1-s_2-s_3\right)\right)}{8 s_2}-\frac{e^{s_1} \left(k_{10}\left(s_2,-s_1-s_2\right)-k_{10}\left(s_2+s_3,-s_1-s_2-s_3\right)\right)}{8 s_3}-\frac{e^{s_1+s_2} \left(k_{13}\left(-s_1-s_2,s_1\right)-e^{s_3} k_{13}\left(-s_1-s_2-s_3,s_1\right)\right)}{8 s_3}-\frac{e^{s_1+s_2+s_3} \left(k_9\left(s_1,s_2\right)-k_9\left(-s_2-s_3,s_2\right)\right)}{8 \left(s_1+s_2+s_3\right)}-\frac{e^{s_1} \left(k_{10}\left(s_2,-s_1-s_2\right)-k_{10}\left(s_2,s_3\right)\right)}{8 \left(s_1+s_2+s_3\right)}-\frac{e^{s_1+s_2} \left(k_{13}\left(-s_1-s_2,s_1\right)-k_{13}\left(s_3,-s_2-s_3\right)\right)}{8 \left(s_1+s_2+s_3\right)}. 
\end{math}
\end{center}

\subsubsection{The function $\widetilde K_{10}$}
\label{basicK10}
We have:

\begin{equation} \label{basicK10eqn}
 \widetilde K_{10}(s_1, s_2, s_3) =
\end{equation}
\begin{center}
\begin{math}
 \frac{1}{15} (-4) \pi  G_3\left(s_1,s_2,s_3\right)+\frac{e^{s_2} \left(e^{s_1} k_3\left(-s_1-s_2\right)-k_3\left(-s_2\right)+e^{s_3} k_3\left(-s_2-s_3\right)-e^{s_1+s_3} k_3\left(-s_1-s_2-s_3\right)\right)}{4 s_1 s_3}+\frac{-k_3\left(s_2\right)+k_3\left(s_1+s_2\right)+k_3\left(s_2+s_3\right)-k_3\left(s_1+s_2+s_3\right)}{4 s_1 s_3}-\frac{1}{2} G_2\left(s_1,s_2\right) k_6\left(s_3\right)+\frac{G_1\left(s_1\right) \left(k_6\left(s_2\right)-k_6\left(s_2+s_3\right)\right)}{s_3}+\frac{k_6\left(s_2\right)-k_6\left(s_1+s_2\right)-k_6\left(s_2+s_3\right)+k_6\left(s_1+s_2+s_3\right)}{s_1 s_3}+\frac{-s_3 k_6\left(s_1\right)+s_2 k_6\left(s_1+s_2\right)+s_3 k_6\left(s_1+s_2\right)-s_2 k_6\left(s_1+s_2+s_3\right)}{2 s_2 s_3 \left(s_2+s_3\right)}+\frac{-s_1 k_6\left(s_3\right)+s_1 k_6\left(s_2+s_3\right)+s_2 k_6\left(s_2+s_3\right)-s_2 k_6\left(s_1+s_2+s_3\right)}{2 s_1 s_2 \left(s_1+s_2\right)}+\frac{e^{s_2} G_1\left(s_1\right) \left(k_7\left(-s_2\right)-e^{s_3} k_7\left(-s_2-s_3\right)\right)}{s_3}-\frac{e^{s_1} \left(s_3 k_7\left(-s_1\right)-e^{s_2} s_2 k_7\left(-s_1-s_2\right)-e^{s_2} s_3 k_7\left(-s_1-s_2\right)+e^{s_2+s_3} s_2 k_7\left(-s_1-s_2-s_3\right)\right)}{2 s_2 s_3 \left(s_2+s_3\right)}-\frac{e^{s_2} \left(e^{s_1} k_7\left(-s_1-s_2\right)-k_7\left(-s_2\right)+e^{s_3} k_7\left(-s_2-s_3\right)-e^{s_1+s_3} k_7\left(-s_1-s_2-s_3\right)\right)}{s_1 s_3}+\frac{e^{s_3} G_1\left(s_1\right) \left(e^{s_2} k_7\left(-s_2-s_3\right)-k_7\left(-s_3\right)\right)}{2 s_2}-\frac{1}{2} e^{s_3} G_2\left(s_1,s_2\right) k_7\left(-s_3\right)+\frac{e^{s_3} \left(e^{s_2} s_1 k_7\left(-s_2-s_3\right)+e^{s_2} s_2 k_7\left(-s_2-s_3\right)-e^{s_1+s_2} s_2 k_7\left(-s_1-s_2-s_3\right)-s_1 k_7\left(-s_3\right)\right)}{2 s_1 s_2 \left(s_1+s_2\right)}+\frac{k_8\left(s_1,s_2+s_3\right)-k_8\left(s_1,s_2\right)}{4 s_3}+\frac{1}{4} G_1\left(s_1\right) k_8\left(s_2,s_3\right)+\frac{k_8\left(s_2,s_3\right)-k_8\left(s_1+s_2,s_3\right)}{4 s_1}+\frac{k_9\left(s_1,s_2+s_3\right)-k_9\left(s_1,s_2\right)}{8 s_3}+\frac{1}{8} G_1\left(s_1\right) k_9\left(s_2,s_3\right)+\frac{k_9\left(s_2,s_3\right)-k_9\left(s_1+s_2,s_3\right)}{8 s_1}+\frac{k_9\left(s_1+s_2,s_3\right)-k_9\left(s_1,s_2+s_3\right)}{8 s_2}+\frac{e^{s_1+s_2} \left(k_9\left(-s_1-s_2,s_1\right)-e^{s_3} k_9\left(-s_1-s_2-s_3,s_1\right)\right)}{8 s_3}-\frac{1}{8} e^{s_2} G_1\left(s_1\right) k_9\left(s_3,-s_2-s_3\right)+\frac{\left(-1+e^{s_1+s_2+s_3}\right) k_{10}\left(s_1,s_2\right)}{8 \left(s_1+s_2+s_3\right)}+\frac{k_{10}\left(s_1,s_2\right)-k_{10}\left(s_1,s_2+s_3\right)}{8 s_3}-\frac{1}{8} e^{s_2+s_3} G_1\left(s_1\right) k_{10}\left(-s_2-s_3,s_2\right)+\frac{1}{8} e^{s_2} G_1\left(s_1\right) k_{10}\left(s_3,-s_2-s_3\right)+\frac{e^{s_2} \left(k_{10}\left(s_3,-s_2-s_3\right)-e^{s_1} k_{10}\left(s_3,-s_1-s_2-s_3\right)\right)}{8 s_1}+\frac{e^{s_1} \left(e^{s_2} k_{10}\left(s_3,-s_1-s_2-s_3\right)-k_{10}\left(s_2+s_3,-s_1-s_2-s_3\right)\right)}{8 s_2}+\frac{1}{4} e^{s_2} G_1\left(s_1\right) k_{11}\left(s_3,-s_2-s_3\right)+\frac{e^{s_2} \left(k_{11}\left(s_3,-s_2-s_3\right)-e^{s_1} k_{11}\left(s_3,-s_1-s_2-s_3\right)\right)}{4 s_1}+\frac{1}{4} e^{s_2+s_3} G_1\left(s_1\right) k_{12}\left(-s_2-s_3,s_2\right)+\frac{e^{s_2+s_3} \left(k_{12}\left(-s_2-s_3,s_2\right)-e^{s_1} k_{12}\left(-s_1-s_2-s_3,s_1+s_2\right)\right)}{4 s_1}-\frac{1}{8} G_1\left(s_1\right) k_{13}\left(s_2,s_3\right)+\frac{k_{13}\left(s_1+s_2,s_3\right)-k_{13}\left(s_2,s_3\right)}{8 s_1}+\frac{1}{8} e^{s_2+s_3} G_1\left(s_1\right) k_{13}\left(-s_2-s_3,s_2\right)+\frac{e^{s_2+s_3} \left(k_{13}\left(-s_2-s_3,s_2\right)-e^{s_1} k_{13}\left(-s_1-s_2-s_3,s_1+s_2\right)\right)}{8 s_1}+\frac{e^{s_1} \left(k_{13}\left(s_2,-s_1-s_2\right)-k_{13}\left(s_2+s_3,-s_1-s_2-s_3\right)\right)}{8 s_3}-\frac{1}{16} k_{17}\left(s_1,s_2,s_3\right)-\frac{1}{16} e^{s_1} k_{17}\left(s_2,s_3,-s_1-s_2-s_3\right)-\frac{1}{16} e^{s_1+s_2+s_3} k_{17}\left(-s_1-s_2-s_3,s_1,s_2\right)-\frac{1}{16} e^{s_1+s_2} k_{17}\left(s_3,-s_1-s_2-s_3,s_1\right)-\frac{1}{16} k_{19}\left(s_1,s_2,s_3\right)-\frac{1}{16} e^{s_1} k_{19}\left(s_2,s_3,-s_1-s_2-s_3\right)-\frac{1}{16} e^{s_1+s_2+s_3} k_{19}\left(-s_1-s_2-s_3,s_1,s_2\right)-\frac{1}{16} e^{s_1+s_2} k_{19}\left(s_3,-s_1-s_2-s_3,s_1\right)-\frac{e^{s_2} \left(k_9\left(s_3,-s_2-s_3\right)-e^{s_1} k_9\left(s_3,-s_1-s_2-s_3\right)\right)}{8 s_1}-\frac{e^{s_2+s_3} \left(k_{10}\left(-s_2-s_3,s_2\right)-e^{s_1} k_{10}\left(-s_1-s_2-s_3,s_1+s_2\right)\right)}{8 s_1}-\frac{G_1\left(s_1\right) \left(k_6\left(s_3\right)-k_6\left(s_2+s_3\right)\right)}{2 s_2}-\frac{e^{s_1+s_2+s_3} \left(k_{13}\left(-s_1-s_2-s_3,s_1\right)-k_{13}\left(-s_1-s_2-s_3,s_1+s_2\right)\right)}{8 s_2}-\frac{e^{s_2} G_1\left(s_1\right) \left(k_3\left(-s_2\right)-e^{s_3} k_3\left(-s_2-s_3\right)\right)}{4 s_3}-\frac{G_1\left(s_1\right) \left(k_3\left(s_2\right)-k_3\left(s_2+s_3\right)\right)}{4 s_3}-\frac{e^{s_1} \left(k_{11}\left(s_2,-s_1-s_2\right)-k_{11}\left(s_2+s_3,-s_1-s_2-s_3\right)\right)}{4 s_3}-\frac{e^{s_1+s_2} \left(k_{12}\left(-s_1-s_2,s_1\right)-e^{s_3} k_{12}\left(-s_1-s_2-s_3,s_1\right)\right)}{4 s_3}-\frac{e^{s_1} \left(k_{10}\left(s_2,-s_1-s_2\right)-k_{10}\left(s_2+s_3,-s_1-s_2-s_3\right)\right)}{8 s_3}-\frac{e^{s_1+s_2} \left(k_{13}\left(-s_1-s_2,s_1\right)-e^{s_3} k_{13}\left(-s_1-s_2-s_3,s_1\right)\right)}{8 s_3}-\frac{e^{s_1+s_2} \left(k_9\left(-s_1-s_2,s_1\right)-k_9\left(s_3,-s_2-s_3\right)\right)}{8 \left(s_1+s_2+s_3\right)}-\frac{e^{s_1+s_2+s_3} \left(k_{10}\left(s_1,s_2\right)-k_{10}\left(-s_2-s_3,s_2\right)\right)}{8 \left(s_1+s_2+s_3\right)}-\frac{e^{s_1} \left(k_{13}\left(s_2,-s_1-s_2\right)-k_{13}\left(s_2,s_3\right)\right)}{8 \left(s_1+s_2+s_3\right)}. 
\end{math}
\end{center}

\subsubsection{The function $\widetilde K_{11}$}
\label{basicK11}
We have: 

\begin{equation} \label{basicK11eqn}
 \widetilde K_{11}(s_1, s_2, s_3) =
\end{equation}
\begin{center}
\begin{math}
 \frac{1}{15} (-2) \pi  G_3\left(s_1,s_2,s_3\right)-\frac{e^{s_2} \left(-e^{s_1} k_4\left(-s_1-s_2\right)+k_4\left(-s_2\right)-e^{s_3} k_4\left(-s_2-s_3\right)+e^{s_1+s_3} k_4\left(-s_1-s_2-s_3\right)\right)}{4 s_1 s_3}+\frac{e^{s_2} \left(e^{s_1} k_4\left(-s_1-s_2\right)-k_4\left(-s_2\right)+e^{s_3} k_4\left(-s_2-s_3\right)-e^{s_1+s_3} k_4\left(-s_1-s_2-s_3\right)\right)}{4 s_1 s_3}+\frac{-k_4\left(s_2\right)+k_4\left(s_1+s_2\right)+k_4\left(s_2+s_3\right)-k_4\left(s_1+s_2+s_3\right)}{2 s_1 s_3}-\frac{1}{4} G_2\left(s_1,s_2\right) k_6\left(s_3\right)+\frac{G_1\left(s_1\right) \left(k_6\left(s_2\right)-k_6\left(s_2+s_3\right)\right)}{2 s_3}+\frac{k_6\left(s_2\right)-k_6\left(s_1+s_2\right)-k_6\left(s_2+s_3\right)+k_6\left(s_1+s_2+s_3\right)}{2 s_1 s_3}+\frac{-s_3 k_6\left(s_1\right)+s_2 k_6\left(s_1+s_2\right)+s_3 k_6\left(s_1+s_2\right)-s_2 k_6\left(s_1+s_2+s_3\right)}{4 s_2 s_3 \left(s_2+s_3\right)}+\frac{-s_1 k_6\left(s_3\right)+s_1 k_6\left(s_2+s_3\right)+s_2 k_6\left(s_2+s_3\right)-s_2 k_6\left(s_1+s_2+s_3\right)}{4 s_1 s_2 \left(s_1+s_2\right)}+\frac{e^{s_2} G_1\left(s_1\right) \left(k_7\left(-s_2\right)-e^{s_3} k_7\left(-s_2-s_3\right)\right)}{2 s_3}+\frac{e^{s_2} \left(-e^{s_1} k_7\left(-s_1-s_2\right)+k_7\left(-s_2\right)-e^{s_3} k_7\left(-s_2-s_3\right)+e^{s_1+s_3} k_7\left(-s_1-s_2-s_3\right)\right)}{4 s_1 s_3}-\frac{e^{s_1} \left(s_3 k_7\left(-s_1\right)-e^{s_2} s_2 k_7\left(-s_1-s_2\right)-e^{s_2} s_3 k_7\left(-s_1-s_2\right)+e^{s_2+s_3} s_2 k_7\left(-s_1-s_2-s_3\right)\right)}{4 s_2 s_3 \left(s_2+s_3\right)}-\frac{e^{s_2} \left(e^{s_1} k_7\left(-s_1-s_2\right)-k_7\left(-s_2\right)+e^{s_3} k_7\left(-s_2-s_3\right)-e^{s_1+s_3} k_7\left(-s_1-s_2-s_3\right)\right)}{4 s_1 s_3}+\frac{e^{s_3} G_1\left(s_1\right) \left(e^{s_2} k_7\left(-s_2-s_3\right)-k_7\left(-s_3\right)\right)}{4 s_2}-\frac{1}{4} e^{s_3} G_2\left(s_1,s_2\right) k_7\left(-s_3\right)+\frac{e^{s_3} \left(e^{s_2} s_1 k_7\left(-s_2-s_3\right)+e^{s_2} s_2 k_7\left(-s_2-s_3\right)-e^{s_1+s_2} s_2 k_7\left(-s_1-s_2-s_3\right)-s_1 k_7\left(-s_3\right)\right)}{4 s_1 s_2 \left(s_1+s_2\right)}+\frac{k_8\left(s_1+s_2,s_3\right)-k_8\left(s_1,s_2+s_3\right)}{8 s_2}+\frac{e^{s_1+s_2} \left(k_8\left(-s_1-s_2,s_1\right)-e^{s_3} k_8\left(-s_1-s_2-s_3,s_1\right)\right)}{8 s_3}-\frac{1}{8} e^{s_2} G_1\left(s_1\right) k_8\left(s_3,-s_2-s_3\right)+\frac{k_9\left(s_1,s_2+s_3\right)-k_9\left(s_1,s_2\right)}{8 s_3}+\frac{1}{8} G_1\left(s_1\right) k_9\left(s_2,s_3\right)+\frac{k_9\left(s_2,s_3\right)-k_9\left(s_1+s_2,s_3\right)}{8 s_1}+\frac{1}{8} e^{s_2} G_1\left(s_1\right) k_{10}\left(s_3,-s_2-s_3\right)+\frac{e^{s_2} \left(k_{10}\left(s_3,-s_2-s_3\right)-e^{s_1} k_{10}\left(s_3,-s_1-s_2-s_3\right)\right)}{8 s_1}+\frac{\left(-1+e^{s_1+s_2+s_3}\right) k_{11}\left(s_1,s_2\right)}{8 \left(s_1+s_2+s_3\right)}+\frac{k_{11}\left(s_1,s_2\right)-k_{11}\left(s_1,s_2+s_3\right)}{8 s_3}-\frac{1}{8} e^{s_2+s_3} G_1\left(s_1\right) k_{11}\left(-s_2-s_3,s_2\right)+\frac{e^{s_1} \left(e^{s_2} k_{11}\left(s_3,-s_1-s_2-s_3\right)-k_{11}\left(s_2+s_3,-s_1-s_2-s_3\right)\right)}{8 s_2}-\frac{1}{8} G_1\left(s_1\right) k_{12}\left(s_2,s_3\right)+\frac{k_{12}\left(s_1+s_2,s_3\right)-k_{12}\left(s_2,s_3\right)}{8 s_1}+\frac{e^{s_1} \left(k_{12}\left(s_2,-s_1-s_2\right)-k_{12}\left(s_2+s_3,-s_1-s_2-s_3\right)\right)}{8 s_3}+\frac{1}{8} e^{s_2+s_3} G_1\left(s_1\right) k_{13}\left(-s_2-s_3,s_2\right)+\frac{e^{s_2+s_3} \left(k_{13}\left(-s_2-s_3,s_2\right)-e^{s_1} k_{13}\left(-s_1-s_2-s_3,s_1+s_2\right)\right)}{8 s_1}-\frac{1}{16} k_{18}\left(s_1,s_2,s_3\right)-\frac{1}{16} e^{s_1} k_{18}\left(s_2,s_3,-s_1-s_2-s_3\right)-\frac{1}{16} e^{s_1+s_2+s_3} k_{18}\left(-s_1-s_2-s_3,s_1,s_2\right)-\frac{1}{16} e^{s_1+s_2} k_{18}\left(s_3,-s_1-s_2-s_3,s_1\right)-\frac{e^{s_2} \left(k_8\left(s_3,-s_2-s_3\right)-e^{s_1} k_8\left(s_3,-s_1-s_2-s_3\right)\right)}{8 s_1}-\frac{e^{s_2+s_3} \left(k_{11}\left(-s_2-s_3,s_2\right)-e^{s_1} k_{11}\left(-s_1-s_2-s_3,s_1+s_2\right)\right)}{8 s_1}-\frac{G_1\left(s_1\right) \left(k_6\left(s_3\right)-k_6\left(s_2+s_3\right)\right)}{4 s_2}-\frac{e^{s_1+s_2+s_3} \left(k_{12}\left(-s_1-s_2-s_3,s_1\right)-k_{12}\left(-s_1-s_2-s_3,s_1+s_2\right)\right)}{8 s_2}-\frac{e^{s_2} G_1\left(s_1\right) \left(k_4\left(-s_2\right)-e^{s_3} k_4\left(-s_2-s_3\right)\right)}{2 s_3}-\frac{G_1\left(s_1\right) \left(k_4\left(s_2\right)-k_4\left(s_2+s_3\right)\right)}{2 s_3}-\frac{e^{s_1} \left(k_{10}\left(s_2,-s_1-s_2\right)-k_{10}\left(s_2+s_3,-s_1-s_2-s_3\right)\right)}{8 s_3}-\frac{e^{s_1+s_2} \left(k_{13}\left(-s_1-s_2,s_1\right)-e^{s_3} k_{13}\left(-s_1-s_2-s_3,s_1\right)\right)}{8 s_3}-\frac{e^{s_1+s_2} \left(k_8\left(-s_1-s_2,s_1\right)-k_8\left(s_3,-s_2-s_3\right)\right)}{8 \left(s_1+s_2+s_3\right)}-\frac{e^{s_1+s_2+s_3} \left(k_{11}\left(s_1,s_2\right)-k_{11}\left(-s_2-s_3,s_2\right)\right)}{8 \left(s_1+s_2+s_3\right)}-\frac{e^{s_1} \left(k_{12}\left(s_2,-s_1-s_2\right)-k_{12}\left(s_2,s_3\right)\right)}{8 \left(s_1+s_2+s_3\right)}. 
\end{math}
\end{center}

\subsubsection{The function $\widetilde K_{12}$}
\label{basicK12}
We have: 

\begin{equation} \label{basicK12eqn}
 \widetilde K_{12}(s_1, s_2, s_3) =
\end{equation}

\begin{center}
\begin{math}
 \frac{1}{15} (-2) \pi  G_3\left(s_1,s_2,s_3\right)+\frac{e^{s_1} \left(s_3 k_4\left(-s_1\right)-e^{s_2} s_2 k_4\left(-s_1-s_2\right)-e^{s_2} s_3 k_4\left(-s_1-s_2\right)+e^{s_2+s_3} s_2 k_4\left(-s_1-s_2-s_3\right)\right)}{2 s_2 s_3 \left(s_2+s_3\right)}+\frac{s_3 k_4\left(s_1\right)-s_2 k_4\left(s_1+s_2\right)-s_3 k_4\left(s_1+s_2\right)+s_2 k_4\left(s_1+s_2+s_3\right)}{2 s_2 s_3 \left(s_2+s_3\right)}-\frac{1}{4} G_2\left(s_1,s_2\right) k_6\left(s_3\right)+\frac{G_1\left(s_1\right) \left(k_6\left(s_2\right)-k_6\left(s_2+s_3\right)\right)}{4 s_3}+\frac{k_6\left(s_2\right)-k_6\left(s_1+s_2\right)-k_6\left(s_2+s_3\right)+k_6\left(s_1+s_2+s_3\right)}{4 s_1 s_3}+\frac{-s_3 k_6\left(s_1\right)+s_2 k_6\left(s_1+s_2\right)+s_3 k_6\left(s_1+s_2\right)-s_2 k_6\left(s_1+s_2+s_3\right)}{2 s_2 s_3 \left(s_2+s_3\right)}+\frac{-s_1 k_6\left(s_3\right)+s_1 k_6\left(s_2+s_3\right)+s_2 k_6\left(s_2+s_3\right)-s_2 k_6\left(s_1+s_2+s_3\right)}{4 s_1 s_2 \left(s_1+s_2\right)}+\frac{e^{s_2} G_1\left(s_1\right) \left(k_7\left(-s_2\right)-e^{s_3} k_7\left(-s_2-s_3\right)\right)}{4 s_3}-\frac{e^{s_1} \left(s_3 k_7\left(-s_1\right)-e^{s_2} s_2 k_7\left(-s_1-s_2\right)-e^{s_2} s_3 k_7\left(-s_1-s_2\right)+e^{s_2+s_3} s_2 k_7\left(-s_1-s_2-s_3\right)\right)}{2 s_2 s_3 \left(s_2+s_3\right)}-\frac{e^{s_2} \left(e^{s_1} k_7\left(-s_1-s_2\right)-k_7\left(-s_2\right)+e^{s_3} k_7\left(-s_2-s_3\right)-e^{s_1+s_3} k_7\left(-s_1-s_2-s_3\right)\right)}{4 s_1 s_3}+\frac{e^{s_3} G_1\left(s_1\right) \left(e^{s_2} k_7\left(-s_2-s_3\right)-k_7\left(-s_3\right)\right)}{4 s_2}-\frac{1}{4} e^{s_3} G_2\left(s_1,s_2\right) k_7\left(-s_3\right)+\frac{e^{s_3} \left(e^{s_2} s_1 k_7\left(-s_2-s_3\right)+e^{s_2} s_2 k_7\left(-s_2-s_3\right)-e^{s_1+s_2} s_2 k_7\left(-s_1-s_2-s_3\right)-s_1 k_7\left(-s_3\right)\right)}{4 s_1 s_2 \left(s_1+s_2\right)}+\frac{1}{8} G_1\left(s_1\right) k_8\left(s_2,s_3\right)+\frac{k_8\left(s_2,s_3\right)-k_8\left(s_1+s_2,s_3\right)}{8 s_1}+\frac{e^{s_1} \left(k_8\left(s_2,-s_1-s_2\right)-k_8\left(s_2+s_3,-s_1-s_2-s_3\right)\right)}{8 s_3}+\frac{k_9\left(s_1,s_2+s_3\right)-k_9\left(s_1,s_2\right)}{8 s_3}+\frac{k_9\left(s_1+s_2,s_3\right)-k_9\left(s_1,s_2+s_3\right)}{8 s_2}+\frac{e^{s_1} \left(e^{s_2} k_{10}\left(s_3,-s_1-s_2-s_3\right)-k_{10}\left(s_2+s_3,-s_1-s_2-s_3\right)\right)}{8 s_2}+\frac{e^{s_1+s_2} \left(k_{11}\left(-s_1-s_2,s_1\right)-e^{s_3} k_{11}\left(-s_1-s_2-s_3,s_1\right)\right)}{8 s_3}+\frac{e^{s_1+s_2+s_3} \left(k_{11}\left(-s_1-s_2-s_3,s_1\right)-k_{11}\left(-s_1-s_2-s_3,s_1+s_2\right)\right)}{8 s_2}+\frac{1}{8} e^{s_2} G_1\left(s_1\right) k_{11}\left(s_3,-s_2-s_3\right)+\frac{e^{s_2} \left(k_{11}\left(s_3,-s_2-s_3\right)-e^{s_1} k_{11}\left(s_3,-s_1-s_2-s_3\right)\right)}{8 s_1}+\frac{\left(-1+e^{s_1+s_2+s_3}\right) k_{12}\left(s_1,s_2\right)}{8 \left(s_1+s_2+s_3\right)}+\frac{k_{12}\left(s_1,s_2\right)-k_{12}\left(s_1,s_2+s_3\right)}{8 s_3}+\frac{k_{12}\left(s_1,s_2+s_3\right)-k_{12}\left(s_1+s_2,s_3\right)}{8 s_2}+\frac{1}{8} e^{s_2+s_3} G_1\left(s_1\right) k_{12}\left(-s_2-s_3,s_2\right)+\frac{e^{s_2+s_3} \left(k_{12}\left(-s_2-s_3,s_2\right)-e^{s_1} k_{12}\left(-s_1-s_2-s_3,s_1+s_2\right)\right)}{8 s_1}-\frac{1}{16} e^{s_1} k_{17}\left(s_2,s_3,-s_1-s_2-s_3\right)-\frac{1}{16} e^{s_1+s_2+s_3} k_{17}\left(-s_1-s_2-s_3,s_1,s_2\right)-\frac{1}{16} k_{19}\left(s_1,s_2,s_3\right)-\frac{1}{16} e^{s_1+s_2} k_{19}\left(s_3,-s_1-s_2-s_3,s_1\right)-\frac{G_1\left(s_1\right) \left(k_6\left(s_3\right)-k_6\left(s_2+s_3\right)\right)}{4 s_2}-\frac{e^{s_1} \left(e^{s_2} k_8\left(s_3,-s_1-s_2-s_3\right)-k_8\left(s_2+s_3,-s_1-s_2-s_3\right)\right)}{8 s_2}-\frac{e^{s_1+s_2+s_3} \left(k_{13}\left(-s_1-s_2-s_3,s_1\right)-k_{13}\left(-s_1-s_2-s_3,s_1+s_2\right)\right)}{8 s_2}-\frac{e^{s_1} \left(k_{10}\left(s_2,-s_1-s_2\right)-k_{10}\left(s_2+s_3,-s_1-s_2-s_3\right)\right)}{8 s_3}-\frac{e^{s_1+s_2} \left(k_{13}\left(-s_1-s_2,s_1\right)-e^{s_3} k_{13}\left(-s_1-s_2-s_3,s_1\right)\right)}{8 s_3}-\frac{e^{s_1} \left(k_8\left(s_2,-s_1-s_2\right)-k_8\left(s_2,s_3\right)\right)}{8 \left(s_1+s_2+s_3\right)}-\frac{e^{s_1+s_2} \left(k_{11}\left(-s_1-s_2,s_1\right)-k_{11}\left(s_3,-s_2-s_3\right)\right)}{8 \left(s_1+s_2+s_3\right)}-\frac{e^{s_1+s_2+s_3} \left(k_{12}\left(s_1,s_2\right)-k_{12}\left(-s_2-s_3,s_2\right)\right)}{8 \left(s_1+s_2+s_3\right)}. 
\end{math}
\end{center}

\subsubsection{The function $\widetilde K_{13}$}
\label{basicK13}
We have:

\begin{equation} \label{basicK13eqn}
\widetilde K_{13}(s_1, s_2, s_3) =
\end{equation}

\begin{center}
\begin{math}
  \frac{1}{15} (-4) \pi  G_3\left(s_1,s_2,s_3\right)+\frac{e^{s_1} \left(s_3 k_3\left(-s_1\right)-e^{s_2} s_2 k_3\left(-s_1-s_2\right)-e^{s_2} s_3 k_3\left(-s_1-s_2\right)+e^{s_2+s_3} s_2 k_3\left(-s_1-s_2-s_3\right)\right)}{4 s_2 s_3 \left(s_2+s_3\right)}+\frac{s_3 k_3\left(s_1\right)-s_2 k_3\left(s_1+s_2\right)-s_3 k_3\left(s_1+s_2\right)+s_2 k_3\left(s_1+s_2+s_3\right)}{4 s_2 s_3 \left(s_2+s_3\right)}-\frac{1}{2} G_2\left(s_1,s_2\right) k_6\left(s_3\right)+\frac{G_1\left(s_1\right) \left(k_6\left(s_2\right)-k_6\left(s_2+s_3\right)\right)}{2 s_3}+\frac{k_6\left(s_2\right)-k_6\left(s_1+s_2\right)-k_6\left(s_2+s_3\right)+k_6\left(s_1+s_2+s_3\right)}{2 s_1 s_3}+\frac{-s_3 k_6\left(s_1\right)+s_2 k_6\left(s_1+s_2\right)+s_3 k_6\left(s_1+s_2\right)-s_2 k_6\left(s_1+s_2+s_3\right)}{s_2 s_3 \left(s_2+s_3\right)}+\frac{-s_1 k_6\left(s_3\right)+s_1 k_6\left(s_2+s_3\right)+s_2 k_6\left(s_2+s_3\right)-s_2 k_6\left(s_1+s_2+s_3\right)}{2 s_1 s_2 \left(s_1+s_2\right)}+\frac{e^{s_2} G_1\left(s_1\right) \left(k_7\left(-s_2\right)-e^{s_3} k_7\left(-s_2-s_3\right)\right)}{2 s_3}-\frac{e^{s_1} \left(s_3 k_7\left(-s_1\right)-e^{s_2} s_2 k_7\left(-s_1-s_2\right)-e^{s_2} s_3 k_7\left(-s_1-s_2\right)+e^{s_2+s_3} s_2 k_7\left(-s_1-s_2-s_3\right)\right)}{s_2 s_3 \left(s_2+s_3\right)}-\frac{e^{s_2} \left(e^{s_1} k_7\left(-s_1-s_2\right)-k_7\left(-s_2\right)+e^{s_3} k_7\left(-s_2-s_3\right)-e^{s_1+s_3} k_7\left(-s_1-s_2-s_3\right)\right)}{2 s_1 s_3}+\frac{e^{s_3} G_1\left(s_1\right) \left(e^{s_2} k_7\left(-s_2-s_3\right)-k_7\left(-s_3\right)\right)}{2 s_2}-\frac{1}{2} e^{s_3} G_2\left(s_1,s_2\right) k_7\left(-s_3\right)+\frac{e^{s_3} \left(e^{s_2} s_1 k_7\left(-s_2-s_3\right)+e^{s_2} s_2 k_7\left(-s_2-s_3\right)-e^{s_1+s_2} s_2 k_7\left(-s_1-s_2-s_3\right)-s_1 k_7\left(-s_3\right)\right)}{2 s_1 s_2 \left(s_1+s_2\right)}+\frac{k_8\left(s_1,s_2+s_3\right)-k_8\left(s_1,s_2\right)}{4 s_3}+\frac{k_8\left(s_1+s_2,s_3\right)-k_8\left(s_1,s_2+s_3\right)}{4 s_2}+\frac{k_9\left(s_1,s_2+s_3\right)-k_9\left(s_1,s_2\right)}{8 s_3}+\frac{1}{8} G_1\left(s_1\right) k_9\left(s_2,s_3\right)+\frac{k_9\left(s_2,s_3\right)-k_9\left(s_1+s_2,s_3\right)}{8 s_1}+\frac{k_9\left(s_1+s_2,s_3\right)-k_9\left(s_1,s_2+s_3\right)}{8 s_2}+\frac{e^{s_1} \left(k_9\left(s_2,-s_1-s_2\right)-k_9\left(s_2+s_3,-s_1-s_2-s_3\right)\right)}{8 s_3}+\frac{e^{s_1+s_2} \left(k_{10}\left(-s_1-s_2,s_1\right)-e^{s_3} k_{10}\left(-s_1-s_2-s_3,s_1\right)\right)}{8 s_3}+\frac{e^{s_1+s_2+s_3} \left(k_{10}\left(-s_1-s_2-s_3,s_1\right)-k_{10}\left(-s_1-s_2-s_3,s_1+s_2\right)\right)}{8 s_2}+\frac{1}{8} e^{s_2} G_1\left(s_1\right) k_{10}\left(s_3,-s_2-s_3\right)+\frac{e^{s_2} \left(k_{10}\left(s_3,-s_2-s_3\right)-e^{s_1} k_{10}\left(s_3,-s_1-s_2-s_3\right)\right)}{8 s_1}+\frac{e^{s_1} \left(e^{s_2} k_{10}\left(s_3,-s_1-s_2-s_3\right)-k_{10}\left(s_2+s_3,-s_1-s_2-s_3\right)\right)}{8 s_2}+\frac{e^{s_1} \left(e^{s_2} k_{11}\left(s_3,-s_1-s_2-s_3\right)-k_{11}\left(s_2+s_3,-s_1-s_2-s_3\right)\right)}{4 s_2}+\frac{\left(-1+e^{s_1+s_2+s_3}\right) k_{13}\left(s_1,s_2\right)}{8 \left(s_1+s_2+s_3\right)}+\frac{k_{13}\left(s_1,s_2\right)-k_{13}\left(s_1,s_2+s_3\right)}{8 s_3}+\frac{k_{13}\left(s_1,s_2+s_3\right)-k_{13}\left(s_1+s_2,s_3\right)}{8 s_2}+\frac{1}{8} e^{s_2+s_3} G_1\left(s_1\right) k_{13}\left(-s_2-s_3,s_2\right)+\frac{e^{s_2+s_3} \left(k_{13}\left(-s_2-s_3,s_2\right)-e^{s_1} k_{13}\left(-s_1-s_2-s_3,s_1+s_2\right)\right)}{8 s_1}-\frac{1}{16} k_{17}\left(s_1,s_2,s_3\right)-\frac{1}{16} e^{s_1+s_2} k_{17}\left(s_3,-s_1-s_2-s_3,s_1\right)-\frac{1}{16} k_{18}\left(s_1,s_2,s_3\right)-\frac{1}{16} e^{s_1} k_{18}\left(s_2,s_3,-s_1-s_2-s_3\right)-\frac{1}{16} e^{s_1+s_2+s_3} k_{18}\left(-s_1-s_2-s_3,s_1,s_2\right)-\frac{1}{16} e^{s_1+s_2} k_{18}\left(s_3,-s_1-s_2-s_3,s_1\right)-\frac{1}{16} e^{s_1} k_{19}\left(s_2,s_3,-s_1-s_2-s_3\right)-\frac{1}{16} e^{s_1+s_2+s_3} k_{19}\left(-s_1-s_2-s_3,s_1,s_2\right)-\frac{G_1\left(s_1\right) \left(k_6\left(s_3\right)-k_6\left(s_2+s_3\right)\right)}{2 s_2}-\frac{e^{s_1+s_2+s_3} \left(k_{12}\left(-s_1-s_2-s_3,s_1\right)-k_{12}\left(-s_1-s_2-s_3,s_1+s_2\right)\right)}{4 s_2}-\frac{e^{s_1} \left(e^{s_2} k_9\left(s_3,-s_1-s_2-s_3\right)-k_9\left(s_2+s_3,-s_1-s_2-s_3\right)\right)}{8 s_2}-\frac{e^{s_1+s_2+s_3} \left(k_{13}\left(-s_1-s_2-s_3,s_1\right)-k_{13}\left(-s_1-s_2-s_3,s_1+s_2\right)\right)}{8 s_2}-\frac{e^{s_1} \left(k_{11}\left(s_2,-s_1-s_2\right)-k_{11}\left(s_2+s_3,-s_1-s_2-s_3\right)\right)}{4 s_3}-\frac{e^{s_1+s_2} \left(k_{12}\left(-s_1-s_2,s_1\right)-e^{s_3} k_{12}\left(-s_1-s_2-s_3,s_1\right)\right)}{4 s_3}-\frac{e^{s_1} \left(k_{10}\left(s_2,-s_1-s_2\right)-k_{10}\left(s_2+s_3,-s_1-s_2-s_3\right)\right)}{8 s_3}-\frac{e^{s_1+s_2} \left(k_{13}\left(-s_1-s_2,s_1\right)-e^{s_3} k_{13}\left(-s_1-s_2-s_3,s_1\right)\right)}{8 s_3}-\frac{e^{s_1} \left(k_9\left(s_2,-s_1-s_2\right)-k_9\left(s_2,s_3\right)\right)}{8 \left(s_1+s_2+s_3\right)}-\frac{e^{s_1+s_2} \left(k_{10}\left(-s_1-s_2,s_1\right)-k_{10}\left(s_3,-s_2-s_3\right)\right)}{8 \left(s_1+s_2+s_3\right)}-\frac{e^{s_1+s_2+s_3} \left(k_{13}\left(s_1,s_2\right)-k_{13}\left(-s_2-s_3,s_2\right)\right)}{8 \left(s_1+s_2+s_3\right)}. 
\end{math}
\end{center}

\subsubsection{The function $\widetilde K_{14}$}
\label{basicK14}
We have: 

\begin{equation} \label{basicK14eqn}
 \widetilde K_{14}(s_1, s_2, s_3) =
\end{equation}

\begin{center}
\begin{math}
 \frac{1}{5} (-2) \pi  G_3\left(s_1,s_2,s_3\right)+\frac{e^{s_1} \left(s_3 k_5\left(-s_1\right)-e^{s_2} s_2 k_5\left(-s_1-s_2\right)-e^{s_2} s_3 k_5\left(-s_1-s_2\right)+e^{s_2+s_3} s_2 k_5\left(-s_1-s_2-s_3\right)\right)}{2 s_2 s_3 \left(s_2+s_3\right)}+\frac{s_3 k_5\left(s_1\right)-s_2 k_5\left(s_1+s_2\right)-s_3 k_5\left(s_1+s_2\right)+s_2 k_5\left(s_1+s_2+s_3\right)}{2 s_2 s_3 \left(s_2+s_3\right)}-\frac{3}{4} G_2\left(s_1,s_2\right) k_6\left(s_3\right)+\frac{3 G_1\left(s_1\right) \left(k_6\left(s_2\right)-k_6\left(s_2+s_3\right)\right)}{4 s_3}+\frac{3 \left(k_6\left(s_2\right)-k_6\left(s_1+s_2\right)-k_6\left(s_2+s_3\right)+k_6\left(s_1+s_2+s_3\right)\right)}{4 s_1 s_3}+\frac{3 \left(-s_1 k_6\left(s_3\right)+s_1 k_6\left(s_2+s_3\right)+s_2 k_6\left(s_2+s_3\right)-s_2 k_6\left(s_1+s_2+s_3\right)\right)}{4 s_1 s_2 \left(s_1+s_2\right)}+\frac{3 e^{s_2} G_1\left(s_1\right) \left(k_7\left(-s_2\right)-e^{s_3} k_7\left(-s_2-s_3\right)\right)}{4 s_3}+\frac{e^{s_2} \left(-e^{s_1} k_7\left(-s_1-s_2\right)+k_7\left(-s_2\right)-e^{s_3} k_7\left(-s_2-s_3\right)+e^{s_1+s_3} k_7\left(-s_1-s_2-s_3\right)\right)}{4 s_1 s_3}-\frac{3 e^{s_1} \left(s_3 k_7\left(-s_1\right)-e^{s_2} s_2 k_7\left(-s_1-s_2\right)-e^{s_2} s_3 k_7\left(-s_1-s_2\right)+e^{s_2+s_3} s_2 k_7\left(-s_1-s_2-s_3\right)\right)}{2 s_2 s_3 \left(s_2+s_3\right)}-\frac{e^{s_2} \left(e^{s_1} k_7\left(-s_1-s_2\right)-k_7\left(-s_2\right)+e^{s_3} k_7\left(-s_2-s_3\right)-e^{s_1+s_3} k_7\left(-s_1-s_2-s_3\right)\right)}{2 s_1 s_3}+\frac{3 e^{s_3} G_1\left(s_1\right) \left(e^{s_2} k_7\left(-s_2-s_3\right)-k_7\left(-s_3\right)\right)}{4 s_2}-\frac{3}{4} e^{s_3} G_2\left(s_1,s_2\right) k_7\left(-s_3\right)+\frac{3 e^{s_3} \left(e^{s_2} s_1 k_7\left(-s_2-s_3\right)+e^{s_2} s_2 k_7\left(-s_2-s_3\right)-e^{s_1+s_2} s_2 k_7\left(-s_1-s_2-s_3\right)-s_1 k_7\left(-s_3\right)\right)}{4 s_1 s_2 \left(s_1+s_2\right)}+\frac{\left(-1+e^{s_1+s_2+s_3}\right) k_{14}\left(s_1,s_2\right)}{8 \left(s_1+s_2+s_3\right)}+\frac{k_{14}\left(s_1,s_2\right)-k_{14}\left(s_1,s_2+s_3\right)}{8 s_3}+\frac{k_{14}\left(s_1,s_2+s_3\right)-k_{14}\left(s_1+s_2,s_3\right)}{8 s_2}+\frac{1}{8} e^{s_2+s_3} G_1\left(s_1\right) k_{14}\left(-s_2-s_3,s_2\right)+\frac{e^{s_2+s_3} \left(k_{14}\left(-s_2-s_3,s_2\right)-e^{s_1} k_{14}\left(-s_1-s_2-s_3,s_1+s_2\right)\right)}{8 s_1}+\frac{k_{15}\left(s_1,s_2+s_3\right)-k_{15}\left(s_1,s_2\right)}{4 s_3}+\frac{1}{8} G_1\left(s_1\right) k_{15}\left(s_2,s_3\right)+\frac{k_{15}\left(s_2,s_3\right)-k_{15}\left(s_1+s_2,s_3\right)}{8 s_1}+\frac{k_{15}\left(s_1+s_2,s_3\right)-k_{15}\left(s_1,s_2+s_3\right)}{4 s_2}+\frac{e^{s_1} \left(k_{15}\left(s_2,-s_1-s_2\right)-k_{15}\left(s_2+s_3,-s_1-s_2-s_3\right)\right)}{8 s_3}+\frac{e^{s_1+s_2} \left(k_{16}\left(-s_1-s_2,s_1\right)-e^{s_3} k_{16}\left(-s_1-s_2-s_3,s_1\right)\right)}{8 s_3}+\frac{e^{s_1+s_2+s_3} \left(k_{16}\left(-s_1-s_2-s_3,s_1\right)-k_{16}\left(-s_1-s_2-s_3,s_1+s_2\right)\right)}{8 s_2}+\frac{1}{8} e^{s_2} G_1\left(s_1\right) k_{16}\left(s_3,-s_2-s_3\right)+\frac{e^{s_2} \left(k_{16}\left(s_3,-s_2-s_3\right)-e^{s_1} k_{16}\left(s_3,-s_1-s_2-s_3\right)\right)}{8 s_1}+\frac{e^{s_1} \left(e^{s_2} k_{16}\left(s_3,-s_1-s_2-s_3\right)-k_{16}\left(s_2+s_3,-s_1-s_2-s_3\right)\right)}{4 s_2}-\frac{1}{16} k_{20}\left(s_1,s_2,s_3\right)-\frac{1}{16} e^{s_1} k_{20}\left(s_2,s_3,-s_1-s_2-s_3\right)-\frac{1}{16} e^{s_1+s_2+s_3} k_{20}\left(-s_1-s_2-s_3,s_1,s_2\right)-\frac{1}{16} e^{s_1+s_2} k_{20}\left(s_3,-s_1-s_2-s_3,s_1\right)-\frac{3 G_1\left(s_1\right) \left(k_6\left(s_3\right)-k_6\left(s_2+s_3\right)\right)}{4 s_2}-\frac{e^{s_1+s_2+s_3} \left(k_{14}\left(-s_1-s_2-s_3,s_1\right)-k_{14}\left(-s_1-s_2-s_3,s_1+s_2\right)\right)}{4 s_2}-\frac{e^{s_1} \left(e^{s_2} k_{15}\left(s_3,-s_1-s_2-s_3\right)-k_{15}\left(s_2+s_3,-s_1-s_2-s_3\right)\right)}{8 s_2}-\frac{e^{s_1+s_2} \left(k_{14}\left(-s_1-s_2,s_1\right)-e^{s_3} k_{14}\left(-s_1-s_2-s_3,s_1\right)\right)}{4 s_3}-\frac{e^{s_1} \left(k_{16}\left(s_2,-s_1-s_2\right)-k_{16}\left(s_2+s_3,-s_1-s_2-s_3\right)\right)}{4 s_3}-\frac{3 \left(s_3 k_6\left(s_1\right)-s_2 k_6\left(s_1+s_2\right)-s_3 k_6\left(s_1+s_2\right)+s_2 k_6\left(s_1+s_2+s_3\right)\right)}{2 s_2 s_3 \left(s_2+s_3\right)}-\frac{e^{s_1+s_2+s_3} \left(k_{14}\left(s_1,s_2\right)-k_{14}\left(-s_2-s_3,s_2\right)\right)}{8 \left(s_1+s_2+s_3\right)}-\frac{e^{s_1} \left(k_{15}\left(s_2,-s_1-s_2\right)-k_{15}\left(s_2,s_3\right)\right)}{8 \left(s_1+s_2+s_3\right)}-\frac{e^{s_1+s_2} \left(k_{16}\left(-s_1-s_2,s_1\right)-k_{16}\left(s_3,-s_2-s_3\right)\right)}{8 \left(s_1+s_2+s_3\right)}. 
\end{math}
\end{center}

\subsubsection{The function $\widetilde K_{15}$}
\label{basicK15}
We have: 

\begin{equation} \label{basicK15eqn}
\widetilde K_{15}(s_1, s_2, s_3) =
\end{equation}

\begin{center}
\begin{math}
  \frac{1}{5} (-2) \pi  G_3\left(s_1,s_2,s_3\right)+\frac{1}{2} e^{s_3} G_2\left(s_1,s_2\right) k_5\left(-s_3\right)-\frac{e^{s_3} \left(e^{s_2} s_1 k_5\left(-s_2-s_3\right)+e^{s_2} s_2 k_5\left(-s_2-s_3\right)-e^{s_1+s_2} s_2 k_5\left(-s_1-s_2-s_3\right)-s_1 k_5\left(-s_3\right)\right)}{2 s_1 s_2 \left(s_1+s_2\right)}+\frac{1}{2} G_2\left(s_1,s_2\right) k_5\left(s_3\right)+\frac{G_1\left(s_1\right) \left(k_5\left(s_3\right)-k_5\left(s_2+s_3\right)\right)}{2 s_2}+\frac{s_1 k_5\left(s_3\right)-s_1 k_5\left(s_2+s_3\right)-s_2 k_5\left(s_2+s_3\right)+s_2 k_5\left(s_1+s_2+s_3\right)}{2 s_1 s_2 \left(s_1+s_2\right)}-\frac{3}{2} G_2\left(s_1,s_2\right) k_6\left(s_3\right)+\frac{3 G_1\left(s_1\right) \left(k_6\left(s_2\right)-k_6\left(s_2+s_3\right)\right)}{4 s_3}+\frac{3 \left(k_6\left(s_2\right)-k_6\left(s_1+s_2\right)-k_6\left(s_2+s_3\right)+k_6\left(s_1+s_2+s_3\right)\right)}{4 s_1 s_3}+\frac{3 \left(-s_3 k_6\left(s_1\right)+s_2 k_6\left(s_1+s_2\right)+s_3 k_6\left(s_1+s_2\right)-s_2 k_6\left(s_1+s_2+s_3\right)\right)}{4 s_2 s_3 \left(s_2+s_3\right)}+\frac{3 e^{s_2} G_1\left(s_1\right) \left(k_7\left(-s_2\right)-e^{s_3} k_7\left(-s_2-s_3\right)\right)}{4 s_3}+\frac{e^{s_2} \left(-e^{s_1} k_7\left(-s_1-s_2\right)+k_7\left(-s_2\right)-e^{s_3} k_7\left(-s_2-s_3\right)+e^{s_1+s_3} k_7\left(-s_1-s_2-s_3\right)\right)}{2 s_1 s_3}-\frac{3 e^{s_1} \left(s_3 k_7\left(-s_1\right)-e^{s_2} s_2 k_7\left(-s_1-s_2\right)-e^{s_2} s_3 k_7\left(-s_1-s_2\right)+e^{s_2+s_3} s_2 k_7\left(-s_1-s_2-s_3\right)\right)}{4 s_2 s_3 \left(s_2+s_3\right)}-\frac{e^{s_2} \left(e^{s_1} k_7\left(-s_1-s_2\right)-k_7\left(-s_2\right)+e^{s_3} k_7\left(-s_2-s_3\right)-e^{s_1+s_3} k_7\left(-s_1-s_2-s_3\right)\right)}{4 s_1 s_3}+\frac{3 e^{s_3} G_1\left(s_1\right) \left(e^{s_2} k_7\left(-s_2-s_3\right)-k_7\left(-s_3\right)\right)}{2 s_2}-\frac{3}{2} e^{s_3} G_2\left(s_1,s_2\right) k_7\left(-s_3\right)+\frac{3 e^{s_3} \left(e^{s_2} s_1 k_7\left(-s_2-s_3\right)+e^{s_2} s_2 k_7\left(-s_2-s_3\right)-e^{s_1+s_2} s_2 k_7\left(-s_1-s_2-s_3\right)-s_1 k_7\left(-s_3\right)\right)}{2 s_1 s_2 \left(s_1+s_2\right)}+\frac{1}{4} e^{s_2+s_3} G_1\left(s_1\right) k_{14}\left(-s_2-s_3,s_2\right)+\frac{e^{s_2+s_3} \left(k_{14}\left(-s_2-s_3,s_2\right)-e^{s_1} k_{14}\left(-s_1-s_2-s_3,s_1+s_2\right)\right)}{4 s_1}-\frac{1}{8} e^{s_2} G_1\left(s_1\right) k_{14}\left(s_3,-s_2-s_3\right)+\frac{\left(-1+e^{s_1+s_2+s_3}\right) k_{15}\left(s_1,s_2\right)}{8 \left(s_1+s_2+s_3\right)}+\frac{k_{15}\left(s_1,s_2+s_3\right)-k_{15}\left(s_1,s_2\right)}{8 s_3}+\frac{1}{4} G_1\left(s_1\right) k_{15}\left(s_2,s_3\right)+\frac{k_{15}\left(s_2,s_3\right)-k_{15}\left(s_1+s_2,s_3\right)}{4 s_1}+\frac{k_{15}\left(s_1+s_2,s_3\right)-k_{15}\left(s_1,s_2+s_3\right)}{4 s_2}-\frac{1}{8} e^{s_2+s_3} G_1\left(s_1\right) k_{15}\left(-s_2-s_3,s_2\right)+\frac{e^{s_1+s_2+s_3} \left(k_{15}\left(-s_1-s_2-s_3,s_1\right)-k_{15}\left(-s_1-s_2-s_3,s_1+s_2\right)\right)}{8 s_2}-\frac{1}{8} G_1\left(s_1\right) k_{16}\left(s_2,s_3\right)+\frac{k_{16}\left(s_1,s_2+s_3\right)-k_{16}\left(s_1+s_2,s_3\right)}{8 s_2}+\frac{k_{16}\left(s_1+s_2,s_3\right)-k_{16}\left(s_2,s_3\right)}{8 s_1}+\frac{1}{4} e^{s_2} G_1\left(s_1\right) k_{16}\left(s_3,-s_2-s_3\right)+\frac{e^{s_2} \left(k_{16}\left(s_3,-s_2-s_3\right)-e^{s_1} k_{16}\left(s_3,-s_1-s_2-s_3\right)\right)}{4 s_1}+\frac{e^{s_1} \left(e^{s_2} k_{16}\left(s_3,-s_1-s_2-s_3\right)-k_{16}\left(s_2+s_3,-s_1-s_2-s_3\right)\right)}{4 s_2}-\frac{1}{16} k_{20}\left(s_1,s_2,s_3\right)-\frac{1}{16} e^{s_1} k_{20}\left(s_2,s_3,-s_1-s_2-s_3\right)-\frac{1}{16} e^{s_1+s_2+s_3} k_{20}\left(-s_1-s_2-s_3,s_1,s_2\right)-\frac{1}{16} e^{s_1+s_2} k_{20}\left(s_3,-s_1-s_2-s_3,s_1\right)-\frac{e^{s_2} \left(k_{14}\left(s_3,-s_2-s_3\right)-e^{s_1} k_{14}\left(s_3,-s_1-s_2-s_3\right)\right)}{8 s_1}-\frac{e^{s_2+s_3} \left(k_{15}\left(-s_2-s_3,s_2\right)-e^{s_1} k_{15}\left(-s_1-s_2-s_3,s_1+s_2\right)\right)}{8 s_1}-\frac{e^{s_3} G_1\left(s_1\right) \left(e^{s_2} k_5\left(-s_2-s_3\right)-k_5\left(-s_3\right)\right)}{2 s_2}-\frac{3 G_1\left(s_1\right) \left(k_6\left(s_3\right)-k_6\left(s_2+s_3\right)\right)}{2 s_2}-\frac{e^{s_1+s_2+s_3} \left(k_{14}\left(-s_1-s_2-s_3,s_1\right)-k_{14}\left(-s_1-s_2-s_3,s_1+s_2\right)\right)}{4 s_2}-\frac{e^{s_1} \left(e^{s_2} k_{14}\left(s_3,-s_1-s_2-s_3\right)-k_{14}\left(s_2+s_3,-s_1-s_2-s_3\right)\right)}{8 s_2}-\frac{3 \left(s_1 k_6\left(s_3\right)-s_1 k_6\left(s_2+s_3\right)-s_2 k_6\left(s_2+s_3\right)+s_2 k_6\left(s_1+s_2+s_3\right)\right)}{2 s_1 s_2 \left(s_1+s_2\right)}-\frac{e^{s_1+s_2} \left(k_{14}\left(-s_1-s_2,s_1\right)-e^{s_3} k_{14}\left(-s_1-s_2-s_3,s_1\right)\right)}{8 s_3}-\frac{e^{s_1} \left(k_{16}\left(s_2,-s_1-s_2\right)-k_{16}\left(s_2+s_3,-s_1-s_2-s_3\right)\right)}{8 s_3}-\frac{e^{s_1+s_2} \left(k_{14}\left(-s_1-s_2,s_1\right)-k_{14}\left(s_3,-s_2-s_3\right)\right)}{8 \left(s_1+s_2+s_3\right)}-\frac{e^{s_1+s_2+s_3} \left(k_{15}\left(s_1,s_2\right)-k_{15}\left(-s_2-s_3,s_2\right)\right)}{8 \left(s_1+s_2+s_3\right)}-\frac{e^{s_1} \left(k_{16}\left(s_2,-s_1-s_2\right)-k_{16}\left(s_2,s_3\right)\right)}{8 \left(s_1+s_2+s_3\right)}. 
\end{math}
\end{center}

\subsubsection{The function $\widetilde K_{16}$}
\label{basicK16}
We have: 
\begin{equation} \label{basicK16eqn}
 \widetilde K_{16}(s_1, s_2, s_3) =
\end{equation}

\begin{center}
\begin{math}
 \frac{1}{5} (-2) \pi  G_3\left(s_1,s_2,s_3\right)-\frac{e^{s_2} \left(-e^{s_1} k_5\left(-s_1-s_2\right)+k_5\left(-s_2\right)-e^{s_3} k_5\left(-s_2-s_3\right)+e^{s_1+s_3} k_5\left(-s_1-s_2-s_3\right)\right)}{4 s_1 s_3}+\frac{e^{s_2} \left(e^{s_1} k_5\left(-s_1-s_2\right)-k_5\left(-s_2\right)+e^{s_3} k_5\left(-s_2-s_3\right)-e^{s_1+s_3} k_5\left(-s_1-s_2-s_3\right)\right)}{4 s_1 s_3}+\frac{-k_5\left(s_2\right)+k_5\left(s_1+s_2\right)+k_5\left(s_2+s_3\right)-k_5\left(s_1+s_2+s_3\right)}{2 s_1 s_3}-\frac{3}{4} G_2\left(s_1,s_2\right) k_6\left(s_3\right)+\frac{3 G_1\left(s_1\right) \left(k_6\left(s_2\right)-k_6\left(s_2+s_3\right)\right)}{2 s_3}+\frac{3 \left(k_6\left(s_2\right)-k_6\left(s_1+s_2\right)-k_6\left(s_2+s_3\right)+k_6\left(s_1+s_2+s_3\right)\right)}{2 s_1 s_3}+\frac{3 \left(-s_3 k_6\left(s_1\right)+s_2 k_6\left(s_1+s_2\right)+s_3 k_6\left(s_1+s_2\right)-s_2 k_6\left(s_1+s_2+s_3\right)\right)}{4 s_2 s_3 \left(s_2+s_3\right)}+\frac{3 \left(-s_1 k_6\left(s_3\right)+s_1 k_6\left(s_2+s_3\right)+s_2 k_6\left(s_2+s_3\right)-s_2 k_6\left(s_1+s_2+s_3\right)\right)}{4 s_1 s_2 \left(s_1+s_2\right)}+\frac{3 e^{s_2} G_1\left(s_1\right) \left(k_7\left(-s_2\right)-e^{s_3} k_7\left(-s_2-s_3\right)\right)}{2 s_3}+\frac{3 e^{s_2} \left(-e^{s_1} k_7\left(-s_1-s_2\right)+k_7\left(-s_2\right)-e^{s_3} k_7\left(-s_2-s_3\right)+e^{s_1+s_3} k_7\left(-s_1-s_2-s_3\right)\right)}{4 s_1 s_3}-\frac{3 e^{s_1} \left(s_3 k_7\left(-s_1\right)-e^{s_2} s_2 k_7\left(-s_1-s_2\right)-e^{s_2} s_3 k_7\left(-s_1-s_2\right)+e^{s_2+s_3} s_2 k_7\left(-s_1-s_2-s_3\right)\right)}{4 s_2 s_3 \left(s_2+s_3\right)}-\frac{3 e^{s_2} \left(e^{s_1} k_7\left(-s_1-s_2\right)-k_7\left(-s_2\right)+e^{s_3} k_7\left(-s_2-s_3\right)-e^{s_1+s_3} k_7\left(-s_1-s_2-s_3\right)\right)}{4 s_1 s_3}+\frac{3 e^{s_3} G_1\left(s_1\right) \left(e^{s_2} k_7\left(-s_2-s_3\right)-k_7\left(-s_3\right)\right)}{4 s_2}-\frac{3}{4} e^{s_3} G_2\left(s_1,s_2\right) k_7\left(-s_3\right)+\frac{3 e^{s_3} \left(e^{s_2} s_1 k_7\left(-s_2-s_3\right)+e^{s_2} s_2 k_7\left(-s_2-s_3\right)-e^{s_1+s_2} s_2 k_7\left(-s_1-s_2-s_3\right)-s_1 k_7\left(-s_3\right)\right)}{4 s_1 s_2 \left(s_1+s_2\right)}-\frac{1}{8} G_1\left(s_1\right) k_{14}\left(s_2,s_3\right)+\frac{k_{14}\left(s_1+s_2,s_3\right)-k_{14}\left(s_2,s_3\right)}{8 s_1}+\frac{1}{4} e^{s_2+s_3} G_1\left(s_1\right) k_{14}\left(-s_2-s_3,s_2\right)+\frac{e^{s_2+s_3} \left(k_{14}\left(-s_2-s_3,s_2\right)-e^{s_1} k_{14}\left(-s_1-s_2-s_3,s_1+s_2\right)\right)}{4 s_1}+\frac{e^{s_1} \left(k_{14}\left(s_2,-s_1-s_2\right)-k_{14}\left(s_2+s_3,-s_1-s_2-s_3\right)\right)}{8 s_3}+\frac{k_{15}\left(s_1,s_2+s_3\right)-k_{15}\left(s_1,s_2\right)}{4 s_3}+\frac{1}{4} G_1\left(s_1\right) k_{15}\left(s_2,s_3\right)+\frac{k_{15}\left(s_2,s_3\right)-k_{15}\left(s_1+s_2,s_3\right)}{4 s_1}+\frac{k_{15}\left(s_1+s_2,s_3\right)-k_{15}\left(s_1,s_2+s_3\right)}{8 s_2}+\frac{e^{s_1+s_2} \left(k_{15}\left(-s_1-s_2,s_1\right)-e^{s_3} k_{15}\left(-s_1-s_2-s_3,s_1\right)\right)}{8 s_3}-\frac{1}{8} e^{s_2} G_1\left(s_1\right) k_{15}\left(s_3,-s_2-s_3\right)+\frac{\left(-1+e^{s_1+s_2+s_3}\right) k_{16}\left(s_1,s_2\right)}{8 \left(s_1+s_2+s_3\right)}+\frac{k_{16}\left(s_1,s_2\right)-k_{16}\left(s_1,s_2+s_3\right)}{8 s_3}-\frac{1}{8} e^{s_2+s_3} G_1\left(s_1\right) k_{16}\left(-s_2-s_3,s_2\right)+\frac{1}{4} e^{s_2} G_1\left(s_1\right) k_{16}\left(s_3,-s_2-s_3\right)+\frac{e^{s_2} \left(k_{16}\left(s_3,-s_2-s_3\right)-e^{s_1} k_{16}\left(s_3,-s_1-s_2-s_3\right)\right)}{4 s_1}+\frac{e^{s_1} \left(e^{s_2} k_{16}\left(s_3,-s_1-s_2-s_3\right)-k_{16}\left(s_2+s_3,-s_1-s_2-s_3\right)\right)}{8 s_2}-\frac{1}{16} k_{20}\left(s_1,s_2,s_3\right)-\frac{1}{16} e^{s_1} k_{20}\left(s_2,s_3,-s_1-s_2-s_3\right)-\frac{1}{16} e^{s_1+s_2+s_3} k_{20}\left(-s_1-s_2-s_3,s_1,s_2\right)-\frac{1}{16} e^{s_1+s_2} k_{20}\left(s_3,-s_1-s_2-s_3,s_1\right)-\frac{e^{s_2} \left(k_{15}\left(s_3,-s_2-s_3\right)-e^{s_1} k_{15}\left(s_3,-s_1-s_2-s_3\right)\right)}{8 s_1}-\frac{e^{s_2+s_3} \left(k_{16}\left(-s_2-s_3,s_2\right)-e^{s_1} k_{16}\left(-s_1-s_2-s_3,s_1+s_2\right)\right)}{8 s_1}-\frac{3 G_1\left(s_1\right) \left(k_6\left(s_3\right)-k_6\left(s_2+s_3\right)\right)}{4 s_2}-\frac{e^{s_1+s_2+s_3} \left(k_{14}\left(-s_1-s_2-s_3,s_1\right)-k_{14}\left(-s_1-s_2-s_3,s_1+s_2\right)\right)}{8 s_2}-\frac{e^{s_2} G_1\left(s_1\right) \left(k_5\left(-s_2\right)-e^{s_3} k_5\left(-s_2-s_3\right)\right)}{2 s_3}-\frac{G_1\left(s_1\right) \left(k_5\left(s_2\right)-k_5\left(s_2+s_3\right)\right)}{2 s_3}-\frac{e^{s_1+s_2} \left(k_{14}\left(-s_1-s_2,s_1\right)-e^{s_3} k_{14}\left(-s_1-s_2-s_3,s_1\right)\right)}{4 s_3}-\frac{e^{s_1} \left(k_{16}\left(s_2,-s_1-s_2\right)-k_{16}\left(s_2+s_3,-s_1-s_2-s_3\right)\right)}{4 s_3}-\frac{e^{s_1} \left(k_{14}\left(s_2,-s_1-s_2\right)-k_{14}\left(s_2,s_3\right)\right)}{8 \left(s_1+s_2+s_3\right)}-\frac{e^{s_1+s_2} \left(k_{15}\left(-s_1-s_2,s_1\right)-k_{15}\left(s_3,-s_2-s_3\right)\right)}{8 \left(s_1+s_2+s_3\right)}-\frac{e^{s_1+s_2+s_3} \left(k_{16}\left(s_1,s_2\right)-k_{16}\left(-s_2-s_3,s_2\right)\right)}{8 \left(s_1+s_2+s_3\right)}. 
\end{math}
\end{center}

\subsection{Functional relations for $\widetilde K_{18}, \widetilde K_{19}, \widetilde K_{20}$}
Now we present the remaining basic functional relations associated with the four variable functions. 

\subsubsection{The function $\widetilde K_{18}$}
\label{basicK18}
We have: 

\begin{equation} \label{basicK18eqn}
\widetilde K_{18}(s_1, s_2, s_3, s_4) =
\end{equation}
\begin{center}
\begin{math}
 \frac{1}{15} (-4) \pi  G_4\left(s_1,s_2,s_3,s_4\right)+\frac{s_3 k_6\left(s_1\right)}{2 s_2 \left(s_2+s_3\right) \left(s_3+s_4\right) \left(s_2+s_3+s_4\right)}+\frac{s_4 k_6\left(s_1\right)}{2 s_2 \left(s_2+s_3\right) \left(s_3+s_4\right) \left(s_2+s_3+s_4\right)}+\frac{k_6\left(s_1+s_2\right)}{2 s_1 s_3 \left(s_3+s_4\right)}+\frac{G_1\left(s_1\right) k_6\left(s_3\right)}{2 s_2 s_4}+\frac{G_2\left(s_1,s_2\right) k_6\left(s_3\right)}{2 s_4}+\frac{k_6\left(s_3\right)}{2 s_2 \left(s_1+s_2\right) s_4}+\frac{G_1\left(s_1\right) k_6\left(s_2+s_3\right)}{2 s_3 \left(s_3+s_4\right)}+\frac{G_1\left(s_1\right) k_6\left(s_2+s_3\right)}{2 s_4 \left(s_3+s_4\right)}+\frac{k_6\left(s_2+s_3\right)}{2 s_1 s_3 \left(s_3+s_4\right)}+\frac{k_6\left(s_2+s_3\right)}{2 s_1 s_4 \left(s_3+s_4\right)}+\frac{k_6\left(s_1+s_2+s_3\right)}{2 s_1 \left(s_1+s_2\right) s_4}+\frac{k_6\left(s_1+s_2+s_3\right)}{\left(s_2+s_3\right) \left(s_3+s_4\right) \left(s_2+s_3+s_4\right)}+\frac{s_2 k_6\left(s_1+s_2+s_3\right)}{2 s_3 \left(s_2+s_3\right) \left(s_3+s_4\right) \left(s_2+s_3+s_4\right)}+\frac{s_4 k_6\left(s_1+s_2+s_3\right)}{2 s_3 \left(s_2+s_3\right) \left(s_3+s_4\right) \left(s_2+s_3+s_4\right)}+\frac{s_2 k_6\left(s_1+s_2+s_3\right)}{2 \left(s_2+s_3\right) s_4 \left(s_3+s_4\right) \left(s_2+s_3+s_4\right)}+\frac{s_3 k_6\left(s_1+s_2+s_3\right)}{2 \left(s_2+s_3\right) s_4 \left(s_3+s_4\right) \left(s_2+s_3+s_4\right)}-\frac{1}{2} G_3\left(s_1,s_2,s_3\right) k_6\left(s_4\right)+\frac{G_1\left(s_1\right) k_6\left(s_3+s_4\right)}{2 s_2 \left(s_2+s_3\right)}+\frac{G_1\left(s_1\right) k_6\left(s_3+s_4\right)}{2 s_3 \left(s_2+s_3\right)}+\frac{G_2\left(s_1,s_2\right) k_6\left(s_3+s_4\right)}{2 s_3}+\frac{k_6\left(s_3+s_4\right)}{\left(s_1+s_2\right) \left(s_2+s_3\right) \left(s_1+s_2+s_3\right)}+\frac{s_1 k_6\left(s_3+s_4\right)}{2 s_2 \left(s_1+s_2\right) \left(s_2+s_3\right) \left(s_1+s_2+s_3\right)}+\frac{s_3 k_6\left(s_3+s_4\right)}{2 s_2 \left(s_1+s_2\right) \left(s_2+s_3\right) \left(s_1+s_2+s_3\right)}+\frac{s_1 k_6\left(s_3+s_4\right)}{2 \left(s_1+s_2\right) s_3 \left(s_2+s_3\right) \left(s_1+s_2+s_3\right)}+\frac{s_2 k_6\left(s_3+s_4\right)}{2 \left(s_1+s_2\right) s_3 \left(s_2+s_3\right) \left(s_1+s_2+s_3\right)}+\frac{G_1\left(s_1\right) k_6\left(s_2+s_3+s_4\right)}{2 s_2 s_4}+\frac{k_6\left(s_2+s_3+s_4\right)}{2 s_1 \left(s_1+s_2\right) s_4}+\frac{k_6\left(s_2+s_3+s_4\right)}{2 s_2 \left(s_1+s_2\right) s_4}+\frac{s_2 k_6\left(s_1+s_2+s_3+s_4\right)}{2 s_1 \left(s_1+s_2\right) \left(s_2+s_3\right) \left(s_1+s_2+s_3\right)}+\frac{s_3 k_6\left(s_1+s_2+s_3+s_4\right)}{2 s_1 \left(s_1+s_2\right) \left(s_2+s_3\right) \left(s_1+s_2+s_3\right)}+\frac{k_6\left(s_1+s_2+s_3+s_4\right)}{2 s_1 s_4 \left(s_3+s_4\right)}+\frac{e^{s_1} s_3 k_7\left(-s_1\right)}{2 s_2 \left(s_2+s_3\right) \left(s_3+s_4\right) \left(s_2+s_3+s_4\right)}+\frac{e^{s_1} s_4 k_7\left(-s_1\right)}{2 s_2 \left(s_2+s_3\right) \left(s_3+s_4\right) \left(s_2+s_3+s_4\right)}+\frac{e^{s_1+s_2} k_7\left(-s_1-s_2\right)}{2 s_1 s_3 \left(s_3+s_4\right)}+\frac{e^{s_2+s_3} G_1\left(s_1\right) k_7\left(-s_2-s_3\right)}{2 s_3 \left(s_3+s_4\right)}+\frac{e^{s_2+s_3} G_1\left(s_1\right) k_7\left(-s_2-s_3\right)}{2 s_4 \left(s_3+s_4\right)}+\frac{e^{s_2+s_3} k_7\left(-s_2-s_3\right)}{2 s_1 s_3 \left(s_3+s_4\right)}+\frac{e^{s_2+s_3} k_7\left(-s_2-s_3\right)}{2 s_1 s_4 \left(s_3+s_4\right)}+\frac{e^{s_1+s_2+s_3} k_7\left(-s_1-s_2-s_3\right)}{2 s_1 \left(s_1+s_2\right) s_4}+\frac{e^{s_1+s_2+s_3} k_7\left(-s_1-s_2-s_3\right)}{\left(s_2+s_3\right) \left(s_3+s_4\right) \left(s_2+s_3+s_4\right)}+\frac{e^{s_1+s_2+s_3} s_2 k_7\left(-s_1-s_2-s_3\right)}{2 s_3 \left(s_2+s_3\right) \left(s_3+s_4\right) \left(s_2+s_3+s_4\right)}+\frac{e^{s_1+s_2+s_3} s_4 k_7\left(-s_1-s_2-s_3\right)}{2 s_3 \left(s_2+s_3\right) \left(s_3+s_4\right) \left(s_2+s_3+s_4\right)}+\frac{e^{s_1+s_2+s_3} s_2 k_7\left(-s_1-s_2-s_3\right)}{2 \left(s_2+s_3\right) s_4 \left(s_3+s_4\right) \left(s_2+s_3+s_4\right)}+\frac{e^{s_1+s_2+s_3} s_3 k_7\left(-s_1-s_2-s_3\right)}{2 \left(s_2+s_3\right) s_4 \left(s_3+s_4\right) \left(s_2+s_3+s_4\right)}+\frac{e^{s_3} G_1\left(s_1\right) k_7\left(-s_3\right)}{2 s_2 s_4}+\frac{e^{s_3} G_2\left(s_1,s_2\right) k_7\left(-s_3\right)}{2 s_4}+\frac{e^{s_3} k_7\left(-s_3\right)}{2 s_2 \left(s_1+s_2\right) s_4}+\frac{e^{s_3+s_4} G_1\left(s_1\right) k_7\left(-s_3-s_4\right)}{2 s_2 \left(s_2+s_3\right)}+\frac{e^{s_3+s_4} G_1\left(s_1\right) k_7\left(-s_3-s_4\right)}{2 s_3 \left(s_2+s_3\right)}+\frac{e^{s_3+s_4} G_2\left(s_1,s_2\right) k_7\left(-s_3-s_4\right)}{2 s_3}+\frac{e^{s_3+s_4} k_7\left(-s_3-s_4\right)}{\left(s_1+s_2\right) \left(s_2+s_3\right) \left(s_1+s_2+s_3\right)}+\frac{e^{s_3+s_4} s_1 k_7\left(-s_3-s_4\right)}{2 s_2 \left(s_1+s_2\right) \left(s_2+s_3\right) \left(s_1+s_2+s_3\right)}+\frac{e^{s_3+s_4} s_3 k_7\left(-s_3-s_4\right)}{2 s_2 \left(s_1+s_2\right) \left(s_2+s_3\right) \left(s_1+s_2+s_3\right)}+\frac{e^{s_3+s_4} s_1 k_7\left(-s_3-s_4\right)}{2 \left(s_1+s_2\right) s_3 \left(s_2+s_3\right) \left(s_1+s_2+s_3\right)}+\frac{e^{s_3+s_4} s_2 k_7\left(-s_3-s_4\right)}{2 \left(s_1+s_2\right) s_3 \left(s_2+s_3\right) \left(s_1+s_2+s_3\right)}+\frac{e^{s_2+s_3+s_4} G_1\left(s_1\right) k_7\left(-s_2-s_3-s_4\right)}{2 s_2 s_4}+\frac{e^{s_2+s_3+s_4} k_7\left(-s_2-s_3-s_4\right)}{2 s_1 \left(s_1+s_2\right) s_4}+\frac{e^{s_2+s_3+s_4} k_7\left(-s_2-s_3-s_4\right)}{2 s_2 \left(s_1+s_2\right) s_4}+\frac{e^{s_1+s_2+s_3+s_4} s_2 k_7\left(-s_1-s_2-s_3-s_4\right)}{2 s_1 \left(s_1+s_2\right) \left(s_2+s_3\right) \left(s_1+s_2+s_3\right)}+\frac{e^{s_1+s_2+s_3+s_4} s_3 k_7\left(-s_1-s_2-s_3-s_4\right)}{2 s_1 \left(s_1+s_2\right) \left(s_2+s_3\right) \left(s_1+s_2+s_3\right)}+\frac{e^{s_1+s_2+s_3+s_4} k_7\left(-s_1-s_2-s_3-s_4\right)}{2 s_1 s_4 \left(s_3+s_4\right)}-\frac{1}{2} e^{s_4} G_3\left(s_1,s_2,s_3\right) k_7\left(-s_4\right)+\frac{k_8\left(s_1,s_2+s_3\right)}{4 s_2 s_4}+\frac{k_8\left(s_1+s_2,s_3+s_4\right)}{4 s_1 s_3}+\frac{k_8\left(s_1+s_2,s_3+s_4\right)}{4 s_2 s_4}+\frac{G_1\left(s_1\right) k_8\left(s_2+s_3,s_4\right)}{4 s_3}+\frac{k_8\left(s_2+s_3,s_4\right)}{4 s_1 s_3}+\frac{k_9\left(s_1,s_2\right)}{8 s_3 \left(s_3+s_4\right)}+\frac{k_9\left(s_1,s_2+s_3+s_4\right)}{8 s_2 \left(s_2+s_3\right)}+\frac{k_9\left(s_1,s_2+s_3+s_4\right)}{8 s_4 \left(s_3+s_4\right)}+\frac{G_1\left(s_1\right) k_9\left(s_2,s_3+s_4\right)}{8 s_4}+\frac{k_9\left(s_2,s_3+s_4\right)}{8 s_1 s_4}+\frac{k_9\left(s_1+s_2,s_3\right)}{8 s_1 s_4}+\frac{G_1\left(s_1\right) k_9\left(s_3,s_4\right)}{8 s_2}+\frac{1}{8} G_2\left(s_1,s_2\right) k_9\left(s_3,s_4\right)+\frac{k_9\left(s_3,s_4\right)}{8 s_2 \left(s_1+s_2\right)}+\frac{k_9\left(s_1+s_2+s_3,s_4\right)}{8 s_1 \left(s_1+s_2\right)}+\frac{k_9\left(s_1+s_2+s_3,s_4\right)}{8 s_3 \left(s_2+s_3\right)}+\frac{e^{s_1} k_{10}\left(s_2,-s_1-s_2\right)}{8 s_3 \left(s_3+s_4\right)}+\frac{e^{s_1+s_2} k_{10}\left(s_3,-s_1-s_2-s_3\right)}{8 s_1 s_4}+\frac{e^{s_3} G_1\left(s_1\right) k_{10}\left(s_4,-s_3-s_4\right)}{8 s_2}+\frac{1}{8} e^{s_3} G_2\left(s_1,s_2\right) k_{10}\left(s_4,-s_3-s_4\right)+\frac{e^{s_3} k_{10}\left(s_4,-s_3-s_4\right)}{8 s_2 \left(s_1+s_2\right)}+\frac{e^{s_1+s_2+s_3} k_{10}\left(s_4,-s_1-s_2-s_3-s_4\right)}{8 s_1 \left(s_1+s_2\right)}+\frac{e^{s_1+s_2+s_3} k_{10}\left(s_4,-s_1-s_2-s_3-s_4\right)}{8 s_3 \left(s_2+s_3\right)}+\frac{e^{s_2} G_1\left(s_1\right) k_{10}\left(s_3+s_4,-s_2-s_3-s_4\right)}{8 s_4}+\frac{e^{s_2} k_{10}\left(s_3+s_4,-s_2-s_3-s_4\right)}{8 s_1 s_4}+\frac{e^{s_1} k_{10}\left(s_2+s_3+s_4,-s_1-s_2-s_3-s_4\right)}{8 s_2 \left(s_2+s_3\right)}+\frac{e^{s_1} k_{10}\left(s_2+s_3+s_4,-s_1-s_2-s_3-s_4\right)}{8 s_4 \left(s_3+s_4\right)}+\frac{e^{s_1} k_{11}\left(s_2+s_3,-s_1-s_2-s_3\right)}{4 s_2 s_4}+\frac{e^{s_2+s_3} G_1\left(s_1\right) k_{11}\left(s_4,-s_2-s_3-s_4\right)}{4 s_3}+\frac{e^{s_2+s_3} k_{11}\left(s_4,-s_2-s_3-s_4\right)}{4 s_1 s_3}+\frac{e^{s_1+s_2} k_{11}\left(s_3+s_4,-s_1-s_2-s_3-s_4\right)}{4 s_1 s_3}+\frac{e^{s_1+s_2} k_{11}\left(s_3+s_4,-s_1-s_2-s_3-s_4\right)}{4 s_2 s_4}+\frac{e^{s_1+s_2+s_3} k_{12}\left(-s_1-s_2-s_3,s_1\right)}{4 s_2 s_4}+\frac{e^{s_2+s_3+s_4} G_1\left(s_1\right) k_{12}\left(-s_2-s_3-s_4,s_2+s_3\right)}{4 s_3}+\frac{e^{s_2+s_3+s_4} k_{12}\left(-s_2-s_3-s_4,s_2+s_3\right)}{4 s_1 s_3}+\frac{e^{s_1+s_2+s_3+s_4} k_{12}\left(-s_1-s_2-s_3-s_4,s_1+s_2\right)}{4 s_1 s_3}+\frac{e^{s_1+s_2+s_3+s_4} k_{12}\left(-s_1-s_2-s_3-s_4,s_1+s_2\right)}{4 s_2 s_4}+\frac{e^{s_1+s_2} k_{13}\left(-s_1-s_2,s_1\right)}{8 s_3 \left(s_3+s_4\right)}+\frac{e^{s_1+s_2+s_3} k_{13}\left(-s_1-s_2-s_3,s_1+s_2\right)}{8 s_1 s_4}+\frac{e^{s_3+s_4} G_1\left(s_1\right) k_{13}\left(-s_3-s_4,s_3\right)}{8 s_2}+\frac{1}{8} e^{s_3+s_4} G_2\left(s_1,s_2\right) k_{13}\left(-s_3-s_4,s_3\right)+\frac{e^{s_3+s_4} k_{13}\left(-s_3-s_4,s_3\right)}{8 s_2 \left(s_1+s_2\right)}+\frac{e^{s_2+s_3+s_4} G_1\left(s_1\right) k_{13}\left(-s_2-s_3-s_4,s_2\right)}{8 s_4}+\frac{e^{s_2+s_3+s_4} k_{13}\left(-s_2-s_3-s_4,s_2\right)}{8 s_1 s_4}+\frac{e^{s_1+s_2+s_3+s_4} k_{13}\left(-s_1-s_2-s_3-s_4,s_1\right)}{8 s_2 \left(s_2+s_3\right)}+\frac{e^{s_1+s_2+s_3+s_4} k_{13}\left(-s_1-s_2-s_3-s_4,s_1\right)}{8 s_4 \left(s_3+s_4\right)}+\frac{e^{s_1+s_2+s_3+s_4} k_{13}\left(-s_1-s_2-s_3-s_4,s_1+s_2+s_3\right)}{8 s_1 \left(s_1+s_2\right)}+\frac{e^{s_1+s_2+s_3+s_4} k_{13}\left(-s_1-s_2-s_3-s_4,s_1+s_2+s_3\right)}{8 s_3 \left(s_2+s_3\right)}+\frac{k_{17}\left(s_1,s_2+s_3,s_4\right)}{16 s_2}+\frac{e^{s_1} k_{17}\left(s_2,s_3+s_4,-s_1-s_2-s_3-s_4\right)}{16 s_3}+\frac{e^{s_1+s_2+s_3+s_4} k_{17}\left(-s_1-s_2-s_3-s_4,s_1,s_2\right)}{16 s_3}+\frac{e^{s_1+s_2+s_3} k_{17}\left(s_4,-s_1-s_2-s_3-s_4,s_1\right)}{16 s_2}+\frac{k_{18}\left(s_1,s_2,s_3\right)}{16 s_4}+\frac{e^{s_1} k_{18}\left(s_2,s_3,-s_1-s_2-s_3\right)}{16 s_4}-\frac{1}{16} G_1\left(s_1\right) k_{18}\left(s_2,s_3,s_4\right)+\frac{e^{s_1} k_{18}\left(s_2,s_3,s_4\right)}{16 \left(s_1+s_2+s_3+s_4\right)}+\frac{k_{18}\left(s_1+s_2,s_3,s_4\right)}{16 s_1}+\frac{e^{s_1+s_2+s_3} k_{18}\left(-s_1-s_2-s_3,s_1,s_2\right)}{16 s_4}+\frac{e^{s_1+s_2} k_{18}\left(s_3,-s_1-s_2-s_3,s_1\right)}{16 s_4}-\frac{1}{16} e^{s_2} G_1\left(s_1\right) k_{18}\left(s_3,s_4,-s_2-s_3-s_4\right)+\frac{e^{s_1+s_2} k_{18}\left(s_3,s_4,-s_2-s_3-s_4\right)}{16 \left(s_1+s_2+s_3+s_4\right)}+\frac{e^{s_1+s_2} k_{18}\left(s_3,s_4,-s_1-s_2-s_3-s_4\right)}{16 s_1}-\frac{1}{16} e^{s_2+s_3+s_4} G_1\left(s_1\right) k_{18}\left(-s_2-s_3-s_4,s_2,s_3\right)+\frac{e^{s_1+s_2+s_3+s_4} k_{18}\left(-s_2-s_3-s_4,s_2,s_3\right)}{16 \left(s_1+s_2+s_3+s_4\right)}+\frac{e^{s_1+s_2+s_3+s_4} k_{18}\left(-s_1-s_2-s_3-s_4,s_1+s_2,s_3\right)}{16 s_1}-\frac{1}{16} e^{s_2+s_3} G_1\left(s_1\right) k_{18}\left(s_4,-s_2-s_3-s_4,s_2\right)+\frac{e^{s_1+s_2+s_3} k_{18}\left(s_4,-s_2-s_3-s_4,s_2\right)}{16 \left(s_1+s_2+s_3+s_4\right)}+\frac{e^{s_1+s_2+s_3} k_{18}\left(s_4,-s_1-s_2-s_3-s_4,s_1+s_2\right)}{16 s_1}+\frac{k_{19}\left(s_1,s_2,s_3+s_4\right)}{16 s_3}+\frac{e^{s_1} k_{19}\left(s_2+s_3,s_4,-s_1-s_2-s_3-s_4\right)}{16 s_2}+\frac{e^{s_1+s_2+s_3+s_4} k_{19}\left(-s_1-s_2-s_3-s_4,s_1,s_2+s_3\right)}{16 s_2}+\frac{e^{s_1+s_2} k_{19}\left(s_3+s_4,-s_1-s_2-s_3-s_4,s_1\right)}{16 s_3}-\frac{k_{18}\left(s_2,s_3,s_4\right)}{16 s_1}-\frac{e^{s_2} k_{18}\left(s_3,s_4,-s_2-s_3-s_4\right)}{16 s_1}-\frac{e^{s_2+s_3+s_4} k_{18}\left(-s_2-s_3-s_4,s_2,s_3\right)}{16 s_1}-\frac{e^{s_2+s_3} k_{18}\left(s_4,-s_2-s_3-s_4,s_2\right)}{16 s_1}-\frac{G_1\left(s_1\right) k_9\left(s_2+s_3,s_4\right)}{8 s_2}-\frac{e^{s_2+s_3} G_1\left(s_1\right) k_{10}\left(s_4,-s_2-s_3-s_4\right)}{8 s_2}-\frac{e^{s_2+s_3+s_4} G_1\left(s_1\right) k_{13}\left(-s_2-s_3-s_4,s_2+s_3\right)}{8 s_2}-\frac{k_{17}\left(s_1+s_2,s_3,s_4\right)}{16 s_2}-\frac{e^{s_1+s_2+s_3} k_{17}\left(s_4,-s_1-s_2-s_3-s_4,s_1+s_2\right)}{16 s_2}-\frac{e^{s_1+s_2} k_{19}\left(s_3,s_4,-s_1-s_2-s_3-s_4\right)}{16 s_2}-\frac{e^{s_1+s_2+s_3+s_4} k_{19}\left(-s_1-s_2-s_3-s_4,s_1+s_2,s_3\right)}{16 s_2}-\frac{k_9\left(s_2+s_3,s_4\right)}{8 s_1 \left(s_1+s_2\right)}-\frac{e^{s_2+s_3} k_{10}\left(s_4,-s_2-s_3-s_4\right)}{8 s_1 \left(s_1+s_2\right)}-\frac{e^{s_2+s_3+s_4} k_{13}\left(-s_2-s_3-s_4,s_2+s_3\right)}{8 s_1 \left(s_1+s_2\right)}-\frac{k_9\left(s_2+s_3,s_4\right)}{8 s_2 \left(s_1+s_2\right)}-\frac{e^{s_2+s_3} k_{10}\left(s_4,-s_2-s_3-s_4\right)}{8 s_2 \left(s_1+s_2\right)}-\frac{e^{s_2+s_3+s_4} k_{13}\left(-s_2-s_3-s_4,s_2+s_3\right)}{8 s_2 \left(s_1+s_2\right)}-\frac{G_2\left(s_1,s_2\right) k_6\left(s_4\right)}{2 s_3}-\frac{e^{s_4} G_2\left(s_1,s_2\right) k_7\left(-s_4\right)}{2 s_3}-\frac{G_1\left(s_1\right) k_8\left(s_2,s_3+s_4\right)}{4 s_3}-\frac{e^{s_2} G_1\left(s_1\right) k_{11}\left(s_3+s_4,-s_2-s_3-s_4\right)}{4 s_3}-\frac{e^{s_2+s_3+s_4} G_1\left(s_1\right) k_{12}\left(-s_2-s_3-s_4,s_2\right)}{4 s_3}-\frac{e^{s_1} k_{17}\left(s_2+s_3,s_4,-s_1-s_2-s_3-s_4\right)}{16 s_3}-\frac{e^{s_1+s_2+s_3+s_4} k_{17}\left(-s_1-s_2-s_3-s_4,s_1,s_2+s_3\right)}{16 s_3}-\frac{k_{19}\left(s_1,s_2+s_3,s_4\right)}{16 s_3}-\frac{e^{s_1+s_2+s_3} k_{19}\left(s_4,-s_1-s_2-s_3-s_4,s_1\right)}{16 s_3}-\frac{k_8\left(s_2,s_3+s_4\right)}{4 s_1 s_3}-\frac{k_8\left(s_1+s_2+s_3,s_4\right)}{4 s_1 s_3}-\frac{e^{s_1+s_2+s_3} k_{11}\left(s_4,-s_1-s_2-s_3-s_4\right)}{4 s_1 s_3}-\frac{e^{s_2} k_{11}\left(s_3+s_4,-s_2-s_3-s_4\right)}{4 s_1 s_3}-\frac{e^{s_2+s_3+s_4} k_{12}\left(-s_2-s_3-s_4,s_2\right)}{4 s_1 s_3}-\frac{e^{s_1+s_2+s_3+s_4} k_{12}\left(-s_1-s_2-s_3-s_4,s_1+s_2+s_3\right)}{4 s_1 s_3}-\frac{G_1\left(s_1\right) k_6\left(s_2+s_3+s_4\right)}{2 s_2 \left(s_2+s_3\right)}-\frac{e^{s_2+s_3+s_4} G_1\left(s_1\right) k_7\left(-s_2-s_3-s_4\right)}{2 s_2 \left(s_2+s_3\right)}-\frac{k_9\left(s_1+s_2,s_3+s_4\right)}{8 s_2 \left(s_2+s_3\right)}-\frac{e^{s_1+s_2} k_{10}\left(s_3+s_4,-s_1-s_2-s_3-s_4\right)}{8 s_2 \left(s_2+s_3\right)}-\frac{e^{s_1+s_2+s_3+s_4} k_{13}\left(-s_1-s_2-s_3-s_4,s_1+s_2\right)}{8 s_2 \left(s_2+s_3\right)}-\frac{G_1\left(s_1\right) k_6\left(s_4\right)}{2 s_3 \left(s_2+s_3\right)}-\frac{e^{s_4} G_1\left(s_1\right) k_7\left(-s_4\right)}{2 s_3 \left(s_2+s_3\right)}-\frac{k_9\left(s_1+s_2,s_3+s_4\right)}{8 s_3 \left(s_2+s_3\right)}-\frac{e^{s_1+s_2} k_{10}\left(s_3+s_4,-s_1-s_2-s_3-s_4\right)}{8 s_3 \left(s_2+s_3\right)}-\frac{e^{s_1+s_2+s_3+s_4} k_{13}\left(-s_1-s_2-s_3-s_4,s_1+s_2\right)}{8 s_3 \left(s_2+s_3\right)}-\frac{k_6\left(s_2+s_3+s_4\right)}{\left(s_1+s_2\right) \left(s_2+s_3\right) \left(s_1+s_2+s_3\right)}-\frac{e^{s_2+s_3+s_4} k_7\left(-s_2-s_3-s_4\right)}{\left(s_1+s_2\right) \left(s_2+s_3\right) \left(s_1+s_2+s_3\right)}-\frac{s_2 k_6\left(s_2+s_3+s_4\right)}{2 s_1 \left(s_1+s_2\right) \left(s_2+s_3\right) \left(s_1+s_2+s_3\right)}-\frac{s_3 k_6\left(s_2+s_3+s_4\right)}{2 s_1 \left(s_1+s_2\right) \left(s_2+s_3\right) \left(s_1+s_2+s_3\right)}-\frac{e^{s_2+s_3+s_4} s_2 k_7\left(-s_2-s_3-s_4\right)}{2 s_1 \left(s_1+s_2\right) \left(s_2+s_3\right) \left(s_1+s_2+s_3\right)}-\frac{e^{s_2+s_3+s_4} s_3 k_7\left(-s_2-s_3-s_4\right)}{2 s_1 \left(s_1+s_2\right) \left(s_2+s_3\right) \left(s_1+s_2+s_3\right)}-\frac{s_1 k_6\left(s_2+s_3+s_4\right)}{2 s_2 \left(s_1+s_2\right) \left(s_2+s_3\right) \left(s_1+s_2+s_3\right)}-\frac{s_3 k_6\left(s_2+s_3+s_4\right)}{2 s_2 \left(s_1+s_2\right) \left(s_2+s_3\right) \left(s_1+s_2+s_3\right)}-\frac{e^{s_2+s_3+s_4} s_1 k_7\left(-s_2-s_3-s_4\right)}{2 s_2 \left(s_1+s_2\right) \left(s_2+s_3\right) \left(s_1+s_2+s_3\right)}-\frac{e^{s_2+s_3+s_4} s_3 k_7\left(-s_2-s_3-s_4\right)}{2 s_2 \left(s_1+s_2\right) \left(s_2+s_3\right) \left(s_1+s_2+s_3\right)}-\frac{s_1 k_6\left(s_4\right)}{2 \left(s_1+s_2\right) s_3 \left(s_2+s_3\right) \left(s_1+s_2+s_3\right)}-\frac{s_2 k_6\left(s_4\right)}{2 \left(s_1+s_2\right) s_3 \left(s_2+s_3\right) \left(s_1+s_2+s_3\right)}-\frac{e^{s_4} s_1 k_7\left(-s_4\right)}{2 \left(s_1+s_2\right) s_3 \left(s_2+s_3\right) \left(s_1+s_2+s_3\right)}-\frac{e^{s_4} s_2 k_7\left(-s_4\right)}{2 \left(s_1+s_2\right) s_3 \left(s_2+s_3\right) \left(s_1+s_2+s_3\right)}-\frac{G_2\left(s_1,s_2\right) k_6\left(s_3+s_4\right)}{2 s_4}-\frac{e^{s_3+s_4} G_2\left(s_1,s_2\right) k_7\left(-s_3-s_4\right)}{2 s_4}-\frac{G_1\left(s_1\right) k_9\left(s_2,s_3\right)}{8 s_4}-\frac{e^{s_2} G_1\left(s_1\right) k_{10}\left(s_3,-s_2-s_3\right)}{8 s_4}-\frac{e^{s_2+s_3} G_1\left(s_1\right) k_{13}\left(-s_2-s_3,s_2\right)}{8 s_4}-\frac{k_{18}\left(s_1,s_2,s_3+s_4\right)}{16 s_4}-\frac{e^{s_1} k_{18}\left(s_2,s_3+s_4,-s_1-s_2-s_3-s_4\right)}{16 s_4}-\frac{e^{s_1+s_2+s_3+s_4} k_{18}\left(-s_1-s_2-s_3-s_4,s_1,s_2\right)}{16 s_4}-\frac{e^{s_1+s_2} k_{18}\left(s_3+s_4,-s_1-s_2-s_3-s_4,s_1\right)}{16 s_4}-\frac{k_9\left(s_2,s_3\right)}{8 s_1 s_4}-\frac{k_9\left(s_1+s_2,s_3+s_4\right)}{8 s_1 s_4}-\frac{e^{s_2} k_{10}\left(s_3,-s_2-s_3\right)}{8 s_1 s_4}-\frac{e^{s_1+s_2} k_{10}\left(s_3+s_4,-s_1-s_2-s_3-s_4\right)}{8 s_1 s_4}-\frac{e^{s_2+s_3} k_{13}\left(-s_2-s_3,s_2\right)}{8 s_1 s_4}-\frac{e^{s_1+s_2+s_3+s_4} k_{13}\left(-s_1-s_2-s_3-s_4,s_1+s_2\right)}{8 s_1 s_4}-\frac{G_1\left(s_1\right) k_6\left(s_2+s_3\right)}{2 s_2 s_4}-\frac{G_1\left(s_1\right) k_6\left(s_3+s_4\right)}{2 s_2 s_4}-\frac{e^{s_2+s_3} G_1\left(s_1\right) k_7\left(-s_2-s_3\right)}{2 s_2 s_4}-\frac{e^{s_3+s_4} G_1\left(s_1\right) k_7\left(-s_3-s_4\right)}{2 s_2 s_4}-\frac{k_8\left(s_1,s_2+s_3+s_4\right)}{4 s_2 s_4}-\frac{k_8\left(s_1+s_2,s_3\right)}{4 s_2 s_4}-\frac{e^{s_1+s_2} k_{11}\left(s_3,-s_1-s_2-s_3\right)}{4 s_2 s_4}-\frac{e^{s_1} k_{11}\left(s_2+s_3+s_4,-s_1-s_2-s_3-s_4\right)}{4 s_2 s_4}-\frac{e^{s_1+s_2+s_3} k_{12}\left(-s_1-s_2-s_3,s_1+s_2\right)}{4 s_2 s_4}-\frac{e^{s_1+s_2+s_3+s_4} k_{12}\left(-s_1-s_2-s_3-s_4,s_1\right)}{4 s_2 s_4}-\frac{k_6\left(s_2+s_3\right)}{2 s_1 \left(s_1+s_2\right) s_4}-\frac{k_6\left(s_1+s_2+s_3+s_4\right)}{2 s_1 \left(s_1+s_2\right) s_4}-\frac{e^{s_2+s_3} k_7\left(-s_2-s_3\right)}{2 s_1 \left(s_1+s_2\right) s_4}-\frac{e^{s_1+s_2+s_3+s_4} k_7\left(-s_1-s_2-s_3-s_4\right)}{2 s_1 \left(s_1+s_2\right) s_4}-\frac{k_6\left(s_2+s_3\right)}{2 s_2 \left(s_1+s_2\right) s_4}-\frac{k_6\left(s_3+s_4\right)}{2 s_2 \left(s_1+s_2\right) s_4}-\frac{e^{s_2+s_3} k_7\left(-s_2-s_3\right)}{2 s_2 \left(s_1+s_2\right) s_4}-\frac{e^{s_3+s_4} k_7\left(-s_3-s_4\right)}{2 s_2 \left(s_1+s_2\right) s_4}-\frac{G_1\left(s_1\right) k_6\left(s_2\right)}{2 s_3 \left(s_3+s_4\right)}-\frac{e^{s_2} G_1\left(s_1\right) k_7\left(-s_2\right)}{2 s_3 \left(s_3+s_4\right)}-\frac{k_9\left(s_1,s_2+s_3\right)}{8 s_3 \left(s_3+s_4\right)}-\frac{e^{s_1} k_{10}\left(s_2+s_3,-s_1-s_2-s_3\right)}{8 s_3 \left(s_3+s_4\right)}-\frac{e^{s_1+s_2+s_3} k_{13}\left(-s_1-s_2-s_3,s_1\right)}{8 s_3 \left(s_3+s_4\right)}-\frac{k_6\left(s_2\right)}{2 s_1 s_3 \left(s_3+s_4\right)}-\frac{k_6\left(s_1+s_2+s_3\right)}{2 s_1 s_3 \left(s_3+s_4\right)}-\frac{e^{s_2} k_7\left(-s_2\right)}{2 s_1 s_3 \left(s_3+s_4\right)}-\frac{e^{s_1+s_2+s_3} k_7\left(-s_1-s_2-s_3\right)}{2 s_1 s_3 \left(s_3+s_4\right)}-\frac{G_1\left(s_1\right) k_6\left(s_2+s_3+s_4\right)}{2 s_4 \left(s_3+s_4\right)}-\frac{e^{s_2+s_3+s_4} G_1\left(s_1\right) k_7\left(-s_2-s_3-s_4\right)}{2 s_4 \left(s_3+s_4\right)}-\frac{k_9\left(s_1,s_2+s_3\right)}{8 s_4 \left(s_3+s_4\right)}-\frac{e^{s_1} k_{10}\left(s_2+s_3,-s_1-s_2-s_3\right)}{8 s_4 \left(s_3+s_4\right)}-\frac{e^{s_1+s_2+s_3} k_{13}\left(-s_1-s_2-s_3,s_1\right)}{8 s_4 \left(s_3+s_4\right)}-\frac{k_6\left(s_1+s_2+s_3\right)}{2 s_1 s_4 \left(s_3+s_4\right)}-\frac{k_6\left(s_2+s_3+s_4\right)}{2 s_1 s_4 \left(s_3+s_4\right)}-\frac{e^{s_1+s_2+s_3} k_7\left(-s_1-s_2-s_3\right)}{2 s_1 s_4 \left(s_3+s_4\right)}-\frac{e^{s_2+s_3+s_4} k_7\left(-s_2-s_3-s_4\right)}{2 s_1 s_4 \left(s_3+s_4\right)}-\frac{k_6\left(s_1+s_2\right)}{\left(s_2+s_3\right) \left(s_3+s_4\right) \left(s_2+s_3+s_4\right)}-\frac{e^{s_1+s_2} k_7\left(-s_1-s_2\right)}{\left(s_2+s_3\right) \left(s_3+s_4\right) \left(s_2+s_3+s_4\right)}-\frac{s_3 k_6\left(s_1+s_2\right)}{2 s_2 \left(s_2+s_3\right) \left(s_3+s_4\right) \left(s_2+s_3+s_4\right)}-\frac{s_4 k_6\left(s_1+s_2\right)}{2 s_2 \left(s_2+s_3\right) \left(s_3+s_4\right) \left(s_2+s_3+s_4\right)}-\frac{e^{s_1+s_2} s_3 k_7\left(-s_1-s_2\right)}{2 s_2 \left(s_2+s_3\right) \left(s_3+s_4\right) \left(s_2+s_3+s_4\right)}-\frac{e^{s_1+s_2} s_4 k_7\left(-s_1-s_2\right)}{2 s_2 \left(s_2+s_3\right) \left(s_3+s_4\right) \left(s_2+s_3+s_4\right)}-\frac{s_2 k_6\left(s_1+s_2\right)}{2 s_3 \left(s_2+s_3\right) \left(s_3+s_4\right) \left(s_2+s_3+s_4\right)}-\frac{s_4 k_6\left(s_1+s_2\right)}{2 s_3 \left(s_2+s_3\right) \left(s_3+s_4\right) \left(s_2+s_3+s_4\right)}-\frac{e^{s_1+s_2} s_2 k_7\left(-s_1-s_2\right)}{2 s_3 \left(s_2+s_3\right) \left(s_3+s_4\right) \left(s_2+s_3+s_4\right)}-\frac{e^{s_1+s_2} s_4 k_7\left(-s_1-s_2\right)}{2 s_3 \left(s_2+s_3\right) \left(s_3+s_4\right) \left(s_2+s_3+s_4\right)}-\frac{s_2 k_6\left(s_1+s_2+s_3+s_4\right)}{2 \left(s_2+s_3\right) s_4 \left(s_3+s_4\right) \left(s_2+s_3+s_4\right)}-\frac{s_3 k_6\left(s_1+s_2+s_3+s_4\right)}{2 \left(s_2+s_3\right) s_4 \left(s_3+s_4\right) \left(s_2+s_3+s_4\right)}-\frac{e^{s_1+s_2+s_3+s_4} s_2 k_7\left(-s_1-s_2-s_3-s_4\right)}{2 \left(s_2+s_3\right) s_4 \left(s_3+s_4\right) \left(s_2+s_3+s_4\right)}-\frac{e^{s_1+s_2+s_3+s_4} s_3 k_7\left(-s_1-s_2-s_3-s_4\right)}{2 \left(s_2+s_3\right) s_4 \left(s_3+s_4\right) \left(s_2+s_3+s_4\right)}-\frac{k_{18}\left(s_1,s_2,s_3\right)}{16 \left(s_1+s_2+s_3+s_4\right)}-\frac{e^{s_1} k_{18}\left(s_2,s_3,-s_1-s_2-s_3\right)}{16 \left(s_1+s_2+s_3+s_4\right)}-\frac{e^{s_1+s_2+s_3} k_{18}\left(-s_1-s_2-s_3,s_1,s_2\right)}{16 \left(s_1+s_2+s_3+s_4\right)}-\frac{e^{s_1+s_2} k_{18}\left(s_3,-s_1-s_2-s_3,s_1\right)}{16 \left(s_1+s_2+s_3+s_4\right)}. 
\end{math}
\end{center}

\subsubsection{The function $\widetilde K_{19}$}
\label{basicK19}
We have: 
\begin{equation} \label{basicK19eqn}
\widetilde K_{19}(s_1, s_2, s_3, s_4) =
\end{equation}

\begin{center}
\begin{math}
 \frac{1}{15} (-4) \pi  G_4\left(s_1,s_2,s_3,s_4\right)+\frac{s_3 k_6\left(s_1\right)}{2 s_2 \left(s_2+s_3\right) \left(s_3+s_4\right) \left(s_2+s_3+s_4\right)}+\frac{s_4 k_6\left(s_1\right)}{2 s_2 \left(s_2+s_3\right) \left(s_3+s_4\right) \left(s_2+s_3+s_4\right)}+\frac{k_6\left(s_1+s_2\right)}{2 s_1 s_3 \left(s_3+s_4\right)}+\frac{G_1\left(s_1\right) k_6\left(s_3\right)}{2 s_2 s_4}+\frac{G_2\left(s_1,s_2\right) k_6\left(s_3\right)}{2 s_4}+\frac{k_6\left(s_3\right)}{2 s_2 \left(s_1+s_2\right) s_4}+\frac{G_1\left(s_1\right) k_6\left(s_2+s_3\right)}{2 s_3 \left(s_3+s_4\right)}+\frac{G_1\left(s_1\right) k_6\left(s_2+s_3\right)}{2 s_4 \left(s_3+s_4\right)}+\frac{k_6\left(s_2+s_3\right)}{2 s_1 s_3 \left(s_3+s_4\right)}+\frac{k_6\left(s_2+s_3\right)}{2 s_1 s_4 \left(s_3+s_4\right)}+\frac{k_6\left(s_1+s_2+s_3\right)}{2 s_1 \left(s_1+s_2\right) s_4}+\frac{k_6\left(s_1+s_2+s_3\right)}{\left(s_2+s_3\right) \left(s_3+s_4\right) \left(s_2+s_3+s_4\right)}+\frac{s_2 k_6\left(s_1+s_2+s_3\right)}{2 s_3 \left(s_2+s_3\right) \left(s_3+s_4\right) \left(s_2+s_3+s_4\right)}+\frac{s_4 k_6\left(s_1+s_2+s_3\right)}{2 s_3 \left(s_2+s_3\right) \left(s_3+s_4\right) \left(s_2+s_3+s_4\right)}+\frac{s_2 k_6\left(s_1+s_2+s_3\right)}{2 \left(s_2+s_3\right) s_4 \left(s_3+s_4\right) \left(s_2+s_3+s_4\right)}+\frac{s_3 k_6\left(s_1+s_2+s_3\right)}{2 \left(s_2+s_3\right) s_4 \left(s_3+s_4\right) \left(s_2+s_3+s_4\right)}-\frac{1}{2} G_3\left(s_1,s_2,s_3\right) k_6\left(s_4\right)+\frac{G_1\left(s_1\right) k_6\left(s_3+s_4\right)}{2 s_2 \left(s_2+s_3\right)}+\frac{G_1\left(s_1\right) k_6\left(s_3+s_4\right)}{2 s_3 \left(s_2+s_3\right)}+\frac{G_2\left(s_1,s_2\right) k_6\left(s_3+s_4\right)}{2 s_3}+\frac{k_6\left(s_3+s_4\right)}{\left(s_1+s_2\right) \left(s_2+s_3\right) \left(s_1+s_2+s_3\right)}+\frac{s_1 k_6\left(s_3+s_4\right)}{2 s_2 \left(s_1+s_2\right) \left(s_2+s_3\right) \left(s_1+s_2+s_3\right)}+\frac{s_3 k_6\left(s_3+s_4\right)}{2 s_2 \left(s_1+s_2\right) \left(s_2+s_3\right) \left(s_1+s_2+s_3\right)}+\frac{s_1 k_6\left(s_3+s_4\right)}{2 \left(s_1+s_2\right) s_3 \left(s_2+s_3\right) \left(s_1+s_2+s_3\right)}+\frac{s_2 k_6\left(s_3+s_4\right)}{2 \left(s_1+s_2\right) s_3 \left(s_2+s_3\right) \left(s_1+s_2+s_3\right)}+\frac{G_1\left(s_1\right) k_6\left(s_2+s_3+s_4\right)}{2 s_2 s_4}+\frac{k_6\left(s_2+s_3+s_4\right)}{2 s_1 \left(s_1+s_2\right) s_4}+\frac{k_6\left(s_2+s_3+s_4\right)}{2 s_2 \left(s_1+s_2\right) s_4}+\frac{s_2 k_6\left(s_1+s_2+s_3+s_4\right)}{2 s_1 \left(s_1+s_2\right) \left(s_2+s_3\right) \left(s_1+s_2+s_3\right)}+\frac{s_3 k_6\left(s_1+s_2+s_3+s_4\right)}{2 s_1 \left(s_1+s_2\right) \left(s_2+s_3\right) \left(s_1+s_2+s_3\right)}+\frac{k_6\left(s_1+s_2+s_3+s_4\right)}{2 s_1 s_4 \left(s_3+s_4\right)}+\frac{e^{s_1} s_3 k_7\left(-s_1\right)}{2 s_2 \left(s_2+s_3\right) \left(s_3+s_4\right) \left(s_2+s_3+s_4\right)}+\frac{e^{s_1} s_4 k_7\left(-s_1\right)}{2 s_2 \left(s_2+s_3\right) \left(s_3+s_4\right) \left(s_2+s_3+s_4\right)}+\frac{e^{s_1+s_2} k_7\left(-s_1-s_2\right)}{2 s_1 s_3 \left(s_3+s_4\right)}+\frac{e^{s_2+s_3} G_1\left(s_1\right) k_7\left(-s_2-s_3\right)}{2 s_3 \left(s_3+s_4\right)}+\frac{e^{s_2+s_3} G_1\left(s_1\right) k_7\left(-s_2-s_3\right)}{2 s_4 \left(s_3+s_4\right)}+\frac{e^{s_2+s_3} k_7\left(-s_2-s_3\right)}{2 s_1 s_3 \left(s_3+s_4\right)}+\frac{e^{s_2+s_3} k_7\left(-s_2-s_3\right)}{2 s_1 s_4 \left(s_3+s_4\right)}+\frac{e^{s_1+s_2+s_3} k_7\left(-s_1-s_2-s_3\right)}{2 s_1 \left(s_1+s_2\right) s_4}+\frac{e^{s_1+s_2+s_3} k_7\left(-s_1-s_2-s_3\right)}{\left(s_2+s_3\right) \left(s_3+s_4\right) \left(s_2+s_3+s_4\right)}+\frac{e^{s_1+s_2+s_3} s_2 k_7\left(-s_1-s_2-s_3\right)}{2 s_3 \left(s_2+s_3\right) \left(s_3+s_4\right) \left(s_2+s_3+s_4\right)}+\frac{e^{s_1+s_2+s_3} s_4 k_7\left(-s_1-s_2-s_3\right)}{2 s_3 \left(s_2+s_3\right) \left(s_3+s_4\right) \left(s_2+s_3+s_4\right)}+\frac{e^{s_1+s_2+s_3} s_2 k_7\left(-s_1-s_2-s_3\right)}{2 \left(s_2+s_3\right) s_4 \left(s_3+s_4\right) \left(s_2+s_3+s_4\right)}+\frac{e^{s_1+s_2+s_3} s_3 k_7\left(-s_1-s_2-s_3\right)}{2 \left(s_2+s_3\right) s_4 \left(s_3+s_4\right) \left(s_2+s_3+s_4\right)}+\frac{e^{s_3} G_1\left(s_1\right) k_7\left(-s_3\right)}{2 s_2 s_4}+\frac{e^{s_3} G_2\left(s_1,s_2\right) k_7\left(-s_3\right)}{2 s_4}+\frac{e^{s_3} k_7\left(-s_3\right)}{2 s_2 \left(s_1+s_2\right) s_4}+\frac{e^{s_3+s_4} G_1\left(s_1\right) k_7\left(-s_3-s_4\right)}{2 s_2 \left(s_2+s_3\right)}+\frac{e^{s_3+s_4} G_1\left(s_1\right) k_7\left(-s_3-s_4\right)}{2 s_3 \left(s_2+s_3\right)}+\frac{e^{s_3+s_4} G_2\left(s_1,s_2\right) k_7\left(-s_3-s_4\right)}{2 s_3}+\frac{e^{s_3+s_4} k_7\left(-s_3-s_4\right)}{\left(s_1+s_2\right) \left(s_2+s_3\right) \left(s_1+s_2+s_3\right)}+\frac{e^{s_3+s_4} s_1 k_7\left(-s_3-s_4\right)}{2 s_2 \left(s_1+s_2\right) \left(s_2+s_3\right) \left(s_1+s_2+s_3\right)}+\frac{e^{s_3+s_4} s_3 k_7\left(-s_3-s_4\right)}{2 s_2 \left(s_1+s_2\right) \left(s_2+s_3\right) \left(s_1+s_2+s_3\right)}+\frac{e^{s_3+s_4} s_1 k_7\left(-s_3-s_4\right)}{2 \left(s_1+s_2\right) s_3 \left(s_2+s_3\right) \left(s_1+s_2+s_3\right)}+\frac{e^{s_3+s_4} s_2 k_7\left(-s_3-s_4\right)}{2 \left(s_1+s_2\right) s_3 \left(s_2+s_3\right) \left(s_1+s_2+s_3\right)}+\frac{e^{s_2+s_3+s_4} G_1\left(s_1\right) k_7\left(-s_2-s_3-s_4\right)}{2 s_2 s_4}+\frac{e^{s_2+s_3+s_4} k_7\left(-s_2-s_3-s_4\right)}{2 s_1 \left(s_1+s_2\right) s_4}+\frac{e^{s_2+s_3+s_4} k_7\left(-s_2-s_3-s_4\right)}{2 s_2 \left(s_1+s_2\right) s_4}+\frac{e^{s_1+s_2+s_3+s_4} s_2 k_7\left(-s_1-s_2-s_3-s_4\right)}{2 s_1 \left(s_1+s_2\right) \left(s_2+s_3\right) \left(s_1+s_2+s_3\right)}+\frac{e^{s_1+s_2+s_3+s_4} s_3 k_7\left(-s_1-s_2-s_3-s_4\right)}{2 s_1 \left(s_1+s_2\right) \left(s_2+s_3\right) \left(s_1+s_2+s_3\right)}+\frac{e^{s_1+s_2+s_3+s_4} k_7\left(-s_1-s_2-s_3-s_4\right)}{2 s_1 s_4 \left(s_3+s_4\right)}-\frac{1}{2} e^{s_4} G_3\left(s_1,s_2,s_3\right) k_7\left(-s_4\right)+\frac{k_8\left(s_1,s_2+s_3+s_4\right)}{4 s_2 \left(s_2+s_3\right)}+\frac{G_1\left(s_1\right) k_8\left(s_2,s_3+s_4\right)}{4 s_4}+\frac{k_8\left(s_2,s_3+s_4\right)}{4 s_1 s_4}+\frac{k_8\left(s_1+s_2,s_3\right)}{4 s_1 s_4}+\frac{k_8\left(s_1+s_2+s_3,s_4\right)}{4 s_3 \left(s_2+s_3\right)}+\frac{k_9\left(s_1,s_2\right)}{8 s_3 \left(s_3+s_4\right)}+\frac{k_9\left(s_1,s_2+s_3\right)}{8 s_2 s_4}+\frac{k_9\left(s_1,s_2+s_3+s_4\right)}{8 s_4 \left(s_3+s_4\right)}+\frac{k_9\left(s_1+s_2,s_3+s_4\right)}{8 s_1 s_3}+\frac{k_9\left(s_1+s_2,s_3+s_4\right)}{8 s_2 s_4}+\frac{G_1\left(s_1\right) k_9\left(s_3,s_4\right)}{8 s_2}+\frac{1}{8} G_2\left(s_1,s_2\right) k_9\left(s_3,s_4\right)+\frac{k_9\left(s_3,s_4\right)}{8 s_2 \left(s_1+s_2\right)}+\frac{G_1\left(s_1\right) k_9\left(s_2+s_3,s_4\right)}{8 s_3}+\frac{k_9\left(s_2+s_3,s_4\right)}{8 s_1 s_3}+\frac{k_9\left(s_1+s_2+s_3,s_4\right)}{8 s_1 \left(s_1+s_2\right)}+\frac{e^{s_1} k_{10}\left(s_2,-s_1-s_2\right)}{8 s_3 \left(s_3+s_4\right)}+\frac{e^{s_1} k_{10}\left(s_2+s_3,-s_1-s_2-s_3\right)}{8 s_2 s_4}+\frac{e^{s_3} G_1\left(s_1\right) k_{10}\left(s_4,-s_3-s_4\right)}{8 s_2}+\frac{1}{8} e^{s_3} G_2\left(s_1,s_2\right) k_{10}\left(s_4,-s_3-s_4\right)+\frac{e^{s_3} k_{10}\left(s_4,-s_3-s_4\right)}{8 s_2 \left(s_1+s_2\right)}+\frac{e^{s_2+s_3} G_1\left(s_1\right) k_{10}\left(s_4,-s_2-s_3-s_4\right)}{8 s_3}+\frac{e^{s_2+s_3} k_{10}\left(s_4,-s_2-s_3-s_4\right)}{8 s_1 s_3}+\frac{e^{s_1+s_2+s_3} k_{10}\left(s_4,-s_1-s_2-s_3-s_4\right)}{8 s_1 \left(s_1+s_2\right)}+\frac{e^{s_1+s_2} k_{10}\left(s_3+s_4,-s_1-s_2-s_3-s_4\right)}{8 s_1 s_3}+\frac{e^{s_1+s_2} k_{10}\left(s_3+s_4,-s_1-s_2-s_3-s_4\right)}{8 s_2 s_4}+\frac{e^{s_1} k_{10}\left(s_2+s_3+s_4,-s_1-s_2-s_3-s_4\right)}{8 s_4 \left(s_3+s_4\right)}+\frac{e^{s_1+s_2} k_{11}\left(s_3,-s_1-s_2-s_3\right)}{4 s_1 s_4}+\frac{e^{s_1+s_2+s_3} k_{11}\left(s_4,-s_1-s_2-s_3-s_4\right)}{4 s_3 \left(s_2+s_3\right)}+\frac{e^{s_2} G_1\left(s_1\right) k_{11}\left(s_3+s_4,-s_2-s_3-s_4\right)}{4 s_4}+\frac{e^{s_2} k_{11}\left(s_3+s_4,-s_2-s_3-s_4\right)}{4 s_1 s_4}+\frac{e^{s_1} k_{11}\left(s_2+s_3+s_4,-s_1-s_2-s_3-s_4\right)}{4 s_2 \left(s_2+s_3\right)}+\frac{e^{s_1+s_2+s_3} k_{12}\left(-s_1-s_2-s_3,s_1+s_2\right)}{4 s_1 s_4}+\frac{e^{s_2+s_3+s_4} G_1\left(s_1\right) k_{12}\left(-s_2-s_3-s_4,s_2\right)}{4 s_4}+\frac{e^{s_2+s_3+s_4} k_{12}\left(-s_2-s_3-s_4,s_2\right)}{4 s_1 s_4}+\frac{e^{s_1+s_2+s_3+s_4} k_{12}\left(-s_1-s_2-s_3-s_4,s_1\right)}{4 s_2 \left(s_2+s_3\right)}+\frac{e^{s_1+s_2+s_3+s_4} k_{12}\left(-s_1-s_2-s_3-s_4,s_1+s_2+s_3\right)}{4 s_3 \left(s_2+s_3\right)}+\frac{e^{s_1+s_2} k_{13}\left(-s_1-s_2,s_1\right)}{8 s_3 \left(s_3+s_4\right)}+\frac{e^{s_1+s_2+s_3} k_{13}\left(-s_1-s_2-s_3,s_1\right)}{8 s_2 s_4}+\frac{e^{s_3+s_4} G_1\left(s_1\right) k_{13}\left(-s_3-s_4,s_3\right)}{8 s_2}+\frac{1}{8} e^{s_3+s_4} G_2\left(s_1,s_2\right) k_{13}\left(-s_3-s_4,s_3\right)+\frac{e^{s_3+s_4} k_{13}\left(-s_3-s_4,s_3\right)}{8 s_2 \left(s_1+s_2\right)}+\frac{e^{s_2+s_3+s_4} G_1\left(s_1\right) k_{13}\left(-s_2-s_3-s_4,s_2+s_3\right)}{8 s_3}+\frac{e^{s_2+s_3+s_4} k_{13}\left(-s_2-s_3-s_4,s_2+s_3\right)}{8 s_1 s_3}+\frac{e^{s_1+s_2+s_3+s_4} k_{13}\left(-s_1-s_2-s_3-s_4,s_1\right)}{8 s_4 \left(s_3+s_4\right)}+\frac{e^{s_1+s_2+s_3+s_4} k_{13}\left(-s_1-s_2-s_3-s_4,s_1+s_2\right)}{8 s_1 s_3}+\frac{e^{s_1+s_2+s_3+s_4} k_{13}\left(-s_1-s_2-s_3-s_4,s_1+s_2\right)}{8 s_2 s_4}+\frac{e^{s_1+s_2+s_3+s_4} k_{13}\left(-s_1-s_2-s_3-s_4,s_1+s_2+s_3\right)}{8 s_1 \left(s_1+s_2\right)}+\frac{e^{s_1} k_{17}\left(s_2,s_3,-s_1-s_2-s_3\right)}{16 s_4}-\frac{1}{16} G_1\left(s_1\right) k_{17}\left(s_2,s_3,s_4\right)+\frac{e^{s_1} k_{17}\left(s_2,s_3,s_4\right)}{16 \left(s_1+s_2+s_3+s_4\right)}+\frac{k_{17}\left(s_1+s_2,s_3,s_4\right)}{16 s_1}+\frac{e^{s_1+s_2+s_3} k_{17}\left(-s_1-s_2-s_3,s_1,s_2\right)}{16 s_4}-\frac{1}{16} e^{s_2+s_3} G_1\left(s_1\right) k_{17}\left(s_4,-s_2-s_3-s_4,s_2\right)+\frac{e^{s_1+s_2+s_3} k_{17}\left(s_4,-s_2-s_3-s_4,s_2\right)}{16 \left(s_1+s_2+s_3+s_4\right)}+\frac{e^{s_1+s_2+s_3} k_{17}\left(s_4,-s_1-s_2-s_3-s_4,s_1+s_2\right)}{16 s_1}+\frac{k_{18}\left(s_1,s_2,s_3+s_4\right)}{16 s_3}+\frac{k_{18}\left(s_1,s_2+s_3,s_4\right)}{16 s_2}+\frac{e^{s_1} k_{18}\left(s_2,s_3+s_4,-s_1-s_2-s_3-s_4\right)}{16 s_3}+\frac{e^{s_1} k_{18}\left(s_2+s_3,s_4,-s_1-s_2-s_3-s_4\right)}{16 s_2}+\frac{e^{s_1+s_2+s_3+s_4} k_{18}\left(-s_1-s_2-s_3-s_4,s_1,s_2\right)}{16 s_3}+\frac{e^{s_1+s_2+s_3+s_4} k_{18}\left(-s_1-s_2-s_3-s_4,s_1,s_2+s_3\right)}{16 s_2}+\frac{e^{s_1+s_2+s_3} k_{18}\left(s_4,-s_1-s_2-s_3-s_4,s_1\right)}{16 s_2}+\frac{e^{s_1+s_2} k_{18}\left(s_3+s_4,-s_1-s_2-s_3-s_4,s_1\right)}{16 s_3}+\frac{k_{19}\left(s_1,s_2,s_3\right)}{16 s_4}+\frac{e^{s_1+s_2} k_{19}\left(s_3,-s_1-s_2-s_3,s_1\right)}{16 s_4}-\frac{1}{16} e^{s_2} G_1\left(s_1\right) k_{19}\left(s_3,s_4,-s_2-s_3-s_4\right)+\frac{e^{s_1+s_2} k_{19}\left(s_3,s_4,-s_2-s_3-s_4\right)}{16 \left(s_1+s_2+s_3+s_4\right)}+\frac{e^{s_1+s_2} k_{19}\left(s_3,s_4,-s_1-s_2-s_3-s_4\right)}{16 s_1}-\frac{1}{16} e^{s_2+s_3+s_4} G_1\left(s_1\right) k_{19}\left(-s_2-s_3-s_4,s_2,s_3\right)+\frac{e^{s_1+s_2+s_3+s_4} k_{19}\left(-s_2-s_3-s_4,s_2,s_3\right)}{16 \left(s_1+s_2+s_3+s_4\right)}+\frac{e^{s_1+s_2+s_3+s_4} k_{19}\left(-s_1-s_2-s_3-s_4,s_1+s_2,s_3\right)}{16 s_1}-\frac{k_{17}\left(s_2,s_3,s_4\right)}{16 s_1}-\frac{e^{s_2+s_3} k_{17}\left(s_4,-s_2-s_3-s_4,s_2\right)}{16 s_1}-\frac{e^{s_2} k_{19}\left(s_3,s_4,-s_2-s_3-s_4\right)}{16 s_1}-\frac{e^{s_2+s_3+s_4} k_{19}\left(-s_2-s_3-s_4,s_2,s_3\right)}{16 s_1}-\frac{G_1\left(s_1\right) k_9\left(s_2+s_3,s_4\right)}{8 s_2}-\frac{e^{s_2+s_3} G_1\left(s_1\right) k_{10}\left(s_4,-s_2-s_3-s_4\right)}{8 s_2}-\frac{e^{s_2+s_3+s_4} G_1\left(s_1\right) k_{13}\left(-s_2-s_3-s_4,s_2+s_3\right)}{8 s_2}-\frac{k_{18}\left(s_1+s_2,s_3,s_4\right)}{16 s_2}-\frac{e^{s_1+s_2} k_{18}\left(s_3,s_4,-s_1-s_2-s_3-s_4\right)}{16 s_2}-\frac{e^{s_1+s_2+s_3+s_4} k_{18}\left(-s_1-s_2-s_3-s_4,s_1+s_2,s_3\right)}{16 s_2}-\frac{e^{s_1+s_2+s_3} k_{18}\left(s_4,-s_1-s_2-s_3-s_4,s_1+s_2\right)}{16 s_2}-\frac{k_9\left(s_2+s_3,s_4\right)}{8 s_1 \left(s_1+s_2\right)}-\frac{e^{s_2+s_3} k_{10}\left(s_4,-s_2-s_3-s_4\right)}{8 s_1 \left(s_1+s_2\right)}-\frac{e^{s_2+s_3+s_4} k_{13}\left(-s_2-s_3-s_4,s_2+s_3\right)}{8 s_1 \left(s_1+s_2\right)}-\frac{k_9\left(s_2+s_3,s_4\right)}{8 s_2 \left(s_1+s_2\right)}-\frac{e^{s_2+s_3} k_{10}\left(s_4,-s_2-s_3-s_4\right)}{8 s_2 \left(s_1+s_2\right)}-\frac{e^{s_2+s_3+s_4} k_{13}\left(-s_2-s_3-s_4,s_2+s_3\right)}{8 s_2 \left(s_1+s_2\right)}-\frac{G_2\left(s_1,s_2\right) k_6\left(s_4\right)}{2 s_3}-\frac{e^{s_4} G_2\left(s_1,s_2\right) k_7\left(-s_4\right)}{2 s_3}-\frac{G_1\left(s_1\right) k_9\left(s_2,s_3+s_4\right)}{8 s_3}-\frac{e^{s_2} G_1\left(s_1\right) k_{10}\left(s_3+s_4,-s_2-s_3-s_4\right)}{8 s_3}-\frac{e^{s_2+s_3+s_4} G_1\left(s_1\right) k_{13}\left(-s_2-s_3-s_4,s_2\right)}{8 s_3}-\frac{k_{18}\left(s_1,s_2+s_3,s_4\right)}{16 s_3}-\frac{e^{s_1} k_{18}\left(s_2+s_3,s_4,-s_1-s_2-s_3-s_4\right)}{16 s_3}-\frac{e^{s_1+s_2+s_3+s_4} k_{18}\left(-s_1-s_2-s_3-s_4,s_1,s_2+s_3\right)}{16 s_3}-\frac{e^{s_1+s_2+s_3} k_{18}\left(s_4,-s_1-s_2-s_3-s_4,s_1\right)}{16 s_3}-\frac{k_9\left(s_2,s_3+s_4\right)}{8 s_1 s_3}-\frac{k_9\left(s_1+s_2+s_3,s_4\right)}{8 s_1 s_3}-\frac{e^{s_1+s_2+s_3} k_{10}\left(s_4,-s_1-s_2-s_3-s_4\right)}{8 s_1 s_3}-\frac{e^{s_2} k_{10}\left(s_3+s_4,-s_2-s_3-s_4\right)}{8 s_1 s_3}-\frac{e^{s_2+s_3+s_4} k_{13}\left(-s_2-s_3-s_4,s_2\right)}{8 s_1 s_3}-\frac{e^{s_1+s_2+s_3+s_4} k_{13}\left(-s_1-s_2-s_3-s_4,s_1+s_2+s_3\right)}{8 s_1 s_3}-\frac{G_1\left(s_1\right) k_6\left(s_2+s_3+s_4\right)}{2 s_2 \left(s_2+s_3\right)}-\frac{e^{s_2+s_3+s_4} G_1\left(s_1\right) k_7\left(-s_2-s_3-s_4\right)}{2 s_2 \left(s_2+s_3\right)}-\frac{k_8\left(s_1+s_2,s_3+s_4\right)}{4 s_2 \left(s_2+s_3\right)}-\frac{e^{s_1+s_2} k_{11}\left(s_3+s_4,-s_1-s_2-s_3-s_4\right)}{4 s_2 \left(s_2+s_3\right)}-\frac{e^{s_1+s_2+s_3+s_4} k_{12}\left(-s_1-s_2-s_3-s_4,s_1+s_2\right)}{4 s_2 \left(s_2+s_3\right)}-\frac{G_1\left(s_1\right) k_6\left(s_4\right)}{2 s_3 \left(s_2+s_3\right)}-\frac{e^{s_4} G_1\left(s_1\right) k_7\left(-s_4\right)}{2 s_3 \left(s_2+s_3\right)}-\frac{k_8\left(s_1+s_2,s_3+s_4\right)}{4 s_3 \left(s_2+s_3\right)}-\frac{e^{s_1+s_2} k_{11}\left(s_3+s_4,-s_1-s_2-s_3-s_4\right)}{4 s_3 \left(s_2+s_3\right)}-\frac{e^{s_1+s_2+s_3+s_4} k_{12}\left(-s_1-s_2-s_3-s_4,s_1+s_2\right)}{4 s_3 \left(s_2+s_3\right)}-\frac{k_6\left(s_2+s_3+s_4\right)}{\left(s_1+s_2\right) \left(s_2+s_3\right) \left(s_1+s_2+s_3\right)}-\frac{e^{s_2+s_3+s_4} k_7\left(-s_2-s_3-s_4\right)}{\left(s_1+s_2\right) \left(s_2+s_3\right) \left(s_1+s_2+s_3\right)}-\frac{s_2 k_6\left(s_2+s_3+s_4\right)}{2 s_1 \left(s_1+s_2\right) \left(s_2+s_3\right) \left(s_1+s_2+s_3\right)}-\frac{s_3 k_6\left(s_2+s_3+s_4\right)}{2 s_1 \left(s_1+s_2\right) \left(s_2+s_3\right) \left(s_1+s_2+s_3\right)}-\frac{e^{s_2+s_3+s_4} s_2 k_7\left(-s_2-s_3-s_4\right)}{2 s_1 \left(s_1+s_2\right) \left(s_2+s_3\right) \left(s_1+s_2+s_3\right)}-\frac{e^{s_2+s_3+s_4} s_3 k_7\left(-s_2-s_3-s_4\right)}{2 s_1 \left(s_1+s_2\right) \left(s_2+s_3\right) \left(s_1+s_2+s_3\right)}-\frac{s_1 k_6\left(s_2+s_3+s_4\right)}{2 s_2 \left(s_1+s_2\right) \left(s_2+s_3\right) \left(s_1+s_2+s_3\right)}-\frac{s_3 k_6\left(s_2+s_3+s_4\right)}{2 s_2 \left(s_1+s_2\right) \left(s_2+s_3\right) \left(s_1+s_2+s_3\right)}-\frac{e^{s_2+s_3+s_4} s_1 k_7\left(-s_2-s_3-s_4\right)}{2 s_2 \left(s_1+s_2\right) \left(s_2+s_3\right) \left(s_1+s_2+s_3\right)}-\frac{e^{s_2+s_3+s_4} s_3 k_7\left(-s_2-s_3-s_4\right)}{2 s_2 \left(s_1+s_2\right) \left(s_2+s_3\right) \left(s_1+s_2+s_3\right)}-\frac{s_1 k_6\left(s_4\right)}{2 \left(s_1+s_2\right) s_3 \left(s_2+s_3\right) \left(s_1+s_2+s_3\right)}-\frac{s_2 k_6\left(s_4\right)}{2 \left(s_1+s_2\right) s_3 \left(s_2+s_3\right) \left(s_1+s_2+s_3\right)}-\frac{e^{s_4} s_1 k_7\left(-s_4\right)}{2 \left(s_1+s_2\right) s_3 \left(s_2+s_3\right) \left(s_1+s_2+s_3\right)}-\frac{e^{s_4} s_2 k_7\left(-s_4\right)}{2 \left(s_1+s_2\right) s_3 \left(s_2+s_3\right) \left(s_1+s_2+s_3\right)}-\frac{G_2\left(s_1,s_2\right) k_6\left(s_3+s_4\right)}{2 s_4}-\frac{e^{s_3+s_4} G_2\left(s_1,s_2\right) k_7\left(-s_3-s_4\right)}{2 s_4}-\frac{G_1\left(s_1\right) k_8\left(s_2,s_3\right)}{4 s_4}-\frac{e^{s_2} G_1\left(s_1\right) k_{11}\left(s_3,-s_2-s_3\right)}{4 s_4}-\frac{e^{s_2+s_3} G_1\left(s_1\right) k_{12}\left(-s_2-s_3,s_2\right)}{4 s_4}-\frac{e^{s_1} k_{17}\left(s_2,s_3+s_4,-s_1-s_2-s_3-s_4\right)}{16 s_4}-\frac{e^{s_1+s_2+s_3+s_4} k_{17}\left(-s_1-s_2-s_3-s_4,s_1,s_2\right)}{16 s_4}-\frac{k_{19}\left(s_1,s_2,s_3+s_4\right)}{16 s_4}-\frac{e^{s_1+s_2} k_{19}\left(s_3+s_4,-s_1-s_2-s_3-s_4,s_1\right)}{16 s_4}-\frac{k_8\left(s_2,s_3\right)}{4 s_1 s_4}-\frac{k_8\left(s_1+s_2,s_3+s_4\right)}{4 s_1 s_4}-\frac{e^{s_2} k_{11}\left(s_3,-s_2-s_3\right)}{4 s_1 s_4}-\frac{e^{s_1+s_2} k_{11}\left(s_3+s_4,-s_1-s_2-s_3-s_4\right)}{4 s_1 s_4}-\frac{e^{s_2+s_3} k_{12}\left(-s_2-s_3,s_2\right)}{4 s_1 s_4}-\frac{e^{s_1+s_2+s_3+s_4} k_{12}\left(-s_1-s_2-s_3-s_4,s_1+s_2\right)}{4 s_1 s_4}-\frac{G_1\left(s_1\right) k_6\left(s_2+s_3\right)}{2 s_2 s_4}-\frac{G_1\left(s_1\right) k_6\left(s_3+s_4\right)}{2 s_2 s_4}-\frac{e^{s_2+s_3} G_1\left(s_1\right) k_7\left(-s_2-s_3\right)}{2 s_2 s_4}-\frac{e^{s_3+s_4} G_1\left(s_1\right) k_7\left(-s_3-s_4\right)}{2 s_2 s_4}-\frac{k_9\left(s_1,s_2+s_3+s_4\right)}{8 s_2 s_4}-\frac{k_9\left(s_1+s_2,s_3\right)}{8 s_2 s_4}-\frac{e^{s_1+s_2} k_{10}\left(s_3,-s_1-s_2-s_3\right)}{8 s_2 s_4}-\frac{e^{s_1} k_{10}\left(s_2+s_3+s_4,-s_1-s_2-s_3-s_4\right)}{8 s_2 s_4}-\frac{e^{s_1+s_2+s_3} k_{13}\left(-s_1-s_2-s_3,s_1+s_2\right)}{8 s_2 s_4}-\frac{e^{s_1+s_2+s_3+s_4} k_{13}\left(-s_1-s_2-s_3-s_4,s_1\right)}{8 s_2 s_4}-\frac{k_6\left(s_2+s_3\right)}{2 s_1 \left(s_1+s_2\right) s_4}-\frac{k_6\left(s_1+s_2+s_3+s_4\right)}{2 s_1 \left(s_1+s_2\right) s_4}-\frac{e^{s_2+s_3} k_7\left(-s_2-s_3\right)}{2 s_1 \left(s_1+s_2\right) s_4}-\frac{e^{s_1+s_2+s_3+s_4} k_7\left(-s_1-s_2-s_3-s_4\right)}{2 s_1 \left(s_1+s_2\right) s_4}-\frac{k_6\left(s_2+s_3\right)}{2 s_2 \left(s_1+s_2\right) s_4}-\frac{k_6\left(s_3+s_4\right)}{2 s_2 \left(s_1+s_2\right) s_4}-\frac{e^{s_2+s_3} k_7\left(-s_2-s_3\right)}{2 s_2 \left(s_1+s_2\right) s_4}-\frac{e^{s_3+s_4} k_7\left(-s_3-s_4\right)}{2 s_2 \left(s_1+s_2\right) s_4}-\frac{G_1\left(s_1\right) k_6\left(s_2\right)}{2 s_3 \left(s_3+s_4\right)}-\frac{e^{s_2} G_1\left(s_1\right) k_7\left(-s_2\right)}{2 s_3 \left(s_3+s_4\right)}-\frac{k_9\left(s_1,s_2+s_3\right)}{8 s_3 \left(s_3+s_4\right)}-\frac{e^{s_1} k_{10}\left(s_2+s_3,-s_1-s_2-s_3\right)}{8 s_3 \left(s_3+s_4\right)}-\frac{e^{s_1+s_2+s_3} k_{13}\left(-s_1-s_2-s_3,s_1\right)}{8 s_3 \left(s_3+s_4\right)}-\frac{k_6\left(s_2\right)}{2 s_1 s_3 \left(s_3+s_4\right)}-\frac{k_6\left(s_1+s_2+s_3\right)}{2 s_1 s_3 \left(s_3+s_4\right)}-\frac{e^{s_2} k_7\left(-s_2\right)}{2 s_1 s_3 \left(s_3+s_4\right)}-\frac{e^{s_1+s_2+s_3} k_7\left(-s_1-s_2-s_3\right)}{2 s_1 s_3 \left(s_3+s_4\right)}-\frac{G_1\left(s_1\right) k_6\left(s_2+s_3+s_4\right)}{2 s_4 \left(s_3+s_4\right)}-\frac{e^{s_2+s_3+s_4} G_1\left(s_1\right) k_7\left(-s_2-s_3-s_4\right)}{2 s_4 \left(s_3+s_4\right)}-\frac{k_9\left(s_1,s_2+s_3\right)}{8 s_4 \left(s_3+s_4\right)}-\frac{e^{s_1} k_{10}\left(s_2+s_3,-s_1-s_2-s_3\right)}{8 s_4 \left(s_3+s_4\right)}-\frac{e^{s_1+s_2+s_3} k_{13}\left(-s_1-s_2-s_3,s_1\right)}{8 s_4 \left(s_3+s_4\right)}-\frac{k_6\left(s_1+s_2+s_3\right)}{2 s_1 s_4 \left(s_3+s_4\right)}-\frac{k_6\left(s_2+s_3+s_4\right)}{2 s_1 s_4 \left(s_3+s_4\right)}-\frac{e^{s_1+s_2+s_3} k_7\left(-s_1-s_2-s_3\right)}{2 s_1 s_4 \left(s_3+s_4\right)}-\frac{e^{s_2+s_3+s_4} k_7\left(-s_2-s_3-s_4\right)}{2 s_1 s_4 \left(s_3+s_4\right)}-\frac{k_6\left(s_1+s_2\right)}{\left(s_2+s_3\right) \left(s_3+s_4\right) \left(s_2+s_3+s_4\right)}-\frac{e^{s_1+s_2} k_7\left(-s_1-s_2\right)}{\left(s_2+s_3\right) \left(s_3+s_4\right) \left(s_2+s_3+s_4\right)}-\frac{s_3 k_6\left(s_1+s_2\right)}{2 s_2 \left(s_2+s_3\right) \left(s_3+s_4\right) \left(s_2+s_3+s_4\right)}-\frac{s_4 k_6\left(s_1+s_2\right)}{2 s_2 \left(s_2+s_3\right) \left(s_3+s_4\right) \left(s_2+s_3+s_4\right)}-\frac{e^{s_1+s_2} s_3 k_7\left(-s_1-s_2\right)}{2 s_2 \left(s_2+s_3\right) \left(s_3+s_4\right) \left(s_2+s_3+s_4\right)}-\frac{e^{s_1+s_2} s_4 k_7\left(-s_1-s_2\right)}{2 s_2 \left(s_2+s_3\right) \left(s_3+s_4\right) \left(s_2+s_3+s_4\right)}-\frac{s_2 k_6\left(s_1+s_2\right)}{2 s_3 \left(s_2+s_3\right) \left(s_3+s_4\right) \left(s_2+s_3+s_4\right)}-\frac{s_4 k_6\left(s_1+s_2\right)}{2 s_3 \left(s_2+s_3\right) \left(s_3+s_4\right) \left(s_2+s_3+s_4\right)}-\frac{e^{s_1+s_2} s_2 k_7\left(-s_1-s_2\right)}{2 s_3 \left(s_2+s_3\right) \left(s_3+s_4\right) \left(s_2+s_3+s_4\right)}-\frac{e^{s_1+s_2} s_4 k_7\left(-s_1-s_2\right)}{2 s_3 \left(s_2+s_3\right) \left(s_3+s_4\right) \left(s_2+s_3+s_4\right)}-\frac{s_2 k_6\left(s_1+s_2+s_3+s_4\right)}{2 \left(s_2+s_3\right) s_4 \left(s_3+s_4\right) \left(s_2+s_3+s_4\right)}-\frac{s_3 k_6\left(s_1+s_2+s_3+s_4\right)}{2 \left(s_2+s_3\right) s_4 \left(s_3+s_4\right) \left(s_2+s_3+s_4\right)}-\frac{e^{s_1+s_2+s_3+s_4} s_2 k_7\left(-s_1-s_2-s_3-s_4\right)}{2 \left(s_2+s_3\right) s_4 \left(s_3+s_4\right) \left(s_2+s_3+s_4\right)}-\frac{e^{s_1+s_2+s_3+s_4} s_3 k_7\left(-s_1-s_2-s_3-s_4\right)}{2 \left(s_2+s_3\right) s_4 \left(s_3+s_4\right) \left(s_2+s_3+s_4\right)}-\frac{e^{s_1} k_{17}\left(s_2,s_3,-s_1-s_2-s_3\right)}{16 \left(s_1+s_2+s_3+s_4\right)}-\frac{e^{s_1+s_2+s_3} k_{17}\left(-s_1-s_2-s_3,s_1,s_2\right)}{16 \left(s_1+s_2+s_3+s_4\right)}-\frac{k_{19}\left(s_1,s_2,s_3\right)}{16 \left(s_1+s_2+s_3+s_4\right)}-\frac{e^{s_1+s_2} k_{19}\left(s_3,-s_1-s_2-s_3,s_1\right)}{16 \left(s_1+s_2+s_3+s_4\right)}. 
\end{math}
\end{center}

\subsubsection{The function $\widetilde K_{20}$}
\label{basicK20}
Finally, for the last four variable function we have:
\begin{equation} \label{basicK20eqn}
\widetilde K_{20}(s_1, s_2, s_3, s_4) =
\end{equation}
\begin{center}
\begin{math}
 \frac{1}{5} (-4) \pi  G_4\left(s_1,s_2,s_3,s_4\right)+\frac{3 s_3 k_6\left(s_1\right)}{2 s_2 \left(s_2+s_3\right) \left(s_3+s_4\right) \left(s_2+s_3+s_4\right)}+\frac{3 s_4 k_6\left(s_1\right)}{2 s_2 \left(s_2+s_3\right) \left(s_3+s_4\right) \left(s_2+s_3+s_4\right)}+\frac{3 k_6\left(s_1+s_2\right)}{2 s_1 s_3 \left(s_3+s_4\right)}+\frac{3 G_1\left(s_1\right) k_6\left(s_3\right)}{2 s_2 s_4}+\frac{3 G_2\left(s_1,s_2\right) k_6\left(s_3\right)}{2 s_4}+\frac{3 k_6\left(s_3\right)}{2 s_2 \left(s_1+s_2\right) s_4}+\frac{3 G_1\left(s_1\right) k_6\left(s_2+s_3\right)}{2 s_3 \left(s_3+s_4\right)}+\frac{3 G_1\left(s_1\right) k_6\left(s_2+s_3\right)}{2 s_4 \left(s_3+s_4\right)}+\frac{3 k_6\left(s_2+s_3\right)}{2 s_1 s_3 \left(s_3+s_4\right)}+\frac{3 k_6\left(s_2+s_3\right)}{2 s_1 s_4 \left(s_3+s_4\right)}+\frac{3 k_6\left(s_1+s_2+s_3\right)}{2 s_1 \left(s_1+s_2\right) s_4}+\frac{3 k_6\left(s_1+s_2+s_3\right)}{\left(s_2+s_3\right) \left(s_3+s_4\right) \left(s_2+s_3+s_4\right)}+\frac{3 s_2 k_6\left(s_1+s_2+s_3\right)}{2 s_3 \left(s_2+s_3\right) \left(s_3+s_4\right) \left(s_2+s_3+s_4\right)}+\frac{3 s_4 k_6\left(s_1+s_2+s_3\right)}{2 s_3 \left(s_2+s_3\right) \left(s_3+s_4\right) \left(s_2+s_3+s_4\right)}+\frac{3 s_2 k_6\left(s_1+s_2+s_3\right)}{2 \left(s_2+s_3\right) s_4 \left(s_3+s_4\right) \left(s_2+s_3+s_4\right)}+\frac{3 s_3 k_6\left(s_1+s_2+s_3\right)}{2 \left(s_2+s_3\right) s_4 \left(s_3+s_4\right) \left(s_2+s_3+s_4\right)}-\frac{3}{2} G_3\left(s_1,s_2,s_3\right) k_6\left(s_4\right)+\frac{3 G_1\left(s_1\right) k_6\left(s_3+s_4\right)}{2 s_2 \left(s_2+s_3\right)}+\frac{3 G_1\left(s_1\right) k_6\left(s_3+s_4\right)}{2 s_3 \left(s_2+s_3\right)}+\frac{3 G_2\left(s_1,s_2\right) k_6\left(s_3+s_4\right)}{2 s_3}+\frac{3 k_6\left(s_3+s_4\right)}{\left(s_1+s_2\right) \left(s_2+s_3\right) \left(s_1+s_2+s_3\right)}+\frac{3 s_1 k_6\left(s_3+s_4\right)}{2 s_2 \left(s_1+s_2\right) \left(s_2+s_3\right) \left(s_1+s_2+s_3\right)}+\frac{3 s_3 k_6\left(s_3+s_4\right)}{2 s_2 \left(s_1+s_2\right) \left(s_2+s_3\right) \left(s_1+s_2+s_3\right)}+\frac{3 s_1 k_6\left(s_3+s_4\right)}{2 \left(s_1+s_2\right) s_3 \left(s_2+s_3\right) \left(s_1+s_2+s_3\right)}+\frac{3 s_2 k_6\left(s_3+s_4\right)}{2 \left(s_1+s_2\right) s_3 \left(s_2+s_3\right) \left(s_1+s_2+s_3\right)}+\frac{3 G_1\left(s_1\right) k_6\left(s_2+s_3+s_4\right)}{2 s_2 s_4}+\frac{3 k_6\left(s_2+s_3+s_4\right)}{2 s_1 \left(s_1+s_2\right) s_4}+\frac{3 k_6\left(s_2+s_3+s_4\right)}{2 s_2 \left(s_1+s_2\right) s_4}+\frac{3 s_2 k_6\left(s_1+s_2+s_3+s_4\right)}{2 s_1 \left(s_1+s_2\right) \left(s_2+s_3\right) \left(s_1+s_2+s_3\right)}+\frac{3 s_3 k_6\left(s_1+s_2+s_3+s_4\right)}{2 s_1 \left(s_1+s_2\right) \left(s_2+s_3\right) \left(s_1+s_2+s_3\right)}+\frac{3 k_6\left(s_1+s_2+s_3+s_4\right)}{2 s_1 s_4 \left(s_3+s_4\right)}+\frac{3 e^{s_1} s_3 k_7\left(-s_1\right)}{2 s_2 \left(s_2+s_3\right) \left(s_3+s_4\right) \left(s_2+s_3+s_4\right)}+\frac{3 e^{s_1} s_4 k_7\left(-s_1\right)}{2 s_2 \left(s_2+s_3\right) \left(s_3+s_4\right) \left(s_2+s_3+s_4\right)}+\frac{3 e^{s_1+s_2} k_7\left(-s_1-s_2\right)}{2 s_1 s_3 \left(s_3+s_4\right)}+\frac{3 e^{s_2+s_3} G_1\left(s_1\right) k_7\left(-s_2-s_3\right)}{2 s_3 \left(s_3+s_4\right)}+\frac{3 e^{s_2+s_3} G_1\left(s_1\right) k_7\left(-s_2-s_3\right)}{2 s_4 \left(s_3+s_4\right)}+\frac{3 e^{s_2+s_3} k_7\left(-s_2-s_3\right)}{2 s_1 s_3 \left(s_3+s_4\right)}+\frac{3 e^{s_2+s_3} k_7\left(-s_2-s_3\right)}{2 s_1 s_4 \left(s_3+s_4\right)}+\frac{3 e^{s_1+s_2+s_3} k_7\left(-s_1-s_2-s_3\right)}{2 s_1 \left(s_1+s_2\right) s_4}+\frac{3 e^{s_1+s_2+s_3} k_7\left(-s_1-s_2-s_3\right)}{\left(s_2+s_3\right) \left(s_3+s_4\right) \left(s_2+s_3+s_4\right)}+\frac{3 e^{s_1+s_2+s_3} s_2 k_7\left(-s_1-s_2-s_3\right)}{2 s_3 \left(s_2+s_3\right) \left(s_3+s_4\right) \left(s_2+s_3+s_4\right)}+\frac{3 e^{s_1+s_2+s_3} s_4 k_7\left(-s_1-s_2-s_3\right)}{2 s_3 \left(s_2+s_3\right) \left(s_3+s_4\right) \left(s_2+s_3+s_4\right)}+\frac{3 e^{s_1+s_2+s_3} s_2 k_7\left(-s_1-s_2-s_3\right)}{2 \left(s_2+s_3\right) s_4 \left(s_3+s_4\right) \left(s_2+s_3+s_4\right)}+\frac{3 e^{s_1+s_2+s_3} s_3 k_7\left(-s_1-s_2-s_3\right)}{2 \left(s_2+s_3\right) s_4 \left(s_3+s_4\right) \left(s_2+s_3+s_4\right)}+\frac{3 e^{s_3} G_1\left(s_1\right) k_7\left(-s_3\right)}{2 s_2 s_4}+\frac{3 e^{s_3} G_2\left(s_1,s_2\right) k_7\left(-s_3\right)}{2 s_4}+\frac{3 e^{s_3} k_7\left(-s_3\right)}{2 s_2 \left(s_1+s_2\right) s_4}+\frac{3 e^{s_3+s_4} G_1\left(s_1\right) k_7\left(-s_3-s_4\right)}{2 s_2 \left(s_2+s_3\right)}+\frac{3 e^{s_3+s_4} G_1\left(s_1\right) k_7\left(-s_3-s_4\right)}{2 s_3 \left(s_2+s_3\right)}+\frac{3 e^{s_3+s_4} G_2\left(s_1,s_2\right) k_7\left(-s_3-s_4\right)}{2 s_3}+\frac{3 e^{s_3+s_4} k_7\left(-s_3-s_4\right)}{\left(s_1+s_2\right) \left(s_2+s_3\right) \left(s_1+s_2+s_3\right)}+\frac{3 e^{s_3+s_4} s_1 k_7\left(-s_3-s_4\right)}{2 s_2 \left(s_1+s_2\right) \left(s_2+s_3\right) \left(s_1+s_2+s_3\right)}+\frac{3 e^{s_3+s_4} s_3 k_7\left(-s_3-s_4\right)}{2 s_2 \left(s_1+s_2\right) \left(s_2+s_3\right) \left(s_1+s_2+s_3\right)}+\frac{3 e^{s_3+s_4} s_1 k_7\left(-s_3-s_4\right)}{2 \left(s_1+s_2\right) s_3 \left(s_2+s_3\right) \left(s_1+s_2+s_3\right)}+\frac{3 e^{s_3+s_4} s_2 k_7\left(-s_3-s_4\right)}{2 \left(s_1+s_2\right) s_3 \left(s_2+s_3\right) \left(s_1+s_2+s_3\right)}+\frac{3 e^{s_2+s_3+s_4} G_1\left(s_1\right) k_7\left(-s_2-s_3-s_4\right)}{2 s_2 s_4}+\frac{3 e^{s_2+s_3+s_4} k_7\left(-s_2-s_3-s_4\right)}{2 s_1 \left(s_1+s_2\right) s_4}+\frac{3 e^{s_2+s_3+s_4} k_7\left(-s_2-s_3-s_4\right)}{2 s_2 \left(s_1+s_2\right) s_4}+\frac{3 e^{s_1+s_2+s_3+s_4} s_2 k_7\left(-s_1-s_2-s_3-s_4\right)}{2 s_1 \left(s_1+s_2\right) \left(s_2+s_3\right) \left(s_1+s_2+s_3\right)}+\frac{3 e^{s_1+s_2+s_3+s_4} s_3 k_7\left(-s_1-s_2-s_3-s_4\right)}{2 s_1 \left(s_1+s_2\right) \left(s_2+s_3\right) \left(s_1+s_2+s_3\right)}+\frac{3 e^{s_1+s_2+s_3+s_4} k_7\left(-s_1-s_2-s_3-s_4\right)}{2 s_1 s_4 \left(s_3+s_4\right)}-\frac{3}{2} e^{s_4} G_3\left(s_1,s_2,s_3\right) k_7\left(-s_4\right)+\frac{e^{s_1+s_2} k_{14}\left(-s_1-s_2,s_1\right)}{4 s_3 \left(s_3+s_4\right)}+\frac{e^{s_1+s_2+s_3} k_{14}\left(-s_1-s_2-s_3,s_1\right)}{4 s_2 s_4}+\frac{e^{s_1+s_2+s_3} k_{14}\left(-s_1-s_2-s_3,s_1+s_2\right)}{4 s_1 s_4}+\frac{e^{s_3+s_4} G_1\left(s_1\right) k_{14}\left(-s_3-s_4,s_3\right)}{4 s_2}+\frac{1}{4} e^{s_3+s_4} G_2\left(s_1,s_2\right) k_{14}\left(-s_3-s_4,s_3\right)+\frac{e^{s_3+s_4} k_{14}\left(-s_3-s_4,s_3\right)}{4 s_2 \left(s_1+s_2\right)}+\frac{e^{s_2+s_3+s_4} G_1\left(s_1\right) k_{14}\left(-s_2-s_3-s_4,s_2\right)}{4 s_4}+\frac{e^{s_2+s_3+s_4} k_{14}\left(-s_2-s_3-s_4,s_2\right)}{4 s_1 s_4}+\frac{e^{s_2+s_3+s_4} G_1\left(s_1\right) k_{14}\left(-s_2-s_3-s_4,s_2+s_3\right)}{4 s_3}+\frac{e^{s_2+s_3+s_4} k_{14}\left(-s_2-s_3-s_4,s_2+s_3\right)}{4 s_1 s_3}+\frac{e^{s_1+s_2+s_3+s_4} k_{14}\left(-s_1-s_2-s_3-s_4,s_1\right)}{4 s_2 \left(s_2+s_3\right)}+\frac{e^{s_1+s_2+s_3+s_4} k_{14}\left(-s_1-s_2-s_3-s_4,s_1\right)}{4 s_4 \left(s_3+s_4\right)}+\frac{e^{s_1+s_2+s_3+s_4} k_{14}\left(-s_1-s_2-s_3-s_4,s_1+s_2\right)}{4 s_1 s_3}+\frac{e^{s_1+s_2+s_3+s_4} k_{14}\left(-s_1-s_2-s_3-s_4,s_1+s_2\right)}{4 s_2 s_4}+\frac{e^{s_1+s_2+s_3+s_4} k_{14}\left(-s_1-s_2-s_3-s_4,s_1+s_2+s_3\right)}{4 s_1 \left(s_1+s_2\right)}+\frac{e^{s_1+s_2+s_3+s_4} k_{14}\left(-s_1-s_2-s_3-s_4,s_1+s_2+s_3\right)}{4 s_3 \left(s_2+s_3\right)}+\frac{k_{15}\left(s_1,s_2\right)}{4 s_3 \left(s_3+s_4\right)}+\frac{k_{15}\left(s_1,s_2+s_3\right)}{4 s_2 s_4}+\frac{k_{15}\left(s_1,s_2+s_3+s_4\right)}{4 s_2 \left(s_2+s_3\right)}+\frac{k_{15}\left(s_1,s_2+s_3+s_4\right)}{4 s_4 \left(s_3+s_4\right)}+\frac{G_1\left(s_1\right) k_{15}\left(s_2,s_3+s_4\right)}{4 s_4}+\frac{k_{15}\left(s_2,s_3+s_4\right)}{4 s_1 s_4}+\frac{k_{15}\left(s_1+s_2,s_3\right)}{4 s_1 s_4}+\frac{k_{15}\left(s_1+s_2,s_3+s_4\right)}{4 s_1 s_3}+\frac{k_{15}\left(s_1+s_2,s_3+s_4\right)}{4 s_2 s_4}+\frac{G_1\left(s_1\right) k_{15}\left(s_3,s_4\right)}{4 s_2}+\frac{1}{4} G_2\left(s_1,s_2\right) k_{15}\left(s_3,s_4\right)+\frac{k_{15}\left(s_3,s_4\right)}{4 s_2 \left(s_1+s_2\right)}+\frac{G_1\left(s_1\right) k_{15}\left(s_2+s_3,s_4\right)}{4 s_3}+\frac{k_{15}\left(s_2+s_3,s_4\right)}{4 s_1 s_3}+\frac{k_{15}\left(s_1+s_2+s_3,s_4\right)}{4 s_1 \left(s_1+s_2\right)}+\frac{k_{15}\left(s_1+s_2+s_3,s_4\right)}{4 s_3 \left(s_2+s_3\right)}+\frac{e^{s_1} k_{16}\left(s_2,-s_1-s_2\right)}{4 s_3 \left(s_3+s_4\right)}+\frac{e^{s_1+s_2} k_{16}\left(s_3,-s_1-s_2-s_3\right)}{4 s_1 s_4}+\frac{e^{s_1} k_{16}\left(s_2+s_3,-s_1-s_2-s_3\right)}{4 s_2 s_4}+\frac{e^{s_3} G_1\left(s_1\right) k_{16}\left(s_4,-s_3-s_4\right)}{4 s_2}+\frac{1}{4} e^{s_3} G_2\left(s_1,s_2\right) k_{16}\left(s_4,-s_3-s_4\right)+\frac{e^{s_3} k_{16}\left(s_4,-s_3-s_4\right)}{4 s_2 \left(s_1+s_2\right)}+\frac{e^{s_2+s_3} G_1\left(s_1\right) k_{16}\left(s_4,-s_2-s_3-s_4\right)}{4 s_3}+\frac{e^{s_2+s_3} k_{16}\left(s_4,-s_2-s_3-s_4\right)}{4 s_1 s_3}+\frac{e^{s_1+s_2+s_3} k_{16}\left(s_4,-s_1-s_2-s_3-s_4\right)}{4 s_1 \left(s_1+s_2\right)}+\frac{e^{s_1+s_2+s_3} k_{16}\left(s_4,-s_1-s_2-s_3-s_4\right)}{4 s_3 \left(s_2+s_3\right)}+\frac{e^{s_2} G_1\left(s_1\right) k_{16}\left(s_3+s_4,-s_2-s_3-s_4\right)}{4 s_4}+\frac{e^{s_2} k_{16}\left(s_3+s_4,-s_2-s_3-s_4\right)}{4 s_1 s_4}+\frac{e^{s_1+s_2} k_{16}\left(s_3+s_4,-s_1-s_2-s_3-s_4\right)}{4 s_1 s_3}+\frac{e^{s_1+s_2} k_{16}\left(s_3+s_4,-s_1-s_2-s_3-s_4\right)}{4 s_2 s_4}+\frac{e^{s_1} k_{16}\left(s_2+s_3+s_4,-s_1-s_2-s_3-s_4\right)}{4 s_2 \left(s_2+s_3\right)}+\frac{e^{s_1} k_{16}\left(s_2+s_3+s_4,-s_1-s_2-s_3-s_4\right)}{4 s_4 \left(s_3+s_4\right)}+\frac{k_{20}\left(s_1,s_2,s_3\right)}{16 s_4}+\frac{k_{20}\left(s_1,s_2,s_3+s_4\right)}{16 s_3}+\frac{k_{20}\left(s_1,s_2+s_3,s_4\right)}{16 s_2}+\frac{e^{s_1} k_{20}\left(s_2,s_3,-s_1-s_2-s_3\right)}{16 s_4}-\frac{1}{16} G_1\left(s_1\right) k_{20}\left(s_2,s_3,s_4\right)+\frac{e^{s_1} k_{20}\left(s_2,s_3,s_4\right)}{16 \left(s_1+s_2+s_3+s_4\right)}+\frac{e^{s_1} k_{20}\left(s_2,s_3+s_4,-s_1-s_2-s_3-s_4\right)}{16 s_3}+\frac{k_{20}\left(s_1+s_2,s_3,s_4\right)}{16 s_1}+\frac{e^{s_1+s_2+s_3} k_{20}\left(-s_1-s_2-s_3,s_1,s_2\right)}{16 s_4}+\frac{e^{s_1+s_2} k_{20}\left(s_3,-s_1-s_2-s_3,s_1\right)}{16 s_4}-\frac{1}{16} e^{s_2} G_1\left(s_1\right) k_{20}\left(s_3,s_4,-s_2-s_3-s_4\right)+\frac{e^{s_1+s_2} k_{20}\left(s_3,s_4,-s_2-s_3-s_4\right)}{16 \left(s_1+s_2+s_3+s_4\right)}+\frac{e^{s_1+s_2} k_{20}\left(s_3,s_4,-s_1-s_2-s_3-s_4\right)}{16 s_1}+\frac{e^{s_1} k_{20}\left(s_2+s_3,s_4,-s_1-s_2-s_3-s_4\right)}{16 s_2}-\frac{1}{16} e^{s_2+s_3+s_4} G_1\left(s_1\right) k_{20}\left(-s_2-s_3-s_4,s_2,s_3\right)+\frac{e^{s_1+s_2+s_3+s_4} k_{20}\left(-s_2-s_3-s_4,s_2,s_3\right)}{16 \left(s_1+s_2+s_3+s_4\right)}+\frac{e^{s_1+s_2+s_3+s_4} k_{20}\left(-s_1-s_2-s_3-s_4,s_1,s_2\right)}{16 s_3}+\frac{e^{s_1+s_2+s_3+s_4} k_{20}\left(-s_1-s_2-s_3-s_4,s_1,s_2+s_3\right)}{16 s_2}+\frac{e^{s_1+s_2+s_3+s_4} k_{20}\left(-s_1-s_2-s_3-s_4,s_1+s_2,s_3\right)}{16 s_1}-\frac{1}{16} e^{s_2+s_3} G_1\left(s_1\right) k_{20}\left(s_4,-s_2-s_3-s_4,s_2\right)+\frac{e^{s_1+s_2+s_3} k_{20}\left(s_4,-s_2-s_3-s_4,s_2\right)}{16 \left(s_1+s_2+s_3+s_4\right)}+\frac{e^{s_1+s_2+s_3} k_{20}\left(s_4,-s_1-s_2-s_3-s_4,s_1\right)}{16 s_2}+\frac{e^{s_1+s_2+s_3} k_{20}\left(s_4,-s_1-s_2-s_3-s_4,s_1+s_2\right)}{16 s_1}+\frac{e^{s_1+s_2} k_{20}\left(s_3+s_4,-s_1-s_2-s_3-s_4,s_1\right)}{16 s_3}-\frac{k_{20}\left(s_2,s_3,s_4\right)}{16 s_1}-\frac{e^{s_2} k_{20}\left(s_3,s_4,-s_2-s_3-s_4\right)}{16 s_1}-\frac{e^{s_2+s_3+s_4} k_{20}\left(-s_2-s_3-s_4,s_2,s_3\right)}{16 s_1}-\frac{e^{s_2+s_3} k_{20}\left(s_4,-s_2-s_3-s_4,s_2\right)}{16 s_1}-\frac{e^{s_2+s_3+s_4} G_1\left(s_1\right) k_{14}\left(-s_2-s_3-s_4,s_2+s_3\right)}{4 s_2}-\frac{G_1\left(s_1\right) k_{15}\left(s_2+s_3,s_4\right)}{4 s_2}-\frac{e^{s_2+s_3} G_1\left(s_1\right) k_{16}\left(s_4,-s_2-s_3-s_4\right)}{4 s_2}-\frac{k_{20}\left(s_1+s_2,s_3,s_4\right)}{16 s_2}-\frac{e^{s_1+s_2} k_{20}\left(s_3,s_4,-s_1-s_2-s_3-s_4\right)}{16 s_2}-\frac{e^{s_1+s_2+s_3+s_4} k_{20}\left(-s_1-s_2-s_3-s_4,s_1+s_2,s_3\right)}{16 s_2}-\frac{e^{s_1+s_2+s_3} k_{20}\left(s_4,-s_1-s_2-s_3-s_4,s_1+s_2\right)}{16 s_2}-\frac{e^{s_2+s_3+s_4} k_{14}\left(-s_2-s_3-s_4,s_2+s_3\right)}{4 s_1 \left(s_1+s_2\right)}-\frac{k_{15}\left(s_2+s_3,s_4\right)}{4 s_1 \left(s_1+s_2\right)}-\frac{e^{s_2+s_3} k_{16}\left(s_4,-s_2-s_3-s_4\right)}{4 s_1 \left(s_1+s_2\right)}-\frac{e^{s_2+s_3+s_4} k_{14}\left(-s_2-s_3-s_4,s_2+s_3\right)}{4 s_2 \left(s_1+s_2\right)}-\frac{k_{15}\left(s_2+s_3,s_4\right)}{4 s_2 \left(s_1+s_2\right)}-\frac{e^{s_2+s_3} k_{16}\left(s_4,-s_2-s_3-s_4\right)}{4 s_2 \left(s_1+s_2\right)}-\frac{3 G_2\left(s_1,s_2\right) k_6\left(s_4\right)}{2 s_3}-\frac{3 e^{s_4} G_2\left(s_1,s_2\right) k_7\left(-s_4\right)}{2 s_3}-\frac{e^{s_2+s_3+s_4} G_1\left(s_1\right) k_{14}\left(-s_2-s_3-s_4,s_2\right)}{4 s_3}-\frac{G_1\left(s_1\right) k_{15}\left(s_2,s_3+s_4\right)}{4 s_3}-\frac{e^{s_2} G_1\left(s_1\right) k_{16}\left(s_3+s_4,-s_2-s_3-s_4\right)}{4 s_3}-\frac{k_{20}\left(s_1,s_2+s_3,s_4\right)}{16 s_3}-\frac{e^{s_1} k_{20}\left(s_2+s_3,s_4,-s_1-s_2-s_3-s_4\right)}{16 s_3}-\frac{e^{s_1+s_2+s_3+s_4} k_{20}\left(-s_1-s_2-s_3-s_4,s_1,s_2+s_3\right)}{16 s_3}-\frac{e^{s_1+s_2+s_3} k_{20}\left(s_4,-s_1-s_2-s_3-s_4,s_1\right)}{16 s_3}-\frac{e^{s_2+s_3+s_4} k_{14}\left(-s_2-s_3-s_4,s_2\right)}{4 s_1 s_3}-\frac{e^{s_1+s_2+s_3+s_4} k_{14}\left(-s_1-s_2-s_3-s_4,s_1+s_2+s_3\right)}{4 s_1 s_3}-\frac{k_{15}\left(s_2,s_3+s_4\right)}{4 s_1 s_3}-\frac{k_{15}\left(s_1+s_2+s_3,s_4\right)}{4 s_1 s_3}-\frac{e^{s_1+s_2+s_3} k_{16}\left(s_4,-s_1-s_2-s_3-s_4\right)}{4 s_1 s_3}-\frac{e^{s_2} k_{16}\left(s_3+s_4,-s_2-s_3-s_4\right)}{4 s_1 s_3}-\frac{3 G_1\left(s_1\right) k_6\left(s_2+s_3+s_4\right)}{2 s_2 \left(s_2+s_3\right)}-\frac{3 e^{s_2+s_3+s_4} G_1\left(s_1\right) k_7\left(-s_2-s_3-s_4\right)}{2 s_2 \left(s_2+s_3\right)}-\frac{e^{s_1+s_2+s_3+s_4} k_{14}\left(-s_1-s_2-s_3-s_4,s_1+s_2\right)}{4 s_2 \left(s_2+s_3\right)}-\frac{k_{15}\left(s_1+s_2,s_3+s_4\right)}{4 s_2 \left(s_2+s_3\right)}-\frac{e^{s_1+s_2} k_{16}\left(s_3+s_4,-s_1-s_2-s_3-s_4\right)}{4 s_2 \left(s_2+s_3\right)}-\frac{3 G_1\left(s_1\right) k_6\left(s_4\right)}{2 s_3 \left(s_2+s_3\right)}-\frac{3 e^{s_4} G_1\left(s_1\right) k_7\left(-s_4\right)}{2 s_3 \left(s_2+s_3\right)}-\frac{e^{s_1+s_2+s_3+s_4} k_{14}\left(-s_1-s_2-s_3-s_4,s_1+s_2\right)}{4 s_3 \left(s_2+s_3\right)}-\frac{k_{15}\left(s_1+s_2,s_3+s_4\right)}{4 s_3 \left(s_2+s_3\right)}-\frac{e^{s_1+s_2} k_{16}\left(s_3+s_4,-s_1-s_2-s_3-s_4\right)}{4 s_3 \left(s_2+s_3\right)}-\frac{3 k_6\left(s_2+s_3+s_4\right)}{\left(s_1+s_2\right) \left(s_2+s_3\right) \left(s_1+s_2+s_3\right)}-\frac{3 e^{s_2+s_3+s_4} k_7\left(-s_2-s_3-s_4\right)}{\left(s_1+s_2\right) \left(s_2+s_3\right) \left(s_1+s_2+s_3\right)}-\frac{3 s_2 k_6\left(s_2+s_3+s_4\right)}{2 s_1 \left(s_1+s_2\right) \left(s_2+s_3\right) \left(s_1+s_2+s_3\right)}-\frac{3 s_3 k_6\left(s_2+s_3+s_4\right)}{2 s_1 \left(s_1+s_2\right) \left(s_2+s_3\right) \left(s_1+s_2+s_3\right)}-\frac{3 e^{s_2+s_3+s_4} s_2 k_7\left(-s_2-s_3-s_4\right)}{2 s_1 \left(s_1+s_2\right) \left(s_2+s_3\right) \left(s_1+s_2+s_3\right)}-\frac{3 e^{s_2+s_3+s_4} s_3 k_7\left(-s_2-s_3-s_4\right)}{2 s_1 \left(s_1+s_2\right) \left(s_2+s_3\right) \left(s_1+s_2+s_3\right)}-\frac{3 s_1 k_6\left(s_2+s_3+s_4\right)}{2 s_2 \left(s_1+s_2\right) \left(s_2+s_3\right) \left(s_1+s_2+s_3\right)}-\frac{3 s_3 k_6\left(s_2+s_3+s_4\right)}{2 s_2 \left(s_1+s_2\right) \left(s_2+s_3\right) \left(s_1+s_2+s_3\right)}-\frac{3 e^{s_2+s_3+s_4} s_1 k_7\left(-s_2-s_3-s_4\right)}{2 s_2 \left(s_1+s_2\right) \left(s_2+s_3\right) \left(s_1+s_2+s_3\right)}-\frac{3 e^{s_2+s_3+s_4} s_3 k_7\left(-s_2-s_3-s_4\right)}{2 s_2 \left(s_1+s_2\right) \left(s_2+s_3\right) \left(s_1+s_2+s_3\right)}-\frac{3 s_1 k_6\left(s_4\right)}{2 \left(s_1+s_2\right) s_3 \left(s_2+s_3\right) \left(s_1+s_2+s_3\right)}-\frac{3 s_2 k_6\left(s_4\right)}{2 \left(s_1+s_2\right) s_3 \left(s_2+s_3\right) \left(s_1+s_2+s_3\right)}-\frac{3 e^{s_4} s_1 k_7\left(-s_4\right)}{2 \left(s_1+s_2\right) s_3 \left(s_2+s_3\right) \left(s_1+s_2+s_3\right)}-\frac{3 e^{s_4} s_2 k_7\left(-s_4\right)}{2 \left(s_1+s_2\right) s_3 \left(s_2+s_3\right) \left(s_1+s_2+s_3\right)}-\frac{3 G_2\left(s_1,s_2\right) k_6\left(s_3+s_4\right)}{2 s_4}-\frac{3 e^{s_3+s_4} G_2\left(s_1,s_2\right) k_7\left(-s_3-s_4\right)}{2 s_4}-\frac{e^{s_2+s_3} G_1\left(s_1\right) k_{14}\left(-s_2-s_3,s_2\right)}{4 s_4}-\frac{G_1\left(s_1\right) k_{15}\left(s_2,s_3\right)}{4 s_4}-\frac{e^{s_2} G_1\left(s_1\right) k_{16}\left(s_3,-s_2-s_3\right)}{4 s_4}-\frac{k_{20}\left(s_1,s_2,s_3+s_4\right)}{16 s_4}-\frac{e^{s_1} k_{20}\left(s_2,s_3+s_4,-s_1-s_2-s_3-s_4\right)}{16 s_4}-\frac{e^{s_1+s_2+s_3+s_4} k_{20}\left(-s_1-s_2-s_3-s_4,s_1,s_2\right)}{16 s_4}-\frac{e^{s_1+s_2} k_{20}\left(s_3+s_4,-s_1-s_2-s_3-s_4,s_1\right)}{16 s_4}-\frac{e^{s_2+s_3} k_{14}\left(-s_2-s_3,s_2\right)}{4 s_1 s_4}-\frac{e^{s_1+s_2+s_3+s_4} k_{14}\left(-s_1-s_2-s_3-s_4,s_1+s_2\right)}{4 s_1 s_4}-\frac{k_{15}\left(s_2,s_3\right)}{4 s_1 s_4}-\frac{k_{15}\left(s_1+s_2,s_3+s_4\right)}{4 s_1 s_4}-\frac{e^{s_2} k_{16}\left(s_3,-s_2-s_3\right)}{4 s_1 s_4}-\frac{e^{s_1+s_2} k_{16}\left(s_3+s_4,-s_1-s_2-s_3-s_4\right)}{4 s_1 s_4}-\frac{3 G_1\left(s_1\right) k_6\left(s_2+s_3\right)}{2 s_2 s_4}-\frac{3 G_1\left(s_1\right) k_6\left(s_3+s_4\right)}{2 s_2 s_4}-\frac{3 e^{s_2+s_3} G_1\left(s_1\right) k_7\left(-s_2-s_3\right)}{2 s_2 s_4}-\frac{3 e^{s_3+s_4} G_1\left(s_1\right) k_7\left(-s_3-s_4\right)}{2 s_2 s_4}-\frac{e^{s_1+s_2+s_3} k_{14}\left(-s_1-s_2-s_3,s_1+s_2\right)}{4 s_2 s_4}-\frac{e^{s_1+s_2+s_3+s_4} k_{14}\left(-s_1-s_2-s_3-s_4,s_1\right)}{4 s_2 s_4}-\frac{k_{15}\left(s_1,s_2+s_3+s_4\right)}{4 s_2 s_4}-\frac{k_{15}\left(s_1+s_2,s_3\right)}{4 s_2 s_4}-\frac{e^{s_1+s_2} k_{16}\left(s_3,-s_1-s_2-s_3\right)}{4 s_2 s_4}-\frac{e^{s_1} k_{16}\left(s_2+s_3+s_4,-s_1-s_2-s_3-s_4\right)}{4 s_2 s_4}-\frac{3 k_6\left(s_2+s_3\right)}{2 s_1 \left(s_1+s_2\right) s_4}-\frac{3 k_6\left(s_1+s_2+s_3+s_4\right)}{2 s_1 \left(s_1+s_2\right) s_4}-\frac{3 e^{s_2+s_3} k_7\left(-s_2-s_3\right)}{2 s_1 \left(s_1+s_2\right) s_4}-\frac{3 e^{s_1+s_2+s_3+s_4} k_7\left(-s_1-s_2-s_3-s_4\right)}{2 s_1 \left(s_1+s_2\right) s_4}-\frac{3 k_6\left(s_2+s_3\right)}{2 s_2 \left(s_1+s_2\right) s_4}-\frac{3 k_6\left(s_3+s_4\right)}{2 s_2 \left(s_1+s_2\right) s_4}-\frac{3 e^{s_2+s_3} k_7\left(-s_2-s_3\right)}{2 s_2 \left(s_1+s_2\right) s_4}-\frac{3 e^{s_3+s_4} k_7\left(-s_3-s_4\right)}{2 s_2 \left(s_1+s_2\right) s_4}-\frac{3 G_1\left(s_1\right) k_6\left(s_2\right)}{2 s_3 \left(s_3+s_4\right)}-\frac{3 e^{s_2} G_1\left(s_1\right) k_7\left(-s_2\right)}{2 s_3 \left(s_3+s_4\right)}-\frac{e^{s_1+s_2+s_3} k_{14}\left(-s_1-s_2-s_3,s_1\right)}{4 s_3 \left(s_3+s_4\right)}-\frac{k_{15}\left(s_1,s_2+s_3\right)}{4 s_3 \left(s_3+s_4\right)}-\frac{e^{s_1} k_{16}\left(s_2+s_3,-s_1-s_2-s_3\right)}{4 s_3 \left(s_3+s_4\right)}-\frac{3 k_6\left(s_2\right)}{2 s_1 s_3 \left(s_3+s_4\right)}-\frac{3 k_6\left(s_1+s_2+s_3\right)}{2 s_1 s_3 \left(s_3+s_4\right)}-\frac{3 e^{s_2} k_7\left(-s_2\right)}{2 s_1 s_3 \left(s_3+s_4\right)}-\frac{3 e^{s_1+s_2+s_3} k_7\left(-s_1-s_2-s_3\right)}{2 s_1 s_3 \left(s_3+s_4\right)}-\frac{3 G_1\left(s_1\right) k_6\left(s_2+s_3+s_4\right)}{2 s_4 \left(s_3+s_4\right)}-\frac{3 e^{s_2+s_3+s_4} G_1\left(s_1\right) k_7\left(-s_2-s_3-s_4\right)}{2 s_4 \left(s_3+s_4\right)}-\frac{e^{s_1+s_2+s_3} k_{14}\left(-s_1-s_2-s_3,s_1\right)}{4 s_4 \left(s_3+s_4\right)}-\frac{k_{15}\left(s_1,s_2+s_3\right)}{4 s_4 \left(s_3+s_4\right)}-\frac{e^{s_1} k_{16}\left(s_2+s_3,-s_1-s_2-s_3\right)}{4 s_4 \left(s_3+s_4\right)}-\frac{3 k_6\left(s_1+s_2+s_3\right)}{2 s_1 s_4 \left(s_3+s_4\right)}-\frac{3 k_6\left(s_2+s_3+s_4\right)}{2 s_1 s_4 \left(s_3+s_4\right)}-\frac{3 e^{s_1+s_2+s_3} k_7\left(-s_1-s_2-s_3\right)}{2 s_1 s_4 \left(s_3+s_4\right)}-\frac{3 e^{s_2+s_3+s_4} k_7\left(-s_2-s_3-s_4\right)}{2 s_1 s_4 \left(s_3+s_4\right)}-\frac{3 k_6\left(s_1+s_2\right)}{\left(s_2+s_3\right) \left(s_3+s_4\right) \left(s_2+s_3+s_4\right)}-\frac{3 e^{s_1+s_2} k_7\left(-s_1-s_2\right)}{\left(s_2+s_3\right) \left(s_3+s_4\right) \left(s_2+s_3+s_4\right)}-\frac{3 s_3 k_6\left(s_1+s_2\right)}{2 s_2 \left(s_2+s_3\right) \left(s_3+s_4\right) \left(s_2+s_3+s_4\right)}-\frac{3 s_4 k_6\left(s_1+s_2\right)}{2 s_2 \left(s_2+s_3\right) \left(s_3+s_4\right) \left(s_2+s_3+s_4\right)}-\frac{3 e^{s_1+s_2} s_3 k_7\left(-s_1-s_2\right)}{2 s_2 \left(s_2+s_3\right) \left(s_3+s_4\right) \left(s_2+s_3+s_4\right)}-\frac{3 e^{s_1+s_2} s_4 k_7\left(-s_1-s_2\right)}{2 s_2 \left(s_2+s_3\right) \left(s_3+s_4\right) \left(s_2+s_3+s_4\right)}-\frac{3 s_2 k_6\left(s_1+s_2\right)}{2 s_3 \left(s_2+s_3\right) \left(s_3+s_4\right) \left(s_2+s_3+s_4\right)}-\frac{3 s_4 k_6\left(s_1+s_2\right)}{2 s_3 \left(s_2+s_3\right) \left(s_3+s_4\right) \left(s_2+s_3+s_4\right)}-\frac{3 e^{s_1+s_2} s_2 k_7\left(-s_1-s_2\right)}{2 s_3 \left(s_2+s_3\right) \left(s_3+s_4\right) \left(s_2+s_3+s_4\right)}-\frac{3 e^{s_1+s_2} s_4 k_7\left(-s_1-s_2\right)}{2 s_3 \left(s_2+s_3\right) \left(s_3+s_4\right) \left(s_2+s_3+s_4\right)}-\frac{3 s_2 k_6\left(s_1+s_2+s_3+s_4\right)}{2 \left(s_2+s_3\right) s_4 \left(s_3+s_4\right) \left(s_2+s_3+s_4\right)}-\frac{3 s_3 k_6\left(s_1+s_2+s_3+s_4\right)}{2 \left(s_2+s_3\right) s_4 \left(s_3+s_4\right) \left(s_2+s_3+s_4\right)}-\frac{3 e^{s_1+s_2+s_3+s_4} s_2 k_7\left(-s_1-s_2-s_3-s_4\right)}{2 \left(s_2+s_3\right) s_4 \left(s_3+s_4\right) \left(s_2+s_3+s_4\right)}-\frac{3 e^{s_1+s_2+s_3+s_4} s_3 k_7\left(-s_1-s_2-s_3-s_4\right)}{2 \left(s_2+s_3\right) s_4 \left(s_3+s_4\right) \left(s_2+s_3+s_4\right)}-\frac{k_{20}\left(s_1,s_2,s_3\right)}{16 \left(s_1+s_2+s_3+s_4\right)}-\frac{e^{s_1} k_{20}\left(s_2,s_3,-s_1-s_2-s_3\right)}{16 \left(s_1+s_2+s_3+s_4\right)}-\frac{e^{s_1+s_2+s_3} k_{20}\left(-s_1-s_2-s_3,s_1,s_2\right)}{16 \left(s_1+s_2+s_3+s_4\right)}-\frac{e^{s_1+s_2} k_{20}\left(s_3,-s_1-s_2-s_3,s_1\right)}{16 \left(s_1+s_2+s_3+s_4\right)}.
\end{math}
\end{center}

\section{Lengthy functional relations among functions $k_3, \dots, k_{20}$}
\label{lengthykfnrelationsappsec}

In Section \ref{FuncRelationsforksSec} we found functional relations 
that involve only the functions $k_3, \dots, k_{20}$. In this appendix 
we provide more functional relations of this type.

\subsection{Functional relation between $k_{3},$ $k_{5},$ $k_{6},$ $\dots,$ $k_{16},$ $k_{17},$ $k_{19},$ $k_{20}$ } 
The explicit formulas given in Appendix \ref{explicit3and4fnsappsec} can be used to see that 
$2 \widetilde K_{16} =3 \widetilde K_{10}$, which combined with the finite difference 
expressions \eqref{basicK10eqn} and \eqref{basicK16eqn}, gives the following functional relation:  
 
  {\small
    \begin{center}
    \begin{math}
     -\frac{3 e^{s_1+s_2} k_3(-s_1-s_2)}{4 s_1 s_3}+\frac{3 e^{s_2} G_1(s_1) k_3(-s_2)}{4 s_3}+\frac{3 e^{s_2} k_3(-s_2)}{4 s_1 s_3}+\frac{3 G_1(s_1) k_3(s_2)}{4 s_3}+\frac{3 k_3(s_2)}{4 s_1 s_3}+\frac{3 e^{s_1+s_2+s_3} k_3(-s_1-s_2-s_3)}{4 s_1 s_3}+\frac{3 k_3(s_1+s_2+s_3)}{4 s_1 s_3}+\frac{e^{s_1+s_2} k_5(-s_1-s_2)}{s_1 s_3}+\frac{k_5(s_1+s_2)}{s_1 s_3}+\frac{e^{s_2+s_3} G_1(s_1) k_5(-s_2-s_3)}{s_3}+\frac{e^{s_2+s_3} k_5(-s_2-s_3)}{s_1 s_3}+\frac{G_1(s_1) k_5(s_2+s_3)}{s_3}+\frac{k_5(s_2+s_3)}{s_1 s_3}+\frac{k_6(s_1)}{2 s_2 s_3}+\frac{k_6(s_3)}{4 s_1 (s_1+s_2)}+\frac{k_6(s_3)}{4 s_2 (s_1+s_2)}+\frac{k_6(s_1+s_2+s_3)}{4 s_1 s_2}+\frac{k_6(s_1+s_2+s_3)}{2 s_2 (s_2+s_3)}+\frac{k_6(s_1+s_2+s_3)}{2 s_3 (s_2+s_3)}+\frac{e^{s_1} k_7(-s_1)}{2 s_2 s_3}+\frac{e^{s_1+s_2+s_3} k_7(-s_1-s_2-s_3)}{4 s_1 s_2}+\frac{e^{s_1+s_2+s_3} k_7(-s_1-s_2-s_3)}{2 s_2 (s_2+s_3)}+\frac{e^{s_1+s_2+s_3} k_7(-s_1-s_2-s_3)}{2 s_3 (s_2+s_3)}+\frac{e^{s_3} k_7(-s_3)}{4 s_1 (s_1+s_2)}+\frac{e^{s_3} k_7(-s_3)}{4 s_2 (s_1+s_2)}+\frac{3 k_8(s_1,s_2)}{4 s_3}-\frac{3}{4} G_1(s_1) k_8(s_2,s_3)+\frac{3 k_8(s_1+s_2,s_3)}{4 s_1}+\frac{3 k_9(s_1,s_2)}{8 s_3}+\frac{3 k_9(s_1,s_2+s_3)}{8 s_2}+\frac{3 e^{s_1+s_2} k_9(-s_1-s_2,s_1)}{8 (s_1+s_2+s_3)}-\frac{3}{8} G_1(s_1) k_9(s_2,s_3)+\frac{3 k_9(s_1+s_2,s_3)}{8 s_1}+\frac{3 e^{s_1+s_2+s_3} k_9(-s_1-s_2-s_3,s_1)}{8 s_3}+\frac{3}{8} e^{s_2} G_1(s_1) k_9(s_3,-s_2-s_3)+\frac{3 e^{s_2} k_9(s_3,-s_2-s_3)}{8 s_1}+\frac{3 e^{s_1+s_2+s_3} k_{10}(s_1,s_2)}{8 (-s_1-s_2-s_3)}+\frac{3 e^{s_1+s_2+s_3} k_{10}(s_1,s_2)}{8 (s_1+s_2+s_3)}+\frac{3 k_{10}(s_1,s_2+s_3)}{8 s_3}+\frac{3 e^{s_1} k_{10}(s_2,-s_1-s_2)}{8 s_3}+\frac{3}{8} e^{s_2+s_3} G_1(s_1) k_{10}(-s_2-s_3,s_2)+\frac{3 e^{s_2+s_3} k_{10}(-s_2-s_3,s_2)}{8 s_1}-\frac{3}{8} e^{s_2} G_1(s_1) k_{10}(s_3,-s_2-s_3)+\frac{3 e^{s_1+s_2} k_{10}(s_3,-s_1-s_2-s_3)}{8 s_1}+\frac{3 e^{s_1} k_{10}(s_2+s_3,-s_1-s_2-s_3)}{8 s_2}+\frac{3 e^{s_1} k_{11}(s_2,-s_1-s_2)}{4 s_3}-\frac{3}{4} e^{s_2} G_1(s_1) k_{11}(s_3,-s_2-s_3)+\frac{3 e^{s_1+s_2} k_{11}(s_3,-s_1-s_2-s_3)}{4 s_1}+\frac{3 e^{s_1+s_2} k_{12}(-s_1-s_2,s_1)}{4 s_3}-\frac{3}{4} e^{s_2+s_3} G_1(s_1) k_{12}(-s_2-s_3,s_2)+\frac{3 e^{s_1+s_2+s_3} k_{12}(-s_1-s_2-s_3,s_1+s_2)}{4 s_1}+\frac{3 e^{s_1+s_2} k_{13}(-s_1-s_2,s_1)}{8 s_3}+\frac{3 e^{s_1} k_{13}(s_2,-s_1-s_2)}{8 (s_1+s_2+s_3)}+\frac{3}{8} G_1(s_1) k_{13}(s_2,s_3)+\frac{3 k_{13}(s_2,s_3)}{8 s_1}-\frac{3}{8} e^{s_2+s_3} G_1(s_1) k_{13}(-s_2-s_3,s_2)+\frac{3 e^{s_1+s_2+s_3} k_{13}(-s_1-s_2-s_3,s_1)}{8 s_2}+\frac{3 e^{s_1+s_2+s_3} k_{13}(-s_1-s_2-s_3,s_1+s_2)}{8 s_1}+\frac{3 e^{s_1} k_{13}(s_2+s_3,-s_1-s_2-s_3)}{8 s_3}+\frac{e^{s_1} k_{14}(s_2,-s_1-s_2)}{4 s_3}-\frac{1}{4} G_1(s_1) k_{14}(s_2,s_3)+\frac{e^{s_1} k_{14}(s_2,s_3)}{4 (s_1+s_2+s_3)}+\frac{k_{14}(s_1+s_2,s_3)}{4 s_1}+\frac{1}{2} e^{s_2+s_3} G_1(s_1) k_{14}(-s_2-s_3,s_2)+\frac{e^{s_2+s_3} k_{14}(-s_2-s_3,s_2)}{2 s_1}+\frac{e^{s_1+s_2+s_3} k_{14}(-s_1-s_2-s_3,s_1)}{2 s_3}+\frac{e^{s_1+s_2+s_3} k_{14}(-s_1-s_2-s_3,s_1+s_2)}{4 s_2}+\frac{k_{15}(s_1,s_2+s_3)}{2 s_3}+\frac{e^{s_1+s_2} k_{15}(-s_1-s_2,s_1)}{4 s_3}+\frac{1}{2} G_1(s_1) k_{15}(s_2,s_3)+\frac{k_{15}(s_2,s_3)}{2 s_1}+\frac{k_{15}(s_1+s_2,s_3)}{4 s_2}-\frac{1}{4} e^{s_2} G_1(s_1) k_{15}(s_3,-s_2-s_3)+\frac{e^{s_1+s_2} k_{15}(s_3,-s_2-s_3)}{4 (s_1+s_2+s_3)}+\frac{e^{s_1+s_2} k_{15}(s_3,-s_1-s_2-s_3)}{4 s_1}+\frac{k_{16}(s_1,s_2)}{4 (-s_1-s_2-s_3)}+\frac{k_{16}(s_1,s_2)}{4 s_3}-\frac{1}{4} e^{s_2+s_3} G_1(s_1) k_{16}(-s_2-s_3,s_2)+\frac{e^{s_1+s_2+s_3} k_{16}(-s_2-s_3,s_2)}{4 (s_1+s_2+s_3)}+\frac{e^{s_1+s_2+s_3} k_{16}(-s_1-s_2-s_3,s_1+s_2)}{4 s_1}+\frac{1}{2} e^{s_2} G_1(s_1) k_{16}(s_3,-s_2-s_3)+\frac{e^{s_2} k_{16}(s_3,-s_2-s_3)}{2 s_1}+\frac{e^{s_1+s_2} k_{16}(s_3,-s_1-s_2-s_3)}{4 s_2}+\frac{e^{s_1} k_{16}(s_2+s_3,-s_1-s_2-s_3)}{2 s_3}+\frac{3}{16} k_{17}(s_1,s_2,s_3)+\frac{3}{16} e^{s_1} k_{17}(s_2,s_3,-s_1-s_2-s_3)+\frac{3}{16} e^{s_1+s_2+s_3} k_{17}(-s_1-s_2-s_3,s_1,s_2)+\frac{3}{16} e^{s_1+s_2} k_{17}(s_3,-s_1-s_2-s_3,s_1)+\frac{3}{16} k_{19}(s_1,s_2,s_3)+\frac{3}{16} e^{s_1} k_{19}(s_2,s_3,-s_1-s_2-s_3)+\frac{3}{16} e^{s_1+s_2+s_3} k_{19}(-s_1-s_2-s_3,s_1,s_2)+\frac{3}{16} e^{s_1+s_2} k_{19}(s_3,-s_1-s_2-s_3,s_1)-\frac{1}{8} k_{20}(s_1,s_2,s_3)-\frac{1}{8} e^{s_1} k_{20}(s_2,s_3,-s_1-s_2-s_3)-\frac{1}{8} e^{s_1+s_2+s_3} k_{20}(-s_1-s_2-s_3,s_1,s_2)-\frac{1}{8} e^{s_1+s_2} k_{20}(s_3,-s_1-s_2-s_3,s_1)-\frac{e^{s_1+s_2+s_3} k_{14}(-s_1-s_2-s_3,s_1+s_2)}{2 s_1}-\frac{k_{15}(s_1+s_2,s_3)}{2 s_1}-\frac{e^{s_1+s_2} k_{16}(s_3,-s_1-s_2-s_3)}{2 s_1}-\frac{3 k_8(s_2,s_3)}{4 s_1}-\frac{3 e^{s_2} k_{11}(s_3,-s_2-s_3)}{4 s_1}-\frac{3 e^{s_2+s_3} k_{12}(-s_2-s_3,s_2)}{4 s_1}-\frac{k_{14}(s_2,s_3)}{4 s_1}-\frac{e^{s_2} k_{15}(s_3,-s_2-s_3)}{4 s_1}-\frac{e^{s_2+s_3} k_{16}(-s_2-s_3,s_2)}{4 s_1}-\frac{3 k_9(s_2,s_3)}{8 s_1}-\frac{3 e^{s_1+s_2} k_9(s_3,-s_1-s_2-s_3)}{8 s_1}-\frac{3 e^{s_1+s_2+s_3} k_{10}(-s_1-s_2-s_3,s_1+s_2)}{8 s_1}-\frac{3 e^{s_2} k_{10}(s_3,-s_2-s_3)}{8 s_1}-\frac{3 k_{13}(s_1+s_2,s_3)}{8 s_1}-\frac{3 e^{s_2+s_3} k_{13}(-s_2-s_3,s_2)}{8 s_1}-\frac{e^{s_1+s_2+s_3} k_{14}(-s_1-s_2-s_3,s_1)}{4 s_2}-\frac{k_{15}(s_1,s_2+s_3)}{4 s_2}-\frac{e^{s_1} k_{16}(s_2+s_3,-s_1-s_2-s_3)}{4 s_2}-\frac{3 k_9(s_1+s_2,s_3)}{8 s_2}-\frac{3 e^{s_1+s_2} k_{10}(s_3,-s_1-s_2-s_3)}{8 s_2}-\frac{3 e^{s_1+s_2+s_3} k_{13}(-s_1-s_2-s_3,s_1+s_2)}{8 s_2}-\frac{k_6(s_3)}{4 s_1 s_2}-\frac{e^{s_3} k_7(-s_3)}{4 s_1 s_2}-\frac{k_6(s_1+s_2+s_3)}{4 s_1 (s_1+s_2)}-\frac{e^{s_1+s_2+s_3} k_7(-s_1-s_2-s_3)}{4 s_1 (s_1+s_2)}-\frac{k_6(s_1+s_2+s_3)}{4 s_2 (s_1+s_2)}-\frac{e^{s_1+s_2+s_3} k_7(-s_1-s_2-s_3)}{4 s_2 (s_1+s_2)}-\frac{e^{s_1+s_2+s_3} k_{16}(s_1,s_2)}{4 (-s_1-s_2-s_3)}-\frac{3 k_{10}(s_1,s_2)}{8 (-s_1-s_2-s_3)}-\frac{e^{s_2} G_1(s_1) k_5(-s_2)}{s_3}-\frac{G_1(s_1) k_5(s_2)}{s_3}-\frac{e^{s_1+s_2} k_{14}(-s_1-s_2,s_1)}{2 s_3}-\frac{k_{15}(s_1,s_2)}{2 s_3}-\frac{e^{s_1} k_{16}(s_2,-s_1-s_2)}{2 s_3}-\frac{3 e^{s_2+s_3} G_1(s_1) k_3(-s_2-s_3)}{4 s_3}-\frac{3 G_1(s_1) k_3(s_2+s_3)}{4 s_3}-\frac{3 k_8(s_1,s_2+s_3)}{4 s_3}-\frac{3 e^{s_1} k_{11}(s_2+s_3,-s_1-s_2-s_3)}{4 s_3}-\frac{3 e^{s_1+s_2+s_3} k_{12}(-s_1-s_2-s_3,s_1)}{4 s_3}-\frac{e^{s_1} k_{14}(s_2+s_3,-s_1-s_2-s_3)}{4 s_3}-\frac{e^{s_1+s_2+s_3} k_{15}(-s_1-s_2-s_3,s_1)}{4 s_3}-\frac{k_{16}(s_1,s_2+s_3)}{4 s_3}-\frac{3 k_9(s_1,s_2+s_3)}{8 s_3}-\frac{3 e^{s_1+s_2} k_9(-s_1-s_2,s_1)}{8 s_3}-\frac{3 k_{10}(s_1,s_2)}{8 s_3}-\frac{3 e^{s_1} k_{10}(s_2+s_3,-s_1-s_2-s_3)}{8 s_3}-\frac{3 e^{s_1} k_{13}(s_2,-s_1-s_2)}{8 s_3}-\frac{3 e^{s_1+s_2+s_3} k_{13}(-s_1-s_2-s_3,s_1)}{8 s_3}-\frac{e^{s_2} k_5(-s_2)}{s_1 s_3}-\frac{k_5(s_2)}{s_1 s_3}-\frac{e^{s_1+s_2+s_3} k_5(-s_1-s_2-s_3)}{s_1 s_3}-\frac{k_5(s_1+s_2+s_3)}{s_1 s_3}-\frac{3 k_3(s_1+s_2)}{4 s_1 s_3}-\frac{3 e^{s_2+s_3} k_3(-s_2-s_3)}{4 s_1 s_3}-\frac{3 k_3(s_2+s_3)}{4 s_1 s_3}-\frac{k_6(s_1+s_2+s_3)}{2 s_2 s_3}-\frac{e^{s_1+s_2+s_3} k_7(-s_1-s_2-s_3)}{2 s_2 s_3}-\frac{k_6(s_1)}{2 s_2 (s_2+s_3)}-\frac{e^{s_1} k_7(-s_1)}{2 s_2 (s_2+s_3)}-\frac{k_6(s_1)}{2 s_3 (s_2+s_3)}-\frac{e^{s_1} k_7(-s_1)}{2 s_3 (s_2+s_3)}-\frac{e^{s_1} k_{14}(s_2,-s_1-s_2)}{4 (s_1+s_2+s_3)}-\frac{e^{s_1+s_2} k_{15}(-s_1-s_2,s_1)}{4 (s_1+s_2+s_3)}-\frac{e^{s_1+s_2+s_3} k_{16}(s_1,s_2)}{4 (s_1+s_2+s_3)}-\frac{3 e^{s_1+s_2} k_9(s_3,-s_2-s_3)}{8 (s_1+s_2+s_3)}-\frac{3 e^{s_1+s_2+s_3} k_{10}(-s_2-s_3,s_2)}{8 (s_1+s_2+s_3)}-\frac{3 e^{s_1} k_{13}(s_2,s_3)}{8 (s_1+s_2+s_3)}=0. 
   \end{math}
   \end{center}
   }

 \subsection{Functional relation between $k_{6},$ $k_{7},$ $\dots,$ $k_{13},$ $k_{17},$ $k_{18},$ $k_{19}$  } 
 Finally, one can use the equality 
 $\widetilde K_{18} = \widetilde K_{19}$, which can be observed 
 by using the final explicit formulas written in Appendix \ref{explicit3and4fnsappsec}, and the different finite difference 
 expression \eqref{basicK18eqn} and \eqref{basicK19eqn} for these functions to obtain the following 
 functional relation:  
 
 \smallskip
 
 {\small  
 \begin{center}
 \begin{math}
    \frac{k_6(s_1)}{2 s_2 (s_2+s_3) s_4}-\frac{k_6(s_1)}{4 s_2 s_3 (s_2+s_3+s_4)}-\frac{k_6(s_1)}{4 s_2 (s_2+s_3) (s_2+s_3+s_4)}+\frac{k_6(s_1)}{4 s_3 (s_2+s_3) (s_2+s_3+s_4)}-\frac{k_6(s_1)}{2 s_2 s_4 (s_2+s_3+s_4)}+\frac{G_1(s_1) k_6(s_2)}{2 s_3 (s_3+s_4)}+\frac{G_1(s_1) k_6(s_2)}{2 s_4 (s_3+s_4)}+\frac{k_6(s_2)}{2 s_1 s_3 (s_3+s_4)}+\frac{k_6(s_2)}{2 s_1 s_4 (s_3+s_4)}+\frac{k_6(s_1+s_2)}{2 s_1 s_3 s_4}+\frac{k_6(s_1+s_2)}{2 s_2 s_3 (s_3+s_4)}+\frac{k_6(s_1+s_2)}{2 s_2 s_4 (s_3+s_4)}+\frac{k_6(s_1+s_2+s_3)}{4 s_2 s_3 s_4}+\frac{G_1(s_1) k_6(s_4)}{4 s_2 (s_2+s_3)}+\frac{G_1(s_1) k_6(s_4)}{4 s_3 (s_2+s_3)}+\frac{k_6(s_4)}{4 s_1 s_2 (s_2+s_3)}+\frac{k_6(s_4)}{4 s_2 (s_1+s_2) (s_1+s_2+s_3)}+\frac{k_6(s_4)}{4 s_3 (s_2+s_3) (s_1+s_2+s_3)}+\frac{G_1(s_1) k_6(s_2+s_3+s_4)}{4 s_2 s_3}+\frac{G_1(s_1) k_6(s_2+s_3+s_4)}{2 s_3 s_4}+\frac{k_6(s_2+s_3+s_4)}{4 s_1 s_2 s_3}+\frac{k_6(s_2+s_3+s_4)}{2 s_1 s_3 s_4}+\frac{k_6(s_1+s_2+s_3+s_4)}{4 s_2 (s_1+s_2) s_3}+\frac{k_6(s_1+s_2+s_3+s_4)}{4 s_1 s_3 (s_2+s_3)}+\frac{k_6(s_1+s_2+s_3+s_4)}{4 s_1 s_2 (s_1+s_2+s_3)}+\frac{k_6(s_1+s_2+s_3+s_4)}{4 s_2 s_3 s_4}+\frac{k_6(s_1+s_2+s_3+s_4)}{4 s_3 (s_2+s_3) s_4}+\frac{k_6(s_1+s_2+s_3+s_4)}{2 s_1 s_3 (s_3+s_4)}+\frac{k_6(s_1+s_2+s_3+s_4)}{2 s_1 s_4 (s_3+s_4)}+\frac{k_6(s_1+s_2+s_3+s_4)}{4 s_2 s_3 (s_2+s_3+s_4)}+\frac{k_6(s_1+s_2+s_3+s_4)}{4 s_2 (s_2+s_3) (s_2+s_3+s_4)}+\frac{k_6(s_1+s_2+s_3+s_4)}{2 s_2 s_4 (s_2+s_3+s_4)}+\frac{e^{s_1} k_7(-s_1)}{2 s_2 (s_2+s_3) s_4}+\frac{e^{s_1} k_7(-s_1)}{4 s_3 (s_2+s_3) (s_2+s_3+s_4)}+\frac{e^{s_1+s_2} k_7(-s_1-s_2)}{2 s_1 s_3 s_4}+\frac{e^{s_1+s_2} k_7(-s_1-s_2)}{2 s_2 s_3 (s_3+s_4)}+\frac{e^{s_1+s_2} k_7(-s_1-s_2)}{2 s_2 s_4 (s_3+s_4)}+\frac{e^{s_2} G_1(s_1) k_7(-s_2)}{2 s_3 (s_3+s_4)}+\frac{e^{s_2} G_1(s_1) k_7(-s_2)}{2 s_4 (s_3+s_4)}+\frac{e^{s_2} k_7(-s_2)}{2 s_1 s_3 (s_3+s_4)}+\frac{e^{s_2} k_7(-s_2)}{2 s_1 s_4 (s_3+s_4)}+\frac{e^{s_1+s_2+s_3} k_7(-s_1-s_2-s_3)}{4 s_2 s_3 s_4}+\frac{e^{s_2+s_3+s_4} G_1(s_1) k_7(-s_2-s_3-s_4)}{4 s_2 s_3}+\frac{e^{s_2+s_3+s_4} G_1(s_1) k_7(-s_2-s_3-s_4)}{2 s_3 s_4}+\frac{e^{s_2+s_3+s_4} k_7(-s_2-s_3-s_4)}{4 s_1 s_2 s_3}+\frac{e^{s_2+s_3+s_4} k_7(-s_2-s_3-s_4)}{2 s_1 s_3 s_4}+\frac{e^{s_1+s_2+s_3+s_4} k_7(-s_1-s_2-s_3-s_4)}{4 s_2 (s_1+s_2) s_3}+\frac{e^{s_1+s_2+s_3+s_4} k_7(-s_1-s_2-s_3-s_4)}{4 s_1 s_3 (s_2+s_3)}+\frac{e^{s_1+s_2+s_3+s_4} k_7(-s_1-s_2-s_3-s_4)}{4 s_1 s_2 (s_1+s_2+s_3)}+\frac{e^{s_1+s_2+s_3+s_4} k_7(-s_1-s_2-s_3-s_4)}{4 s_2 s_3 s_4}+\frac{e^{s_1+s_2+s_3+s_4} k_7(-s_1-s_2-s_3-s_4)}{4 s_3 (s_2+s_3) s_4}+\frac{e^{s_1+s_2+s_3+s_4} k_7(-s_1-s_2-s_3-s_4)}{2 s_1 s_3 (s_3+s_4)}+\frac{e^{s_1+s_2+s_3+s_4} k_7(-s_1-s_2-s_3-s_4)}{2 s_1 s_4 (s_3+s_4)}+\frac{e^{s_1+s_2+s_3+s_4} k_7(-s_1-s_2-s_3-s_4)}{4 s_2 s_3 (s_2+s_3+s_4)}+\frac{e^{s_1+s_2+s_3+s_4} k_7(-s_1-s_2-s_3-s_4)}{4 s_2 (s_2+s_3) (s_2+s_3+s_4)}+\frac{e^{s_1+s_2+s_3+s_4} k_7(-s_1-s_2-s_3-s_4)}{2 s_2 s_4 (s_2+s_3+s_4)}+\frac{e^{s_4} G_1(s_1) k_7(-s_4)}{4 s_2 (s_2+s_3)}+\frac{e^{s_4} G_1(s_1) k_7(-s_4)}{4 s_3 (s_2+s_3)}+\frac{e^{s_4} k_7(-s_4)}{4 s_1 s_2 (s_2+s_3)}+\frac{e^{s_4} k_7(-s_4)}{4 s_2 (s_1+s_2) (s_1+s_2+s_3)}+\frac{e^{s_4} k_7(-s_4)}{4 s_3 (s_2+s_3) (s_1+s_2+s_3)}+\frac{k_8(s_1,s_2+s_3)}{4 s_2 s_4}+\frac{k_8(s_1,s_2+s_3+s_4)}{8 s_3 (s_2+s_3)}+\frac{G_1(s_1) k_8(s_2,s_3)}{4 s_4}+\frac{k_8(s_2,s_3)}{4 s_1 s_4}+\frac{k_8(s_1+s_2,s_3+s_4)}{4 s_1 s_3}+\frac{k_8(s_1+s_2,s_3+s_4)}{4 s_2 s_3}+\frac{k_8(s_1+s_2,s_3+s_4)}{4 s_1 s_4}+\frac{k_8(s_1+s_2,s_3+s_4)}{4 s_2 s_4}+\frac{G_1(s_1) k_8(s_2+s_3,s_4)}{4 s_3}+\frac{k_8(s_2+s_3,s_4)}{4 s_1 s_3}+\frac{k_8(s_1+s_2+s_3,s_4)}{8 s_2 (s_2+s_3)}+\frac{k_9(s_1,s_2)}{8 s_3 s_4}+\frac{k_9(s_1,s_2+s_3+s_4)}{8 s_2 (s_2+s_3)}+\frac{k_9(s_1,s_2+s_3+s_4)}{8 s_2 s_4}+\frac{k_9(s_1,s_2+s_3+s_4)}{8 s_3 (s_3+s_4)}+\frac{k_9(s_1,s_2+s_3+s_4)}{8 s_4 (s_3+s_4)}+\frac{G_1(s_1) k_9(s_2,s_3+s_4)}{8 s_3}+\frac{G_1(s_1) k_9(s_2,s_3+s_4)}{8 s_4}+\frac{k_9(s_2,s_3+s_4)}{8 s_1 s_3}+\frac{k_9(s_2,s_3+s_4)}{8 s_1 s_4}+\frac{k_9(s_1+s_2,s_3)}{8 s_1 s_4}+\frac{k_9(s_1+s_2,s_3)}{8 s_2 s_4}+\frac{k_9(s_1+s_2+s_3,s_4)}{8 s_1 s_3}+\frac{k_9(s_1+s_2+s_3,s_4)}{8 s_2 s_3}+\frac{e^{s_1} k_{10}(s_2,-s_1-s_2)}{8 s_3 s_4}+\frac{e^{s_1+s_2} k_{10}(s_3,-s_1-s_2-s_3)}{8 s_1 s_4}+\frac{e^{s_1+s_2} k_{10}(s_3,-s_1-s_2-s_3)}{8 s_2 s_4}+\frac{e^{s_1+s_2+s_3} k_{10}(s_4,-s_1-s_2-s_3-s_4)}{8 s_1 s_3}+\frac{e^{s_1+s_2+s_3} k_{10}(s_4,-s_1-s_2-s_3-s_4)}{8 s_2 s_3}+\frac{e^{s_2} G_1(s_1) k_{10}(s_3+s_4,-s_2-s_3-s_4)}{8 s_3}+\frac{e^{s_2} G_1(s_1) k_{10}(s_3+s_4,-s_2-s_3-s_4)}{8 s_4}+\frac{e^{s_2} k_{10}(s_3+s_4,-s_2-s_3-s_4)}{8 s_1 s_3}+\frac{e^{s_2} k_{10}(s_3+s_4,-s_2-s_3-s_4)}{8 s_1 s_4}+\frac{e^{s_1} k_{10}(s_2+s_3+s_4,-s_1-s_2-s_3-s_4)}{8 s_2 (s_2+s_3)}+\frac{e^{s_1} k_{10}(s_2+s_3+s_4,-s_1-s_2-s_3-s_4)}{8 s_2 s_4}+\frac{e^{s_1} k_{10}(s_2+s_3+s_4,-s_1-s_2-s_3-s_4)}{8 s_3 (s_3+s_4)}+\frac{e^{s_1} k_{10}(s_2+s_3+s_4,-s_1-s_2-s_3-s_4)}{8 s_4 (s_3+s_4)}+\frac{e^{s_2} G_1(s_1) k_{11}(s_3,-s_2-s_3)}{4 s_4}+\frac{e^{s_2} k_{11}(s_3,-s_2-s_3)}{4 s_1 s_4}+\frac{e^{s_1} k_{11}(s_2+s_3,-s_1-s_2-s_3)}{4 s_2 s_4}+\frac{e^{s_2+s_3} G_1(s_1) k_{11}(s_4,-s_2-s_3-s_4)}{4 s_3}+\frac{e^{s_2+s_3} k_{11}(s_4,-s_2-s_3-s_4)}{4 s_1 s_3}+\frac{e^{s_1+s_2+s_3} k_{11}(s_4,-s_1-s_2-s_3-s_4)}{8 s_2 (s_2+s_3)}+\frac{e^{s_1+s_2} k_{11}(s_3+s_4,-s_1-s_2-s_3-s_4)}{4 s_1 s_3}+\frac{e^{s_1+s_2} k_{11}(s_3+s_4,-s_1-s_2-s_3-s_4)}{4 s_2 s_3}+\frac{e^{s_1+s_2} k_{11}(s_3+s_4,-s_1-s_2-s_3-s_4)}{4 s_1 s_4}+\frac{e^{s_1+s_2} k_{11}(s_3+s_4,-s_1-s_2-s_3-s_4)}{4 s_2 s_4}+\frac{e^{s_1} k_{11}(s_2+s_3+s_4,-s_1-s_2-s_3-s_4)}{8 s_3 (s_2+s_3)}+\frac{e^{s_2+s_3} G_1(s_1) k_{12}(-s_2-s_3,s_2)}{4 s_4}+\frac{e^{s_2+s_3} k_{12}(-s_2-s_3,s_2)}{4 s_1 s_4}+\frac{e^{s_1+s_2+s_3} k_{12}(-s_1-s_2-s_3,s_1)}{4 s_2 s_4}+\frac{e^{s_2+s_3+s_4} G_1(s_1) k_{12}(-s_2-s_3-s_4,s_2+s_3)}{4 s_3}+\frac{e^{s_2+s_3+s_4} k_{12}(-s_2-s_3-s_4,s_2+s_3)}{4 s_1 s_3}+\frac{e^{s_1+s_2+s_3+s_4} k_{12}(-s_1-s_2-s_3-s_4,s_1)}{8 s_3 (s_2+s_3)}+\frac{e^{s_1+s_2+s_3+s_4} k_{12}(-s_1-s_2-s_3-s_4,s_1+s_2)}{4 s_1 s_3}+\frac{e^{s_1+s_2+s_3+s_4} k_{12}(-s_1-s_2-s_3-s_4,s_1+s_2)}{4 s_2 s_3}+\frac{e^{s_1+s_2+s_3+s_4} k_{12}(-s_1-s_2-s_3-s_4,s_1+s_2)}{4 s_1 s_4}+\frac{e^{s_1+s_2+s_3+s_4} k_{12}(-s_1-s_2-s_3-s_4,s_1+s_2)}{4 s_2 s_4}+\frac{e^{s_1+s_2+s_3+s_4} k_{12}(-s_1-s_2-s_3-s_4,s_1+s_2+s_3)}{8 s_2 (s_2+s_3)}+\frac{e^{s_1+s_2} k_{13}(-s_1-s_2,s_1)}{8 s_3 s_4}+\frac{e^{s_1+s_2+s_3} k_{13}(-s_1-s_2-s_3,s_1+s_2)}{8 s_1 s_4}+\frac{e^{s_1+s_2+s_3} k_{13}(-s_1-s_2-s_3,s_1+s_2)}{8 s_2 s_4}+\frac{e^{s_2+s_3+s_4} G_1(s_1) k_{13}(-s_2-s_3-s_4,s_2)}{8 s_3}+\frac{e^{s_2+s_3+s_4} G_1(s_1) k_{13}(-s_2-s_3-s_4,s_2)}{8 s_4}+\frac{e^{s_2+s_3+s_4} k_{13}(-s_2-s_3-s_4,s_2)}{8 s_1 s_3}+\frac{e^{s_2+s_3+s_4} k_{13}(-s_2-s_3-s_4,s_2)}{8 s_1 s_4}+\frac{e^{s_1+s_2+s_3+s_4} k_{13}(-s_1-s_2-s_3-s_4,s_1)}{8 s_2 (s_2+s_3)}+\frac{e^{s_1+s_2+s_3+s_4} k_{13}(-s_1-s_2-s_3-s_4,s_1)}{8 s_2 s_4}+\frac{e^{s_1+s_2+s_3+s_4} k_{13}(-s_1-s_2-s_3-s_4,s_1)}{8 s_3 (s_3+s_4)}+\frac{e^{s_1+s_2+s_3+s_4} k_{13}(-s_1-s_2-s_3-s_4,s_1)}{8 s_4 (s_3+s_4)}+\frac{e^{s_1+s_2+s_3+s_4} k_{13}(-s_1-s_2-s_3-s_4,s_1+s_2+s_3)}{8 s_1 s_3}+\frac{e^{s_1+s_2+s_3+s_4} k_{13}(-s_1-s_2-s_3-s_4,s_1+s_2+s_3)}{8 s_2 s_3}+\frac{k_{17}(s_1,s_2+s_3,s_4)}{16 s_2}+\frac{e^{s_1} k_{17}(s_2,s_3,-s_1-s_2-s_3)}{16 (s_1+s_2+s_3+s_4)}+\frac{1}{16} G_1(s_1) k_{17}(s_2,s_3,s_4)+\frac{k_{17}(s_2,s_3,s_4)}{16 s_1}+\frac{e^{s_1} k_{17}(s_2,s_3+s_4,-s_1-s_2-s_3-s_4)}{16 s_3}+\frac{e^{s_1} k_{17}(s_2,s_3+s_4,-s_1-s_2-s_3-s_4)}{16 s_4}+\frac{e^{s_1+s_2+s_3} k_{17}(-s_1-s_2-s_3,s_1,s_2)}{16 (s_1+s_2+s_3+s_4)}+\frac{e^{s_1+s_2+s_3+s_4} k_{17}(-s_1-s_2-s_3-s_4,s_1,s_2)}{16 s_3}+\frac{e^{s_1+s_2+s_3+s_4} k_{17}(-s_1-s_2-s_3-s_4,s_1,s_2)}{16 s_4}+\frac{1}{16} e^{s_2+s_3} G_1(s_1) k_{17}(s_4,-s_2-s_3-s_4,s_2)+\frac{e^{s_2+s_3} k_{17}(s_4,-s_2-s_3-s_4,s_2)}{16 s_1}+\frac{e^{s_1+s_2+s_3} k_{17}(s_4,-s_1-s_2-s_3-s_4,s_1)}{16 s_2}+\frac{k_{18}(s_1,s_2,s_3)}{16 (-s_1-s_2-s_3-s_4)}+\frac{k_{18}(s_1,s_2,s_3)}{16 s_4}+\frac{k_{18}(s_1,s_2+s_3,s_4)}{16 s_3}+\frac{e^{s_1} k_{18}(s_2,s_3,-s_1-s_2-s_3)}{16 s_4}-\frac{1}{16} G_1(s_1) k_{18}(s_2,s_3,s_4)+\frac{e^{s_1} k_{18}(s_2,s_3,s_4)}{16 (s_1+s_2+s_3+s_4)}+\frac{k_{18}(s_1+s_2,s_3,s_4)}{16 s_1}+\frac{k_{18}(s_1+s_2,s_3,s_4)}{16 s_2}+\frac{e^{s_1+s_2+s_3} k_{18}(-s_1-s_2-s_3,s_1,s_2)}{16 s_4}+\frac{e^{s_1+s_2} k_{18}(s_3,-s_1-s_2-s_3,s_1)}{16 s_4}-\frac{1}{16} e^{s_2} G_1(s_1) k_{18}(s_3,s_4,-s_2-s_3-s_4)+\frac{e^{s_1+s_2} k_{18}(s_3,s_4,-s_2-s_3-s_4)}{16 (s_1+s_2+s_3+s_4)}+\frac{e^{s_1+s_2} k_{18}(s_3,s_4,-s_1-s_2-s_3-s_4)}{16 s_1}+\frac{e^{s_1+s_2} k_{18}(s_3,s_4,-s_1-s_2-s_3-s_4)}{16 s_2}+\frac{e^{s_1} k_{18}(s_2+s_3,s_4,-s_1-s_2-s_3-s_4)}{16 s_3}-\frac{1}{16} e^{s_2+s_3+s_4} G_1(s_1) k_{18}(-s_2-s_3-s_4,s_2,s_3)+\frac{e^{s_1+s_2+s_3+s_4} k_{18}(-s_2-s_3-s_4,s_2,s_3)}{16 (s_1+s_2+s_3+s_4)}+\frac{e^{s_1+s_2+s_3+s_4} k_{18}(-s_1-s_2-s_3-s_4,s_1,s_2+s_3)}{16 s_3}+\frac{e^{s_1+s_2+s_3+s_4} k_{18}(-s_1-s_2-s_3-s_4,s_1+s_2,s_3)}{16 s_1}+\frac{e^{s_1+s_2+s_3+s_4} k_{18}(-s_1-s_2-s_3-s_4,s_1+s_2,s_3)}{16 s_2}-\frac{1}{16} e^{s_2+s_3} G_1(s_1) k_{18}(s_4,-s_2-s_3-s_4,s_2)+\frac{e^{s_1+s_2+s_3} k_{18}(s_4,-s_2-s_3-s_4,s_2)}{16 (s_1+s_2+s_3+s_4)}+\frac{e^{s_1+s_2+s_3} k_{18}(s_4,-s_1-s_2-s_3-s_4,s_1)}{16 s_3}+\frac{e^{s_1+s_2+s_3} k_{18}(s_4,-s_1-s_2-s_3-s_4,s_1+s_2)}{16 s_1}+\frac{e^{s_1+s_2+s_3} k_{18}(s_4,-s_1-s_2-s_3-s_4,s_1+s_2)}{16 s_2}+\frac{e^{s_1+s_2+s_3+s_4} k_{19}(s_1,s_2,s_3)}{16 (-s_1-s_2-s_3-s_4)}+\frac{e^{s_1+s_2+s_3+s_4} k_{19}(s_1,s_2,s_3)}{16 (s_1+s_2+s_3+s_4)}+\frac{k_{19}(s_1,s_2,s_3+s_4)}{16 s_3}+\frac{k_{19}(s_1,s_2,s_3+s_4)}{16 s_4}+\frac{e^{s_1+s_2} k_{19}(s_3,-s_1-s_2-s_3,s_1)}{16 (s_1+s_2+s_3+s_4)}+\frac{1}{16} e^{s_2} G_1(s_1) k_{19}(s_3,s_4,-s_2-s_3-s_4)+\frac{e^{s_2} k_{19}(s_3,s_4,-s_2-s_3-s_4)}{16 s_1}+\frac{e^{s_1} k_{19}(s_2+s_3,s_4,-s_1-s_2-s_3-s_4)}{16 s_2}+\frac{1}{16} e^{s_2+s_3+s_4} G_1(s_1) k_{19}(-s_2-s_3-s_4,s_2,s_3)+\frac{e^{s_2+s_3+s_4} k_{19}(-s_2-s_3-s_4,s_2,s_3)}{16 s_1}+\frac{e^{s_1+s_2+s_3+s_4} k_{19}(-s_1-s_2-s_3-s_4,s_1,s_2+s_3)}{16 s_2}+\frac{e^{s_1+s_2} k_{19}(s_3+s_4,-s_1-s_2-s_3-s_4,s_1)}{16 s_3}+\frac{e^{s_1+s_2} k_{19}(s_3+s_4,-s_1-s_2-s_3-s_4,s_1)}{16 s_4}-\frac{k_{17}(s_1+s_2,s_3,s_4)}{16 s_1}-\frac{e^{s_1+s_2+s_3} k_{17}(s_4,-s_1-s_2-s_3-s_4,s_1+s_2)}{16 s_1}-\frac{k_{18}(s_2,s_3,s_4)}{16 s_1}-\frac{e^{s_2} k_{18}(s_3,s_4,-s_2-s_3-s_4)}{16 s_1}-\frac{e^{s_2+s_3+s_4} k_{18}(-s_2-s_3-s_4,s_2,s_3)}{16 s_1}-\frac{e^{s_2+s_3} k_{18}(s_4,-s_2-s_3-s_4,s_2)}{16 s_1}-\frac{e^{s_1+s_2} k_{19}(s_3,s_4,-s_1-s_2-s_3-s_4)}{16 s_1}-\frac{e^{s_1+s_2+s_3+s_4} k_{19}(-s_1-s_2-s_3-s_4,s_1+s_2,s_3)}{16 s_1}-\frac{k_{17}(s_1+s_2,s_3,s_4)}{16 s_2}-\frac{e^{s_1+s_2+s_3} k_{17}(s_4,-s_1-s_2-s_3-s_4,s_1+s_2)}{16 s_2}-\frac{k_{18}(s_1,s_2+s_3,s_4)}{16 s_2}-\frac{e^{s_1} k_{18}(s_2+s_3,s_4,-s_1-s_2-s_3-s_4)}{16 s_2}-\frac{e^{s_1+s_2+s_3+s_4} k_{18}(-s_1-s_2-s_3-s_4,s_1,s_2+s_3)}{16 s_2}-\frac{e^{s_1+s_2+s_3} k_{18}(s_4,-s_1-s_2-s_3-s_4,s_1)}{16 s_2}-\frac{e^{s_1+s_2} k_{19}(s_3,s_4,-s_1-s_2-s_3-s_4)}{16 s_2}-\frac{e^{s_1+s_2+s_3+s_4} k_{19}(-s_1-s_2-s_3-s_4,s_1+s_2,s_3)}{16 s_2}-\frac{G_1(s_1) k_8(s_2,s_3+s_4)}{4 s_3}-\frac{e^{s_2} G_1(s_1) k_{11}(s_3+s_4,-s_2-s_3-s_4)}{4 s_3}-\frac{e^{s_2+s_3+s_4} G_1(s_1) k_{12}(-s_2-s_3-s_4,s_2)}{4 s_3}-\frac{G_1(s_1) k_9(s_2+s_3,s_4)}{8 s_3}-\frac{e^{s_2+s_3} G_1(s_1) k_{10}(s_4,-s_2-s_3-s_4)}{8 s_3}-\frac{e^{s_2+s_3+s_4} G_1(s_1) k_{13}(-s_2-s_3-s_4,s_2+s_3)}{8 s_3}-\frac{e^{s_1} k_{17}(s_2+s_3,s_4,-s_1-s_2-s_3-s_4)}{16 s_3}-\frac{e^{s_1+s_2+s_3+s_4} k_{17}(-s_1-s_2-s_3-s_4,s_1,s_2+s_3)}{16 s_3}-\frac{k_{18}(s_1,s_2,s_3+s_4)}{16 s_3}-\frac{e^{s_1} k_{18}(s_2,s_3+s_4,-s_1-s_2-s_3-s_4)}{16 s_3}-\frac{e^{s_1+s_2+s_3+s_4} k_{18}(-s_1-s_2-s_3-s_4,s_1,s_2)}{16 s_3}-\frac{e^{s_1+s_2} k_{18}(s_3+s_4,-s_1-s_2-s_3-s_4,s_1)}{16 s_3}-\frac{k_{19}(s_1,s_2+s_3,s_4)}{16 s_3}-\frac{e^{s_1+s_2+s_3} k_{19}(s_4,-s_1-s_2-s_3-s_4,s_1)}{16 s_3}-\frac{k_8(s_2,s_3+s_4)}{4 s_1 s_3}-\frac{k_8(s_1+s_2+s_3,s_4)}{4 s_1 s_3}-\frac{e^{s_1+s_2+s_3} k_{11}(s_4,-s_1-s_2-s_3-s_4)}{4 s_1 s_3}-\frac{e^{s_2} k_{11}(s_3+s_4,-s_2-s_3-s_4)}{4 s_1 s_3}-\frac{e^{s_2+s_3+s_4} k_{12}(-s_2-s_3-s_4,s_2)}{4 s_1 s_3}-\frac{e^{s_1+s_2+s_3+s_4} k_{12}(-s_1-s_2-s_3-s_4,s_1+s_2+s_3)}{4 s_1 s_3}-\frac{k_9(s_1+s_2,s_3+s_4)}{8 s_1 s_3}-\frac{k_9(s_2+s_3,s_4)}{8 s_1 s_3}-\frac{e^{s_2+s_3} k_{10}(s_4,-s_2-s_3-s_4)}{8 s_1 s_3}-\frac{e^{s_1+s_2} k_{10}(s_3+s_4,-s_1-s_2-s_3-s_4)}{8 s_1 s_3}-\frac{e^{s_2+s_3+s_4} k_{13}(-s_2-s_3-s_4,s_2+s_3)}{8 s_1 s_3}-\frac{e^{s_1+s_2+s_3+s_4} k_{13}(-s_1-s_2-s_3-s_4,s_1+s_2)}{8 s_1 s_3}-\frac{G_1(s_1) k_6(s_4)}{4 s_2 s_3}-\frac{e^{s_4} G_1(s_1) k_7(-s_4)}{4 s_2 s_3}-\frac{k_8(s_1,s_2+s_3+s_4)}{8 s_2 s_3}-\frac{k_8(s_1+s_2+s_3,s_4)}{8 s_2 s_3}-\frac{k_9(s_1+s_2,s_3+s_4)}{8 s_2 s_3}-\frac{e^{s_1+s_2} k_{10}(s_3+s_4,-s_1-s_2-s_3-s_4)}{8 s_2 s_3}-\frac{e^{s_1+s_2+s_3} k_{11}(s_4,-s_1-s_2-s_3-s_4)}{8 s_2 s_3}-\frac{e^{s_1} k_{11}(s_2+s_3+s_4,-s_1-s_2-s_3-s_4)}{8 s_2 s_3}-\frac{e^{s_1+s_2+s_3+s_4} k_{12}(-s_1-s_2-s_3-s_4,s_1)}{8 s_2 s_3}-\frac{e^{s_1+s_2+s_3+s_4} k_{12}(-s_1-s_2-s_3-s_4,s_1+s_2+s_3)}{8 s_2 s_3}-\frac{e^{s_1+s_2+s_3+s_4} k_{13}(-s_1-s_2-s_3-s_4,s_1+s_2)}{8 s_2 s_3}-\frac{k_6(s_1+s_2+s_3+s_4)}{4 s_1 s_2 s_3}-\frac{e^{s_1+s_2+s_3+s_4} k_7(-s_1-s_2-s_3-s_4)}{4 s_1 s_2 s_3}-\frac{k_6(s_4)}{4 s_2 (s_1+s_2) s_3}-\frac{e^{s_4} k_7(-s_4)}{4 s_2 (s_1+s_2) s_3}-\frac{G_1(s_1) k_6(s_2+s_3+s_4)}{4 s_2 (s_2+s_3)}-\frac{e^{s_2+s_3+s_4} G_1(s_1) k_7(-s_2-s_3-s_4)}{4 s_2 (s_2+s_3)}-\frac{k_8(s_1,s_2+s_3+s_4)}{8 s_2 (s_2+s_3)}-\frac{k_9(s_1+s_2+s_3,s_4)}{8 s_2 (s_2+s_3)}-\frac{e^{s_1+s_2+s_3} k_{10}(s_4,-s_1-s_2-s_3-s_4)}{8 s_2 (s_2+s_3)}-\frac{e^{s_1} k_{11}(s_2+s_3+s_4,-s_1-s_2-s_3-s_4)}{8 s_2 (s_2+s_3)}-\frac{e^{s_1+s_2+s_3+s_4} k_{12}(-s_1-s_2-s_3-s_4,s_1)}{8 s_2 (s_2+s_3)}-\frac{e^{s_1+s_2+s_3+s_4} k_{13}(-s_1-s_2-s_3-s_4,s_1+s_2+s_3)}{8 s_2 (s_2+s_3)}-\frac{k_6(s_2+s_3+s_4)}{4 s_1 s_2 (s_2+s_3)}-\frac{e^{s_2+s_3+s_4} k_7(-s_2-s_3-s_4)}{4 s_1 s_2 (s_2+s_3)}-\frac{G_1(s_1) k_6(s_2+s_3+s_4)}{4 s_3 (s_2+s_3)}-\frac{e^{s_2+s_3+s_4} G_1(s_1) k_7(-s_2-s_3-s_4)}{4 s_3 (s_2+s_3)}-\frac{k_8(s_1+s_2+s_3,s_4)}{8 s_3 (s_2+s_3)}-\frac{e^{s_1+s_2+s_3} k_{11}(s_4,-s_1-s_2-s_3-s_4)}{8 s_3 (s_2+s_3)}-\frac{e^{s_1+s_2+s_3+s_4} k_{12}(-s_1-s_2-s_3-s_4,s_1+s_2+s_3)}{8 s_3 (s_2+s_3)}-\frac{k_6(s_2+s_3+s_4)}{4 s_1 s_3 (s_2+s_3)}-\frac{e^{s_2+s_3+s_4} k_7(-s_2-s_3-s_4)}{4 s_1 s_3 (s_2+s_3)}-\frac{k_6(s_4)}{4 s_1 s_2 (s_1+s_2+s_3)}-\frac{e^{s_4} k_7(-s_4)}{4 s_1 s_2 (s_1+s_2+s_3)}-\frac{k_6(s_1+s_2+s_3+s_4)}{4 s_2 (s_1+s_2) (s_1+s_2+s_3)}-\frac{e^{s_1+s_2+s_3+s_4} k_7(-s_1-s_2-s_3-s_4)}{4 s_2 (s_1+s_2) (s_1+s_2+s_3)}-\frac{k_6(s_1+s_2+s_3+s_4)}{4 s_3 (s_2+s_3) (s_1+s_2+s_3)}-\frac{e^{s_1+s_2+s_3+s_4} k_7(-s_1-s_2-s_3-s_4)}{4 s_3 (s_2+s_3) (s_1+s_2+s_3)}-\frac{e^{s_1+s_2+s_3+s_4} k_{18}(s_1,s_2,s_3)}{16 (-s_1-s_2-s_3-s_4)}-\frac{k_{19}(s_1,s_2,s_3)}{16 (-s_1-s_2-s_3-s_4)}-\frac{G_1(s_1) k_8(s_2,s_3+s_4)}{4 s_4}-\frac{e^{s_2} G_1(s_1) k_{11}(s_3+s_4,-s_2-s_3-s_4)}{4 s_4}-\frac{e^{s_2+s_3+s_4} G_1(s_1) k_{12}(-s_2-s_3-s_4,s_2)}{4 s_4}-\frac{G_1(s_1) k_9(s_2,s_3)}{8 s_4}-\frac{e^{s_2} G_1(s_1) k_{10}(s_3,-s_2-s_3)}{8 s_4}-\frac{e^{s_2+s_3} G_1(s_1) k_{13}(-s_2-s_3,s_2)}{8 s_4}-\frac{e^{s_1} k_{17}(s_2,s_3,-s_1-s_2-s_3)}{16 s_4}-\frac{e^{s_1+s_2+s_3} k_{17}(-s_1-s_2-s_3,s_1,s_2)}{16 s_4}-\frac{k_{18}(s_1,s_2,s_3+s_4)}{16 s_4}-\frac{e^{s_1} k_{18}(s_2,s_3+s_4,-s_1-s_2-s_3-s_4)}{16 s_4}-\frac{e^{s_1+s_2+s_3+s_4} k_{18}(-s_1-s_2-s_3-s_4,s_1,s_2)}{16 s_4}-\frac{e^{s_1+s_2} k_{18}(s_3+s_4,-s_1-s_2-s_3-s_4,s_1)}{16 s_4}-\frac{k_{19}(s_1,s_2,s_3)}{16 s_4}-\frac{e^{s_1+s_2} k_{19}(s_3,-s_1-s_2-s_3,s_1)}{16 s_4}-\frac{k_8(s_2,s_3+s_4)}{4 s_1 s_4}-\frac{k_8(s_1+s_2,s_3)}{4 s_1 s_4}-\frac{e^{s_1+s_2} k_{11}(s_3,-s_1-s_2-s_3)}{4 s_1 s_4}-\frac{e^{s_2} k_{11}(s_3+s_4,-s_2-s_3-s_4)}{4 s_1 s_4}-\frac{e^{s_1+s_2+s_3} k_{12}(-s_1-s_2-s_3,s_1+s_2)}{4 s_1 s_4}-\frac{e^{s_2+s_3+s_4} k_{12}(-s_2-s_3-s_4,s_2)}{4 s_1 s_4}-\frac{k_9(s_2,s_3)}{8 s_1 s_4}-\frac{k_9(s_1+s_2,s_3+s_4)}{8 s_1 s_4}-\frac{e^{s_2} k_{10}(s_3,-s_2-s_3)}{8 s_1 s_4}-\frac{e^{s_1+s_2} k_{10}(s_3+s_4,-s_1-s_2-s_3-s_4)}{8 s_1 s_4}-\frac{e^{s_2+s_3} k_{13}(-s_2-s_3,s_2)}{8 s_1 s_4}-\frac{e^{s_1+s_2+s_3+s_4} k_{13}(-s_1-s_2-s_3-s_4,s_1+s_2)}{8 s_1 s_4}-\frac{k_8(s_1,s_2+s_3+s_4)}{4 s_2 s_4}-\frac{k_8(s_1+s_2,s_3)}{4 s_2 s_4}-\frac{e^{s_1+s_2} k_{11}(s_3,-s_1-s_2-s_3)}{4 s_2 s_4}-\frac{e^{s_1} k_{11}(s_2+s_3+s_4,-s_1-s_2-s_3-s_4)}{4 s_2 s_4}-\frac{e^{s_1+s_2+s_3} k_{12}(-s_1-s_2-s_3,s_1+s_2)}{4 s_2 s_4}-\frac{e^{s_1+s_2+s_3+s_4} k_{12}(-s_1-s_2-s_3-s_4,s_1)}{4 s_2 s_4}-\frac{k_9(s_1,s_2+s_3)}{8 s_2 s_4}-\frac{k_9(s_1+s_2,s_3+s_4)}{8 s_2 s_4}-\frac{e^{s_1} k_{10}(s_2+s_3,-s_1-s_2-s_3)}{8 s_2 s_4}-\frac{e^{s_1+s_2} k_{10}(s_3+s_4,-s_1-s_2-s_3-s_4)}{8 s_2 s_4}-\frac{e^{s_1+s_2+s_3} k_{13}(-s_1-s_2-s_3,s_1)}{8 s_2 s_4}-\frac{e^{s_1+s_2+s_3+s_4} k_{13}(-s_1-s_2-s_3-s_4,s_1+s_2)}{8 s_2 s_4}-\frac{G_1(s_1) k_6(s_2)}{2 s_3 s_4}-\frac{e^{s_2} G_1(s_1) k_7(-s_2)}{2 s_3 s_4}-\frac{k_9(s_1,s_2+s_3+s_4)}{8 s_3 s_4}-\frac{e^{s_1} k_{10}(s_2+s_3+s_4,-s_1-s_2-s_3-s_4)}{8 s_3 s_4}-\frac{e^{s_1+s_2+s_3+s_4} k_{13}(-s_1-s_2-s_3-s_4,s_1)}{8 s_3 s_4}-\frac{k_6(s_2)}{2 s_1 s_3 s_4}-\frac{k_6(s_1+s_2+s_3+s_4)}{2 s_1 s_3 s_4}-\frac{e^{s_2} k_7(-s_2)}{2 s_1 s_3 s_4}-\frac{e^{s_1+s_2+s_3+s_4} k_7(-s_1-s_2-s_3-s_4)}{2 s_1 s_3 s_4}-\frac{k_6(s_1+s_2)}{2 s_2 s_3 s_4}-\frac{e^{s_1+s_2} k_7(-s_1-s_2)}{2 s_2 s_3 s_4}-\frac{k_6(s_1+s_2+s_3)}{4 s_2 (s_2+s_3) s_4}-\frac{k_6(s_1+s_2+s_3+s_4)}{4 s_2 (s_2+s_3) s_4}-\frac{e^{s_1+s_2+s_3} k_7(-s_1-s_2-s_3)}{4 s_2 (s_2+s_3) s_4}-\frac{e^{s_1+s_2+s_3+s_4} k_7(-s_1-s_2-s_3-s_4)}{4 s_2 (s_2+s_3) s_4}-\frac{k_6(s_1+s_2+s_3)}{4 s_3 (s_2+s_3) s_4}-\frac{e^{s_1+s_2+s_3} k_7(-s_1-s_2-s_3)}{4 s_3 (s_2+s_3) s_4}-\frac{G_1(s_1) k_6(s_2+s_3+s_4)}{2 s_3 (s_3+s_4)}-\frac{e^{s_2+s_3+s_4} G_1(s_1) k_7(-s_2-s_3-s_4)}{2 s_3 (s_3+s_4)}-\frac{k_9(s_1,s_2)}{8 s_3 (s_3+s_4)}-\frac{e^{s_1} k_{10}(s_2,-s_1-s_2)}{8 s_3 (s_3+s_4)}-\frac{e^{s_1+s_2} k_{13}(-s_1-s_2,s_1)}{8 s_3 (s_3+s_4)}-\frac{k_6(s_1+s_2)}{2 s_1 s_3 (s_3+s_4)}-\frac{k_6(s_2+s_3+s_4)}{2 s_1 s_3 (s_3+s_4)}-\frac{e^{s_1+s_2} k_7(-s_1-s_2)}{2 s_1 s_3 (s_3+s_4)}-\frac{e^{s_2+s_3+s_4} k_7(-s_2-s_3-s_4)}{2 s_1 s_3 (s_3+s_4)}-\frac{k_6(s_1+s_2+s_3+s_4)}{2 s_2 s_3 (s_3+s_4)}-\frac{e^{s_1+s_2+s_3+s_4} k_7(-s_1-s_2-s_3-s_4)}{2 s_2 s_3 (s_3+s_4)}-\frac{G_1(s_1) k_6(s_2+s_3+s_4)}{2 s_4 (s_3+s_4)}-\frac{e^{s_2+s_3+s_4} G_1(s_1) k_7(-s_2-s_3-s_4)}{2 s_4 (s_3+s_4)}-\frac{k_9(s_1,s_2)}{8 s_4 (s_3+s_4)}-\frac{e^{s_1} k_{10}(s_2,-s_1-s_2)}{8 s_4 (s_3+s_4)}-\frac{e^{s_1+s_2} k_{13}(-s_1-s_2,s_1)}{8 s_4 (s_3+s_4)}-\frac{k_6(s_1+s_2)}{2 s_1 s_4 (s_3+s_4)}-\frac{k_6(s_2+s_3+s_4)}{2 s_1 s_4 (s_3+s_4)}-\frac{e^{s_1+s_2} k_7(-s_1-s_2)}{2 s_1 s_4 (s_3+s_4)}-\frac{e^{s_2+s_3+s_4} k_7(-s_2-s_3-s_4)}{2 s_1 s_4 (s_3+s_4)}-\frac{k_6(s_1+s_2+s_3+s_4)}{2 s_2 s_4 (s_3+s_4)}-\frac{e^{s_1+s_2+s_3+s_4} k_7(-s_1-s_2-s_3-s_4)}{2 s_2 s_4 (s_3+s_4)}-\frac{e^{s_1} k_7(-s_1)}{4 s_2 s_3 (s_2+s_3+s_4)}-\frac{e^{s_1} k_7(-s_1)}{4 s_2 (s_2+s_3) (s_2+s_3+s_4)}-\frac{k_6(s_1+s_2+s_3+s_4)}{4 s_3 (s_2+s_3) (s_2+s_3+s_4)}-\frac{e^{s_1+s_2+s_3+s_4} k_7(-s_1-s_2-s_3-s_4)}{4 s_3 (s_2+s_3) (s_2+s_3+s_4)}-\frac{e^{s_1} k_7(-s_1)}{2 s_2 s_4 (s_2+s_3+s_4)}-\frac{e^{s_1} k_{17}(s_2,s_3,s_4)}{16 (s_1+s_2+s_3+s_4)}-\frac{e^{s_1+s_2+s_3} k_{17}(s_4,-s_2-s_3-s_4,s_2)}{16 (s_1+s_2+s_3+s_4)}-\frac{e^{s_1+s_2+s_3+s_4} k_{18}(s_1,s_2,s_3)}{16 (s_1+s_2+s_3+s_4)}-\frac{e^{s_1} k_{18}(s_2,s_3,-s_1-s_2-s_3)}{16 (s_1+s_2+s_3+s_4)}-\frac{e^{s_1+s_2+s_3} k_{18}(-s_1-s_2-s_3,s_1,s_2)}{16 (s_1+s_2+s_3+s_4)}-\frac{e^{s_1+s_2} k_{18}(s_3,-s_1-s_2-s_3,s_1)}{16 (s_1+s_2+s_3+s_4)}-\frac{e^{s_1+s_2} k_{19}(s_3,s_4,-s_2-s_3-s_4)}{16 (s_1+s_2+s_3+s_4)}-\frac{e^{s_1+s_2+s_3+s_4} k_{19}(-s_2-s_3-s_4,s_2,s_3)}{16 (s_1+s_2+s_3+s_4)}=0. 
   \end{math}
   \end{center}
   }

\section{Lengthy explicit formulas}
\label{explicit3and4fnsappsec}

In Section \ref{ExplicitFormulasSec} we presented the explicit formula for the one and two 
variable functions $K_1, \dots, K_7$ that appear in the formula \eqref{a_4expression} for 
the term $a_4 \in \CNT$. Since the expressions associated with the three and four variable 
functions $K_8, \dots, K_{20}$ are quite lengthy, they are provided in this appendix.

\subsection{The three variable functions $K_8, \dots, K_{16}$}
\label{ExplicitThreeVarsSec}

First we cover the explicit formulas for the three variable functions  
$K_8, K_9, \dots, K_{16}$.

\subsubsection{The function $K_8$} 
We have 
\[
K_8(s_1, s_2, s_3) = \sum_{i=1}^{18} K_{8, i}(s_1, s_2, s_3), 
\]
where 
\[
K_{8, 1}(s_1, s_2, s_3)= -\frac{4 \pi  e^{\frac{3 s_1}{2}+\frac{5 s_2}{2}} \left(\left(e^{s_1} \left(2 e^{s_2}-1\right)-1\right) s_1+\left(e^{s_1}-1\right) s_2\right)}{\left(e^{s_2}-1\right) \left(e^{s_1+s_2}-1\right){}^2 \left(e^{\frac{s_3}{2}}-1\right){}^2 s_1 s_2 \left(s_1+s_2\right)}, 
\]
\[
K_{8, 2}(s_1, s_2, s_3)=\frac{4 \pi  e^{\frac{3 s_1}{2}+\frac{5 s_2}{2}} \left(\left(e^{s_1} \left(2 e^{s_2}-1\right)-1\right) s_1+\left(e^{s_1}-1\right) s_2\right)}{\left(e^{s_2}-1\right) \left(e^{s_1+s_2}-1\right){}^2 \left(e^{\frac{s_3}{2}}+1\right){}^2 s_1 s_2 \left(s_1+s_2\right)}, 
\]
\[
K_{8, 3}(s_1, s_2, s_3)=-\frac{2 \pi  e^{\frac{3 s_1}{2}} \left(s_2-2 s_3\right)}{\left(e^{s_1}-1\right) \left(e^{\frac{1}{2} \left(s_2+s_3\right)}-1\right){}^3 s_1 s_2 s_3}, 
\]
\[
K_{8, 4}(s_1, s_2, s_3)=-\frac{2 \pi  e^{\frac{3 s_1}{2}} \left(s_2-2 s_3\right)}{\left(e^{s_1}-1\right) \left(e^{\frac{1}{2} \left(s_2+s_3\right)}+1\right){}^3 s_1 s_2 s_3}, 
\]
\[
K_{8, 5}(s_1, s_2, s_3)=-\frac{3 \pi  \left(s_1^2-s_3 s_1-s_2^2+2 s_3^2+s_2 s_3\right)}{\left(e^{\frac{1}{2} \left(s_1+s_2+s_3\right)}-1\right){}^4 s_1 \left(s_1+s_2\right) s_3 \left(s_2+s_3\right)}, 
\]
\[
K_{8, 6}(s_1, s_2, s_3)=\frac{3 \pi  \left(s_1^2-s_3 s_1-s_2^2+2 s_3^2+s_2 s_3\right)}{\left(e^{\frac{1}{2} \left(s_1+s_2+s_3\right)}+1\right){}^4 s_1 \left(s_1+s_2\right) s_3 \left(s_2+s_3\right)}, 
\]
\[
K_{8, 7}(s_1, s_2, s_3)=\frac{4 \pi  e^{\frac{3 s_1}{2}} \left(2 \left(e^{s_2} \left(s_3-3\right)-s_3+2\right) s_3+s_2 \left(-2 s_3+e^{s_2} \left(2 s_3-1\right)-1\right)\right)}{\left(e^{s_1}-1\right) \left(e^{s_2}-1\right) \left(e^{\frac{1}{2} \left(s_2+s_3\right)}-1\right) s_1 s_2 s_3 \left(s_2+s_3\right)}, 
\]
\[
K_{8, 8}(s_1, s_2, s_3)=\frac{4 \pi  e^{\frac{3 s_1}{2}} \left(2 \left(e^{s_2} \left(s_3-3\right)-s_3+2\right) s_3+s_2 \left(-2 s_3+e^{s_2} \left(2 s_3-1\right)-1\right)\right)}{\left(e^{s_1}-1\right) \left(e^{s_2}-1\right) \left(e^{\frac{1}{2} \left(s_2+s_3\right)}+1\right) s_1 s_2 s_3 \left(s_2+s_3\right)}, 
\]
{\tiny
\[
K_{8, 9}(s_1, s_2, s_3)=\frac{\pi  e^{\frac{3 s_1}{2}} \left(\left(e^{s_2}+3\right) s_2^2+\left(-7 s_3+e^{s_2} \left(15 s_3+4\right)-4\right) s_2+2 s_3 \left(-5 s_3+e^{s_2} \left(7 s_3-4\right)+4\right)\right)}{\left(e^{s_1}-1\right) \left(e^{s_2}-1\right) \left(e^{\frac{1}{2} \left(s_2+s_3\right)}-1\right){}^2 s_1 s_2 s_3 \left(s_2+s_3\right)}
\]
\[
K_{8, 10}(s_1, s_2, s_3)=\frac{\pi  e^{\frac{3 s_1}{2}} \left(-\left(e^{s_2}+3\right) s_2^2+\left(7 s_3-e^{s_2} \left(15 s_3+4\right)+4\right) s_2-2 s_3 \left(-5 s_3+e^{s_2} \left(7 s_3-4\right)+4\right)\right)}{\left(e^{s_1}-1\right) \left(e^{s_2}-1\right) \left(e^{\frac{1}{2} \left(s_2+s_3\right)}+1\right){}^2 s_1 s_2 s_3 \left(s_2+s_3\right)}, 
\]
\[
K_{8, 11}(s_1, s_2, s_3)=\frac{8 \pi  e^{\frac{3}{2} \left(s_1+s_2\right)} \left(e^{s_2} \left(e^{s_1}-1\right) s_2+s_1 \left(-e^{s_1+2 s_2} \left(s_3-2\right)+e^{s_2} \left(e^{s_1}+1\right) \left(s_3-1\right)-s_3\right)\right)}{\left(e^{s_2}-1\right) \left(e^{s_1+s_2}-1\right){}^2 \left(e^{\frac{s_3}{2}}-1\right) s_1 s_2 \left(s_1+s_2\right) s_3}, 
\]
\[
K_{8, 12}(s_1, s_2, s_3)=\frac{8 \pi  e^{\frac{3}{2} \left(s_1+s_2\right)} \left(e^{s_2} \left(e^{s_1}-1\right) s_2+s_1 \left(-e^{s_1+2 s_2} \left(s_3-2\right)+e^{s_2} \left(e^{s_1}+1\right) \left(s_3-1\right)-s_3\right)\right)}{\left(e^{s_2}-1\right) \left(e^{s_1+s_2}-1\right){}^2 \left(e^{\frac{s_3}{2}}+1\right) s_1 s_2 \left(s_1+s_2\right) s_3}, 
\]
\[
K_{8, 13}(s_1, s_2, s_3)\left(e^{s_1}-1\right) \left(e^{s_1+s_2}-1\right) \left(e^{\frac{1}{2} \left(s_1+s_2+s_3\right)}-1\right){}^3 s_1 s_2 \left(s_1+s_2\right) s_3 \left(s_2+s_3\right) \left(s_1+s_2+s_3\right)=
\]
$\pi  (((e^{s_1} (e^{s_2} (7-5 e^{s_1})+9)-11) s_2-4 e^{s_1} (e^{s_2}-1) s_3) s_1^3+((e^{s_1} (e^{s_2} (11-5 e^{s_1})+9)-15) s_2^2+2 (3 (e^{s_1}-1) (e^{s_1+s_2}-1)-2 (e^{s_1} (2 e^{s_2}-3)+1) s_3) s_2-8 e^{s_1} (e^{s_2}-1) s_3^2) s_1^2+((e^{s_1} (e^{s_2} (5 e^{s_1}+1)-9)+3) s_2^3-4 (e^{s_1} (e^{s_2} (4 e^{s_1}-3)-3)+2) s_3 s_2^2+s_3 ((e^{s_1} (25-7 e^{s_2} (3 e^{s_1}-1))-11) s_3-6 (e^{s_1}-1) (e^{s_1+s_2}-1)) s_2-4 e^{s_1} (e^{s_2}-1) s_3^3) s_1+s_2 (s_2+s_3) (7 s_2^2-11 s_3 s_2-6 s_2-18 s_3^2+12 s_3+e^{2 s_1+s_2} (5 s_2^2-3 (7 s_3+2) s_2+2 s_3 (6-13 s_3))+e^{s_1} (-3 (e^{s_2}+3) s_2^2+(13 s_3+e^{s_2} (19 s_3+6)+6) s_2+2 (e^{s_2}+1) s_3 (11 s_3-6))))$, 
\[
K_{8, 14}(s_1, s_2, s_3)=\frac{K_{8, 13}^{\text{num}}(s_1, s_2, s_3)} {\left(e^{s_1}-1\right) \left(e^{s_1+s_2}-1\right) \left(e^{\frac{1}{2} \left(s_1+s_2+s_3\right)}+1\right){}^3 s_1 s_2 \left(s_1+s_2\right) s_3 \left(s_2+s_3\right) \left(s_1+s_2+s_3\right)},   
\]

\[
K_{8, 15}(s_1, s_2, s_3)\left(e^{s_1}-1\right) \left(e^{s_1+s_2}-1\right){}^2 \left(e^{\frac{1}{2} \left(s_1+s_2+s_3\right)}-1\right) s_1 s_2 \left(s_1+s_2\right) s_3 \left(s_2+s_3\right) \left(s_1+s_2+s_3\right)=
\]
$
4 \pi  (-2 (e^{s_1+s_2}-1) (e^{s_1} (e^{s_2}-1) s_3-(e^{s_1}-1) s_2) s_1^3+(4 (e^{s_1}-1) (e^{s_1+s_2}-1) s_2^2+((-8 e^{s_1}+4 e^{s_1+s_2}-6 e^{2 (s_1+s_2)}+8 e^{2 s_1+s_2}+2) s_3+4 e^{s_1}+6 e^{s_1+s_2}+e^{2 (s_1+s_2)}-6 e^{2 s_1+s_2}-5) s_2-2 e^{s_1} (e^{s_2}-1) s_3 (2 (e^{s_1+s_2}-1) s_3-3 e^{s_1+s_2}+2)) s_1^2+(2 (e^{s_1}-1) (e^{s_1+s_2}-1) s_2^3-2 (e^{s_1+s_2}-1) ((-5 e^{s_1}+e^{s_1+s_2}+2 e^{2 s_1+s_2}+2) s_3+e^{s_1}-e^{2 s_1+s_2}-2) s_2^2-s_3 (2 (5 e^{s_1}-e^{s_1+s_2}+2 e^{2 (s_1+s_2)}-7 e^{2 s_1+s_2}+2 e^{3 s_1+2 s_2}-1) s_3-6 e^{s_1}+10 e^{s_1+s_2}-11 e^{2 (s_1+s_2)}+10 e^{2 s_1+s_2}-2 e^{3 s_1+2 s_2}-1) s_2-2 e^{s_1} (e^{s_2}-1) s_3^2 ((e^{s_1+s_2}-1) s_3-3 e^{s_1+s_2}+2)) s_1-s_2 (s_2+s_3) (2 s_3 ((2 e^{s_1}-1) s_3 (e^{s_1+s_2}-1){}^2-5 e^{s_1}-6 e^{s_1+s_2}+4 e^{2 (s_1+s_2)}+11 e^{2 s_1+s_2}-7 e^{3 s_1+2 s_2}+3)+s_2 (2 (2 e^{s_1}-1) s_3 (e^{s_1+s_2}-1){}^2+2 e^{s_1}+2 e^{s_1+s_2}+e^{2 (s_1+s_2)}-2 e^{2 s_1+s_2}-2 e^{3 s_1+2 s_2}-1))), 
$
\[
K_{8, 16}(s_1, s_2, s_3)=\frac{K_{8, 15}^{\text{num}}(s_1, s_2, s_3)}{\left(e^{s_1}-1\right) \left(e^{s_1+s_2}-1\right){}^2 \left(e^{\frac{1}{2} \left(s_1+s_2+s_3\right)}+1\right) s_1 s_2 \left(s_1+s_2\right) s_3 \left(s_2+s_3\right) \left(s_1+s_2+s_3\right)}, 
\]
\[
K_{8, 17}(s_1, s_2, s_3) 2 \left(e^{s_1}-1\right) \left(e^{s_1+s_2}-1\right){}^2 \left(e^{\frac{1}{2} \left(s_1+s_2+s_3\right)}-1\right){}^2 s_1 s_2 \left(s_1+s_2\right) s_3 \left(s_2+s_3\right) \left(s_1+s_2+s_3\right)=
\]
$
-\pi  (((-29+23 e^{s_1}+38 e^{s_1+s_2}-e^{2 (s_1+s_2)}-34 e^{2 s_1+s_2}+3 e^{3 s_1+2 s_2}) s_2+4 e^{s_1} (-1+e^{s_2}) (-5+7 e^{s_1+s_2}) s_3) s_1^3+((-49+31 e^{s_1}+62 e^{s_1+s_2}-5 e^{2 (s_1+s_2)}-50 e^{2 s_1+s_2}+11 e^{3 s_1+2 s_2}) s_2^2+2 (2 (-5+17 e^{s_1}-9 e^{s_1+s_2}+20 e^{2 (s_1+s_2)}-25 e^{2 s_1+s_2}+2 e^{3 s_1+2 s_2}) s_3-15 e^{s_1}-30 e^{s_1+s_2}+11 e^{2 (s_1+s_2)}+22 e^{2 s_1+s_2}-7 e^{3 s_1+2 s_2}+19) s_2+8 e^{s_1} (-1+e^{s_2}) s_3 ((-5+7 e^{s_1+s_2}) s_3-2 e^{s_1+s_2}+2)) s_1^2+((-11-7 e^{s_1}+10 e^{s_1+s_2}-7 e^{2 (s_1+s_2)}+2 e^{2 s_1+s_2}+13 e^{3 s_1+2 s_2}) s_2^3+4 (4 (-1+e^{s_1+s_2}){}^2+(-10+19 e^{s_1}+9 e^{s_1+s_2}-e^{2 (s_1+s_2)}-41 e^{2 s_1+s_2}+24 e^{3 s_1+2 s_2}) s_3) s_2^2+s_3 (2 (-11-e^{s_1}+30 e^{s_1+s_2}-19 e^{2 (s_1+s_2)}-6 e^{2 s_1+s_2}+7 e^{3 s_1+2 s_2})+(-29+103 e^{s_1}+6 e^{s_1+s_2}+31 e^{2 (s_1+s_2)}-194 e^{2 s_1+s_2}+83 e^{3 s_1+2 s_2}) s_3) s_2+4 e^{s_1} (-1+e^{s_2}) s_3^2 ((-5+7 e^{s_1+s_2}) s_3-4 e^{s_1+s_2}+4)) s_1+s_2 (s_2+s_3) ((9-15 e^{s_1}-14 e^{s_1+s_2}-3 e^{2 (s_1+s_2)}+18 e^{2 s_1+s_2}+5 e^{3 s_1+2 s_2}) s_2^2+((-29+43 e^{s_1}+66 e^{s_1+s_2}-53 e^{2 (s_1+s_2)}-110 e^{2 s_1+s_2}+83 e^{3 s_1+2 s_2}) s_3+30 e^{s_1}+28 e^{s_1+s_2}-6 e^{2 (s_1+s_2)}-44 e^{2 s_1+s_2}+14 e^{3 s_1+2 s_2}-22) s_2+2 s_3 ((-19+29 e^{s_1}+40 e^{s_1+s_2}-25 e^{2 (s_1+s_2)}-64 e^{2 s_1+s_2}+39 e^{3 s_1+2 s_2}) s_3-38 e^{s_1}-68 e^{s_1+s_2}+38 e^{2 (s_1+s_2)}+84 e^{2 s_1+s_2}-46 e^{3 s_1+2 s_2}+30))), 
$
\[
K_{8, 18}(s_1, s_2, s_3)=\frac{-K_{8, 17}^{\text{num}}(s_1, s_2, s_3)}{2 \left(e^{s_1}-1\right) \left(e^{s_1+s_2}-1\right){}^2 \left(e^{\frac{1}{2} \left(s_1+s_2+s_3\right)}+1\right){}^2 s_1 s_2 \left(s_1+s_2\right) s_3 \left(s_2+s_3\right) \left(s_1+s_2+s_3\right)}. 
\]
}

\smallskip

Therefore, we can write the function $K_8$ as a quotient 
\[
K_8(s_1, s_2, s_3) 
= 
\frac{K_8^{\text{num}}(s_1, s_2, s_3)}{K_8^{\text{den}}(s_1, s_2, s_3)}, 
\]
where
\begin{equation} \label{K8den}
K_8^{\text{den}}(s_1, s_2, s_3)= 
\left(e^{s_1}-1\right) \left(e^{s_2}-1\right)  \left(e^{s_1+s_2}-1\right){}^2 \left(e^{s_3}-1\right){}^2 \left(e^{s_2+s_3}-1\right){}^3 \times
\end{equation}
\[
 \left(e^{s_1+s_2+s_3}-1\right){}^4 s_1 s_2 \left(s_1+s_2\right) s_3 \left(s_2+s_3\right) \left(s_1+s_2+s_3\right), 
\]
and $K_8^{\text{num}}(s_1, s_2, s_3)$
is a polynomial in $s_1, s_2, s_3, e^{s_1/2}, e^{s_2/2}, e^{s_3/2}$. The points 
$(i, j, m)$ and $(n, p, q)$ such that 
$s_1^i s_2^j s_3^m  e^{n s_1/2} e^{p s_2/2} e^{q s_3/2}$ appears in $K_8^{\text{num}}(s_1, s_2, s_3)$ are plotted in Figure \ref{K8spowers} and Figure \ref{K8Etospowers}.

\begin{figure}
\includegraphics[scale=0.7]{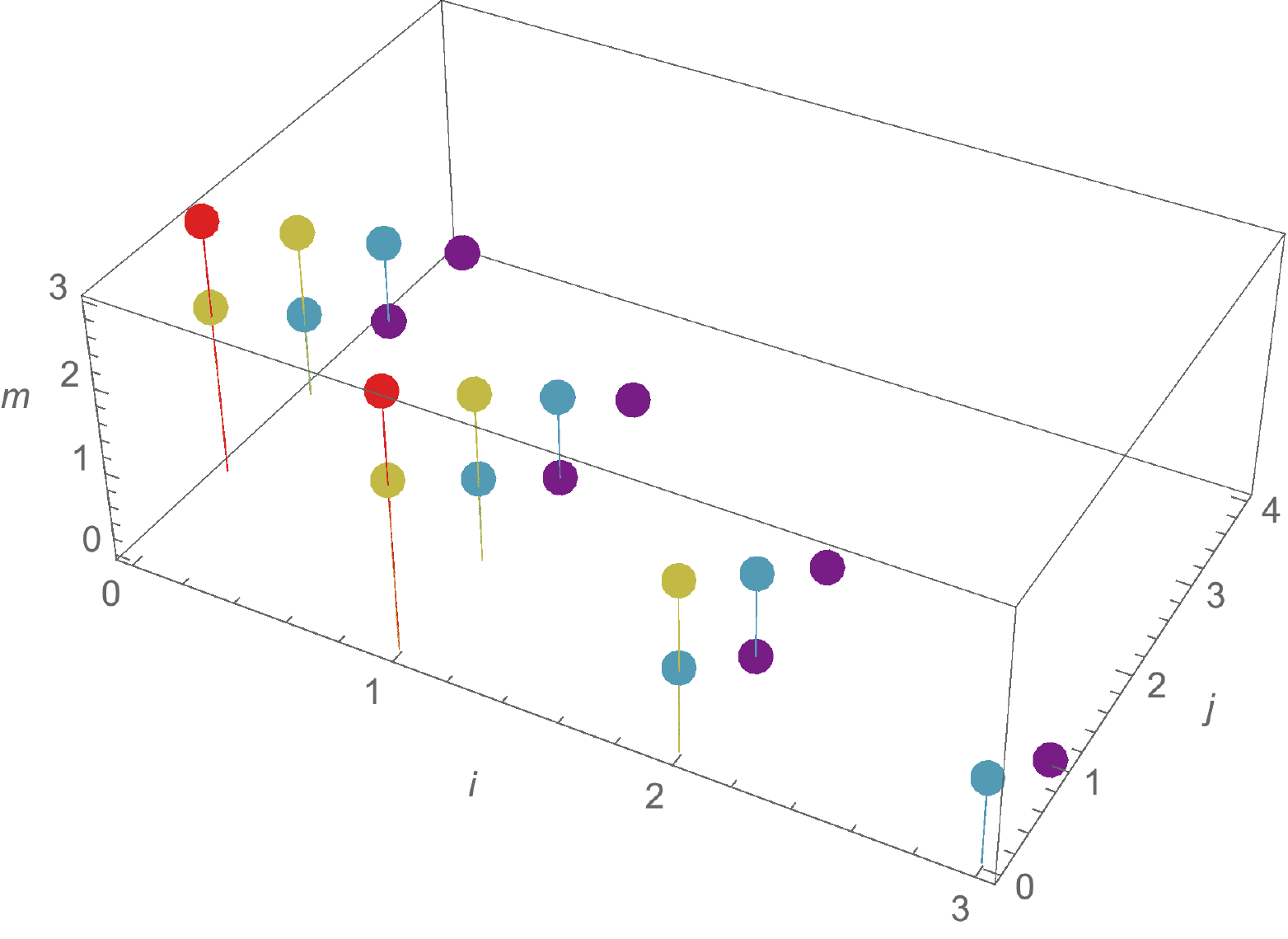}
\caption{The points  $(i, j, m)$  such that 
$s_1^i\, s_2^j\, s_3^m \, e^{n s_1/2}\, e^{p s_2/2} \, e^{q s_3/2}$  appears in the expressions    
for $K_8^{\text{num}}$, $K_9^{\text{num}}$, $\dots,$ $K_{16}^{\text{num}}$.}
\label{K8spowers}
\end{figure}

\begin{figure}
\includegraphics[scale=0.7]{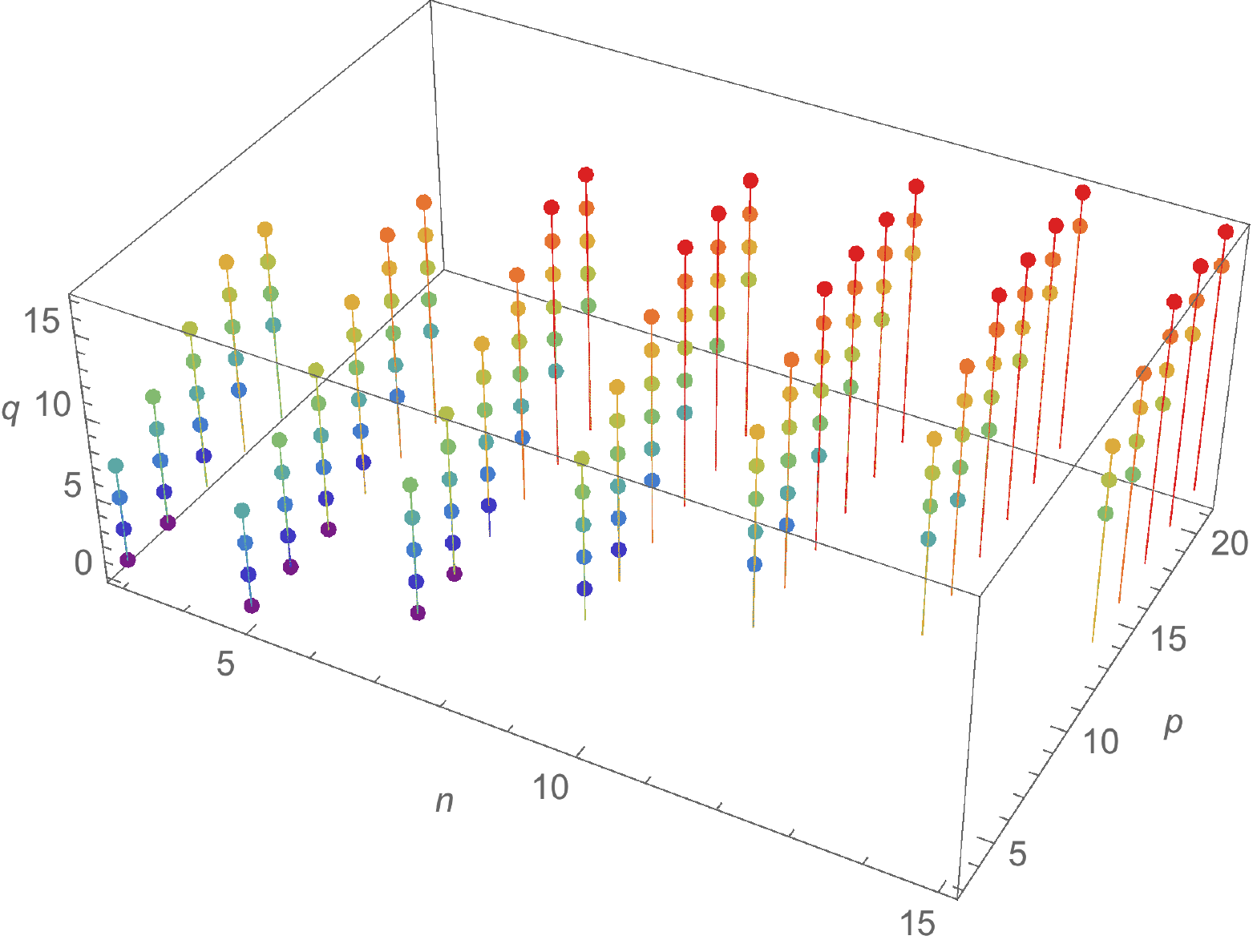}
\caption{The points  $(n, p, q)$  such that 
$s_1^i\, s_2^j\, s_3^m \, e^{n s_1/2}\, e^{p s_2/2} \, e^{q s_3/2}$  appears in the expressions   
for $K_8^{\text{num}}$ and $K_{15}^{\text{num}}$.}
\label{K8Etospowers}
\end{figure}

\subsubsection{The function $K_9$}
We have 
\[
K_9(s_1, s_2, s_3) = \sum_{i=1}^{18} K_{9, i}(s_1, s_2, s_3), 
\]
where 
\[
K_{9, 1}(s_1, s_2, s_3)=-\frac{8 \pi  e^{\frac{3}{2} \left(s_1+s_2\right)} \left(e^{s_2} \left(s_1-s_2+e^{s_1} \left(s_1+s_2\right)\right)-2 s_1\right)}{\left(e^{s_2}-1\right) \left(e^{s_1+s_2}-1\right){}^2 \left(e^{\frac{s_3}{2}}-1\right){}^2 s_1 s_2 \left(s_1+s_2\right)},
\]
\[
K_{9, 2}(s_1, s_2, s_3)= \frac{8 \pi  e^{\frac{3}{2} \left(s_1+s_2\right)} \left(e^{s_2} \left(s_1-s_2+e^{s_1} \left(s_1+s_2\right)\right)-2 s_1\right)}{\left(e^{s_2}-1\right) \left(e^{s_1+s_2}-1\right){}^2 \left(e^{\frac{s_3}{2}}+1\right){}^2 s_1 s_2 \left(s_1+s_2\right)}, 
\]
\[
K_{9, 3}(s_1, s_2, s_3)=-\frac{4 \pi  e^{\frac{3 s_1}{2}} \left(s_2-2 s_3\right)}{\left(e^{s_1}-1\right) \left(e^{\frac{1}{2} \left(s_2+s_3\right)}-1\right){}^3 s_1 s_2 s_3},
\]
\[
K_{9, 4}(s_1, s_2, s_3)=-\frac{4 \pi  e^{\frac{3 s_1}{2}} \left(s_2-2 s_3\right)}{\left(e^{s_1}-1\right) \left(e^{\frac{1}{2} \left(s_2+s_3\right)}+1\right){}^3 s_1 s_2 s_3}, 
\]
\[
K_{9, 5}(s_1, s_2, s_3)=-\frac{16 \pi  e^{\frac{3}{2} \left(s_1+s_2\right)} \left(e^{s_2} \left(s_1-s_2+e^{s_1} \left(s_1+s_2\right)\right)-2 s_1\right) \left(s_3-1\right)}{\left(e^{s_2}-1\right) \left(e^{s_1+s_2}-1\right){}^2 \left(e^{\frac{s_3}{2}}-1\right) s_1 s_2 \left(s_1+s_2\right) s_3}, 
\]
\[
K_{9, 6}(s_1, s_2, s_3)=-\frac{16 \pi  e^{\frac{3}{2} \left(s_1+s_2\right)} \left(e^{s_2} \left(s_1-s_2+e^{s_1} \left(s_1+s_2\right)\right)-2 s_1\right) \left(s_3-1\right)}{\left(e^{s_2}-1\right) \left(e^{s_1+s_2}-1\right){}^2 \left(e^{\frac{s_3}{2}}+1\right) s_1 s_2 \left(s_1+s_2\right) s_3}, 
\]
\[
K_{9, 7}(s_1, s_2, s_3)=-\frac{6 \pi  \left(s_1^2-s_3 s_1-\left(s_2-2 s_3\right) \left(s_2+s_3\right)\right)}{\left(e^{\frac{1}{2} \left(s_1+s_2+s_3\right)}-1\right){}^4 s_1 \left(s_1+s_2\right) s_3 \left(s_2+s_3\right)}, 
\]
\[
K_{9, 8}(s_1, s_2, s_3)=\frac{6 \pi  \left(s_1^2-s_3 s_1-\left(s_2-2 s_3\right) \left(s_2+s_3\right)\right)}{\left(e^{\frac{1}{2} \left(s_1+s_2+s_3\right)}+1\right){}^4 s_1 \left(s_1+s_2\right) s_3 \left(s_2+s_3\right)},
\]
{\tiny 
\[
K_{9, 9}(s_1, s_2, s_3)=\frac{8 \pi  e^{\frac{3 s_1}{2}} \left(2 s_2^2+\left(2 \left(e^{s_2}+1\right) s_3+e^{s_2}-3\right) s_2+2 \left(e^{s_2} \left(s_3-2\right)+1\right) s_3\right)}{\left(e^{s_1}-1\right) \left(e^{s_2}-1\right) \left(e^{\frac{1}{2} \left(s_2+s_3\right)}-1\right) s_1 s_2 s_3 \left(s_2+s_3\right)}, 
\]
\[
K_{9, 10}(s_1, s_2, s_3)=\frac{8 \pi  e^{\frac{3 s_1}{2}} \left(2 s_2^2+\left(2 \left(e^{s_2}+1\right) s_3+e^{s_2}-3\right) s_2+2 \left(e^{s_2} \left(s_3-2\right)+1\right) s_3\right)}{\left(e^{s_1}-1\right) \left(e^{s_2}-1\right) \left(e^{\frac{1}{2} \left(s_2+s_3\right)}+1\right) s_1 s_2 s_3 \left(s_2+s_3\right)},
\]
\[
K_{9, 11}(s_1, s_2, s_3)=\frac{2 \pi  e^{\frac{3 s_1}{2}} \left(\left(7-3 e^{s_2}\right) s_2^2+\left(s_3+e^{s_2} \left(7 s_3+4\right)-4\right) s_2+2 s_3 \left(-3 s_3+e^{s_2} \left(5 s_3-4\right)+4\right)\right)}{\left(e^{s_1}-1\right) \left(e^{s_2}-1\right) \left(e^{\frac{1}{2} \left(s_2+s_3\right)}-1\right){}^2 s_1 s_2 s_3 \left(s_2+s_3\right)},
\]
\[
K_{9, 12}(s_1, s_2, s_3)=\frac{2 \pi  e^{\frac{3 s_1}{2}} \left(\left(3 e^{s_2}-7\right) s_2^2-\left(s_3+e^{s_2} \left(7 s_3+4\right)-4\right) s_2-2 s_3 \left(-3 s_3+e^{s_2} \left(5 s_3-4\right)+4\right)\right)}{\left(e^{s_1}-1\right) \left(e^{s_2}-1\right) \left(e^{\frac{1}{2} \left(s_2+s_3\right)}+1\right){}^2 s_1 s_2 s_3 \left(s_2+s_3\right)}, 
\]
\[
K_{9, 13}(s_1, s_2, s_3) \left(e^{s_1}-1\right) \left(e^{s_1+s_2}-1\right) \left(e^{\frac{1}{2} \left(s_1+s_2+s_3\right)}-1\right){}^3 s_1 s_2 \left(s_1+s_2\right) s_3 \left(s_2+s_3\right) \left(s_1+s_2+s_3\right)=
\]
$
-2 \pi  (((-9 e^{s_1}-7 e^{s_1+s_2}+5 e^{2 s_1+s_2}+11) s_2+4 e^{s_1} (e^{s_2}-1) s_3) s_1^3+((-5 e^{s_1}-7 e^{s_1+s_2}+e^{2 s_1+s_2}+11) s_2^2-2 (3 (e^{s_1}-1) (e^{s_1+s_2}-1)+2 e^{s_1} (-3 e^{s_2}+e^{s_1+s_2}+2) s_3) s_2+8 e^{s_1} (e^{s_2}-1) s_3^2) s_1^2+((17 e^{s_1}+7 e^{s_1+s_2}-13 e^{2 s_1+s_2}-11) s_2^3+4 (e^{s_1}+e^{s_1+s_2}-2) s_3 s_2^2+s_3 (6 (e^{s_1}-1) (e^{s_1+s_2}-1)+(-17 e^{s_1}+e^{s_1+s_2}+13 e^{2 s_1+s_2}+3) s_3) s_2+4 e^{s_1} (e^{s_2}-1) s_3^3) s_1-s_2 (s_2+s_3) ((-13 e^{s_1}-7 e^{s_1+s_2}+9 e^{2 s_1+s_2}+11) s_2^2+((5 e^{s_1}+11 e^{s_1+s_2}-13 e^{2 s_1+s_2}-3) s_3-6 (e^{s_1}-1) (e^{s_1+s_2}-1)) s_2-2 s_3 ((-9 e^{s_1}-9 e^{s_1+s_2}+11 e^{2 s_1+s_2}+7) s_3-6 (e^{s_1}-1) (e^{s_1+s_2}-1)))), 
$
\[
K_{9, 14}(s_1, s_2, s_3)=\frac{K_{9, 13}^{\text{num}}(s_1, s_2, s_3)}{\left(e^{s_1}-1\right) \left(e^{s_1+s_2}-1\right) \left(e^{\frac{1}{2} \left(s_1+s_2+s_3\right)}+1\right){}^3 s_1 s_2 \left(s_1+s_2\right) s_3 \left(s_2+s_3\right) \left(s_1+s_2+s_3\right)},
\]
\[
K_{9, 15}(s_1, s_2, s_3)\left(e^{s_1}-1\right) \left(e^{s_1+s_2}-1\right){}^2 \left(e^{\frac{1}{2} \left(s_1+s_2+s_3\right)}-1\right) s_1 s_2 \left(s_1+s_2\right) s_3 \left(s_2+s_3\right) \left(s_1+s_2+s_3\right)=
\]
$
8 \pi  (2 ((1-2 e^{s_1+s_2}+e^{2 s_1+s_2}) s_2-e^{2 s_1+s_2} (-1+e^{s_2}) s_3) s_1^3+(-2 (-1-2 e^{s_1}+2 e^{s_1+s_2}+e^{2 s_1+s_2}) s_2^2-(2 e^{s_1} (-2-e^{s_1+s_2}+3 e^{s_1+2 s_2}) s_3-2 e^{s_1}-8 e^{s_1+s_2}+e^{2 (s_1+s_2)}+4 e^{2 s_1+s_2}+5) s_2-2 e^{s_1} (-1+e^{s_2}) s_3 (2 e^{s_1+s_2} s_3-2 e^{s_1+s_2}+1)) s_1^2-(2 (1-4 e^{s_1}-2 e^{s_1+s_2}+5 e^{2 s_1+s_2}) s_2^3+2 (e^{s_1} (3-4 e^{s_1+s_2}+e^{2 (s_1+s_2)})+(2-6 e^{s_1}-4 e^{s_1+s_2}+e^{2 (s_1+s_2)}+5 e^{2 s_1+s_2}+2 e^{3 s_1+2 s_2}) s_3) s_2^2+s_3 ((2-4 e^{s_1}-4 e^{s_1+s_2}+4 e^{2 (s_1+s_2)}-2 e^{2 s_1+s_2}+4 e^{3 s_1+2 s_2}) s_3+4 e^{s_1}+12 e^{s_1+s_2}-9 e^{2 (s_1+s_2)}-4 e^{2 s_1+s_2}+2 e^{3 s_1+2 s_2}-5) s_2+2 e^{s_1} (-1+e^{s_2}) s_3^2 (e^{s_1+s_2} s_3-2 e^{s_1+s_2}+1)) s_1-s_2 (s_2+s_3) ((2-4 e^{s_1}-4 e^{s_1+s_2}+6 e^{2 s_1+s_2}) s_2^2+((2-4 e^{s_1}-4 e^{s_1+s_2}-2 e^{2 (s_1+s_2)}+4 e^{2 s_1+s_2}+4 e^{3 s_1+2 s_2}) s_3+8 e^{s_1}+8 e^{s_1+s_2}-e^{2 (s_1+s_2)}-12 e^{2 s_1+s_2}+2 e^{3 s_1+2 s_2}-5) s_2+2 s_3 (e^{2 s_1+s_2} (-1-e^{s_2}+2 e^{s_1+s_2}) s_3-2 e^{s_1}-3 e^{s_1+s_2}+3 e^{2 (s_1+s_2)}+6 e^{2 s_1+s_2}-5 e^{3 s_1+2 s_2}+1))), 
$
\[
K_{9, 16}(s_1, s_2, s_3)=\frac{K_{9, 15}^{\text{num}}(s_1, s_2, s_3)}{\left(e^{s_1}-1\right) \left(e^{s_1+s_2}-1\right){}^2 \left(e^{\frac{1}{2} \left(s_1+s_2+s_3\right)}+1\right) s_1 s_2 \left(s_1+s_2\right) s_3 \left(s_2+s_3\right) \left(s_1+s_2+s_3\right)},  
\]
\[
K_{9, 17}(s_1, s_2, s_3)\left(e^{s_1}-1\right) \left(e^{s_1+s_2}-1\right){}^2 \left(e^{\frac{1}{2} \left(s_1+s_2+s_3\right)}-1\right){}^2 s_1 s_2 \left(s_1+s_2\right) s_3 \left(s_2+s_3\right) \left(s_1+s_2+s_3\right)=
\]
$
-\pi  (((-29+15 e^{s_1}+46 e^{s_1+s_2}-9 e^{2 (s_1+s_2)}-26 e^{2 s_1+s_2}+3 e^{3 s_1+2 s_2}) s_2+4 e^{s_1} (-1+e^{s_2}) (-3+5 e^{s_1+s_2}) s_3) s_1^3-((29+13 e^{s_1}-46 e^{s_1+s_2}+9 e^{2 (s_1+s_2)}-14 e^{2 s_1+s_2}+9 e^{3 s_1+2 s_2}) s_2^2+2 (2 e^{s_1} (-2+9 e^{s_2}+5 e^{s_1+s_2}+3 e^{2 (s_1+s_2)}-15 e^{s_1+2 s_2}) s_3+15 e^{s_1}+30 e^{s_1+s_2}-11 e^{2 (s_1+s_2)}-22 e^{2 s_1+s_2}+7 e^{3 s_1+2 s_2}-19) s_2-8 e^{s_1} (-1+e^{s_2}) s_3 ((-3+5 e^{s_1+s_2}) s_3-2 e^{s_1+s_2}+2)) s_1^2+((29-71 e^{s_1}-46 e^{s_1+s_2}+9 e^{2 (s_1+s_2)}+106 e^{2 s_1+s_2}-27 e^{3 s_1+2 s_2}) s_2^3+4 (4 e^{s_1} (-1+e^{s_1+s_2}){}^2+(10-15 e^{s_1}-17 e^{s_1+s_2}+5 e^{2 (s_1+s_2)}+13 e^{2 s_1+s_2}+4 e^{3 s_1+2 s_2}) s_3) s_2^2+s_3 ((11+23 e^{s_1}-34 e^{s_1+s_2}+31 e^{2 (s_1+s_2)}-74 e^{2 s_1+s_2}+43 e^{3 s_1+2 s_2}) s_3+14 e^{s_1}+92 e^{s_1+s_2}-54 e^{2 (s_1+s_2)}-44 e^{2 s_1+s_2}+30 e^{3 s_1+2 s_2}-38) s_2+4 e^{s_1} (-1+e^{s_2}) s_3^2 ((-3+5 e^{s_1+s_2}) s_3-4 e^{s_1+s_2}+4)) s_1-s_2 (s_2+s_3) ((-29+43 e^{s_1}+46 e^{s_1+s_2}-9 e^{2 (s_1+s_2)}-66 e^{2 s_1+s_2}+15 e^{3 s_1+2 s_2}) s_2^2+((-11+13 e^{s_1}-2 e^{s_1+s_2}+29 e^{2 (s_1+s_2)}+14 e^{2 s_1+s_2}-43 e^{3 s_1+2 s_2}) s_3-46 e^{s_1}-60 e^{s_1+s_2}+22 e^{2 (s_1+s_2)}+76 e^{2 s_1+s_2}-30 e^{3 s_1+2 s_2}+38) s_2-2 s_3 ((-9+15 e^{s_1}+24 e^{s_1+s_2}-19 e^{2 (s_1+s_2)}-40 e^{2 s_1+s_2}+29 e^{3 s_1+2 s_2}) s_3-30 e^{s_1}-52 e^{s_1+s_2}+30 e^{2 (s_1+s_2)}+68 e^{2 s_1+s_2}-38 e^{3 s_1+2 s_2}+22))), 
$
\[
K_{9, 18}(s_1, s_2, s_3)=\frac{-K_{9, 17}^{\text{num}}(s_1, s_2, s_3)}{\left(e^{s_1}-1\right) \left(e^{s_1+s_2}-1\right){}^2 \left(e^{\frac{1}{2} \left(s_1+s_2+s_3\right)}+1\right){}^2 s_1 s_2 \left(s_1+s_2\right) s_3 \left(s_2+s_3\right) \left(s_1+s_2+s_3\right)}.
\]
}

\smallskip

By putting together the above expressions, we have 
\[
K_9(s_1, s_2, s_3) 
= 
\frac{K_9^{\text{num}}(s_1, s_2, s_3)}{K_9^{\text{den}}(s_1, s_2, s_3)}, 
\]
where
\[
K_9^{\text{den}}(s_1, s_2, s_3) = K_8^{\text{den}}(s_1, s_2, s_3),  
\]
which is given by \eqref{K8den}, and $K_9^{\text{num}}$ 
is a polynomial in $s_1, s_2, s_3, e^{s_1/2}, e^{s_2/2}, e^{s_3/2}$. 
The points 
$(i, j, m)$ and $(n, p, q)$ such that 
$s_1^i s_2^j s_3^m  e^{n s_1/2} e^{p s_2/2} e^{q s_3/2}$ appears in $K_9^{\text{num}}(s_1, s_2, s_3)$ are plotted in Figure \ref{K8spowers} and 
Figure \ref{K9Etospowers}.   

\begin{figure}
\includegraphics[scale=0.7]{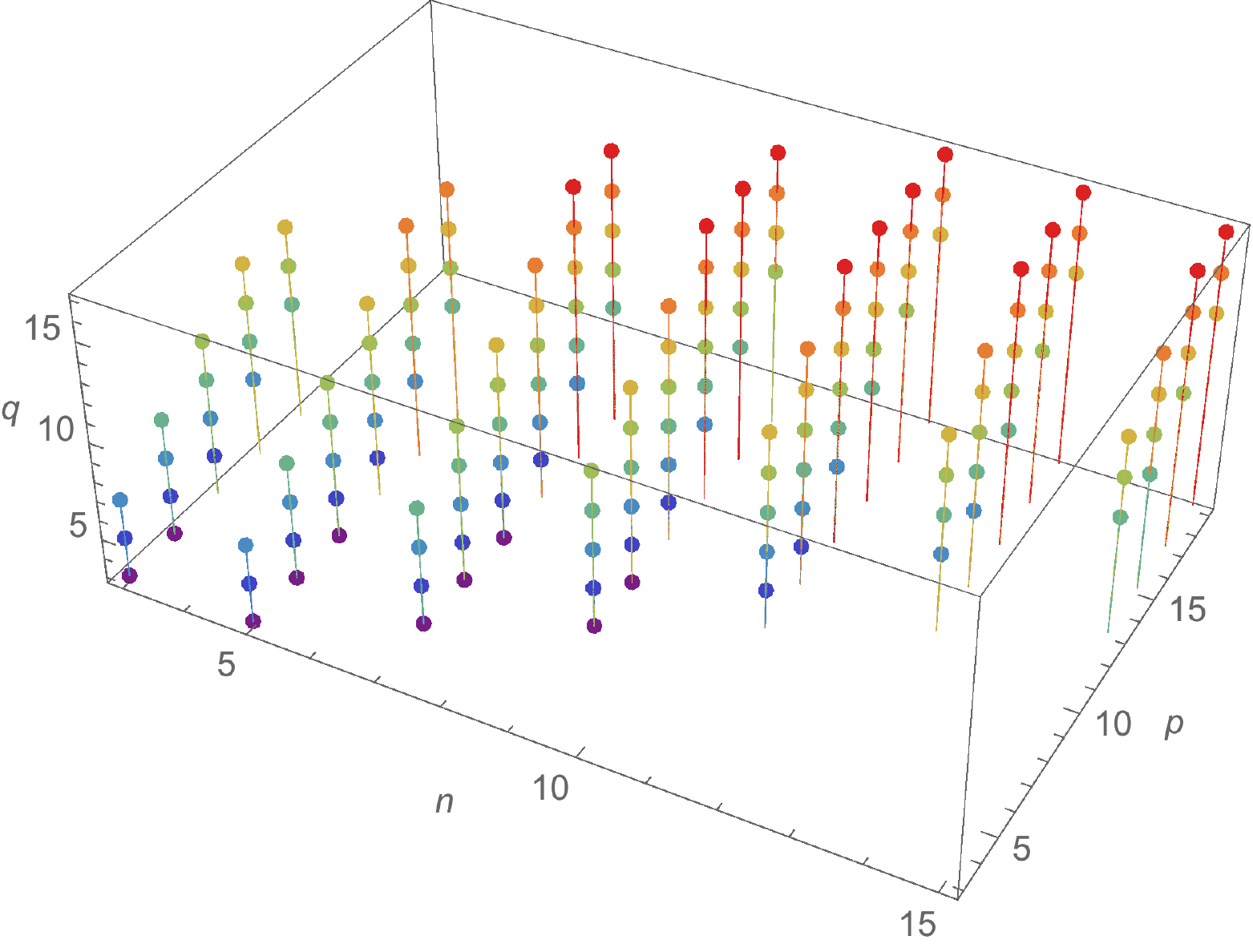}
\caption{The points  $(n, p, q)$  such that 
$s_1^i\, s_2^j\, s_3^m \, e^{n s_1/2}\, e^{p s_2/2} \, e^{q s_3/2}$  appears in the expression   
for $K_9^{\text{num}}$.}
\label{K9Etospowers}
\end{figure}

\subsubsection{The functions $K_{10}, K_{11}, K_{16}$}
We have 
\[
K_{10}(s_1, s_2, s_3) = \sum_{i=1}^{16} K_{10, i}(s_1, s_2, s_3), 
\]
where 
{\tiny 
\[
K_{10, 1}(s_1, s_2, s_3) =  \frac{16 \pi  e^{\frac{3}{2} \left(s_1+s_2\right)} \left(-\left(e^{s_1+2 s_2} \left(e^{s_2} \left(e^{s_1}+1\right)-3\right)+1\right) s_1-e^{s_2} \left(e^{s_1}-1\right) \left(e^{s_2}+e^{s_1+2 s_2}-2\right) s_2\right)}{\left(e^{s_2}-1\right){}^2 \left(e^{s_1+s_2}-1\right){}^3 \left(e^{\frac{s_3}{2}}-1\right) s_1 s_2 \left(s_1+s_2\right)},
\]
\[
K_{10, 2}(s_1, s_2, s_3) =\frac{16 \pi  e^{\frac{3}{2} \left(s_1+s_2\right)} \left(-\left(e^{s_1+2 s_2} \left(e^{s_2} \left(e^{s_1}+1\right)-3\right)+1\right) s_1-e^{s_2} \left(e^{s_1}-1\right) \left(e^{s_2}+e^{s_1+2 s_2}-2\right) s_2\right)}{\left(e^{s_2}-1\right){}^2 \left(e^{s_1+s_2}-1\right){}^3 \left(e^{\frac{s_3}{2}}+1\right) s_1 s_2 \left(s_1+s_2\right)},
\]
}
\[
K_{10, 3}(s_1, s_2, s_3) =-\frac{4 \pi  e^{\frac{3 s_1}{2}} \left(2 s_2-s_3\right)}{\left(e^{s_1}-1\right) \left(e^{\frac{1}{2} \left(s_2+s_3\right)}-1\right){}^3 s_1 s_2 s_3}, 
\]
\[
K_{10, 4}(s_1, s_2, s_3) =-\frac{4 \pi  e^{\frac{3 s_1}{2}} \left(2 s_2-s_3\right)}{\left(e^{s_1}-1\right) \left(e^{\frac{1}{2} \left(s_2+s_3\right)}+1\right){}^3 s_1 s_2 s_3}, 
\]
\[
K_{10, 5}(s_1, s_2, s_3) =\frac{6 \pi  \left(-s_1^2+\left(s_2+2 s_3\right) s_1+\left(2 s_2-s_3\right) \left(s_2+s_3\right)\right)}{\left(e^{\frac{1}{2} \left(s_1+s_2+s_3\right)}-1\right){}^4 s_1 \left(s_1+s_2\right) s_3 \left(s_2+s_3\right)},
\]
\[
K_{10, 6}(s_1, s_2, s_3) =\frac{6 \pi  \left(s_1^2-\left(s_2+2 s_3\right) s_1-\left(2 s_2-s_3\right) \left(s_2+s_3\right)\right)}{\left(e^{\frac{1}{2} \left(s_1+s_2+s_3\right)}+1\right){}^4 s_1 \left(s_1+s_2\right) s_3 \left(s_2+s_3\right)}, 
\]
{\tiny
\[
K_{10, 7}(s_1, s_2, s_3) =\frac{2 \pi  e^{\frac{3 s_1}{2}} \left(\left(10-6 e^{s_2}\right) s_2^2+\left(7 s_3+e^{s_2} \left(s_3+8\right)-8\right) s_2+s_3 \left(-3 s_3+e^{s_2} \left(7 s_3-4\right)+4\right)\right)}{\left(e^{s_1}-1\right) \left(e^{s_2}-1\right) \left(e^{\frac{1}{2} \left(s_2+s_3\right)}-1\right){}^2 s_1 s_2 s_3 \left(s_2+s_3\right)},
\]
\[
K_{10, 8}(s_1, s_2, s_3) =\frac{2 \pi  e^{\frac{3 s_1}{2}} \left(2 \left(3 e^{s_2}-5\right) s_2^2-\left(7 s_3+e^{s_2} \left(s_3+8\right)-8\right) s_2-s_3 \left(-3 s_3+e^{s_2} \left(7 s_3-4\right)+4\right)\right)}{\left(e^{s_1}-1\right) \left(e^{s_2}-1\right) \left(e^{\frac{1}{2} \left(s_2+s_3\right)}+1\right){}^2 s_1 s_2 s_3 \left(s_2+s_3\right)},
\]
$
K_{10, 9}(s_1, s_2, s_3) =
$
\[
\frac{8 \pi  e^{\frac{3 s_1}{2}} \left(\left(4 e^{s_2}-2\right) s_2^2+2 \left(\left(2 e^{s_2}+e^{2 s_2}-1\right) s_3-3 e^{s_2}+e^{2 s_2}+2\right) s_2+s_3 \left(2 e^{2 s_2} s_3+4 e^{s_2}-3 e^{2 s_2}-1\right)\right)}{\left(e^{s_1}-1\right) \left(e^{s_2}-1\right){}^2 \left(e^{\frac{1}{2} \left(s_2+s_3\right)}-1\right) s_1 s_2 s_3 \left(s_2+s_3\right)}, 
\]
$
K_{10, 10}(s_1, s_2, s_3) =
$
\[
\frac{8 \pi  e^{\frac{3 s_1}{2}} \left(\left(4 e^{s_2}-2\right) s_2^2+2 \left(\left(2 e^{s_2}+e^{2 s_2}-1\right) s_3-3 e^{s_2}+e^{2 s_2}+2\right) s_2+s_3 \left(2 e^{2 s_2} s_3+4 e^{s_2}-3 e^{2 s_2}-1\right)\right)}{\left(e^{s_1}-1\right) \left(e^{s_2}-1\right){}^2 \left(e^{\frac{1}{2} \left(s_2+s_3\right)}+1\right) s_1 s_2 s_3 \left(s_2+s_3\right)}, 
\]
\[
K_{10, 11}(s_1, s_2, s_3) \left(e^{s_1}-1\right) \left(e^{s_1+s_2}-1\right) \left(e^{\frac{1}{2} \left(s_1+s_2+s_3\right)}-1\right){}^3 s_1 s_2 \left(s_1+s_2\right) s_3 \left(s_2+s_3\right) \left(s_1+s_2+s_3\right)=
\]
$
-2 \pi  (((-7 e^{s_1}-9 e^{s_1+s_2}+5 e^{2 s_1+s_2}+11) s_2+2 e^{s_1} (e^{s_2}-1) s_3) s_1^3+(-4 (2 e^{s_1}+1) (e^{s_1+s_2}-1) s_2^2+((9 e^{s_1}+11 e^{s_1+s_2}-13 e^{2 s_1+s_2}-7) s_3-6 (e^{s_1}-1) (e^{s_1+s_2}-1)) s_2+4 e^{s_1} (e^{s_2}-1) s_3^2) s_1^2+((37 e^{s_1}+19 e^{s_1+s_2}-31 e^{2 s_1+s_2}-25) s_2^3+(6 (e^{s_1}-1) (e^{s_1+s_2}-1)+(46 e^{s_1}+26 e^{s_1+s_2}-36 e^{2 s_1+s_2}-36) s_3) s_2^2-s_3 ((-7 e^{s_1}-9 e^{s_1+s_2}+5 e^{2 s_1+s_2}+11) s_3-12 (e^{s_1}-1) (e^{s_1+s_2}-1)) s_2+2 e^{s_1} (e^{s_2}-1) s_3^3) s_1-s_2 (s_2+s_3) (2 (-11 e^{s_1}-7 e^{s_1+s_2}+9 e^{2 s_1+s_2}+9) s_2^2+((-13 e^{s_1}-3 e^{s_1+s_2}+5 e^{2 s_1+s_2}+11) s_3-12 (e^{s_1}-1) (e^{s_1+s_2}-1)) s_2-s_3 ((-9 e^{s_1}-11 e^{s_1+s_2}+13 e^{2 s_1+s_2}+7) s_3-6 (e^{s_1}-1) (e^{s_1+s_2}-1)))), 
$
\[
K_{10, 12}(s_1, s_2, s_3)= \frac{K_{10, 11}^{\text{num}}(s_1, s_2, s_3)}{\left(e^{s_1}-1\right) \left(e^{s_1+s_2}-1\right) \left(e^{\frac{1}{2} \left(s_1+s_2+s_3\right)}+1\right){}^3 s_1 s_2 \left(s_1+s_2\right) s_3 \left(s_2+s_3\right) \left(s_1+s_2+s_3\right)},  
\]
\[
K_{10, 13}(s_1, s_2, s_3)\left(e^{s_1}-1\right) \left(e^{s_1+s_2}-1\right){}^2 \left(e^{\frac{1}{2} \left(s_1+s_2+s_3\right)}+1\right){}^2 s_1 s_2 \left(s_1+s_2\right) s_3 \left(s_2+s_3\right) \left(s_1+s_2+s_3\right) = 
\]
$
\pi  (((-29+9 e^{s_1}+52 e^{s_1+s_2}-15 e^{2 (s_1+s_2)}-20 e^{2 s_1+s_2}+3 e^{3 s_1+2 s_2}) s_2+2 e^{s_1} (-1+e^{s_2}) (-3+7 e^{s_1+s_2}) s_3) s_1^3-(4 (1+2 e^{s_1}) (5-8 e^{s_1+s_2}+3 e^{2 (s_1+s_2)}) s_2^2+((-9+31 e^{s_1}+38 e^{s_1+s_2}-45 e^{2 (s_1+s_2)}-42 e^{2 s_1+s_2}+27 e^{3 s_1+2 s_2}) s_3+22 e^{s_1}+68 e^{s_1+s_2}-30 e^{2 (s_1+s_2)}-36 e^{2 s_1+s_2}+14 e^{3 s_1+2 s_2}-38) s_2-4 e^{s_1} (-1+e^{s_2}) s_3 ((-3+7 e^{s_1+s_2}) s_3-2 e^{s_1+s_2}+2)) s_1^2+((47-107 e^{s_1}-92 e^{s_1+s_2}+21 e^{2 (s_1+s_2)}+188 e^{2 s_1+s_2}-57 e^{3 s_1+2 s_2}) s_2^3-2 ((-38+69 e^{s_1}+77 e^{s_1+s_2}-19 e^{2 (s_1+s_2)}-111 e^{2 s_1+s_2}+22 e^{3 s_1+2 s_2}) s_3-27 e^{s_1}-18 e^{s_1+s_2}+7 e^{2 (s_1+s_2)}+50 e^{2 s_1+s_2}-23 e^{3 s_1+2 s_2}+11) s_2^2+s_3 (60 (-1+e^{s_1}) (-1+e^{s_1+s_2}){}^2+(29-25 e^{s_1}-68 e^{s_1+s_2}+31 e^{2 (s_1+s_2)}+20 e^{2 s_1+s_2}+13 e^{3 s_1+2 s_2}) s_3) s_2+2 e^{s_1} (-1+e^{s_2}) s_3^2 ((-3+7 e^{s_1+s_2}) s_3-4 e^{s_1+s_2}+4)) s_1-s_2 (s_2+s_3) (2 (-19+29 e^{s_1}+36 e^{s_1+s_2}-9 e^{2 (s_1+s_2)}-52 e^{2 s_1+s_2}+15 e^{3 s_1+2 s_2}) s_2^2+((-29+43 e^{s_1}+50 e^{s_1+s_2}+11 e^{2 (s_1+s_2)}-62 e^{2 s_1+s_2}-13 e^{3 s_1+2 s_2}) s_3-76 e^{s_1}-104 e^{s_1+s_2}+44 e^{2 (s_1+s_2)}+136 e^{2 s_1+s_2}-60 e^{3 s_1+2 s_2}+60) s_2-s_3 ((-9+15 e^{s_1}+22 e^{s_1+s_2}-29 e^{2 (s_1+s_2)}-42 e^{2 s_1+s_2}+43 e^{3 s_1+2 s_2}) s_3-30 e^{s_1}-60 e^{s_1+s_2}+38 e^{2 (s_1+s_2)}+76 e^{2 s_1+s_2}-46 e^{3 s_1+2 s_2}+22))), 
$
\[
K_{10, 14}(s_1, s_2, s_3)=\frac{-K_{10, 13}^{\text{num}}(s_1, s_2, s_3)}{\left(e^{s_1}-1\right) \left(e^{s_1+s_2}-1\right){}^2 \left(e^{\frac{1}{2} \left(s_1+s_2+s_3\right)}-1\right){}^2 s_1 s_2 \left(s_1+s_2\right) s_3 \left(s_2+s_3\right) \left(s_1+s_2+s_3\right)},
\]
\[
K_{10, 15}(s_1, s_2, s_3) \left(e^{s_1}-1\right) \left(e^{s_1+s_2}-1\right){}^3 \left(e^{\frac{1}{2} \left(s_1+s_2+s_3\right)}-1\right) s_1 s_2 \left(s_1+s_2\right) s_3 \left(s_2+s_3\right) \left(s_1+s_2+s_3\right) =
\]
$
8 \pi  (2 ((-1+3 e^{s_1+s_2}-3 e^{2 (s_1+s_2)}+e^{3 s_1+2 s_2}) s_2-e^{3 s_1+2 s_2} (-1+e^{s_2}) s_3) s_1^3-(2 (1+2 e^{s_1}) (1-3 e^{s_1+s_2}+2 e^{2 (s_1+s_2)}) s_2^2+(2 e^{s_1} (2-6 e^{s_1+s_2}+2 e^{2 (s_1+s_2)}-e^{s_1+2 s_2}+3 e^{2 s_1+3 s_2}) s_3+e^{s_1}+14 e^{s_1+s_2}-11 e^{2 (s_1+s_2)}+2 e^{3 (s_1+s_2)}-4 e^{2 s_1+s_2}+3 e^{3 s_1+2 s_2}-5) s_2+e^{s_1} (-1+e^{s_2}) s_3 (4 e^{2 (s_1+s_2)} s_3+4 e^{s_1+s_2}-3 e^{2 (s_1+s_2)}-1)) s_1^2-(2 (-1+4 e^{s_1}+3 e^{s_1+s_2}-5 e^{2 (s_1+s_2)}-12 e^{2 s_1+s_2}+11 e^{3 s_1+2 s_2}) s_2^3+(2 (-2+6 e^{s_1}+6 e^{s_1+s_2}-10 e^{2 (s_1+s_2)}+e^{3 (s_1+s_2)}-18 e^{2 s_1+s_2}+15 e^{3 s_1+2 s_2}+2 e^{4 s_1+3 s_2}) s_3-9 e^{s_1}-4 e^{s_1+s_2}+3 e^{2 (s_1+s_2)}+24 e^{2 s_1+s_2}-19 e^{3 s_1+2 s_2}+4 e^{4 s_1+3 s_2}+1) s_2^2+2 s_3 ((3-4 e^{s_1}-4 e^{s_1+s_2}+2 e^{2 s_1+s_2}) (-1+e^{s_1+s_2}){}^2+(-1+2 e^{s_1}+3 e^{s_1+s_2}-5 e^{2 (s_1+s_2)}+2 e^{3 (s_1+s_2)}-6 e^{2 s_1+s_2}+3 e^{3 s_1+2 s_2}+2 e^{4 s_1+3 s_2}) s_3) s_2+e^{s_1} (-1+e^{s_2}) s_3^2 (2 e^{2 (s_1+s_2)} s_3+4 e^{s_1+s_2}-3 e^{2 (s_1+s_2)}-1)) s_1-s_2 (s_2+s_3) (2 (-1+2 e^{s_1}+3 e^{s_1+s_2}-4 e^{2 (s_1+s_2)}-6 e^{2 s_1+s_2}+6 e^{3 s_1+2 s_2}) s_2^2+2 ((-1+2 e^{s_1}+3 e^{s_1+s_2}-5 e^{2 (s_1+s_2)}-e^{3 (s_1+s_2)}-6 e^{2 s_1+s_2}+6 e^{3 s_1+2 s_2}+2 e^{4 s_1+3 s_2}) s_3-5 e^{s_1}-9 e^{s_1+s_2}+7 e^{2 (s_1+s_2)}-e^{3 (s_1+s_2)}+14 e^{2 s_1+s_2}-11 e^{3 s_1+2 s_2}+2 e^{4 s_1+3 s_2}+3) s_2+s_3 (2 e^{2 (s_1+s_2)} (-1-e^{s_1+s_2}+2 e^{2 s_1+s_2}) s_3+2 e^{s_1}+3 e^{s_1+s_2}-7 e^{2 (s_1+s_2)}+5 e^{3 (s_1+s_2)}-8 e^{2 s_1+s_2}+14 e^{3 s_1+2 s_2}-8 e^{4 s_1+3 s_2}-1))), 
$
\[
K_{10, 16}(s_1, s_2, s_3)=\frac{K_{10, 15}^{\text{num}}(s_1, s_2, s_3)}{ \left(e^{s_1}-1\right) \left(e^{s_1+s_2}-1\right){}^3 \left(e^{\frac{1}{2} \left(s_1+s_2+s_3\right)}+1\right) s_1 s_2 \left(s_1+s_2\right) s_3 \left(s_2+s_3\right) \left(s_1+s_2+s_3\right)}. 
\]
}

\smallskip

In fact, by putting together the above expressions, we can 
write 
\[
K_{10}(s_1, s_2, s_3) 
= 
\frac{K_{10}^{\text{num}}(s_1, s_2, s_3)}{K_{10}^{\text{den}}(s_1, s_2, s_3)}, 
\]
where 
\[
K_{10}^{\text{den}}(s_1, s_2, s_3) =
\left(e^{s_1}-1\right) \left(e^{s_2}-1\right){}^2 \left(e^{s_1+s_2}-1\right){}^3 \left(e^{s_3}-1\right) \left(e^{s_2+s_3}-1\right){}^3 \times 
\]
\[
 \left(e^{s_1+s_2+s_3}-1\right){}^4  
s_1 s_2 \left(s_1+s_2\right) s_3 \left(s_2+s_3\right) \left(s_1+s_2+s_3\right), 
\]
and 
$K_{10}^{\text{num}}$ 
is a polynomial in $s_1, s_2, s_3, e^{s_1/2}, e^{s_2/2}, e^{s_3/2}$. 
The points 
$(i, j, m)$ and $(n, p, q)$ such that 
$s_1^i s_2^j s_3^m  e^{n s_1/2} e^{p s_2/2} e^{q s_3/2}$ appears in $K_{10}^{\text{num}}(s_1, s_2, s_3)$ are plotted in Figure \ref{K8spowers} 
and Figure \ref{K10Etospowers}.  

\begin{figure}
\includegraphics[scale=0.7]{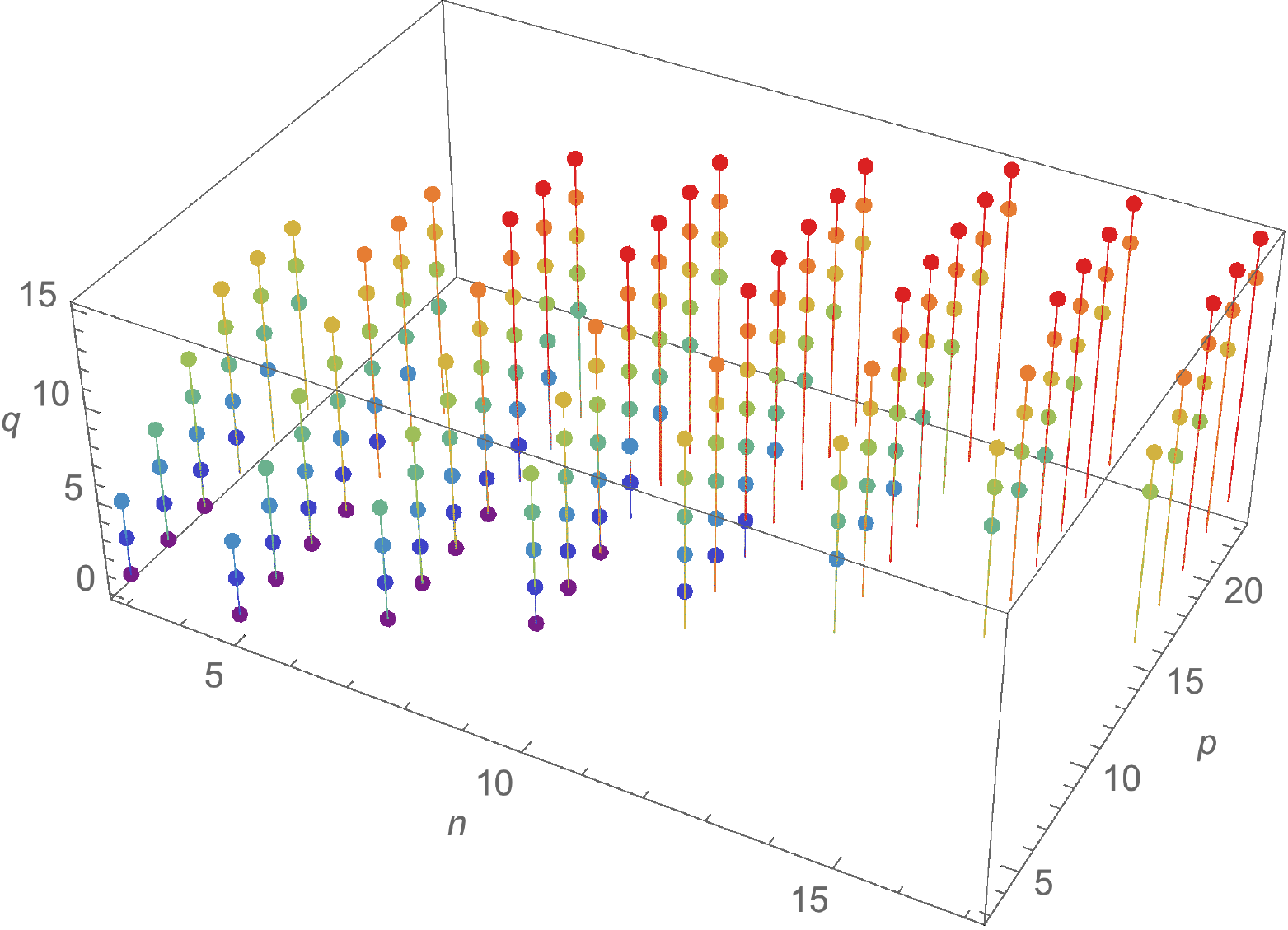}
\caption{The points $(n, p, q)$  such that 
$s_1^i\, s_2^j\, s_3^m \, e^{n s_1/2}\, e^{p s_2/2} \, e^{q s_3/2}$  appears in the expression   
for $K_{10}^{\text{num}}$.}
\label{K10Etospowers}
\end{figure}

\smallskip

The functions $K_{11}$ and $K_{16}$ are scalar multiples of $K_{10}$: 
\[
K_{11}(s_1, s_2, s_3) =\frac{1}{2} K_{10}(s_1, s_2, s_3), 
\]
\[
K_{16}(s_1, s_2, s_3) = \frac{3}{2} K_{10}(s_1, s_2, s_3). 
\]

\subsubsection{The function $K_{12}$} We have
\[
K_{12}(s_1, s_2, s_3) = \sum_{i=1}^{14} K_{12, i}(s_1, s_2, s_3), 
\]
where 
\[
 K_{12, 1}(s_1, s_2, s_3) = 
\frac{8 \pi  e^{\frac{3}{2} \left(s_1+s_2\right)} \left(e^{s_2} \left(2 e^{s_1}-1\right) s_1+s_1-e^{s_1+2 s_2} \left(s_1-s_2+e^{s_1} \left(s_1+s_2\right)\right)\right)}{\left(e^{s_2}-1\right) \left(e^{s_1+s_2}-1\right){}^3 \left(e^{\frac{s_3}{2}}-1\right) s_1 s_2 \left(s_1+s_2\right)},
\]
\[
 K_{12, 2}(s_1, s_2, s_3) =\frac{8 \pi  e^{\frac{3}{2} \left(s_1+s_2\right)} \left(e^{s_2} \left(2 e^{s_1}-1\right) s_1+s_1-e^{s_1+2 s_2} \left(s_1-s_2+e^{s_1} \left(s_1+s_2\right)\right)\right)}{\left(e^{s_2}-1\right) \left(e^{s_1+s_2}-1\right){}^3 \left(e^{\frac{s_3}{2}}+1\right) s_1 s_2 \left(s_1+s_2\right)}, 
\]
\[
 K_{12, 3}(s_1, s_2, s_3) =-\frac{4 \pi  e^{\frac{5 s_1}{2}} \left(s_2-s_3\right)}{\left(e^{s_1}-1\right){}^2 \left(e^{\frac{1}{2} \left(s_2+s_3\right)}-1\right){}^2 s_1 s_2 s_3}, 
\]
\[
 K_{12, 4}(s_1, s_2, s_3) =\frac{4 \pi  e^{\frac{5 s_1}{2}} \left(s_2-s_3\right)}{\left(e^{s_1}-1\right){}^2 \left(e^{\frac{1}{2} \left(s_2+s_3\right)}+1\right){}^2 s_1 s_2 s_3}, 
\]
\[
 K_{12, 5}(s_1, s_2, s_3) =-\frac{3 \pi  \left(2 s_1^2+\left(s_2-s_3\right) s_1-s_2^2+s_3^2\right)}{\left(e^{\frac{1}{2} \left(s_1+s_2+s_3\right)}-1\right){}^4 s_1 \left(s_1+s_2\right) s_3 \left(s_2+s_3\right)}, 
\]
\[
 K_{12, 6}(s_1, s_2, s_3) =\frac{3 \pi  \left(2 s_1^2+\left(s_2-s_3\right) s_1-s_2^2+s_3^2\right)}{\left(e^{\frac{1}{2} \left(s_1+s_2+s_3\right)}+1\right){}^4 s_1 \left(s_1+s_2\right) s_3 \left(s_2+s_3\right)}, 
\]
{\tiny
$
 K_{12, 7}(s_1, s_2, s_3) =
$
\[
\frac{8 \pi  e^{\frac{3 s_1}{2}} \left(e^{s_1} s_2^2+\left(e^{s_1} \left(s_3+e^{s_2} \left(s_3+1\right)-1\right)-\left(e^{s_2}-1\right) s_3\right) s_2+s_3 \left(e^{s_2} \left(e^{s_1} \left(s_3-1\right)-s_3\right)+e^{s_1}+s_3\right)\right)}{\left(e^{s_1}-1\right){}^2 \left(e^{s_2}-1\right) \left(e^{\frac{1}{2} \left(s_2+s_3\right)}-1\right) s_1 s_2 s_3 \left(s_2+s_3\right)}, 
\]
$
 K_{12, 8}(s_1, s_2, s_3) =
$
\[
\frac{8 \pi  e^{\frac{3 s_1}{2}} \left(e^{s_1} s_2^2+\left(e^{s_1} \left(s_3+e^{s_2} \left(s_3+1\right)-1\right)-\left(e^{s_2}-1\right) s_3\right) s_2+s_3 \left(e^{s_2} \left(e^{s_1} \left(s_3-1\right)-s_3\right)+e^{s_1}+s_3\right)\right)}{\left(e^{s_1}-1\right){}^2 \left(e^{s_2}-1\right) \left(e^{\frac{1}{2} \left(s_2+s_3\right)}+1\right) s_1 s_2 s_3 \left(s_2+s_3\right)}, 
\]
\[
 K_{12, 9}(s_1, s_2, s_3)\left(e^{s_1}-1\right) \left(e^{s_1+s_2}-1\right) \left(e^{\frac{1}{2} \left(s_1+s_2+s_3\right)}-1\right){}^3 s_1 s_2 \left(s_1+s_2\right) s_3 \left(s_2+s_3\right) \left(s_1+s_2+s_3\right) =
\]
$
\pi  (-2 ((-7 e^{s_1}-7 e^{s_1+s_2}+5 e^{2 s_1+s_2}+9) s_2+2 e^{s_1} (e^{s_2}-1) s_3) s_1^3+((13 e^{s_1}+19 e^{s_1+s_2}-7 e^{2 s_1+s_2}-25) s_2^2+(12 (e^{s_1}-1) (e^{s_1+s_2}-1)+(11 e^{s_1}-7 e^{s_1+s_2}+3 e^{2 s_1+s_2}-7) s_3) s_2-8 e^{s_1} (e^{s_2}-1) s_3^2) s_1^2+2 (2 (4 e^{s_1}-1) (e^{s_1+s_2}-1) s_2^3+(3 (e^{s_1}-1) (e^{s_1+s_2}-1)+(-6 e^{s_1}-6 e^{s_1+s_2}+8 e^{2 s_1+s_2}+4) s_3) s_2^2-s_3 (3 (e^{s_1}-1) (e^{s_1+s_2}-1)+(-4 e^{s_1}+6 e^{s_1+s_2}-2) s_3) s_2-2 e^{s_1} (e^{s_2}-1) s_3^3) s_1+s_2 (s_2+s_3) ((-15 e^{s_1}-9 e^{s_1+s_2}+13 e^{2 s_1+s_2}+11) s_2^2-2 (e^{s_1}-1) (2 s_3+3 e^{s_1+s_2}-3) s_2-s_3 ((-11 e^{s_1}-9 e^{s_1+s_2}+13 e^{2 s_1+s_2}+7) s_3-6 (e^{s_1}-1) (e^{s_1+s_2}-1)))), 
$
\[
 K_{12, 10}(s_1, s_2, s_3)=\frac{K_{12, 9}^{\text{num}}(s_1, s_2, s_3)}{\left(e^{s_1}-1\right) \left(e^{s_1+s_2}-1\right) \left(e^{\frac{1}{2} \left(s_1+s_2+s_3\right)}+1\right){}^3 s_1 s_2 \left(s_1+s_2\right) s_3 \left(s_2+s_3\right) \left(s_1+s_2+s_3\right) }, 
\]
\[
 K_{12, 11}(s_1, s_2, s_3)2 \left(e^{s_1}-1\right){}^2 \left(e^{s_1+s_2}-1\right){}^2 \left(e^{\frac{1}{2} \left(s_1+s_2+s_3\right)}-1\right){}^2 s_1 s_2 \left(s_1+s_2\right) s_3 \left(s_2+s_3\right) \left(s_1+s_2+s_3\right) =
\]
$
\pi  (2 ((-19+28 e^{s_1}-5 e^{2 s_1}+36 e^{s_1+s_2}-9 e^{2 (s_1+s_2)}-56 e^{2 s_1+s_2}+12 e^{3 s_1+s_2}+12 e^{3 s_1+2 s_2}+e^{4 s_1+2 s_2}) s_2-2 e^{s_1} (-1+e^{s_2}) (1-3 e^{s_1}-5 e^{s_1+s_2}+7 e^{2 s_1+s_2}) s_3) s_1^3+((-47+34 e^{s_1}+37 e^{2 s_1}+92 e^{s_1+s_2}-21 e^{2 (s_1+s_2)}-88 e^{2 s_1+s_2}-52 e^{3 s_1+s_2}+6 e^{3 s_1+2 s_2}+39 e^{4 s_1+2 s_2}) s_2^2+(4 (-1+e^{s_1}) (-15+11 e^{s_1}+26 e^{s_1+s_2}-11 e^{2 (s_1+s_2)}-18 e^{2 s_1+s_2}+7 e^{3 s_1+2 s_2})+(-9-10 e^{s_1}+11 e^{2 s_1}+8 e^{s_1+s_2}+57 e^{2 (s_1+s_2)}+8 e^{3 s_1+s_2}-102 e^{3 s_1+2 s_2}+37 e^{4 s_1+2 s_2}) s_3) s_2-8 e^{s_1} (-1+e^{s_2}) s_3 ((1-3 e^{s_1}-5 e^{s_1+s_2}+7 e^{2 s_1+s_2}) s_3-2 (-1+e^{s_1}) (-1+e^{s_1+s_2}))) s_1^2+2 (2 (5-25 e^{s_1}+26 e^{2 s_1}-8 e^{s_1+s_2}+3 e^{2 (s_1+s_2)}+40 e^{2 s_1+s_2}-44 e^{3 s_1+s_2}-15 e^{3 s_1+2 s_2}+18 e^{4 s_1+2 s_2}) s_2^3+(2 (10-37 e^{s_1}+29 e^{2 s_1}-17 e^{s_1+s_2}+11 e^{2 (s_1+s_2)}+54 e^{2 s_1+s_2}-41 e^{3 s_1+s_2}-25 e^{3 s_1+2 s_2}+16 e^{4 s_1+2 s_2}) s_3-(-1+e^{s_1}) (11+5 e^{s_1}-18 e^{s_1+s_2}+7 e^{2 (s_1+s_2)}-14 e^{2 s_1+s_2}+9 e^{3 s_1+2 s_2})) s_2^2-s_3 ((-1+e^{s_1}) (-19+11 e^{s_1}+50 e^{s_1+s_2}-31 e^{2 (s_1+s_2)}-34 e^{2 s_1+s_2}+23 e^{3 s_1+2 s_2})+2 (-5+11 e^{s_1}+10 e^{s_1+s_2}-13 e^{2 (s_1+s_2)}-12 e^{2 s_1+s_2}-10 e^{3 s_1+s_2}+17 e^{3 s_1+2 s_2}+2 e^{4 s_1+2 s_2}) s_3) s_2-2 e^{s_1} (-1+e^{s_2}) s_3^2 ((1-3 e^{s_1}-5 e^{s_1+s_2}+7 e^{2 s_1+s_2}) s_3-4 (-1+e^{s_1}) (-1+e^{s_1+s_2}))) s_1+s_2 (s_2+s_3) ((29-78 e^{s_1}+57 e^{2 s_1}-52 e^{s_1+s_2}+15 e^{2 (s_1+s_2)}+136 e^{2 s_1+s_2}-100 e^{3 s_1+s_2}-42 e^{3 s_1+2 s_2}+35 e^{4 s_1+2 s_2}) s_2^2-2 (-1+e^{s_1}) (2 (5-9 e^{s_1}-7 e^{s_1+s_2}-2 e^{2 (s_1+s_2)}+11 e^{2 s_1+s_2}+2 e^{3 s_1+2 s_2}) s_3+27 e^{s_1}+34 e^{s_1+s_2}-15 e^{2 (s_1+s_2)}-50 e^{2 s_1+s_2}+23 e^{3 s_1+2 s_2}-19) s_2-s_3 ((9-22 e^{s_1}+21 e^{2 s_1}-24 e^{s_1+s_2}+23 e^{2 (s_1+s_2)}+64 e^{2 s_1+s_2}-56 e^{3 s_1+s_2}-58 e^{3 s_1+2 s_2}+43 e^{4 s_1+2 s_2}) s_3-2 (-1+e^{s_1}) (-11+19 e^{s_1}+26 e^{s_1+s_2}-15 e^{2 (s_1+s_2)}-42 e^{2 s_1+s_2}+23 e^{3 s_1+2 s_2})))), 
$
\[
 K_{12, 12}(s_1, s_2, s_3)= \frac{-K_{12, 11}^{\text{num}}(s_1, s_2, s_3)}{2 \left(e^{s_1}-1\right){}^2 \left(e^{s_1+s_2}-1\right){}^2 \left(e^{\frac{1}{2} \left(s_1+s_2+s_3\right)}+1\right){}^2 s_1 s_2 \left(s_1+s_2\right) s_3 \left(s_2+s_3\right) \left(s_1+s_2+s_3\right)} , 
\]
\[
 K_{12, 13}(s_1, s_2, s_3)\left(e^{s_1}-1\right){}^2 \left(e^{s_1+s_2}-1\right){}^3 \left(e^{\frac{1}{2} \left(s_1+s_2+s_3\right)}-1\right) s_1 s_2 \left(s_1+s_2\right) s_3 \left(s_2+s_3\right) \left(s_1+s_2+s_3\right) =
\]
$
4 \pi  (2 ((1-e^{s_1}-e^{2 s_1}-3 e^{s_1+s_2}+4 e^{2 (s_1+s_2)}+3 e^{2 s_1+s_2}+2 e^{3 s_1+s_2}-6 e^{3 s_1+2 s_2}+e^{4 s_1+2 s_2}) s_2-e^{2 s_1} (-1+e^{s_2}) (-1-e^{s_2}+2 e^{s_1+s_2}+e^{2 (s_1+s_2)}-e^{s_1+2 s_2}) s_3) s_1^3-2 ((-1-e^{s_1}+5 e^{2 s_1}+3 e^{s_1+s_2}-5 e^{2 (s_1+s_2)}+3 e^{2 s_1+s_2}-12 e^{3 s_1+s_2}+4 e^{3 s_1+2 s_2}+4 e^{4 s_1+2 s_2}) s_2^2+(e^{s_1} (-2+7 e^{s_1}+6 e^{s_1+s_2}+7 e^{2 (s_1+s_2)}+3 e^{3 (s_1+s_2)}-16 e^{2 s_1+s_2}-4 e^{s_1+2 s_2}+2 e^{3 s_1+2 s_2}-3 e^{2 s_1+3 s_2}) s_3-4 e^{s_1}-9 e^{s_1+s_2}+7 e^{2 (s_1+s_2)}-e^{3 (s_1+s_2)}+13 e^{2 s_1+s_2}-e^{3 s_1+s_2}-10 e^{3 s_1+2 s_2}+e^{4 s_1+3 s_2}+e^{5 s_1+3 s_2}+3) s_2+e^{2 s_1} (-1+e^{s_2}) s_3 (2 (-1-e^{s_2}+2 e^{s_1+s_2}+e^{2 (s_1+s_2)}-e^{s_1+2 s_2}) s_3-2 e^{s_2}+4 e^{s_1+s_2}-3 e^{2 (s_1+s_2)}+2 e^{s_1+2 s_2}-1)) s_1^2-(2 (1-5 e^{s_1}+7 e^{2 s_1}-3 e^{s_1+s_2}+2 e^{2 (s_1+s_2)}+15 e^{2 s_1+s_2}-18 e^{3 s_1+s_2}-10 e^{3 s_1+2 s_2}+11 e^{4 s_1+2 s_2}) s_2^3+(2 (2-8 e^{s_1}+13 e^{2 s_1}-6 e^{s_1+s_2}+3 e^{2 (s_1+s_2)}-e^{3 (s_1+s_2)}+24 e^{2 s_1+s_2}-32 e^{3 s_1+s_2}-11 e^{3 s_1+2 s_2}+15 e^{4 s_1+2 s_2}-e^{4 s_1+3 s_2}+2 e^{5 s_1+3 s_2}) s_3+6 e^{s_1}-11 e^{2 s_1}-4 e^{s_1+s_2}+3 e^{2 (s_1+s_2)}-12 e^{2 s_1+s_2}+28 e^{3 s_1+s_2}+10 e^{3 s_1+2 s_2}-25 e^{4 s_1+2 s_2}-4 e^{4 s_1+3 s_2}+8 e^{5 s_1+3 s_2}+1) s_2^2+s_3 (2 (1-3 e^{s_1}+7 e^{2 s_1}-3 e^{s_1+s_2}-2 e^{3 (s_1+s_2)}+9 e^{2 s_1+s_2}-16 e^{3 s_1+s_2}+2 e^{3 s_1+2 s_2}+3 e^{4 s_1+2 s_2}+2 e^{5 s_1+3 s_2}) s_3+14 e^{s_1}-7 e^{2 s_1}+14 e^{s_1+s_2}-19 e^{2 (s_1+s_2)}+10 e^{3 (s_1+s_2)}-34 e^{2 s_1+s_2}+14 e^{3 s_1+s_2}+38 e^{3 s_1+2 s_2}-13 e^{4 s_1+2 s_2}-18 e^{4 s_1+3 s_2}+6 e^{5 s_1+3 s_2}-5) s_2+2 e^{2 s_1} (-1+e^{s_2}) s_3^2 ((-1-e^{s_2}+2 e^{s_1+s_2}+e^{2 (s_1+s_2)}-e^{s_1+2 s_2}) s_3-2 e^{s_2}+4 e^{s_1+s_2}-3 e^{2 (s_1+s_2)}+2 e^{s_1+2 s_2}-1)) s_1-s_2 (s_2+s_3) (2 (1-3 e^{s_1}+3 e^{2 s_1}-3 e^{s_1+s_2}+3 e^{2 (s_1+s_2)}+9 e^{2 s_1+s_2}-8 e^{3 s_1+s_2}-8 e^{3 s_1+2 s_2}+6 e^{4 s_1+2 s_2}) s_2^2+(2 (1-3 e^{s_1}+4 e^{2 s_1}-3 e^{s_1+s_2}+3 e^{2 (s_1+s_2)}+e^{3 (s_1+s_2)}+9 e^{2 s_1+s_2}-10 e^{3 s_1+s_2}-7 e^{3 s_1+2 s_2}+6 e^{4 s_1+2 s_2}-3 e^{4 s_1+3 s_2}+2 e^{5 s_1+3 s_2}) s_3+14 e^{s_1}-11 e^{2 s_1}+14 e^{s_1+s_2}-11 e^{2 (s_1+s_2)}+2 e^{3 (s_1+s_2)}-38 e^{2 s_1+s_2}+30 e^{3 s_1+s_2}+30 e^{3 s_1+2 s_2}-25 e^{4 s_1+2 s_2}-6 e^{4 s_1+3 s_2}+6 e^{5 s_1+3 s_2}-5) s_2+s_3 (2 e^{2 s_1} (1-2 e^{s_1+s_2}+2 e^{3 (s_1+s_2)}+e^{s_1+2 s_2}+e^{s_1+3 s_2}-3 e^{2 s_1+3 s_2}) s_3-2 e^{s_1}+3 e^{2 s_1}-4 e^{s_1+s_2}+7 e^{2 (s_1+s_2)}-4 e^{3 (s_1+s_2)}+10 e^{2 s_1+s_2}-12 e^{3 s_1+s_2}-18 e^{3 s_1+2 s_2}+17 e^{4 s_1+2 s_2}+10 e^{4 s_1+3 s_2}-8 e^{5 s_1+3 s_2}+1))), 
$
\[
 K_{12, 14}(s_1, s_2, s_3)=\frac{K_{12, 13}^{\text{num}}(s_1, s_2, s_3)}{\left(e^{s_1}-1\right){}^2 \left(e^{s_1+s_2}-1\right){}^3 \left(e^{\frac{1}{2} \left(s_1+s_2+s_3\right)}+1\right) s_1 s_2 \left(s_1+s_2\right) s_3 \left(s_2+s_3\right) \left(s_1+s_2+s_3\right)}.  
\]
}

\smallskip

Using the above expressions, we have  
\[
K_{12}(s_1, s_2, s_3) 
= 
\frac{K_{12}^{\text{num}}(s_1, s_2, s_3)}{K_{12}^{\text{den}}(s_1, s_2, s_3)}, 
\]
where 
\begin{equation} \label{K12den}
K_{12}^{\text{den}}(s_1, s_2, s_3) = 
\left(e^{s_1}-1\right){}^2 \left(e^{s_2}-1\right) \left(e^{s_1+s_2}-1\right){}^3 \left(e^{s_3}-1\right) \left(e^{s_2+s_3}-1\right){}^2 \times 
\end{equation}
\[
 \left(e^{s_1+s_2+s_3}-1\right){}^4 s_1 s_2 \left(s_1+s_2\right) s_3 \left(s_2+s_3\right) \left(s_1+s_2+s_3\right),
\]
and $K_{12}^{\text{num}}$ 
is a polynomial in $s_1, s_2, s_3, e^{s_1/2}, e^{s_2/2}, e^{s_3/2}$. 
The points 
$(i, j, m)$ and $(n, p, q)$ such that 
$s_1^i s_2^j s_3^m  e^{n s_1/2} e^{p s_2/2} e^{q s_3/2}$ appears in $K_{12}^{\text{num}}(s_1, s_2, s_3)$ are plotted in Figure \ref{K8spowers} 
and Figure \ref{K12Etospowers}.  

\begin{figure}
\includegraphics[scale=0.7]{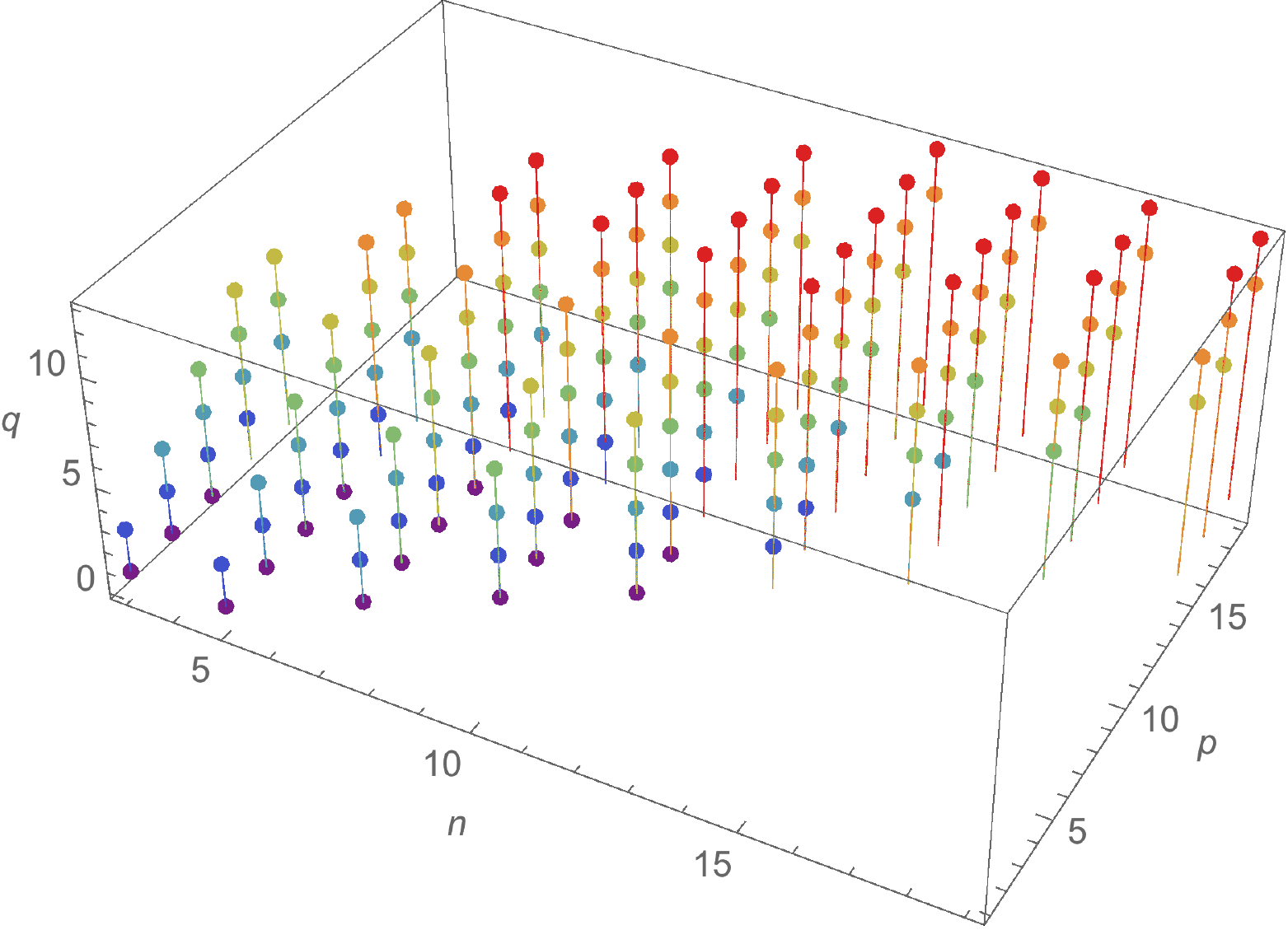}
\caption{The points $(n, p, q)$  such that 
$s_1^i\, s_2^j\, s_3^m \, e^{n s_1/2}\, e^{p s_2/2} \, e^{q s_3/2}$  appears in the expressions   
for $K_{12}^{\text{num}}$ and $K_{14}^{\text{num}}$.}
\label{K12Etospowers}
\end{figure}

\subsubsection{The function $K_{13}$}
We have 
\[
K_{13}(s_1, s_2, s_3) = \sum_{i=1}^{14} K_{13, i}(s_1, s_2, s_3), 
\]
where 
\[
 K_{13, 1}(s_1, s_2, s_3) = -\frac{16 \pi  e^{\frac{3}{2} \left(s_1+s_2\right)} \left(\left(e^{s_1+s_2} \left(2 e^{s_2}-1\right)-1\right) s_1+e^{s_2} \left(e^{s_1}-1\right) s_2\right)}{\left(e^{s_2}-1\right) \left(e^{s_1+s_2}-1\right){}^3 \left(e^{\frac{s_3}{2}}-1\right) s_1 s_2 \left(s_1+s_2\right)},
\]
\[
 K_{13, 2}(s_1, s_2, s_3) =-\frac{16 \pi  e^{\frac{3}{2} \left(s_1+s_2\right)} \left(\left(e^{s_1+s_2} \left(2 e^{s_2}-1\right)-1\right) s_1+e^{s_2} \left(e^{s_1}-1\right) s_2\right)}{\left(e^{s_2}-1\right) \left(e^{s_1+s_2}-1\right){}^3 \left(e^{\frac{s_3}{2}}+1\right) s_1 s_2 \left(s_1+s_2\right)},
\]
\[
 K_{13, 3}(s_1, s_2, s_3) =-\frac{8 \pi  e^{\frac{3 s_1}{2}} \left(s_2-s_3\right)}{\left(e^{s_1}-1\right){}^2 \left(e^{\frac{1}{2} \left(s_2+s_3\right)}-1\right){}^2 s_1 s_2 s_3}, 
\]
\[
 K_{13, 4}(s_1, s_2, s_3) =\frac{8 \pi  e^{\frac{3 s_1}{2}} \left(s_2-s_3\right)}{\left(e^{s_1}-1\right){}^2 \left(e^{\frac{1}{2} \left(s_2+s_3\right)}+1\right){}^2 s_1 s_2 s_3}, 
\]
\[
 K_{13, 5}(s_1, s_2, s_3) =-\frac{6 \pi  \left(2 s_1^2+\left(s_2-s_3\right) s_1-s_2^2+s_3^2\right)}{\left(e^{\frac{1}{2} \left(s_1+s_2+s_3\right)}-1\right){}^4 s_1 \left(s_1+s_2\right) s_3 \left(s_2+s_3\right)}, 
\]
\[
 K_{13, 6}(s_1, s_2, s_3) =\frac{6 \pi  \left(2 s_1^2+\left(s_2-s_3\right) s_1-s_2^2+s_3^2\right)}{\left(e^{\frac{1}{2} \left(s_1+s_2+s_3\right)}+1\right){}^4 s_1 \left(s_1+s_2\right) s_3 \left(s_2+s_3\right)}, 
\]
\[
 K_{13, 7}(s_1, s_2, s_3) =\frac{16 \pi  e^{\frac{3 s_1}{2}} \left(s_2^2+\left(e^{s_2} \left(2 s_3+1\right)-1\right) s_2+s_3 \left(-s_3+e^{s_2} \left(2 s_3-1\right)+1\right)\right)}{\left(e^{s_1}-1\right){}^2 \left(e^{s_2}-1\right) \left(e^{\frac{1}{2} \left(s_2+s_3\right)}-1\right) s_1 s_2 s_3 \left(s_2+s_3\right)}, 
\]
\[
 K_{13, 8}(s_1, s_2, s_3) = \frac{16 \pi  e^{\frac{3 s_1}{2}} \left(s_2^2+\left(e^{s_2} \left(2 s_3+1\right)-1\right) s_2+s_3 \left(-s_3+e^{s_2} \left(2 s_3-1\right)+1\right)\right)}{\left(e^{s_1}-1\right){}^2 \left(e^{s_2}-1\right) \left(e^{\frac{1}{2} \left(s_2+s_3\right)}+1\right) s_1 s_2 s_3 \left(s_2+s_3\right)}, 
\]
{\tiny 
\[
 K_{13, 9}(s_1, s_2, s_3)\left(e^{s_1}-1\right) \left(e^{s_1+s_2}-1\right) \left(e^{\frac{1}{2} \left(s_1+s_2+s_3\right)}-1\right){}^3 s_1 s_2 \left(s_1+s_2\right) s_3 \left(s_2+s_3\right) \left(s_1+s_2+s_3\right) = 
\]
$
-2 \pi  (2 ((-9 e^{s_1}-9 e^{s_1+s_2}+7 e^{2 s_1+s_2}+11) s_2+2 e^{s_1} (e^{s_2}-1) s_3) s_1^3+((-25 e^{s_1}-31 e^{s_1+s_2}+19 e^{2 s_1+s_2}+37) s_2^2+((-19 e^{s_1}-e^{s_1+s_2}+5 e^{2 s_1+s_2}+15) s_3-12 (e^{s_1}-1) (e^{s_1+s_2}-1)) s_2+8 e^{s_1} (e^{s_2}-1) s_3^2) s_1^2+(-4 (e^{s_1}+2) (e^{s_1+s_2}-1) s_2^3+(-6 (e^{s_1}-1) (e^{s_1+s_2}-1)-4 (e^{s_1}+e^{s_1+s_2}-2) s_3) s_2^2+2 s_3 (3 (e^{s_1}-1) (e^{s_1+s_2}-1)+2 e^{s_1} (2 e^{s_2}+e^{s_1+s_2}-3) s_3) s_2+4 e^{s_1} (e^{s_2}-1) s_3^3) s_1-s_2 (s_2+s_3) ((-11 e^{s_1}-5 e^{s_1+s_2}+9 e^{2 s_1+s_2}+7) s_2^2-2 (e^{s_1}-1) (2 e^{s_1+s_2} s_3+3 e^{s_1+s_2}-3) s_2-s_3 ((-11 e^{s_1}-9 e^{s_1+s_2}+13 e^{2 s_1+s_2}+7) s_3-6 (e^{s_1}-1) (e^{s_1+s_2}-1)))), 
$
\[
 K_{13, 10}(s_1, s_2, s_3)=\frac{K_{13, 9}^{\text{num}}(s_1, s_2, s_3)}{\left(e^{s_1}-1\right) \left(e^{s_1+s_2}-1\right) \left(e^{\frac{1}{2} \left(s_1+s_2+s_3\right)}+1\right){}^3 s_1 s_2 \left(s_1+s_2\right) s_3 \left(s_2+s_3\right) \left(s_1+s_2+s_3\right)},  
\]
\[
 K_{13, 11}(s_1, s_2, s_3)\left(e^{s_1}-1\right){}^2 \left(e^{s_1+s_2}-1\right){}^2 \left(e^{\frac{1}{2} \left(s_1+s_2+s_3\right)}-1\right){}^2 s_1 s_2 \left(s_1+s_2\right) s_3 \left(s_2+s_3\right) \left(s_1+s_2+s_3\right) = 
\]
$
-\pi  (2 ((29-52 e^{s_1}+19 e^{2 s_1}-52 e^{s_1+s_2}+15 e^{2 (s_1+s_2)}+96 e^{2 s_1+s_2}-36 e^{3 s_1+s_2}-28 e^{3 s_1+2 s_2}+9 e^{4 s_1+2 s_2}) s_2+2 e^{s_1} (-1+e^{s_2}) (3-5 e^{s_1}-7 e^{s_1+s_2}+9 e^{2 s_1+s_2}) s_3) s_1^3+((107-178 e^{s_1}+47 e^{2 s_1}-188 e^{s_1+s_2}+57 e^{2 (s_1+s_2)}+328 e^{2 s_1+s_2}-92 e^{3 s_1+s_2}-102 e^{3 s_1+2 s_2}+21 e^{4 s_1+2 s_2}) s_2^2+((49-110 e^{s_1}+69 e^{2 s_1}-48 e^{s_1+s_2}-57 e^{2 (s_1+s_2)}+160 e^{2 s_1+s_2}-128 e^{3 s_1+s_2}+62 e^{3 s_1+2 s_2}+3 e^{4 s_1+2 s_2}) s_3-4 (-1+e^{s_1}) (-19+15 e^{s_1}+34 e^{s_1+s_2}-15 e^{2 (s_1+s_2)}-26 e^{2 s_1+s_2}+11 e^{3 s_1+2 s_2})) s_2+8 e^{s_1} (-1+e^{s_2}) s_3 ((3-5 e^{s_1}-7 e^{s_1+s_2}+9 e^{2 s_1+s_2}) s_3-2 (-1+e^{s_1}) (-1+e^{s_1+s_2}))) s_1^2-2 (2 (-10+11 e^{s_1}+5 e^{2 s_1}+16 e^{s_1+s_2}-6 e^{2 (s_1+s_2)}-20 e^{2 s_1+s_2}-8 e^{3 s_1+s_2}+9 e^{3 s_1+2 s_2}+3 e^{4 s_1+2 s_2}) s_2^3+((-1+e^{s_1}) (-27+11 e^{s_1}+50 e^{s_1+s_2}-23 e^{2 (s_1+s_2)}-18 e^{2 s_1+s_2}+7 e^{3 s_1+2 s_2})-2 (10-17 e^{s_1}+5 e^{2 s_1}-9 e^{s_1+s_2}-5 e^{2 (s_1+s_2)}+26 e^{2 s_1+s_2}-13 e^{3 s_1+s_2}-e^{3 s_1+2 s_2}+4 e^{4 s_1+2 s_2}) s_3) s_2^2-s_3 ((-1+e^{s_1}) (-11+3 e^{s_1}+34 e^{s_1+s_2}-23 e^{2 (s_1+s_2)}-18 e^{2 s_1+s_2}+15 e^{3 s_1+2 s_2})+2 e^{s_1} (-9+15 e^{s_1}+10 e^{s_2}+8 e^{s_1+s_2}+17 e^{2 (s_1+s_2)}-30 e^{2 s_1+s_2}-18 e^{s_1+2 s_2}+7 e^{3 s_1+2 s_2}) s_3) s_2-2 e^{s_1} (-1+e^{s_2}) s_3^2 ((3-5 e^{s_1}-7 e^{s_1+s_2}+9 e^{2 s_1+s_2}) s_3-4 (-1+e^{s_1}) (-1+e^{s_1+s_2}))) s_1-s_2 (s_2+s_3) ((9-30 e^{s_1}+29 e^{2 s_1}-20 e^{s_1+s_2}+3 e^{2 (s_1+s_2)}+56 e^{2 s_1+s_2}-52 e^{3 s_1+s_2}-10 e^{3 s_1+2 s_2}+15 e^{4 s_1+2 s_2}) s_2^2-2 (-1+e^{s_1}) (2 e^{s_1+s_2} (-1-3 e^{s_1}-3 e^{s_1+s_2}+7 e^{2 s_1+s_2}) s_3+19 e^{s_1}+18 e^{s_1+s_2}-7 e^{2 (s_1+s_2)}-34 e^{2 s_1+s_2}+15 e^{3 s_1+2 s_2}-11) s_2-s_3 ((9-30 e^{s_1}+29 e^{2 s_1}-16 e^{s_1+s_2}+15 e^{2 (s_1+s_2)}+64 e^{2 s_1+s_2}-64 e^{3 s_1+s_2}-50 e^{3 s_1+2 s_2}+43 e^{4 s_1+2 s_2}) s_3-2 (-1+e^{s_1}) (-11+19 e^{s_1}+26 e^{s_1+s_2}-15 e^{2 (s_1+s_2)}-42 e^{2 s_1+s_2}+23 e^{3 s_1+2 s_2})))), 
$
\[
 K_{13, 12}(s_1, s_2, s_3)=\frac{-K_{13, 11}^{\text{num}}(s_1, s_2, s_3)}{\left(e^{s_1}-1\right){}^2 \left(e^{s_1+s_2}-1\right){}^2 \left(e^{\frac{1}{2} \left(s_1+s_2+s_3\right)}+1\right){}^2 s_1 s_2 \left(s_1+s_2\right) s_3 \left(s_2+s_3\right) \left(s_1+s_2+s_3\right)}, 
\]
\[
 K_{13, 13}(s_1, s_2, s_3) \left(e^{s_1}-1\right){}^2 \left(e^{s_1+s_2}-1\right){}^3 \left(e^{\frac{1}{2} \left(s_1+s_2+s_3\right)}-1\right) s_1 s_2 \left(s_1+s_2\right) s_3 \left(s_2+s_3\right) \left(s_1+s_2+s_3\right)= 
\]
$
8 \pi  (2 ((2-4 e^{s_1}+e^{2 s_1}-6 e^{s_1+s_2}+6 e^{2 (s_1+s_2)}+11 e^{2 s_1+s_2}-3 e^{3 s_1+s_2}-11 e^{3 s_1+2 s_2}+4 e^{4 s_1+2 s_2}) s_2-e^{2 s_1} (-1+e^{s_2}) (1-3 e^{s_1+s_2}+4 e^{2 (s_1+s_2)}-2 e^{s_1+2 s_2}) s_3) s_1^3+2 ((4-8 e^{s_1}+e^{2 s_1}-12 e^{s_1+s_2}+11 e^{2 (s_1+s_2)}+21 e^{2 s_1+s_2}-3 e^{3 s_1+s_2}-19 e^{3 s_1+2 s_2}+5 e^{4 s_1+2 s_2}) s_2^2+((2-4 e^{s_1}+3 e^{2 s_1}-6 e^{s_1+s_2}+5 e^{2 (s_1+s_2)}+6 e^{3 (s_1+s_2)}+7 e^{2 s_1+s_2}-9 e^{3 s_1+s_2}-5 e^{3 s_1+2 s_2}+13 e^{4 s_1+2 s_2}-12 e^{4 s_1+3 s_2}) s_3+9 e^{s_1}-3 e^{2 s_1}+14 e^{s_1+s_2}-11 e^{2 (s_1+s_2)}+2 e^{3 (s_1+s_2)}-26 e^{2 s_1+s_2}+9 e^{3 s_1+s_2}+21 e^{3 s_1+2 s_2}-7 e^{4 s_1+2 s_2}-4 e^{4 s_1+3 s_2}+e^{5 s_1+3 s_2}-5) s_2-e^{s_1} (-1+e^{s_2}) s_3 (2 e^{s_1} (1-3 e^{s_1+s_2}+4 e^{2 (s_1+s_2)}-2 e^{s_1+2 s_2}) s_3-2 e^{s_1}-4 e^{s_1+s_2}+3 e^{2 (s_1+s_2)}+6 e^{2 s_1+s_2}-4 e^{3 s_1+2 s_2}+1)) s_1^2-(2 (-2+4 e^{s_1}+e^{2 s_1}+6 e^{s_1+s_2}-4 e^{2 (s_1+s_2)}-9 e^{2 s_1+s_2}-3 e^{3 s_1+s_2}+5 e^{3 s_1+2 s_2}+2 e^{4 s_1+2 s_2}) s_2^3+(2 (-2+4 e^{s_1}-e^{2 s_1}+6 e^{s_1+s_2}-2 e^{2 (s_1+s_2)}-4 e^{3 (s_1+s_2)}-5 e^{2 s_1+s_2}+3 e^{3 s_1+s_2}-3 e^{3 s_1+2 s_2}-6 e^{4 s_1+2 s_2}+8 e^{4 s_1+3 s_2}+2 e^{5 s_1+3 s_2}) s_3-14 e^{s_1}+e^{2 s_1}-24 e^{s_1+s_2}+19 e^{2 (s_1+s_2)}-4 e^{3 (s_1+s_2)}+40 e^{2 s_1+s_2}-4 e^{3 s_1+s_2}-34 e^{3 s_1+2 s_2}+3 e^{4 s_1+2 s_2}+8 e^{4 s_1+3 s_2}+9) s_2^2+s_3 (2 e^{2 s_1} (-3+5 e^{s_2}+2 e^{2 s_2}+9 e^{s_1+s_2}-12 e^{2 (s_1+s_2)}+2 e^{3 (s_1+s_2)}-9 e^{s_1+2 s_2}-6 e^{s_1+3 s_2}+12 e^{2 s_1+3 s_2}) s_3+3 e^{2 s_1}+8 e^{s_1+s_2}-19 e^{2 (s_1+s_2)}+12 e^{3 (s_1+s_2)}-4 e^{2 s_1+s_2}-10 e^{3 s_1+s_2}+20 e^{3 s_1+2 s_2}+5 e^{4 s_1+2 s_2}-16 e^{4 s_1+3 s_2}+2 e^{5 s_1+3 s_2}-1) s_2+2 e^{s_1} (-1+e^{s_2}) s_3^2 (e^{s_1} (1-3 e^{s_1+s_2}+4 e^{2 (s_1+s_2)}-2 e^{s_1+2 s_2}) s_3-2 e^{s_1}-4 e^{s_1+s_2}+3 e^{2 (s_1+s_2)}+6 e^{2 s_1+s_2}-4 e^{3 s_1+2 s_2}+1)) s_1-s_2 (s_2+s_3) (2 e^{2 s_1} (1+e^{s_2}+e^{2 s_2}-3 e^{s_1+s_2}+3 e^{2 (s_1+s_2)}-3 e^{s_1+2 s_2}) s_2^2+(4 e^{2 s_1+s_2} (1+e^{s_2}-3 e^{s_1+s_2}+e^{3 s_1+2 s_2}) s_3+4 e^{s_1}-5 e^{2 s_1}+4 e^{s_1+s_2}-3 e^{2 (s_1+s_2)}-12 e^{2 s_1+s_2}+14 e^{3 s_1+s_2}+8 e^{3 s_1+2 s_2}-11 e^{4 s_1+2 s_2}+2 e^{5 s_1+3 s_2}-1) s_2+s_3 (2 e^{2 s_1} (-1+e^{s_2}+e^{2 s_2}+3 e^{s_1+s_2}-3 e^{2 (s_1+s_2)}+2 e^{3 (s_1+s_2)}-3 e^{s_1+2 s_2}) s_3-4 e^{s_1}+5 e^{2 s_1}-2 e^{s_1+s_2}+3 e^{2 (s_1+s_2)}-2 e^{3 (s_1+s_2)}+12 e^{2 s_1+s_2}-16 e^{3 s_1+s_2}-16 e^{3 s_1+2 s_2}+19 e^{4 s_1+2 s_2}+8 e^{4 s_1+3 s_2}-8 e^{5 s_1+3 s_2}+1))), 
$
\[
 K_{13, 14}(s_1, s_2, s_3)=\frac{K_{13, 13}^{\text{num}}(s_1, s_2, s_3)}{\left(e^{s_1}-1\right){}^2 \left(e^{s_1+s_2}-1\right){}^3 \left(e^{\frac{1}{2} \left(s_1+s_2+s_3\right)}+1\right) s_1 s_2 \left(s_1+s_2\right) s_3 \left(s_2+s_3\right) \left(s_1+s_2+s_3\right)}. 
\]
}

\smallskip

Using the above expressions, we have  
\[
K_{13}(s_1, s_2, s_3) 
= 
\frac{K_{13}^{\text{num}}(s_1, s_2, s_3)}{K_{13}^{\text{den}}(s_1, s_2, s_3)}, 
\]
where 
\[
K_{13}^{\text{den}}(s_1, s_2, s_3) =  K_{12}^{\text{den}}(s_1, s_2, s_3), 
\]
which is given by \eqref{K12den}, 
and $K_{13}^{\text{num}}$ 
is a polynomial in $s_1, s_2, s_3, e^{s_1/2}, e^{s_2/2}, e^{s_3/2}$. 
The points 
$(i, j, m)$ and $(n, p, q)$ such that 
$s_1^i s_2^j s_3^m  e^{n s_1/2} e^{p s_2/2} e^{q s_3/2}$ appears in $K_{13}^{\text{num}}(s_1, s_2, s_3)$ are plotted in Figure \ref{K8spowers} 
and Figure \ref{K13Etospowers}.  

\begin{figure}
\includegraphics[scale=0.7]{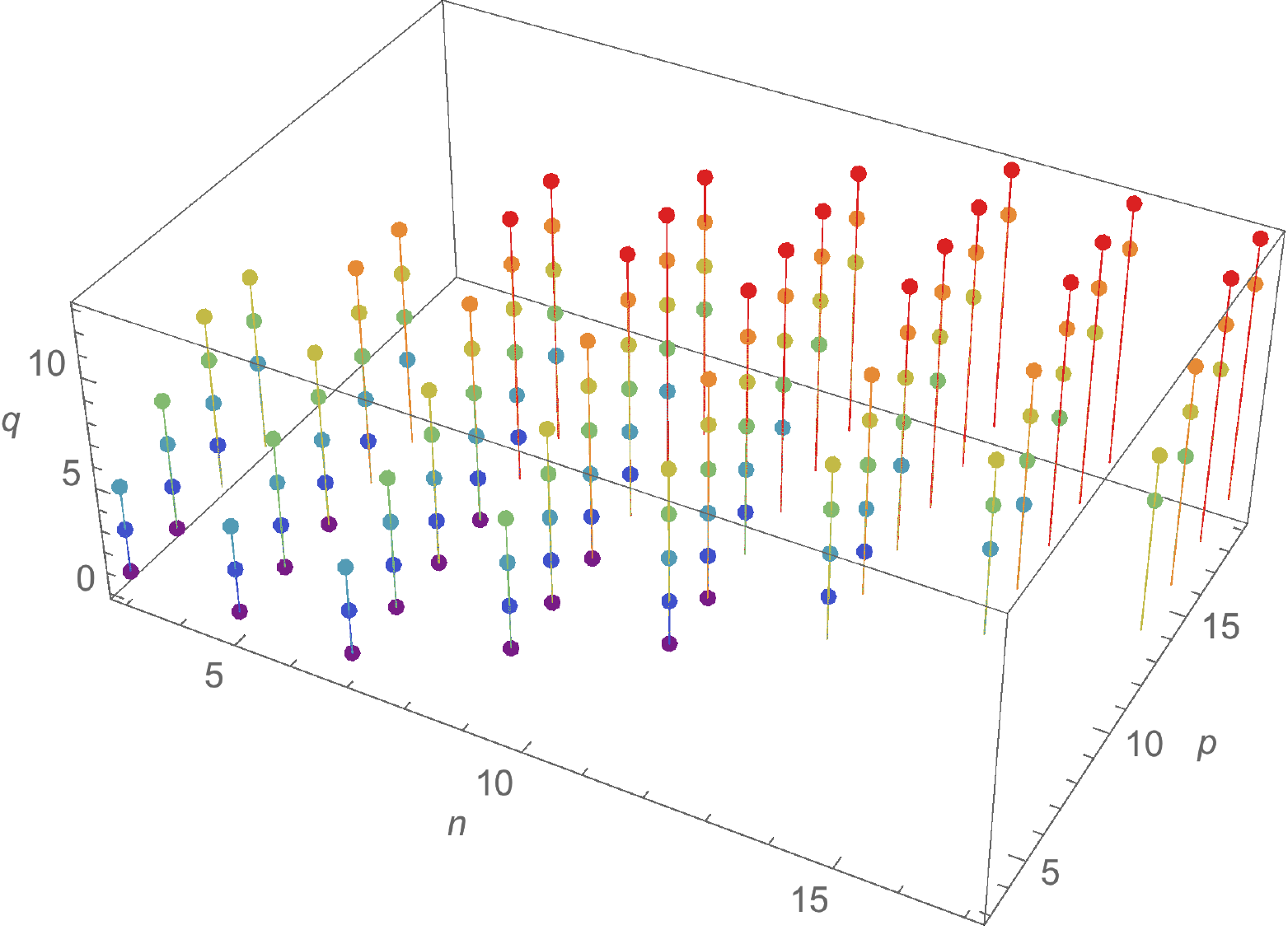}
\caption{The points $(n, p, q)$  such that 
$s_1^i\, s_2^j\, s_3^m \, e^{n s_1/2}\, e^{p s_2/2} \, e^{q s_3/2}$  appears in the expression   
for $K_{13}^{\text{num}}$.}
\label{K13Etospowers}
\end{figure}

\subsubsection{The function $K_{14}$} We have 
\[
K_{14}(s_1, s_2, s_3) = \sum_{i=1}^{14} K_{14, i}(s_1, s_2, s_3), 
\]
where 
{\tiny
\[
 K_{14, 1}(s_1, s_2, s_3) =\frac{8 \pi  e^{\frac{3}{2} \left(s_1+s_2\right)} \left(-\left(e^{s_2} \left(e^{s_1} \left(e^{s_2} \left(e^{s_1}+5\right)-4\right)+1\right)-3\right) s_1-e^{s_2} \left(e^{s_1}-1\right) \left(e^{s_1+s_2}+2\right) s_2\right)}{\left(e^{s_2}-1\right) \left(e^{s_1+s_2}-1\right){}^3 \left(e^{\frac{s_3}{2}}-1\right) s_1 s_2 \left(s_1+s_2\right)},
\]
\[
 K_{14, 2}(s_1, s_2, s_3) = \frac{8 \pi  e^{\frac{3}{2} \left(s_1+s_2\right)} \left(-\left(e^{s_2} \left(e^{s_1} \left(e^{s_2} \left(e^{s_1}+5\right)-4\right)+1\right)-3\right) s_1-e^{s_2} \left(e^{s_1}-1\right) \left(e^{s_1+s_2}+2\right) s_2\right)}{\left(e^{s_2}-1\right) \left(e^{s_1+s_2}-1\right){}^3 \left(e^{\frac{s_3}{2}}+1\right) s_1 s_2 \left(s_1+s_2\right)},
\]
}
\[
 K_{14, 3}(s_1, s_2, s_3) = -\frac{4 \pi  e^{\frac{3 s_1}{2}} \left(e^{s_1}+2\right) \left(s_2-s_3\right)}{\left(e^{s_1}-1\right){}^2 \left(e^{\frac{1}{2} \left(s_2+s_3\right)}-1\right){}^2 s_1 s_2 s_3}, 
\]
\[
 K_{14, 4}(s_1, s_2, s_3) = \frac{4 \pi  e^{\frac{3 s_1}{2}} \left(e^{s_1}+2\right) \left(s_2-s_3\right)}{\left(e^{s_1}-1\right){}^2 \left(e^{\frac{1}{2} \left(s_2+s_3\right)}+1\right){}^2 s_1 s_2 s_3}, 
\]
\[
 K_{14, 5}(s_1, s_2, s_3) = -\frac{9 \pi  \left(2 s_1^2+\left(s_2-s_3\right) s_1-s_2^2+s_3^2\right)}{\left(e^{\frac{1}{2} \left(s_1+s_2+s_3\right)}-1\right){}^4 s_1 \left(s_1+s_2\right) s_3 \left(s_2+s_3\right)},
\]
\[
 K_{14, 6}(s_1, s_2, s_3) = \frac{9 \pi  \left(2 s_1^2+\left(s_2-s_3\right) s_1-s_2^2+s_3^2\right)}{\left(e^{\frac{1}{2} \left(s_1+s_2+s_3\right)}+1\right){}^4 s_1 \left(s_1+s_2\right) s_3 \left(s_2+s_3\right)}, 
\]
{\tiny
\[
 K_{14, 7}(s_1, s_2, s_3)\left(e^{s_1}-1\right){}^2 \left(e^{s_2}-1\right) \left(e^{\frac{1}{2} \left(s_2+s_3\right)}-1\right) s_1 s_2 s_3 \left(s_2+s_3\right) = 
\]
$
8 \pi  e^{\frac{3 s_1}{2}} ((e^{s_1}+2) s_2^2+((e^{s_1}+2) (e^{s_2}-1)+(e^{s_1}+3 e^{s_2}+e^{s_1+s_2}+1) s_3) s_2+s_3 ((3 e^{s_2}+e^{s_1+s_2}-1) s_3-(e^{s_1}+2) (e^{s_2}-1))), 
$
\[
 K_{14, 8}(s_1, s_2, s_3) = \frac{K_{14, 7}^{\text{num}}(s_1, s_2, s_3)}{\left(e^{s_1}-1\right){}^2 \left(e^{s_2}-1\right) \left(e^{\frac{1}{2} \left(s_2+s_3\right)}+1\right) s_1 s_2 s_3 \left(s_2+s_3\right)}, 
\]
\[
 K_{14, 9}(s_1, s_2, s_3)\left(e^{s_1}-1\right) \left(e^{s_1+s_2}-1\right) \left(e^{\frac{1}{2} \left(s_1+s_2+s_3\right)}-1\right){}^3 s_1 s_2 \left(s_1+s_2\right) s_3 \left(s_2+s_3\right) \left(s_1+s_2+s_3\right) = 
\]
$
-\pi  (2 ((-25 e^{s_1}-25 e^{s_1+s_2}+19 e^{2 s_1+s_2}+31) s_2+6 e^{s_1} (e^{s_2}-1) s_3) s_1^3+(9 (-7 e^{s_1}-9 e^{s_1+s_2}+5 e^{2 s_1+s_2}+11) s_2^2+((-49 e^{s_1}+5 e^{s_1+s_2}+7 e^{2 s_1+s_2}+37) s_3-36 (e^{s_1}-1) (e^{s_1+s_2}-1)) s_2+24 e^{s_1} (e^{s_2}-1) s_3^2) s_1^2-2 (6 (2 e^{s_1}+1) (e^{s_1+s_2}-1) s_2^3+(9 (e^{s_1}-1) (e^{s_1+s_2}-1)+(-2 e^{s_1}-2 e^{s_1+s_2}+8 e^{2 s_1+s_2}-4) s_3) s_2^2-s_3 (9 (e^{s_1}-1) (e^{s_1+s_2}-1)+2 (-8 e^{s_1}+7 e^{s_1+s_2}+2 e^{2 s_1+s_2}-1) s_3) s_2-6 e^{s_1} (e^{s_2}-1) s_3^3) s_1-s_2 (s_2+s_3) ((-37 e^{s_1}-19 e^{s_1+s_2}+31 e^{2 s_1+s_2}+25) s_2^2-2 (e^{s_1}-1) (9 (e^{s_1+s_2}-1)+(4 e^{s_1+s_2}+2) s_3) s_2-3 s_3 ((-11 e^{s_1}-9 e^{s_1+s_2}+13 e^{2 s_1+s_2}+7) s_3-6 (e^{s_1}-1) (e^{s_1+s_2}-1)))), 
$
\[
 K_{14, 10}(s_1, s_2, s_3) = \frac{K_{14, 9}^{\text{num}}(s_1, s_2, s_3)}{\left(e^{s_1}-1\right) \left(e^{s_1+s_2}-1\right) \left(e^{\frac{1}{2} \left(s_1+s_2+s_3\right)}+1\right){}^3 s_1 s_2 \left(s_1+s_2\right) s_3 \left(s_2+s_3\right) \left(s_1+s_2+s_3\right)},
\]
\[
 K_{14, 11}(s_1, s_2, s_3)2 \left(e^{s_1}-1\right){}^2 \left(e^{s_1+s_2}-1\right){}^2 \left(e^{\frac{1}{2} \left(s_1+s_2+s_3\right)}-1\right){}^2 s_1 s_2 \left(s_1+s_2\right) s_3 \left(s_2+s_3\right) \left(s_1+s_2+s_3\right) = 
\]
$
-\pi  (2 ((77-132 e^{s_1}+43 e^{2 s_1}-140 e^{s_1+s_2}+39 e^{2 (s_1+s_2)}+248 e^{2 s_1+s_2}-84 e^{3 s_1+s_2}-68 e^{3 s_1+2 s_2}+17 e^{4 s_1+2 s_2}) s_2+2 e^{s_1} (-1+e^{s_2}) (7-13 e^{s_1}-19 e^{s_1+s_2}+25 e^{2 s_1+s_2}) s_3) s_1^3+(3 (87-130 e^{s_1}+19 e^{2 s_1}-156 e^{s_1+s_2}+45 e^{2 (s_1+s_2)}+248 e^{2 s_1+s_2}-44 e^{3 s_1+s_2}-70 e^{3 s_1+2 s_2}+e^{4 s_1+2 s_2}) s_2^2-(4 (-1+e^{s_1}) (-53+41 e^{s_1}+94 e^{s_1+s_2}-41 e^{2 (s_1+s_2)}-70 e^{2 s_1+s_2}+29 e^{3 s_1+2 s_2})+(-107+210 e^{s_1}-127 e^{2 s_1}+104 e^{s_1+s_2}+171 e^{2 (s_1+s_2)}-320 e^{2 s_1+s_2}+264 e^{3 s_1+s_2}-226 e^{3 s_1+2 s_2}+31 e^{4 s_1+2 s_2}) s_3) s_2+8 e^{s_1} (-1+e^{s_2}) s_3 ((7-13 e^{s_1}-19 e^{s_1+s_2}+25 e^{2 s_1+s_2}) s_3-6 (-1+e^{s_1}) (-1+e^{s_1+s_2}))) s_1^2-2 (6 (-5-e^{s_1}+12 e^{2 s_1}+8 e^{s_1+s_2}-3 e^{2 (s_1+s_2)}-20 e^{3 s_1+s_2}+e^{3 s_1+2 s_2}+8 e^{4 s_1+2 s_2}) s_2^3+((-1+e^{s_1}) (-65+17 e^{s_1}+118 e^{s_1+s_2}-53 e^{2 (s_1+s_2)}-22 e^{2 s_1+s_2}+5 e^{3 s_1+2 s_2})+2 (-10-3 e^{s_1}+19 e^{2 s_1}+e^{s_1+s_2}+21 e^{2 (s_1+s_2)}+2 e^{2 s_1+s_2}-15 e^{3 s_1+s_2}-23 e^{3 s_1+2 s_2}+8 e^{4 s_1+2 s_2}) s_3) s_2^2-s_3 ((-1+e^{s_1}) (-41+17 e^{s_1}+118 e^{s_1+s_2}-77 e^{2 (s_1+s_2)}-70 e^{2 s_1+s_2}+53 e^{3 s_1+2 s_2})+2 (-5-7 e^{s_1}+30 e^{2 s_1}+30 e^{s_1+s_2}-49 e^{2 (s_1+s_2)}+4 e^{2 s_1+s_2}-70 e^{3 s_1+s_2}+51 e^{3 s_1+2 s_2}+16 e^{4 s_1+2 s_2}) s_3) s_2-2 e^{s_1} (-1+e^{s_2}) s_3^2 ((7-13 e^{s_1}-19 e^{s_1+s_2}+25 e^{2 s_1+s_2}) s_3-12 (-1+e^{s_1}) (-1+e^{s_1+s_2}))) s_1-s_2 (s_2+s_3) ((47-138 e^{s_1}+115 e^{2 s_1}-92 e^{s_1+s_2}+21 e^{2 (s_1+s_2)}+248 e^{2 s_1+s_2}-204 e^{3 s_1+s_2}-62 e^{3 s_1+2 s_2}+65 e^{4 s_1+2 s_2}) s_2^2-2 (-1+e^{s_1}) (2 (5-9 e^{s_1}-9 e^{s_1+s_2}-8 e^{2 (s_1+s_2)}+5 e^{2 s_1+s_2}+16 e^{3 s_1+2 s_2}) s_3+65 e^{s_1}+70 e^{s_1+s_2}-29 e^{2 (s_1+s_2)}-118 e^{2 s_1+s_2}+53 e^{3 s_1+2 s_2}-41) s_2-s_3 ((27-82 e^{s_1}+79 e^{2 s_1}-56 e^{s_1+s_2}+53 e^{2 (s_1+s_2)}+192 e^{2 s_1+s_2}-184 e^{3 s_1+s_2}-158 e^{3 s_1+2 s_2}+129 e^{4 s_1+2 s_2}) s_3-6 (-1+e^{s_1}) (-11+19 e^{s_1}+26 e^{s_1+s_2}-15 e^{2 (s_1+s_2)}-42 e^{2 s_1+s_2}+23 e^{3 s_1+2 s_2})))),
$
\[
 K_{14, 12}(s_1, s_2, s_3) = \frac{-K_{14, 11}^{\text{num}}(s_1, s_2, s_3)}{2 \left(e^{s_1}-1\right){}^2 \left(e^{s_1+s_2}-1\right){}^2 \left(e^{\frac{1}{2} \left(s_1+s_2+s_3\right)}+1\right){}^2 s_1 s_2 \left(s_1+s_2\right) s_3 \left(s_2+s_3\right) \left(s_1+s_2+s_3\right)},
\]
\[
 K_{14, 13}(s_1, s_2, s_3) \left(e^{s_1}-1\right){}^2 \left(e^{s_1+s_2}-1\right){}^3 \left(e^{\frac{1}{2} \left(s_1+s_2+s_3\right)}-1\right) s_1 s_2 \left(s_1+s_2\right) s_3 \left(s_2+s_3\right) \left(s_1+s_2+s_3\right)= 
\]
$
4 \pi  (2 ((5-9 e^{s_1}+e^{2 s_1}-15 e^{s_1+s_2}+16 e^{2 (s_1+s_2)}+25 e^{2 s_1+s_2}-4 e^{3 s_1+s_2}-28 e^{3 s_1+2 s_2}+9 e^{4 s_1+2 s_2}) s_2-e^{2 s_1} (-1+e^{s_2}) (1-e^{s_2}-4 e^{s_1+s_2}+9 e^{2 (s_1+s_2)}-5 e^{s_1+2 s_2}) s_3) s_1^3+2 (3 (3-5 e^{s_1}-e^{2 s_1}-9 e^{s_1+s_2}+9 e^{2 (s_1+s_2)}+13 e^{2 s_1+s_2}+2 e^{3 s_1+s_2}-14 e^{3 s_1+2 s_2}+2 e^{4 s_1+2 s_2}) s_2^2+(-(-4+6 e^{s_1}+e^{2 s_1}+12 e^{s_1+s_2}-14 e^{2 (s_1+s_2)}-15 e^{3 (s_1+s_2)}-8 e^{2 s_1+s_2}+2 e^{3 s_1+s_2}+17 e^{3 s_1+2 s_2}-24 e^{4 s_1+2 s_2}+27 e^{4 s_1+3 s_2}) s_3+22 e^{s_1}-6 e^{2 s_1}+37 e^{s_1+s_2}-29 e^{2 (s_1+s_2)}+5 e^{3 (s_1+s_2)}-65 e^{2 s_1+s_2}+19 e^{3 s_1+s_2}+52 e^{3 s_1+2 s_2}-14 e^{4 s_1+2 s_2}-9 e^{4 s_1+3 s_2}+e^{5 s_1+3 s_2}-13) s_2-e^{s_1} (-1+e^{s_2}) s_3 (2 e^{s_1} (1-e^{s_2}-4 e^{s_1+s_2}+9 e^{2 (s_1+s_2)}-5 e^{s_1+2 s_2}) s_3-5 e^{s_1}-10 e^{s_1+s_2}+8 e^{2 (s_1+s_2)}+16 e^{2 s_1+s_2}-11 e^{3 s_1+2 s_2}+2)) s_1^2-(6 (-1+e^{s_1}+3 e^{2 s_1}+3 e^{s_1+s_2}-2 e^{2 (s_1+s_2)}-e^{2 s_1+s_2}-8 e^{3 s_1+s_2}+5 e^{4 s_1+2 s_2}) s_2^3+(2 (-2+11 e^{2 s_1}+6 e^{s_1+s_2}-e^{2 (s_1+s_2)}-9 e^{3 (s_1+s_2)}+14 e^{2 s_1+s_2}-26 e^{3 s_1+s_2}-17 e^{3 s_1+2 s_2}+3 e^{4 s_1+2 s_2}+15 e^{4 s_1+3 s_2}+6 e^{5 s_1+3 s_2}) s_3-22 e^{s_1}-9 e^{2 s_1}-52 e^{s_1+s_2}+41 e^{2 (s_1+s_2)}-8 e^{3 (s_1+s_2)}+68 e^{2 s_1+s_2}+20 e^{3 s_1+s_2}-58 e^{3 s_1+2 s_2}-19 e^{4 s_1+2 s_2}+12 e^{4 s_1+3 s_2}+8 e^{5 s_1+3 s_2}+19) s_2^2+s_3 (2 (1-3 e^{s_1}+e^{2 s_1}-3 e^{s_1+s_2}+4 e^{2 (s_1+s_2)}-14 e^{3 (s_1+s_2)}+19 e^{2 s_1+s_2}+2 e^{3 s_1+s_2}-16 e^{3 s_1+2 s_2}-21 e^{4 s_1+2 s_2}+24 e^{4 s_1+3 s_2}+6 e^{5 s_1+3 s_2}) s_3+14 e^{s_1}-e^{2 s_1}+30 e^{s_1+s_2}-57 e^{2 (s_1+s_2)}+34 e^{3 (s_1+s_2)}-42 e^{2 s_1+s_2}-6 e^{3 s_1+s_2}+78 e^{3 s_1+2 s_2}-3 e^{4 s_1+2 s_2}-50 e^{4 s_1+3 s_2}+10 e^{5 s_1+3 s_2}-7) s_2+2 e^{s_1} (-1+e^{s_2}) s_3^2 (e^{s_1} (1-e^{s_2}-4 e^{s_1+s_2}+9 e^{2 (s_1+s_2)}-5 e^{s_1+2 s_2}) s_3-5 e^{s_1}-10 e^{s_1+s_2}+8 e^{2 (s_1+s_2)}+16 e^{2 s_1+s_2}-11 e^{3 s_1+2 s_2}+2)) s_1-s_2 (s_2+s_3) (2 (1-3 e^{s_1}+5 e^{2 s_1}-3 e^{s_1+s_2}+5 e^{2 (s_1+s_2)}+11 e^{2 s_1+s_2}-14 e^{3 s_1+s_2}-14 e^{3 s_1+2 s_2}+12 e^{4 s_1+2 s_2}) s_2^2+(2 (1-3 e^{s_1}+4 e^{2 s_1}-3 e^{s_1+s_2}+7 e^{2 (s_1+s_2)}+e^{3 (s_1+s_2)}+13 e^{2 s_1+s_2}-10 e^{3 s_1+s_2}-19 e^{3 s_1+2 s_2}+6 e^{4 s_1+2 s_2}-3 e^{4 s_1+3 s_2}+6 e^{5 s_1+3 s_2}) s_3+22 e^{s_1}-21 e^{2 s_1}+22 e^{s_1+s_2}-17 e^{2 (s_1+s_2)}+2 e^{3 (s_1+s_2)}-62 e^{2 s_1+s_2}+58 e^{3 s_1+s_2}+46 e^{3 s_1+2 s_2}-47 e^{4 s_1+2 s_2}-6 e^{4 s_1+3 s_2}+10 e^{5 s_1+3 s_2}-7) s_2+s_3 (2 e^{2 s_1} (-1+2 e^{s_2}+2 e^{2 s_2}+4 e^{s_1+s_2}-6 e^{2 (s_1+s_2)}+6 e^{3 (s_1+s_2)}-5 e^{s_1+2 s_2}+e^{s_1+3 s_2}-3 e^{2 s_1+3 s_2}) s_3-10 e^{s_1}+13 e^{2 s_1}-8 e^{s_1+s_2}+13 e^{2 (s_1+s_2)}-8 e^{3 (s_1+s_2)}+34 e^{2 s_1+s_2}-44 e^{3 s_1+s_2}-50 e^{3 s_1+2 s_2}+55 e^{4 s_1+2 s_2}+26 e^{4 s_1+3 s_2}-24 e^{5 s_1+3 s_2}+3))), 
$
\[
 K_{14, 14}(s_1, s_2, s_3) = \frac{K_{14, 13}^{\text{num}}(s_1, s_2, s_3)}{\left(e^{s_1}-1\right){}^2 \left(e^{s_1+s_2}-1\right){}^3 \left(e^{\frac{1}{2} \left(s_1+s_2+s_3\right)}+1\right) s_1 s_2 \left(s_1+s_2\right) s_3 \left(s_2+s_3\right) \left(s_1+s_2+s_3\right)}. 
\]
}

\smallskip

By putting together the above expressions, we have  
\[
K_{14}(s_1, s_2, s_3) 
= 
\frac{K_{14}^{\text{num}}(s_1, s_2, s_3)}{K_{14}^{\text{den}}(s_1, s_2, s_3)}, 
\]
where 
\[
K_{14}^{\text{den}}(s_1, s_2, s_3) = K_{12}^{\text{den}}(s_1, s_2, s_3), 
\]
which is given by \eqref{K12den}, and $K_{14}^{\text{num}}$ 
is a polynomial in $s_1, s_2, s_3, e^{s_1/2}, e^{s_2/2}, e^{s_3/2}$.
The points 
$(i, j, m)$ and $(n, p, q)$ such that 
$s_1^i s_2^j s_3^m  e^{n s_1/2} e^{p s_2/2} e^{q s_3/2}$ appears in 
the expression of $K_{14}^{\text{num}}(s_1, s_2, s_3)$ are plotted in Figure \ref{K8spowers} 
and Figure \ref{K12Etospowers}.

\subsubsection{The function $K_{15}$} We have 
\[
K_{15}(s_1, s_2, s_3) = \sum_{i=1}^{18} K_{15, i}(s_1, s_2, s_3), 
\]
where 
\[
 K_{15, 1}(s_1, s_2, s_3)=-\frac{4 \pi  e^{\frac{3}{2} \left(s_1+s_2\right)} \left(\left(e^{s_2} \left(e^{s_1} \left(2 e^{s_2}+1\right)+1\right)-4\right) s_1+3 e^{s_2} \left(e^{s_1}-1\right) s_2\right)}{\left(e^{s_2}-1\right) \left(e^{s_1+s_2}-1\right){}^2 \left(e^{\frac{s_3}{2}}-1\right){}^2 s_1 s_2 \left(s_1+s_2\right)},
\]
\[
 K_{15, 2}(s_1, s_2, s_3)=\frac{4 \pi  e^{\frac{3}{2} \left(s_1+s_2\right)} \left(\left(e^{s_2} \left(e^{s_1} \left(2 e^{s_2}+1\right)+1\right)-4\right) s_1+3 e^{s_2} \left(e^{s_1}-1\right) s_2\right)}{\left(e^{s_2}-1\right) \left(e^{s_1+s_2}-1\right){}^2 \left(e^{\frac{s_3}{2}}+1\right){}^2 s_1 s_2 \left(s_1+s_2\right)},
\]
\[
 K_{15, 3}(s_1, s_2, s_3)=-\frac{6 \pi  e^{\frac{3 s_1}{2}} \left(s_2-2 s_3\right)}{\left(e^{s_1}-1\right) \left(e^{\frac{1}{2} \left(s_2+s_3\right)}-1\right){}^3 s_1 s_2 s_3}, 
\]
\[
 K_{15, 4}(s_1, s_2, s_3)=-\frac{6 \pi  e^{\frac{3 s_1}{2}} \left(s_2-2 s_3\right)}{\left(e^{s_1}-1\right) \left(e^{\frac{1}{2} \left(s_2+s_3\right)}+1\right){}^3 s_1 s_2 s_3}, 
\]
\[
 K_{15, 5}(s_1, s_2, s_3)=-\frac{9 \pi  \left(s_1^2-s_3 s_1-\left(s_2-2 s_3\right) \left(s_2+s_3\right)\right)}{\left(e^{\frac{1}{2} \left(s_1+s_2+s_3\right)}-1\right){}^4 s_1 \left(s_1+s_2\right) s_3 \left(s_2+s_3\right)},
\]
\[
 K_{15, 6}(s_1, s_2, s_3)=\frac{9 \pi  \left(s_1^2-s_3 s_1-\left(s_2-2 s_3\right) \left(s_2+s_3\right)\right)}{\left(e^{\frac{1}{2} \left(s_1+s_2+s_3\right)}+1\right){}^4 s_1 \left(s_1+s_2\right) s_3 \left(s_2+s_3\right)}, 
\]
{\tiny
\[
 K_{15, 7}(s_1, s_2, s_3)=\frac{4 \pi  e^{\frac{3 s_1}{2}} \left(4 s_2^2+\left(\left(6 e^{s_2}+2\right) s_3+e^{s_2}-7\right) s_2+2 s_3 \left(-s_3+e^{s_2} \left(3 s_3-7\right)+4\right)\right)}{\left(e^{s_1}-1\right) \left(e^{s_2}-1\right) \left(e^{\frac{1}{2} \left(s_2+s_3\right)}-1\right) s_1 s_2 s_3 \left(s_2+s_3\right)},
\]
\[
 K_{15, 8}(s_1, s_2, s_3)=\frac{4 \pi  e^{\frac{3 s_1}{2}} \left(4 s_2^2+\left(\left(6 e^{s_2}+2\right) s_3+e^{s_2}-7\right) s_2+2 s_3 \left(-s_3+e^{s_2} \left(3 s_3-7\right)+4\right)\right)}{\left(e^{s_1}-1\right) \left(e^{s_2}-1\right) \left(e^{\frac{1}{2} \left(s_2+s_3\right)}+1\right) s_1 s_2 s_3 \left(s_2+s_3\right)},
\]
\[
 K_{15, 9}(s_1, s_2, s_3)=\frac{\pi  e^{\frac{3 s_1}{2}} \left(\left(17-5 e^{s_2}\right) s_2^2+\left(-5 s_3+e^{s_2} \left(29 s_3+12\right)-12\right) s_2+2 s_3 \left(-11 s_3+e^{s_2} \left(17 s_3-12\right)+12\right)\right)}{\left(e^{s_1}-1\right) \left(e^{s_2}-1\right) \left(e^{\frac{1}{2} \left(s_2+s_3\right)}-1\right){}^2 s_1 s_2 s_3 \left(s_2+s_3\right)}, 
\]
\[
 K_{15, 10}(s_1, s_2, s_3)=\frac{\pi  e^{\frac{3 s_1}{2}} \left(\left(5 e^{s_2}-17\right) s_2^2+\left(5 s_3-e^{s_2} \left(29 s_3+12\right)+12\right) s_2-2 s_3 \left(-11 s_3+e^{s_2} \left(17 s_3-12\right)+12\right)\right)}{\left(e^{s_1}-1\right) \left(e^{s_2}-1\right) \left(e^{\frac{1}{2} \left(s_2+s_3\right)}+1\right){}^2 s_1 s_2 s_3 \left(s_2+s_3\right)}, 
\]
$
 K_{15, 11}(s_1, s_2, s_3)=$
\[
\frac{8 \pi  e^{\frac{3}{2} \left(s_1+s_2\right)} \left(s_1 \left(-e^{s_1+2 s_2} \left(s_3-2\right)-e^{s_2} \left(e^{s_1}+1\right) \left(s_3-1\right)+3 s_3-4\right)-e^{s_2} \left(e^{s_1}-1\right) s_2 \left(2 s_3-3\right)\right)}{\left(e^{s_2}-1\right) \left(e^{s_1+s_2}-1\right){}^2 \left(e^{\frac{s_3}{2}}-1\right) s_1 s_2 \left(s_1+s_2\right) s_3}, 
\]
$
 K_{15, 12}(s_1, s_2, s_3)=
$
\[
\frac{8 \pi  e^{\frac{3}{2} \left(s_1+s_2\right)} \left(s_1 \left(-e^{s_1+2 s_2} \left(s_3-2\right)-e^{s_2} \left(e^{s_1}+1\right) \left(s_3-1\right)+3 s_3-4\right)-e^{s_2} \left(e^{s_1}-1\right) s_2 \left(2 s_3-3\right)\right)}{\left(e^{s_2}-1\right) \left(e^{s_1+s_2}-1\right){}^2 \left(e^{\frac{s_3}{2}}+1\right) s_1 s_2 \left(s_1+s_2\right) s_3}, 
\]
\[
 K_{15, 13}(s_1, s_2, s_3)\left(e^{s_1}-1\right) \left(e^{s_1+s_2}-1\right) \left(e^{\frac{1}{2} \left(s_1+s_2+s_3\right)}-1\right){}^3 s_1 s_2 \left(s_1+s_2\right) s_3 \left(s_2+s_3\right) \left(s_1+s_2+s_3\right)=
\]
$
\pi  (-3 ((-9 e^{s_1}-7 e^{s_1+s_2}+5 e^{2 s_1+s_2}+11) s_2+4 e^{s_1} (e^{s_2}-1) s_3) s_1^3+((19 e^{s_1}+25 e^{s_1+s_2}-7 e^{2 s_1+s_2}-37) s_2^2+2 (9 (e^{s_1}-1) (e^{s_1+s_2}-1)+2 (7 e^{s_1}-8 e^{s_1+s_2}+2 e^{2 s_1+s_2}-1) s_3) s_2-24 e^{s_1} (e^{s_2}-1) s_3^2) s_1^2+((-43 e^{s_1}-13 e^{s_1+s_2}+31 e^{2 s_1+s_2}+25) s_2^3-4 (-e^{s_1}-e^{s_1+s_2}+4 e^{2 s_1+s_2}-2) s_3 s_2^2-s_3 (18 (e^{s_1}-1) (e^{s_1+s_2}-1)+(-59 e^{s_1}-5 e^{s_1+s_2}+47 e^{2 s_1+s_2}+17) s_3) s_2-12 e^{s_1} (e^{s_2}-1) s_3^3) s_1+s_2 (s_2+s_3) ((-35 e^{s_1}-17 e^{s_1+s_2}+23 e^{2 s_1+s_2}+29) s_2^2+((23 e^{s_1}+41 e^{s_1+s_2}-47 e^{2 s_1+s_2}-17) s_3-18 (e^{s_1}-1) (e^{s_1+s_2}-1)) s_2-2 s_3 ((-29 e^{s_1}-29 e^{s_1+s_2}+35 e^{2 s_1+s_2}+23) s_3-18 (e^{s_1}-1) (e^{s_1+s_2}-1)))),
$
\[
 K_{15, 14}(s_1, s_2, s_3)=\frac{K_{15,13}^{\text{num}}(s_1, s_2, s_3)}{\left(e^{s_1}-1\right) \left(e^{s_1+s_2}-1\right) \left(e^{\frac{1}{2} \left(s_1+s_2+s_3\right)}+1\right){}^3 s_1 s_2 \left(s_1+s_2\right) s_3 \left(s_2+s_3\right) \left(s_1+s_2+s_3\right)},
\]
\[
 K_{15, 15}(s_1, s_2, s_3)\left(e^{s_1}-1\right) \left(e^{s_1+s_2}-1\right){}^2 \left(e^{\frac{1}{2} \left(s_1+s_2+s_3\right)}-1\right) s_1 s_2 \left(s_1+s_2\right) s_3 \left(s_2+s_3\right) \left(s_1+s_2+s_3\right)=
\]
$
4 \pi  (2 ((3-e^{s_1}-5 e^{s_1+s_2}+3 e^{2 s_1+s_2}) s_2-e^{s_1} (-1+e^{s_2}) (-1+3 e^{s_1+s_2}) s_3) s_1^3+(4 (2+e^{s_1}-3 e^{s_1+s_2}) s_2^2+((2+4 e^{s_1+s_2}-18 e^{2 (s_1+s_2)}+12 e^{2 s_1+s_2}) s_3+8 e^{s_1}+22 e^{s_1+s_2}-e^{2 (s_1+s_2)}-14 e^{2 s_1+s_2}-15) s_2-2 e^{s_1} (-1+e^{s_2}) s_3 ((-2+6 e^{s_1+s_2}) s_3-7 e^{s_1+s_2}+4)) s_1^2-(2 (1-7 e^{s_1}-3 e^{s_1+s_2}+9 e^{2 s_1+s_2}) s_2^3+2 ((2-7 e^{s_1}-7 e^{s_1+s_2}+3 e^{2 (s_1+s_2)}+3 e^{2 s_1+s_2}+6 e^{3 s_1+2 s_2}) s_3+5 e^{s_1}-2 e^{s_1+s_2}-6 e^{2 s_1+s_2}+e^{3 s_1+2 s_2}+2) s_2^2+s_3 (2 (1+e^{s_1}-5 e^{s_1+s_2}+6 e^{2 (s_1+s_2)}-9 e^{2 s_1+s_2}+6 e^{3 s_1+2 s_2}) s_3+2 e^{s_1}+34 e^{s_1+s_2}-29 e^{2 (s_1+s_2)}+2 e^{2 s_1+s_2}+2 e^{3 s_1+2 s_2}-11) s_2+2 e^{s_1} (-1+e^{s_2}) s_3^2 ((-1+3 e^{s_1+s_2}) s_3-7 e^{s_1+s_2}+4)) s_1-s_2 (s_2+s_3) (4 (1-2 e^{s_1}-2 e^{s_1+s_2}+3 e^{2 s_1+s_2}) s_2^2+(2 (1-2 e^{s_1}-2 e^{s_1+s_2}-3 e^{2 (s_1+s_2)}+6 e^{3 s_1+2 s_2}) s_3+18 e^{s_1}+18 e^{s_1+s_2}-e^{2 (s_1+s_2)}-26 e^{2 s_1+s_2}+2 e^{3 s_1+2 s_2}-11) s_2+2 s_3 ((-1+2 e^{s_1}+2 e^{s_1+s_2}-3 e^{2 (s_1+s_2)}-6 e^{2 s_1+s_2}+6 e^{3 s_1+2 s_2}) s_3-9 e^{s_1}-12 e^{s_1+s_2}+10 e^{2 (s_1+s_2)}+23 e^{2 s_1+s_2}-17 e^{3 s_1+2 s_2}+5))), 
$
\[
 K_{15, 16}(s_1, s_2, s_3)=\frac{K_{15,15}^{\text{num}}(s_1, s_2, s_3)}{\left(e^{s_1}-1\right) \left(e^{s_1+s_2}-1\right){}^2 \left(e^{\frac{1}{2} \left(s_1+s_2+s_3\right)}+1\right) s_1 s_2 \left(s_1+s_2\right) s_3 \left(s_2+s_3\right) \left(s_1+s_2+s_3\right)}, 
\]
\[
 K_{15, 17}(s_1, s_2, s_3)2 \left(e^{s_1}-1\right) \left(e^{s_1+s_2}-1\right){}^2 \left(e^{\frac{1}{2} \left(s_1+s_2+s_3\right)}-1\right){}^2 s_1 s_2 \left(s_1+s_2\right) s_3 \left(s_2+s_3\right) \left(s_1+s_2+s_3\right)=
\]
$
\pi  (((87-53 e^{s_1}-130 e^{s_1+s_2}+19 e^{2 (s_1+s_2)}+86 e^{2 s_1+s_2}-9 e^{3 s_1+2 s_2}) s_2-4 e^{s_1} (-1+e^{s_2}) (-11+17 e^{s_1+s_2}) s_3) s_1^3+((107-5 e^{s_1}-154 e^{s_1+s_2}+23 e^{2 (s_1+s_2)}+22 e^{2 s_1+s_2}+7 e^{3 s_1+2 s_2}) s_2^2+2 (2 (5-21 e^{s_1}+27 e^{s_1+s_2}-50 e^{2 (s_1+s_2)}+35 e^{2 s_1+s_2}+4 e^{3 s_1+2 s_2}) s_3+45 e^{s_1}+90 e^{s_1+s_2}-33 e^{2 (s_1+s_2)}-66 e^{2 s_1+s_2}+21 e^{3 s_1+2 s_2}-57) s_2-8 e^{s_1} (-1+e^{s_2}) s_3 ((-11+17 e^{s_1+s_2}) s_3-6 e^{s_1+s_2}+6)) s_1^2+((-47+149 e^{s_1}+82 e^{s_1+s_2}-11 e^{2 (s_1+s_2)}-214 e^{2 s_1+s_2}+41 e^{3 s_1+2 s_2}) s_2^3-4 (4 (1+2 e^{s_1}) (-1+e^{s_1+s_2}){}^2+(10-11 e^{s_1}-25 e^{s_1+s_2}+9 e^{2 (s_1+s_2)}-15 e^{2 s_1+s_2}+32 e^{3 s_1+2 s_2}) s_3) s_2^2-s_3 (2 (-49+13 e^{s_1}+122 e^{s_1+s_2}-73 e^{2 (s_1+s_2)}-50 e^{2 s_1+s_2}+37 e^{3 s_1+2 s_2})+(-7+149 e^{s_1}-62 e^{s_1+s_2}+93 e^{2 (s_1+s_2)}-342 e^{2 s_1+s_2}+169 e^{3 s_1+2 s_2}) s_3) s_2-4 e^{s_1} (-1+e^{s_2}) s_3^2 ((-11+17 e^{s_1+s_2}) s_3-12 (-1+e^{s_1+s_2}))) s_1+s_2 (s_2+s_3) ((-67+101 e^{s_1}+106 e^{s_1+s_2}-15 e^{2 (s_1+s_2)}-150 e^{2 s_1+s_2}+25 e^{3 s_1+2 s_2}) s_2^2+(2 (49-61 e^{s_1}-74 e^{s_1+s_2}+25 e^{2 (s_1+s_2)}+98 e^{2 s_1+s_2}-37 e^{3 s_1+2 s_2})+(7-17 e^{s_1}-70 e^{s_1+s_2}+111 e^{2 (s_1+s_2)}+138 e^{2 s_1+s_2}-169 e^{3 s_1+2 s_2}) s_3) s_2-2 s_3 ((-37+59 e^{s_1}+88 e^{s_1+s_2}-63 e^{2 (s_1+s_2)}-144 e^{2 s_1+s_2}+97 e^{3 s_1+2 s_2}) s_3-98 e^{s_1}-172 e^{s_1+s_2}+98 e^{2 (s_1+s_2)}+220 e^{2 s_1+s_2}-122 e^{3 s_1+2 s_2}+74))), 
$
\[
 K_{15, 18}(s_1, s_2, s_3)=\frac{-K_{15,17}^{\text{num}}(s_1, s_2, s_3)}{2 \left(e^{s_1}-1\right) \left(e^{s_1+s_2}-1\right){}^2 \left(e^{\frac{1}{2} \left(s_1+s_2+s_3\right)}+1\right){}^2 s_1 s_2 \left(s_1+s_2\right) s_3 \left(s_2+s_3\right) \left(s_1+s_2+s_3\right)}. 
\]
}

\smallskip

Using the above expressions, we have  
\[
K_{15}(s_1, s_2, s_3) 
= 
\frac{K_{15}^{\text{num}}(s_1, s_2, s_3)}{K_{15}^{\text{den}}(s_1, s_2, s_3)}, 
\]
where 
\[
K_{15}^{\text{den}}(s_1, s_2, s_3) =K_{8}^{\text{den}}(s_1, s_2, s_3),  
\]
which is given by \eqref{K8den}, and $K_{15}^{\text{num}}$ 
is a polynomial in $s_1, s_2, s_3, e^{s_1/2}, e^{s_2/2}, e^{s_3/2}$. 
The points 
$(i, j, m)$ and $(n, p, q)$ such that 
$s_1^i s_2^j s_3^m  e^{n s_1/2} e^{p s_2/2} e^{q s_3/2}$ appears in the 
expression of $K_{15}^{\text{num}}(s_1, s_2, s_3)$ are plotted in Figure \ref{K8spowers} 
and Figure \ref{K8Etospowers}.

\smallskip

\subsection{The four variable functions $K_{17}, \dots, K_{20}$}
\label{ExplicitFourVarsSec}

The functions of four variables $K_{17}, K_{18},$ $K_{19}, K_{20}$ 
appearing in \eqref{a_4expression} have 
very lengthy expressions. These functions are of the form 
\[
K_j(s_1,s_2, s_3, s_4)=\frac{K_j^{\text{num}}(s_1,s_2, s_3, s_4)}{K_j^{\text{den}}(s_1,s_2, s_3, s_4)}, \qquad j=17, 18, 19, 20, 
\]
where each $K_j^{\text{num}}(s_1,s_2, s_3, s_4)$ is a polynomial 
in $s_1, s_2, s_3, s_4$, $e^{s_1/2}, e^{s_1/2},$ $e^{s_3/2}, e^{s_4/2}$, 
and 
\begin{equation} \label{K17den}
K_{17}^{\text{den}}(s_1,s_2, s_3, s_4)=K_{18}^{\text{den}}(s_1,s_2, s_3, s_4)=K_{19}^{\text{den}}(s_1,s_2, s_3, s_4)=
\end{equation}
\[
K_{20}^{\text{den}}(s_1,s_2, s_3, s_4)=\left(e^{s_1}-1\right) \left(e^{s_2}-1\right) \left(e^{s_1+s_2}-1\right){}^2 \left(e^{s_3}-1\right) \left(e^{s_2+s_3}-1\right){}^2
\times
\]
\[
 \left(e^{s_1+s_2+s_3}-1\right){}^3 \left(e^{s_4}-1\right) \left(e^{s_3+s_4}-1\right){}^2 \left(e^{s_2+s_3+s_4}-1\right){}^3 \left(e^{s_1+s_2+s_3+s_4}-1\right){}^4 \times
\]
\[ s_1 s_2 \left(s_1+s_2\right) s_3 \left(s_2+s_3\right) \left(s_1+s_2+s_3\right) s_4 \left(s_3+s_4\right) \left(s_2+s_3+s_4\right) \left(s_1+s_2+s_3+s_4\right). 
\]
In fact, it turns out that 
\[
K_{18}(s_1, s_2, s_3, s_4)= K_{19}(s_1, s_2, s_3, s_4). 
\]

\smallskip

It is made clear in Section \ref{a_4Sec}, using the functional 
relations stated in Theorem \ref{FuncRelationsThm}, that the functions $K_{17}, K_{18},$ $K_{19}, K_{20}$ 
can be constructed from the one, two and three variable 
functions $K_1,$ $\dots,$ $K_{16}$ presented in Section \ref{ExplicitFormulasSec} 
and earlier in this Appendix, the functions 
$G_1, \dots, G_4$ given by \eqref{ExplicitBabyFunctions}, and the following three variable 
functions:  
\[
k_j(s_1, s_2, s_3)= K_j(s_1, s_2, s_3, -s_1-s_2-s_3), \qquad j=17, 18, 19, 20, 
\]
which were in fact introduced already by \eqref{littlekfunctions}. 
 Therefore, in the rest of this  appendix we present the latter functions explicitly. 

\subsubsection{The function $k_{17}$} We have 
\[
k_{17}(s_1, s_2, s_3)= 
K_{17}(s_1, s_2, s_3, -s_1-s_2-s_3)
=\sum_{i=1}^{30} k_{17, i}(s_1, s_2, s_3), 
\]
where 

{\tiny
\[
k_{17, 1}(s_1, s_2, s_3)\left(e^{s_1}-1\right){}^4 \left(e^{s_2}-1\right) \left(e^{s_1+s_2}-1\right){}^4 s_1^2 \left(s_1+s_2\right){}^2 \left(s_1+s_2+s_3\right){}^2=
\]
$
16 \pi  e^{2 s_1} (2 e^{3 s_2}-8 e^{3 (s_1+s_2)}-39 e^{2 s_1+s_2}-40 e^{s_1+2 s_2}+10 e^{4 s_1+2 s_2}+5 e^{2 s_1+4 s_2}+16 e^{5 s_1+4 s_2}+12 e^{4 s_1+5 s_2}-12) s_2 (s_2+s_3),
$
\[
k_{17, 2}(s_1, s_2, s_3)= \frac{64 \pi  e^{2 s_1} \left(2 e^{3 s_2}-2 e^{4 s_1+2 s_2}+7 e^{4 s_1+5 s_2}-7\right) s_2 \left(s_2+s_3\right)}{5 \left(e^{s_1}-1\right){}^4 \left(e^{s_2}-1\right) \left(e^{s_1+s_2}-1\right){}^4 s_1 \left(s_1+s_2\right){}^2 \left(s_1+s_2+s_3\right){}^2}, 
\]
\[
k_{17, 3}(s_1, s_2, s_3)\frac{3}{16} \left(e^{s_1}-1\right){}^4 \left(e^{s_2}-1\right) \left(e^{s_1+s_2}-1\right){}^4 s_1^2 \left(s_1+s_2\right){}^2 \left(s_1+s_2+s_3\right){}^2= 
\]
$
\pi  (-(-4 e^{s_1}-28 e^{3 s_1}-5 e^{4 s_1}-e^{s_2}-10 e^{2 (s_1+s_2)}+20 e^{3 (s_1+s_2)}+4 e^{5 (s_1+s_2)}-20 e^{2 s_1+s_2}-116 e^{3 s_1+s_2}+20 e^{5 s_1+s_2}+4 e^{s_1+2 s_2}+104 e^{4 s_1+2 s_2}-188 e^{5 s_1+2 s_2}-200 e^{4 s_1+3 s_2}+142 e^{6 s_1+3 s_2}+20 e^{7 s_1+3 s_2}+4 e^{3 s_1+4 s_2}+140 e^{5 s_1+4 s_2}-76 e^{6 s_1+4 s_2}-5 e^{8 s_1+4 s_2}-e^{4 s_1+5 s_2}+28 e^{7 s_1+5 s_2}+5 e^{8 s_1+5 s_2}+1)) s_2 (s_2+s_3), 
$
\[
k_{17, 4}(s_1, s_2, s_3)\frac{3}{16} \left(e^{s_1}-1\right){}^4 \left(e^{s_2}-1\right) \left(e^{s_1+s_2}-1\right){}^4 s_1 \left(s_1+s_2\right){}^2 s_3 \left(s_2+s_3\right) \left(s_1+s_2+s_3\right){}^2= 
\]
$
\pi  (-11 e^{s_1}-29 e^{3 s_1}+11 e^{4 s_1}-2 e^{s_2}+2 e^{s_1+s_2}+82 e^{3 (s_1+s_2)}-53 e^{4 (s_1+s_2)}-88 e^{3 s_1+s_2}+106 e^{4 s_1+s_2}-136 e^{4 s_1+3 s_2}+26 e^{6 s_1+3 s_2}-26 e^{7 s_1+3 s_2}+11 e^{3 s_1+4 s_2}-83 e^{6 s_1+4 s_2}+26 e^{7 s_1+4 s_2}+2) s_2^3, 
$
\[
k_{17, 5}(s_1, s_2, s_3)\frac{1}{16} \left(e^{s_1}-1\right){}^4 \left(e^{s_2}-1\right) \left(e^{s_1+s_2}-1\right){}^4 s_1 \left(s_1+s_2\right){}^2 s_3 \left(s_2+s_3\right) \left(s_1+s_2+s_3\right){}^2= 
\]
$
\pi  e^{s_1} s_2^3 (2 e^{2 s_1+4 s_2} s_2 (e^{s_1}-1){}^5+9 e^{s_1}+3 e^{2 s_2}+8 e^{s_1+s_2}+8 e^{2 (s_1+s_2)}+33 e^{4 (s_1+s_2)}-14 e^{4 s_1+s_2}-11 e^{s_1+2 s_2}+24 e^{3 s_1+2 s_2}-43 e^{4 s_1+2 s_2}+19 e^{5 s_1+2 s_2}-6 e^{s_1+3 s_2}+24 e^{4 s_1+3 s_2})
$, 
\[
k_{17, 6}(s_1, s_2, s_3)\frac{3}{32} \left(e^{s_1}-1\right){}^4 \left(e^{s_2}-1\right) \left(e^{s_1+s_2}-1\right){}^4 \left(s_1+s_2\right){}^2 s_3 \left(s_2+s_3\right) \left(s_1+s_2+s_3\right){}^2= 
\]
$
-\frac{1}{2} \pi  ((19 e^{s_1}+37 e^{3 s_1}-13 e^{4 s_1}+4 e^{s_2}-10 e^{s_1+s_2}-50 e^{3 (s_1+s_2)}+73 e^{4 (s_1+s_2)}+80 e^{3 s_1+s_2}-104 e^{4 s_1+s_2}+68 e^{4 s_1+3 s_2}-64 e^{6 s_1+3 s_2}+34 e^{7 s_1+3 s_2}-19 e^{3 s_1+4 s_2}+43 e^{6 s_1+4 s_2}+14 e^{7 s_1+4 s_2}-4) s_1^2+(49 e^{s_1}+103 e^{3 s_1}-37 e^{4 s_1}+10 e^{s_2}-22 e^{s_1+s_2}-182 e^{3 (s_1+s_2)}+199 e^{4 (s_1+s_2)}+248 e^{3 s_1+s_2}-314 e^{4 s_1+s_2}+272 e^{4 s_1+3 s_2}-154 e^{6 s_1+3 s_2}+94 e^{7 s_1+3 s_2}-49 e^{3 s_1+4 s_2}+169 e^{6 s_1+4 s_2}+2 e^{7 s_1+4 s_2}-10) s_2 s_1-(-41 e^{s_1}-95 e^{3 s_1}+35 e^{4 s_1}-8 e^{s_2}+14 e^{s_1+s_2}+214 e^{3 (s_1+s_2)}-179 e^{4 (s_1+s_2)}-256 e^{3 s_1+s_2}+316 e^{4 s_1+s_2}-340 e^{4 s_1+3 s_2}+116 e^{6 s_1+3 s_2}-86 e^{7 s_1+3 s_2}+41 e^{3 s_1+4 s_2}-209 e^{6 s_1+4 s_2}+38 e^{7 s_1+4 s_2}+8) s_2^2),
$
\[
k_{17, 7}(s_1, s_2, s_3)\frac{1}{32} \left(e^{s_1}-1\right){}^4 \left(e^{s_2}-1\right) \left(e^{s_1+s_2}-1\right){}^4 \left(s_1+s_2\right){}^2 s_3 \left(s_2+s_3\right) \left(s_1+s_2+s_3\right){}^2= 
\]
$
-\frac{1}{2} e^{s_1} \pi  (2 e^{s_1} (-1+e^{s_1}) (-1+e^{s_2}) (-e^{s_1}-2 e^{2 s_2}-e^{s_1+s_2}-2 e^{2 (s_1+s_2)}+6 e^{3 (s_1+s_2)}+4 e^{2 s_1+s_2}+5 e^{s_1+2 s_2}-4 e^{3 s_1+2 s_2}+3 e^{s_1+3 s_2}-8 e^{2 s_1+3 s_2}) s_1^3+(-2 e^{s_1} (-1+e^{s_1}) (-3 e^{s_1}-6 e^{2 s_2}+6 e^{3 s_2}-18 e^{2 (s_1+s_2)}+30 e^{3 (s_1+s_2)}-4 e^{4 (s_1+s_2)}+12 e^{2 s_1+s_2}+18 e^{s_1+2 s_2}-12 e^{3 s_1+2 s_2}-6 e^{s_1+3 s_2}-18 e^{2 s_1+3 s_2}-8 e^{s_1+4 s_2}+20 e^{2 s_1+4 s_2}-12 e^{3 s_1+4 s_2}+e^{5 s_1+4 s_2}) s_2-13 e^{s_1}-3 e^{2 s_2}-4 e^{s_1+s_2}-16 e^{2 (s_1+s_2)}-33 e^{4 (s_1+s_2)}+24 e^{5 (s_1+s_2)}+14 e^{4 s_1+s_2}+13 e^{s_1+2 s_2}-12 e^{3 s_1+2 s_2}+35 e^{4 s_1+2 s_2}-17 e^{5 s_1+2 s_2}+4 e^{s_1+3 s_2}-4 e^{7 s_1+4 s_2}+4 e^{3 s_1+5 s_2}-16 e^{4 s_1+5 s_2}-16 e^{6 s_1+5 s_2}+4 e^{7 s_1+5 s_2}) s_1^2-s_2 (6 e^{s_1} (-1+e^{s_1}) (-e^{s_1}-2 e^{2 s_2}+2 e^{3 s_2}-6 e^{2 (s_1+s_2)}+10 e^{3 (s_1+s_2)}-4 e^{4 (s_1+s_2)}+4 e^{2 s_1+s_2}+6 e^{s_1+2 s_2}-4 e^{3 s_1+2 s_2}-2 e^{s_1+3 s_2}-6 e^{2 s_1+3 s_2}-2 e^{s_1+4 s_2}+4 e^{2 s_1+4 s_2}+e^{5 s_1+4 s_2}) s_2+35 e^{s_1}+9 e^{2 s_2}+16 e^{s_1+s_2}+40 e^{2 (s_1+s_2)}+99 e^{4 (s_1+s_2)}-48 e^{5 (s_1+s_2)}-42 e^{4 s_1+s_2}-37 e^{s_1+2 s_2}+48 e^{3 s_1+2 s_2}-113 e^{4 s_1+2 s_2}+53 e^{5 s_1+2 s_2}-14 e^{s_1+3 s_2}+24 e^{4 s_1+3 s_2}+8 e^{7 s_1+4 s_2}-8 e^{3 s_1+5 s_2}+32 e^{4 s_1+5 s_2}+32 e^{6 s_1+5 s_2}-8 e^{7 s_1+5 s_2}) s_1-s_2^2 (2 e^{s_1} (-1+e^{s_1}) (-e^{s_1}-2 e^{2 s_2}+2 e^{3 s_2}-6 e^{2 (s_1+s_2)}+10 e^{3 (s_1+s_2)}-12 e^{4 (s_1+s_2)}+4 e^{2 s_1+s_2}+6 e^{s_1+2 s_2}-4 e^{3 s_1+2 s_2}-2 e^{s_1+3 s_2}-6 e^{2 s_1+3 s_2}-4 e^{2 s_1+4 s_2}+12 e^{3 s_1+4 s_2}+3 e^{5 s_1+4 s_2}) s_2+31 e^{s_1}+9 e^{2 s_2}+20 e^{s_1+s_2}+32 e^{2 (s_1+s_2)}+99 e^{4 (s_1+s_2)}-24 e^{5 (s_1+s_2)}-42 e^{4 s_1+s_2}-35 e^{s_1+2 s_2}+60 e^{3 s_1+2 s_2}-121 e^{4 s_1+2 s_2}+55 e^{5 s_1+2 s_2}-16 e^{s_1+3 s_2}+48 e^{4 s_1+3 s_2}+4 e^{7 s_1+4 s_2}-4 e^{3 s_1+5 s_2}+16 e^{4 s_1+5 s_2}+16 e^{6 s_1+5 s_2}-4 e^{7 s_1+5 s_2})), 
$
\[
k_{17, 8}(s_1, s_2, s_3)\frac{3}{32} \left(e^{s_1}-1\right){}^4 \left(e^{s_2}-1\right) \left(e^{s_1+s_2}-1\right){}^4 \left(s_1+s_2\right){}^2 \left(s_2+s_3\right) \left(s_1+s_2+s_3\right){}^2= 
\]
$
\frac{1}{2} \pi  (4 e^{2 s_1+s_2} (8+74 e^{s_1}+13 e^{s_2}+98 e^{2 (s_1+s_2)}-131 e^{3 (s_1+s_2)}-98 e^{2 s_1+s_2}-88 e^{3 s_1+s_2}+10 e^{s_1+2 s_2}+89 e^{4 s_1+2 s_2}+25 e^{4 s_1+3 s_2}) s_2^2-e^{s_1} (55-55 e^{3 s_1}-13 e^{2 s_2}-176 e^{s_1+s_2}-256 e^{3 (s_1+s_2)}-79 e^{4 (s_1+s_2)}+182 e^{4 s_1+s_2}+149 e^{s_1+2 s_2}-128 e^{3 s_1+2 s_2}-155 e^{4 s_1+2 s_2}+94 e^{6 s_1+3 s_2}-8 e^{7 s_1+4 s_2}+40 e^{3 s_1+5 s_2}-164 e^{4 s_1+5 s_2}+20 e^{6 s_1+5 s_2}+8 e^{7 s_1+5 s_2}) s_2-2 s_3 (2 e^{2 s_1+s_2} (-4-40 e^{s_1}-5 e^{s_2}-46 e^{2 (s_1+s_2)}+40 e^{3 (s_1+s_2)}+70 e^{2 s_1+s_2}+50 e^{3 s_1+s_2}-2 e^{s_1+2 s_2}-37 e^{4 s_1+2 s_2}+34 e^{4 s_1+3 s_2}) s_3+14 e^{s_1}-26 e^{3 s_1}-14 e^{4 s_1}+e^{s_2}-14 e^{s_1+s_2}+31 e^{2 (s_1+s_2)}-104 e^{3 (s_1+s_2)}+206 e^{4 (s_1+s_2)}-118 e^{5 (s_1+s_2)}-79 e^{2 s_1+s_2}-94 e^{3 s_1+s_2}+146 e^{4 s_1+s_2}+40 e^{5 s_1+s_2}-244 e^{5 s_1+2 s_2}-344 e^{4 s_1+3 s_2}+131 e^{6 s_1+3 s_2}+20 e^{7 s_1+3 s_2}-52 e^{3 s_1+4 s_2}+46 e^{5 s_1+4 s_2}-179 e^{6 s_1+4 s_2}-14 e^{7 s_1+4 s_2}-7 e^{8 s_1+4 s_2}+7 e^{8 s_1+5 s_2}-1)+2 s_1 (-2 e^{2 s_1+s_2} (-16-94 e^{s_1}-23 e^{s_2}-154 e^{2 (s_1+s_2)}+25 e^{3 (s_1+s_2)}+58 e^{2 s_1+s_2}+83 e^{3 s_1+s_2}-38 e^{s_1+2 s_2}-127 e^{4 s_1+2 s_2}+67 e^{4 s_1+3 s_2}) s_2-17 e^{s_1}-19 e^{3 s_1}+17 e^{4 s_1}-e^{s_2}+44 e^{3 (s_1+s_2)}+4 e^{3 s_1+s_2}-7 e^{s_1+2 s_2}+4 e^{4 s_1+2 s_2}+140 e^{4 s_1+3 s_2}+e^{6 s_1+3 s_2}-32 e^{7 s_1+3 s_2}+31 e^{3 s_1+4 s_2}+74 e^{6 s_1+4 s_2}-26 e^{4 s_1+5 s_2}+56 e^{7 s_1+5 s_2}-2 e^{2 s_1+s_2} (-16-100 e^{s_1}-17 e^{s_2}-142 e^{2 (s_1+s_2)}+25 e^{3 (s_1+s_2)}+118 e^{2 s_1+s_2}+92 e^{3 s_1+s_2}-26 e^{s_1+2 s_2}-97 e^{4 s_1+2 s_2}+121 e^{4 s_1+3 s_2}) s_3+1)), 
$
\[
k_{17, 9}(s_1, s_2, s_3)\frac{3}{32} \left(e^{s_1}-1\right){}^4 \left(e^{s_1+s_2}-1\right){}^4 s_2 \left(s_1+s_2\right){}^2 \left(s_2+s_3\right) \left(s_1+s_2+s_3\right){}^2= 
\]
$
\frac{1}{2} \pi  e^{s_1} ((-7 e^{3 s_1}-7 e^{s_2}-37 e^{3 (s_1+s_2)}-5 e^{3 s_1+s_2}+27 e^{3 s_1+2 s_2}+10 e^{6 s_1+3 s_2}-8 e^{3 s_1+4 s_2}+20 e^{6 s_1+4 s_2}+7) s_1^2+2 (-e^{3 s_1}-e^{s_2}-85 e^{3 (s_1+s_2)}+43 e^{3 s_1+s_2}+27 e^{3 s_1+2 s_2}+4 e^{6 s_1+3 s_2}-14 e^{3 s_1+4 s_2}+26 e^{6 s_1+4 s_2}+1) s_3^2), 
$
\[
k_{17, 10}(s_1, s_2, s_3)\frac{1}{32} \left(e^{s_1}-1\right){}^4 \left(e^{s_2}-1\right) \left(e^{s_1+s_2}-1\right){}^4 \left(s_1+s_2\right){}^2 \left(s_2+s_3\right) \left(s_1+s_2+s_3\right){}^2= 
\]
$
\frac{1}{2} \pi  (2 e^{2 s_1} (-8+2 e^{s_1}-4 e^{2 s_1}+8 e^{s_2}+10 e^{2 s_2}-10 e^{3 s_2}+36 e^{s_1+s_2}-4 e^{2 (s_1+s_2)}-24 e^{3 (s_1+s_2)}-46 e^{4 (s_1+s_2)}-4 e^{5 (s_1+s_2)}-10 e^{2 s_1+s_2}+16 e^{3 s_1+s_2}-72 e^{s_1+2 s_2}-20 e^{3 s_1+2 s_2}-14 e^{4 s_1+2 s_2}+20 e^{s_1+3 s_2}+68 e^{2 s_1+3 s_2}+50 e^{4 s_1+3 s_2}-4 e^{5 s_1+3 s_2}+13 e^{s_1+4 s_2}-47 e^{2 s_1+4 s_2}+26 e^{3 s_1+4 s_2}+3 e^{5 s_1+4 s_2}+e^{6 s_1+4 s_2}+2 e^{2 s_1+5 s_2}-8 e^{3 s_1+5 s_2}+20 e^{4 s_1+5 s_2}) s_1^2-2 e^{s_1} (e^{2 s_1+s_2} (-6 e^{s_1}+104 e^{s_2}-8 e^{2 (s_1+s_2)}+39 e^{s_1+3 s_2}+18 e^{4 s_1+3 s_2}) s_2-6 e^{s_1}-8 e^{s_2}-5 e^{s_1+s_2}+20 e^{2 (s_1+s_2)}-15 e^{4 (s_1+s_2)}+41 e^{5 (s_1+s_2)}-2 e^{3 s_1+s_2}+16 e^{4 s_1+s_2}+12 e^{s_1+2 s_2}-17 e^{4 s_1+2 s_2}-16 e^{5 s_1+2 s_2}-e^{s_1+3 s_2}+52 e^{4 s_1+3 s_2}+47 e^{3 s_1+4 s_2}+8 e^{6 s_1+4 s_2}-5 e^{7 s_1+4 s_2}-36 e^{4 s_1+5 s_2}+5 e^{7 s_1+5 s_2}-e^{2 s_1+s_2} (24 e^{s_1}-88 e^{s_2}+64 e^{2 (s_1+s_2)}-29 e^{s_1+3 s_2}+5 e^{4 s_1+3 s_2}) s_3) s_1-2 e^{3 s_1+s_2} (-26 e^{s_1}+56 e^{s_2}-56 e^{2 (s_1+s_2)}-19 e^{s_1+3 s_2}+45 e^{4 s_1+3 s_2}) s_2^2-2 e^{2 s_1} s_3 (2 e^{s_1+s_2} (-10 e^{s_1}+12 e^{s_2}-24 e^{2 (s_1+s_2)}-2 e^{s_1+3 s_2}+3 e^{4 s_1+3 s_2}) s_3+7 e^{3 s_2}-8 e^{2 (s_1+s_2)}+92 e^{3 (s_1+s_2)}-2 e^{5 (s_1+s_2)}+92 e^{s_1+2 s_2}-13 e^{4 s_1+2 s_2}+10 e^{2 s_1+5 s_2}+29 e^{4 s_1+5 s_2}+9)+s_2 (-2 e^{2 s_1+s_2} (-8-76 e^{s_1}-50 e^{2 s_1}-10 e^{s_2}+72 e^{s_1+s_2}-92 e^{2 (s_1+s_2)}+98 e^{3 (s_1+s_2)}+124 e^{2 s_1+s_2}+90 e^{3 s_1+s_2}-4 e^{s_1+2 s_2}-120 e^{3 s_1+2 s_2}-74 e^{4 s_1+2 s_2}-19 e^{2 s_1+3 s_2}+30 e^{4 s_1+3 s_2}+36 e^{5 s_1+3 s_2}) s_3+9 e^{2 s_1}-11 e^{3 s_1}-2 e^{s_2}+14 e^{s_1+s_2}+114 e^{3 (s_1+s_2)}-135 e^{4 (s_1+s_2)}-12 e^{3 s_1+s_2}+2 e^{4 s_1+s_2}-120 e^{3 s_1+2 s_2}+71 e^{6 s_1+2 s_2}-18 e^{2 s_1+3 s_2}-72 e^{5 s_1+3 s_2}-78 e^{6 s_1+3 s_2}+29 e^{3 s_1+4 s_2}+39 e^{6 s_1+4 s_2}+38 e^{7 s_1+4 s_2}-32 e^{6 s_1+5 s_2}+2)), 
$
\[
k_{17, 11}(s_1, s_2, s_3)\frac{5}{32} \left(e^{s_1}-1\right){}^4 \left(e^{s_2}-1\right) \left(e^{s_1+s_2}-1\right){}^4 \left(s_1+s_2\right){}^2 \left(s_2+s_3\right) \left(s_1+s_2+s_3\right){}^2= 
\]
$
\pi  (2 e^{2 s_1} (-11 e^{3 s_2}-19 e^{4 s_1+2 s_2}+54 e^{4 s_1+5 s_2}-24) s_2^2+(8 e^{s_1}-72 e^{2 s_1}-67 e^{3 s_1}-7 e^{4 s_1}+2 e^{s_2}-88 e^{5 (s_1+s_2)}+28 e^{5 s_1+s_2}-8 e^{s_1+2 s_2}-12 e^{6 s_1+2 s_2}-18 e^{2 s_1+3 s_2}-32 e^{7 s_1+3 s_2}-13 e^{3 s_1+4 s_2}+38 e^{8 s_1+4 s_2}+22 e^{4 s_1+5 s_2}+192 e^{6 s_1+5 s_2}+12 e^{7 s_1+5 s_2}+2 e^{8 s_1+5 s_2}-2) s_3 s_2+6 e^{2 s_1} (-e^{3 s_2}+e^{4 s_1+2 s_2}+14 e^{4 s_1+5 s_2}-4) s_3^2+2 e^{2 s_1} s_1 ((-33 e^{3 s_2}-42 e^{4 s_1+2 s_2}+92 e^{4 s_1+5 s_2}-32) s_2+(-23 e^{3 s_2}-7 e^{4 s_1+2 s_2}+122 e^{4 s_1+5 s_2}-32) s_3)), 
$
\[
k_{17, 12}(s_1, s_2, s_3)\frac{15}{32} \left(e^{s_1}-1\right){}^4 \left(e^{s_2}-1\right) \left(e^{s_1+s_2}-1\right){}^4 \left(s_1+s_2\right){}^2 \left(s_2+s_3\right) \left(s_1+s_2+s_3\right){}^2= 
\]
$
\pi  ((16 e^{s_1}-104 e^{3 s_1}-34 e^{4 s_1}+4 e^{s_2}-136 e^{5 (s_1+s_2)}+136 e^{5 s_1+s_2}-16 e^{s_1+2 s_2}-44 e^{7 s_1+3 s_2}-31 e^{3 s_1+4 s_2}+131 e^{8 s_1+4 s_2}+34 e^{4 s_1+5 s_2}+44 e^{7 s_1+5 s_2}+4 e^{8 s_1+5 s_2}-4) s_2^2+2 (4 e^{s_1}-41 e^{3 s_1}-e^{4 s_1}+e^{s_2}-64 e^{5 (s_1+s_2)}+4 e^{5 s_1+s_2}-4 e^{s_1+2 s_2}-26 e^{7 s_1+3 s_2}-4 e^{3 s_1+4 s_2}+14 e^{8 s_1+4 s_2}+16 e^{4 s_1+5 s_2}-4 e^{7 s_1+5 s_2}+e^{8 s_1+5 s_2}-1) s_3^2+s_1 ((8 e^{s_1}-37 e^{3 s_1}-77 e^{4 s_1}+2 e^{s_2}-248 e^{5 (s_1+s_2)}+308 e^{5 s_1+s_2}-8 e^{s_1+2 s_2}-112 e^{7 s_1+3 s_2}+187 e^{3 s_1+4 s_2}+88 e^{8 s_1+4 s_2}+62 e^{4 s_1+5 s_2}-68 e^{7 s_1+5 s_2}+2 e^{8 s_1+5 s_2}-2) s_2+(8 e^{s_1}-112 e^{3 s_1}-32 e^{4 s_1}+2 e^{s_2}-368 e^{5 (s_1+s_2)}+128 e^{5 s_1+s_2}-8 e^{s_1+2 s_2}-172 e^{7 s_1+3 s_2}+127 e^{3 s_1+4 s_2}+43 e^{8 s_1+4 s_2}+92 e^{4 s_1+5 s_2}-128 e^{7 s_1+5 s_2}+2 e^{8 s_1+5 s_2}-2) s_3)), 
$
\[
k_{17, 13}(s_1, s_2, s_3)\frac{1}{32} \left(e^{s_1}-1\right){}^4 \left(e^{s_1+s_2}-1\right){}^4 s_2 \left(s_1+s_2\right){}^2 \left(s_2+s_3\right) \left(s_1+s_2+s_3\right){}^2= 
\]
$
-\frac{1}{2} e^{s_1} \pi  (-2 e^{s_1} (-1+e^{s_2}) (-2+e^{s_1}-e^{2 s_1}-2 e^{s_2}+10 e^{s_1+s_2}-11 e^{2 (s_1+s_2)}-6 e^{2 s_1+s_2}+4 e^{3 s_1+s_2}+3 e^{s_1+2 s_2}+6 e^{3 s_1+2 s_2}-4 e^{4 s_1+2 s_2}+2 e^{4 s_1+3 s_2}) s_1^3+e^{s_1} (2 (-4-e^{s_1}-e^{2 s_1}+4 e^{2 s_2}+19 e^{s_1+s_2}-15 e^{2 (s_1+s_2)}-2 e^{3 (s_1+s_2)}-16 e^{4 (s_1+s_2)}+e^{2 s_1+s_2}+4 e^{3 s_1+s_2}-13 e^{s_1+2 s_2}-10 e^{3 s_1+2 s_2}-2 e^{4 s_1+2 s_2}-5 e^{s_1+3 s_2}+17 e^{2 s_1+3 s_2}+18 e^{4 s_1+3 s_2}-4 e^{5 s_1+3 s_2}-2 e^{2 s_1+4 s_2}+8 e^{3 s_1+4 s_2}+4 e^{5 s_1+4 s_2}) s_3-3 e^{s_1}+e^{s_2}-4 e^{2 s_2}+e^{s_1+s_2}-9 e^{3 (s_1+s_2)}+14 e^{4 (s_1+s_2)}-6 e^{3 s_1+s_2}+e^{s_1+2 s_2}+7 e^{3 s_1+2 s_2}+5 e^{4 s_1+2 s_2}-e^{s_1+3 s_2}+e^{4 s_1+3 s_2}-12 e^{3 s_1+4 s_2}+2 e^{6 s_1+4 s_2}+3) s_1^2+s_3 (3 (-1+e^{s_1}) (1+2 e^{s_1}+e^{2 s_1}-e^{s_2}-2 e^{s_1+s_2}+7 e^{2 (s_1+s_2)}-18 e^{3 (s_1+s_2)}+12 e^{4 (s_1+s_2)}-11 e^{2 s_1+s_2}-2 e^{3 s_1+s_2}+16 e^{3 s_1+2 s_2}+e^{4 s_1+2 s_2}+5 e^{2 s_1+3 s_2}-e^{4 s_1+3 s_2}-2 e^{5 s_1+3 s_2}-4 e^{3 s_1+4 s_2}-6 e^{5 s_1+4 s_2}+2 e^{6 s_1+4 s_2})+4 e^{s_1} (-1-e^{s_1}+e^{2 s_2}+5 e^{s_1+s_2}-5 e^{2 (s_1+s_2)}+2 e^{3 (s_1+s_2)}-7 e^{4 (s_1+s_2)}+3 e^{2 s_1+s_2}-3 e^{s_1+2 s_2}-6 e^{3 s_1+2 s_2}+e^{4 s_1+2 s_2}-e^{s_1+3 s_2}+3 e^{2 s_1+3 s_2}+6 e^{4 s_1+3 s_2}-2 e^{5 s_1+3 s_2}-e^{2 s_1+4 s_2}+4 e^{3 s_1+4 s_2}+2 e^{5 s_1+4 s_2}) s_3) s_1+2 e^{s_1} (-3+3 e^{s_1}+e^{s_2}+2 e^{2 s_2}+13 e^{s_1+s_2}-21 e^{3 (s_1+s_2)}+20 e^{4 (s_1+s_2)}-11 e^{s_1+2 s_2}+19 e^{3 s_1+2 s_2}-e^{4 s_1+2 s_2}-7 e^{s_1+3 s_2}+e^{4 s_1+3 s_2}-18 e^{3 s_1+4 s_2}+2 e^{6 s_1+4 s_2}) s_3^2), 
$
\[
k_{17, 14}(s_1, s_2, s_3)\frac{3}{32} \left(e^{s_1}-1\right){}^4 \left(e^{s_2}-1\right) \left(e^{s_1+s_2}-1\right){}^4 s_1 \left(s_1+s_2\right){}^2 \left(s_2+s_3\right) \left(s_1+s_2+s_3\right){}^2= 
\]
$
\frac{1}{2} \pi  (4 e^{2 s_1+s_2} (1+22 e^{s_1}+2 e^{s_2}+22 e^{2 (s_1+s_2)}-76 e^{3 (s_1+s_2)}-46 e^{2 s_1+s_2}-32 e^{3 s_1+s_2}-4 e^{s_1+2 s_2}+22 e^{4 s_1+2 s_2}+35 e^{4 s_1+3 s_2}) s_2^3+e^{s_1+s_2} (4 e^{s_1} (2+44 e^{s_1}+4 e^{s_2}+44 e^{2 (s_1+s_2)}-107 e^{3 (s_1+s_2)}-92 e^{2 s_1+s_2}-64 e^{3 s_1+s_2}-8 e^{s_1+2 s_2}+44 e^{4 s_1+2 s_2}+25 e^{4 s_1+3 s_2}) s_3+160 e^{s_1}-125 e^{3 s_1}-124 e^{4 s_1}+16 e^{s_2}-82 e^{s_1+s_2}-137 e^{3 (s_1+s_2)}-20 e^{4 (s_1+s_2)}-16 e^{3 s_1+s_2}+280 e^{4 s_1+s_2}-160 e^{4 s_1+3 s_2}+200 e^{6 s_1+3 s_2}-11 e^{7 s_1+3 s_2}+5 e^{3 s_1+4 s_2}-140 e^{6 s_1+4 s_2}+11 e^{7 s_1+4 s_2}+8) s_2^2+e^{s_1} s_3 (4 e^{s_1+s_2} (1+22 e^{s_1}+2 e^{s_2}+22 e^{2 (s_1+s_2)}-46 e^{3 (s_1+s_2)}-46 e^{2 s_1+s_2}-32 e^{3 s_1+s_2}-4 e^{s_1+2 s_2}+22 e^{4 s_1+2 s_2}+5 e^{4 s_1+3 s_2}) s_3+29 e^{3 s_1}+5 e^{2 s_2}+196 e^{s_1+s_2}+656 e^{3 (s_1+s_2)}-469 e^{4 (s_1+s_2)}-106 e^{4 s_1+s_2}-49 e^{s_1+2 s_2}-176 e^{3 s_1+2 s_2}+715 e^{4 s_1+2 s_2}-26 e^{6 s_1+3 s_2}-26 e^{7 s_1+4 s_2}+10 e^{3 s_1+5 s_2}-44 e^{4 s_1+5 s_2}-244 e^{6 s_1+5 s_2}+26 e^{7 s_1+5 s_2}-11) s_2-(1-2 e^{s_1}-46 e^{3 s_1}-7 e^{4 s_1}-e^{s_2}-40 e^{3 (s_1+s_2)}-152 e^{3 s_1+s_2}+2 e^{s_1+2 s_2}+88 e^{4 s_1+2 s_2}-280 e^{4 s_1+3 s_2}+274 e^{6 s_1+3 s_2}-8 e^{7 s_1+3 s_2}-2 e^{3 s_1+4 s_2}-142 e^{6 s_1+4 s_2}-5 e^{4 s_1+5 s_2}+104 e^{7 s_1+5 s_2}) s_3^2), 
$
\[
k_{17, 15}(s_1, s_2, s_3)\frac{1}{32} \left(e^{s_1}-1\right){}^4 \left(e^{s_2}-1\right) \left(e^{s_1+s_2}-1\right){}^4 s_1 \left(s_1+s_2\right){}^2 \left(s_2+s_3\right) \left(s_1+s_2+s_3\right){}^2= 
\]
$
\frac{1}{2} \pi  (-4 e^{3 s_1+s_2} (-7 e^{s_1}+4 e^{s_2}-16 e^{2 (s_1+s_2)}-11 e^{s_1+3 s_2}+12 e^{4 s_1+3 s_2}) s_2^3+(-2 e^{3 s_1+s_2} (-28 e^{s_1}+16 e^{s_2}-64 e^{2 (s_1+s_2)}-29 e^{s_1+3 s_2}+33 e^{4 s_1+3 s_2}) s_3-8 e^{s_1}-12 e^{2 s_1}+8 e^{3 s_1}+11 e^{4 s_1}-e^{s_2}+76 e^{3 (s_1+s_2)}+28 e^{3 s_1+s_2}-120 e^{3 s_1+2 s_2}+54 e^{6 s_1+2 s_2}-14 e^{2 s_1+3 s_2}+80 e^{4 s_1+3 s_2}+8 e^{5 s_1+3 s_2}-130 e^{6 s_1+3 s_2}-20 e^{7 s_1+3 s_2}+8 e^{3 s_1+4 s_2}+28 e^{6 s_1+4 s_2}+48 e^{6 s_1+5 s_2}+1) s_2^2+e^{s_1} s_3 (-4 e^{2 s_1+s_2} (-7 e^{s_1}+4 e^{s_2}-16 e^{2 (s_1+s_2)}-6 e^{s_1+3 s_2}+7 e^{4 s_1+3 s_2}) s_3-39 e^{s_1}+33 e^{2 s_1}+2 e^{s_2}-208 e^{2 (s_1+s_2)}+84 e^{5 (s_1+s_2)}+108 e^{2 s_1+s_2}-140 e^{3 s_1+s_2}+43 e^{5 s_1+2 s_2}-10 e^{s_1+3 s_2}+62 e^{2 s_1+3 s_2}-32 e^{4 s_1+3 s_2}-230 e^{5 s_1+3 s_2}+5 e^{2 s_1+4 s_2}-33 e^{3 s_1+4 s_2}+103 e^{5 s_1+4 s_2}+90 e^{6 s_1+4 s_2}) s_2+e^{2 s_1} (-18+20 e^{s_2}-2 e^{3 s_2}-16 e^{3 (s_1+s_2)}-63 e^{2 s_1+s_2}-8 e^{3 s_1+s_2}-80 e^{s_1+2 s_2}+102 e^{3 s_1+2 s_2}+8 e^{4 s_1+2 s_2}-5 e^{2 s_1+4 s_2}-70 e^{3 s_1+4 s_2}+32 e^{5 s_1+4 s_2}-5 e^{6 s_1+4 s_2}-8 e^{3 s_1+5 s_2}+36 e^{4 s_1+5 s_2}+5 e^{6 s_1+5 s_2}) s_3^2), 
$
\[
k_{17, 16}(s_1, s_2, s_3)\frac{15}{32} \left(e^{s_1}-1\right){}^4 \left(e^{s_2}-1\right) \left(e^{s_1+s_2}-1\right){}^4 s_1 \left(s_1+s_2\right){}^2 \left(s_2+s_3\right) \left(s_1+s_2+s_3\right){}^2= 
\]
$
\pi  s_2 (2 (4 e^{s_1}-26 e^{3 s_1}-e^{4 s_1}+e^{s_2}-4 e^{5 (s_1+s_2)}+4 e^{5 s_1+s_2}-4 e^{s_1+2 s_2}+4 e^{7 s_1+3 s_2}-34 e^{3 s_1+4 s_2}+29 e^{8 s_1+4 s_2}+e^{4 s_1+5 s_2}+26 e^{7 s_1+5 s_2}+e^{8 s_1+5 s_2}-1) s_2^2+(16 e^{s_1}-104 e^{3 s_1}-4 e^{4 s_1}+4 e^{s_2}-16 e^{5 (s_1+s_2)}+16 e^{5 s_1+s_2}-16 e^{s_1+2 s_2}+16 e^{7 s_1+3 s_2}-91 e^{3 s_1+4 s_2}+71 e^{8 s_1+4 s_2}+4 e^{4 s_1+5 s_2}+104 e^{7 s_1+5 s_2}+4 e^{8 s_1+5 s_2}-4) s_3 s_2+2 (4 e^{s_1}-26 e^{3 s_1}-e^{4 s_1}+e^{s_2}-4 e^{5 (s_1+s_2)}+4 e^{5 s_1+s_2}-4 e^{s_1+2 s_2}+4 e^{7 s_1+3 s_2}-19 e^{3 s_1+4 s_2}+14 e^{8 s_1+4 s_2}+e^{4 s_1+5 s_2}+26 e^{7 s_1+5 s_2}+e^{8 s_1+5 s_2}-1) s_3^2), 
$
\[
k_{17, 17}(s_1, s_2, s_3)= \frac{4 \pi  \left(s_1^2-s_3 s_1-s_2^2+s_3^2\right)}{\left(e^{\frac{1}{2} \left(s_1+s_2+s_3\right)}-1\right){}^4 s_1 \left(s_1+s_2\right) s_3 \left(s_2+s_3\right) \left(s_1+s_2+s_3\right)}, 
\]
\[
k_{17, 18}(s_1, s_2, s_3)= \frac{4 \pi  \left(s_1^2-s_3 s_1-s_2^2+s_3^2\right)}{\left(e^{\frac{1}{2} \left(s_1+s_2+s_3\right)}+1\right){}^4 s_1 \left(s_1+s_2\right) s_3 \left(s_2+s_3\right) \left(s_1+s_2+s_3\right)}, 
\]
\[
k_{17, 19}(s_1, s_2, s_3)= \frac{8 \pi  e^{2 s_1} \left(s_2-s_3\right)}{\left(e^{s_1}-1\right){}^2 \left(e^{\frac{1}{2} \left(s_2+s_3\right)}-1\right){}^2 s_1 s_2 s_3 \left(s_1+s_2+s_3\right)}, 
\]
\[
k_{17, 20}(s_1, s_2, s_3)= \frac{8 \pi  e^{2 s_1} \left(s_2-s_3\right)}{\left(e^{s_1}-1\right){}^2 \left(e^{\frac{1}{2} \left(s_2+s_3\right)}+1\right){}^2 s_1 s_2 s_3 \left(s_1+s_2+s_3\right)}, 
\]
\[
k_{17, 21}(s_1, s_2, s_3)\left(e^{s_1}-1\right){}^3 \left(e^{s_2}-1\right) \left(e^{\frac{1}{2} \left(s_2+s_3\right)}-1\right) s_1 s_2 s_3 \left(s_2+s_3\right) \left(s_1+s_2+s_3\right)= 
\]
$
-8 \pi  e^{2 s_1} ((-e^{s_1}-e^{s_2}+3 e^{s_1+s_2}-1) s_2^2+2 (e^{s_1}-1) ((4 e^{s_2}-2) s_3+e^{s_2}-1) s_2+s_3 ((-3 e^{s_1}-7 e^{s_2}+5 e^{s_1+s_2}+5) s_3-2 (e^{s_1}-1) (e^{s_2}-1))),
$
\[
k_{17, 22}(s_1, s_2, s_3)\left(e^{s_1}-1\right) \left(e^{s_1+s_2}-1\right) \left(e^{\frac{1}{2} \left(s_1+s_2+s_3\right)}+1\right){}^3 s_1 s_2 \left(s_1+s_2\right) s_3 \left(s_2+s_3\right) \left(s_1+s_2+s_3\right){}^2= 
\]
$
-4 \pi  (((-3 e^{s_1}-3 e^{s_1+s_2}+2 e^{2 s_1+s_2}+4) s_2+e^{s_1} (e^{s_2}-1) s_3) s_1^3+((-3 e^{s_1}-5 e^{s_1+s_2}+2 e^{2 s_1+s_2}+6) s_2^2+((-3 e^{s_1}+e^{s_1+s_2}+2) s_3-2 (e^{s_1}-1) (e^{s_1+s_2}-1)) s_2+2 e^{s_1} (e^{s_2}-1) s_3^2) s_1^2-(e^{s_1} (e^{s_2}+2 e^{s_1+s_2}-3) s_2^3-(-e^{s_1}-3 e^{s_1+s_2}+2 e^{2 s_1+s_2}+2) s_3 s_2^2-s_3 (2 (e^{s_1}-1) (e^{s_1+s_2}-1)+(-5 e^{s_1}-e^{s_1+s_2}+4 e^{2 s_1+s_2}+2) s_3) s_2-e^{s_1} (e^{s_2}-1) s_3^3) s_1-s_2 (s_2+s_3) ((-3 e^{s_1}-e^{s_1+s_2}+2 e^{2 s_1+s_2}+2) s_2^2-2 (e^{s_1}-1) ((2 e^{s_1+s_2}-1) s_3+e^{s_1+s_2}-1) s_2-s_3 ((-5 e^{s_1}-5 e^{s_1+s_2}+6 e^{2 s_1+s_2}+4) s_3-2 (e^{s_1}-1) (e^{s_1+s_2}-1)))), 
$
\[
k_{17, 23}(s_1, s_2, s_3)\left(e^{s_1}-1\right){}^3 \left(e^{s_2}-1\right) \left(e^{\frac{1}{2} \left(s_2+s_3\right)}+1\right) s_1 s_2 s_3 \left(s_2+s_3\right) \left(s_1+s_2+s_3\right)= 
\]
$
8 \pi  e^{2 s_1} ((-e^{s_1}-e^{s_2}+3 e^{s_1+s_2}-1) s_2^2+2 (e^{s_1}-1) ((4 e^{s_2}-2) s_3+e^{s_2}-1) s_2+s_3 ((-3 e^{s_1}-7 e^{s_2}+5 e^{s_1+s_2}+5) s_3-2 (e^{s_1}-1) (e^{s_2}-1))), 
$
\[
k_{17, 24}(s_1, s_2, s_3) = \frac{-k_{17, 22}^{\text{num}}(s_1, s_2, s_3)}{\left(e^{s_1}-1\right) \left(e^{s_1+s_2}-1\right) \left(e^{\frac{1}{2} \left(s_1+s_2+s_3\right)}-1\right){}^3 s_1 s_2 \left(s_1+s_2\right) s_3 \left(s_2+s_3\right) \left(s_1+s_2+s_3\right){}^2},
\]
\[
k_{17, 25}(s_1, s_2, s_3)\left(e^{s_1}-1\right) \left(e^{s_1+s_2}-1\right){}^2 \left(e^{\frac{1}{2} \left(s_1+s_2+s_3\right)}-1\right){}^2 s_1 s_2 \left(s_1+s_2\right) s_3 \left(s_2+s_3\right) \left(s_1+s_2+s_3\right){}^2= 
\]
$
2 \pi  (((-11+10 e^{s_1}+13 e^{s_1+s_2}+2 e^{2 (s_1+s_2)}-15 e^{2 s_1+s_2}+e^{3 s_1+2 s_2}) s_2+e^{s_1} (-1+e^{s_2}) (-9+13 e^{s_1+s_2}) s_3) s_1^3+((-21+18 e^{s_1}+25 e^{s_1+s_2}-31 e^{2 s_1+s_2}+9 e^{3 s_1+2 s_2}) s_2^2+((-10+35 e^{s_1}-15 e^{s_1+s_2}+37 e^{2 (s_1+s_2)}-55 e^{2 s_1+s_2}+8 e^{3 s_1+2 s_2}) s_3-2 (-7+5 e^{s_1}+12 e^{s_1+s_2}-5 e^{2 (s_1+s_2)}-8 e^{2 s_1+s_2}+3 e^{3 s_1+2 s_2})) s_2+2 e^{s_1} (-1+e^{s_2}) s_3 ((-9+13 e^{s_1+s_2}) s_3-2 e^{s_1+s_2}+2)) s_1^2+((-9+6 e^{s_1}+11 e^{s_1+s_2}-6 e^{2 (s_1+s_2)}-17 e^{2 s_1+s_2}+15 e^{3 s_1+2 s_2}) s_2^3+(8 (-1+e^{s_1+s_2}){}^2+(-19+42 e^{s_1}+15 e^{s_1+s_2}-89 e^{2 s_1+s_2}+51 e^{3 s_1+2 s_2}) s_3) s_2^2+s_3 (2 (-3+e^{s_1}+8 e^{s_1+s_2}-5 e^{2 (s_1+s_2)}-4 e^{2 s_1+s_2}+3 e^{3 s_1+2 s_2})+(-10+45 e^{s_1}-5 e^{s_1+s_2}+19 e^{2 (s_1+s_2)}-85 e^{2 s_1+s_2}+36 e^{3 s_1+2 s_2}) s_3) s_2+e^{s_1} (-1+e^{s_2}) s_3^2 ((-9+13 e^{s_1+s_2}) s_3-4 e^{s_1+s_2}+4)) s_1+s_2 (s_2+s_3) ((1-2 e^{s_1}-e^{s_1+s_2}-4 e^{2 (s_1+s_2)}-e^{2 s_1+s_2}+7 e^{3 s_1+2 s_2}) s_2^2+2 ((-5+9 e^{s_1}+11 e^{s_1+s_2}-10 e^{2 (s_1+s_2)}-23 e^{2 s_1+s_2}+18 e^{3 s_1+2 s_2}) s_3+5 e^{s_1}+4 e^{s_1+s_2}-e^{2 (s_1+s_2)}-8 e^{2 s_1+s_2}+3 e^{3 s_1+2 s_2}-3) s_2+s_3 (2 (7-9 e^{s_1}-16 e^{s_1+s_2}+9 e^{2 (s_1+s_2)}+20 e^{2 s_1+s_2}-11 e^{3 s_1+2 s_2})+(-11+20 e^{s_1}+23 e^{s_1+s_2}-16 e^{2 (s_1+s_2)}-45 e^{2 s_1+s_2}+29 e^{3 s_1+2 s_2}) s_3))), 
$
\[
k_{17, 26}(s_1, s_2, s_3)= \frac{k_{17, 25}^{\text{num}}(s_1, s_2, s_3)}{\left(e^{s_1}-1\right) \left(e^{s_1+s_2}-1\right){}^2 \left(e^{\frac{1}{2} \left(s_1+s_2+s_3\right)}+1\right){}^2 s_1 s_2 \left(s_1+s_2\right) s_3 \left(s_2+s_3\right) \left(s_1+s_2+s_3\right){}^2}, 
\]
\[
k_{17, 27}(s_1, s_2, s_3)3 \left(e^{s_1}-1\right){}^3 \left(e^{s_1+s_2}-1\right){}^3 \left(e^{\frac{1}{2} \left(s_1+s_2+s_3\right)}+1\right) s_1 s_2 \left(s_1+s_2\right) s_3 \left(s_2+s_3\right) \left(s_1+s_2+s_3\right){}^2= 
\]
$
2 \pi  (3 ((-5+16 e^{s_1}-17 e^{2 s_1}+14 e^{3 s_1}+8 e^{s_1+s_2}+11 e^{2 (s_1+s_2)}-22 e^{3 (s_1+s_2)}-23 e^{2 s_1+s_2}+30 e^{3 s_1+s_2}-39 e^{4 s_1+s_2}-30 e^{3 s_1+2 s_2}+11 e^{4 s_1+2 s_2}+32 e^{5 s_1+2 s_2}+61 e^{4 s_1+3 s_2}-48 e^{5 s_1+3 s_2}+e^{6 s_1+3 s_2}) s_2-e^{s_1} (-7+30 e^{s_1}-15 e^{2 s_1}+7 e^{s_2}-4 e^{s_1+s_2}+57 e^{2 (s_1+s_2)}-70 e^{3 (s_1+s_2)}-69 e^{2 s_1+s_2}+42 e^{3 s_1+s_2}-26 e^{s_1+2 s_2}+28 e^{3 s_1+2 s_2}-35 e^{4 s_1+2 s_2}+27 e^{2 s_1+3 s_2}+35 e^{4 s_1+3 s_2}) s_3) s_1^3-(3 (11-36 e^{s_1}+39 e^{2 s_1}-38 e^{3 s_1}-18 e^{s_1+s_2}-25 e^{2 (s_1+s_2)}+40 e^{3 (s_1+s_2)}+51 e^{2 s_1+s_2}-72 e^{3 s_1+s_2}+111 e^{4 s_1+s_2}+66 e^{3 s_1+2 s_2}-9 e^{4 s_1+2 s_2}-104 e^{5 s_1+2 s_2}-105 e^{4 s_1+3 s_2}+66 e^{5 s_1+3 s_2}+23 e^{6 s_1+3 s_2}) s_2^2+(2 (-23+23 e^{s_1}+45 e^{s_1+s_2}-9 e^{2 (s_1+s_2)}-13 e^{3 (s_1+s_2)}-57 e^{2 s_1+s_2}+33 e^{3 s_1+2 s_2}+e^{4 s_1+3 s_2}) (-1+e^{s_1}){}^2+3 (6-41 e^{s_1}+112 e^{2 s_1}-69 e^{3 s_1}+11 e^{s_1+s_2}-92 e^{2 (s_1+s_2)}+99 e^{3 (s_1+s_2)}+16 e^{2 s_1+s_2}-249 e^{3 s_1+s_2}+198 e^{4 s_1+s_2}+207 e^{3 s_1+2 s_2}+86 e^{4 s_1+2 s_2}-177 e^{5 s_1+2 s_2}-254 e^{4 s_1+3 s_2}+123 e^{5 s_1+3 s_2}+24 e^{6 s_1+3 s_2}) s_3) s_2+6 e^{s_1} (-1+e^{s_2}) s_3 ((7-30 e^{s_1}+15 e^{2 s_1}-26 e^{s_1+s_2}+27 e^{2 (s_1+s_2)}+84 e^{2 s_1+s_2}-42 e^{3 s_1+s_2}-70 e^{3 s_1+2 s_2}+35 e^{4 s_1+2 s_2}) s_3-4 (-1+e^{s_1}){}^2 (2-5 e^{s_1+s_2}+3 e^{2 (s_1+s_2)}))) s_1^2+(-3 (7-24 e^{s_1}+27 e^{2 s_1}-34 e^{3 s_1}-12 e^{s_1+s_2}-17 e^{2 (s_1+s_2)}+14 e^{3 (s_1+s_2)}+33 e^{2 s_1+s_2}-54 e^{3 s_1+s_2}+105 e^{4 s_1+s_2}+42 e^{3 s_1+2 s_2}+15 e^{4 s_1+2 s_2}-112 e^{5 s_1+2 s_2}-27 e^{4 s_1+3 s_2}-12 e^{5 s_1+3 s_2}+49 e^{6 s_1+3 s_2}) s_2^3-3 ((13-64 e^{s_1}+137 e^{2 s_1}-94 e^{3 s_1}-14 e^{s_1+s_2}-63 e^{2 (s_1+s_2)}+56 e^{3 (s_1+s_2)}+89 e^{2 s_1+s_2}-336 e^{3 s_1+s_2}+285 e^{4 s_1+s_2}+114 e^{3 s_1+2 s_2}+217 e^{4 s_1+2 s_2}-292 e^{5 s_1+2 s_2}-115 e^{4 s_1+3 s_2}-42 e^{5 s_1+3 s_2}+109 e^{6 s_1+3 s_2}) s_3-16 (-1+e^{s_1}){}^2 (-1+e^{s_1+s_2}){}^2 (1-e^{s_1}+e^{2 s_1+s_2})) s_2^2+s_3 (2 (-1+e^{s_1}){}^2 (1-49 e^{s_1}+45 e^{s_1+s_2}-105 e^{2 (s_1+s_2)}+59 e^{3 (s_1+s_2)}+135 e^{2 s_1+s_2}-111 e^{3 s_1+2 s_2}+25 e^{4 s_1+3 s_2})-3 (6-47 e^{s_1}+140 e^{2 s_1}-75 e^{3 s_1}+5 e^{s_1+s_2}-72 e^{2 (s_1+s_2)}+69 e^{3 (s_1+s_2)}+52 e^{2 s_1+s_2}-351 e^{3 s_1+s_2}+222 e^{4 s_1+s_2}+129 e^{3 s_1+2 s_2}+230 e^{4 s_1+2 s_2}-215 e^{5 s_1+2 s_2}-158 e^{4 s_1+3 s_2}+5 e^{5 s_1+3 s_2}+60 e^{6 s_1+3 s_2}) s_3) s_2-3 e^{s_1} (-1+e^{s_2}) s_3^2 ((7-30 e^{s_1}+15 e^{2 s_1}-26 e^{s_1+s_2}+27 e^{2 (s_1+s_2)}+84 e^{2 s_1+s_2}-42 e^{3 s_1+s_2}-70 e^{3 s_1+2 s_2}+35 e^{4 s_1+2 s_2}) s_3-8 (-1+e^{s_1}){}^2 (2-5 e^{s_1+s_2}+3 e^{2 (s_1+s_2)}))) s_1-s_2 (s_2+s_3) (3 (1-4 e^{s_1}+5 e^{2 s_1}-10 e^{3 s_1}-2 e^{s_1+s_2}-3 e^{2 (s_1+s_2)}-4 e^{3 (s_1+s_2)}+5 e^{2 s_1+s_2}-12 e^{3 s_1+s_2}+33 e^{4 s_1+s_2}+6 e^{3 s_1+2 s_2}+13 e^{4 s_1+2 s_2}-40 e^{5 s_1+2 s_2}+17 e^{4 s_1+3 s_2}-30 e^{5 s_1+3 s_2}+25 e^{6 s_1+3 s_2}) s_2^2+2 (-1+e^{s_1}) (3 (-3+10 e^{s_1}-15 e^{2 s_1}+8 e^{s_1+s_2}-3 e^{2 (s_1+s_2)}+6 e^{3 (s_1+s_2)}-24 e^{2 s_1+s_2}+48 e^{3 s_1+s_2}+18 e^{3 s_1+2 s_2}-55 e^{4 s_1+2 s_2}-20 e^{4 s_1+3 s_2}+30 e^{5 s_1+3 s_2}) s_3-(-1+e^{s_1}) (1-e^{s_1}-3 e^{s_1+s_2}+15 e^{2 (s_1+s_2)}-13 e^{3 (s_1+s_2)}+15 e^{2 s_1+s_2}-39 e^{3 s_1+2 s_2}+25 e^{4 s_1+3 s_2})) s_2+s_3 (3 (5-22 e^{s_1}+45 e^{2 s_1}-20 e^{3 s_1}-14 e^{s_1+s_2}+9 e^{2 (s_1+s_2)}-8 e^{3 (s_1+s_2)}+59 e^{2 s_1+s_2}-132 e^{3 s_1+s_2}+63 e^{4 s_1+s_2}-48 e^{3 s_1+2 s_2}+133 e^{4 s_1+2 s_2}-70 e^{5 s_1+2 s_2}+35 e^{4 s_1+3 s_2}-70 e^{5 s_1+3 s_2}+35 e^{6 s_1+3 s_2}) s_3-2 (-1+e^{s_1}){}^2 (23-47 e^{s_1}-69 e^{s_1+s_2}+81 e^{2 (s_1+s_2)}-35 e^{3 (s_1+s_2)}+153 e^{2 s_1+s_2}-177 e^{3 s_1+2 s_2}+71 e^{4 s_1+3 s_2})))), 
$
\[
k_{17, 28}(s_1, s_2, s_3)= \frac{k_{17, 27}^{\text{num}}(s_1, s_2, s_3)}{3 \left(e^{s_1}-1\right){}^3 \left(e^{s_1+s_2}-1\right){}^3 \left(e^{\frac{1}{2} \left(s_1+s_2+s_3\right)}-1\right) s_1 s_2 \left(s_1+s_2\right) s_3 \left(s_2+s_3\right) \left(s_1+s_2+s_3\right){}^2}, 
\]
\[
k_{17, 29}(s_1, s_2, s_3)=-\frac{16 \pi  e^{2 s_1+3 s_2} \left(\left(e^{s_1} \left(2 e^{s_2}-1\right)-1\right) s_1+\left(e^{s_1}-1\right) s_2\right)}{\left(e^{s_2}-1\right) \left(e^{s_1+s_2}-1\right){}^3 \left(e^{\frac{s_3}{2}}+1\right) s_1 s_2 \left(s_1+s_2\right) \left(s_1+s_2+s_3\right)}, 
\]
\[
k_{17, 30}(s_1, s_2, s_3)= \frac{16 \pi  e^{2 s_1+3 s_2} \left(\left(e^{s_1} \left(2 e^{s_2}-1\right)-1\right) s_1+\left(e^{s_1}-1\right) s_2\right)}{\left(e^{s_2}-1\right) \left(e^{s_1+s_2}-1\right){}^3 \left(e^{\frac{s_3}{2}}-1\right) s_1 s_2 \left(s_1+s_2\right) \left(s_1+s_2+s_3\right)}.
\]
}

\subsubsection{The functions $k_{18}$ and $k_{19}$} 
As we mentioned earlier, the functions $K_{18}$ and $K_{19}$ 
appearing in \eqref{a_4expression} are identical. Therefore, 
the derived three variable functions $k_{18}$ and $k_{19}$ match 
precisely. For the function $k_{18}$ we have: 
\[
k_{18}(s_1, s_2, s_3)= 
K_{18}(s_1, s_2, s_3, -s_1-s_2-s_3)
=\sum_{i=1}^{30} k_{18, i}(s_1, s_2, s_3), 
\]
where  
{\tiny 
\[
k_{18, 1}(s_1, s_2, s_3) \left(e^{s_1}-1\right){}^4 \left(e^{s_2}-1\right) \left(e^{s_1+s_2}-1\right){}^4 s_1^2 \left(s_1+s_2\right){}^2 \left(s_1+s_2+s_3\right){}^2 =
\]
$
16 \pi  e^{2 s_1} (2 e^{3 s_2}-39 e^{2 s_1+s_2}-40 e^{s_1+2 s_2}+10 e^{4 s_1+2 s_2}-8 e^{3 s_1+3 s_2}+5 e^{2 s_1+4 s_2}+16 e^{5 s_1+4 s_2}+12 e^{4 s_1+5 s_2}-12) s_2 (s_2+s_3),
$
\[
k_{18, 2}(s_1, s_2, s_3)3 \left(e^{s_1}-1\right){}^4 \left(e^{s_2}-1\right) \left(e^{s_1+s_2}-1\right){}^4 s_1^2 \left(s_1+s_2\right){}^2 \left(s_1+s_2+s_3\right){}^2=
\]
$
-16 \pi  (-4 e^{s_1}-28 e^{3 s_1}-5 e^{4 s_1}-e^{s_2}-20 e^{2 s_1+s_2}-116 e^{3 s_1+s_2}+20 e^{5 s_1+s_2}+4 e^{s_1+2 s_2}-10 e^{2 s_1+2 s_2}+104 e^{4 s_1+2 s_2}-188 e^{5 s_1+2 s_2}+20 e^{3 s_1+3 s_2}-200 e^{4 s_1+3 s_2}+142 e^{6 s_1+3 s_2}+20 e^{7 s_1+3 s_2}+4 e^{3 s_1+4 s_2}+140 e^{5 s_1+4 s_2}-76 e^{6 s_1+4 s_2}-5 e^{8 s_1+4 s_2}-e^{4 s_1+5 s_2}+4 e^{5 s_1+5 s_2}+28 e^{7 s_1+5 s_2}+5 e^{8 s_1+5 s_2}+1) s_2 (s_2+s_3), 
$
\[
k_{18, 3}(s_1, s_2, s_3)3 \left(e^{s_1}-1\right){}^4 \left(e^{s_2}-1\right) \left(e^{s_1+s_2}-1\right){}^4 s_1 \left(s_1+s_2\right){}^2 s_3 \left(s_2+s_3\right) \left(s_1+s_2+s_3\right){}^2=
\]
$
-16 \pi  (-25 e^{s_1}-31 e^{3 s_1}+7 e^{4 s_1}-4 e^{s_2}+4 e^{s_1+s_2}-140 e^{3 s_1+s_2}+104 e^{4 s_1+s_2}+104 e^{3 s_1+3 s_2}-140 e^{4 s_1+3 s_2}+4 e^{6 s_1+3 s_2}-4 e^{7 s_1+3 s_2}+7 e^{3 s_1+4 s_2}-31 e^{4 s_1+4 s_2}-25 e^{6 s_1+4 s_2}+4 e^{7 s_1+4 s_2}+4) s_2^3,
$
\[
k_{18, 4}(s_1, s_2, s_3)\left(e^{s_1}-1\right){}^4 \left(e^{s_2}-1\right) \left(e^{s_1+s_2}-1\right){}^4 s_1 \left(s_1+s_2\right){}^2 s_3 \left(s_2+s_3\right) \left(s_1+s_2+s_3\right){}^2=
\]
$
-16 \pi  e^{s_1} s_2^3 (2 e^{s_1+3 s_2} (e^{s_1+s_2}-2) s_2 (e^{s_1}-1){}^5+15 e^{s_1}+7 e^{2 s_2}+20 e^{s_1+s_2}-8 e^{4 s_1+s_2}-27 e^{s_1+2 s_2}+20 e^{2 s_1+2 s_2}+20 e^{3 s_1+2 s_2}-27 e^{4 s_1+2 s_2}+7 e^{5 s_1+2 s_2}-8 e^{s_1+3 s_2}+20 e^{4 s_1+3 s_2}+15 e^{4 s_1+4 s_2}),
$
\[
k_{18, 5}(s_1, s_2, s_3)3 \left(e^{s_1}-1\right){}^4 \left(e^{s_2}-1\right) \left(e^{s_1+s_2}-1\right){}^4 \left(s_1+s_2\right){}^2 s_3 \left(s_2+s_3\right) \left(s_1+s_2+s_3\right){}^2=
\]
$
16 \pi  ((-7 e^{s_1}-e^{3 s_1}+e^{4 s_1}-4 e^{s_2}-14 e^{s_1+s_2}-80 e^{3 s_1+s_2}+56 e^{4 s_1+s_2}+74 e^{3 s_1+3 s_2}-92 e^{4 s_1+3 s_2}+4 e^{6 s_1+3 s_2}-10 e^{7 s_1+3 s_2}+7 e^{3 s_1+4 s_2}-25 e^{4 s_1+4 s_2}-31 e^{6 s_1+4 s_2}+10 e^{7 s_1+4 s_2}+4) s_1^2+(11 e^{s_1}+29 e^{3 s_1}-5 e^{4 s_1}-4 e^{s_2}-32 e^{s_1+s_2}-20 e^{3 s_1+s_2}+8 e^{4 s_1+s_2}+44 e^{3 s_1+3 s_2}-44 e^{4 s_1+3 s_2}+4 e^{6 s_1+3 s_2}-16 e^{7 s_1+3 s_2}+7 e^{3 s_1+4 s_2}-19 e^{4 s_1+4 s_2}-37 e^{6 s_1+4 s_2}+16 e^{7 s_1+4 s_2}+4) s_2 s_1+(43 e^{s_1}+61 e^{3 s_1}-13 e^{4 s_1}+4 e^{s_2}-22 e^{s_1+s_2}+200 e^{3 s_1+s_2}-152 e^{4 s_1+s_2}-134 e^{3 s_1+3 s_2}+188 e^{4 s_1+3 s_2}-4 e^{6 s_1+3 s_2}-2 e^{7 s_1+3 s_2}-7 e^{3 s_1+4 s_2}+37 e^{4 s_1+4 s_2}+19 e^{6 s_1+4 s_2}+2 e^{7 s_1+4 s_2}-4) s_2^2), 
$
\[
k_{18, 6}(s_1, s_2, s_3)\left(e^{s_1}-1\right){}^4 \left(e^{s_2}-1\right) \left(e^{s_1+s_2}-1\right){}^4 \left(s_1+s_2\right){}^2 s_3 \left(s_2+s_3\right) \left(s_1+s_2+s_3\right){}^2=
\]
$
-16 e^{s_1} \pi  (2 e^{s_1} (-1+e^{s_1}) (-1+e^{s_2}) (-2+e^{s_1}-2 e^{s_2}+2 e^{2 s_2}+9 e^{s_1+s_2}-4 e^{2 s_1+s_2}-3 e^{s_1+2 s_2}-4 e^{2 s_1+2 s_2}+2 e^{3 s_1+2 s_2}-e^{s_1+3 s_2}+2 e^{2 s_1+3 s_2}) s_1^3+(2 e^{s_1} (-1+e^{s_1}) (6-3 e^{s_1}-12 e^{2 s_2}+4 e^{3 s_2}-24 e^{s_1+s_2}+12 e^{2 s_1+s_2}+36 e^{s_1+2 s_2}-6 e^{3 s_1+2 s_2}+2 e^{s_1+3 s_2}-30 e^{2 s_1+3 s_2}+14 e^{3 s_1+3 s_2}-2 e^{4 s_1+3 s_2}-2 e^{s_1+4 s_2}+2 e^{2 s_1+4 s_2}+6 e^{3 s_1+4 s_2}-4 e^{4 s_1+4 s_2}+e^{5 s_1+4 s_2}) s_2-e^{s_1}-7 e^{2 s_2}-20 e^{s_1+s_2}+6 e^{4 s_1+s_2}+13 e^{s_1+2 s_2}-20 e^{3 s_1+2 s_2}+23 e^{4 s_1+2 s_2}-9 e^{5 s_1+2 s_2}+8 e^{s_1+3 s_2}-16 e^{4 s_1+3 s_2}-13 e^{4 s_1+4 s_2}) s_1^2+s_2 (6 e^{s_1} (-1+e^{s_1}) (2-e^{s_1}-4 e^{2 s_2}-8 e^{s_1+s_2}+4 e^{2 s_1+s_2}+12 e^{s_1+2 s_2}-2 e^{3 s_1+2 s_2}+6 e^{s_1+3 s_2}-18 e^{2 s_1+3 s_2}+10 e^{3 s_1+3 s_2}-2 e^{4 s_1+3 s_2}-2 e^{2 s_1+4 s_2}+6 e^{3 s_1+4 s_2}-4 e^{4 s_1+4 s_2}+e^{5 s_1+4 s_2}) s_2+13 e^{s_1}-7 e^{2 s_2}-20 e^{s_1+s_2}+4 e^{4 s_1+s_2}-e^{s_1+2 s_2}+20 e^{2 s_1+2 s_2}-20 e^{3 s_1+2 s_2}+19 e^{4 s_1+2 s_2}-11 e^{5 s_1+2 s_2}+8 e^{s_1+3 s_2}-12 e^{4 s_1+3 s_2}-11 e^{4 s_1+4 s_2}) s_1+s_2^2 (2 e^{s_1} (-1+e^{s_1}) (2-e^{s_1}-4 e^{2 s_2}-4 e^{3 s_2}-8 e^{s_1+s_2}+4 e^{2 s_1+s_2}+12 e^{s_1+2 s_2}-2 e^{3 s_1+2 s_2}+22 e^{s_1+3 s_2}-42 e^{2 s_1+3 s_2}+26 e^{3 s_1+3 s_2}-6 e^{4 s_1+3 s_2}+2 e^{s_1+4 s_2}-10 e^{2 s_1+4 s_2}+18 e^{3 s_1+4 s_2}-12 e^{4 s_1+4 s_2}+3 e^{5 s_1+4 s_2}) s_2+29 e^{s_1}+7 e^{2 s_2}+20 e^{s_1+s_2}-10 e^{4 s_1+s_2}-41 e^{s_1+2 s_2}+40 e^{2 s_1+2 s_2}+20 e^{3 s_1+2 s_2}-31 e^{4 s_1+2 s_2}+5 e^{5 s_1+2 s_2}-8 e^{s_1+3 s_2}+24 e^{4 s_1+3 s_2}+17 e^{4 s_1+4 s_2})), 
$
\[
k_{18, 7}(s_1, s_2, s_3)5 \left(e^{s_1}-1\right){}^4 \left(e^{s_2}-1\right) \left(e^{s_1+s_2}-1\right){}^4 s_1 \left(s_1+s_2\right){}^2 \left(s_2+s_3\right) \left(s_1+s_2+s_3\right){}^2=
\]
$
64 \pi  e^{2 s_1} s_2 ((-13 e^{3 s_2}-2 e^{4 s_1+2 s_2}+7 e^{4 s_1+5 s_2}-7) s_2^2+(-11 e^{3 s_2}-4 e^{4 s_1+2 s_2}+14 e^{4 s_1+5 s_2}-14) s_3 s_2+(-3 e^{3 s_2}-2 e^{4 s_1+2 s_2}+7 e^{4 s_1+5 s_2}-7) s_3^2),
$
\[
k_{18, 8}(s_1, s_2, s_3)3 \left(e^{s_1}-1\right){}^4 \left(e^{s_2}-1\right) \left(e^{s_1+s_2}-1\right){}^4 \left(s_1+s_2\right){}^2 \left(s_2+s_3\right) \left(s_1+s_2+s_3\right){}^2=
\]
$
-16 \pi  (4 e^{2 s_1+s_2} (-8-35 e^{s_1}+14 e^{s_2}+194 e^{2 s_1+s_2}+82 e^{3 s_1+s_2}-76 e^{s_1+2 s_2}+136 e^{2 s_1+2 s_2}+52 e^{4 s_1+2 s_2}-13 e^{3 s_1+3 s_2}+140 e^{4 s_1+3 s_2}) s_2^2+(14-97 e^{s_1}-215 e^{3 s_1}+13 e^{4 s_1}-14 e^{s_2}-676 e^{3 s_1+s_2}+61 e^{s_1+2 s_2}+380 e^{4 s_1+2 s_2}+256 e^{3 s_1+3 s_2}-908 e^{4 s_1+3 s_2}+512 e^{6 s_1+3 s_2}+20 e^{7 s_1+3 s_2}+47 e^{3 s_1+4 s_2}-329 e^{6 s_1+4 s_2}-10 e^{4 s_1+5 s_2}+136 e^{7 s_1+5 s_2}) s_2+2 e^{s_1} s_3 ((-1+e^{s_2}) (23+7 e^{3 s_1}+10 e^{s_2}-190 e^{3 s_1+s_2}-360 e^{3 s_1+2 s_2}-38 e^{3 s_1+3 s_2}-28 e^{6 s_1+3 s_2}-4 e^{3 s_1+4 s_2}+67 e^{6 s_1+4 s_2})+2 e^{s_1+s_2} (-4-25 e^{s_1}+4 e^{s_2}+94 e^{2 s_1+s_2}+44 e^{3 s_1+s_2}-2 e^{s_1+2 s_2}-16 e^{2 s_1+2 s_2}-22 e^{4 s_1+2 s_2}+28 e^{3 s_1+3 s_2}+37 e^{4 s_1+3 s_2}) s_3)+2 s_1 (2 e^{2 s_1+s_2} (-16-4 e^{s_1}+40 e^{s_2}+286 e^{2 s_1+s_2}+71 e^{3 s_1+s_2}-23 e^{s_1+2 s_2}+56 e^{2 s_1+2 s_2}+38 e^{4 s_1+2 s_2}+31 e^{3 s_1+3 s_2}+118 e^{4 s_1+3 s_2}) s_2-22 e^{s_1}-62 e^{3 s_1}+7 e^{4 s_1}-2 e^{s_2}-97 e^{3 s_1+s_2}+e^{s_1+2 s_2}+86 e^{4 s_1+2 s_2}+22 e^{3 s_1+3 s_2}-158 e^{4 s_1+3 s_2}+173 e^{6 s_1+3 s_2}-16 e^{7 s_1+3 s_2}+23 e^{3 s_1+4 s_2}-113 e^{6 s_1+4 s_2}-7 e^{4 s_1+5 s_2}+55 e^{7 s_1+5 s_2}+2 e^{2 s_1+s_2} (-16-31 e^{s_1}+28 e^{s_2}+262 e^{2 s_1+s_2}+74 e^{3 s_1+s_2}+13 e^{s_1+2 s_2}-58 e^{2 s_1+2 s_2}-31 e^{4 s_1+2 s_2}+67 e^{3 s_1+3 s_2}+82 e^{4 s_1+3 s_2}) s_3+2)), 
$
\[
k_{18, 9}(s_1, s_2, s_3)3 \left(e^{s_1}-1\right){}^4 \left(e^{s_1+s_2}-1\right){}^4 s_2 \left(s_1+s_2\right){}^2 \left(s_2+s_3\right) \left(s_1+s_2+s_3\right){}^2=
\]
$
-16 \pi  e^{s_1} ((-5 e^{3 s_1}-5 e^{s_2}-55 e^{3 s_1+s_2}-63 e^{3 s_1+2 s_2}-23 e^{3 s_1+3 s_2}+14 e^{6 s_1+3 s_2}-4 e^{3 s_1+4 s_2}+28 e^{6 s_1+4 s_2}+5) s_1^2+2 (4 e^{3 s_1}-2 e^{s_2}-e^{3 s_1+s_2}-45 e^{3 s_1+2 s_2}-17 e^{3 s_1+3 s_2}-7 e^{6 s_1+3 s_2}-e^{3 s_1+4 s_2}+13 e^{6 s_1+4 s_2}+2) s_3^2), 
$
\[
k_{18, 10}(s_1, s_2, s_3)\left(e^{s_1}-1\right){}^4 \left(e^{s_2}-1\right) \left(e^{s_1+s_2}-1\right){}^4 \left(s_1+s_2\right){}^2 \left(s_2+s_3\right) \left(s_1+s_2+s_3\right){}^2=
\]
$
16 \pi  (-2 e^{2 s_1} (-2+16 e^{s_1}-4 e^{2 s_1}-8 e^{s_2}+20 e^{2 s_2}-8 e^{3 s_2}+6 e^{s_1+s_2}-64 e^{2 s_1+s_2}+16 e^{3 s_1+s_2}-48 e^{s_1+2 s_2}+116 e^{2 s_1+2 s_2}+16 e^{3 s_1+2 s_2}-4 e^{4 s_1+2 s_2}+12 e^{s_1+3 s_2}-16 e^{2 s_1+3 s_2}-92 e^{3 s_1+3 s_2}+8 e^{4 s_1+3 s_2}-4 e^{5 s_1+3 s_2}+3 e^{s_1+4 s_2}-7 e^{2 s_1+4 s_2}+30 e^{3 s_1+4 s_2}+24 e^{4 s_1+4 s_2}-e^{5 s_1+4 s_2}+e^{6 s_1+4 s_2}-8 e^{4 s_1+5 s_2}-2 e^{5 s_1+5 s_2}) s_1^2+2 e^{s_1} (e^{2 s_1+s_2} (106 e^{s_1}+64 e^{s_2}+228 e^{2 s_1+2 s_2}-3 e^{s_1+3 s_2}+8 e^{4 s_1+3 s_2}) s_2-25 e^{s_1}-7 e^{s_2}+18 e^{s_1+s_2}-44 e^{3 s_1+s_2}+2 e^{s_1+2 s_2}-38 e^{2 s_1+2 s_2}+65 e^{4 s_1+2 s_2}+5 e^{s_1+3 s_2}+2 e^{4 s_1+3 s_2}+20 e^{3 s_1+4 s_2}-60 e^{4 s_1+4 s_2}+13 e^{6 s_1+4 s_2}-3 e^{7 s_1+4 s_2}-7 e^{4 s_1+5 s_2}+20 e^{5 s_1+5 s_2}+3 e^{7 s_1+5 s_2}-e^{2 s_1+s_2} (-86 e^{s_1}-32 e^{s_2}-152 e^{2 s_1+2 s_2}-3 e^{s_1+3 s_2}+11 e^{4 s_1+3 s_2}) s_3) s_1+2 e^{3 s_1+s_2} (68 e^{s_1}+16 e^{s_2}+244 e^{2 s_1+2 s_2}-15 e^{s_1+3 s_2}+13 e^{4 s_1+3 s_2}) s_2^2-2 s_3 (2 e^{4 s_1+s_2} (-15-e^{3 s_2}-34 e^{s_1+2 s_2}+4 e^{3 s_1+3 s_2}) s_3+32 e^{2 s_1}-23 e^{3 s_1}-e^{s_2}-18 e^{2 s_1+s_2}-67 e^{3 s_1+s_2}+16 e^{5 s_1+s_2}-7 e^{2 s_1+2 s_2}+74 e^{3 s_1+2 s_2}-104 e^{5 s_1+2 s_2}-23 e^{6 s_1+2 s_2}-7 e^{2 s_1+3 s_2}+10 e^{3 s_1+3 s_2}+2 e^{5 s_1+3 s_2}+93 e^{6 s_1+3 s_2}+6 e^{3 s_1+4 s_2}+83 e^{5 s_1+4 s_2}-46 e^{6 s_1+4 s_2}+3 e^{5 s_1+5 s_2}-24 e^{6 s_1+5 s_2}+1)+e^{s_1} s_2 (-2 e^{s_1+s_2} (-8-48 e^{s_1}-86 e^{2 s_1}+8 e^{s_2}+180 e^{2 s_1+s_2}+82 e^{3 s_1+s_2}-26 e^{s_1+2 s_2}+16 e^{2 s_1+2 s_2}-244 e^{3 s_1+2 s_2}-12 e^{4 s_1+2 s_2}+e^{2 s_1+3 s_2}+38 e^{3 s_1+3 s_2}+88 e^{4 s_1+3 s_2}+10 e^{5 s_1+3 s_2}) s_3-95 e^{s_1}-12 e^{s_2}-4 e^{s_1+s_2}-214 e^{3 s_1+s_2}+71 e^{s_1+2 s_2}-196 e^{2 s_1+2 s_2}+259 e^{4 s_1+2 s_2}+13 e^{5 s_1+2 s_2}+28 e^{s_1+3 s_2}-68 e^{4 s_1+3 s_2}+39 e^{3 s_1+4 s_2}-183 e^{4 s_1+4 s_2}+52 e^{6 s_1+4 s_2}-2 e^{7 s_1+4 s_2}-8 e^{4 s_1+5 s_2}+48 e^{5 s_1+5 s_2}+2 e^{7 s_1+5 s_2})), 
$
\[
k_{18, 11}(s_1, s_2, s_3)5 \left(e^{s_1}-1\right){}^4 \left(e^{s_2}-1\right) \left(e^{s_1+s_2}-1\right){}^4 \left(s_1+s_2\right){}^2 \left(s_2+s_3\right) \left(s_1+s_2+s_3\right){}^2=
\]
$
32 \pi  (2 e^{2 s_1} (-16 e^{3 s_2}-4 e^{4 s_1+2 s_2}+24 e^{4 s_1+5 s_2}-9) s_2^2+(8 e^{s_1}-42 e^{2 s_1}-107 e^{3 s_1}+3 e^{4 s_1}+2 e^{s_2}-12 e^{5 s_1+s_2}-8 e^{s_1+2 s_2}-42 e^{6 s_1+2 s_2}-8 e^{2 s_1+3 s_2}+78 e^{7 s_1+3 s_2}-3 e^{3 s_1+4 s_2}-12 e^{8 s_1+4 s_2}+2 e^{4 s_1+5 s_2}-8 e^{5 s_1+5 s_2}+72 e^{6 s_1+5 s_2}+72 e^{7 s_1+5 s_2}+2 e^{8 s_1+5 s_2}-2) s_3 s_2+2 e^{2 s_1} (2 e^{3 s_2}-12 e^{4 s_1+2 s_2}+12 e^{4 s_1+5 s_2}-7) s_3^2+2 e^{2 s_1} s_1 ((12 e^{3 s_2}+3 e^{4 s_1+2 s_2}+32 e^{4 s_1+5 s_2}+3) s_2+(17 e^{3 s_2}-22 e^{4 s_1+2 s_2}+32 e^{4 s_1+5 s_2}-7) s_3)),
$
\[
k_{18, 12}(s_1, s_2, s_3)15 \left(e^{s_1}-1\right){}^4 \left(e^{s_2}-1\right) \left(e^{s_1+s_2}-1\right){}^4 \left(s_1+s_2\right){}^2 \left(s_2+s_3\right) \left(s_1+s_2+s_3\right){}^2=
\]
$
32 \pi  ((16 e^{s_1}-254 e^{3 s_1}+26 e^{4 s_1}+4 e^{s_2}-104 e^{5 s_1+s_2}-16 e^{s_1+2 s_2}+256 e^{7 s_1+3 s_2}+29 e^{3 s_1+4 s_2}-79 e^{8 s_1+4 s_2}+4 e^{4 s_1+5 s_2}-16 e^{5 s_1+5 s_2}+134 e^{7 s_1+5 s_2}+4 e^{8 s_1+5 s_2}-4) s_2^2+2 (4 e^{s_1}-56 e^{3 s_1}-e^{4 s_1}+e^{s_2}+4 e^{5 s_1+s_2}-4 e^{s_1+2 s_2}+34 e^{7 s_1+3 s_2}-4 e^{3 s_1+4 s_2}-e^{8 s_1+4 s_2}+e^{4 s_1+5 s_2}-4 e^{5 s_1+5 s_2}+41 e^{7 s_1+5 s_2}+e^{8 s_1+5 s_2}-1) s_3^2+s_1 ((8 e^{s_1}-397 e^{3 s_1}+73 e^{4 s_1}+2 e^{s_2}-292 e^{5 s_1+s_2}-8 e^{s_1+2 s_2}+218 e^{7 s_1+3 s_2}-23 e^{3 s_1+4 s_2}-62 e^{8 s_1+4 s_2}+2 e^{4 s_1+5 s_2}-8 e^{5 s_1+5 s_2}+112 e^{7 s_1+5 s_2}+2 e^{8 s_1+5 s_2}-2) s_2+(8 e^{s_1}-322 e^{3 s_1}+28 e^{4 s_1}+2 e^{s_2}-112 e^{5 s_1+s_2}-8 e^{s_1+2 s_2}+158 e^{7 s_1+3 s_2}-23 e^{3 s_1+4 s_2}-17 e^{8 s_1+4 s_2}+2 e^{4 s_1+5 s_2}-8 e^{5 s_1+5 s_2}+142 e^{7 s_1+5 s_2}+2 e^{8 s_1+5 s_2}-2) s_3)),
$
\[
k_{18, 13}(s_1, s_2, s_3)\left(e^{s_1}-1\right){}^4 \left(e^{s_1+s_2}-1\right){}^4 s_2 \left(s_1+s_2\right){}^2 \left(s_2+s_3\right) \left(s_1+s_2+s_3\right){}^2=
\]
$
16 e^{s_1} \pi  (2 e^{s_1} (-1+e^{s_2}) (-3 e^{s_1}+e^{2 s_1}+2 e^{s_2}-6 e^{s_1+s_2}+14 e^{2 s_1+s_2}-4 e^{3 s_1+s_2}-e^{s_1+2 s_2}+3 e^{2 s_1+2 s_2}-10 e^{3 s_1+2 s_2}+2 e^{4 s_1+2 s_2}+2 e^{4 s_1+3 s_2}) s_1^3+e^{s_1} (2 (7 e^{s_1}-e^{2 s_1}-4 e^{s_2}+4 e^{2 s_2}+5 e^{s_1+s_2}-29 e^{2 s_1+s_2}+4 e^{3 s_1+s_2}-11 e^{s_1+2 s_2}+27 e^{2 s_1+2 s_2}+14 e^{3 s_1+2 s_2}+2 e^{4 s_1+2 s_2}-e^{s_1+3 s_2}+3 e^{2 s_1+3 s_2}-18 e^{3 s_1+3 s_2}-6 e^{4 s_1+3 s_2}-2 e^{5 s_1+3 s_2}+4 e^{4 s_1+4 s_2}+2 e^{5 s_1+4 s_2}) s_3-9 e^{s_1}-5 e^{s_2}-13 e^{s_1+s_2}-2 e^{3 s_1+s_2}+3 e^{s_1+2 s_2}-27 e^{3 s_1+2 s_2}+3 e^{4 s_1+2 s_2}-3 e^{s_1+3 s_2}-27 e^{3 s_1+3 s_2}+27 e^{4 s_1+3 s_2}-4 e^{3 s_1+4 s_2}+10 e^{4 s_1+4 s_2}+2 e^{6 s_1+4 s_2}+9) s_1^2+s_3 ((-1+e^{s_1}) (3-14 e^{s_1}-e^{2 s_1}-3 e^{s_2}+8 e^{s_1+s_2}+31 e^{2 s_1+s_2}+12 e^{3 s_1+s_2}-2 e^{s_1+2 s_2}-3 e^{2 s_1+2 s_2}-54 e^{3 s_1+2 s_2}-13 e^{4 s_1+2 s_2}+7 e^{2 s_1+3 s_2}-12 e^{3 s_1+3 s_2}+53 e^{4 s_1+3 s_2}-2 e^{3 s_1+4 s_2}+4 e^{4 s_1+4 s_2}-16 e^{5 s_1+4 s_2}+2 e^{6 s_1+4 s_2})+4 e^{s_1} (2 e^{s_1}-e^{s_2}+e^{2 s_2}+e^{s_1+s_2}-8 e^{2 s_1+s_2}-3 e^{s_1+2 s_2}+8 e^{2 s_1+2 s_2}+4 e^{3 s_1+2 s_2}+2 e^{4 s_1+2 s_2}-4 e^{3 s_1+3 s_2}-3 e^{4 s_1+3 s_2}-e^{5 s_1+3 s_2}+e^{4 s_1+4 s_2}+e^{5 s_1+4 s_2}) s_3) s_1+2 e^{s_1} (4-2 e^{s_1}-3 e^{s_2}+e^{2 s_2}-5 e^{s_1+s_2}+7 e^{3 s_1+s_2}-e^{s_1+2 s_2}-7 e^{3 s_1+2 s_2}-8 e^{4 s_1+2 s_2}-2 e^{s_1+3 s_2}-19 e^{3 s_1+3 s_2}+13 e^{4 s_1+3 s_2}-e^{3 s_1+4 s_2}+5 e^{4 s_1+4 s_2}) s_3^2),
$
\[
k_{18, 14}(s_1, s_2, s_3)3 \left(e^{s_1}-1\right){}^4 \left(e^{s_2}-1\right) \left(e^{s_1+s_2}-1\right){}^4 s_1 \left(s_1+s_2\right){}^2 \left(s_2+s_3\right) \left(s_1+s_2+s_3\right){}^2=
\]
$
-16 \pi  (4 e^{2 s_1+s_2} (-1-22 e^{s_1}-2 e^{s_2}+46 e^{2 s_1+s_2}+32 e^{3 s_1+s_2}-41 e^{s_1+2 s_2}+68 e^{2 s_1+2 s_2}+23 e^{4 s_1+2 s_2}-14 e^{3 s_1+3 s_2}+55 e^{4 s_1+3 s_2}) s_2^3+(4 e^{2 s_1+s_2} (-2-44 e^{s_1}-4 e^{s_2}+92 e^{2 s_1+s_2}+64 e^{3 s_1+s_2}-37 e^{s_1+2 s_2}+46 e^{2 s_1+2 s_2}+e^{4 s_1+2 s_2}+17 e^{3 s_1+3 s_2}+65 e^{4 s_1+3 s_2}) s_3-62 e^{s_1}-146 e^{3 s_1}-e^{4 s_1}-11 e^{s_2}+4 e^{s_1+s_2}+58 e^{2 s_1+s_2}-610 e^{3 s_1+s_2}+545 e^{4 s_1+s_2}+14 e^{5 s_1+s_2}+58 e^{s_1+2 s_2}-196 e^{2 s_1+2 s_2}+320 e^{4 s_1+2 s_2}-650 e^{5 s_1+2 s_2}+214 e^{3 s_1+3 s_2}-832 e^{4 s_1+3 s_2}+380 e^{6 s_1+3 s_2}+58 e^{7 s_1+3 s_2}+14 e^{3 s_1+4 s_2}-31 e^{4 s_1+4 s_2}+398 e^{5 s_1+4 s_2}-230 e^{6 s_1+4 s_2}-140 e^{7 s_1+4 s_2}-11 e^{8 s_1+4 s_2}-e^{4 s_1+5 s_2}-2 e^{5 s_1+5 s_2}+82 e^{7 s_1+5 s_2}+11 e^{8 s_1+5 s_2}+11) s_2^2+s_3 (4 e^{2 s_1+s_2} (-1-22 e^{s_1}-2 e^{s_2}+46 e^{2 s_1+s_2}+32 e^{3 s_1+s_2}-11 e^{s_1+2 s_2}+8 e^{2 s_1+2 s_2}-7 e^{4 s_1+2 s_2}+16 e^{3 s_1+3 s_2}+25 e^{4 s_1+3 s_2}) s_3-53 e^{s_1}-163 e^{3 s_1}-19 e^{4 s_1}-10 e^{s_2}-692 e^{3 s_1+s_2}+53 e^{s_1+2 s_2}+388 e^{4 s_1+2 s_2}+140 e^{3 s_1+3 s_2}-1012 e^{4 s_1+3 s_2}+580 e^{6 s_1+3 s_2}+112 e^{7 s_1+3 s_2}+7 e^{3 s_1+4 s_2}-277 e^{6 s_1+4 s_2}-2 e^{4 s_1+5 s_2}+128 e^{7 s_1+5 s_2}+10) s_2+e^{s_1} (-16-11 e^{3 s_1}+16 e^{2 s_2}-10 e^{s_1+s_2}+46 e^{4 s_1+s_2}-50 e^{s_1+2 s_2}+128 e^{3 s_1+2 s_2}-256 e^{4 s_1+2 s_2}-320 e^{3 s_1+3 s_2}+50 e^{6 s_1+3 s_2}+160 e^{4 s_1+4 s_2}-7 e^{7 s_1+4 s_2}-e^{3 s_1+5 s_2}+2 e^{4 s_1+5 s_2}+46 e^{6 s_1+5 s_2}+7 e^{7 s_1+5 s_2}) s_3^2), 
$
\[
k_{18, 15}(s_1, s_2, s_3)\left(e^{s_1}-1\right){}^4 \left(e^{s_2}-1\right) \left(e^{s_1+s_2}-1\right){}^4 s_1 \left(s_1+s_2\right){}^2 \left(s_2+s_3\right) \left(s_1+s_2+s_3\right){}^2=
\]
$
-16 \pi  (-4 e^{3 s_1+s_2} (7 e^{s_1}-4 e^{s_2}+46 e^{2 s_1+2 s_2}-4 e^{s_1+3 s_2}+3 e^{4 s_1+3 s_2}) s_2^3+2 e^{2 s_1} (e^{s_1+s_2} (-28 e^{s_1}+16 e^{s_2}-124 e^{2 s_1+2 s_2}+e^{s_1+3 s_2}+3 e^{4 s_1+3 s_2}) s_3-10 e^{3 s_2}+88 e^{s_1+2 s_2}-10 e^{4 s_1+2 s_2}+40 e^{3 s_1+3 s_2}-15 e^{4 s_1+5 s_2}+33) s_2^2+e^{2 s_1} s_3 (4 e^{s_1+s_2} (-7 e^{s_1}+4 e^{s_2}-26 e^{2 s_1+2 s_2}-e^{s_1+3 s_2}+2 e^{4 s_1+3 s_2}) s_3-4 e^{s_2}-55 e^{2 s_2}-16 e^{3 s_2}+210 e^{2 s_1+s_2}+28 e^{3 s_1+s_2}+236 e^{s_1+2 s_2}-275 e^{3 s_1+2 s_2}-53 e^{4 s_1+2 s_2}+76 e^{3 s_1+3 s_2}+5 e^{2 s_1+4 s_2}+171 e^{3 s_1+4 s_2}-80 e^{5 s_1+4 s_2}-6 e^{6 s_1+4 s_2}-48 e^{4 s_1+5 s_2}+6 e^{6 s_1+5 s_2}+75) s_2-(-24 e^{2 s_1}+16 e^{3 s_1}+e^{s_2}+74 e^{3 s_1+s_2}-63 e^{4 s_1+s_2}-80 e^{3 s_1+2 s_2}+26 e^{6 s_1+2 s_2}+4 e^{2 s_1+3 s_2}-10 e^{3 s_1+3 s_2}-16 e^{5 s_1+3 s_2}-68 e^{6 s_1+3 s_2}-5 e^{4 s_1+4 s_2}+24 e^{6 s_1+4 s_2}+32 e^{7 s_1+4 s_2}+18 e^{6 s_1+5 s_2}-1) s_3^2),
$
\[
k_{18, 16}(s_1, s_2, s_3)15 \left(e^{s_1}-1\right){}^4 \left(e^{s_2}-1\right) \left(e^{s_1+s_2}-1\right){}^4 s_1 \left(s_1+s_2\right){}^2 \left(s_2+s_3\right) \left(s_1+s_2+s_3\right){}^2=
\]
$
32 \pi  s_2 (2 (4 e^{s_1}-26 e^{3 s_1}-e^{4 s_1}+e^{s_2}+4 e^{5 s_1+s_2}-4 e^{s_1+2 s_2}+49 e^{7 s_1+3 s_2}+11 e^{3 s_1+4 s_2}-16 e^{8 s_1+4 s_2}+e^{4 s_1+5 s_2}-4 e^{5 s_1+5 s_2}+26 e^{7 s_1+5 s_2}+e^{8 s_1+5 s_2}-1) s_2^2+(16 e^{s_1}-104 e^{3 s_1}-4 e^{4 s_1}+4 e^{s_2}+16 e^{5 s_1+s_2}-16 e^{s_1+2 s_2}+106 e^{7 s_1+3 s_2}-e^{3 s_1+4 s_2}-19 e^{8 s_1+4 s_2}+4 e^{4 s_1+5 s_2}-16 e^{5 s_1+5 s_2}+104 e^{7 s_1+5 s_2}+4 e^{8 s_1+5 s_2}-4) s_3 s_2+2 (4 e^{s_1}-26 e^{3 s_1}-e^{4 s_1}+e^{s_2}+4 e^{5 s_1+s_2}-4 e^{s_1+2 s_2}+19 e^{7 s_1+3 s_2}-4 e^{3 s_1+4 s_2}-e^{8 s_1+4 s_2}+e^{4 s_1+5 s_2}-4 e^{5 s_1+5 s_2}+26 e^{7 s_1+5 s_2}+e^{8 s_1+5 s_2}-1) s_3^2),
$
\[
k_{18, 17}(s_1, s_2, s_3)=\frac{4 \pi  \left(s_1^2-s_3 s_1-s_2^2+s_3^2\right)}{\left(e^{\frac{1}{2} \left(s_1+s_2+s_3\right)}-1\right){}^4 s_1 \left(s_1+s_2\right) s_3 \left(s_2+s_3\right) \left(s_1+s_2+s_3\right)}, 
\]
\[
k_{18, 18}(s_1, s_2, s_3)=\frac{4 \pi  \left(s_1^2-s_3 s_1-s_2^2+s_3^2\right)}{\left(e^{\frac{1}{2} \left(s_1+s_2+s_3\right)}+1\right){}^4 s_1 \left(s_1+s_2\right) s_3 \left(s_2+s_3\right) \left(s_1+s_2+s_3\right)}, 
\]
\[
k_{18, 19}(s_1, s_2, s_3)=\frac{8 \pi  e^{2 s_1} \left(s_2-s_3\right)}{\left(e^{s_1}-1\right){}^2 \left(e^{\frac{1}{2} \left(s_2+s_3\right)}-1\right){}^2 s_1 s_2 s_3 \left(s_1+s_2+s_3\right)}, 
\]
\[
k_{18, 20}(s_1, s_2, s_3)=\frac{8 \pi  e^{2 s_1} \left(s_2-s_3\right)}{\left(e^{s_1}-1\right){}^2 \left(e^{\frac{1}{2} \left(s_2+s_3\right)}+1\right){}^2 s_1 s_2 s_3 \left(s_1+s_2+s_3\right)}, 
\]
\[
k_{18, 21}(s_1, s_2, s_3)\left(e^{s_1}-1\right){}^3 \left(e^{s_2}-1\right) \left(e^{\frac{1}{2} \left(s_2+s_3\right)}-1\right) s_1 s_2 s_3 \left(s_2+s_3\right) \left(s_1+s_2+s_3\right)=
\]
$
8 \pi  e^{2 s_1} ((-3 e^{s_1}-3 e^{s_2}+e^{s_1+s_2}+5) s_2^2-2 (e^{s_1}-1) (2 s_3+e^{s_2}-1) s_2-s_3 ((e^{s_1}-3 e^{s_2}+e^{s_1+s_2}+1) s_3-2 (e^{s_1}-1) (e^{s_2}-1))),
$
\[
k_{18, 22}(s_1, s_2, s_3)\left(e^{s_1}-1\right) \left(e^{s_1+s_2}-1\right) \left(e^{\frac{1}{2} \left(s_1+s_2+s_3\right)}+1\right){}^3 s_1 s_2 \left(s_1+s_2\right) s_3 \left(s_2+s_3\right) \left(s_1+s_2+s_3\right){}^2=
\]
$
-4 \pi  (((-3 e^{s_1}-3 e^{s_1+s_2}+2 e^{2 s_1+s_2}+4) s_2+e^{s_1} (e^{s_2}-1) s_3) s_1^3-((e^{s_1}+3 e^{s_1+s_2}-4) s_2^2+(2 (e^{s_1}-1) (e^{s_1+s_2}-1)+e^{s_1} (-3 e^{s_2}+2 e^{s_1+s_2}+1) s_3) s_2-2 e^{s_1} (e^{s_2}-1) s_3^2) s_1^2-((-7 e^{s_1}-3 e^{s_1+s_2}+6 e^{2 s_1+s_2}+4) s_2^3+(-7 e^{s_1}-5 e^{s_1+s_2}+6 e^{2 s_1+s_2}+6) s_3 s_2^2-s_3 (2 (e^{s_1}-1) (e^{s_1+s_2}-1)+(-e^{s_1}+3 e^{s_1+s_2}-2) s_3) s_2-e^{s_1} (e^{s_2}-1) s_3^3) s_1-s_2 (s_2+s_3) ((-5 e^{s_1}-3 e^{s_1+s_2}+4 e^{2 s_1+s_2}+4) s_2^2-2 (e^{s_1}-1) (s_3+e^{s_1+s_2}-1) s_2-s_3 ((-3 e^{s_1}-3 e^{s_1+s_2}+4 e^{2 s_1+s_2}+2) s_3-2 (e^{s_1}-1) (e^{s_1+s_2}-1)))), 
$
\[
k_{18, 23}(s_1, s_2, s_3)=\frac{-k_{18, 21}^{\text{num}}(s_1, s_2, s_3)}{\left(e^{s_1}-1\right){}^3 \left(e^{s_2}-1\right) \left(e^{\frac{1}{2} \left(s_2+s_3\right)}+1\right) s_1 s_2 s_3 \left(s_2+s_3\right) \left(s_1+s_2+s_3\right)},
\]
\[
k_{18, 24}(s_1, s_2, s_3)=\frac{-k_{18, 22}^{\text{num}}(s_1, s_2, s_3)}{\left(e^{s_1}-1\right) \left(e^{s_1+s_2}-1\right) \left(e^{\frac{1}{2} \left(s_1+s_2+s_3\right)}-1\right){}^3 s_1 s_2 \left(s_1+s_2\right) s_3 \left(s_2+s_3\right) \left(s_1+s_2+s_3\right){}^2}, 
\]
\[
k_{18, 25}(s_1, s_2, s_3)\left(e^{s_1}-1\right) \left(e^{s_1+s_2}-1\right){}^2 \left(e^{\frac{1}{2} \left(s_1+s_2+s_3\right)}-1\right){}^2 s_1 s_2 \left(s_1+s_2\right) s_3 \left(s_2+s_3\right) \left(s_1+s_2+s_3\right){}^2=
\]
$
2 \pi  (((-11+2 e^{s_1}+21 e^{s_1+s_2}-7 e^{2 s_1+s_2}-6 e^{2 s_1+2 s_2}+e^{3 s_1+2 s_2}) s_2+e^{s_1} (-1+e^{s_2}) (-1+5 e^{s_1+s_2}) s_3) s_1^3-((11+16 e^{s_1}-21 e^{s_1+s_2}-21 e^{2 s_1+s_2}+6 e^{2 s_1+2 s_2}+9 e^{3 s_1+2 s_2}) s_2^2+(2 (-7+5 e^{s_1}+12 e^{s_1+s_2}-8 e^{2 s_1+s_2}-5 e^{2 s_1+2 s_2}+3 e^{3 s_1+2 s_2})+e^{s_1} (15+3 e^{s_2}-13 e^{s_1+s_2}-15 e^{s_1+2 s_2}+10 e^{2 s_1+2 s_2}) s_3) s_2-2 e^{s_1} (-1+e^{s_2}) s_3 ((-1+5 e^{s_1+s_2}) s_3-2 e^{s_1+s_2}+2)) s_1^2+((11-38 e^{s_1}-21 e^{s_1+s_2}+63 e^{2 s_1+s_2}+6 e^{2 s_1+2 s_2}-21 e^{3 s_1+2 s_2}) s_2^3+(8 e^{s_1} (-1+e^{s_1+s_2}){}^2+(21-54 e^{s_1}-41 e^{s_1+s_2}+79 e^{2 s_1+s_2}+16 e^{2 s_1+2 s_2}-21 e^{3 s_1+2 s_2}) s_3) s_2^2+s_3 (2 (-7+5 e^{s_1}+16 e^{s_1+s_2}-12 e^{2 s_1+s_2}-9 e^{2 s_1+2 s_2}+7 e^{3 s_1+2 s_2})+(10-15 e^{s_1}-21 e^{s_1+s_2}+11 e^{2 s_1+s_2}+15 e^{2 s_1+2 s_2}) s_3) s_2+e^{s_1} (-1+e^{s_2}) s_3^2 ((-1+5 e^{s_1+s_2}) s_3-4 e^{s_1+s_2}+4)) s_1-s_2 (s_2+s_3) ((-11+20 e^{s_1}+21 e^{s_1+s_2}-35 e^{2 s_1+s_2}-6 e^{2 s_1+2 s_2}+11 e^{3 s_1+2 s_2}) s_2^2-2 ((5-9 e^{s_1}-9 e^{s_1+s_2}+13 e^{2 s_1+s_2}) s_3+9 e^{s_1}+12 e^{s_1+s_2}-16 e^{2 s_1+s_2}-5 e^{2 s_1+2 s_2}+7 e^{3 s_1+2 s_2}-7) s_2-s_3 ((-1+2 e^{s_1}+3 e^{s_1+s_2}-9 e^{2 s_1+s_2}-6 e^{2 s_1+2 s_2}+11 e^{3 s_1+2 s_2}) s_3-2 (-3+5 e^{s_1}+8 e^{s_1+s_2}-12 e^{2 s_1+s_2}-5 e^{2 s_1+2 s_2}+7 e^{3 s_1+2 s_2})))), 
$
\[
k_{18, 26}(s_1, s_2, s_3)=\frac{k_{18, 25}^{\text{num}}(s_1, s_2, s_3)}{\left(e^{s_1}-1\right) \left(e^{s_1+s_2}-1\right){}^2 \left(e^{\frac{1}{2} \left(s_1+s_2+s_3\right)}+1\right){}^2 s_1 s_2 \left(s_1+s_2\right) s_3 \left(s_2+s_3\right) \left(s_1+s_2+s_3\right){}^2}
\]
\[
k_{18, 27}(s_1, s_2, s_3)3 \left(e^{s_1}-1\right){}^3 \left(e^{s_1+s_2}-1\right){}^3 \left(e^{\frac{1}{2} \left(s_1+s_2+s_3\right)}+1\right) s_1 s_2 \left(s_1+s_2\right) s_3 \left(s_2+s_3\right) \left(s_1+s_2+s_3\right){}^2=
\]
$
2 \pi  (3 ((-5+8 e^{s_1}+15 e^{2 s_1}-10 e^{3 s_1}+16 e^{s_1+s_2}-23 e^{2 s_1+s_2}-42 e^{3 s_1+s_2}+25 e^{4 s_1+s_2}-21 e^{2 s_1+2 s_2}+42 e^{3 s_1+2 s_2}+11 e^{4 s_1+2 s_2}-8 e^{5 s_1+2 s_2}+2 e^{3 s_1+3 s_2}-3 e^{4 s_1+3 s_2}-8 e^{5 s_1+3 s_2}+e^{6 s_1+3 s_2}) s_2+e^{s_1} (-1+2 e^{s_1}-9 e^{2 s_1}+e^{s_2}+4 e^{s_1+s_2}-3 e^{2 s_1+s_2}+22 e^{3 s_1+s_2}-6 e^{s_1+2 s_2}+15 e^{2 s_1+2 s_2}-28 e^{3 s_1+2 s_2}-5 e^{4 s_1+2 s_2}-3 e^{2 s_1+3 s_2}+6 e^{3 s_1+3 s_2}+5 e^{4 s_1+3 s_2}) s_3) s_1^3+(3 (-5-6 e^{s_1}+75 e^{2 s_1}-40 e^{3 s_1}+16 e^{s_1+s_2}+27 e^{2 s_1+s_2}-222 e^{3 s_1+s_2}+107 e^{4 s_1+s_2}-21 e^{2 s_1+2 s_2}+159 e^{4 s_1+2 s_2}-66 e^{5 s_1+2 s_2}+2 e^{3 s_1+3 s_2}+3 e^{4 s_1+3 s_2}-36 e^{5 s_1+3 s_2}+7 e^{6 s_1+3 s_2}) s_2^2+(3 e^{s_1} (-17+66 e^{s_1}-57 e^{2 s_1}+3 e^{s_2}+62 e^{s_1+s_2}-189 e^{2 s_1+s_2}+148 e^{3 s_1+s_2}-18 e^{s_1+2 s_2}+3 e^{2 s_1+2 s_2}+64 e^{3 s_1+2 s_2}-73 e^{4 s_1+2 s_2}-9 e^{2 s_1+3 s_2}+24 e^{3 s_1+3 s_2}-13 e^{4 s_1+3 s_2}+6 e^{5 s_1+3 s_2}) s_3-2 (-1+e^{s_1}){}^2 (-23-e^{s_1}+69 e^{s_1+s_2}-9 e^{2 s_1+s_2}-57 e^{2 s_1+2 s_2}+9 e^{3 s_1+2 s_2}+11 e^{3 s_1+3 s_2}+e^{4 s_1+3 s_2})) s_2+6 e^{s_1} (-1+e^{s_2}) s_3 (4 e^{s_1+s_2} (-1+e^{s_1+s_2}) (-1+e^{s_1}){}^2+(1-2 e^{s_1}+9 e^{2 s_1}-6 e^{s_1+s_2}+12 e^{2 s_1+s_2}-22 e^{3 s_1+s_2}-3 e^{2 s_1+2 s_2}+6 e^{3 s_1+2 s_2}+5 e^{4 s_1+2 s_2}) s_3)) s_1^2+(3 (5-36 e^{s_1}+105 e^{2 s_1}-50 e^{3 s_1}-16 e^{s_1+s_2}+123 e^{2 s_1+s_2}-318 e^{3 s_1+s_2}+139 e^{4 s_1+s_2}+21 e^{2 s_1+2 s_2}-126 e^{3 s_1+2 s_2}+285 e^{4 s_1+2 s_2}-108 e^{5 s_1+2 s_2}-2 e^{3 s_1+3 s_2}+15 e^{4 s_1+3 s_2}-48 e^{5 s_1+3 s_2}+11 e^{6 s_1+3 s_2}) s_2^3+3 ((11-64 e^{s_1}+159 e^{2 s_1}-98 e^{3 s_1}-34 e^{s_1+s_2}+223 e^{2 s_1+s_2}-480 e^{3 s_1+s_2}+267 e^{4 s_1+s_2}+39 e^{2 s_1+2 s_2}-210 e^{3 s_1+2 s_2}+383 e^{4 s_1+2 s_2}-188 e^{5 s_1+2 s_2}-8 e^{3 s_1+3 s_2}+27 e^{4 s_1+3 s_2}-38 e^{5 s_1+3 s_2}+11 e^{6 s_1+3 s_2}) s_3-16 e^{s_1} (-1+e^{s_1}){}^2 (-2+e^{s_1+s_2}) (-1+e^{s_1+s_2}){}^2) s_2^2+s_3 (3 (6-29 e^{s_1}+56 e^{2 s_1}-57 e^{3 s_1}-17 e^{s_1+s_2}+104 e^{2 s_1+s_2}-165 e^{3 s_1+s_2}+150 e^{4 s_1+s_2}+12 e^{2 s_1+2 s_2}-69 e^{3 s_1+2 s_2}+70 e^{4 s_1+2 s_2}-85 e^{5 s_1+2 s_2}-9 e^{3 s_1+3 s_2}+18 e^{4 s_1+3 s_2}+15 e^{5 s_1+3 s_2}) s_3-2 (-1+e^{s_1}){}^2 (23-47 e^{s_1}-69 e^{s_1+s_2}+105 e^{2 s_1+s_2}+81 e^{2 s_1+2 s_2}-81 e^{3 s_1+2 s_2}-35 e^{3 s_1+3 s_2}+23 e^{4 s_1+3 s_2})) s_2+3 e^{s_1} (-1+e^{s_2}) s_3^2 (8 e^{s_1+s_2} (-1+e^{s_1+s_2}) (-1+e^{s_1}){}^2+(1-2 e^{s_1}+9 e^{2 s_1}-6 e^{s_1+s_2}+12 e^{2 s_1+s_2}-22 e^{3 s_1+s_2}-3 e^{2 s_1+2 s_2}+6 e^{3 s_1+2 s_2}+5 e^{4 s_1+2 s_2}) s_3)) s_1+s_2 (s_2+s_3) (3 (5-22 e^{s_1}+45 e^{2 s_1}-20 e^{3 s_1}-16 e^{s_1+s_2}+73 e^{2 s_1+s_2}-138 e^{3 s_1+s_2}+57 e^{4 s_1+s_2}+21 e^{2 s_1+2 s_2}-84 e^{3 s_1+2 s_2}+137 e^{4 s_1+2 s_2}-50 e^{5 s_1+2 s_2}-2 e^{3 s_1+3 s_2}+9 e^{4 s_1+3 s_2}-20 e^{5 s_1+3 s_2}+5 e^{6 s_1+3 s_2}) s_2^2-2 (-1+e^{s_1}) ((-1+e^{s_1}) (23-47 e^{s_1}-69 e^{s_1+s_2}+129 e^{2 s_1+s_2}+57 e^{2 s_1+2 s_2}-105 e^{3 s_1+2 s_2}-11 e^{3 s_1+3 s_2}+23 e^{4 s_1+3 s_2})+3 (3-10 e^{s_1}+15 e^{2 s_1}-10 e^{s_1+s_2}+36 e^{2 s_1+s_2}-42 e^{3 s_1+s_2}+15 e^{2 s_1+2 s_2}-42 e^{3 s_1+2 s_2}+35 e^{4 s_1+2 s_2}) s_3) s_2+s_3 (2 (-1+e^{s_1}){}^2 (-1+e^{s_1}+3 e^{s_1+s_2}+9 e^{2 s_1+s_2}+9 e^{2 s_1+2 s_2}-33 e^{3 s_1+2 s_2}-11 e^{3 s_1+3 s_2}+23 e^{4 s_1+3 s_2})-3 (-1+4 e^{s_1}-5 e^{2 s_1}+10 e^{3 s_1}+4 e^{s_1+s_2}-19 e^{2 s_1+s_2}+18 e^{3 s_1+s_2}-27 e^{4 s_1+s_2}-9 e^{2 s_1+2 s_2}+30 e^{3 s_1+2 s_2}-17 e^{4 s_1+2 s_2}+20 e^{5 s_1+2 s_2}-2 e^{3 s_1+3 s_2}+9 e^{4 s_1+3 s_2}-20 e^{5 s_1+3 s_2}+5 e^{6 s_1+3 s_2}) s_3))), 
$
\[
k_{18, 28}(s_1, s_2, s_3)=\frac{-k_{18, 27}^{\text{num}}(s_1, s_2, s_3)}{3 \left(e^{s_1}-1\right){}^3 \left(e^{s_1+s_2}-1\right){}^3 \left(e^{\frac{1}{2} \left(s_1+s_2+s_3\right)}-1\right) s_1 s_2 \left(s_1+s_2\right) s_3 \left(s_2+s_3\right) \left(s_1+s_2+s_3\right){}^2},
\]
\[
k_{18, 29}(s_1, s_2, s_3)=\frac{16 \pi  e^{2 \left(s_1+s_2\right)} \left(2 s_1-e^{s_2} \left(s_1-s_2+e^{s_1} \left(s_1+s_2\right)\right)\right)}{\left(e^{s_2}-1\right) \left(e^{s_1+s_2}-1\right){}^3 \left(e^{\frac{s_3}{2}}+1\right) s_1 s_2 \left(s_1+s_2\right) \left(s_1+s_2+s_3\right)}, 
\]
\[
k_{18, 30}(s_1, s_2, s_3)=\frac{16 \pi  e^{2 \left(s_1+s_2\right)} \left(e^{s_2} \left(s_1-s_2+e^{s_1} \left(s_1+s_2\right)\right)-2 s_1\right)}{\left(e^{s_2}-1\right) \left(e^{s_1+s_2}-1\right){}^3 \left(e^{\frac{s_3}{2}}-1\right) s_1 s_2 \left(s_1+s_2\right) \left(s_1+s_2+s_3\right)}. 
\]
}

\subsubsection{The function $k_{20}$} We have 
\[
k_{20}(s_1, s_2, s_3)= 
K_{20}(s_1, s_2, s_3, -s_1-s_2-s_3)
=\sum_{i=1}^{24} k_{20, i}(s_1, s_2, s_3), 
\]
where 
{\tiny 
\[
k_{20, 1}(s_1, s_2, s_3)=-\frac{16 \pi  \left(-3 e^{s_1}+33 e^{2 s_1}+5 e^{3 s_1}+1\right) s_2 \left(s_2+s_3\right)}{\left(e^{s_1}-1\right){}^3 s_1^2 \left(s_1+s_2\right){}^2 \left(s_1+s_2+s_3\right){}^2}, 
\]
\[
k_{20, 2}(s_1, s_2, s_3)\left(e^{s_1}-1\right){}^2 \left(e^{s_2}-1\right) \left(e^{s_1+s_2}-1\right){}^4 s_1 \left(s_1+s_2\right){}^2 s_3 \left(s_2+s_3\right) \left(s_1+s_2+s_3\right){}^2=
\]
$
-16 \pi  s_2^3 (2 e^{2 s_1+3 s_2} (e^{s_1+s_2}-4) s_2 (e^{s_1}-1){}^3-9 e^{s_1}+e^{2 s_1}-2 e^{s_2}-2 e^{s_1+s_2}+30 e^{2 s_1+s_2}-2 e^{3 s_1+s_2}+11 e^{s_1+2 s_2}-21 e^{2 s_1+2 s_2}-21 e^{3 s_1+2 s_2}-5 e^{4 s_1+2 s_2}-10 e^{2 s_1+3 s_2}+22 e^{3 s_1+3 s_2}+6 e^{4 s_1+3 s_2}+6 e^{5 s_1+3 s_2}+e^{3 s_1+4 s_2}-e^{4 s_1+4 s_2}-6 e^{5 s_1+4 s_2}+2), 
$
\[
k_{20, 3}(s_1, s_2, s_3)\left(e^{s_1}-1\right){}^3 \left(e^{s_2}-1\right) \left(e^{s_1+s_2}-1\right){}^4 \left(s_1+s_2\right){}^2 s_3 \left(s_2+s_3\right) \left(s_1+s_2+s_3\right){}^2=
\]
$
-16 \pi  (2 e^{2 s_1} (-1+e^{s_2}) (-4+e^{s_1}-4 e^{s_2}+2 e^{2 s_2}+17 e^{s_1+s_2}-4 e^{2 s_1+s_2}-e^{s_1+2 s_2}-10 e^{2 s_1+2 s_2}+e^{s_1+3 s_2}-4 e^{2 s_1+3 s_2}+6 e^{3 s_1+3 s_2}) s_1^3+((-1+e^{s_1}) (-4+3 e^{s_1}-5 e^{2 s_1}+4 e^{s_2}+14 e^{s_1+s_2}-20 e^{2 s_1+s_2}+26 e^{3 s_1+s_2}-17 e^{s_1+2 s_2}+5 e^{2 s_1+2 s_2}+11 e^{3 s_1+2 s_2}-35 e^{4 s_1+2 s_2}+20 e^{2 s_1+3 s_2}-26 e^{3 s_1+3 s_2}+12 e^{4 s_1+3 s_2}+18 e^{5 s_1+3 s_2}-11 e^{3 s_1+4 s_2}+19 e^{4 s_1+4 s_2}-10 e^{5 s_1+4 s_2}-4 e^{6 s_1+4 s_2}+4 e^{4 s_1+5 s_2}-8 e^{5 s_1+5 s_2}+4 e^{6 s_1+5 s_2})+2 e^{2 s_1} (12-3 e^{s_1}-18 e^{2 s_2}+2 e^{3 s_2}-48 e^{s_1+s_2}+12 e^{2 s_1+s_2}+54 e^{s_1+2 s_2}+18 e^{2 s_1+2 s_2}+10 e^{s_1+3 s_2}-42 e^{2 s_1+3 s_2}-2 e^{3 s_1+3 s_2}-4 e^{4 s_1+3 s_2}+4 e^{s_1+4 s_2}-16 e^{2 s_1+4 s_2}+24 e^{3 s_1+4 s_2}-4 e^{4 s_1+4 s_2}+e^{5 s_1+4 s_2}) s_2) s_1^2+(-1+e^{s_1+s_2}) s_2 ((-1+e^{s_1}) (6+3 e^{s_1}+9 e^{2 s_1}-6 e^{s_2}-20 e^{s_1+s_2}+13 e^{2 s_1+s_2}-41 e^{3 s_1+s_2}+17 e^{s_1+2 s_2}-9 e^{2 s_1+2 s_2}+12 e^{3 s_1+2 s_2}+34 e^{4 s_1+2 s_2}-13 e^{2 s_1+3 s_2}+21 e^{3 s_1+3 s_2}-18 e^{4 s_1+3 s_2}-8 e^{5 s_1+3 s_2}+8 e^{3 s_1+4 s_2}-16 e^{4 s_1+4 s_2}+8 e^{5 s_1+4 s_2})+6 e^{2 s_1} (-4+e^{s_1}+6 e^{2 s_2}+2 e^{3 s_2}+12 e^{s_1+s_2}-3 e^{2 s_1+s_2}-18 e^{s_1+2 s_2}+6 e^{2 s_1+2 s_2}-3 e^{3 s_1+2 s_2}-8 e^{s_1+3 s_2}+12 e^{2 s_1+3 s_2}-4 e^{3 s_1+3 s_2}+e^{4 s_1+3 s_2}) s_2) s_1+e^{s_1} s_2^2 ((-1+e^{s_1}) (-1+e^{s_2}) (15+3 e^{s_1}+5 e^{s_2}-37 e^{s_1+s_2}-22 e^{2 s_1+s_2}+9 e^{2 s_1+2 s_2}+45 e^{3 s_1+2 s_2}-9 e^{2 s_1+3 s_2}+21 e^{3 s_1+3 s_2}-30 e^{4 s_1+3 s_2}+4 e^{3 s_1+4 s_2}-8 e^{4 s_1+4 s_2}+4 e^{5 s_1+4 s_2})+2 e^{s_1} (4-e^{s_1}-6 e^{2 s_2}-10 e^{3 s_2}-16 e^{s_1+s_2}+4 e^{2 s_1+s_2}+18 e^{s_1+2 s_2}+6 e^{2 s_1+2 s_2}+46 e^{s_1+3 s_2}-78 e^{2 s_1+3 s_2}+42 e^{3 s_1+3 s_2}-12 e^{4 s_1+3 s_2}+4 e^{s_1+4 s_2}-16 e^{2 s_1+4 s_2}+24 e^{3 s_1+4 s_2}-12 e^{4 s_1+4 s_2}+3 e^{5 s_1+4 s_2}) s_2)),
$
\[
k_{20, 4}(s_1, s_2, s_3)3 \left(e^{s_1}-1\right){}^4 \left(e^{s_2}-1\right) \left(e^{s_1+s_2}-1\right){}^4 \left(s_1+s_2\right){}^2 \left(s_2+s_3\right) \left(s_1+s_2+s_3\right){}^2=
\]
$
16 \pi  ((139 e^{s_1}+397 e^{3 s_1}+29 e^{4 s_1}+22 e^{s_2}+152 e^{2 s_1+s_2}+1316 e^{3 s_1+s_2}-182 e^{5 s_1+s_2}-109 e^{s_1+2 s_2}+277 e^{2 s_1+2 s_2}-632 e^{4 s_1+2 s_2}+1709 e^{5 s_1+2 s_2}-170 e^{3 s_1+3 s_2}+2072 e^{4 s_1+3 s_2}-1258 e^{6 s_1+3 s_2}-134 e^{7 s_1+3 s_2}-7 e^{3 s_1+4 s_2}-1019 e^{5 s_1+4 s_2}+775 e^{6 s_1+4 s_2}-4 e^{8 s_1+4 s_2}-20 e^{4 s_1+5 s_2}+116 e^{5 s_1+5 s_2}-292 e^{7 s_1+5 s_2}+4 e^{8 s_1+5 s_2}-22) s_2-2 (-32 e^{s_1}-164 e^{3 s_1}-28 e^{4 s_1}-5 e^{s_2}-187 e^{2 s_1+s_2}-496 e^{3 s_1+s_2}+136 e^{5 s_1+s_2}+20 e^{s_1+2 s_2}-11 e^{2 s_1+2 s_2}+316 e^{4 s_1+2 s_2}-868 e^{5 s_1+2 s_2}-44 e^{3 s_1+3 s_2}-988 e^{4 s_1+3 s_2}+689 e^{6 s_1+3 s_2}+76 e^{7 s_1+3 s_2}-16 e^{3 s_1+4 s_2}+544 e^{5 s_1+4 s_2}-455 e^{6 s_1+4 s_2}-7 e^{8 s_1+4 s_2}+22 e^{4 s_1+5 s_2}-100 e^{5 s_1+5 s_2}+128 e^{7 s_1+5 s_2}+7 e^{8 s_1+5 s_2}+5) s_3),
$
\[
k_{20, 5}(s_1, s_2, s_3)\left(e^{s_1}-1\right){}^4 \left(e^{s_2}-1\right) \left(e^{s_1+s_2}-1\right){}^4 \left(s_1+s_2\right){}^2 \left(s_2+s_3\right) \left(s_1+s_2+s_3\right){}^2=
\]
$
-16 \pi  (2 e^{2 s_1} (4+30 e^{s_1}-4 e^{2 s_1}-24 e^{s_2}+30 e^{2 s_2}-6 e^{3 s_2}-24 e^{s_1+s_2}-118 e^{2 s_1+s_2}+16 e^{3 s_1+s_2}-24 e^{s_1+2 s_2}+236 e^{2 s_1+2 s_2}+52 e^{3 s_1+2 s_2}+6 e^{4 s_1+2 s_2}+4 e^{s_1+3 s_2}-100 e^{2 s_1+3 s_2}-160 e^{3 s_1+3 s_2}-34 e^{4 s_1+3 s_2}-4 e^{5 s_1+3 s_2}-7 e^{s_1+4 s_2}+33 e^{2 s_1+4 s_2}+34 e^{3 s_1+4 s_2}+94 e^{4 s_1+4 s_2}-5 e^{5 s_1+4 s_2}+e^{6 s_1+4 s_2}-2 e^{2 s_1+5 s_2}+8 e^{3 s_1+5 s_2}-36 e^{4 s_1+5 s_2}) s_1^2-2 ((-1+e^{s_1}) (1-8 e^{s_1}+36 e^{2 s_1}+e^{3 s_1}-e^{s_2}+5 e^{s_1+s_2}-36 e^{2 s_1+s_2}-102 e^{3 s_1+s_2}-16 e^{4 s_1+s_2}+3 e^{s_1+2 s_2}+11 e^{2 s_1+2 s_2}+107 e^{3 s_1+2 s_2}+163 e^{4 s_1+2 s_2}+16 e^{5 s_1+2 s_2}-11 e^{2 s_1+3 s_2}-11 e^{3 s_1+3 s_2}-163 e^{4 s_1+3 s_2}-115 e^{5 s_1+3 s_2}+5 e^{3 s_1+4 s_2}+12 e^{4 s_1+4 s_2}+117 e^{5 s_1+4 s_2}+17 e^{6 s_1+4 s_2}-e^{7 s_1+4 s_2}+4 e^{4 s_1+5 s_2}-18 e^{5 s_1+5 s_2}-17 e^{6 s_1+5 s_2}+e^{7 s_1+5 s_2})-e^{2 s_1+s_2} (-32-68 e^{s_1}-218 e^{2 s_1}+38 e^{s_2}-24 e^{s_1+s_2}+420 e^{2 s_1+s_2}+150 e^{3 s_1+s_2}-56 e^{s_1+2 s_2}-28 e^{2 s_1+2 s_2}-464 e^{3 s_1+2 s_2}-34 e^{4 s_1+2 s_2}+45 e^{2 s_1+3 s_2}+58 e^{3 s_1+3 s_2}+202 e^{4 s_1+3 s_2}+2 e^{5 s_1+3 s_2}) s_2-e^{2 s_1+s_2} (-32-108 e^{s_1}-196 e^{2 s_1}+26 e^{s_2}+24 e^{s_1+s_2}+428 e^{2 s_1+s_2}+160 e^{3 s_1+s_2}-172 e^{2 s_1+2 s_2}-368 e^{3 s_1+2 s_2}-106 e^{4 s_1+2 s_2}+23 e^{2 s_1+3 s_2}+106 e^{3 s_1+3 s_2}+190 e^{4 s_1+3 s_2}+17 e^{5 s_1+3 s_2}) s_3) s_1+e^{s_1} (2 e^{s_1+s_2} (-16-96 e^{s_1}-162 e^{2 s_1}+10 e^{s_2}+24 e^{s_1+s_2}+324 e^{2 s_1+s_2}+168 e^{3 s_1+s_2}-108 e^{s_1+2 s_2}+116 e^{2 s_1+2 s_2}-544 e^{3 s_1+2 s_2}+10 e^{4 s_1+2 s_2}+11 e^{2 s_1+3 s_2}+70 e^{3 s_1+3 s_2}+170 e^{4 s_1+3 s_2}+19 e^{5 s_1+3 s_2}) s_2^2+(2 e^{s_1+s_2} (-24-172 e^{s_1}-222 e^{2 s_1}+6 e^{s_2}+72 e^{s_1+s_2}+484 e^{2 s_1+s_2}+254 e^{3 s_1+s_2}-56 e^{s_1+2 s_2}-60 e^{2 s_1+2 s_2}-608 e^{3 s_1+2 s_2}-98 e^{4 s_1+2 s_2}-17 e^{2 s_1+3 s_2}+174 e^{3 s_1+3 s_2}+206 e^{4 s_1+3 s_2}+56 e^{5 s_1+3 s_2}) s_3+181 e^{s_1}+10 e^{s_2}+426 e^{3 s_1+s_2}+512 e^{2 s_1+2 s_2}-97 e^{5 s_1+2 s_2}-38 e^{s_1+3 s_2}+208 e^{4 s_1+3 s_2}+57 e^{3 s_1+4 s_2}-142 e^{6 s_1+4 s_2}-64 e^{5 s_1+5 s_2}) s_2+2 s_3 (2 e^{s_1+s_2} (-4-30 e^{s_1}-40 e^{2 s_1}+e^{s_2}+12 e^{s_1+s_2}+86 e^{2 s_1+s_2}+46 e^{3 s_1+s_2}-2 e^{s_1+2 s_2}-26 e^{2 s_1+2 s_2}-92 e^{3 s_1+2 s_2}-27 e^{4 s_1+2 s_2}-4 e^{2 s_1+3 s_2}+32 e^{3 s_1+3 s_2}+36 e^{4 s_1+3 s_2}+11 e^{5 s_1+3 s_2}) s_3+73 e^{s_1}+4 e^{s_2}+180 e^{3 s_1+s_2}+240 e^{2 s_1+2 s_2}-59 e^{5 s_1+2 s_2}-7 e^{s_1+3 s_2}+96 e^{4 s_1+3 s_2}+46 e^{3 s_1+4 s_2}-68 e^{6 s_1+4 s_2}-19 e^{5 s_1+5 s_2}))), 
$
\[
k_{20, 6}(s_1, s_2, s_3)3 \left(e^{s_1}-1\right){}^4 \left(e^{s_2}-1\right) \left(e^{s_1+s_2}-1\right){}^4 s_1 \left(s_1+s_2\right){}^2 \left(s_2+s_3\right) \left(s_1+s_2+s_3\right){}^2=
\]
$
-16 \pi  ((19-100 e^{s_1}-316 e^{3 s_1}-35 e^{4 s_1}-19 e^{s_2}-44 e^{2 s_1+s_2}-1304 e^{3 s_1+s_2}+152 e^{5 s_1+s_2}+100 e^{s_1+2 s_2}-310 e^{2 s_1+2 s_2}+656 e^{4 s_1+2 s_2}-1580 e^{5 s_1+2 s_2}+200 e^{3 s_1+3 s_2}-1904 e^{4 s_1+3 s_2}+1150 e^{6 s_1+3 s_2}+176 e^{7 s_1+3 s_2}+4 e^{3 s_1+4 s_2}+956 e^{5 s_1+4 s_2}-544 e^{6 s_1+4 s_2}-11 e^{8 s_1+4 s_2}-7 e^{4 s_1+5 s_2}+16 e^{5 s_1+5 s_2}+304 e^{7 s_1+5 s_2}+11 e^{8 s_1+5 s_2}) s_2^2+(20-95 e^{s_1}-425 e^{3 s_1}-67 e^{4 s_1}-20 e^{s_2}-220 e^{2 s_1+s_2}-1708 e^{3 s_1+s_2}+274 e^{5 s_1+s_2}+101 e^{s_1+2 s_2}-281 e^{2 s_1+2 s_2}+952 e^{4 s_1+2 s_2}-2365 e^{5 s_1+2 s_2}+94 e^{3 s_1+3 s_2}-2680 e^{4 s_1+3 s_2}+1850 e^{6 s_1+3 s_2}+250 e^{7 s_1+3 s_2}-e^{3 s_1+4 s_2}+1495 e^{5 s_1+4 s_2}-863 e^{6 s_1+4 s_2}-10 e^{8 s_1+4 s_2}-14 e^{4 s_1+5 s_2}+44 e^{5 s_1+5 s_2}+500 e^{7 s_1+5 s_2}+10 e^{8 s_1+5 s_2}) s_3 s_2-(-7+34 e^{s_1}+142 e^{3 s_1}+29 e^{4 s_1}+7 e^{s_2}+80 e^{2 s_1+s_2}+596 e^{3 s_1+s_2}-116 e^{5 s_1+s_2}-34 e^{s_1+2 s_2}+100 e^{2 s_1+2 s_2}-344 e^{4 s_1+2 s_2}+818 e^{5 s_1+2 s_2}-20 e^{3 s_1+3 s_2}+920 e^{4 s_1+3 s_2}-682 e^{6 s_1+3 s_2}-92 e^{7 s_1+3 s_2}+2 e^{3 s_1+4 s_2}-530 e^{5 s_1+4 s_2}+286 e^{6 s_1+4 s_2}-e^{8 s_1+4 s_2}+7 e^{4 s_1+5 s_2}-28 e^{5 s_1+5 s_2}-196 e^{7 s_1+5 s_2}+e^{8 s_1+5 s_2}) s_3^2), 
$
\[
k_{20, 7}(s_1, s_2, s_3)5 \left(e^{s_1}-1\right){}^4 \left(e^{s_2}-1\right) \left(e^{s_1+s_2}-1\right){}^4 \left(s_1+s_2\right){}^2 \left(s_2+s_3\right) \left(s_1+s_2+s_3\right){}^2=
\]
$
32 \pi  ((-4+16 e^{s_1}-84 e^{2 s_1}-204 e^{3 s_1}+6 e^{4 s_1}+4 e^{s_2}-24 e^{5 s_1+s_2}-16 e^{s_1+2 s_2}-54 e^{6 s_1+2 s_2}-86 e^{2 s_1+3 s_2}+156 e^{7 s_1+3 s_2}+9 e^{3 s_1+4 s_2}-9 e^{8 s_1+4 s_2}+14 e^{4 s_1+5 s_2}-56 e^{5 s_1+5 s_2}+204 e^{6 s_1+5 s_2}+104 e^{7 s_1+5 s_2}+4 e^{8 s_1+5 s_2}) s_2^2+(-6+24 e^{s_1}-156 e^{2 s_1}-281 e^{3 s_1}-e^{4 s_1}+6 e^{s_2}+4 e^{5 s_1+s_2}-24 e^{s_1+2 s_2}-96 e^{6 s_1+2 s_2}-34 e^{2 s_1+3 s_2}+124 e^{7 s_1+3 s_2}-19 e^{3 s_1+4 s_2}+14 e^{8 s_1+4 s_2}+26 e^{4 s_1+5 s_2}-104 e^{5 s_1+5 s_2}+336 e^{6 s_1+5 s_2}+156 e^{7 s_1+5 s_2}+6 e^{8 s_1+5 s_2}) s_3 s_2+2 (-1+4 e^{s_1}-26 e^{2 s_1}-51 e^{3 s_1}-e^{4 s_1}+e^{s_2}+4 e^{5 s_1+s_2}-4 e^{s_1+2 s_2}-21 e^{6 s_1+2 s_2}+e^{2 s_1+3 s_2}+14 e^{7 s_1+3 s_2}-4 e^{3 s_1+4 s_2}+4 e^{8 s_1+4 s_2}+6 e^{4 s_1+5 s_2}-24 e^{5 s_1+5 s_2}+66 e^{6 s_1+5 s_2}+26 e^{7 s_1+5 s_2}+e^{8 s_1+5 s_2}) s_3^2+s_1 ((-2+8 e^{s_1}-52 e^{2 s_1}-277 e^{3 s_1}+23 e^{4 s_1}+2 e^{s_2}-92 e^{5 s_1+s_2}-8 e^{s_1+2 s_2}-72 e^{6 s_1+2 s_2}-18 e^{2 s_1+3 s_2}+108 e^{7 s_1+3 s_2}+47 e^{3 s_1+4 s_2}-12 e^{8 s_1+4 s_2}+22 e^{4 s_1+5 s_2}-88 e^{5 s_1+5 s_2}+312 e^{6 s_1+5 s_2}+52 e^{7 s_1+5 s_2}+2 e^{8 s_1+5 s_2}) s_2+(-2+8 e^{s_1}-92 e^{2 s_1}-252 e^{3 s_1}+8 e^{4 s_1}+2 e^{s_2}-32 e^{5 s_1+s_2}-8 e^{s_1+2 s_2}-102 e^{6 s_1+2 s_2}+22 e^{2 s_1+3 s_2}+48 e^{7 s_1+3 s_2}+27 e^{3 s_1+4 s_2}+3 e^{8 s_1+4 s_2}+32 e^{4 s_1+5 s_2}-128 e^{5 s_1+5 s_2}+372 e^{6 s_1+5 s_2}+52 e^{7 s_1+5 s_2}+2 e^{8 s_1+5 s_2}) s_3)),
$
\[
k_{20, 8}(s_1, s_2, s_3)\left(e^{s_1}-1\right){}^4 \left(e^{s_1+s_2}-1\right){}^4 s_2 \left(s_1+s_2\right){}^2 \left(s_2+s_3\right) \left(s_1+s_2+s_3\right){}^2=
\]
$
16 e^{s_1} \pi  (2 e^{s_1} (-1+e^{s_2}) (-2-5 e^{s_1}+e^{2 s_1}+2 e^{s_2}-2 e^{s_1+s_2}+22 e^{2 s_1+s_2}-4 e^{3 s_1+s_2}+e^{s_1+2 s_2}-5 e^{2 s_1+2 s_2}-14 e^{3 s_1+2 s_2}+6 e^{4 s_1+3 s_2}) s_1^3+((-1+e^{s_1}) (1-14 e^{s_1}+e^{2 s_1}-e^{s_2}+10 e^{s_1+s_2}+37 e^{2 s_1+s_2}+2 e^{3 s_1+s_2}-4 e^{s_1+2 s_2}-9 e^{2 s_1+2 s_2}-60 e^{3 s_1+2 s_2}+e^{4 s_1+2 s_2}+5 e^{2 s_1+3 s_2}+2 e^{3 s_1+3 s_2}+47 e^{4 s_1+3 s_2}-6 e^{5 s_1+3 s_2}-4 e^{4 s_1+4 s_2}-10 e^{5 s_1+4 s_2}+2 e^{6 s_1+4 s_2})+2 e^{s_1} (-1+e^{s_2}) (-4-15 e^{s_1}+e^{2 s_1}+4 e^{s_2}-6 e^{s_1+s_2}+60 e^{2 s_1+s_2}-4 e^{3 s_1+s_2}+3 e^{s_1+2 s_2}-9 e^{2 s_1+2 s_2}-42 e^{3 s_1+2 s_2}-6 e^{4 s_1+2 s_2}+2 e^{2 s_1+3 s_2}-8 e^{3 s_1+3 s_2}+24 e^{4 s_1+3 s_2}) s_3) s_1^2+s_3 (4 e^{s_1} (-1+e^{s_2}) (-1-5 e^{s_1}+e^{s_2}-2 e^{s_1+s_2}+19 e^{2 s_1+s_2}+e^{s_1+2 s_2}-2 e^{2 s_1+2 s_2}-14 e^{3 s_1+2 s_2}-3 e^{4 s_1+2 s_2}+e^{2 s_1+3 s_2}-4 e^{3 s_1+3 s_2}+9 e^{4 s_1+3 s_2}) s_3-(-1+e^{s_1}) (-3+34 e^{s_1}+5 e^{2 s_1}+3 e^{s_2}-22 e^{s_1+s_2}-95 e^{2 s_1+s_2}-30 e^{3 s_1+s_2}+4 e^{s_1+2 s_2}+27 e^{2 s_1+2 s_2}+156 e^{3 s_1+2 s_2}+29 e^{4 s_1+2 s_2}+e^{2 s_1+3 s_2}-30 e^{3 s_1+3 s_2}-109 e^{4 s_1+3 s_2}-6 e^{5 s_1+3 s_2}-8 e^{3 s_1+4 s_2}+28 e^{4 s_1+4 s_2}+14 e^{5 s_1+4 s_2}+2 e^{6 s_1+4 s_2})) s_1-2 (-1+e^{s_1}) (-1+10 e^{s_1}+3 e^{2 s_1}+e^{s_2}-6 e^{s_1+s_2}-29 e^{2 s_1+s_2}-14 e^{3 s_1+s_2}+9 e^{2 s_1+2 s_2}+48 e^{3 s_1+2 s_2}+15 e^{4 s_1+2 s_2}+3 e^{2 s_1+3 s_2}-14 e^{3 s_1+3 s_2}-31 e^{4 s_1+3 s_2}-6 e^{5 s_1+3 s_2}-4 e^{3 s_1+4 s_2}+12 e^{4 s_1+4 s_2}+2 e^{5 s_1+4 s_2}+2 e^{6 s_1+4 s_2}) s_3^2),
$
\[
k_{20, 9}(s_1, s_2, s_3)\left(e^{s_1}-1\right){}^4 \left(e^{s_2}-1\right) \left(e^{s_1+s_2}-1\right){}^4 s_1 \left(s_1+s_2\right){}^2 \left(s_2+s_3\right) \left(s_1+s_2+s_3\right){}^2=
\]
$
-16 e^{s_1} \pi  (4 e^{s_1+s_2} (-1-22 e^{s_1}-21 e^{2 s_1}-2 e^{s_2}+12 e^{s_1+s_2}+46 e^{2 s_1+s_2}+32 e^{3 s_1+s_2}-26 e^{s_1+2 s_2}+38 e^{2 s_1+2 s_2}-108 e^{3 s_1+2 s_2}+8 e^{4 s_1+2 s_2}-3 e^{2 s_1+3 s_2}+16 e^{3 s_1+3 s_2}+25 e^{4 s_1+3 s_2}+6 e^{5 s_1+3 s_2}) s_2^3+e^{s_1} (2 e^{s_2} (-4-88 e^{s_1}-84 e^{2 s_1}-8 e^{s_2}+48 e^{s_1+s_2}+184 e^{2 s_1+s_2}+128 e^{3 s_1+s_2}-44 e^{s_1+2 s_2}+32 e^{2 s_1+2 s_2}-312 e^{3 s_1+2 s_2}-28 e^{4 s_1+2 s_2}-27 e^{2 s_1+3 s_2}+94 e^{3 s_1+3 s_2}+70 e^{4 s_1+3 s_2}+39 e^{5 s_1+3 s_2}) s_3-26 e^{3 s_2}+405 e^{2 s_1+s_2}+472 e^{s_1+2 s_2}-94 e^{4 s_1+2 s_2}+152 e^{3 s_1+3 s_2}+25 e^{2 s_1+4 s_2}-160 e^{5 s_1+4 s_2}-108 e^{4 s_1+5 s_2}+144) s_2^2+s_3 (4 e^{s_1+s_2} (-1-22 e^{s_1}-21 e^{2 s_1}-2 e^{s_2}+12 e^{s_1+s_2}+46 e^{2 s_1+s_2}+32 e^{3 s_1+s_2}-6 e^{s_1+2 s_2}-2 e^{2 s_1+2 s_2}-68 e^{3 s_1+2 s_2}-12 e^{4 s_1+2 s_2}-8 e^{2 s_1+3 s_2}+26 e^{3 s_1+3 s_2}+15 e^{4 s_1+3 s_2}+11 e^{5 s_1+3 s_2}) s_3+189 e^{s_1}-2 e^{s_2}+560 e^{3 s_1+s_2}+680 e^{2 s_1+2 s_2}-149 e^{5 s_1+2 s_2}-22 e^{s_1+3 s_2}+184 e^{4 s_1+3 s_2}+43 e^{3 s_1+4 s_2}-250 e^{6 s_1+4 s_2}-180 e^{5 s_1+5 s_2}) s_2-3 e^{s_1} (-22+2 e^{3 s_2}-63 e^{2 s_1+s_2}-80 e^{s_1+2 s_2}+20 e^{4 s_1+2 s_2}-16 e^{3 s_1+3 s_2}-5 e^{2 s_1+4 s_2}+32 e^{5 s_1+4 s_2}+24 e^{4 s_1+5 s_2}) s_3^2),
$
\[
k_{20, 10}(s_1, s_2, s_3)5 \left(e^{s_1}-1\right){}^4 \left(e^{s_2}-1\right) \left(e^{s_1+s_2}-1\right){}^4 s_1 \left(s_1+s_2\right){}^2 \left(s_2+s_3\right) \left(s_1+s_2+s_3\right){}^2=
\]
$
32 \pi  s_2 (2 (4 e^{s_1}-21 e^{2 s_1}-26 e^{3 s_1}-e^{4 s_1}+e^{s_2}+4 e^{5 s_1+s_2}-4 e^{s_1+2 s_2}-6 e^{6 s_1+2 s_2}-24 e^{2 s_1+3 s_2}+34 e^{7 s_1+3 s_2}-4 e^{3 s_1+4 s_2}-e^{8 s_1+4 s_2}+e^{4 s_1+5 s_2}-4 e^{5 s_1+5 s_2}+21 e^{6 s_1+5 s_2}+26 e^{7 s_1+5 s_2}+e^{8 s_1+5 s_2}-1) s_2^2+(16 e^{s_1}-84 e^{2 s_1}-104 e^{3 s_1}-4 e^{4 s_1}+4 e^{s_2}+16 e^{5 s_1+s_2}-16 e^{s_1+2 s_2}-24 e^{6 s_1+2 s_2}-36 e^{2 s_1+3 s_2}+76 e^{7 s_1+3 s_2}-31 e^{3 s_1+4 s_2}+11 e^{8 s_1+4 s_2}+4 e^{4 s_1+5 s_2}-16 e^{5 s_1+5 s_2}+84 e^{6 s_1+5 s_2}+104 e^{7 s_1+5 s_2}+4 e^{8 s_1+5 s_2}-4) s_3 s_2+2 (4 e^{s_1}-21 e^{2 s_1}-26 e^{3 s_1}-e^{4 s_1}+e^{s_2}+4 e^{5 s_1+s_2}-4 e^{s_1+2 s_2}-6 e^{6 s_1+2 s_2}-4 e^{2 s_1+3 s_2}+14 e^{7 s_1+3 s_2}-9 e^{3 s_1+4 s_2}+4 e^{8 s_1+4 s_2}+e^{4 s_1+5 s_2}-4 e^{5 s_1+5 s_2}+21 e^{6 s_1+5 s_2}+26 e^{7 s_1+5 s_2}+e^{8 s_1+5 s_2}-1) s_3^2),
$
\[
k_{20, 11}(s_1, s_2, s_3)=\frac{12 \pi  \left(s_1^2-s_3 s_1-s_2^2+s_3^2\right)}{\left(e^{\frac{1}{2} \left(s_1+s_2+s_3\right)}-1\right){}^4 s_1 \left(s_1+s_2\right) s_3 \left(s_2+s_3\right) \left(s_1+s_2+s_3\right)}, 
\]
\[
k_{20, 12}(s_1, s_2, s_3)=\frac{12 \pi  \left(s_1^2-s_3 s_1-s_2^2+s_3^2\right)}{\left(e^{\frac{1}{2} \left(s_1+s_2+s_3\right)}+1\right){}^4 s_1 \left(s_1+s_2\right) s_3 \left(s_2+s_3\right) \left(s_1+s_2+s_3\right)},
\]
\[
k_{20, 13}(s_1, s_2, s_3)=\frac{24 \pi  e^{2 s_1} \left(s_2-s_3\right)}{\left(e^{s_1}-1\right){}^2 \left(e^{\frac{1}{2} \left(s_2+s_3\right)}-1\right){}^2 s_1 s_2 s_3 \left(s_1+s_2+s_3\right)}, 
\]
\[
k_{20, 14}(s_1, s_2, s_3)=\frac{24 \pi  e^{2 s_1} \left(s_2-s_3\right)}{\left(e^{s_1}-1\right){}^2 \left(e^{\frac{1}{2} \left(s_2+s_3\right)}+1\right){}^2 s_1 s_2 s_3 \left(s_1+s_2+s_3\right)}, 
\]
\[
k_{20, 15}(s_1, s_2, s_3)\left(e^{s_1}-1\right) \left(e^{s_1+s_2}-1\right) \left(e^{\frac{1}{2} \left(s_1+s_2+s_3\right)}+1\right){}^3 s_1 s_2 \left(s_1+s_2\right) s_3 \left(s_2+s_3\right) \left(s_1+s_2+s_3\right){}^2=
\]
$
-4 \pi  (3 ((-3 e^{s_1}-3 e^{s_1+s_2}+2 e^{2 s_1+s_2}+4) s_2+e^{s_1} (e^{s_2}-1) s_3) s_1^3+((-5 e^{s_1}-11 e^{s_1+s_2}+2 e^{2 s_1+s_2}+14) s_2^2+((-5 e^{s_1}+7 e^{s_1+s_2}-4 e^{2 s_1+s_2}+2) s_3-6 (e^{s_1}-1) (e^{s_1+s_2}-1)) s_2+6 e^{s_1} (e^{s_2}-1) s_3^2) s_1^2+((17 e^{s_1}+5 e^{s_1+s_2}-14 e^{2 s_1+s_2}-8) s_2^3-(-13 e^{s_1}-7 e^{s_1+s_2}+10 e^{2 s_1+s_2}+10) s_3 s_2^2+s_3 (6 (e^{s_1}-1) (e^{s_1+s_2}-1)+(-7 e^{s_1}+5 e^{s_1+s_2}+4 e^{2 s_1+s_2}-2) s_3) s_2+3 e^{s_1} (e^{s_2}-1) s_3^3) s_1-s_2 (s_2+s_3) ((-13 e^{s_1}-7 e^{s_1+s_2}+10 e^{2 s_1+s_2}+10) s_2^2-2 (e^{s_1}-1) (2 e^{s_1+s_2} s_3+s_3+3 e^{s_1+s_2}-3) s_2-s_3 ((-11 e^{s_1}-11 e^{s_1+s_2}+14 e^{2 s_1+s_2}+8) s_3-6 (e^{s_1}-1) (e^{s_1+s_2}-1)))),
$
\[
k_{20, 16}(s_1, s_2, s_3)\left(e^{s_1}-1\right){}^3 \left(e^{s_2}-1\right) \left(e^{\frac{1}{2} \left(s_2+s_3\right)}-1\right) s_1 s_2 s_3 \left(s_2+s_3\right) \left(s_1+s_2+s_3\right)=
\]
$
-8 \pi  e^{2 s_1} ((5 e^{s_1}+5 e^{s_2}+e^{s_1+s_2}-11) s_2^2+2 (e^{s_1}-1) (3 (e^{s_2}-1)+(4 e^{s_2}+2) s_3) s_2+s_3 ((-e^{s_1}-13 e^{s_2}+7 e^{s_1+s_2}+7) s_3-6 (e^{s_1}-1) (e^{s_2}-1))),
$
\[
k_{20, 17}(s_1, s_2, s_3)=\frac{-k_{20, 16}^{\text{num}}(s_1, s_2, s_3)}{\left(e^{s_1}-1\right){}^3 \left(e^{s_2}-1\right) \left(e^{\frac{1}{2} \left(s_2+s_3\right)}+1\right) s_1 s_2 s_3 \left(s_2+s_3\right) \left(s_1+s_2+s_3\right)},
\]
\[
k_{20, 18}(s_1, s_2, s_3)\left(e^{s_1}-1\right) \left(e^{s_1+s_2}-1\right) \left(e^{\frac{1}{2} \left(s_1+s_2+s_3\right)}-1\right){}^3 s_1 s_2 \left(s_1+s_2\right) s_3 \left(s_2+s_3\right) \left(s_1+s_2+s_3\right){}^2=
\]
$
4 \pi  (3 ((-3 e^{s_1}-3 e^{s_1+s_2}+2 e^{2 s_1+s_2}+4) s_2+e^{s_1} (e^{s_2}-1) s_3) s_1^3+((-5 e^{s_1}-11 e^{s_1+s_2}+2 e^{2 s_1+s_2}+14) s_2^2+((-5 e^{s_1}+7 e^{s_1+s_2}-4 e^{2 s_1+s_2}+2) s_3-6 (e^{s_1}-1) (e^{s_1+s_2}-1)) s_2+6 e^{s_1} (e^{s_2}-1) s_3^2) s_1^2+((17 e^{s_1}+5 e^{s_1+s_2}-14 e^{2 s_1+s_2}-8) s_2^3-(-13 e^{s_1}-7 e^{s_1+s_2}+10 e^{2 s_1+s_2}+10) s_3 s_2^2+s_3 (6 (e^{s_1}-1) (e^{s_1+s_2}-1)+(-7 e^{s_1}+5 e^{s_1+s_2}+4 e^{2 s_1+s_2}-2) s_3) s_2+3 e^{s_1} (e^{s_2}-1) s_3^3) s_1-s_2 (s_2+s_3) ((-13 e^{s_1}-7 e^{s_1+s_2}+10 e^{2 s_1+s_2}+10) s_2^2-2 (e^{s_1}-1) (2 e^{s_1+s_2} s_3+s_3+3 e^{s_1+s_2}-3) s_2-s_3 ((-11 e^{s_1}-11 e^{s_1+s_2}+14 e^{2 s_1+s_2}+8) s_3-6 (e^{s_1}-1) (e^{s_1+s_2}-1)))), 
$
\[
k_{20, 19}(s_1, s_2, s_3)\left(e^{s_1}-1\right) \left(e^{s_1+s_2}-1\right){}^2 \left(e^{\frac{1}{2} \left(s_1+s_2+s_3\right)}-1\right){}^2 s_1 s_2 \left(s_1+s_2\right) s_3 \left(s_2+s_3\right) \left(s_1+s_2+s_3\right){}^2=
\]
$
2 \pi  (((-33+14 e^{s_1}+55 e^{s_1+s_2}-29 e^{2 s_1+s_2}-10 e^{2 s_1+2 s_2}+3 e^{3 s_1+2 s_2}) s_2+e^{s_1} (-1+e^{s_2}) (-11+23 e^{s_1+s_2}) s_3) s_1^3-((43+14 e^{s_1}-67 e^{s_1+s_2}-11 e^{2 s_1+s_2}+12 e^{2 s_1+2 s_2}+9 e^{3 s_1+2 s_2}) s_2^2+(6 (-7+5 e^{s_1}+12 e^{s_1+s_2}-8 e^{2 s_1+s_2}-5 e^{2 s_1+2 s_2}+3 e^{3 s_1+2 s_2})+(10-5 e^{s_1}+21 e^{s_1+s_2}+29 e^{2 s_1+s_2}-67 e^{2 s_1+2 s_2}+12 e^{3 s_1+2 s_2}) s_3) s_2-2 e^{s_1} (-1+e^{s_2}) s_3 ((-11+23 e^{s_1+s_2}) s_3-6 e^{s_1+s_2}+6)) s_1^2+((13-70 e^{s_1}-31 e^{s_1+s_2}+109 e^{2 s_1+s_2}+6 e^{2 s_1+2 s_2}-27 e^{3 s_1+2 s_2}) s_2^3+(8 (1+2 e^{s_1}) (-1+e^{s_1+s_2}){}^2+(23-66 e^{s_1}-67 e^{s_1+s_2}+69 e^{2 s_1+s_2}+32 e^{2 s_1+2 s_2}+9 e^{3 s_1+2 s_2}) s_3) s_2^2+s_3 ((10+15 e^{s_1}-47 e^{s_1+s_2}-63 e^{2 s_1+s_2}+49 e^{2 s_1+2 s_2}+36 e^{3 s_1+2 s_2}) s_3+22 e^{s_1}+80 e^{s_1+s_2}-56 e^{2 s_1+s_2}-46 e^{2 s_1+2 s_2}+34 e^{3 s_1+2 s_2}-34) s_2+e^{s_1} (-1+e^{s_2}) s_3^2 ((-11+23 e^{s_1+s_2}) s_3-12 (-1+e^{s_1+s_2}))) s_1-s_2 (s_2+s_3) ((-23+42 e^{s_1}+43 e^{s_1+s_2}-69 e^{2 s_1+s_2}-8 e^{2 s_1+2 s_2}+15 e^{3 s_1+2 s_2}) s_2^2-2 ((5-9 e^{s_1}-7 e^{s_1+s_2}+3 e^{2 s_1+s_2}-10 e^{2 s_1+2 s_2}+18 e^{3 s_1+2 s_2}) s_3+23 e^{s_1}+28 e^{s_1+s_2}-40 e^{2 s_1+s_2}-11 e^{2 s_1+2 s_2}+17 e^{3 s_1+2 s_2}-17) s_2-s_3 ((-13+24 e^{s_1}+29 e^{s_1+s_2}-63 e^{2 s_1+s_2}-28 e^{2 s_1+2 s_2}+51 e^{3 s_1+2 s_2}) s_3-38 e^{s_1}-64 e^{s_1+s_2}+88 e^{2 s_1+s_2}+38 e^{2 s_1+2 s_2}-50 e^{3 s_1+2 s_2}+26))),
$
\[
k_{20, 20}(s_1, s_2, s_3)=\frac{k_{20, 19}^{\text{num}}(s_1, s_2, s_3)}{\left(e^{s_1}-1\right) \left(e^{s_1+s_2}-1\right){}^2 \left(e^{\frac{1}{2} \left(s_1+s_2+s_3\right)}+1\right){}^2 s_1 s_2 \left(s_1+s_2\right) s_3 \left(s_2+s_3\right) \left(s_1+s_2+s_3\right){}^2},
\]
\[
k_{20, 21}(s_1, s_2, s_3)\left(e^{s_1}-1\right){}^3 \left(e^{s_1+s_2}-1\right){}^3 \left(e^{\frac{1}{2} \left(s_1+s_2+s_3\right)}+1\right) s_1 s_2 \left(s_1+s_2\right) s_3 \left(s_2+s_3\right) \left(s_1+s_2+s_3\right){}^2=
\]
$
2 \pi  (((-15+32 e^{s_1}+13 e^{2 s_1}-6 e^{3 s_1}+40 e^{s_1+s_2}-69 e^{2 s_1+s_2}-54 e^{3 s_1+s_2}+11 e^{4 s_1+s_2}-31 e^{2 s_1+2 s_2}+54 e^{3 s_1+2 s_2}+33 e^{4 s_1+2 s_2}+16 e^{5 s_1+2 s_2}-18 e^{3 s_1+3 s_2}+55 e^{4 s_1+3 s_2}-64 e^{5 s_1+3 s_2}+3 e^{6 s_1+3 s_2}) s_2-e^{s_1} (-5+26 e^{s_1}+3 e^{2 s_1}+5 e^{s_2}-12 e^{s_1+s_2}-63 e^{2 s_1+s_2}-2 e^{3 s_1+s_2}-14 e^{s_1+2 s_2}+27 e^{2 s_1+2 s_2}+84 e^{3 s_1+2 s_2}-25 e^{4 s_1+2 s_2}+33 e^{2 s_1+3 s_2}-82 e^{3 s_1+3 s_2}+25 e^{4 s_1+3 s_2}) s_3) s_1^3-((21-24 e^{s_1}-111 e^{2 s_1}+42 e^{3 s_1}-50 e^{s_1+s_2}-3 e^{2 s_1+s_2}+372 e^{3 s_1+s_2}-103 e^{4 s_1+s_2}+17 e^{2 s_1+2 s_2}+66 e^{3 s_1+2 s_2}-327 e^{4 s_1+2 s_2}+28 e^{5 s_1+2 s_2}+36 e^{3 s_1+3 s_2}-111 e^{4 s_1+3 s_2}+138 e^{5 s_1+3 s_2}+9 e^{6 s_1+3 s_2}) s_2^2+(2 (-23+7 e^{s_1}+61 e^{s_1+s_2}-25 e^{2 s_1+s_2}-41 e^{2 s_1+2 s_2}+17 e^{3 s_1+2 s_2}+3 e^{3 s_1+3 s_2}+e^{4 s_1+3 s_2}) (-1+e^{s_1}){}^2+(6-7 e^{s_1}-20 e^{2 s_1}+45 e^{3 s_1}+5 e^{s_1+s_2}-108 e^{2 s_1+s_2}+129 e^{3 s_1+s_2}-98 e^{4 s_1+s_2}-56 e^{2 s_1+2 s_2}+201 e^{3 s_1+2 s_2}-42 e^{4 s_1+2 s_2}-31 e^{5 s_1+2 s_2}+117 e^{3 s_1+3 s_2}-302 e^{4 s_1+3 s_2}+149 e^{5 s_1+3 s_2}+12 e^{6 s_1+3 s_2}) s_3) s_2+2 e^{s_1} (-1+e^{s_2}) s_3 ((5-26 e^{s_1}-3 e^{2 s_1}-14 e^{s_1+s_2}+60 e^{2 s_1+s_2}+2 e^{3 s_1+s_2}+33 e^{2 s_1+2 s_2}-82 e^{3 s_1+2 s_2}+25 e^{4 s_1+2 s_2}) s_3-4 (-1+e^{s_1}){}^2 (2-7 e^{s_1+s_2}+5 e^{2 s_1+2 s_2}))) s_1^2+(-(-3+48 e^{s_1}-183 e^{2 s_1}+66 e^{3 s_1}+20 e^{s_1+s_2}-213 e^{2 s_1+s_2}+582 e^{3 s_1+s_2}-173 e^{4 s_1+s_2}-59 e^{2 s_1+2 s_2}+294 e^{3 s_1+2 s_2}-555 e^{4 s_1+2 s_2}+104 e^{5 s_1+2 s_2}+18 e^{3 s_1+3 s_2}-57 e^{4 s_1+3 s_2}+84 e^{5 s_1+3 s_2}+27 e^{6 s_1+3 s_2}) s_2^3+(-16 (-1+e^{s_1}){}^2 (-1-3 e^{s_1}+e^{2 s_1+s_2}) (-1+e^{s_1+s_2}){}^2-(-9+64 e^{s_1}-181 e^{2 s_1}+102 e^{3 s_1}+54 e^{s_1+s_2}-357 e^{2 s_1+s_2}+624 e^{3 s_1+s_2}-249 e^{4 s_1+s_2}-141 e^{2 s_1+2 s_2}+534 e^{3 s_1+2 s_2}-549 e^{4 s_1+2 s_2}+84 e^{5 s_1+2 s_2}+72 e^{3 s_1+3 s_2}-169 e^{4 s_1+3 s_2}+34 e^{5 s_1+3 s_2}+87 e^{6 s_1+3 s_2}) s_3) s_2^2+s_3 ((6-11 e^{s_1}-28 e^{2 s_1}-39 e^{3 s_1}-39 e^{s_1+s_2}+156 e^{2 s_1+s_2}+21 e^{3 s_1+s_2}+78 e^{4 s_1+s_2}+96 e^{2 s_1+2 s_2}-267 e^{3 s_1+2 s_2}-90 e^{4 s_1+2 s_2}+45 e^{5 s_1+2 s_2}-87 e^{3 s_1+3 s_2}+194 e^{4 s_1+3 s_2}+25 e^{5 s_1+3 s_2}-60 e^{6 s_1+3 s_2}) s_3-2 (-1+e^{s_1}){}^2 (15-15 e^{s_1}-61 e^{s_1+s_2}+25 e^{2 s_1+s_2}+89 e^{2 s_1+2 s_2}-17 e^{3 s_1+2 s_2}-43 e^{3 s_1+3 s_2}+7 e^{4 s_1+3 s_2})) s_2-e^{s_1} (-1+e^{s_2}) s_3^2 ((5-26 e^{s_1}-3 e^{2 s_1}-14 e^{s_1+s_2}+60 e^{2 s_1+s_2}+2 e^{3 s_1+s_2}+33 e^{2 s_1+2 s_2}-82 e^{3 s_1+2 s_2}+25 e^{4 s_1+2 s_2}) s_3-8 (-1+e^{s_1}){}^2 (2-7 e^{s_1+s_2}+5 e^{2 s_1+2 s_2}))) s_1-s_2 (s_2+s_3) ((-9+40 e^{s_1}-85 e^{2 s_1}+30 e^{3 s_1}+30 e^{s_1+s_2}-141 e^{2 s_1+s_2}+264 e^{3 s_1+s_2}-81 e^{4 s_1+s_2}-45 e^{2 s_1+2 s_2}+174 e^{3 s_1+2 s_2}-261 e^{4 s_1+2 s_2}+60 e^{5 s_1+2 s_2}-e^{4 s_1+3 s_2}+10 e^{5 s_1+3 s_2}+15 e^{6 s_1+3 s_2}) s_2^2+2 (-1+e^{s_1}) ((-1+e^{s_1}) (15-31 e^{s_1}-45 e^{s_1+s_2}+81 e^{2 s_1+s_2}+33 e^{2 s_1+2 s_2}-57 e^{3 s_1+2 s_2}-3 e^{3 s_1+3 s_2}+7 e^{4 s_1+3 s_2})+(3-10 e^{s_1}+15 e^{2 s_1}-12 e^{s_1+s_2}+48 e^{2 s_1+s_2}-36 e^{3 s_1+s_2}+27 e^{2 s_1+2 s_2}-66 e^{3 s_1+2 s_2}+15 e^{4 s_1+2 s_2}+6 e^{3 s_1+3 s_2}-20 e^{4 s_1+3 s_2}+30 e^{5 s_1+3 s_2}) s_3) s_2+s_3 ((3-14 e^{s_1}+35 e^{2 s_1}-6 e^{s_1+s_2}+21 e^{2 s_1+s_2}-96 e^{3 s_1+s_2}+9 e^{4 s_1+s_2}-9 e^{2 s_1+2 s_2}+12 e^{3 s_1+2 s_2}+99 e^{4 s_1+2 s_2}-30 e^{5 s_1+2 s_2}-12 e^{3 s_1+3 s_2}+53 e^{4 s_1+3 s_2}-110 e^{5 s_1+3 s_2}+45 e^{6 s_1+3 s_2}) s_3-2 (-1+e^{s_1}){}^2 (7-15 e^{s_1}-21 e^{s_1+s_2}+57 e^{2 s_1+s_2}+33 e^{2 s_1+2 s_2}-81 e^{3 s_1+2 s_2}-19 e^{3 s_1+3 s_2}+39 e^{4 s_1+3 s_2})))),
$
\[
k_{20, 22}(s_1, s_2, s_3)=\frac{-k_{20, 21}^{\text{num}}(s_1, s_2, s_3)}{\left(e^{s_1}-1\right){}^3 \left(e^{s_1+s_2}-1\right){}^3 \left(e^{\frac{1}{2} \left(s_1+s_2+s_3\right)}-1\right) s_1 s_2 \left(s_1+s_2\right) s_3 \left(s_2+s_3\right) \left(s_1+s_2+s_3\right){}^2},
\]
\[
k_{20, 23}(s_1, s_2, s_3)=-\frac{16 \pi  e^{2 \left(s_1+s_2\right)} \left(\left(e^{s_2} \left(e^{s_1} \left(2 e^{s_2}+1\right)+1\right)-4\right) s_1+3 e^{s_2} \left(e^{s_1}-1\right) s_2\right)}{\left(e^{s_2}-1\right) \left(e^{s_1+s_2}-1\right){}^3 \left(e^{\frac{s_3}{2}}+1\right) s_1 s_2 \left(s_1+s_2\right) \left(s_1+s_2+s_3\right)}, 
\]
\[
k_{20, 24}(s_1, s_2, s_3)=\frac{16 \pi  e^{2 \left(s_1+s_2\right)} \left(\left(e^{s_2} \left(e^{s_1} \left(2 e^{s_2}+1\right)+1\right)-4\right) s_1+3 e^{s_2} \left(e^{s_1}-1\right) s_2\right)}{\left(e^{s_2}-1\right) \left(e^{s_1+s_2}-1\right){}^3 \left(e^{\frac{s_3}{2}}-1\right) s_1 s_2 \left(s_1+s_2\right) \left(s_1+s_2+s_3\right)}. 
\]
}

\section*{Acknowledgments}

F.F. is indebted to Eric Schost for very helpful and 
generous conversations on computer algebra programming, 
which played a crucial role for handling the heavy calculations 
of the present paper. He also thanks the Institut des Hautes \'Etudes 
Scientifiques (I.H.E.S.), in particular Francois Bachelier in the I.T.  
department of the institute, for the excellent environment and 
facilities that were used for carrying out this work partially during 
stays in the Fall of 2013 and the Summer of 2015.

\end{document}